\documentclass[12pt, a4paper, parskip=half, abstracton]{scrartcl}
\pdfoutput=1

\usepackage{array}
\usepackage{marginnote}
\usepackage[dvipsnames]{xcolor}
\usepackage{amscd,amssymb,amsfonts,amsmath,latexsym,amsthm}
\usepackage{hyperref}
\usepackage[all,cmtip]{xy}
\usepackage{mathrsfs}
\usepackage{graphicx}
\usepackage{bm} % for bold symbols in math mode \boldsymbol{}
\usepackage{enumitem} % for the \resume command 

\usepackage{etoolbox}
\def\hB{\hspace*{\fill}$\qed$}% \newline\noindent}

\usepackage[nottoc]{tocbibind} % Bibliography im toc

\usepackage{defs_pp1}
\usepackage{slashed}
\usepackage[utf8]{inputenc}
\usepackage[stretch=10,shrink=30]{microtype}
\usepackage[english]{babel}
\usepackage{mathtools}
\usepackage{bm}
\usepackage{esvect}

\title{Additive $\bm{C^{*}}$-categories and $\bm{K}$-theory}
\author{
Ulrich Bunke\thanks{Fakult{\"a}t f{\"u}r Mathematik,
Universit{\"a}t Regensburg,
93040 Regensburg,
GERMANY\newline
\href{mailto:ulrich.bunke@mathematik.uni-regensburg.de}{ulrich.bunke@mathematik.uni-regensburg.de}}  
\and
Alexander Engel\thanks{Universit\"at Greifswald,
Walther-Rathenau-Strasse 47,
17489 Greifswald,
GERMANY\newline
\href{mailto:alexander.engel@uni-greifswald.de}{alexander.engel@uni-greifswald.de}}
}

\numberwithin{equation}{section}
\newtheorem{theorem}{Theorem}[section] 
\newtheorem{prop}[theorem]{Proposition}
\newtheorem{lem}[theorem]{Lemma}

\newtheorem{ddd}[theorem]{Definition}
\newtheorem{kor}[theorem]{Corollary}

\newtheorem{construction}[theorem]{Construction}

\theoremstyle{remark}
\theoremstyle{definition}

\newtheorem{ex}[theorem]{Example}
\newtheorem{rem}[theorem]{Remark}

\newcommand{\la}{\mathrm{la}}
\newcommand{\disc}{\mathrm{disc}}
\newcommand{\std}{\mathrm{std}}
\newcommand{\LM}{\mathrm{LM}}
\newcommand{\RM}{\mathrm{RM}}

\newcommand{\Lzwei}{\mathbf{L}^{2}}
 \newcommand{\Ban}{\mathbf{Ban}}

\newcommand{\Homol}{\mathrm{Hg}}

\newcommand{\Ho}{\mathbf{Ho}}

\newcommand{\bX}{\mathbf{X}}

\newcommand{\Res}{\mathrm{Res}}
\newcommand{\Rep}{\mathbf{Rep}}

\newcommand{\Orb}{\mathbf{Orb}}

\newcommand{\Hilb}{\mathbf{Hilb}}

\newcommand{\cR}{\mathcal{R}}

\newcommand{\bQ}{\mathbf{Q}}

\newcommand{\Fin}{\mathbf{Fin}}
\newcommand{\Am}{\mathbf{Am}}
\newcommand{\KAm}{\mathbf{K\textbf{-}Am}}

\newcommand{\Ob}{\mathrm{Ob}}

\newcommand{\bB}{{\mathbf{B}}}

\newcommand{\incl}{\mathrm{incl}}
\newcommand{\Emb}{\mathrm{Emb}}

\newcommand{\DL}{\mathrm{DL}}
\newcommand{\Groupoids}{\mathbf{Groupoids}}

\newcommand{\bL}{{\mathbf{L}}}
\newcommand{\bH}{{\mathbf{H}}}
\newcommand{\bM}{\mathbf{M}}

\newcommand{\bF}{{\mathbf{F}}}

\newcommand{\cL}{{\mathcal{L}}}

\newcommand{\PSh}{{\mathbf{PSh}}}

\newcommand{\bA}{{\mathbf{A}}}

\newcommand{\bK}{{\mathbf{K}}}

\newcommand{\cD}{{\mathcal{D}}}

 \newcommand{\Cat}{{\mathbf{Cat}}}

\DeclareMathOperator{\proj}{proj}

\newcommand{\IK}{\mathbb{K}}
\newcommand{\K}{\mathrm{K}}

\newcommand{\aMHom}{\mathtt{M}^{\mathrm{alg}}\mathtt{Hom}}
\newcommand{\MHom}{\mathtt{MHom}}

\newcommand{\Idem}{\mathrm{Idem}}

\newcommand{\Spc}{\mathbf{Spc}}

\newcommand{\IR}{\mathbb{R}}
\newcommand{\IC}{\mathbb{C}}

\newcommand{\Ccat}{{\mathbf{C}^{\ast}\mathbf{Cat}}}
\newcommand{\Calg}{{\mathbf{C}^{\ast}\mathbf{Alg}}}

\newcommand{\bd}{\mathrm{bd}}
\newcommand{\op}{\mathrm{op}}

\newcommand{\add}{\mathrm{add}}

\newcommand{\Wcat}{W^{*}\mathbf{Cat}}
\newcommand{\nCcat}{C^{*}\mathbf{Cat}^{\mathrm{nu}}}
\renewcommand{\Ccat}{C^{*}\mathbf{Cat}}
\newcommand{\alg}{\mathrm{alg}}
\newcommand{\nClincat}{\mathbf{{}^{*}\Cat^{\mathrm{nu}}_{\C}}}
\newcommand{\npClincat}{\mathbf{{}_{\mathrm{pre}}^{*}\Cat^{\mathrm{nu}}_{\C}}}
\newcommand{\pClincat}{\mathbf{{}_{\mathrm{pre}}^{*}\Cat_{\C}}}

\newcommand{\Clincat}{\mathbf{{}^{*}\Cat_{\C}}}
\renewcommand{\Calg}{C^{*}\mathbf{Alg}}
\newcommand{\nCalg}{C^{*}\mathbf{Alg}^{\mathbf{nu}}}
\renewcommand{\nCalg}{C^{*}\mathbf{Alg}^{\mathrm{nu}}}

\newcommand{\npAlgc}{{}_{\mathrm{pre}}^{*}\mathbf{Alg}^{\mathrm{nu}}_{\mathbb{C}}}

\newcommand{\Compl}{\mathrm{Compl}}
 \newcommand{\Kcat}{\mathrm{K}^{\mathrm{C^{*}Cat}}}
  
  \newcommand{\Kast}{\mathrm{K}^{C^{*}}}

\newcommand{\Morita}{\mathrm{Morita}}

\begin{document}  

\maketitle

\begin{abstract}
We introduce and study the notion of an orthogonal sum of a (possibly infinite) family of objects in a $C^{*}$-category.  Further, we construct   reduced crossed products  of $C^{*}$-categories with groups.   We axiomatize the basic properties of the $K$-theory for $C^{*}$-categories
in the notion of a homological functor. We then study various rigidity properties of homological functors in general, and special additional  features   of the $K$-theory of $C^{*}$-categories.
As an application  we construct and study interesting functors on the orbit category of a group from $C^{*}$-categorical data.
%
%   
%
%notion of a homological functor for $C^{*}$-categories and show that topological $K$-theory for $C^{*}$-categories is an example. \Alex{Furthermore, we} show that this $K$-theory functor preserves \Alex{(possibly infinite)} products and is Morita invariant. \Alex{Finally, we provide a generalization of the Davis--Lück functor for topological $K$-theory from trivial coefficients $\IC$ to coefficients in an arbitrary unital $C^{*}$-algebra $A$ with $G$-action.}
\end{abstract}

\tableofcontents
\setcounter{tocdepth}{5}

% \section{Controlled objects}

 \newcommand{\sep}{\mathrm{sep}}
 \newcommand{\free}{\mathrm{free}}
  \newcommand{\fg}{\mathrm{fg}}
\newcommand{\fin}{\mathrm{fin}}

 \section{Introduction}

 The goal of this paper is to provide a reference for foundational results on $C^{*}$-categories and their topological $K$-theory. The three main themes
 are  orthogonal sums of (infinite) families of objects in a $C^{*}$-category, reduced crossed products of $C^{*}$-categories with groups, and rigidity properties 
 of the $K$-theory of $C^{*}$-categories and more general homological functors.      
  The results of the present paper will be used in the subsequent papers  \cite{coarsek}, \cite{KKG} and \cite{bel-paschke}.
 
 % In \cite{coarsek} we   construct equivariant coarse homology theories associated to $C^{*}$-categories and apply them to prove split-injectivity results for assembly map, see Remark \ref{wtoigjwotgergwegwerg} for a more detailed preview. In  \cite{compass} we will construct a stable $\infty$-category $KK^{G}$ modelling equivariant Kasparov $KK$-theory and  derive an equivariant version of Paschke duality which is in turn used to compare
% the analytic and homotopy theoretic versions of  the Baum--Connes assembly map.

The notion of a $C^{*}$-category was introduced in \cite{ghr}; see also the further references \cite{DellAmbrogio:2010aa}, \cite{davis_lueck}, \cite{joachimcat}, \cite{mitchc}, \cite{Mitchener:aa}, \cite{startcats}. The category $\Ccat$ of $C^{*}$-categories has an interesting homotopy theory  based on the notion of unitary equivalence
which is studied in   \cite{DellAmbrogio:2010aa} and \cite{startcats}. 

 The main topic of  \cite{crosscat} are the categorical properties of the category $\nCcat$ of possibly non-unital $C^{*}$-categories.   In particular  it was shown that this category 
 is  complete and cocomplete.  Furthermore, for  $C^{*}$-categories with $G$-action the  maximal crossed product  was  introduced  and recognized as a homotopy colimit.
 
The main goal  of    \cite{coarsek}  is to  construct  equivariant coarse homology theories  in the sense of \cite{buen}, \cite{equicoarse} associated to a coefficient $C^{*}$-category. Thereby we   follow  the  recipe  of  \cite{buen}, \cite{equicoarse} and \cite{unik}.  The paper  \cite{coarsek}  concentrates 
on the construction of $C^{*}$-categories of controlled objects and the verification of their homological properties.  The present paper provides all the necessary background concerning  orthogonal sums, reduced crossed products and homological functors.  
%The present paper provides the 
%
%
%
%
%In a first  step we  functorially  associate to every bornological coarse space a $C^{*}$-category of controlled  objects in the   coefficient  category. We then apply a homological functor from $C^{*}$-categories to some target category. In order to verify that this construction satisfies the axioms of an equivariant coarse homology theory we first verify related properties of these controlled object categories and then use the defining properties of the homological functor.
% While   \cite{coarsek}  focusses on the controlled object functors, the  purpose of the present paper  is to  provide the necessary  background about  $C^{*}$-categories and  homological functors.

In  \cite{KKG} we construct a stable $\infty$-category $\mathrm{KK}^{G}$ modelling equivariant Kasparov $K\!K$-theory. In \cite{bel-paschke} we then derive an equivariant version of Paschke duality which is in turn used to compare
the analytic and homotopy theoretic versions of  the Baum--Connes assembly map.  Both papers use orthogonal sums, crossed products and various properties of $K$-theory shown in the present paper.

%The recipe of this construction  is similar to the one used in  \cite{unik}, which involves left-exact $\infty$-categories  instead of $C^{*}$-categories. 
%  In \cite{unik} the construction proceeds in two steps. Given a left-exact $\infty$-category in the first step one associates to a bornological coarse space a new left-exact $\infty$-category of controlled objects, and in the second one applies a homological functor to the result. 
%In order to show that this construction produces a coarse homology theory one verifies some natural properties of the controlled object functors. The rest of the argument uses  features of   the category of  left-exact $\infty$-categories and the homological functor which can be captured in an axiomatic way.
%The  paper  \cite{coarsek} follows the same route and is focussed on the controlled object functors. 
% The goal of the present paper is to  provide the necessary  background about  $C^{*}$-categories and  homological functors.

In the remainder of this introduction we describe the content of the present paper in greater detail.

Section \ref{qrioghqoirfewffeqwfqef} serves as a reminder of basic notions from the theory of $C^{*}$- and $W^{*}$-categories. The  introduction of the $W^{*}$-envelope of a $C^{*}$-category   in  Theorem \ref{rijogegergwerg9}
seems to fill a gap in the literature. In Section \ref{weigjogfgsgdfgsfdg} we present a detailed discussion of the concept of a multiplier $C^{*}$-category
and its relation with the $W^{*}$-envelope. We use multiplier categories in order to extend the notion of a unitary equivalence between $C^{*}$-categories to the non-unital case. In Section \ref{efigosgsfgsfdgsdfg}
we use the two-categorical structure of the category of $C^{*}$-categories in order to   introduce 
  the notion of a weakly  equivariant functor  in Definition \ref{rgjeqrgoieqrgrgqewgfq}.

The first main topic of the present paper are orthogonal sums of families of objects in a $C^{*}$-category which will be defined in Section \ref{qroifjqerofwefewfefwqfe}.
We have various reasons for considering such sums:
\begin{enumerate}
\item  Let $X$ be a set. The main feature of the definition of an $X$-controlled object $C$    in a $C^{*}$-category \cite{coarsek}  is a presentation of $C$ as an orthogonal  sum of a family of objects $(C_{x})_{x\in X}$ indexed by the set $X$.  %  See 
%Construction  \ref{qrfiohqrioffdewfeqwdqewdqed} for an instance of this idea.
\item In Definition \ref{ewtiojgwergerggwggr} the reduced crossed product of a $C^{*}$-category $  \bC$ with $G$-action will be constructed by completing  the algebraic crossed product (see \cite[Def.\ 5.1]{crosscat}) with respect to a norm obtained from a representation  on a $C^{*}$ category derived from $\bC$ which we will denote suggestively by $\bL^{2}(G,  \bC)$. In particular,
the morphism spaces  of the latter are given, using  orthogonal sums of families of objects indexed by $G$,  in terms of the morphism spaces of $\bC$  by $\Hom_{\bW^{\mathrm{nu}}\bC}(\bigoplus_{g\in G} gC,\bigoplus_{g\in G} gC')$.
\item Frequently the fact that a $C^{*}$-category $\bC$ has trivial $K$-theory is 
deduced from an Eilenberg swindle. This will  be encoded in the notion of flasqueness of $\bC$,  see the Definition~\ref{bjiobdfbfsferq}. The usual verification
of flasqueness  of $\bC$ consists in showing that for every object $C$  the infinite sum $\bigoplus_{\nat}C$ of countably many copies of $C$ exists in $\bC$.
\end{enumerate}

%The first \Alex{main} topic of the present paper is the notion of an orthogonal sum of a  family of objects in a $C^{*}$-category. This is  in particular interesting if the family is infinite. 
%In the following we explain what kind of problem we will encounter.

 If $A$ is a  $C^{*}$-algebra, then the category $\Hilb(A)$ of Hilbert $A$-modules is an example of a $C^{*}$-category.  Given a family $(M_{i})_{i\in I}$ of objects  in $\Hilb(A)$ we can construct the classical orthogonal sum $\bigoplus_{i\in I}M_{i}$ in    $  \Hilb(A)$ as a completion of the algebraic direct sum with respect to the norm induced by  an  explicitly given $A$-valued scalar product.
 One can then  characterize the sum $\bigoplus_{i\in I}M_{i}$ by describing     the spaces of bounded adjointable operators $B(\bigoplus_{i\in I}M_{i},M)$ or $B(M,\bigoplus_{i\in I}M_{i})$  for any object $M$ in $\Hilb(A)$ in terms of the spaces $B(M_{i},M)$ and $B(M,M_{i})$ for all $i$.
For general $C^{*}$-categories we will use a similar idea.
%For general $C^{*}$-categories
% we do not have a direct access to its objects. 
%our idea is to find  characterizations of the spaces $B(\bigoplus_{i\in I}M_{i},M)$ and $B(M,\bigoplus_{i\in I}M_{i})$ without using the construction of $\bigoplus_{i\in I}M_{i}$. The details are  inspired by the abstract construction of multiplier algebras. 
Our final definition of an orthogonal sum of a family of objects in a unital $C^{*}$-category is  Definition \ref{erogwfsfdvbsbfdbsdfbsfdbv}.
In Theorem  \ref{ewrgowergwerfgrfwerf} we show that  in the case of $\Hilb(A)$  {the} classical  definition of an orthogonal sum coincides with our notion of an orthogonal  sum interpreted in the $W^{*}$-envelope $\bW\Hilb(\bC)$.
Section \ref{sec_morphisms_sums} provides additional material which is helpful when working with sums.
In Remark \ref{euwifhqiufqwfewffqewffqfef} we show that our notion of an orthogonal sum is equivalent to the notion previously 
 introduced in \cite{fritz}.
 
The notion of an orthogonal sum introduced in Definition  \ref{erogwfsfdvbsbfdbsdfbsfdbv} is not adjusted to multiplier categories. 
In this respect the notion of orthogonal sums  (in the present paper we call them  AV-sums) due to Antoun and Voigt  \cite{antoun_voigt}\footnote{This preprint appeared while we were finishing a first version of the present paper.}
and  described  in Definition \ref{peofjopwebgwregwgreg} is better behaved. It will be discussed in detail in 
  Section~\ref{sec_comparison_other_definition}.  %we review in detail the notion (see Definition \ref{peofjopwebgwregwgreg}) of orthogonal sums  (called AV-sums) due to Antoun and Voigt   \cite{antoun_voigt}\footnote{This preprint appeared while we were finishing a first version of the present paper.}. 
 %The notion of an AV-sum is based on multiplier categories and differs from our notion. In order to produce a detailed reference we first recall the construction of multiplier categories. 
In   Theorem  \ref{ewrgowergwerfgrfwerf} we also show\footnote{This  fact was   stated in \cite{antoun_voigt}.} that  classical sums of
Hilbert $A$-modules correspond to AV-sums interpreted in the ideal $\Hilb_{c}(A)$ of compact operators in $\Hilb(A)$.

In Section \ref{wtogwepgfereggwrferf} we describe a Yoneda type embedding of any $C^{*}$-category into a certain $C^{*}$-category of Hilbert modules. In  Theorem \ref{mainyondea}  we state its compatibility with various notions of orthogonal sums. The   Yoneda type embedding will be used subsequently in order to  find for every $C^{*}$-category an embedding into some $C^{*}$-category admitting all small sums.

Given a family of functors with target in a $C^{*}$-category
we can form the orthogonal sum of these functors objectwise provided
the target category admits the corresponding sums.  This and related material is discussed in Section \ref{qergoijeroigjeroigjqer90gu9384tuergqegerggw}.
In particular
we use this sum of functors in order to introduce   the notion of flasqueness   in Definition  \ref{bjiobdfbfsferq}.
In the equivariant case,
since sums are only unique up to unique unitary isomorphism,  an orthogonal   sum of equivariant functors is in general not equivariant anymore, but by Proposition \ref{rgoihegiqwefqfqwfqwefqwefq} it extends to a  weakly equivariant functor.

%
%  question 
%whether one can construct additive completion functors with good formal properties. Since $C^{*}$-categories are enriched over abelian groups it is clear that unital functors between unital $C^{*}$-categories preserve finite sums.  
%In the case of infinite sum we provide criteria for the preservation of sums on the one hand (e.g. Proposition \ref{kor_compare_orthogonal_sums}) and examples of functors which do not preserve sums on the other (see e.g. Example \ref{ex_sum_becomes_different}).  
% In Proposition \ref{ergijoewgergrweg9} we show, that in contrast to the case of finite sums, infinite sum completion functors can not serve as left adjoints of a Bousfield localization.
%
%
%

% of the notion of an orthogonal sum of a family of objects which works in arbitrary \Alex{unital} $C^{*}$-categories.
%   We then verify that the notion of an orthogonal sum has various (surely expected) properties. 
%    \Alex{Shortly before we finished this paper, Antoun and Voigt put their preprint \cite{antoun_voigt} onto the arXiv where they also discuss orthogonal sums for infinite families of objects in a $C^*$-category. In  we provide a comparison of their definition to ours.}

Given a   $C^{*}$-category $  \bC$ with an action of a group $G$ in \cite{crosscat}  we  introduced  the maximal crossed product  $\bC\rtimes G$ as the completion of an algebraic crossed product with respect to the maximal norm.  Equivalently, in the unital case,  it can be understood as the $C^{*}$-category of homotopy $G$-orbits in $\bC$.   As in the case of $C^{*}$-algebras besides the maximal  one there are  other   choices for the completion of the algebraic crossed product. In general these choices are less functorial but analytically more interesting. One natural choice is the  reduced crossed product ${\bC}\rtimes_{r}G$. The main result of that section is Theorem \ref{ejgwoierferfewrferfwe}  which 
asserts that the reduced crossed product functor  exists and  states its basic properties.  As explained above, 
the construction of the reduced crossed product   heavily relies on our notion of infinite orthogonal sums in $C^*$-categories. Our reason for considering the reduced crossed product is twofold. First of all it appears naturally in the calculation of the values on discrete bornological coarse spaces of the coarse homology theories constructed in \cite{coarsek}. On the other hand, the functors on the orbit category which  provide the topological side of the Baum--Connes assembly map (see Definition \ref{iuhquifhiwefqewfqwefqwefef}) involve the reduced crossed product in their construction. As for a $C^{*}$-algebra,  also for a $C^{*}$-category $\bC$  with an action by an amenable group $G$ the canonical functor
$\bC\rtimes G\to \bC\rtimes_{r} G$ from the maximal to the  reduced crossed product is an   isomorphism.

%\Alex{Our reason for introducing the reduced crossed product is that the functors on the orbit category for which we can show injectivity of assembly maps 
%involve this reduced crossed product. We describe these functors in Definition~\ref{iuhquifhiwefqewfqwefqwefef}
%and calculate their values on orbits in Proposition~\ref{wetgiojweirogwergrwegwergw}. As explained in Example \ref{efgiuwheviufvfsdv} our construction in particular extends the construction of the topological $K$-theory functor $\bK^{\DL}\colon G\Orb\to \Sp$ (which goes into the homotopy theoretic Baum--Connes assembly map) introduced by Davis--Lück \cite{davis_lueck} from trivial coefficients $\C$ to coefficients in an arbitrary unital $C^{*}$-algebra $A$ with $G$-action.
%For completeness we will also state Theorem~\ref{qrfoiqfewewfqfewfeqf} shown in \cite{coarsek} about these functors, namely that they are CP-functors. The notion of a CP-functor was introduced in \cite{desc} and has interesting consequences for the injectivity of assembly maps involving such functors. For example, if $G$ is a group that admits a finite-dimensional model for its classifying space $E^{\mathrm{top}}_{\Fin}G$ for proper actions and if $G$ has finite decomposition complexity, then the associated assembly map admits a left-inverse \cite[Cor.~1.14]{desc}.}

In Definition \ref{oihgjeorgwergergergwerg} of a homological functor we 
axiomatize some of the properties of the $K$-theory functor $\Kcat$ for $C^{*}$-categories. The construction of coarse homology theories in \cite{coarsek}
only relies on these axioms. In Section \ref{ewgiowegergregregergrwrf} we further derive some immediate consequences of the axioms like additivity or annihilation of flasques.  

In Section \ref{ergiojerfrfqefqwefqwef} we verify that  the $K$-theory functor  for $C^{*}$-categories    $\Kcat$  introduced by \cite{joachimcat}  is  indeed an example of a homological functor.

In Theorem \ref{ojgweorergwerfewrfwerfw} we show that the $K$-theory functor for $C^{*}$-categories  $\Kcat$ preserves arbitrary  products of additive $C^{*}$-category.  This a special property of $K$-theory  which we do not expect for arbitrary homological functors.  It is similar in spirit  with the results shown in \cite{MR1351941}, \cite{Kasprowski:2017aa}, \cite{kaswin}. The fact that 
 $\Kcat$ preserves products is one of the main inputs for the  proof of  Theorem \ref{qrfoiqfewewfqfewfeqf}  provided in  \cite{coarsek}.

In Section \ref{erogijogergergwgerg9} we consider the algebraic notion of Morita equivalences between $C^{*}$-categories introduced in  \cite{MR3123758} and homological functors preserving them. In Theorem \ref{wtoijgwergergrewfwergrg} we show that   the $K$-theory functor for $C^{*}$-categories $\Kcat$ preserves Morita equivalences. Furthermore, in  Proposition \ref{werijguhwerigvwergwer9} we show that  the reduced crossed product functor preserves Morita equivalences.   

So far Morita equivalences and  idempotent completions  
were considered for unital $C^{*}$-categories. In  Definitions  \ref{sitogdghsgfgsfg} and \ref{qrihweiofdwefffqfwefqwef}  we generalize these notions  to the  relative situation of an ideal in a unital $C^{*}$-category
and show  in Propositions \ref{qrgojqeorgqfewfeqwfqewf} and \ref{qrgojqeorgqfewfeqwfqewf1} that Morita invariant functors send relative Moria equivalences and relative idempotent completions to equivalences.
In Definition \ref{wkjrthowrtgewgwergw} we furthermore introduce the notion of a Murray-von Neumann (MvN) equivalence between morphisms between $C^{*}$-categories and verify in Proposition \ref{wkthgwfgwegwesdf}  that homological functors send MvN-equivalent morphisms to equivalent morphisms.

 It is well-known that   the left upper corner embedding of a unital $C^{*}$-algebra into the compact operators on a 
free Hilbert $C^{*}$-module induces an equivalence in the $K$-theory of $C^{*}$-algebras.  
In Section  \ref{eoigwjeoigregewrgwergwerg} we generalize 
this situation by introducing in Definition \ref{ergiowergerfwrfwfref} the notion of a weak Morita equivalence between $C^{*}$-categories.  As in the case of $C^{*}$-algebras it is a condition about the  approximability of morphisms in the bigger category by conjugates of morphisms in the smaller.  In particular, the notion of a weak Morita equivalence  belongs to the functional analytic corner of the field and has no counterpart in algebra. Our main result is Theorem \ref{eowigjwoigwregrgwegrwreg} saying that the K-theory of $C^{*}$-categories $\Kcat$ sends weak Morita equivalences to equivalences. This result will be used in \cite{bel-paschke}.

The final Section \ref{qeroigjeoigergergewgwegewrgwerg} is devoted to the construction of equivariant homology theories   from the data of a unital  $C^{*}$-category $\bC$ with a strict $G$-action on the one hand,  and some auxiliary functor $\Homol\colon \nCcat \to \bS$ (e.g.\ $\Kcat\colon \nCcat\to \Sp$) on the other.  Here in view of Elmendorf's theorem equivariant homology theories are by definition   functors $G\Orb\to \bS$ from the orbit category $G\Orb$ of $G$ to a cocomplete stable $\infty$-category $\bS$.
 
 Using homotopy theoretic methods, following \cite{startcats} we construct a functor 
$$\Homol^{G}_{\bC,\max}\colon G\Orb\to \bS$$ whose values on orbits $G/H$ are given by  $\Homol( \bC\rtimes H)$  and involve  the  maximal crossed product. 

Using the theory of orthogonal sums and reduced crossed products of $C^{*}$-categories with groups   we furthermore provide an explicit construction  of a functor 
$$\Homol^{G}_{\bC,r}\colon G\Orb\to \bS$$ together with a comparison map $\Homol^{G}_{\bC ,\max}\to \Homol^{G}_{  \bC,r}$
which on orbits $G/H$ reduces to the canonical morphism $\Homol(\bC\rtimes H)\to \Homol(\bC\rtimes_{r}H)$ from the maximal to the reduced crossed product. 

If $A$ is a unital algebra and $\bC:=\Hilb(A)^{\mathrm{fg,proj}}$ is the full sub-category of $\Hilb(A)$ of finitely generated projective Hilbert $A$-modules, then, as shown in Proposition \ref{prop_compare_DL} our functor is equivalent to the  functor constructed by Davis--L\"uck in  \cite{davis_lueck}.
While the homotopy theoretic approach provides insights in the formal properties of $\Homol^{G}_{ \bC,\max}$, the functor $\Homol^{G}_{\bC,r}$ is relevant for the Baum--Connes assembly map as discussed in \cite{bel-paschke} and is the subject of one of the main results of \cite{coarsek} reproduced here as Theorem \ref{qrfoiqfewewfqfewfeqf}.
We   refer to Proposition  \ref{eroighqeirgoregwergwergewrgewrgwe} for an interesting application of the comparison map.

{\em Acknowledgement: U.B.\ was supported by the SFB 1085 (Higher Invariants) funded by the Deutsche Forschungsgemeinschaft (DFG).}

\section{\texorpdfstring{$\bm{\C}$}{C}-linear \texorpdfstring{$*$}{star}-categories and \texorpdfstring{$\bm{C^{*}}$}{Cstar}- and  \texorpdfstring{$\bm{W^{*}}$}{Wstar}-categories}\label{qrioghqoirfewffeqwfqef}
 
 In this section we recall  the definitions of   $\C$-linear $*$-categories,    $C^{*}$-categories and $W^{*}$-categories.  These  concepts  were originally introduced   in \cite{ghr}. 
In  Theorem \ref{rijogegergwerg9} we show that  the inclusion of $W^{*}$-categories and normal functors into unital $C^{*}$-categories and unital functors has a left-adjoint which sends a    $C^{*}$-category to its $W^{*}$-envelope. This statement seems to fill a gap in the literature.

  %We refer to \cite{startcats}, \cite{crosscat} for further details and the comparison with the classical definitions \cite[Rem.~2.15]{startcats}.
   
 In order to fix set-theoretic size issues we consider a sequence of three Grothendieck universes whose elements are called  very small, small and large sets, respectively.

 A   possibly non-unital small category $\bC$ consists of a  small set of objects $\Ob(\bC)$, for every two objects $C,C'$ a small set of
 morphisms $\Hom_{\bC}(C,C')$, and an associative law of composition.  A functor $\phi\colon \bC\to \bD$ between two possibly non-unital categories is given by a map  between the sets of objects $ \Ob(\bC)\to \Ob(\bD)$, and for every two objects $C,C'$ in $\bC$ a map of morphism sets
 $\Hom_{\bC}(C,C')\to \Hom_{\bD}(\phi(C),\phi(C'))$ which respects the laws of compositions.  
 The possibly non-unital small categories and  functors  form the large category of possibly non-unital small categories. 
 
  A small category is a possibly non-unital small category which admits units for all its objects. A unital functor between categories is
  a functor which preserves units. We get the large category of small categories and unital functors. It is a subcategory of the large category of possibly non-unital small categories.  The inclusion  is neither full nor wide.

A  possibly non-unital small  $\C$-linear category is a   possibly non-unital small category which is enriched in $\C$-vector spaces.  Thus its morphism sets have the additional structure of $\C$-vector spaces, and the composition laws are required to be bi-linear. Functors between   possibly non-unital small $\C$-linear categories are required to respect the enrichment in $\C$-vector spaces.

A  possibly non-unital small $\C$-linear $*$-category is a     possibly non-unital small $\C$-linear category equipped with an involution $*$ (a contravariant endofunctor of the underlying  possibly non-unital  category) fixing objects, reversing the direction of morphisms,  and acting complex anti-linearly on the morphism spaces. 

\begin{rem}
In comparison with the notion of a complex $*$-category as defined in \cite[Def.~1.1]{ghr} we dropped the third axiom A3 requiring positivity of morphisms of the form $f^{*}f$. \hB
 \end{rem}

A functor between 
   possibly non-unital small $\C$-linear $*$-categories is a functor between  possibly non-unital small $\C$-linear categories which in addition preserves the involutions. 
 
\begin{ddd}
 We let $\nClincat$ denote the large category  of    possibly non-unital small $\C$-linear $*$-categories, and we let
 $\Clincat$ denote the subcategory of     unital  small  $\C$-linear $*$-categories and unital functors.
\end{ddd}

 \begin{ex}
 A non-unital $*$-algebra over $\C$ can be considered as an object of $\nClincat$ which has a single object.  An example is the $*$-algebra of  finite-rank operators on an $\infty$-dimensional Hilbert space.
 
The unital $*$-algebras $\C$ and  $\Mat(2,\C)$ are   objects of $\Clincat$. The upper left corner inclusion $\C\to \Mat(2,\C)$ is a morphism in $\nClincat$, but not in $\Clincat$.
 
The category   of very small Hilbert spaces and finite-rank linear operators $\Hilb_{\text{fin-rk}}(\C)$ is an object of $\nClincat$. The $*$-operation sends an operator to its adjoint. %given by taking adjoints. 
Its full subcategory $\Hilb^{\fg}(\C) $  of finite-dimensional Hilbert spaces is an object of $\Clincat$.
\hB
\end{ex}

 If $H$ is a Hilbert space, then by $B(H)$ we denote the $C^{*}$-algebra of bounded operators.  It has a norm $\|-\|_{B(H)}$. If $H$ is small, 
then  we will consider $B(H)$ as an object of $\nClincat$.  
If $\bC$ is in $\nClincat$ and $f$ is a morphism in $\bC$, then we define its  maximal norm   by 
\begin{equation}\label{ervwervwe}
\|f\|_{\max} \coloneqq {\sup}_{\rho} \|\rho(f)\|_{B(H)}\, ,
\end{equation}
where $\rho$ runs over all functors $\rho \colon \bC\to B(H)$ for all    small complex Hilbert spaces $H$. 
Since there is at least the zero functor we know that  $\|f\|_{\max}$   takes values in $[0,\infty]$.

\begin{rem}
This definition is equivalent to the definition (used e.g.\ in \cite{startcats},   \cite{crosscat}) where  the maximal norm is defined  as a supremum over all representations into  small $C^{*}$-algebras since every small $C^{*}$-algebra admits an isometric embedding into $B(H)$ for some small complex Hilbert space.
  \hB
\end{rem}

In general, the maximal norm can take the value $\infty$. In order to talk about completeness or to construct completions with respect to the maximal norm we need its  finiteness. We therefore
  introduce  the notion of  a pre-$C^{*}$-category. Let $\bC$ be in $\nClincat$.
\begin{ddd}\label{defn_pre_Cstar_cat}
 $\bC$ is called a   pre-$C^{*}$-category if $\|f\|_{\max}<\infty$ for all morphisms $f$ in~$\bC$.
  We denote by $\npClincat$  and $\pClincat$ the full subcategories of $\nClincat$ and $\Clincat$ of pre-$C^{*}$-categories, respectively.
\end{ddd}

If  $\bC$ is in $\npClincat$, then $\|-\|_{\max}$ induces semi-norms on the morphism spaces of $\bC$.
 A semi-normed complex vector space is said to be complete if the semi-norm is a norm and if in addition the vector space is complete with respect to the metric induced by the norm. In the following, completeness always refers to  $\|-\|_{\max}$.
 
Let $\bC$ be in $\npClincat$.
 \begin{ddd}\label{defn_Cstar_cat}
 $\bC$ is called a $C^{*}$-category if the morphism spaces of $\bC$ are complete.
 We denote by $\nCcat$ and $\Ccat$ the full  subcategories of $\npClincat$  and $\pClincat$ of $C^{*}$-categories, respectively.
 \end{ddd}

The advantage of this definition compared to the classical definitions (see Remark \ref{sfbfdvdfvfdswr} below)  is that being a $C^{*}$-category is just a property of a $\C$-linear $*$-category. It does not require any additional data like norms on the morphism spaces.   

\begin{rem}\label{sfbfdvdfvfdswr} 
Classically the notions of a pre-$C^{*}$-algebra and a pre-$C^{*}$-category have a different meaning. A pre-$C^{*}$-algebra in the classical sense is a sub-multiplicatively  normed  *-algebra  $A$ such that the $C^{*}$-identity
$\|a^{*}a\|=\|a\|^{2}$ holds for all elements $a$ of $A$. Then $A$ is a $C^{*}$-algebra if it is in addition complete. Any pre-$C^{*}$-algebra $A$  in the classical sense can be completed to a $C^{*}$-algebra $\bar A$. A selfadjoint element $a$ in a pre-$C^{*}$-algebra is called positive if its image in $\bar A$ is positive, i.e., has a spectrum contained in $[0,\infty)$.

Similarly, 
a sub-multiplicatively normed *-category $\bC$  is a pre-$C^*$-category in the classical sense \cite[Defn.\ 2.4]{mitchc}  if the following conditions hold:
\begin{enumerate}
\item \label{qerihugiqefqwefd} The $C^*$-identity $\|x\|^2 = \|x^* x\|$ is satisfied for all morphisms $x$ in $\Hom_\bC(A,B)$ and for all objects $A,B$ of $\bC$.
\item  \label{qerihugiqefqwefd1} For every morphism $x$ in $\Hom_\bC(A,B)$ the morphism $x^* x$ is a positive element of the pre-$C^*$-algebra $\Hom_\bC(A,A)$.
\end{enumerate}
A $C^{*}$-category in the classical sense  \cite[Def. 1.1]{ghr} is then a pre-$C^{*}$-category in the classical sense whose morphism spaces are complete.

Alternatively to \ref{qerihugiqefqwefd} and  \ref{qerihugiqefqwefd1}, by \cite[Thm.\ 2.7 \& Defn.\ 2.9]{mitchc}   one can require the  Condition~\ref{qerihugiqefqwefd1} together with:
\begin{enumerate}[resume]
\item\label{efijoibsderewrfwefbsfdbsdfb} The $C^*$-inequality $\|x\|^2 \le \|x^* x + y^* y\|$ is satisfied for all morphisms $x,y$ in $\Hom_\bC(A,B)$ and for all objects $A,B$ of $\bC$.
%\item For every morphism $x$ in $\Hom_\bC(A,B)$ the morphism $x^* x$ is a positive element of the pre-$C^*$-algebra $\Hom_\bC(A,A)$. 
\end{enumerate}
In \cite[Ex.\ 2.10]{mitchc} Mitchener provides an example of a sub-multiplicatively normed ${}^*$-category which satisfies both the $C^*$-identity and $C^*$-inequality, but not the positivity condition in Point~\ref{qerihugiqefqwefd1}.

\begin{ddd} A normed *-category $\bC$  satisfies  the {strong $C^*$-inequality} if for all objects $A,B,C$ of $\bC$ and all morphisms $x$ in $\Hom_\bC(A,B)$ and $y$ in $\Hom_\bC(A,C)$  we have 
\begin{equation}
\label{eq_strong_Cstar_inequality}
\|x\|^2 \le \|x^* x + y^* y\|\,.
\end{equation}
\end{ddd} 
Note that the difference to  the $C^{*}$-inequality in Point~\ref{efijoibsderewrfwefbsfdbsdfb} above is that $x$ and $y$ may have different targets.

The strong $C^*$-inequality implies both the $C^*$-inequality in Point~\ref{efijoibsderewrfwefbsfdbsdfb}  and the $C^*$-identity in Point~\ref{qerihugiqefqwefd}, and it implies the positivity  condition  in Point~\ref{qerihugiqefqwefd1} by exploiting the following property of $C^*$-algebras: A self-adjoint element $b$ in a $C^*$-algebra $A$ is positive if and only if for all positive elements $a$ in $A$ we have $\|a\| \le \|a + b\|$. 
On the other hand, this property of $C^*$-algebras also implies that the strong $C^*$-inequality is true for the maximal norm on a pre-$C^*$-category in the sense of Definition \ref{defn_pre_Cstar_cat}.

Since the norm on a $C^{*}$-category in the classical sense is equal to the maximal norm we see 
that the definitions of a $C^{*}$-category in the sense of Definition \ref{defn_Cstar_cat} and in the classical sense are equivalent.\footnote{This was already  stated in \cite[Rem.\ 2.15]{startcats}, but one must delete the word ``parallel'' in the statement of the $C^{*}$-inequality in order to turn it into the strong $C^{*}$-inequality.}
\hB
\end{rem}

 \begin{ex}\label{ex_hilbert_modules}
 A $C^{*}$-algebra is a $C^{*}$-category with a single object. It could be unital or non-unital.  
If,  according to the classical definition, a $C^{*}$-algebra is considered as a closed $*$-subalgebra $A$ of $B(H)$ for some Hilbert space $H$, then the maximal norm on $A$ coincides 
 with the restriction of the usual operator norm  from $B(H)$ to $A$.

 If $A$ is a very small $C^{*}$-algebra, then the category of  very small   Hilbert $A$-modules $\Hilb(A)$ and bounded adjointable operators is an object of $\Ccat$.  It contains the wide subcategory $\Hilb_{c}(A)$ whose morphisms are compact operators (in the sense of Hilbert $A$-modules). For $C,D$ in $\Hilb_{c}(A)$ the space of morphisms $\Hom_{\Hilb_{c}(A)}(C,D)$ is the closure  in $\Hom_{\Hilb(A)}(C,D)$ of the linear subspace generated by the operators
 $\theta_{d,c}$ for all $c$ in $C$ and  $d$ in $D$ which are given by \begin{equation}\label{ssrgfdsgsreg}c'\mapsto \theta_{d,c}(c'):=d\langle c,c'\rangle_{C} \end{equation}  for all $c'$ in $C$.
  The inclusion $\Hilb_{c} (A)\to \Hilb (A)$ is a functor in $\nCcat$.
 
  We can consider $ A$ as an object  of $\Hilb(  A)$ with the scalar product $\langle a,a'\rangle_{A}:=a^{*}a'$. The left multiplication of $A$ on itself  identifies 
  $A$ with $\End_{\Hilb_{c}(A)}(A)$.
 %\bli{For every very small set $X$ we can consider the free Hilbert $A$-module   $L^{2}(X,A)$ on $X$.   
%An arbitrary Hilbert $A$-module is called projective if it is a direct summand in  a free Hilbert $A$-module.  We let $\pHilb(A)$ denote the full subcategory of $\Hilb(A)$ of projective Hilbert $A$-modules.
\hB
\end{ex}

Let $G$ be a {very small} group, and let $BG$ be the category with a single object $*_{BG}$ and the  monoid  of endomorphisms $\End_{BG}(*_{BG}) \coloneqq G$. If $\cC$ is any category, then $\Fun(BG,\cC)$ is the category of $G$-objects and equivariant morphisms in $\cC$. We have a forgetful functor $\Fun(BG,\cC)\to \cC$ which forgets the $G$-action.
{
If $C$ is in $ \Fun(BG,\cC)$ we will often use the symbol $C$ also for the underlying object in $\cC$ obtained by forgetting the $G$-action. But some times, in order to avoid confusion, 
  we will  use the longer notation $\Res^{G}(C)$.} 
\begin{ex}\label{qregiuheqrigegwegergwegergw}%\fuli{was wir hier als $\Hilb^{G}(A)$ bezeichnet hatten, ist eigentlich $\Hilb(A)^{G}$ (die $2$-kategoriellen $G$-fixpunkte in $\Hilb(A)$. Eigentlich meinen wir aber $\Hilb(A)$ wie jetzt neu im Beispiel. Das m\"ussen wir in allen Arbeiten korrigieren.} \color{red}
In this example we construct for every  $A$  in $\Fun(BG,\nCalg)$   a large $C^{*}$-category $ \Hilb( A)$  with strict $G$-action.  If $A$ is very small, then  we will require that the objects of $\Hilb(A)$ are very small as well so that $ \Hilb( A)$  belongs to  $\Fun(BG,\Ccat)$.

The underlying $C^{*}$-category  of $\Hilb(A)$ 
is   the $C^{*}$-category  of small (respectively very small) Hilbert $A$-modules from  Example \ref{ex_hilbert_modules}.  It remains to describe the strict action of $G$.   In the following we describe how $h$ in $G$ acts as an endomorphism  of $\Hilb(A)$. In the  formulas below
 the action   of a group element  $g$  on $A$ will   be  written as $a\mapsto  {}^{g}a$.

%We first define \uli{the underlying  $C^{*}$-category   of $( A)$.}
%We start with  the $C^{*}$-category  $\Hilb(A)$ of very small $A$-Hilbert $C^{*}$-modules.
%  
% 
% \begin{enumerate}
%\item  The objects of $\Hilb^{G}(  A)$ are  pairs $(M,\lambda)$ of  an object $M$ of $\Hilb(A)$  with a compatible action $\lambda $ of $G$ on $M$ by linear maps (not necessarily $A$-module homomorphisms) with:
%\begin{enumerate}
%\item For all $a$ in $A$ and $m$ in $M$ we have $\lambda(g)(m\cdot a)={\lambda(g)(m)}\cdot {}^{g}a$.
%\item \label{wtoigjwgrgerfgwef} For all $m,m'$ in $M$ we have $\langle \lambda(g) (m), \lambda(g)(m')\rangle_{M}= {}^{g}\langle m,  m \rangle_{M}$.
%\end{enumerate}
%\item 
%A morphism $f\colon (M,\lambda)\to (M',\lambda')$ in $\Hilb^{G}( A)$ is a morphism $f\colon M\to M'$ in $\Hilb(A)$ such that $f(\lambda(g)(m))=\lambda'(g)(f(m))$ for all $g$ in $G$ and $m$ in $M$.
% \end{enumerate}
%Note that $\lambda(g)$ is invertible and unitary by Condition \ref{wtoigjwgrgerfgwef}.
%This finishes the description of the underlying $C^{*}$-category of $\Hilb^{G}(A)$.
%
%
% 
%
%
%\uli{In order to finish the description of $\Hilb^{G}(A)$ we 
% define  a strict $G$-action on the $C^{*}$-category described above.}  
%An element $h$ in $G$ acts as an endofunctor of \uli{this category} as follows:
\begin{enumerate}
\item objects:
 The morphism $h$ sends  a Hilbert $A$-module
 $M$ with structures  $(\cdot, \langle-,- \rangle_{M})$  (the right-$A$-module structure and the  $A$-valued scalar product),   to  the Hilbert $A$-module
$hM$ with structures $(\cdot_{h},  \langle-,- \rangle_{hM}) $, 
where  $hM$ is the $\C$-vector space $M$ with the right multiplication by $A$ given by $$ {m\cdot_{h}a}:=   m\cdot ({}^{h^{-1}} a)$$ and the $A$-valued scalar product $$ \langle m,m' \rangle_{hM} \coloneqq {}^{h}\langle m,m'\rangle_{M}\ .$$ 
  \item morphisms:
If $f\colon M \to M' $ is a morphism in  $\Hilb(A)$, then its image under $h$   is the same linear map now considered as a morphism  $hf \colon hM \to hM' $.
\end{enumerate} 
One easily checks that the describes a strict $G$-action on $\Hilb(A)$. 
This action preserves the ideal $\Hilb_{c}(A)$ of compact operators so that in the case of a very small $A$
we get an object $\Hilb_{c}(A)$ in $\Fun(BG,\nCcat)$.

 %{This finishes the description of $ \Hilb(  A) $ as a $C^{*}$-category with a strict $G$-action. If $A$ is very small, then $ \Hilb(  A) $ is an object of  $\Fun(BG,\Ccat)$.}
  
  If $G$ acts trivially on $A$, then it also acts trivially on $\Hilb(A)$.
   \hB 
\end{ex}

\begin{ex}
The functors from  $\nClincat$, $\Clincat$, $\nCcat$ and $\Ccat$ to small sets which take the sets of objects,
have right-adjoints, see \cite[Lem.~2.4 and~3.8]{crosscat}. In all cases the right-adjoint $0[-]$ sends a set $X$ to the category $0[X]$ with the set of objects $X$, and  whose morphism vector spaces are all trivial. The value of the counit of the adjunction at $\bC$  is a functor 
\begin{equation}\label{sgbsfdvfsvfvdfvsv}
\bC\to 0[\Ob(\bC)] .
\end{equation} 
\hB
\end{ex}

Let $\bC$ be in $\Ccat$ and $f$ be a morphism in $\bC$.  It is an immediate consequence of the definition of the maximal norm that
\begin{equation}\label{eq_functor_contractive}
\|\sigma(f)\|_{\max}\le \|f\|_{\max}
\end{equation}
for every morphism $\sigma \colon \bC\to  \bC'$ in $\nCcat$.

In  the following    we show that the maximal norm of a morphism in a unital $C^{*}$-category can be generated by unital representations in the object $\Hilb(\IC)$ of $\Ccat$ of very small Hilbert spaces and bounded linear operators. %{Let $\Hilb(\IC)$ in $\Ccat$ be the $C^{*}$-category of very small Hilbert spaces and bounded linear operators.} 
We use the notation $\|-\|$ in order to denote the operator norm of bounded operators between Hilbert spaces. Of course it coincides with the maximal norm on $\Hilb(\IC)$ considered as a pre-$C^{*}$-category.
\begin{lem}\label{ewrgoipwergwregrwefw}
We have $\|f\|_{\max}=\sup_{\sigma\in \Hom_{\Ccat}(\bC,\Hilb(\IC))} \|\sigma(f)\|$.
\end{lem}

\begin{proof}
 {From \eqref{eq_functor_contractive} we get} that $ \sup_{\sigma} \|\sigma(f)\|\le \|f\|_{\max}$,
 where $\sigma$ runs over the set of unital representations  $\ \Hom_{\Ccat}(\bC,\Hilb(\IC))$.
It therefore suffices to show the reverse inequality.  Let $\rho\colon \bC\to B(H_{\rho})$ be a functor in $\nCcat$.  Since $\bC$ is unital, applying the left-adjoint of  the adjunction $$\widehat{(-)}:\nCcat\leftrightarrows \Ccat:\incl$$ from 
\cite[(3.10)]{crosscat} and using that $\bC$ is already unital we get a unital functor
$\hat \rho\colon \bC\to \Hilb(\IC)$. It sends an object $C$ in $\bC$ to the image $\hat \rho(C):= \rho(\id_{C})({H_\rho})$ of the projection $  \rho(\id_{C})$,
and a morphism $f\colon C\to C'$ to the morphism $\rho(\id_{C'}) f_{|\hat{\rho}(C)} \colon \hat{\rho}(C) \to \hat{\rho}(C')$. Note that in contrast to $\rho$ the functor $\hat \rho$ is not constant on objects anymore.
By an inspection we see that $\|\rho(f)\|=\|\hat \rho(f)\|$. Consequently,
\[\|f\|_{\max} = {{\sup}_{\rho}\|\rho(f)\|} = {\sup}_{\rho}\|\hat{\rho}(f)\| \le  {\sup}_{\sigma} \|\sigma(f)\|\,.\qedhere\]
\end{proof}

\begin{ddd}
A morphism $\bC\to \bD$  in $\nClincat$  is called faithful if it induces injective maps of morphism spaces. \end{ddd} Note that a faithful morphism between $C^{*}$-categories is automatically  isometric.

%\section{Properties of morphisms}\label{oijoifjewofewfqfeqef}
  
%We start with projections.
 
We end this introduction to $\C$-linear $*$-categories and $C^{*}$-categories be recalling some elements of their internal language. Let $\bC$ in $\nClincat$ and let $C$ be an object of $\bC$.
\begin{ddd}\label{ebgiojoigrevsfdvsdfvsfdvsdfv}
%\mbox{}
%\begin{enumerate}
%\item 
The object $C$ is called unital if there exists an identity morphism in $\End_{\bC}(C)$.
%\item 
By $\bC^{u}$
%in $\Ccat$
%we denote the full subcategory of  unital objects in $\bC$.
%\end{enumerate}
%\end{ddd}
%
%Let $\bC$ in $\nClincat$.
%\begin{ddd} 
%By $\bC^{u}$
%in $\Ccat$
we denote the full subcategory of  unital objects in $\bC$.\end{ddd}
 Note that $\bC^{u}$ is an object of  $\Clincat$. 
  If $\bC$ is in $\nCcat$, then $\bC^{u}$ is in $\Ccat$.

\begin{rem} Unital objects are preserved by automorphisms. Therefore, if 
 $G$ is a group and  $  \bC$ is in $\Fun(BG,\nClincat)$ or in $\Fun(BG,\nCcat)$, then we naturally get an object $  \bC^{u}$  in $\Fun(BG,\Clincat)$ or $\Fun(BG,\Ccat)$, respectively. \hB
\end{rem}

Let $\bC$ in $\nClincat$.
\begin{ddd}\label{r9erqgrgqrg0}
\mbox{}
\begin{enumerate}
\item  \label{ergiegregwergergw9}
A  projection   is an  endomorphism $p$    such that $p^{*}=p$ and $p^{2}=p$. 
\item A partial isometry is a morphism $u $  such that $uu^{*}$ and $u^{*}u$ are projections.
\item
An   isometry  is a partial  isometry $u\colon C\to C'$ such that $u^{*}u=\id_{C}$.
\item A unitary  is an  isometry $u\colon C\to C'$ such that  $uu^{*}=\id_{C'}$.
 \end{enumerate}
\end{ddd}
 
 \begin{rem}
 Note that the condition $p^{*}=p$  in Definition~\ref{r9erqgrgqrg0}(\ref{ergiegregwergergw9}) describes orthogonal projections.
 In the present paper we will only consider orthogonal projections and therefore drop the  adjective ``orthogonal''.
 
 If $u \colon C\to C^{\prime}$ is an isometry, then implicitly the object $C$  is unital.
 Similarly, unitaries can only exist between unital objects.   \hB
 \end{rem}

Let $\bC$ in $\nClincat$ and $C$ be an object of $\bC$.
Let $p$ be a projection on $C$.
\begin{ddd}\label{gewergregwegegwerg}
An image of $p$ is a pair $(D,u)$ of an object $D$ in $\bC$ and an  isometry $u\colon D\to C$ such that $p=uu^{*}$.
\end{ddd}

The  image of a projection  is uniquely determined up to  unique unitary isomorphism. In fact, 
 let $(D,u)$ and $(D',u')$ be both images of $p$.
Then $v \coloneqq u^{\prime,*}u \colon D\to D'$ is the unique  unitary such that $u'v=u$. %Indeed, $u'v=u'u^{\prime,*}u=pu=uu^{*}u=u$,
%$v^{*}v=(u^{\prime,*}u)^{*}u^{\prime,*}u=u^{*}u'u^{\prime,*}u=u^{*}pu=u^{*}uu^{*}u=\id_{D}$ and similarly $vv^{*}=\id_{D'}$.
%\uli{For the uniqueness note that if} $w:D\to D'$ is any  morphism such that $ u'w=u$, then $w=u^{\prime,*}u'w=u^{\prime,*}u=v$.

\begin{ddd}\label{regiuhjigwergwgrewgrg}\mbox{}
\begin{enumerate}
\item A projection is called effective if it admits an image.
\item $\bC$ is called idempotent complete if  every projection in $\bC$  is effective.
\end{enumerate}
\end{ddd}

\begin{ex}
If $A$ is a very small $C^{*}$-algebra, then $\Hilb(A)$   is  idempotent complete.  
The full subcategory $\Hilb^{\dim= \infty}(\C)$ of $\infty$-dimensional Hilbert spaces in $\Hilb(\C)$ is an example which is not idempotent complete.
\hB
\end{ex}

Let $\bC$ be in $\nCcat$ and $C$ be an object of $\bC$. Then $\End_{\bC}(C)$ is a $C^{*}$-algebra. Recall that a net $(h_{i})_{i}$ in $ \End_{\bC}(C)$ is an approximate unit if for every element $f$ in $  \End_{\bC}(C)$ we have
$\lim_{i}h_{i}f=f=\lim_{i}fh_{i}$ in norm. This has the following generalization.
 \begin{lem}\label{multmult}  We have $\lim_{i} h_{i}l=l$ 
 for every morphism $l$ in $\bC$ with target $C$ and $\lim_{i} k h_{i}=k$
 for  every morphism $k$ with domain $C$. 
\end{lem}
\begin{proof}
We give the argument for the first case. Note that 
$$\|h_{i}l-l\|^{2}=\|(h_{i}l-l)(h_{i}l- l)^{*}\|=\|( ll^{*} - {h_{i}} ll^{*}     ) +(ll^{*} - ll^{*}h_{i}  ) - (ll^{*} -h_{i}ll^{*}h_{i})\|\, .
$$ 
We can rewrite the last term in the form  $ h_{i}ll^{*}h_{i}= h_{i}\sqrt{ll^{*}}   \sqrt{ll^{*}}h_{i} $.
Since
$\lim_{i}h_{i}ll^{*}=ll^{*} =\lim_{i} ll^{*}h_{i}$ and
$\lim_{i} h_{i}\sqrt{ll^{*}}=  \sqrt{ll^{*}}=\lim_{i}   \sqrt{ll^{*}} h_{i}$ we conclude that  $\lim_{i}\|h_{i}l-l\|^{2}=0$.
\end{proof}

From now one,   we will usually call functors between $\C$-linear $*$-categories or $C^{*}$-categories morphisms as they are morphisms in categories $\nCcat$, $\nClincat$, etc.  We 
  use the word functor on the next categorical level, e.g.\  for functors with domain or target $\nCcat$, $\nClincat$, etc.

 In the remainder of the present section we   recall the definition   of  a $W^{*}$-category, the category $\Wcat$ of $W^{*}$-categories, and the construction of the $W^{*}$-envelope of a $C^{*}$-category.
 
  We say that  a Banach space $E$ admits a predual if there exists a Banach space $E_{*}$ (a pre-dual) such that  $E$ is the dual Banach space of $E_{*}$, i.e., the Banach space of bounded  linear functionals on $E_{*}$.
  Given a pre-dual  $E_{*}$ the $\sigma$-weak topology on $E$ is the topology of point-wise convergence on $E_{*}$.

Let $\bC$ be a unital $C^{*}$-category. 
\begin{ddd}[{\cite[Def. 2.1]{ghr}}]
$\bC$ is a $W^{*}$-category if for every two objects $C,C'$ of $\bC$ the Banach space $\Hom_{\bC}(C,C')$ admits a pre-dual.
\end{ddd}
It is known that the preduals of the morphism spaces of a $W^{*}$-category are
 unique (as subspaces of the duals of the morphism spaces). In particular, the $\sigma$-weak topology is well-defined.

\begin{ex}  A $W^{*}$-algebra (i.e., a von Neumann algebra) is a
 $W^{*}$-category with a single object. If $\bC$ is a $W^{*}$-category, then
 for every $C$ in $\bC$ the $C^{*}$-algebra $\End_{\bC}(C)$ is a $W^{*}$-algebra.    \hB \end{ex}
 
  \begin{ex}\label{ex_HilbC_Wstar}
  The $C^{*}$-category $\Hilb(\C) $ is a $W^{*}$-category, see \cite[Ex.~2.2]{ghr}. 
 As a  predual of $\Hom_{\Hilb (\C)}(C,C')$ one can take  the space $L^{1}(C',C)$ of trace class operators from $C'$ to $C$. Thereby 
 $A$   in $\Hom_{\Hilb(\C) }(C,C')$ is viewed as  the bounded linear functional
 $T\mapsto \tr (AT)$ on $L^{1}(C',C)$.   \hB\end{ex}
  
 \begin{ex}
 If $\bC$ is in $\nCcat$,  then we can form the $C^{*}$-category $\Rep(\bC)$ in $\Ccat$  of representations of $\bC$  on Hilbert spaces as follows. 
 \begin{enumerate}
 \item objects: The objects of $\Rep(\bC)$ are the morphisms  $\bC\to \Hilb(\C)$ in $\nCcat$. 
 \item morphisms: The morphisms of $\Rep(\bC)$ are uniformly  bounded natural transformations between representations. 
 Note that a natural transformation $v \colon \sigma\to \rho$ between representations is given by a family $(v_{C})_{C\in \Ob(\bC)}$ of bounded operators $v_{C}\colon \sigma(C)\to \rho(C)$ between Hilbert spaces such that $v_{C'}\sigma(f)=\rho(f)v_{C}$ for all morphisms $f\colon C\to C'$ in $\bC$. The natural transformation $v$ is called uniformly bounded if
 $\|v\| \coloneqq \sup_{C\in \Ob(\bC)} \|v_{C}\|$ is finite. 
 \item composition: The composition in $\Rep(\bC)$ is the composition of natural transformations. 
 \item $\C$-enrichment and involution: These structures are induced from the morphism spaces of $\Hilb(\bC)$.
 % If $v,v'$ are parallel morphisms in $\Rep(\bC)$ given by $v=(v_{C})_{C\in \Ob(\bC)}$ and $v'=(v'_{C})_{C\in \Ob(\bC)}$, and $\lambda$ is in $\C$, then $v+\lambda v':=(v_{C}+\lambda v_{C}')_{C\in \Ob(\bC)}$. Furthermore,  
 % $v^{*}:=(v^{*}_{C})_{C\in \Ob(\bC)}$.
 \end{enumerate}
The norm $v\mapsto \|v\|$ exhibits $\Rep(\bC)$ as a $C^{*}$-category.
 By \cite[Ex.~2.5]{ghr} the $C^{*}$-category $\Rep(\bC)$ is actually a $W^{*}$-category.
 \hB
   \end{ex}

 We let $\Hilb(\C)^{\la}$ denote the large $C^{*}$-category of possibly  small Hilbert spaces and bounded operators. By replacing $\Hilb(\C)$ with $\Hilb(\C)^{\la}$ we can also consider the $W^{*}$-category $\Rep^{\la}(\bC)$ of representations of $\bC$ on possibly small  Hilbert spaces.
Given a small family $(\sigma_{i})_{i\in I}$ in $\Rep^{\la}(\bC)$ we can form $\bigoplus_{i\in I}\sigma_{i}$ in $\Rep(\bC)^{\la}$  in the straightforward manner. For example we could take  the orthogonal sum of all objects of $\Rep(\bC)$.

\begin{ex}
Let $\bC$ be in $\nCcat$. 
%  A representation $\sigma$ in $\Rep(\bC)^{\la}$ is called faithful if $f$ is an isometric inclusion. 
  Then    $\Rep(\bC)^{\la}$ contains a faithful representation. For example we could 
 %In order to see this note that  
 %for every morphism $f$ in $\bC$ we have
 %$\|f\| = \sup_{\sigma\in \Rep(\bC)} \|\sigma(f)\|$.
 %If we  
  take $\hat \sigma:=\bigoplus_{\sigma\in \Rep(\bC)} \sigma$.
     %in $\Rep(\bC)^{\la}$, then $ \sup_{\sigma\in \Rep(\bC)} \|\sigma(f)\|=\|\hat \sigma(f)\|$.  These two equalities together imply that $\|\hat \sigma(f) \|=\|f\|$ for all morphisms $f$ in $\bC$. Hence $\hat \sigma$ is faithful.
 The universal representation constructed in the proof of \cite[Prop. 1.14]{ghr} gives another faithful representation. \hB
\end{ex}
\begin{ddd}\label{ewkogwtrgwegwergferf}
We call $\sigma$ in $\Rep(\bC)^{\la}$ non-degenerate if
for every object $C$ in $\bC$ the set $\sigma(\End_{\bC}(C))\sigma(C)$ generates a  dense subspace of the Hilbert space $\sigma(C)$. 
\end{ddd}

If $\sigma$ in $\Rep(\bC)^{\la}$ is any representation, then we can find a non-degenerate sub-representation $\tilde \sigma$ of $\sigma$  by setting  \begin{equation}\label{dsfgdgsgfdsgdgdgrqe}
\tilde \sigma(C)=\overline{\sigma(\End_{\bC}(C))\sigma(C)}
\end{equation}  for every object $C$ in $\bC$.  If $\sigma$ was faithful, then so is $\tilde \sigma$.

Let $\bC$ be in $\nCcat$   and $\sigma$ be  an  object of  $\Rep(\bC)^{\la}$.

\begin{ddd}\label{wrkbhorbgfdbdfgbdfgb} We  define   the bicommutant $C^{*}$-category $\bC_{\sigma}''$ together with a morphism $w:\bC\to \bC_{\sigma}''$ as follows:
\begin{enumerate}
\item objects: The objects of $\bC_{\sigma}''$ are the objects of $\bC$ and $w:\bC\to \bC_{\sigma}''$ is the identity on objects.
\item morphisms:
  For objects $C,C'$ in $\bC$ we let $\Hom_{\bC_{\sigma}^{''}}(C,C')$ be the set of all $A$ in $\Hom_{\Hilb(\C)^{\la}}( \tilde \sigma(C), \tilde \sigma(C'))$ (see \eqref{dsfgdgsgfdsgdgdgrqe} for $\tilde \sigma$) such that $Av_{C}=v_{C'}
A$ for all $v$ in $\End_{\Rep(\bC)}(\sigma)$. On morphisms $w$ is given by $\sigma$.
\item composition and involution: The composition and the involution  of $\bC_{\sigma}'' $ are inherited from $\Hilb(\bC)^{\la}$.
\end{enumerate}
\end{ddd}
The norm induced from $\Hilb(\C)^{\la}$ exhibits $\bC_{\sigma}''$ as a $C^{*}$-category. By the bicommutant theorem   \cite[Thm. 4.2]{ghr} it is actually a $W^{*}$-category and $w:\bC\to \bC''_{\sigma}$ has a $\sigma$-weakly dense range.  The last observation implies that $  \bC''_{\sigma}$ is small.
If $\sigma$ is faithful, then  the morphism $w:\bC\to \bC_{\sigma}''$ is also faithul. Finally, if $\sigma$ is faithful and  $\bC$ is a $W^{*}$-category, then $w:\bC\to \bC_{\sigma}''$
 is an isomorphism.

If $A$ and $B$ are $W^{*}$-algebras and $\phi:A\to B$ is a morphism in $\Calg$, then $\phi$ is called normal if for every  increasing  bounded net $(a_{\nu})_{\nu}$ of positive elements in $A$ we have
$\sup_{\nu} \phi(a_{\nu})=\phi(\sup_{\nu} a_{\nu})$.
Note that a morphism between $W^{*}$-algebras is normal if and only if it is $\sigma$-weakly continuous. The notion of a normal morphism between $W^{*}$-algebras  extends to $W^{*}$-categories   in the obvious way.
Let $\bC,\bD$ be $W^{*}$-categories
and $\phi:\bC\to \bD$ a morphism in $\Ccat$.
Note that the endomorphism algebras of all objects of $\bC$ and $\bD$ are  $W^{*}$-algebras.

\begin{ddd}[{\cite[Def. 2.11]{ghr}}]
We say that $\phi$ is normal if $\phi:\End_{\bC}(C)\to \End_{\bD}(\phi(C))$ is normal for every object $C$ of $\bC$.
\end{ddd}
Again, $\phi$ is normal if and only if $\phi$ is $\sigma$-weakly continuous on morphism spaces  \cite[Prop. 2.12]{ghr}.

Let $\bC$ be a $W^{*}$-category. 
\begin{ddd} \label{weiojgowegerfwrf}The weak operator topology on the morphism spaces $\Hom_{\bC}(C,C')$  of $\bC$
is generated by the functionals
$\langle x',\sigma(-)x\rangle$ for all normal representations $\sigma:\bC\to \Hilb(\C)^{\la}$, $x$ in $\sigma(C)$ and $x'$ in $\sigma(C')$.
\end{ddd}

Since these functionals are $\sigma$-weakly continuous it is clear that
the weak operator topology is  smaller than the $\sigma$-weak topology.

\begin{ddd}
We let $\Wcat$ denote the sub-category  of $\Ccat$ of $W^{*}$-categories and normal morphisms.
\end{ddd}

The analog of the following theorem for algebras is well-known \cite{gui}.
A proof can be found  e.g. in  \cite{lurie-n8}.

\begin{theorem}\label{rijogegergwerg9}
We have an adjunction
$$\bW:\Ccat\leftrightarrows \Wcat:\incl\ .$$
\end{theorem}
\begin{proof}
The argument is a straightforward generalization of the argument given in  \cite{lurie-n8}. 
We will construct for every $\bC$ in $\Ccat$ a   morphism
$i_{\bC}:\bC\to \bW\bC$ in $\Ccat$ such that $\bW\bC$ belongs to 
$\Wcat$ and the following universal property is satisfied:
\begin{enumerate}
\item\label{iiuaoidfafasdfasdfasdf} The image of $\bC$ is $\sigma$-weakly dense in $\bW\bC$.
\item \label{wemiotgwegergerfw} Every representation $\sigma:\bC\to \Hilb(\C)^{\la}$ in $\Rep(\bC)^{\la}$ extends to a normal morphism $\hat \sigma:\bW\bC\to  \Hilb(\C)^{\la}$.
\end{enumerate}
We claim that this implies the theorem. 
First of all we must extend the construction to a functor $\bW:\Ccat\to \Wcat$
by defining $\bW$ on morphisms.
Let $\phi:\bC\to \bD$ be a morphism.
By \cite[Prop. 2.13]{ghr} there exists a normal faithful functor $\rho: \bW\bD\to \Hilb(\C)^{\la}$ {whose image is $\sigma$-weakly closed}. Then $\sigma:=\rho\circ \circ i_{\bD}\circ \phi:\bC\to \Hilb(\C)^{\la}$ is in $ \Rep(\bC)^{\la}$. We let $\hat \sigma:\bW\bC\to \Hilb(\C)^{\la}$ be its normal extension whose existence is ensured by \ref{wemiotgwegergerfw}. Since $\sigma(\bC)\subseteq 
\rho(i_{\bD}(\bD))\subseteq \rho(\bW\bD)$, $\hat \sigma$ is   $\sigma$-weakly continuous, and $\rho(\bW\bD)$ is $\sigma$-weakly closed,
we conclude  using \ref{iiuaoidfafasdfasdfasdf} that $\hat \sigma(\bW\bC)\subseteq \rho(\bW\bD)$.
We therefore get a $\sigma$-weakly continuous  and hence normal morphism
$\bW\phi:\bW\bC\to \bW\bD$. Again using \ref{iiuaoidfafasdfasdfasdf} we see that this extension is actually unique. 
Using this uniqueness we further conclude that $\bW$ is  functor, i.e., compatible with the  compositions in $\Ccat$ and $\Wcat$. 

We now assume that $\bD$ is in $\Wcat$. Then we can assume  by
\cite[Prop. 2.13]{ghr} that $\bD$ itself is a $\sigma$-weakly closed subcategory of $\Hilb(\C)^{\la}$.  A similar  argument as above shows that 
 the restriction map along $\bC\to \bW\bC$ induces a surjection
$$\Hom_{\Wcat}(\bW\bC,\bD)\to \Hom_{\Ccat}(\bC,\bD)\ .$$
Using again that morphisms in $ \Wcat$ are $\sigma$-weakly continuous
and \ref{iiuaoidfafasdfasdfasdf} we see that this restriction map also injective. 
 This finishes the proof of the claim.
 
We now show the existence of the morphisms $i_{\bC}:\bC\to \bW\bC$ with the required universal property.  The construction of the GNS representation of a $C^{*}$-algebra from a positive state
generalizes to $C^{*}$-categories \cite[Prop. 1.9]{ghr}.
For every positive linear functional  $\nu$ on $\End_{\bC}(C)$ for some object $C$ of $\bC$ we have the non-degenerate GNS representation $\sigma_{\nu}:\bC\to \Hilb(\C)^{\la}$ constructed as follows. For an object $C'$ of $\bC$ the Hilbert space $\sigma_{\nu}(C')$ is the closure
of $\Hom_{\bC}(C,C')$ with respect to  the scalar product $\langle f,g\rangle:=\nu(f\circ g)$.
For $h:C'\to C''$ the operator $\sigma_{\nu}(h):\sigma_{\nu}(C')\to \sigma_{\nu}(C'')$ is given by the left composition with $h$.

We define the universal representation  $\sigma_{u}:\bC\to \Hilb(\C)^{\la}$  of $\bC$ as the orthogonal sum
of all $\sigma_{\nu}$ for all objects $C$ of $\bC$ and positive  linear functionals  $\nu$ on $\End_{\bC}(C)$. The representation $\sigma_{u}$  is non-degenerate by construction and faithful by \cite[Prop. 1.14]{ghr}.  We now define
$\bW\bC:=\bC''_{\sigma_{u}}$, see Definition \ref{wrkbhorbgfdbdfgbdfgb}. Then the canonical morphism  $i_{\bC}:\bC\to \bW\bC$ is faithful and has a $\sigma$-weakly dense range as required by \ref{iiuaoidfafasdfasdfasdf}. %see Definition \ref{wrkbhorbgfdbdfgbdfgb}. 

Assume that $\sigma:\bC\to \Hilb(\C)^{\la}$ is any non-zero representation.
Then we can find an object of $\bC$ such that $\sigma(C)\not=0$ and  $x$ in  $\sigma(C)$ such that $\nu:A\mapsto \langle \sigma(A)x,x\rangle$ is a non-zero positive functional  on $\End_{\bC}(C)$.  We claim that there exists a summand of $\sigma$ which is isomorphic to $\sigma_{\nu}$.  To this end we consider  for every object $C'$ of $\bC$  the subspace
$\bar \sigma(C')$ of $\sigma(C)$ generated by $\sigma(\Hom_{\bC}(C,C'))(x)$. These subspaces form a subrepresentation of $\sigma$ which has the cyclic vector $x$. By \cite[Prop. 1.9]{ghr} it is unitarily isomorphic to $\sigma_{\nu}$.

We now verify Condition \ref{wemiotgwegergerfw}.
Let $\sigma:\bC\to \Hilb(\C)^{\la}$ be a representation.
First assume that $\sigma$ is isomorphic to a GNS-representation $\sigma_{\nu}$.  
Then $\sigma$ is a direct summand of  the universal representation $\sigma_{u}$ and the  projection
$p$ onto this summand    induces a homomorphism
$\hat \sigma_{\nu}:p\cdots p:\bW\bC\to \Hilb(\C)^{\la}$ which extends $\sigma_{\nu}$.

If $(\sigma_{i})_{i\in I}$ is a family of representations such that
$\sigma_{i} $ admits an extension $\hat \sigma_{i}$, %:\bW\bC\to \Hilb(\C)^{\la}$ for every $i$ in $I$,
 then
$\oplus_{i\in I} \sigma_{i}$ admits the extension $\widehat{\oplus_{i\in I} \sigma_{i}}=\oplus_{i\in I}  \hat \sigma_{i} $.

It thus  remains to show that every representation $\sigma :\bC\to \Hilb(\C)^{\la}$ decomposes as an orthogonal sum of GNS-representations.
To this end
we consider the poset (by inclusion) $\cD$  of subrepresentations of $\sigma$ which decompose as an orthogonal sum of GNS-representations.
We must show that $\sigma\in \cD$.

Every increasing chain in $\cD$ has an upper bound given by the representation generated by the members  of the chain.
By Zorn's lemma there exists a maximal element $\sigma' $ in $\cD$.
If $\sigma'\not=\sigma$, then  in  its  orthogonal complement we can find again a summand which is a GNS-representation. 
But this contradicts the maximality of $\sigma$. 
Hence $\sigma$ itself belongs to $\cD$. 
%  
%
%
%We construct the functor $\bW$. Given $\bC$ in $\Ccat$
%we consider the representation $ \sigma_{\bC}:=\bigoplus_{\sigma\in \Rep(\bC)}\sigma$ in $\Rep^{\la}(\bC)$. Then we set
%$\bW\bC:=\bC^{''}_{\sigma}$.
%
%Let $\phi:\bC\to \bD$ be a morphism in $\Ccat$. Then we define $\bW\phi$ as follows.
%\begin{enumerate}
%\item objects: On objects $\bW\phi$ acts as $\phi$.
%\item morphisms: Let $C,C'$ be objects in $\bC$ and $f\in \Hom_{\bW\bC}(C,C')$ be a morphism. 
%Then $f$ is a bounded operator from $\oplus_{\sigma\in \Rep(\bC)} \sigma(C)$ to $\oplus_{\sigma\in \Rep(\bC)} \sigma(C')$. 
%We have a natural  isometry
%$\bigoplus_{\rho\in \Rep(\bD)} \rho(C)\to  \oplus_{\sigma\in \Rep(\bC)} \sigma(C) $ which identifies the subspace $\rho(C)$ with $(\phi^{*}\rho)(C)$.
% 
%
%
%
%\end{enumerate}
%
\end{proof}

%Let $\bC$ be in $\Ccat$.
% Using  Lemma \ref{wekogwergerfgerwfw} we   choose a faithful representation $\sigma $ in $\Rep(\bC)^{\la}$.
% Then we get a
% $W^{*}$-category $\bW\bC:=\bC_{\sigma}^{''}$ together with an isometric inclusion $\bC\to \bW\bC$. 
%Since the faithful representation $\sigma$ sends $\bC$ to a closed subcategory of $\Hilb(\bC)^{\la}$, by the bicommutant theorem \cite[Thm. 4.2]{ghr}  we can identify
%  the morphism spaces of $\bW\bC$ with the bidual  spaces of the morphism spaces of $\bC$. This shows that $\bW\bC$ does not depend on the choice of the faithful representation up to isomorphism.
%  %The $\sigma$-topology on $\bW\bC$ is the topology of point wise convergence on  the$\sigma(\bC)^{*}$.
%   
% 
%\uli{Diese Definition ist noch nicht richtig!}

Let $\bC$ be in $\Ccat$.
\begin{ddd}\label{wijtgoewgewgefw}
We call $\bW\bC$ the $W^{*}$-envelope of $\bC$.
\end{ddd}

Let $\phi:\bC\to \bD$ be a morphism in $\Ccat$.
\begin{prop}\label{rjigowergwregregwfer}
If $\phi$  is fully faithful, then so is $\bW\phi$.
\end{prop}
\begin{proof}
We first show that $\bW\phi$ is isometric.
If $C$ is an object of $\bC$ and $\nu$ 
is a weight on 
$\End_{\bC}(C)$,
then using the isomorphism of $C^{*}$-algebras
$\phi: \End_{\bC}(C)\stackrel{\cong}{\to}
\End_{\bD}(\phi(C))$ we get a weight $\phi_{*}\nu$ on $ 
\End_{\bD}(\phi(C))$.
 For   every object $C'$  in $\bC$ we have an isometry   of  GNS spaces
$\sigma_{\nu}(C')\cong \sigma_{\phi_{*}\nu}(\phi(C'))$ induced by
the  isomorphism  $\phi:\Hom_{\bC}(C,C')\to \Hom_{\bD}(\phi(C),\phi(C'))$.
Using the construction of the $W^{*}$-envelope
 $\bW\bC$ given in the proof of Theorem \ref{rijogegergwerg9}  
 we see that the norm of  $f:C'\to C''$ in $\bW\bC$   is given by
  $\|f\|=\sup_{\nu} \|\sigma_{\nu}(f)\|$, where $\nu$ runs over all GNS representations of $\bC$.   Here we implicitly extended $\sigma_{\nu}$ to the $W^{*}$-envelope.  Similarly, $\|\bW\phi(f)\|=\sup_{\nu'} \|\sigma_{\nu'}(\bW\phi(f))\|$, where $\nu'$ runs over all GNS-representations of $\bD$.  But then
  $$\|f\|\ge \|\bW\phi(f)\|={\sup}_{\nu'}\|\sigma_{\nu'}(\bW\phi(f))\|\ge 
  {\sup}_{\nu}\|\sigma_{\phi_{*}\nu}(\bW\phi(f))\|=
  {\sup}_{\nu}\|\sigma_{\nu}(f)\|=\|f\|$$
  which implies the equality $\|f\|=\|\bW\phi(f)\|$.
  
  We next show that $\bW\phi$  creates the $\sigma$-weak topology on $\bW\bC$.  Note that the universal representations are defined
  as orthogonal sums over all GNS-representations. We thus  have an isometric embedding of representations $\sigma^{\bC}_{u}\to \phi^{*}\sigma_{u}^\bD$ whose complement is generated by GNS-representations of $\bD$  for   weights on endomorphism algebras of objects  which do not belong to the image of $\phi$.  For an object $C$ in $\bC$ we let 
  $p_{C}:\sigma^{\bD}_{u}(\phi(C))\to \sigma_{u}^{\bC}(C)$ denote the orthogonal projection. Then for $ f$ in $\Hom_{\bW\bC}(C,C')$  
  we have   $\sigma^{\bC}_{u}(f)=p_{C'}\sigma^{\bD}_{u}(\phi(f)) p_{C}$.
Now the $\sigma$-weak topologies on $\bW\bC$ and $\bW\bD$ are
induced via $\sigma^{\bC}_{u}$ and $\sigma_{u}^{\bD}$ from
$\Hilb(\C)^{\la}$, respectively. If $(f_{i})_{i}$ is a net in $\Hom_{\bW\bC}(C,C')$
such that the  $\sigma$-weak limit  of  $ ( \phi(f_{i}))_{i}$ exists in $\Hom_{\bW\bD}(\phi(C),\phi(C'))$, then  the $\sigma$-weak limits  of 
$(\sigma^{\bD}_{u}(\phi(f_{i})))_{i}$    and  therefore of 
$(\sigma^{\bC}_{u}(f_{i}))_{i}$  exist   in $ 
\Hilb(\C)^{\la}$. Hence the $\sigma$-weak limit   $(f_{i})_{i}$ exists in $\Hom_{\bW\bC}(C,C')$.

 Since $\bD$ is $\sigma$-weakly dense in $\bW\bD$,  $\phi$ is full, 
  and  $\bW\phi$ is $\sigma$-weakly continuous and  detects $\sigma$-weak convergence,  we can conclude that 
  $\bW\phi$ is surjective.    \end{proof}

Finally we define the $W^{*}$-envelope    of a possibly non-unital $C^{*}$-category.  
For $\bC$ in $\nCcat$ we let $\bC^{+}$ denote its unitalization. We have a faithful morphism $\bC\to \bC^{+}\to \bW\bC^{+}$.
 
\begin{ddd}\label{jgiewrgojwerfewrfeff}
We define   $\bW^{\mathrm{nu}}\bC$  in $\Wcat$ as the $\sigma$-weak closure of
$\bC$ in $\bW\bC^{+}$.
\end{ddd}

The induced  morphism  $\bC\to \bW^{\mathrm{nu}}\bC$ is faithful.
Its universal property will be explained in Proposition \ref{jigogewrgeferfref} below after  recalling of the concept of the multiplier category of $\bC$.

If $\phi:\bC\to \bD$ is a morphism in $\nCcat$, then the morphism $\bW(\phi^{+}):\bW\bC^{+}\to \bW\bD^{+}$ restricts to a $\sigma$-weakly continuous morphism $\bW^{\mathrm{nu}}(\phi):\bW^{\mathrm{nu}}\bC\to \bW^{\mathrm{nu}}\bD$ such that $$\xymatrix{\bC\ar[r]^{\phi}\ar[d]&\bD\ar[d]\\\bW^{\mathrm{nu}}\bC\ar[r]^{\bW^{\mathrm{nu}}(\phi)}&\bW^{\mathrm{nu}}\bD}$$
 commutes. In particular we obtain a functor $\bW^{\mathrm{nu}}:\nCcat\to \Wcat$. We will see in Corollary \ref{wropr} below that in analogy to Proposition \ref{rjigowergwregregwfer} the functor 
 $ \bW^{\mathrm{nu}}$ preserves fully faithfulness.

 The following will be used later. Let $\bC$ be in $\nCcat$.
 \begin{lem}\label{wtrogkprgerwgwegw}
 For any unital representation $\sigma:\bW^{\mathrm{nu}}\bC\to \Hilb(\C)^{\la}$
 the induced representation $\bC\to \bW^{\mathrm{nu}}\bC\to \Hilb(\C)^{\la}$ is
 non-degenerate.
 \end{lem}
\begin{proof}
Assume the contrary. Then there exists an object $C$ of $\bC$ and a non-zero vector $x$ in $\sigma(C)$ such that $\langle x,\sigma(f)x\rangle=0$ for every $f$ in $\End_{\bC}(C)$.
Note that $\langle x,\sigma(-)x\rangle$ is a continuous functional on $\End_{\bC}(C)$.  Since $ \bW^{\mathrm{nu}}\bC$ is the $\sigma$-weak closure of $\bC$ in $ \bW\bC^{+}$ there exists  
 a net $(f_{i})_{i\in I}$ in $ \End_{\bC}(C)$ which $\sigma$-weakly converges to
 $1_{C}$ in $\End_{ \bW^{\mathrm{nu}}\bC}(C)$. Since $\sigma$ is unital
 we have
  $$0\not=\langle x, x\rangle=\langle x, \sigma(1_{C})x\rangle
  ={\lim}_{i}\langle x, \sigma(f_{i})x\rangle=0\ ,$$
  a contradiction.
\end{proof}

%If $\phi:\bC\to \bC'$ is an unital  morphism of $C^{*}$-categories, then
%we get an extension $\bW\phi:\bW\bC\to \bW\bC'$ which is normal. 
 
\section{Multiplier categories}\label{weigjogfgsgdfgsfdg}

In this section we  discuss the  concept of a multiplier category of a $C^{*}$-category. 
It is an immediate generalization of the notion of a multiplier algebra of a $C^{*}$-algebra.
Since this concept is crucial for the present paper and the subsequent work
\cite{coarsek}, \cite{bel-paschke} we provide a detailed account. The main result is Theorem \ref{mmain} describing an explicit model for the multiplier category of a $C^{*}$-category  that is  characterized in Definition \ref{weighwoegerwferfwef}   by a universal property.
In Theorem \ref{jigogewrgeferfref} we furthermore describe the relation between the multiplier category and the  $W^{*}$-envelope introduced in  Definition \ref{jgiewrgojwerfewrfeff}.
 %
%
%In this section we  compare our approach to orthogonal sums in $C^{*}$-categories with the version of Antoun and Voigt.\footnote{The preprint   \cite{antoun_voigt} appeared on the arXiv shortly before we finished writing this paper.
%For earlier and similar approaches see also   Kandelaki \cite{kandelaki} and Vasselli \cite{vasselli}.}  
%Since our notion of an orthogonal sum differs from that in this reference,   in this section we take the chance  to recall their versions in detail and explain the similarities and differences.
%\uli{We will use the Antoun-Voigt theory in order to show for any $C^{*}$-algebra $A$ that in 
%$\pHilb(A)$ the classical notion  of an orthogonal sum of Hilbert $A$-modules 
% is equivalent to    the concept of an  orthogonal sum introduced in Definition \ref{erogwfsfdvbsbfdbsdfbsfdbv}.}
%
% 

% Antoun and Voigt put their preprint \cite{antoun_voigt} onto the arXiv in which they also discuss orthogonal sums for infinite families of objects in a $C^*$-category; similar approaches where also taken earlier by  In this section we compare their definition to ours.

The concept of   the multiplier category $\bM\bC$ of a $C^{*}$-category $\bC$ was  introduced in \cite{kandelaki} or \cite[Sec.~2]{antoun_voigt}.
Most of constructions and statements  concerning multiplier categories   together with their proofs are direct generalizations of  constructions and statements in \cite{busby_double_centralizers} from $C^{*}$-algebras to $C^{*}$-categories.

%We start {with} recalling the notion of  {a multiplier morphism}   {and the multiplier category $\bMC$  \cite[Sec.~2]{antoun_voigt}. For the sake of self-containedness we provide complete proofs.} 
  
Let $\bC$ be in $\nCcat$.
\begin{ddd}\label{weighwoegerwferfwef}
A multiplier category of $\bC$ is a unital $C^{*}$-category $\bM\bC$ with an ideal inclusion   $\bC\to \bM\bC$ such that for any other ideal inclusion $\bC\to\bD$  with $\bD$ a unital $C^{*}$-category   there is a unique unital morphism $\bD\to \bM\bC$ such that
$$\xymatrix{&\bC\ar[dr]\ar[dl]&\\\bD\ar[rr]&&\bM\bC}$$
commutes.
\end{ddd}

It is clear that if a multiplier category  $\bM\bC$  of $\bC$ exists, then it is determined uniquely up to unique isomorphism by the universal property.
  In the following we will  show the existence of a  multiplier category
  by providing an explicit model which we will also denote by $\bM\bC$.

 Let $\Ban$ denote the category of Banach spaces and continuous linear maps and consider  
  $\bC$  in $\nCcat$. An object   $C$  of $\bC$ represents  the
 $\Ban$-valued functors $\Hom_{\bC}(-,C) \colon \bC^{\op}\to \Ban$ and
 $\Hom_{\bC}(C,-) \colon \bC\to \Ban$.
 If $v$ is a natural transformation between $\Ban$-valued functors on $\bC$ given by a family
 $(v_{C})_{C\in \Ob(\bC)}$ of morphisms in $\Ban$, then we say that $v$ is uniformly bounded if 
 $\|v\| \coloneqq \sup_{C\in \Ob(\bC)} \|v_{C}\|<\infty$. The space of  uniformly bounded natural transformations between two $\Ban$-valued functors is again a Banach space with respect to this norm.

 Let $\bC$ be in $\nCcat$, and let $C,D$ be objects of $\bC$.  
\begin{ddd}\label{ewirjgoethwfegergewrg}\mbox{}
\begin{enumerate}
\item The Banach space of left multiplier morphisms from $C$ to $D$  is the Banach space of uniformly bounded natural transformations of $\Ban$-valued functors
$\Hom_{\bC}(-,C)\to \Hom_{\bC}(-,D)$ on $\bC^{\op}$. % (see Definition \ref{woiegbfdsb9}).
\item {The Banach space of right multiplier morphisms from $C$ to $D$  is the Banach space of  uniformly bounded natural transformations of $\Ban$-valued functors
$\Hom_{\bC}(D,-)\to \Hom_{\bC}(C,-)$ on $\bC$.}
\item \label{qeroigjqoergerqf} A multiplier morphism from $C$ to $D$ is a pair $(L,R)$ of a left and a right multiplier morphism from $C$ to $D$ such that for every $f$ in $\Hom_{\bC}(F,C)$ and every $g$ in $\Hom_{\bC}(D,E)$ we have
\begin{equation}\label{3r2fpokpwerverfw}
g L_F(f) = R_E(g) f\, .
\end{equation} 

We write $\MHom_{\bC}(C,D)$ for the $\C$-vector space of multiplier morphisms from $C$ to~$D$.
\end{enumerate}
\end{ddd}
In the following we spell out this definition in detail and explain the notation appearing in \eqref{3r2fpokpwerverfw}.
A left multiplier morphism $L\colon C \to D$ is a uniformly bounded family $(L_E)_{E \in \Ob(\bC)}$ of $\IC$-linear maps $L_E\colon \Hom_{\bC}(E,C) \to \Hom_{\bC}(E,D)$ such that for every $h$ in $\Hom_{\bC}(E,C)$ %,  {every object $F$ of $\bC$, 
  and every  $g$ in $\Hom_{\bC}(F,E)$  we have $L_F(hg) = L_E(h)g$.

Similarly,  a right multiplier morphism $R\colon C \to D$ is given by a uniformly bounded family $(R_E)_{E \in \Ob(\bC)}$ of $\IC$-linear maps $R_E\colon \Hom_{\bC}(D,E) \to \Hom_{\bC}(C,E)$ such that for every $h$ in $\Hom_{\bC}(D,E)$ and %every   object $F$ of $\bC$, and 
 every  $g$ in $\Hom_{\bC}(E,F)$  we have $R_F(gh) = gR_E(h)$.

 {Below, in order to simplify the notation,  we will omit the subscripts and write $L(h)$ instead of $L_{E}(h)$ or $R(g)$ instead of $R_{E}(g)$.}

Let $C,D,E$ be objects of $\bC$,  let 
$(L,R)$ be in $\MHom_{\bC}(C,D) $, and  $(L',R')$ be in $ \MHom_{\bC}(D,E)$. Then the pair of compositions 
$(L'L,RR')$ belongs to $\MHom_{\bC}(C,E)$. In this way we get a $\C$-bilinear  and associative law  of composition of multiplier morphisms
\begin{equation}\label{eq_composition_mult_morph}
\MHom_{\bC}(C,D) \times \MHom_{\bC}(D,E) \to \MHom_{\bC}(C,E)\, .
\end{equation}
 {For every object $C$ in $\bC$} we have an identity multiplier morphism $\id_{C}$ in $\MHom_{\bC}(C,C )$.
 Finally, the involution of
 $\bC$ induces an anti-linear  involution
 $$(-)^{*} \colon \MHom_{\bC}(C,D)\to  \MHom_{\bC}(D,C)\, , \quad (L,R)^{*} \coloneqq (R^{*},L^{*})\, .$$
 In detail,  if $L=(L_{E})_{E\in \Ob(\bC)}$ and $R=(R_{E})_{E\in \Ob(\bC)}$, then 
  $R^* =(R_E^*)_{E \in \Ob(\bC)}$ with %$L_E^*\colon \Hom_{\bC}(E,D) \to \Hom_{\bC}(E,C)$ defined by 
  $R_E^*(f) \coloneqq R_E(f^*)^*$ for every $f$ in $\Hom_{\bC}(E,C)$, 
  and analogously  $L^*=(L^{*}_{E})_{E\in \Ob(\bC)}$ with $L^{*}_{E}(f):=L_{E}(f^{*})^{*} $  for every $f$ in $\Hom_{\bC}(C,E)$.

The multiplier category $\bM\bC$ of $\bC$ is the object of $\Clincat$ defined as follows.
\begin{ddd}\label{wtiogjoweggfrewgre9}\mbox{}
\begin{enumerate}
\item The objects of $\bM\bC$ are the objects of $\bC$.
\item The $\C$-vector space of morphisms in $\bM\bC$ from $C$ to $D$ is the space of  multiplier morphisms $\MHom_{\bC}(C,D)$.
\item The composition and the involution are defined as described above.
\end{enumerate}
\end{ddd}
The name and notation for  $\bM\bC$ will be justified in  Theorem \ref{mmain} below.

Every morphism $f$ in $\Hom_{\bC}(C,D)$ naturally defines a  multiplier morphism 
$M_{f} \coloneqq (L_f,R_f)$ in $\MHom_{\bC}(C,D)$, where
$L_{f}=f\circ -$ and $R_{f}=-\circ f$.  %a multiplier morphism  of the same norm as $f$, 
We thus have a  morphism $\bC \to \bM\bC$ in $\nClincat$ which is the identity on the objects and given by    $f \mapsto (L_f,R_f)$ on morphisms.

Let $\bC$ be in $\nCcat$.
\begin{theorem}\label{mmain}\mbox{}
\begin{enumerate}
\item\label{wjithgowgegfergegw} $\bM\bC$ is a unital $C^{*}$-category. 
\item \label{qeroigjoqergfqewfqwefqewf}  The morphism
$\bC \to \bM\bC$ is the inclusion of an   ideal.
\item \label{rieogfergegwferf}
 The inclusion $\bC\to \bM\bC$ presents $\bM\bC$ as the multiplier category of $\bC$.\end{enumerate}
 \end{theorem}
\begin{proof}
 \ref{wjithgowgegfergegw}:
We will show that the norm of the Banach space of multiplier morphisms  exhibits $\bM\bC$ as a $C^{*}$-category. To this end we define the norm of
  a multiplier morphism $M = (L,R)$ by
\begin{equation}\label{eq_norm_multiplier}
\|M\| \coloneqq \max\{\|L\|,\|R\|\}\,.
\end{equation}
%see Definition \ref{woiegbfdsb9}. 
The involution $*$  on $\bM\bC$ is then isometric.

%where
%\[
%\|L\| \coloneqq \sup_{E \in \Ob(\bC)} \|L_E\|\,, \quad \|R\| \coloneqq \sup_{E \in \Ob(\bC)} \|R_E\|\,.
%\]
Next we show that actually   $\|L\| = \|R\|$. The argument is similar to 
the argument   for   double centralizers of $C^*$-algebras; cf.~\cite[Lem.~2.1.4]{murphy}. 
First of all note that for a morphism $f \colon C\to D$ in a $C^{*}$-category we have \begin{equation}\label{qwepfojpofqefqwefqewf}
\|f\|=\sup_{g }\|fg\|=\sup_{h}\|hf\|\,,
\end{equation} 
where $g$ runs over all morphisms with  target $C$ and $\|g\|\le 1$, and $h$ runs over all morphisms  
with domain $D$ and $\|h\|\le 1$. In fact, assume that $f\not=0$. Then we have  
$$ \|f\|\ge  \sup_{g }\|fg\|\ge \Big\| f\frac{f^{*}}{\|f\|} \Big \|=\|f^{*}f\|\|f\|^{-1}=\|f\|^{2}\|f\|^{-1}=\|f\|\, ,$$
 {where the} first inequality follows from the sub-multiplicativity of the norm, while
 the second inequality follows from specializing at $g=f^{*} /\|f\|$.

Assume that $M=(L,R)$ is a multiplier morphism from $C$ to $D$.
Then we have
 \begin{eqnarray}
\|L\|=\sup_{f}\|L(f)\|&\stackrel{\eqref{qwepfojpofqefqwefqewf}}{=}&
\sup_{h}\sup_{f}\|hL(f)\| \nonumber\\
&\stackrel{\eqref{3r2fpokpwerverfw}}{=}&\sup_{h}\sup_{f}\|R(h)f\|
\stackrel{\eqref{qwepfojpofqefqwefqewf}}{=}\sup_{h}\|R(h)\|
=\|R\|\, ,\label{tokjpherthertgetrgetrg}
\end{eqnarray}
 {where} $f$ runs over all morphisms with target $C$ and $\|f\|\le 1$, and
$h$ runs over all morphisms with domain $D$ and $\|h\|\le 1$.

 It is clear that $\bM\bC$ is complete since the spaces of left and right multipliers are complete. It is furthermore easy to see that the norm is sub-multiplicative for compositions.
 
 {In order to show that $\bM\bC$ is a $C^{*}$-category it remains to verify the strong $C^*$-inequality \eqref{eq_strong_Cstar_inequality}.}
%It remains to verify the $C^{*}$-equality and the $C^{*}$-inequality, see \cite[Rem.~2.15]{startcats}.
%The argument for the $C^{*}$-equality is similar to  the case of double centralizers for $C^{*}$-algebras; see \cite[Thm.~2.1.5]{murphy}. Let $M=(L,R)$ be a multiplier morphism from $C$ to $D$ with adjoint $(R^{*},L^{*})$.
%Then we have 
%\begin{eqnarray*}
%\|M\|^{2}&=&\|L\|^{2}=\sup_{f} \|L(f)\|^{2}= \sup_{f} \|L(f)^{*}L(f)\|\\
%&=&   \sup_{f} \|L^{*}(f^{*})L(f)\|
%\stackrel{\eqref{3r2fpokpwerverfw}}{=}
% \sup_{f} \|f^{*   }R^{*}(L(f))\|\\&\le&  \sup_{f} \| R^{*}(L(f))\|=\|R^{*}L\|=\|M^{*}M\|\\&\le&\|M^{*}\|\|M\|=\|M\|^{2}\ ,
%\end{eqnarray*}
%where $f$ runs over all morphisms with target $C$ and $\|f\|\le 1$. 
%We conclude that all inequalities in this chain are equalities.
%%%%%
%We finally show the $C^{*}$-inequality.
We consider a multiplier morphisms $M_{0} =(L_{0},R_{0})$ from $C$ to  $E$ and a  multiplier morphism  $M_{1}=(L_{1},R_{1})$ from  $C$ to $D$.
%For $f$ in $\Hom_{\bC}(E,C)$ with $\|f\|\le 1$ 
We then have
\begin{eqnarray*}
\|M_{0}^{*}M_{0}+M_{1}^{*}M_{1}\| &=& \|R_{0}^{*}L_{0}+R_{1}^{*}L_{1}\|\\&=&
\sup_{f}
\|R_{0}^{*}(L_{0}(f))+R_{1}^{*}(L_{1}(f)) \|\\&\ge&\sup_{f} \|f^{*}R_{0}^{*}(L_{0}(f))+f^{*}R_{1}^{*}(L_{1}(f)) \|\\&\stackrel{\eqref{3r2fpokpwerverfw}}{=}&\sup_{f}\| L_{0}^{*}(f^{*})L_{0}(f)+L_{1}^{*}(f^{*}) L_{1}(f)\|\\&=&\sup_{f}\| L_{0}(f)^{*}L_{0}(f)+L_{1}(f)^{*} L_{1}(f)\|\\&\stackrel{!}{\ge}&\sup_{f}
\| L_{0}(f)^{*}L_{0}(f) \| =\sup_{f}\|L_{0}(f)\|^{2}\\&=&\|M_{0}\|^{2}=\|M_{0}^{*}M_{0}\|\, ,
\end{eqnarray*}
where the supremum runs over all $f$ with target $C$ and $\|f\|\le 1$, and at the marked inequality we used the strong $C^{*}$-inequality of $\bC$.
% 
% 
%
%
%
%
% 
%$\bMC$ is a (unital) $C^*$-category if equipped with the involution \eqref{eq_involution_multipliers} and norm \eqref{eq_norm_multiplier}, which is proven similarly as in the case of $C^*$-algebras; see \cite[Thm.~2.1.5]{murphy}.
%
%\end{proof}

%\begin{lem}
\ref{qeroigjoqergfqewfqwefqewf}:  
%The morphism
%$\bC \to \bM\bC$ is the inclusion of an   ideal.
%\end{lem}
%\begin{proof}  
For $f$ in $\Hom_{\bC}(C,D)$ we have
$\|M_{f}\|=\|L_{f}\|\stackrel{\eqref{qwepfojpofqefqwefqewf}}{=}\|f\|$.
This implies that   $\bC \to \bM\bC$ is an isometric inclusion and therefore $\bC$ is closed in $\bM\bC$.

Consider a multiplier $M=(L,R)$ from $C$ to $D$ in $\bC$ and $h \colon D\to E$. Then using  \eqref{3r2fpokpwerverfw} we calculate that 
 $M_{h}M=M_{R(h)}$. These identities show that $\bC$ is an ideal in $\bM\bC$.
%(L_{h}L,RR_{h})$, {and} we have
%\[
%L_{h}(L(l))=hL(l)=R(h)l=L_{R(h)}(l)
%\]
%and
%\[
%R (R_{h}(g) ) =R(gh)=gR(h)=R_{R(h)}(g)\,.
%\]
%We conclude that $M_{h}M=M_{R(h)}$.
Similarly  for $l \colon E\to C$ we have 
$MM_{l}=M_{L(l)}$.
%\end{proof}

 %\begin{prop}
 \ref{rieogfergegwferf}:
% The inclusion $\bC\to \bM\bC$ presents $\bM\bC$ as the multiplier category of $\bC$.
% \end{prop}
 Let $\bD$ be any unital $C^{*}$-category containing $\bC$ as an ideal.
Then  we  have a unital morphism $\bD\to \bM\bC$ which sends a morphism in $\bD$ to the induced multiplier on $\bC$ given by the left and right compositions with the morphism. It induces the identity on $\bC$.
One further checks that there is no other unital morphism  $\bD\to \bM\bC$ inducing the identity on $\bC$.
 \end{proof}

\begin{rem}
Let $ \bD$ be in $\Ccat$ and $\bC\to \bD$ be an ideal  inclusion in $\nCcat$.
Then we can define the  orthogonal complement   $\bC^{\perp}$ of $\bC$  to be the   ideal of $\bD$ consisting of all morphisms 
  which compose to zero with all morphisms from $\bC$.
The ideal $\bC$ is called essential if $\bC^{\perp}=0$.

We now fix $\bC$ in $\nCcat$
 and consider the poset with respect to inclusion of all $\bD$ in $\Ccat$ with the same objects as $\bC$ and  containing $\bC$ as an essential ideal.
 The unitalization $\bC\to \bC^{+}$ is a minimal element of this poset.

One can characterize the multiplier category $\bM\bC$ as a maximal element of this poset. First of all it is clear that $\bC$ is an essential ideal of $\bM\bC$. Assume that  $\bC\to \bD$ is a bigger essential ideal inclusion. Then we have morphisms
$\alpha, \beta$ as in $$\xymatrix{&\bC\ar[dl]\ar[dr]&\\\bM\bC\ar@/^0.5cm/[rr]^{\alpha}&&\bD\ar@/^0.5cm/[ll]^{\beta}}$$
where $\alpha$ witnesses the condition of being bigger,  and $\beta$ arrises from the universal property of $\bM\bC$.
From the uniqueness clause of this universal property $\beta\circ \alpha=\id_{\bM\bC}$. Since $\bC\to \bD$ is essential $\beta$ is faithful and consequently also $\alpha\circ \beta=\id_{\bD}$.

The poset admits maximal elements by Zorn's Lemma. It is applicable since if  $(\bD_{i})_{i}$  is a chain in this poset, then $\bar \bD:=\colim_{i}\bD$
is an upper bound of this chain. \hB 
\end{rem}

%To rephrase the above definition: a net of multiplier morphisms $(M_i)_{i \in I}$ in $\MHom_{\bC}(C,D)$ converges to $M$ in $\MHom_{\bC}(C,D)$ if and only if for every admissible $f$ and $h$ we have that $(L_i(f))_{i \in I}$ and $(R_i(h))_{i \in I}$ converge in norm to $L(f)$, resp.~to $R(h)$.

Next we argue that the algebraic conditions on multiplier morphisms {alone} already imply the boundedness assumptions. We further introduce and study the strict topology on the multiplier category. 

Let $\bC$ be in $\nCcat$, and let $C,D$ be objects of $\bC$. 
We define an algebraic left multiplier from $C$ to $D$  as a natural transformation of $\C$-vector space valued functors
$\Hom_{\bC}(-,C)\to \Hom_{\bC}(-,D)$ on $\bC^{\op}$. Similarly, 
an algebraic right multiplier from $C$ to $D$  is a natural transformation of $\C$-vector space valued functors
$\Hom_{\bC}(D,-)\to \Hom_{\bC}(C,-)$ on $\bC$.  An algebraic multiplier morphism   is a pair
$(L,R)$ of an algebraic left multiplier morphism $L=(L_{E})_{E\in \Ob(\bC)}$, and an algebraic right multiplier morphism $R=(R_{E})_{E\in \Ob(\bC)}$ such that for every $f$ in $\Hom_{\bC}(F,C)$ and every $g$ in $\Hom_{\bC}(D,E)$ we have $
g L_F(f) = R_E(g) f
$. We let $\aMHom_{\bC}(C,D)$ denote the $\C$-vector space of algebraic multiplier morphisms.
 
 %For a multiplier morphism $M$ we retain the notation $M = (L,R)$ for its constituents.
 
 In order to define the strict topology on multipliers we introduce  the following collections of seminorms.
%\begin{ddd}\label{welirgjowergergergferwf}
For every morphism $f$ with target $C$ % an object $E$ of $\bC$ and a morphism $f$ in $\Hom_{\bC}(E,C)$ 
we define the seminorm
\begin{equation}\label{eq_left_strict_topology}
l_f\colon \aMHom_{\bC}(C,D) \to \IR_{\ge 0}\,,\quad l_f((L,R)) \coloneqq \|L(f)\|\,.
\end{equation}
For every morphism $h$ with domain $D$  we define the seminorm%for an object $F$ of $\bC$ and a morphism $h$ in $\Hom_{\bC}(D,F)$ we define
\begin{equation}\label{eq_right_strict_topology}
r_h\colon \aMHom_{\bC}(C,D) \to \IR_{\ge 0}\,,\quad r_h((L,R)) \coloneqq \|R(h)\|\,.
\end{equation}

\begin{ddd}\label{erguihierverqerf}
The strict topology on $\aMHom_{\bC}(C,D)$ is the locally convex topology given by the family of semi-norms $(l_f)_{f}\cup (r_h)_{h}$, where $f$ runs over morphisms with target $C$ and $h$ runs over morphisms with domain $D$. 
\end{ddd}
%through all morphisms in $\Hom_{\bC}(E,C)$ for every object $E$ of $\bC$ and $h$ runs through all morphisms in $\Hom_{\bC}(D,F)$ for every object $F$ of $\bC$.
%\end{ddd}

%\begin{lem} Then
% natural inclusion $\MHom_{\bC}(C,D)\to \aMHom_{\bC}(C,D)$ is an isomorphism.
% \end{lem}
% \begin{proof}
% \end{proof}
%
% 

\begin{prop}\label{wrtohijworthrtwregergwerg}\mbox{}
\begin{enumerate} \item  \label{oiergjoiwegwrrgwr}The
 natural inclusion $\MHom_{\bC}(C,D)\to \aMHom_{\bC}(C,D)$ is an isomorphism.
 \item \label{oiergjoiwegwrrgwr1}   $\Hom_{\bC}(C,D)$ is strictly dense in $\MHom_{\bC}(C,D)$.
\item\label{oiergjoiwegwrrgwr2}   $\MHom_{\bC}(C,D)$ is complete with respect to the strict topology.
 
\end{enumerate}
\end{prop}
\begin{proof}
Let $(L,R)$ be an algebraic multiplier morphism from $C$ to $D$. We first show that
the members of  the families $L=(L_{E})_{E\in \Ob(\bC)}$ and $R=(R_{E})_{E\in \Ob(\bC)}$ are bounded.
We fix an object $E$ of $\bC$. Then $L_{E} \colon \Hom_{\bC}(E,C)\to  \Hom_{\bC}(E,D)$ is a linear map of Banach spaces. 
We show that its graph is closed and conclude that it is continuous and hence bounded. 
Let $(f_{i})_{i}$ be a net in $\Hom_{\bC}(E,C)$ such that $\lim_{i} f_{i}=f$ and assume that $\lim_{i}L(f_{i})=:g$ exists.
For every $h$ in $\Hom_{\bC}(D,F)$ we have
\begin{eqnarray*}
\|h L(f)-hg\|&\le&\|h L(f)-hL(f_{i})\|+\|hL(f_{i})-hg\|\\&=&
\|R(h) ( f-f_{i})\|+\|h(L(f_{i})-g)\|\,.
\end{eqnarray*}
Applying $\lim_{i}$ we get
$\|h(L(f)-g)\|=0$ for all $h$. By \eqref{qwepfojpofqefqwefqewf} we can conclude that
$L(f)=g$.
This is a first step towards the verification of Assertion \ref{oiergjoiwegwrrgwr}.

We show now the Assertion \ref{oiergjoiwegwrrgwr1}.  We actually show the stronger assertion hat  $\Hom_{\bC}(C,D)$ is strictly dense in $\aMHom_{\bC}(C,D)$. 
Let $M=(L,R)$ be in $\aMHom_{\bC}(C,D)$ and let $(h_{i})_{i}$ be a selfadjoint approximate unit of $\End_{\bC}(D)$.
%We first show that $\lim_{i} h_{i}l=l$ for every morphism $l$ with target $D$ and $\lim_{i} k h_{i}=k$ for every morphism with domain $D$. We give the argument for the first case. Note that 
%$$\|h_{i}l-l\|^{2}=\|(h_{i}l-l)(h_{i}l- l)^{*}\|=\|( ll^{*} - {h_{i}} ll^{*}     ) +(ll^{*} - ll^{*}h_{i}  ) - (ll^{*} -h_{i}ll^{*}h_{i})\|\, .
%$$ 
%We can rewrite the last term in the form  $ h_{i}ll^{*}h_{i}= h_{i}\sqrt{ll^{*}}   \sqrt{ll^{*}}h_{i} $.
%Since
%$\lim_{i}h_{i}ll^{*}=ll^{*} =\lim_{i} ll^{*}h_{i}$ and
%$\lim_{i} h_{i}\sqrt{ll^{*}}=  \sqrt{ll^{*}}=\lim_{i}   \sqrt{ll^{*}} h_{i}$ we conclude that 
%$\lim_{i}\|h_{i}l-l\|^{2}=0$.
 Then we have  $M_{h_{i}}M=M_{R(h_{i})}$. We show that
\[
\lim_{i}M_{R(h_{i})}=M
\]
in the strict topology. 
Let $f$ be in $\Hom_{\bC}(C,D)$. Then we have $L_{R(h_{i})}(f)=R(h_{i})f=h_{i}L(f)$ {and hence}
$\lim_{i}L_{R(h_{i})}(f)=\lim_{i} h_{i}L(f)=L(f)$ by Lemma \ref{multmult}. Similarly, for $g$ in $\Hom_{\bC}(D,C)$ we have 
$\lim_{i} R_{L(h_{i})}(g)=R(g)$.  {This proves Assertion \ref{oiergjoiwegwrrgwr1}.}
 
We now finish the proof of Assertion \ref{oiergjoiwegwrrgwr}. If $(L,R)$
is in  $\aMHom_{\bC}(C,D)$, then we have already seen that $L$ and $R$ are implemented by families $L=(L_{E})_{E\in \Ob(\bC)}$ and $R=(R_{E})_{E\in \Ob(\bC)}$ of bounded  maps. It remains to show that these families are uniformly bounded.
We now note that $L=\lim_{i} L_{R(h_{i})}$ in the strict topology, {where $(h_{i})_{i}$ is a selfadjoint bounded approximate unit}. 
Thus $L(f)=\lim_{i} R_{D}(h_{i})f$ for all $f$ in $\Hom_{\bC}(E,C)$. In particular, 
$\|L(f)\| \le \sup_{i} \|R_{D}\|\| h_{i}\| \|f\|\le \|R_{D}\|{\|f\|}$ since  
$\sup_{i}\| h_{i}\|\le 1$. This shows that $\|L\|\le \|R_{D}\|$. Similarly one shows that
$\|R\|\le \|L_{D}\|$.

We finally show Assertion \ref{oiergjoiwegwrrgwr2}.
The arguments are the same as for double centralizers for $C^{*}$-algebras \cite[Prop.\ 3.6]{busby_double_centralizers}.
Let $(M_{\nu})_{\nu}$ be a Cauchy net with respect to the strict topology in $\MHom_{\bC}(C,D)$.
Set $M_{\nu}=(L_{\nu},R_{\nu})$. Then $L \coloneqq \lim_{\nu} L_{\nu}$ and
$R \coloneqq \lim_{\nu} R_{\nu}$ exist pointwise and obviously define an element $M=(L,R)$ of $\aMHom_{\bC}(C,D)$.
We now use  Assertion \ref{oiergjoiwegwrrgwr} in order to conclude that 
$M$ belongs to $\MHom_{\bC}(C,D)$.
\end{proof}

\begin{rem}\label{qirojegwergwergrweg}
The proof of Assertion  \ref{oiergjoiwegwrrgwr1} shows that every multiplier in $\bM\bC$ is the strict limit of a uniformly bounded net in $\bC$. \hB
\end{rem}

%As for $C^*$-algebras we can prove that $\MHom_{\bC}(C,D)$ is complete with respect to the strict topology and $\Hom_{\bC}(C,D)$ is strictly dense in $\MHom_{\bC}(C,D)$; cf.~\cite[Prop.~3.5 \& 3.6]{busby_double_centralizers} for the $C^*$-algebraic proof. 
One can also check that the composition \eqref{eq_composition_mult_morph} is separately strictly continuous, and jointly strictly continuous on bounded subsets.

Let $C$ be an object of $\bC$.
\begin{lem}\label{lem_unital_strict_equals_norm}
Assume that $C$ is unital and that $D$ is any object in $\bC$.
\begin{enumerate}
\item \label{eqroigqwgwefqewfqf}
We have equalities
\[
\Hom_{\bC}(C,D) = \MHom_{\bC}(C,D) \quad \text{and} \quad \Hom_{\bC}(D,C) = \MHom_{\bC}(D,C)\,.
\] 
 \item  \label{eqroigqwgwefqewfqf1}  On  $\MHom_{\bC}(C,D)$ and $\MHom_{\bC}(D,C)$ the strict and norm topologies  coincide.
 \end{enumerate}
\end{lem}
 \begin{proof}
Assertion \ref{eqroigqwgwefqewfqf} is an immediate consequence of  Assertion \ref{mmain}.\ref{qeroigjoqergfqewfqwefqewf}.  
We now show Assertion \ref{eqroigqwgwefqewfqf1}.
 It is clear that the norm on $\MHom_{\bC}(C,D)$ bounds (up to scale)  all the seminorms $l_{f}$ in \eqref{eq_left_strict_topology} and $r_{h}$ in \eqref{eq_right_strict_topology}.
 %  in Definition \ref{welirgjowergergergferwf}.
  In particular, for a   multiplier morphism $M=(L,R)$  
 we have 
 $$l_{\id_{C}}(L)\le \|L\|=\|M\|\, .$$  
 On the  other hand we have $$\|M\|=\|R\|=\sup_{g}\|R(g)\|= \sup_{g}\|R(g)\id_{C}\|=\sup_{g}\|gL(\id_{C})\| \le  l_{\id_{C}}(L)\, ,$$
where the supremum runs over all morphisms $g$ with domain $D$ and $\|g\|\le 1$. 
This shows that the seminorm $  l_{\id_{C}} $ is equivalent to the norm on $\MHom_{\bC}(C,D)$.
%We only show the equality $\Hom_{\bC}(C,D) = \MHom_{\bC}(C,D)$, because the other one is completely analogous.
%\color{blue}
%
%Let $M = (L,R)$ be a multiplier morphism in $\MHom_{\bC}(C,D)$. Then there exists a net $(f_i)_{i \in I}$ in $\Hom_{\bC}(C,D)$ such that the net of multiplier morphisms $((L_{f_i},R_{f_i}))_{i \in I}$ converges strictly to $M$. By the definition of strict convergence, this implies that $(L_{f_i}(\id_C))_{i \in I}$ converges in norm to $L(\id_C) \eqqcolon f$. Since for any morphism $g$ in $\Hom_{\bC}(E,C)$ for any object $E$ we have $L(g) = L(\id_C g) = L(\id_C) g$, we conclude that $L$ is given by left-multiplication by $f$. Analogously, one shows that $R$ is given by right-multiplication by $f$. The argument that the strict and the norm topology coincide is similar.
\end{proof}

From now on for a multiplier morphisms $(L,R)$ from $C$ to $D$ we use the same notation $f$    as for morphisms    and write $gf$ instead of $R(g)$ and $fh$ instead of $L(h)$. 
%Let $\bC, \bD$ be in $\nCcat$, let $C_1,C_2$ be objects of $\bC$, and let $F \colon \bC \to \bD$ be a $C^*$-functor.
%\begin{ddd}
%We call the map $F(C_1,C_2)\colon \Hom_{\bC}(C_1,C_2) \to \Hom_{\bD}(F(C_1),F(C_2))$ strictly continuous on bounded subsets, if for every norm bounded subset $X$ of $\Hom_{\bC}(C_1,C_2)$ the composition
%\[
%X \xrightarrow{F(C_1,C_2)|_{X}} \Hom_{\bD}(F(C_1),F(C_2)) \to \MHom_{\bD}(F(C_1),F(C_2))
%\]
%is continuous with respect to the strict topology on the target space $\MHom_{\bD}(F(C_1),F(C_2))$ and the restriction of the strict topology on $\MHom_{\bC}(C_1,C_2)$ to the domain $X$.
%
%$F$ is strict, if for all objects $C_1$ and $C_2$ of $\bC$ the map $F(C_1,C_2)$ is strictly continuous on bounded subsets.
%\end{ddd}
%
%\textbf{todotodotodo}
%
%\[
%M\colon \nCcat_{\mathrm{strict}} \to \Ccat\,.
%\]
%\textbf{Adjunktion?} \textbf{Identität auf $\Ccat$!?}

In the following we discuss the relation between multiplier categories and $W^{*}$-envelopes.  We furthermore discuss the functoriality of the multiplier category.  
Let $\bC$ be in $\nCcat$ and $\sigma$ be in $\Rep(\bC)^{\la}$.
Recall the Definition \ref{ewkogwtrgwegwergferf} of non-degeneracy of a representation.
 \begin{lem}\label{qeirogqergqfqefqewfq}
If $\sigma$ is faithful and  non-degenerate, then it uniquely extends to a unital and faithful representation $\bM\sigma \colon \bM\bC\to \Hilb(\C)^{\la}$. 
\end{lem}
\begin{proof} Let $f:C\to C'$ be a morphism in $\bM\bC$.
For $x$ in the dense subspace generated by $ \sigma(\End_{\bC}(C))\sigma(C)$  of $\sigma(C)$ we choose finite families $(u_{i})_{i}$ in $
\End_{\bC}(C)$ and $(y_{i})_{i}$ in $\sigma(C)$ such that $x=\sum_{i}\sigma(u_{i})y_{i}$.
Then we must define $\bM\sigma(f)(x):=\sum_{i}\sigma(fu_{i})(y_{i})$. 
In order to see that this element of $\sigma(C')$ is well-defined we consider  other choices of finite families $(u'_{j})_{j}$ and $(y'_{j})_{j}$  such that $\sum_{j}\sigma(u'_{j})y_{j}'=x$.
Then for any $v$ in $\End_{\bC}(C')$ we have  \begin{equation}\label{erferffsdfssdferfwe}
 0=\sigma(vf)(\sum_{i}\sigma(u_{i})y_{i}-\sum_{j}\sigma(u_{j}')y_{j}')=\sigma(v) (  \sum_{i}\sigma(fu_{i})y_{i}-\sum_{j}\sigma(fu_{j}')y_{j}')\ . 
\end{equation}
 Since $\sigma$ is non-degenerate we have
$\bigcap_{v\in \End_{\bC}(C')}\ker(\sigma(v)) =\{0\}$. 
Since $v$ is arbitrary  we conclude from \eqref{erferffsdfssdferfwe}  that $ \sum_{i}\sigma(fu_{i})y_{i}=\sum_{j}\sigma(fu'_{j})y_{j}'$.
 
 We now show that $\bM\sigma(f)$ is bounded and hence extends continuously to an operator defined on all of $\sigma(C)$.
 To this end we let $v$ run over a  normalized approximate unit of the ideal $ \End_{\bC}(C')$ in $\End_{\bM\bC}(C')$. Then
 $$\lim_{v}   \sigma(vf)x=\lim_{v} \sum_{i}\sigma (vfu_{i})y_{i}=    \sum_{i}\sigma(fu_{i})y_{i}= \bM \sigma(f)x \ .$$ On the other hand, for every member $v$ of the normalized approximate unit we have
 $$\| \sigma(vf)x\|\le \|\sigma(vf)\|\|x\|\le \|vf\| \|x\|\le \|f\|\|x\|\ .$$
 Hence
 $\|\bM\sigma(f)x\|\le \|f\|\|x\|$ and  $\|\bM\sigma(f)\|$ is bounded by $\|f\|$. This finishes the construction of a 
  unital morphism $\bM\sigma:\bM\bC\to \Hilb(\C)^{\la}$ extending $\sigma$.

In order to see that it is faithful we note that for any $\epsilon$ in $(0,\infty)$ there exists $u$ in $\End_{\bC}(C)$ with $\|u\|\le 1$ such that $\|fu\|\ge (1-\epsilon)\|f\|$.
But since $\sigma$ is faithful we then have the  last inequality in
$$\|\bM\sigma(f)\| \ge \|\bM\sigma(f)\|\|  \sigma(u)\|\ge \|\sigma(fu)\|\ge  (1-\epsilon)\|f\|\ .$$  
  Since $\epsilon$ is arbitrary and clearly $\|\bM\sigma(f)\|\le \|f\|$ we have $  \|\bM\sigma(f)=\|f\|$.\end{proof}
Let $\phi:\bC\to \bD$ be a morphism in $\nCcat$.
\begin{ddd}\label{fkjbofgbwtgdgbdfgb}
The morphism $\phi$  is called non-degenerate if for every two objects $C,C'$ of $\bC$
the sets $\phi(\End_{\bC}(C'))\Hom_{\bD}(\phi(C),\phi(C'))$ %
and $\Hom_{\bD}(\phi(C),\phi(C'))\phi(\End_{\bC}(C))$ generate dense linear subspaces in $\Hom_{\bD}(\phi(C),\phi(C'))$.
\end{ddd}

\begin{rem} One should not confuse  the notion of non-degeneracy  from Definition \ref{fkjbofgbwtgdgbdfgb} with the notion of a non-degenerate representation on $\Hilb(\C)^{\la}$ introduced in Definition \ref{ewkogwtrgwegwergferf}.
A representation $\sigma:\bC\to \Hilb(\C)^{\la}$ which is non-degenerate 
according to Definition \ref{ewkogwtrgwegwergferf} will in general not be 
non-degenerate according to Definition \ref{fkjbofgbwtgdgbdfgb}. But note that the converse holds.
\hB
\end{rem}
\begin{ex}\label{nijdohnfdndfgndfgndfgnd}
We claim that for every $\bC$ in $\nCcat$ the canonical morphism $\bC\to \bW^{\mathrm{nu}}\bC$
is non-degenerate. 
Assume the contrary. Then there exists an object $C$ of $\bC$ such that
$\End_{\bC}(C)$ is not norm-dense in $\End_{\bW^{\mathrm{nu}}\bC}(C)$. 
 By Hahn--Banach we can find a non-trivial state $\lambda$ of $\End_{\bW^{\mathrm{nu}}\bC}(C)$  which vanishes on $\End_{\bC}(C)$. 
 The  GNS-representation $\sigma_{\lambda}$  
 annihilates $\End_{\bC}(C)$,  but $\sigma_{\lambda}(C)\not=0$.
  This contradicts the assertion of Lemma \ref{wtrogkprgerwgwegw} that
  $\bC\to \bW^{\mathrm{nu}}\bC\stackrel{\sigma_{\lambda}}{\to} 
  \Hilb(\C)^{\la}$ is non-degenerate.
\hB
\end{ex}

\begin{lem}\label{wrtjioghferferwfrfrw}
A faithful and non-degenerate morphism $\kappa:\bC\to \bD$ in $\nCcat$ with $\bD$ a $W^{*}$-category uniquely extends to a faithful  morphism
$\bM\kappa:\bM\bC\to \bD$.
\end{lem}
\begin{proof}
We can choose a faithful representation $\rho:\bD\to \Hilb(\C)^{\la}$. Then the canonical inclusion $\bD\to \bD''_{\rho}$ is an isomorphism. By Lemma \ref{qeirogqergqfqefqewfq} 
 the composition  $\sigma\colon \bC\stackrel{\kappa}{\to} \bD\stackrel{\rho}{\to}  \Hilb(\C)^{\la}$   extends to a faithful morphism $\bM\sigma:\bM\bC\to \Hilb(\C)^{\la}$. 
  It remains to show that  for every two objects $C,C'$ in $\bC$
  we have \begin{equation}\label{vfsdfvqrwevafv}
\bM\sigma(\Hom_{\bM\bC}(C,C'))\subseteq \rho( \Hom_{\bD}(\kappa(C),\kappa(C')))\ .
\end{equation}
  Then 
 we can define $\bM\kappa$ as $\kappa$ on the level of objects and by $\rho^{-1}\circ \bM\sigma$ on the level of morphisms.
   
   We now verify the relation \eqref{vfsdfvqrwevafv}.
 Let $v=(v_{C})_{C\in \bC}$ be in $\End_{\Rep(\bD)}(\rho)$. Then using the notation from the proof of Lemma \ref{qeirogqergqfqefqewfq}, we have for $f$ in $\Hom_{\bM\bC}(C,C')$ that
\begin{align*}\bM\sigma(f)v_{C}x=\bM\sigma(f)v_{C}\sum_{i}\sigma(u_{i})y_{i}&=
 \bM\sigma(f)\sum_{i}\sigma(u_{i})v_{C} y_{i}\\=&\sum_{i}\sigma(fu_{i})v_{C}y_{i}=v_{C'}\sum_{i}\sigma(fu_{i})y_{i}=v_{C'}\bM\sigma(f)x\ .\end{align*} We conclude that 
 $\bM\sigma(f)v_{C}=v_{C'}\bM\sigma(f)$. 
 Since $v$ is arbitrary this shows that $\bM\sigma(f)$ belongs to $\Hom_{\bD_{\rho}''}(\kappa(C),\kappa(C')) $, and hence  to   $\rho( \Hom_{\bD}(\kappa(C),\kappa(C')))$, both viewed as subspaces of $\Hom_{\Hilb(\C)^{\la}}(\sigma(C),\sigma(C'))$.
 %Hence we get the desired extension by restricting the target of $\bM\sigma$.
 \end{proof}

Let $\bC$ be in $\nCcat$. By Lemma  \ref{wrtjioghferferwfrfrw} the faithful and non-degenerate (see Example \ref{nijdohnfdndfgndfgndfgnd}) morphism
$\bC\to \bW^{\mathrm{nu}}\bC$ uniquely extends to   a faithful morphism $ \bM\bC\to \bW^{\mathrm{nu}}\bC$.

%
%and consider the faithful morphism
%
%
%We let $u:\bC\to \bC^{+}$   denote the unitalization of $\bC$.
%Then we consider the following diagram:
%$$\xymatrix{ \bC\ar[rrrd]^{\tilde \sigma} \ar[drr]\ar[d]\ar@/^1cm/[rrr]^{ \sigma} \ar[r]^{u} &\bC^{+} \ar[r]^{i_{\bC^{+}}} &\bW\bC^{+}\ar[r]^{\hat \sigma_{u}}\ar[r]&\Hilb(\C)^{\la} \\ \bM\bC  \ar@/_1cm/[rrr]^{\bM\tilde\sigma}\ar@{..>}[rr]^{i_{\bC}}&&\bC''_{\tilde \sigma} \ar[r] \ar[u]_{j}& \Hilb(\C)^{\la}}\ .$$
%Here $\tilde \sigma$ is the non-degenerate subrepresentation of the 
%representation $\sigma:=\hat \sigma_{u}i_{\bC^{+}} u$.  Since $\sigma$ is faithful, $\tilde \sigma$ is faithful, too. We note that $\bC''_{\tilde \sigma}$ is the $\sigma$-weak closure of $\bC$ in $\bW\bC^{+}$.
%The inclusion $\tilde \sigma\to \sigma$ induces a faithful  non-unital morphism $j: \bC''_{\tilde \sigma}\to \bW\bC^{+}$ as indicated. By Lemma \ref{qeirogqergqfqefqewfq}
%we have the  faithful  extension $\bM\tilde \sigma$ of $\tilde \sigma$ to the multiplier algebra $\bM\bC$.
%% It actually takes values in $\bW\bC^{+}$. In order to see this note that 
%% the orthogonal projection  from $\sigma$ to $\tilde \sigma$ belongs to
%% $\End_{\Rep(\bC^{+})}(\sigma)$. Furthermore, using the explicit construction from the proof of Lemma \ref{qeirogqergqfqefqewfq}  we see that $\bM\tilde \sigma(f)$  
%% belongs to the commutant of $\End_{\Res(\bC^{+})}(\sigma)$ for every morphism in $\bM\bC$.
%% 

\begin{theorem}\label{jigogewrgeferfref}\mbox{}
The morphism $ \bM\bC\to   \bW^{\mathrm{nu}}\bC$
 identifies $\bM\bC$ with the idealizer of $\bC$ in $ \bW^{\mathrm{nu}}\bC$ and
 presents $\bW^{\mathrm{nu}}\bC$ as  the $W^{*}$-envelope of $\bM\bC$. 
\end{theorem}
\begin{proof}
Recall that the idealizer $\bI(\bC\subseteq  \bW^{\mathrm{nu}}\bC)$ of $\bC$ in $ \bW^{\mathrm{nu}}\bC$ is the wide subcategory consisting of all morphisms $f$ with the property that all compositions of $f$ with morphisms from $\bC$ again belong to $\bC$. Using that  $\bC$ is an ideal in $\bM\bC$,
by restricting the target of  $ \bM\bC\to   \bW^{\mathrm{nu}}\bC$  we get a morphism $i_{\bC}:\bM\bC\to \bI(\bC\subseteq  \bW^{\mathrm{nu}}\bC)$.
Since  $\bC$ is closed in $ \bW^{\mathrm{nu}}\bC$ 
we conclude that $\bI(\bC\subseteq \bW^{\mathrm{nu}}\bC)$ is a closed subcategory
of $ \bW^{\mathrm{nu}}\bC$. It clearly contains $\bC$.
By the universal property of the multiplier category  we get   a canonical morphism
$k_{\bC}:\bI(\bC\subseteq \bW^{\mathrm{nu}}\bC)\to \bM\bC$ which sends a morphism in the idealizer to the corresponding multiplier.
We check that   $i_{\bC}$ and $k_{\bC}$ are inverses to each other.
From the uniqueness clause of the  universal property of the multiplier category   it is clear that $k_{\bC}i_{\bC}=\id_{\bM\bC}$. 
We claim that $k_{\bC}$ is injective. The claim implies that also $i_{\bC}k_{\bC}=\id_{\bI(\bC\subseteq \bW^{\mathrm{nu}}\bC)}$.
In order to show the claim assume that  $f:C\to C'$ in $\bW^{\mathrm{nu}}\bC$ is a morphism  with $k_{\bC}(f)=0$. 
By \cite[Prop. 2.13]{ghr} we can choose a unital and faithful representation  $\sigma:\bW^{\mathrm{nu}}\bC\to \Hilb(\C)^{\la}$.
Then for every $x$ in $\sigma(C)$ and $u$ in $ \End_{\bC}(C)$ we have
$fu=0$ and hence  $ \sigma(f) \sigma(u)x=\sigma(fu)x=0$.
 Using Lemma \ref{wtrogkprgerwgwegw} we conclude that
 $\sigma(f)=0$. Since $\sigma$ is faithful we get $f=0$.

In order to show the second assertion we verify the universal properties stated in the proof of Theorem \ref{rijogegergwerg9}.
Since the image of $\bC$ is $\sigma$-weakly dense in $\bW^{\mathrm{nu}}\bC$, so is 
the image of $\bM\bC$. This shows Condition \ref{iiuaoidfafasdfasdfasdf}.
In order verify Condition \ref{wemiotgwegergerfw} we consider
a representation $\sigma:\bM\bC\to \Hilb(\C)^{\la}$.
We then consider the following diagram
$$\xymatrix{\bC\ar[r]^{u} \ar[dd]&\bC^{+} \ar@{..>}[ddl]_{1}\ar[r]&\bW\bC^{+}\ar@{..>}[dd]^{2}\\&\bW^{\mathrm{nu}}\bC\ar@{..>}[dr]^{3}\ar[ur]&\\\bM\bC\ar[ur]\ar[rr]^{\sigma}&&\Hilb(\C)^{\la}}$$
The arrow marked by $1$ is the canonical extension of $\bC\to \bM\bC$    given by the universal property of the unitalization $u$.
The morphism
$2$ is  the canonical extension of $\sigma\circ 1$ given by the universal property of the $W^{*}$ envelope of $\bC^{+}$. Finally, the morphism $3$ is given by the composition of the inclusion $\bW^{\mathrm{nu}}\bC\to \bW\bC^{+}$
with the arrow $2$.
\end{proof}

We now study the functoriality of the multiplier category.
 Let $\phi:\bC\to \bD$ be a morphism in $\nCcat$.
The following is a generalization of \cite[Prop. 3.12]{busby_double_centralizers}.
\begin{prop}\label{stkghosgfgsrgsegs}
If $\phi$ is  full    (non-degenerate, resp.), then it has a unique   extension 
$\bM\phi:\bM\bC\to \bM\bD$ which is strictly continuous (strictly continuous on  bounded subsets, resp.). 
If $\phi$ is fully faithful, then so is $\bM\phi$.
\end{prop}
\begin{proof}
We consider the more complicated case where $\phi$ is non-degenerate. The argument in the case where $\phi$ is full  is similar.
We consider the diagram
$$\xymatrix{\bC\ar[r]\ar[d]^{\phi}&\bM\bC\ar@{..>}[d]^{\bM\phi}\ar[r]&\bW^{\mathrm{nu}}\bC\ar[d]^{\bW^{\mathrm{nu}}\phi}\\\bD\ar[r]&\bM\bD\ar[r]&\bW^{\mathrm{nu}}\bD}$$
We claim that $\bW^{\mathrm{nu}}\phi$ restricts to a morphism
$ \bM\phi:\bM\bC\to \bM\bD$ which is in addition strictly continuous on bounded subsets.
Let $f$ be a morphism in $\bM\bC$.  Since $\phi$ is non-degenerate,  any morphism  in  $\bD$ can be approximated in norm by  linear compositions of compositions
$\phi(u)h$  for a morphisms $u$   in $\bC$ and morphisms $h$ in $\bD$. Then using that $f$  idealizes $\bC$   we see that $\bW^{\mathrm{nu}}\phi(f)\phi(u)h = \phi(fu)h $ is a morphism in $\bD$. 
 This implies that $\bW^{\mathrm{nu}}(\phi)(f)$ belongs to the idealizer of $\bD$ in $\bW^{\mathrm{nu}}\bD$ and hence to $\bM\bD$.
 We thus get the morphism $\bM\phi$. 
 
 In order to see that $\bM\phi$ is strictly continuous on bounded subsets let $(f_{i})_{i\in I}$ be a bounded   net  in $\bM\bC$ such that $\lim_{i}f_{i}=f$ strictly. Then $(\bM\phi(f_{i}))_{i\in I}$ is  a bounded net  in $\bM\bD$.  Hence it suffices to test left strict convergence on the set of morphisms of the form $\phi(u) h$ as above. In fact we have
$$\lim_{i} \bM\phi (f_{i}) \phi(u)h=\lim_{i}\phi(f_{i}u)h=\phi(fu)h=\bM(f)\phi(u)h\ .$$
Right strict convergence can be shown similarly.
If $\phi$ is full, then using a similar argument one can drop the condition that the net is bounded.

The uniqueness of the extension $\bM\phi$ follows from  continuity %\fuli{nicht wirklich: kann man alles durch beschraenkte Netze approximieren?} 
and the fact that $\bC$ is strictly dense in $\bM\bC$ by Proposition \ref{wrtohijworthrtwregergwerg}.\ref{oiergjoiwegwrrgwr1} together with Remark \ref{qirojegwergwergrweg}.

Now assume that $\phi$ is fully faithul. We first show that $\bM\phi$ is an isometric inclusion.
For  a morphism $(L,R):C\to C'$  in $\bM\bC$ write $\bM\phi(L,R)=(L',R')$. We have  the inequalities 
$$ \|(L,R)\|  \ge  \|\ (L',R')\|= \sup_{f} \| L'(f) \|\ge 
\sup_{g} \| L(g)\|=\|(L,R)\|\ , $$ where
$f$ runs over all morphisms in $\bD$ with target $\phi(C)$ and $\|f\|\le 1$, and $g$ runs over all morphisms in $\bC$ with target $C$ and $\|g\|\le 1$.
  The second inequality holds since $\phi$ is fully faithful and  the collection of $f$'s is bigger than the collection of $g$'s as $f$ may have a domain which does not belong to the image of $\phi$. We further used the fact observed in the proof of Theorem \ref{mmain}  that the norm of a multiplier is equal to the  norm of its left multiplier.
The chain of inequalities implies that
$\bM\phi$ is an isometric inclusion.

Finally we show that $\bM\phi$ is  full. To this end
  we show that $\bM\phi$ detects strict convergence.
If $(f_{i})_{i}$ is a net in $\Hom_{\bM\bC}(C,C')$  such that 
$(\bM\phi(f_{i}))_{i}$ strictly converges in $\Hom_{\bM\bD}(\phi(C),\phi(C'))$, then  for every
$h$ in $\Hom_{\bC}( C' , C'' )$  the net $(\phi(h) \bM\phi(f_{i}))_{i}$ in $\bD$ converges in norm.
Using the identities $\phi(h) \bM\phi(f_{i}) =   \phi(h f_{i})$ and that $\phi$ is fully faithful 
we see that the net $(hf_{i})_{i}$  in $\bC$ converges in norm. This shows that $\bM\phi$ detects right-strict convergence. Similarly we show that it detects left-strict convergence.

Since $\phi$ is fully faithful,  $\bM\phi$ is strictly continuous and detects  strict convergence, and  $\bD$ is strictly  dense in $\bM\bD$ by 
Proposition \ref{oiergjoiwegwrrgwr1}, we can conclude that $\bM\phi$ is surjective.
  \end{proof}

  \begin{kor}\label{wropr}  If $\phi:\bC\to \bD$ is a fully faithful morphism in $\nCcat$, then $\bW^{\mathrm{nu}}\phi:\bW^{\mathrm{nu}}\bC\to \bW^{\mathrm{nu}}\bD$ is fully faithful.
  \end{kor}
  \begin{proof}
  By Proposition \ref{stkghosgfgsrgsegs} the morphism $\bM\phi$ is fully faithful. Therefore $\bW\bM\phi$ is fully faithful by Proposition \ref{rjigowergwregregwfer}. By Theorem \ref{jigogewrgeferfref} we have an isomorphism
  $\bW\bM\phi\cong \bW^{\mathrm{nu}}\phi$ and the assertions follows.
    \end{proof}

The following extends the notion of a (unitary) natural transformation between morphisms in $\Ccat$ to the non-unital case.
Let $\phi,\phi':\bC\to \bD$ be morphisms in $\nCcat$.
\begin{ddd}
A (unitary) natural multiplier transformation $u:\phi\to \phi'$ is a family $u:=(u_{C})_{C\in \Ob(\bC)}$ of (unitary) morphisms in $\bM\bD$ such that
for every morphism $f:C\to C'$ in $\bC$ we have $u_{C'}\phi(f)=\phi(f)u_{C}$.
\end{ddd}

Recall that a morphism $\phi:\bC\to \bD$ in $\Ccat$  is  a unitary equivalence of 
it admits an inverse up to unitary isomorphism. Equivalently,  it is a fully faithful morphism that is in addition essentially surjective, i.e., every object in $\bD$ is unitarily isomorphic to an object in the image of $\phi$. Using the functoriality of the
multiplier category for fully faithful morphisms we extend the notion of a unitary equivalence to the non-unital case as follows. 

Let $\phi:\bC\to \bD$ in $\nCcat$ be a  morphism.  If $\phi$ is fully faithful, then  $\bM\phi:\bM\bC\to \bM\bD$ is defined  and also fully faithful by Proposition \ref{stkghosgfgsrgsegs}.
\begin{ddd}\label{ihjigwjegoerwgwerffwerfw}
$\phi$ is called a unitary equivalence if it is fully faithful and  $\bM\phi$ is a unitary equivalence.
\end{ddd}

\begin{rem} \label{retjheirotjhoergrtgertgetrgterg}
For morphisms in $\Ccat$ Definition  \ref{ihjigwjegoerwgwerffwerfw} reproduces  the classical notion.

For a general morphism  $\phi:\bC\to \bD$ in $\nCcat$ we provide three further  equivalent 
characterisations of being a unitary equivalence.

\begin{enumerate} \item
 The morphism $\phi:\bC\to \bD$ in $\nCcat$ is a unitary equivalence if and only it is fully faithful and   every object of $\bD$ is isomorphic  by a unitary multiplier morphism to an object in the image of $\phi$.
\item
The morphism  
 $\phi:\bC\to \bD$  in $\nCcat$ is a unitary equivalence if and only if it admits an inverse up to a   unitary natural  multiplier isomorphisms.   
 \item \label{etherhthgetgetrgetget}
 The morphism $\phi$ is a unitary equivalence if and only if it is fully faithful and a   part
 of a square $$\xymatrix{\bC\ar[d]^{\phi}\ar[r]&\bE\ar[d]^{\psi}\\\bD\ar[r]&\bF}$$ 
 where the horizontal morphisms are ideal inclusions
 and the morphism $\psi$ is a unitary equivalence in $\Ccat$.  Indeed, if $\phi$ is a unitary equivalence in the sense of Definition  \ref{ihjigwjegoerwgwerffwerfw},
 then we can take $\psi=\bM\phi:\bM\bC\to \bM\bD$.
 Vice versa, if we have the data of such a square, then any  object of $\bD$  is isomorphic to an object in the image of $\phi$ by a unitary in $\bF$. The image of this unitary under the canonical map $\bF\to \bM\bD$ yields a multiplier isomorphism of  our object with an object in the image of $\bD$. 
  \end{enumerate}
 The data in Point \ref{etherhthgetgetrgetget} could also be called a relative unitary equivalence.
 \hB
\end{rem}

We concider $\bC$ in $\nCcat$.
Note that we have the strict topology on the morphism spaces of $\bM\bC$ and the weak operator topology described in  Definition  \ref{weiojgowegerfwrf} on the morphism spaces of $\bW^{\mathrm{nu}}\bC$.

\begin{lem}\label{ewtiojgowtrgwergwergf}
The restriction of the inclusion  morphism $\bM\bC\to \bW^{\mathrm{nu}}\bC$ to bounded subsets 
  is  continuous with respect to the strict topology on the domain and the weak operator topology on the target.
\end{lem}
\begin{proof} 
%By Lemma \ref{qeirogqergqfqefqewfq}
%we can assume that the faithful representation $\sigma$ of $\bM\bK$  used to construct $\bC$ as $\bM\bK_{\sigma}''$ (here we  use the notation from Example \ref{wrkbhorbgfdbdfgbdfgb}) is unital  and  restricts  to a non-degenerate representation of $\bK$.

We can assume that the representations used to define the weak  operator topology described in  Definition  \ref{weiojgowegerfwrf} are unital.
Let $\sigma:\bW^{\mathrm{nu}}\bC\to \Hilb(\C)^{\la}$ be any normal and unital  representation.

Let $(f_{i})_{i\in I}$ be a bounded net in  $\bM\bC$ of   morphisms from $C$ to $C'$  which strictly converges to $f$. 
Let $x'$ be in $\sigma(C')$ and $x$ be in $\sigma(C)$. Then we must show 
that  the net $(\langle x', \sigma(f_{i})x\rangle)_{i\in I}$ in $\C$ converges to
$  \langle x', \sigma(f)x\rangle$

By Lemma \ref{wtrogkprgerwgwegw} we know that $\sigma:\bC\to \bW^{\mathrm{nu}}\bC\to \Hilb(\C)^{\la}$ is non-degenerate.
Since $ (\|\sigma(f_{i})\|)_{i\in I}$ is bounded    it suffices by linearity to show this convergence for all $x$ in the  subset $\sigma(\End_{\bC}(C))\sigma(C)$ of $\sigma(C)$  and 
$x'$ in the   subset $\sigma(\End_{\bC}(C'))\sigma(C')$ of $\sigma(C')$.
But then there are $u$ in $  \End_{\bC}(C)$ and $y$ in $  \sigma(C)$ such that $x=\sigma(u)y$, and $u'$ in $ \End_{\bC}(C')$ and $y'$ in $  \sigma(C')$
such that $x'=\sigma(u')y'$.
We then have
$\langle x', \sigma(f_{i})x\rangle=\langle y', \sigma(u^{\prime,*}f_{i} u)y\rangle$.
Since $ \lim_{i}u^{\prime,*}f_{i} u=u^{\prime,*}f u$ in norm
 we have 
 $$\lim_{i\in I}\langle x', \sigma(f_{i})x\rangle= 
\lim_{i\in I}\langle y', \sigma(u^{\prime,*}f_{i} u)y\rangle=\langle y', \sigma( u^{\prime,*}f u ),y\rangle=   \langle x', \sigma(f)x\rangle$$
as needed.
  \end{proof}

% 
%In order to see that $\bW\bC$ does not depend on the choice of the faithful representation $\iwe provide another characterization of $\bW\bC$ which does not depend on the embedding. 
%Indeed, by the bicommutant theorem \cite[Thm. 4.2]{ghr} 
%for every object $C$ of $\bC$ the $W^{*}$-algebra 
%$\End_{\bW\bC}(C)$ is the $W^{*}$-algebra generated by  
%the $C^{*}$-algebra $\End_{\bW\bC}(C)$. For objects $C,C'$ of $\bC$ the 
%  Banach space $\Hom_{\bW\bC}(C,C')$ is the bidual of $\Hom_{\bC}(C,C')$. 
 %Moreover, by    \cite[prop. 2.14]{ghr}  for any two objects $C,C'$ the set
%$\Hom_{\bWbC}(C,C')$ is the set of right $\End_{\bW\bC}(C)$- module maps
%$
\section{Weakly equivariant functors}\label{efigosgsfgsfdgsdfg}

%Let  $  \bC$ be in $\Fun(BG,\Ccat)$. %[or $\bK\in \Fun(BG,\nCcat)$]. In other words, $  \bC$ [or $\bK$] 

%\begin{rem}
%Since  $\Ccat$ is the underlying category of a $(2,1)$-category  {(the $2$-morphisms are the unitary isomorphisms)} we could also consider
%actions of $G$ with  relaxed associativity constraints. The adjective ``strict'' indicates that here we consider actions in the classical sense. 

 %The  $(2,1)$-categorial nature of $\Ccat$ also allows  to weaken  the equivariance condition on morphisms. 
 %This leads to the notion of a weakly invariant {morphism} {as defined below}.
%\hB
%\end{rem}

In this section we consider non-unital $C^{*}$-categories with strict $G$-action. Morphisms in $\Fun(BG,\nCcat)$ are equivariant functors.  
By introducing the notion of weakly equivariant morphism we relax the equivariance condition.  We will see that a morphism  in $\Fun(BG,\nCcat)$ which is a unitary equivalence 
in the sense of Definition \ref{ihjigwjegoerwgwerffwerfw} non-equivariantly
  admits a weakly equivariant inverse up to unitary multiplier isomorphism
  of weakly equivariant morphisms, but not an equivariant  inverse in general.
 The relaxed equivariance introduced in the present section  is relevant since, e.g., the Yoneda type embedding considered in Section \ref{wtogwepgfereggwrferf}   is not equivariant, but only weakly equivariant.   

 The following definition
extends  \cite[Def. 7.10]{crosscat} from unital to non-unital categories.
Let $\bC,\bC'$ be in $\Fun(BG,\nCcat)$. 
\begin{ddd}%[{\cite[Def. 7.10]{crosscat} }]
\label{rgjeqrgoieqrgrgqewgfq}
A weakly equivariant  morphism from $\bC$ to $\bC'$ is a pair $(\phi,\rho)$ consisting of the following data:
\begin{enumerate} 
\item a morphism
$\phi\colon  \bC \to  \bC' $ between the underlying $C^{*}$-categories in $\nCcat$,
\item a family $\rho=(\rho(g))_{g\in G}$ of unitary multiplier   transformations $\rho(g)\colon\phi\to g^{-1}\phi g$ 
 such that for all $g,g'$ in $G$ we have $g^{-1}\rho(g') g \circ  \rho(g)=\rho(g'g)$.
 \end{enumerate}
\end{ddd}
If $\phi \colon  \bC \to \bC' $ is a morphism  between the underlying $C^{*}$-categories in $\nCcat$, then weak equivariance of $\phi$ is an additional structure.
\begin{ex}\label{wrthjowthoewrgwergwerg}
If $  \phi \colon  \bC\to  \bC'$ is a  morphism in $\Fun(BG,\nCcat)$, i.e., an equivariant functor, then $(\phi,(\id_{\phi}))$ is a weakly  equivariant morphism from $\bC$ to $\bC'$. Here $(\id_{\phi})$ denotes the constant family on the multiplier morphism  $\id_{\phi}$. 
%If $\phi \colon \Res^{G}(\bC)\to \Res^{G}(\bC')$ is a  morphism such that  $g^{-1}\phi g=\phi$ for all $g$, then $\phi$ comes from a uniquely determined equivariant functor $\tilde \phi \colon \tilde \bC\to \tilde \bC'$.
\hB
\end{ex}

In the unital case weakly equivariant functors can be composed, see   
\cite[(7.8)]{crosscat}. In contrast, the composition of 
weakly equivariant functors in the non-unital case is only partially defined. 
The reason is that not every morphism in $\phi:\bC\to \bC'$ in
$\nCcat$ extends to the multiplier categories  in the sense that the morphism
$\bW^{\mathrm{nu}}\phi:\bW^{\mathrm{nu}}\bC\to \bW^{\mathrm{nu}}\bD$ 
sends  $\bM\bC$ considered by Theorem \ref{jigogewrgeferfref} as a subcategory of $ \bW^{\mathrm{nu}}\bC$ to 
  $\bM\bD$  considered as a subcategory of $\bW^{\mathrm{nu}}\bD$.

 Assume that $(\phi,\rho):\bC\to \bC'$ and $(\phi',\rho'):\bC'\to \bC''$ are weakly equivariant morphisms  between objects of $\Fun(BG,\nCcat)$.
 Then we want to define a composition
 $\rho'\circ \rho=((\rho'\circ\rho)(g))_{g\in G}$, where  
 $(\rho'\circ\rho)(g)$ is the  composition of
  natural multiplier transformations
\begin{equation}\label{wergrefwerfwerfwerfef}
\phi'\circ \phi\stackrel{\phi'\circ \rho(g)}{\to} \phi'\circ g^{-1}\phi g\stackrel{\rho'(g)\circ g^{-1}\phi g}{\to} g^{-1}(\phi'\circ \phi)g\ .
\end{equation}   
For the moment   we  interpret $ \phi'\circ \rho(g)$    as the family  $(\bW^{\mathrm{nu}}\phi'(\rho_{C}(g))_{C\in \Ob(\bC)}$ of morphisms in $\bW^{\mathrm{nu}}\bC'$. Similarly the whole composition in \eqref{wergrefwerfwerfwerfef} a priori consists of morphisms in $\bW^{\mathrm{nu}}\bC''$. 
We say that the composition of the weakly equivariant functor is defined if the family  $(\bW^{\mathrm{nu}}\phi'(\rho_{C}(g))_{C\in \Ob(\bC)}$  belongs to the subcategory $\bM\bC'$. 
  If the composition of the weakly equivariant functor is defined, then  it is given by  
 $$(\phi',\rho')\circ (\phi,\rho):=(\phi'\circ \phi,\rho'\circ \rho)\ .$$
 If $\bW^{\mathrm{nu}}\phi'$ restricts to a functor $\bM\phi':\bM\bC\to \bM\bD$, then a  composition with $(\phi',\rho')\circ-$ is defined. This is the case e.g. if $\phi'$ is fully faithful.
 Similarly, since $\bW^{\mathrm{nu}}\phi'$ is unital,  for an equivariant morphism $\phi$ a  composition  
 $-\circ (\phi,(\id_{\phi}))$ is defined.

 Let $(\phi,\rho)$ and $(\phi',\rho'):\bC\to \bC'$ be weakly equivariant morphisms. 
 \begin{ddd}
 A (unitary) natural multiplier  transformation $\kappa:(\phi,\rho)\to (\phi',\rho')$  
 is a (unitary) natural multiplier transformation $\kappa:\phi\to \phi'$ such that
 for every $g$  in $G$ we have $g^{-1}\kappa g\circ \rho(g)=\rho'(g)\circ \kappa$. 
   \end{ddd}
   Explicitly this condition means that for every $g$ in $G$ and object $C$ in $\bC$ we have the equality
   $g^{-1}\kappa_{gC}\circ \rho(g)_{C}=\rho'(g)_{C}\circ \kappa_{C}$.

Let $(\phi,\rho):\bC\to \bD$ be a weakly equivariant   morphism between objects of $\Fun(BG,\nCcat)$.
\begin{ddd}
Let  $(\phi,\rho):\bC\to \bD$  be called a unitary equivalence if $\phi$ is a unitary equivalence in the sense of Definition \ref{ihjigwjegoerwgwerffwerfw}. 
\end{ddd}
If $\phi$ is equivariant, then we apply this definition to $(\phi,(\id_{\phi}))$, see Example \ref{wrthjowthoewrgwergwerg}.
If $(\phi,\rho)$ is   a weakly equivariant unitary equivalence,
then we can choose an inverse $\psi:\bD\to \bC$ in $\nCcat$ such that there are natural  unitary multiplier isomorphisms
$\theta:  \phi\circ \psi\to \id_{\bD}$ and $\kappa:\id_{\bC} \to\psi\circ \phi$.

\begin{lem}\label{ekgjosgwregrewgwerwfrwef}
There exists an extension $(\psi,\lambda)$ of $\psi$ to a weakly equivariant
functor such that  $\theta: (\phi,\rho)\circ (\psi,\lambda) \to (\id_{\bD},(\id_{\id_{\bD}})) $ and $\kappa:(\id_{\bC},(\id_{\id_{\bC}})) \to(\psi,\lambda)\circ (\phi,\rho)$ are unitary multiplier isomorphisms between weakly equivariant functors.
\end{lem}
Note that the compositions above are defined since $\phi$ and $\psi$ are fully faithful.
\begin{proof}
For $g$ in $G$ and $D$ in $\bD$ the multiplier morphism $\lambda(g)_{D}$  must satisfy  \begin{equation}\label{ewfqwedwdwedweddeq}
g^{-1}\theta_{g   D  } \circ    \rho(g)_{g^{-1}\psi(gD)}   \circ \bM\phi(\lambda(g)_{D})=\theta_{D}\ .
\end{equation}
 
Since
$\bM\phi$ is fully faithful this determines $\lambda(g)_{D}$
uniquely. One checks that
$\lambda(g):=(\lambda(g)_{D})_{D\in \bD}$ is a natural multiplier isomorphism
from $\psi$ to $g^{-1}\psi g$ and that $\lambda:=(\lambda(g))_{g\in G}$ extends
$\psi$ to a weakly invariant functor.
One further checks that $\kappa$ and $\theta$ are then multiplier isomorphisms between weakly invariant functors.  
 \end{proof}

The following discussion   shows that if we invert
  unitary equivalences in  $\Fun(BG,\nCcat)$, then
every weakly equivariant morphism becomes equivalent to 
an equivariant one. To this end we construct  an endofunctor $Q$ of the 
$\Fun(BG,\nCcat)$ as follows:
  \begin{enumerate}
   \item 
 objects: For $\bD$ in $\Fun(BG,\nCcat)$ the category $Q(\bD)$ is given by:
  \begin{enumerate}
\item objects: The set of objects of $Q(\bD)$ is the set $\Ob(\bD)\times G$.
\item morphisms: For $(D,g)$ and $(D',g')$ in $Q(\bD)$ we define
$\Hom_{Q(\bD)}((D,g),(D',g')):=\Hom_{\bD}(D,D')$.
\item composition and involution: These structures are inherited from $\bD$.
\item $G$-action: The element $k$ in $G$ sends $(D,g)$ to $(kD,gk^{-1})$
and the morphism $f:(D,g)\to (D',g')$ to the morphism
$kf:(kD,gk^{-1})\to (kD',g'k^{-1})$.
\end{enumerate}
\item morphisms: If $\phi:\bD\to \bD'$  is a morphism in $\Fun(BG,\nCcat)$, then
we define the morphism  $Q(\phi)$ as follows:
\begin{enumerate}
\item objects: We set $Q(\phi)(D,g):=(\phi(D),g)$
\item morphisms: For  $f:(D,g)\to (D',g')$  in $Q(\bD)$ we set  $Q(\phi)(f):=\phi(f):(\phi(D),g)\to (\phi(D'),g)$ in $Q(\bD')$.
\end{enumerate}
\end{enumerate}
We have a natural transformation
$p:Q\to \id$ in $\Fun(BG,\nCcat)$ given by $p=(p_{\bD})_{\bD\in \Fun(BG,\nCcat)}$, where
$p_{\bD}:Q(\bD)\to \bD$ is the functor given by:
\begin{enumerate}
\item objects: $p_{\bD}(D,g):=D$ for every object $(D,g)$ of $Q(\bD)$.
\item morphisms: $p_{\bD}(f):=f$ for every morphism $f:(D,g)\to (D',g')$  in $Q(\bD)$.
\end{enumerate}
The morphism $p_{\bD}$ is a unitary equivalence. Indeed, 
we can choose a non-equivariant  inverse
$q_{\bD}:\bD\to Q(\bD)$ given by
\begin{enumerate}
\item objects: $q_{\bD}(D):=(D,e)$ 
\item morphisms: $q_{\bD}(f):=f$.
\end{enumerate}
Then $p_{\bD}\circ q_{\bD}=\id_{\bD}$ and
there is a unitary multiplier isomorphism 
$  \theta:q_{\bD}\circ p_{\bD}\to \id_{Q(\bD)} $
given by $\theta=(\theta_{(D,g)})_{(D,g)\in \Ob(Q(\bD))}$ with
$\theta_{(D,g)}=1_{D}:(D,e)\to (D,g)$.
Note that $1_{D}$ is only a multiplier isomorphism if $\bD$ is not unital.
By Lemma \ref{ekgjosgwregrewgwerwfrwef} we can extend $q_{\bD}$ to a weakly invariant morphism $(q_{\bD},\lambda):\bD\to Q(\bD)$.
In this case the formula \eqref{ewfqwedwdwedweddeq}  for
$\lambda(g)_{(D,k)}$ gives
$\lambda(g)_{(D,k)}=1_{D}:(D,e)\to (D,g)$.

%
%We have an equivariant fully faithful  forgetful functor  $p_{\bD}: Q(\bD)\to \bD$  which sends
%$(D,g)$ to $gD$ and the morphism $f:(D,g)\to (D',g')$ to $f:D\to D'$. 
%The family $(p_{\bD})_{\bD\in \nCcat}$ is a natural transformation. 
% The morphism $p_{\bD}$ is an equivalence  
%up to unitary multiplier morphism.  
%We can extend $q_{\bD}$ to a weakly 
% 
 
Going from $\bD$ to $Q(\bD)$ has the effect of making the $G$-action on the set of objects  free.  The functor $Q$ is a non-unital analog of the cofibrant replacement functor for the projective model category structure on $\Fun(BG,\Ccat)$ considered in \cite[Sec. 15]{startcats}.  

We  consider $\bE$ in $\Fun(BG,\nCcat)$ and a weakly  equivariant morphism  $(\phi,\rho)\colon \bD\to \bE$. 
\begin{lem}\label{eoigjosegbsfdgsertsfd}
There exists an equivariant morphism $\hat \phi:Q(\bD)\to \bE$ such that the triangle $$\xymatrix{&Q(\bD)\ar[dr]^{\hat \phi}\ar[dl]_{p_{\bD}}&\\\bD\ar[rr]^{\phi}&&\bE}$$ commutes up to a 
unitary  natural multiplier  isomorphism between weakly equivariant functors.
\end{lem}
\begin{proof}
Note that the composition of weakly equivariant morphisms
$\phi\circ p_{\bD}$ exists. 
We define the morphism $\hat \phi$ as follows:
\begin{enumerate}
\item objects: For every object $(D,g)$ of $Q(\bD)$ we set 
$\hat \phi(D,g):=g^{-1}\phi(g D)$.
\item morphisms:  For every morphism  $f:(D,g)\to (D',g')$ in $Q(\bD)$ we set 
$\hat \phi(f) := \rho(g)_{D'} \phi(f) \rho(g)_{D}^{-1}$. 
Since $\bE$ is an ideal in $\bM\bE$ this formula defines a morphism in $\bE$. \end{enumerate}
One checks that 
$\hat \phi$ is an equivariant morphism. 
Further, the unitary multiplier isomorphism $\phi\circ p_{\bD}\to \hat \phi$ filling the triangle is given by the family
$(\rho(g)_{(D,g)})_{(D,g)\in \Ob(Q(\bD))}$.
 \end{proof}
 
\section{Orthogonal sums in  \texorpdfstring{$\bm{C^{*}}$}{C-star}-categories}\label{qroifjqerofwefewfefwqfe}

The notion of a finite  orthogonal  sum of objects  {in a unital $C^{*}$-category} can be defined in the standard way. %for every unital $\C$-linear $*$-category. 
 In the present section we are mainly interested in infinite sums of objects in $C^{*}$-categories. 
 After briefly recalling the finite case (see e.g.\ \cite{MR3123758}) we introduce our notion of an orthogonal sum of an arbitrary family of objects in a unital $C^{*}$-category. A posteriori it is equivalent to the concept introduced in \cite{fritz}, see Remark \ref{euwifhqiufqwfewffqewffqfef}.
 
 %In the case of the category of Hilbert $C^{*}$-modules over  a $C^{*}$-algebra we check that our definition of an orthogonal sum is equivalent to the classical definition.}
 
 %Nevertheless we start with the finite case \cite{MR3123758}.

%\subsection{\Alex{Finite orthogonal sums}}

Let $\bC$ be in $\nClincat$, and let $(e_{i})_{i\in I}$ be a family of morphisms $C_{i}\to C$ in $\bC$  with the same target.
\begin{ddd}\label{wtihojwegwergwergef}
The family $(e_{i})_{i\in I}$ is mutually orthogonal if for all $i,j$ in $I$ with $i\not=j$ we have $e_{j}^{*}e_{i}=0$.
\end{ddd}

% In this section we fix a $C^{*}$-category $\bC$. We  define the notions of a  projection and the  image of a projection in $\bC$. 
% We furthermore introduce the notion of  a direct sum of a family of objects of $\bC$ and discuss various properties of this notion. 
%For infinite families this is not absolutely trivial. 

%Let $\bC$ be in $\nClincat$. 
Let $(C_{i})_{i\in I}$ be a finite family  {of} objects in $\bC$.
\begin{ddd}\label{regiuhqrogefewfqwfqef}
 An  orthogonal  sum of the family $(C_{i})_{i\in I}$ is a pair $(C,(e_{i})_{i\in I})$ of an object $C$ in $\bC$ and a family of isometries $e_{i}\colon C_{i}\to C$ such that:
\begin{enumerate}
\item\label{asfvasvadsvsdvasdvisddddda} The family $(e_{i})_{i\in I}$ is mutually orthogonal. % we have    $e^{*}_{j}e_{i}=\left\{\begin{array}{cc}\id_{C_{i}}&i=j\\0&i\not=j\end{array}\right.  \ .$
\item\label{asfvasvadsvsdvasdvisddddda1} $\sum_{i\in I} e_{i}e_{i}^{*}=\id_{C}$.
 \end{enumerate}
 \end{ddd}

Note that by this definition only families of unital objects can admit orthogonal sums.
The sum of such a family is also unital. Since any morphism in $\Ccat$ is unital and linear on morphism spaces it preserves orthogonal sums.

\begin{ex}
%If $\bC$ is not empty, then 
{The} orthogonal sum of {an} empty family is a zero object.
\hB
\end{ex}

 If $(C,(e_{i})_{i\in I})$ is an orthogonal sum of the finite  family $(C_{i})_{i\in I}$, then it represents the categorical coproduct of the family  $(C_{i})_{i\in I}$.
 The pair $(C,(e_{i}^{*})_{i\in I})$ represents the   categorical product 
 of the family  $(C_{i})_{i\in I}$. In particular, an orthogonal sum is uniquely determined   up to unique isomorphism. By the following lemma this isomorphism is actually unitary.
% The following lemma qualifies this isomorphism to be unitary.}
 
 \begin{lem}
 An orthogonal sum of a finite family  is unique up to unique unitary isomorphism.
 \end{lem}
 
 \begin{proof}
 Let  $(C,(e_{i})_{i\in I})$ and  $(C',(e'_{i})_{i\in I})$ be two  orthogonal  sums of the finite family  $(C_{i})_{i\in I}$. 
 Then $v \coloneqq \sum_{i\in I} e_{i}'e_{i}^{*}\colon C\to C'$ is the unique unitary isomorphism such that $ve_{j}=e_{j}'$ for all $j$ in $I$.
 %{One easily checks by a calculation that it is unitary.}
 %Indeed, using Condition~\ref{asfvasvadsvsdvasdvisddddda}  {of Definition~\ref{regiuhqrogefewfqwfqef}} we get  $ve_{j}=\sum_{i\in I} e_{i}'e_{i}^{*} e_{j}=e'_{j}$, and using Conditions~\ref{asfvasvadsvsdvasdvisddddda} and  \ref{asfvasvadsvsdvasdvisddddda1} we get
%Using that $(e'_{i})_{i\in I}$ is mutually orthogonal by Condition \ref{regiuhqrogefewfqwfqef}.\ref{asfvasvadsvsdvasdvisddddda}  and the Condition \ref{regiuhqrogefewfqwfqef}.\ref{asfvasvadsvsdvasdvisddddda1} for    $(e_{i})_{i\in I}$  we get 
 %$$v^{*}v=\sum_{i,j\in I} e_{j} e_{j}^{\prime,*}e_{i}'e_{i}^{*}=\sum_{i\in I} e_{i} e_{i}^{*}=\id_{C}\, .$$
% Similarly  we show $vv^{*}=\id_{C'}$.
 %For uniqueness,   if $v'\colon C\to C'$ is a morphism such that $v'e_{i}=e_{i}'$ for all $i$ in $I$, then we have
% $v'= \sum_{i\in I} v'e_{i}e_{i}^{*}= \sum_{i\in I} e'_{i}e_{i}^{*}=v$.
 \end{proof}

Let $\bC$ be in $\nClincat$.
\begin{ddd}\label{ergiehjioferfqffrf}
$\bC$ is additive if it admits  orthogonal  sums for all finite families of objects.
\end{ddd}

If $\bC$ is additive, then it  is unital since it must admit sums of all one-member families.
 
\begin{ex}
If $A$ is a very small $C^{*}$-algebra, then the full subcategory $\Hilb^{\fg}(A)$ of $\Hilb(A)$ of finitely generated Hilbert $A$-modules  is additive.
\hB
\end{ex}

%\subsection{\Alex{Infinite orthogonal sums}} 
 
For the discussion of infinite orthogonal  sums we specialize to unital $C^{*}$-categories. % \Alex{to discuss infinite orthogonal sums.}
%For simplicity, 
%We will consider sums of infinite families only in unital $C^{*}$-categories. 
Let $\bC$ be in $\Ccat$.
%The following construction prepares the definition of an  orthogonal sum for infinite families.  

%Note that a  $C^{*}$-category is enriched in Banach spaces. We let $\Ban$ denote the category of small complex Banach spaces and continuous linear maps. An object $C$ in $\bC$ (co)represents functors
%$$ \Hom_{\bC}(-,C)\colon \bC^{\op}\to \Ban \, , \quad  \Hom_{\bC}(C,-)\colon \bC\to \Ban\, .$$

 Let $(C_{i})_{i\in I}$ be a family of objects in $\bC$ and $(C,(e_{i})_{i\in I})$ be a pair consisting of
  an object $C$ of $\bC$ and 
 a mutually orthogonal family  of isometries $e_{i}\colon C_{i}\to C$.
 %such that \begin{equation}\label{asfvasvadsvsdvasdva}e^{*}_{j}e_{i}=\left\{\begin{array}{cc}\id_{C_{i}}&i=j\\0&i\not=j\end{array}\right.  \end{equation} for all $i,j$ in $I$.
Using this data we are going to define  two subfunctors
\[
\mathbb{K}(-,C)\colon \bC^{\op}\to \Ban   \, , \quad \mathbb{K}(C,-)\colon \bC\to\Ban 
\]
of $\Hom_{\bC}(-,C)$, or of $\Hom_{\bC}(C,-)$, respectively.
 
 Let $D$ be an object of $\bC$. A morphism $f\colon D\to C$ of the form $f=e_{i}\tilde f$ for some morphism $\tilde f\colon D\to C_{i}$ is called a generator for  $\mathbb{K}(D,C)$.  
Similarly, a morphism $f'\colon C\to D$ of the form $f'=f'_{i}e_{i}^{*}$ for some morphism $f'_{i}\colon C_{i}\to D$ is a called a generator for $\mathbb{K}(C,D)$. 

 A finite linear combination of generators will be called finite. One checks that the subspaces of finite morphisms form subfunctors 
 $   \Hom^{\fin}_{\bC}(-,C)$ and  $ \Hom^{\fin}_{\bC}(C,-) $ of   $\Hom_{\bC}(-,C)$ and  $\Hom_{\bC}(C,-)$, respectively, considered as $\C$-vector space valued functors.

% For every object $D$ of $\bC$ we have 
% embeddings  \begin{equation}\label{ergohioqegfdbfdsv} \bigoplus_{i\in I } \Hom_{\bC}(D,C_{i})\to \Hom_{\bC}(D,C)\ , \quad \bigoplus_{i\in I} \Hom_{\bC}(C_{i},D)\to \Hom_{\bC}(C,D)\end{equation}
%sending $f$ in $ \Hom_{\bC}(D,C_{i})$ to $e_{i}f $ in $\Hom_{\bC}(D,C)$, or $f'$ in $ \Hom_{\bC}(C_{i},D)$ to $fe_{i}^{*}$, in $\Hom_{\bC}(C,D)$, respectively.  
\begin{ddd} 
 We define the subfunctors
\[
\mathbb{K}(-,C)\colon \bC^{\op}\to \Ban\ , \quad 
  \mathbb{K}(C,-)\colon \bC\to \Ban
\]
of $\Hom_{\bC}(-,C)$  and  $\Hom_{\bC}(C,-)$ 
  by taking objectwise the norm closures of $   \Hom^{\fin}_{\bC}(-,C)$ and  $ \Hom^{\fin}_{\bC}(C,-) $.
\end{ddd}
  
 In order to see that these subfunctors are well-defined we use
  the {sub-multiplicativity of the norm on $\bC$ in order to check} % inequalities of the form $\|f\circ f'\|\le \|f\|\|f'\|$ 
  these subspaces are preserved by precompositions or 
postcompositions with morphisms in $\bC$, respectively. 
%We conclude   that they  form   subfunctors. 
The involution of $\bC $ provides  an antilinear isomorphism between 
$\mathbb{K}(C,D)$ and $
\mathbb{K}(D,C)$ for all $D$ in $\bC$.

Note that $\mathbb{K}(C,D)$ and $
\mathbb{K}(D,C)$ depend on $C$ and the family $(e_{i})_{i\in I}$.
 If we want to stress  the dependence of these subspaces on the family $(e_{i})_{i\in I}$, then  we will write $\mathbb{K} ((C,(e_{i})_{i\in I}),D)$ and $\mathbb{K}(D,(C,(e_{i})_{i\in I}))$.  This notation will in particular be used in order  to avoid confusion if we want to consider the case where $D=C$. 
 
\begin{ex}
For $i$ in $I$ the morphism $e_{i}\colon C_{i}\to C$ belongs to $\mathbb{K}(C_{i},C)$. It is actually a generator.  Similarly, $e_{i}^{*}$ is a generator in $\mathbb{K}(C,C_{i})$. In fact, $\mathbb{K}(-,C)$ is the smallest $\Ban$-valued subfunctor of $\Hom_{\bC}(-,C)$ whose value on $C_{i}$ contains $e_{i}$ for every $i$ in $I$. Similarly, $\mathbb{K}(C,-)$  is the smallest $\Ban$-valued subfunctor of $\Hom_{\bC}(C,-)$ whose value on $C_{i}$ contains $e_{i}^{*}$ for every $i$ in $I$.
\hB
\end{ex}

 \begin{ex}
 Let  $A$ be a very small $C^{*}$-algebra and consider the $C^{*}$-category $\Hilb(A)$ of Hilbert $A$-modules and continuous adjointable operators. Let $C$ be in $\Hilb(A)$ and assume that
 $(e_{i})_{i\in I}$ is a mutually orthogonal  family of isometries $e_{i}\colon C_{i}\to C$. Furthermore, assume that the
 images of  the morphisms $e_{i}$  together generate $C$ as an Hilbert $A$-module. Then we have inclusions 
 \begin{equation}\label{eroiheoivjoevsef}
K(D,C)\subseteq \mathbb{K}(D,C)\, , \quad K(C,D)\subseteq \mathbb{K}(C,D)\, ,
\end{equation}
 where  $K(D,C)$ and $K(C,D)$ denote the spaces of all compact operators (in the sense of Hilbert $A$-modules) from $D$ to $C$ and vice versa.
{In general, these inclusions are proper. But 
 % Then $\mathbb{K}(D,C)$ and $\mathbb{K}(C,D)$ are the spaces $K(D,C)$ and $K(C,D)$ of all compact operators (in the sense of $A$-Hilbert $C^{*}$-modules) from $D$ to $C$ and vice versa. This motivates the notation. 
% 
if }$\id_{C_{i}}$ is compact   for every $i$ in $I$, then the  inclusions  in \eqref{eroiheoivjoevsef} are equalities.
\hB
\end{ex}

Let $\cC$ be a category and consider   two functors  $F,F'\colon \cC\to \Ban$. %be , and let  $u\colon F\to F'$ be a natural transformation. The latter is given by a family $u=(u_{C})_{C\in \Ob(\cC)}$ of continuous maps $u_{C}\colon F(C)\to F'(C)$.
%\begin{ddd}\label{woiegbfdsb9}
%The transformation $u$ is called uniformly bounded if 
 %$$\|u\|\coloneqq \sup_{C\in \cC} \|u_{C}\|_{\Hom_{\Ban}(F(C),F'(C))}<\infty\, .$$
%\end{ddd}
We let $\Hom^{\bd}_{\Fun(\cC,\Ban)}(F,F')$ denote the Banach space of 
%$\Hom_{\Fun(\cC,\Ban)}(F,F')$ 
%of 
all uniformly bounded natural transformations.
%One easily checks that
%$\Hom^{\bd}_{\Fun(\cC,\Ban)}(F,F')$ is a Banach space with norm $\|-\|$.
We let $(C,(e_{i})_{i\in I})$ be as above and fix an object $D$.

\begin{ddd}\label{eriguqegwfwefqwefqwef}\mbox{}
\begin{enumerate}
\item\label{qrgjuiergqwfwefqwfqwefq} 
The Banach space  of   right multipliers  from $D$ to $C$ is defined by
$$\RM(D,C)\coloneqq \Hom^{\bd}_{\Fun(\bC,\Ban)}(\mathbb{K}(C,-), \Hom_{\bC}(D,-))\, .$$
\item 
The Banach space  of   left multipliers  from $C$ to $D$ is defined by
$$\LM(C,D)\coloneqq \Hom^{\bd}_{\Fun(\bC^{\op},\Ban)}(\mathbb{K}(-,C), \Hom_{\bC}(-,D))\, .$$
 \end{enumerate}
\end{ddd}

If we want to stress  the dependence of these objects on the family $(e_{i})_{i\in I}$ or insert $C$ in place of $D$,  then we will write  
$\RM (D,(C,(e_{i})_{i\in I}))$ and $\LM ((C,(e_{i})_{i\in I}),D)$. 

A right multiplier in $\RM(D,C)$ is given by a uniformly bounded  family $R=(R_{D'})_{D'\in \Ob(\bC)}$ of  {bounded linear} maps $R_{D'}\colon \mathbb{K}(C,D') \to \Hom_{\bC}(D,D')$ satisfying the conditions for a natural transformation. {In particular,}  for a morphism $f$ in $\Hom_{\bC}(D',D'')$ and $g$ in $\mathbb{K}(C,D')$ we have $R_{D''}(fg) = f R_{D'}(g)$. We use a similar notation  $(L_{D'})_{{D'\in \Ob(\bC)}}$ for left multipliers. The involution of $\bC$ induces an antilinear isomorphism between
$\RM(D,C)$ and $\LM(C,D)$.

Let $(C_{i})_{i\in I}$ be a family of objects of $\bC$, and let $(h_{i})_{i\in I}$ and $(h_{i}')_{i\in I}$ be families of morphisms $h_{i} \colon D\to C_{i}$ and $h_{i}' \colon C_{i}\to D$ in $\bC$.

\begin{ddd}\label{oijfqoifweeqfewfqefqfewf}\mbox{}
We say that $(h_{i})_{i\in I}$  is square summable if
\begin{equation} \label{eq_condition_multiplier}
\sup_{J\subseteq I} \big\|\sum_{i\in J} h_{i}^{*}h_{i}\big\| < \infty  \, ,
\end{equation} where $J$ runs over the {set of} finite subsets of $I$.
 We say that  $(h_{i}')_{i\in I}$   is square summable if
$(h_{i}^{\prime,*})_{i\in I}$ is square summable
%
% \begin{equation}
%\label{eq_condition_multiplier_2}
%\sup_{J\subseteq I} \big\|\sum_{i\in J} h^\prime_{i} h^{\prime,*}_{i}\big\|<\infty\, ,
%\end{equation} where $J$ runs over the {set of} finite subsets of $I$.
 %or (\begin{equation}
%\label{eq_condition_multiplier_2}
%\sup_{J\subseteq I} \|\sum_{i\in J} h^\prime_{i} h^{\prime,*}_{i}\|<\infty
%\end{equation}
%)
%\end{enumerate}
\end{ddd}
 
 \begin{lem}\label{qoirgujoewgwergwergw}
 %\mbox{}\begin{enumerate}
% \item\label{rgqhuiegegwegegwregwergg}
 If   the family $(h_{i}^{*})_{i\in I}$ is mutually orthogonal (see Definition \ref{wtihojwegwergwergef}) and  uniformly bounded, {then $(h_{i})_{i\in I}$ is square summable.} 
% \item\label{rgqhuiegegwegegwregwergg1}  If   the family $(h'_{i})_{i\in I}$ is mutually orthogonal and  uniformly bounded, {then it is square summable.}  \end{enumerate}
 \end{lem}
\begin{proof}
%We show Assertion \ref{rgqhuiegegwegegwregwergg}. Assertion \ref{rgqhuiegegwegegwregwergg1} can then be deduced using  the involution.
Let $J$ be a finite subset of $I$.
Using the fact that $(h^{*}_{i})_{i\in I}$ is mutually orthogonal we calculate that for every $k$ in $\nat$  
\[\big(\sum_{i\in J} h_{i}^{*}h_{i}\big)^{2^{k}} = \sum_{i\in J} (h_{i}^{*}h_{i})^{2^{k}}\,.\]
 Using repeatedly the $C^{*}$-equality, that $\sum_{i\in J} h_{i}^{*}h_{i}$ is self-adjoint, and the triangle inequality for the norm
 we get
\begin{align*}
\big\|\sum_{i\in J} h_{i}^{*}h_{i}\big\|^{2^{k}} & = \big\|\big(\sum_{i\in J} h_{i}^{*}h_{i}\big)^{2^{k}}\big\|
 = \big\|\sum_{i\in J} (h_{i}^{*}h_{i})^{2^{k}}\big\|\\
 & \le \sum_{i\in J} \|h_{i}^{*}h_{i}\|^{2^{k}}
 \le |J| \max_{i\in J}  \|h_{i}^{*}h_{i}\|^{2^{k}}\,.
\end{align*}
We take the $2^{k}$-th root and form the limit for $k\to \infty$. Using that $\lim_{k\to \infty} |J|^{\frac{1}{2^{k}}}=1$ we get 
$$\big\|\sum_{i\in J} h_{i}^{*}h_{i}\big\|\le \max_{i\in J}  \|h_{i}^{*}h_{i}\| \le \sup_{i\in I} \|h_{i}\|^{2}\, .$$
Since the right-hand side is finite by assumption and does not depend on $J$, {we conclude the square summability of $(h_{i})_{i\in I}$.} 
 \end{proof}
%The associed multiplier maps $m_{D}$ and $n_{D}$ are linear and 
%sending $h'$ in $ \Hom_{\bC}(C,D)$ to $L(h')$ in $LM(C,D)$ such that
%$L(h')_{D'}(f):=h'f$  for all $f$ in $ \mathbb{K}(D',C)$.   

%We note that \eqref{wergwergwggwegegr42t} and \eqref{wergwergwggwegegr42t1} are natural transformations of $\Ban$-valued functors in the variable $D$, which are $\IC$-linear and non-expansive. %\fuli{habe Erkl\"arung auskommentiert weil sie nicht-definierte Notation enth\"alt. Sagt man wirklich non-expansive und nicht contractive?}
%, i.e., $\|R(h)\| \le \|h\|$ and $\|L(h)\| \le \|h\|$ for every morphism $h$.

The following lemma provides a tool to construct interesting multipliers.  {Let $(C,(e_{i})_{i\in I})$ be a pair consisting of an object $C$ of $\bC$ and a mutually orthogonal family  of  isometries $ e_{i}\colon C_{i}\to C$.}
Let $h:=(h_{i})_{i\in I}$ be a family of morphisms $h_{i}\colon D\to C_{i}$, and let $h' \coloneqq (h'_{i})_{i\in I}$ be a family of morphisms $h'_{i}\colon C_{i}\to D$. 
\begin{lem}\label{qergijoqregqewfewfqewfwef}\mbox{}
\begin{enumerate} 
\item\label{qerjgqkrfewfewfqfefqef}
There exists a  right-multiplier $R(h)$ in $\RM(D,C)$ with $R(h)_{C_{i}}(e^{*}_{i})=h_{i}$ for all $i$ in $I$ if and only if  {$h$ is square summable.}
%\begin{equation}
%\label{eq_condition_multiplier}
%\sup_{J\subseteq I} \|\sum_{i\in J} h_{i}^{*}h_{i}\| < \infty
%\end{equation}
%(where $J$ runs over the {set of} finite subsets of $I$).

{If  $h$ is square summable}, %\eqref{eq_condition_multiplier} is satisfied, 
then $R(h)$ is uniquely determined {and its norm is given by}
\begin{equation}\label{eq_multiplier_norm}
 {\|R(h)\| = \sqrt{\sup_{J\subseteq I} \|\sum_{i\in J} h_{i}^{*}h_{i}\|}\,.}
\end{equation}
%Moreover, if the family $(h_{i}^*)_{i\in I}$ is mutually orthogonal and $\sup_{i\in I} \|h_{i}\|<\infty$, {then it is square summable.} %\eqref{eq_condition_multiplier} is satisfied.

\item\label{qerjgqkrfewfewfqfefqef1}
 There exists a   left-multiplier $L(h')$ in $\LM(C,D)$ with $L(h')_{C_{i}}(e_{i})=h'_{i}$ for all $i$ in $I$ if and only if $h'$ {is square summable.}
%\begin{equation}
%\label{eq_condition_multiplier_2}
%\sup_{J\subseteq I} \|\sum_{i\in J} h^\prime_{i} h^{\prime,*}_{i}\|<\infty
%\end{equation}
%(where $J$ runs over the  {set of} finite subsets of $I$).

{If  $h'$ is square summable}, then $L({h^\prime})$ is uniquely determined  {and its norm is given by}
\begin{equation}\label{eq_multiplier_norm_L}
 {\|L(h^\prime)\| = \sqrt{\sup_{J\subseteq I} \|\sum_{i\in J} h^\prime_{i} h^{\prime,*}_{i}\|}\,.}
\end{equation}
%Moreover,
%If  $h':=(h'_{i})_{i\in I}$ is a family of morphisms $h'_{i}\colon C_{i}\to D$ such that
%\Alex{
%\begin{equation}
%\label{eq_condition_multiplier_2}
%\sup_{J\subseteq I} \|\sum_{i\in J} h^\prime_{i} h^{\prime,*}_{i}\|<\infty
%\end{equation}
%(where $J$ runs over the finite subsets of $I$),}
%%$\sup_{i\in I} \|h'_{i}\|<\infty$,
%then there is a unique left-multiplier $L(h')$ in $\LM(C,D)$ such that $L(h')_{C_{i}}(e_{i})=h'_{i}$ for all $i$ in $I$.
%if the family $(h^\prime_{i})_{i\in I}$ is mutually orthogonal and $\sup_{i\in I} \|h^\prime_{i}\|<\infty$, {then it is square summable.} %\eqref{eq_condition_multiplier_2} is satisfied.
\end{enumerate}
\end{lem}
Note that the mere existence of the multiplier with the indicated property implies  the  square summability of the corresponding family. In turn it follows the multiplier is actually uniquely determined by the property.  
\begin{proof}
It suffices to prove Assertion~\ref{qerjgqkrfewfewfqfefqef}. Assertion \ref{qerjgqkrfewfewfqfefqef1} then follows from  \ref{qerjgqkrfewfewfqfefqef} using  the  involution. 

We first assume that $h$ is square summable.  % Condition~\eqref{eq_condition_multiplier}. 
Let $J$ be a finite subset of $I$ and consider
\begin{equation}\label{eq_partial_multiplier}
R^{J}(h) \coloneqq \sum_{j\in J} e_{j}h_{j}
\end{equation}
in $\Hom_{\bC}(D,C)$. Right composition with $R^{J}(h)$  provides a right-multiplier  
$(R^{J}(h)_{D'})_{D'\in \Ob(\bC)}$ such that $R^{J}(h)_{D'}$  sends $f$ in $\mathbb{K}(C,D')$ to $
 \sum_{j\in J} f e_{j}h_{j}$ in $\Hom_{\bC}(D,D')$.    
  In particular, we have $R^{J}(h)_{C_{i}}(e_{i}^{*})=h_{i}$ provided $i\in J$.

We now show that $\lim_{J\subseteq I} R^{J}(h)$ exists pointwise on finite morphisms and defines the required multiplier $R(h)$ by continuous extension. Here the limit is taken over the filtered poset of finite subsets of $I$.
So let $D'$ be in $\Ob(\bC)$ and let $f$ be in  $\Hom^{\fin}_{\bC}(C,D')$.   Then $J\mapsto  R^{J}(h)_{D'}(f)$ is eventually constant and the limit  $\lim_{J\subseteq I} R^{J}(h)_{D'}(f)$ exists.    
We get a natural transformation 
$$\lim_{J\subseteq I}R^{J}(h) \colon \Hom_{\bC}^{\fin}(C,-)\to \Hom_{\bC}(D,-)$$
of $\C$-vector space valued functors.  
We now argue that $ \lim_{J\subseteq I}R^{J}(h)$ extends by continuity to a natural transformation
$$R(h)\colon \mathbb{K}(C,-)\to  \Hom_{\bC}(D,-)\, .$$  
To this end it suffices to verify that $  \lim_{J\subseteq I}R^{J}(h)$ is  bounded  for any $D'$ in $\Ob(\bC)$ separately. 
Let $f$ be in $\Hom^{\fin}(C,D')$.
Then   for sufficiently large finite subsets  $J_{f}$ of $I$   we calculate  using the sub-multiplicativity of the norm under composition, the $C^{*}$-equality for the norm on a $C^{*}$-category,  the mutual orthogonality of the family $(e_{i})_{i\in I}$,  and that $e_i^* e_i = \id_{C_i}$ for every $i$ in $I$ that
%\begin{eqnarray*}
%\| \lim_{J\subseteq I}R^{J}(h)_{D'}(f)\|^{2}&=&\| f\sum_{j\in J_{f}}   e_{j}h_{j}\|^{2}\le \|f\|^{2} \cdot \|\sum_{j\in J_{f}}   e_{j}h_{j}\|^{2}\\
%&=&\|f\|^{2} \cdot  \|( \sum_{i\in J_{f}}    e_{i}h_{i})^{*}  (\sum_{j\in J_{f}}e_{j}h_{j})  \|=\|f\|^{2}  \cdot \|\sum_{i\in J_{f}}    \Alex{h_{i}^*e_{i}^*e_{i}h_{i}} \|\\
%&\le& \Alex{\|f\|^2  \cdot \|\sum_{i\in J_{f}} h_{i}^* h_{i} \|}\\
%&\le& \Alex{\|f\|^2  \cdot \sup_{J \subset I} \|\sum_{i\in J} h_{i}^* h_{i} \|}\,,
%\end{eqnarray*}
\begin{align}
\|\lim_{J\subseteq I}R^{J}(h)_{D'}(f)\|^{2} & = \big\| f\sum_{j\in J_{f}}   e_{j}h_{j}\big\|^{2}\label{eq_estimate_norm_Rmult_onesided}
 \le \|f\|^{2} \cdot \big\|\sum_{j\in J_{f}}   e_{j}h_{j}\big\|^{2}\\
 & = \|f\|^{2} \cdot  \big\|\big( \sum_{i\in J_{f}}    e_{i}h_{i}\big)^{*}  \big(\sum_{j\in J_{f}}e_{j}h_{j}\big)  \big\|
 = \|f\|^{2}  \cdot \big\|\sum_{i\in J_{f}}     {h_{i}^*e_{i}^*e_{i}h_{i}} \big\|\notag\\
 & = \|f\|^2  \cdot \big\|\sum_{i\in J_{f}} h_{i}^* h_{i} \big\|
 \le \|f\|^2  \cdot \sup_{J \subseteq I} \big\|\sum_{i\in J} h_{i}^* h_{i} \big\|\,,\notag
\end{align}
where the supremum runs over all finite subsets of $I$. Using the assumption of square summability of  the family $(h_{i})_{i\in I}$  and   \eqref{eq_condition_multiplier} it follows that
% Because $\sup_{i\in I}\|h_{i}\|$ is finite by assumption we see that
$\lim_{J\subseteq I}R^{J}(h)_{D'}$ is bounded.
Since the right-hand side does not depend on $D'$  we further see that $\lim_{J\subseteq I}R^{J}(h)$ is uniformly bounded.  The above estimate also implies that
\begin{equation}\label{wergoijoiwergwegwegr}
\|R(h)\|^2 \le \sup_{J \subseteq I} \big\|\sum_{i\in J} h_{i}^* h_{i}\big\|\, .
\end{equation}  
The converse estimate implying the equality \eqref{eq_multiplier_norm} will
be shown below while proving the converse to the existence statement.

We now assume the existence of a right-multiplier $R(h)$ in $\RM(D,C)$ with $R(h)_{C_{i}}(e^{*}_{i})=h_{i}$ for all $i$ in $I$ and verify that {$h$ is square summable.} %\eqref{eq_condition_multiplier} is satisfied. 
So let $J$ be a finite subset of $I$. Then
\[
R(h)_D \big( \sum_{i \in J} h_i^* e_i^* \big) = \sum_{i \in J} h_i^* R(h)_{C_i}(e_i^*) = \sum_{i \in J} h_i^* h_i
\]
and consequently
\[
\big\| \sum_{i \in J} h_i^* h_i \big\| = \big\| R(h)_D \big( \sum_{i \in J} h_i^* e_i^* \big) \big\| \le \| R(h) \| \cdot \big\| \sum_{i \in J} h_i^* e_i^* \big\| \,.
\]
Using the involution we get a left-multiplier $R(h)^*$ in $\LM(C,D)$ with $R(h)^*_{C_i}(e_i) = h_i^*$ for all $i$ in $I$. It satisfies
\[
\sum_{i \in J} h_i^* e_i^* = \sum_{i \in J} R(h)^*_{C_i}(e_i) e_i^* = \sum_{i \in J} R(h)^*_{C}(e_i e_i^*) = R(h)^*_{C} \big( \sum_{i \in J} e_i e_i^* \big)
\]
and consequently
\[
\big\| \sum_{i \in J} h_i^* e_i^* \big\| = \big\| R(h)^*_C \big( \sum_{i \in J} e_i e_i^* \big) \big\| \le \| R(h)^* \| \cdot \big\| \sum_{i \in J} e_i e_i^* \big\| \le \| R(h)^* \|\,,
\]
where the last inequality sign holds because $(e_i e_i^*)_{i \in I}$ is a mutually orthogonal family of projections. Putting all together, we conclude the inequality
$$\big\|\sum_{i\in J} h_{i}^{*}h_{i}\big\| \le \| R(h) \|^2$$
for every finite subset $J$ of $I$. This implies {square summability of $h$.} %\eqref{eq_condition_multiplier}.
Applying $\sup_{J}$ we get the 
%Furthermore, since $h_{i}^{*}h_{i}$ is a positive operator for every $i$ in $I$ it also implies the 
converse inequality to  \eqref{wergoijoiwergwegwegr} which finishes  the {verification of the equality} \eqref{eq_multiplier_norm}.

We now assume that the family $h$ is square summable  and $R'$ is a multiplier with $R'_{C_{i}}(e_{i}^{*})=h_{i}$ for all $i$ in $I$. Then we have $(R(h)_{C_{i}}-R'_{C_{i}})(e_{i}^{*})=0$ and hence
$(R(h)-R)(f)=0$ for all finite morphisms $f$. By continuity of $R(h)$ and $R'$  we conclude that $R(h)=R'$.
%\Alex{It remains to prove that if the family $(h_{i}^*)_{i\in I}$ is mutually orthogonal and $\sup_{i\in I} \|h_{i}\|<\infty$, then \eqref{eq_condition_multiplier} is satisfied. Let $J$ be a finite subset of $I$. We form the unital $C^\ast$-algebra
%\begin{equation*}
%A(J):=\bigoplus_{E,E'\in \{D,C_i\colon i \in J\}} \Hom_{\bC}(E,E')\,.
%\end{equation*}
%We choose a faithful representation of $A(J)$ on a Hilbert space $H$. Since the family $(h_{i}^*)_{i\in J}$ is mutually orthogonal, the family $(h_{i}^*h_i)_{i\in J}$ is also mutually orthogonal and it therefore becomes a family of operators on $H$ with mutually orthogonal ranges. Since for each $i$ in $I$ the element $h_{i}^*h_i$ is self-adjoint, those operators on $H$ also have mutually orthogonal co-ranges. This implies by basic operator theory that
%\[
%\|\sum_{i\in J} h_{i}^{*}h_{i}\| = \max_{i \in J} \|h_{i}^{*}h_{i}\|\,.
%\]
%Since $\|h_{i}^{*}h_{i}\| = \|h_i\|^2$ by the $C^*$-identity, the claim follows.
%}
%&\end{proof}
%estimate \eqref{eq_condition_multiplier}.
\end{proof}

Let $(C,(e_{i})_{i\in I})$ be a pair consisting of an object $C$ of $\bC$ and a mutually orthogonal family  of  isometries $ e_{i}\colon C_{i}\to C$.
The following maps will play a crucial role in the characterization of infinite orthogonal sums. They send morphisms to the corresponding multipliers.
\begin{ddd}
{For every object $D$ in $\bC$  we define the associated right multiplier map}
\begin{eqnarray}\label{wergwergwggwegegr42t}
\lefteqn{{m^{R}_{D}} \colon \Hom_{\bC}(D,C) \to \Hom^{\bd}_{\Fun (\bC,\Ban)}(\Hom_{\bC}(C,-),\Hom_{\bC}(D,-))}&&\\&&\quad \to   \Hom^{\bd}_{\Fun (\bC,\Ban)}( \mathbb{K}(C,-),\Hom_{\bC}(D,-))= \RM(D,C)\, .  \nonumber
\end{eqnarray}
%It sends $h$ in $ \Hom_{\bC}(D,C)$ to $R(h)$ in $RM(D,C)$ such that
%$R(h)_{D'}(f):=fh$  for all  $f$ in $ \mathbb{K}(C,D')$.   
Similarly we define the associated left multiplier map
\begin{eqnarray}\label{wergwergwggwegegr42t1}
\lefteqn{{m^{L}_{D}} \colon \Hom_{\bC}(C,D)\to  \Hom^{\bd}_{\Fun (\bC^{\op},\Ban)}(\Hom_{\bC}(-,C),\Hom_{\bC}(-,D))}&&\\ && \quad \to   \Hom^{\bd}_{\Fun (\bC^{\op},\Ban)}(\mathbb{K}(-,C),\Hom_{\bC}(-,D))= \LM(C,D)\, .  \nonumber
\end{eqnarray}
\end{ddd}
We have
\begin{equation}\label{vjfjosjvvsfvsvsfvsfv}
m^{R} \coloneqq (m^{R}_{D})_{D\in \Ob(\bC)}\in \Hom^{\bd}_{\Fun(\bC^{\op},\Ban)}(\Hom_{\bC}(-,C), {\RM}(-,C)) 
 \, , \quad \|m^{R}\|\le 1
\end{equation}
and 
\begin{equation}\label{vjfjosjvvsfvsvsfvsfv1}
m^{L} \coloneqq (m^{L}_{D})_{D\in \Ob(\bC)}\in \Hom^{\bd}_{\Fun(\bC,\Ban)}(\Hom_{\bC}(C,-), {\LM}(C,-))\, , \quad  \|m^{L}\|\le 1\, ,
\end{equation} 
  {where the norm estimate  follows from the sub-multiplicativity of the norm on $\bC$.  %If we want to stress the dependence of the objects on $(C,(e_{i})_{i\in I})$, then we will also use the notation $m^{(C,(e_{i})_{i\in I})}$ for $m$ or $ n^{(C,(e_{i})_{i\in I})}$ for $n$.}

We can now define the notion of an orthogonal sum of a family $(C_{i})_{i\in I}$ of objects in $\bC$.
\begin{ddd}\label{erogwfsfdvbsbfdbsdfbsfdbv}
An  orthogonal sum of the family $(C_{i})_{i\in I}$ is a pair $(C,(e_{i})_{i\in I})$ of an object $C$ in $\bC$ together with a  {mutually} orthogonal  family of  isometries 
$e_{i}\colon C_{i}\to C$ such that the  associated 
 multiplier transformations 
    \eqref{wergwergwggwegegr42t} and \eqref{wergwergwggwegegr42t1} are  bijective for any object $D$ of~$\bC$.
%\label{weiobweervwerv}   The spaces  $  e_{i'}\Hom_{\bC}(C_{i},C_{i'})p_{i}$ for all $i,i'$ in $I$
%  generate an ideal  $\mathbb{K}(\bigoplus_{i\in I}C_{i})$ of $\End_{\bC}(\bigoplus_{i\in I}C_{i})$.
%  \item \label{ergiuwegegergwreg}The induced morphism 
%  $$ \End_{\bC}(\bigoplus_{i\in I}C_{i})\to M(\mathbb{K}(\bigoplus_{i\in I}C_{i}))$$ is an isomorphism.
%  \item  We have
%   $\sum_{i\in I} e_{i}p_{i}=1$ in the multiplier algebra $M(\mathbb{K}(\bigoplus_{i\in I}C_{i}))$, where  the convergence is understood in the strict sense.
%   \item If $C'$ is an object of $\bC$ and $(f_{i})_{i\in I}$ is a family $f_{i}:C'\to C_{i}$ such that 
%   $$\sup_{J\subseteq I\ , |J|<\infty} \| \sum_{j\in J}e_{j} f_{i} f_{j}^{*}p_{i}\|_{\bigoplus_{i\in I}C_{i}}\vee \|\sum_{j\in J} f_{j}^{*} f_{j} \|_{C'}$$   is finite, then there exists a unique $f:C'\to  \bigoplus_{i\in I}C_{i}$ such that $p_{i}f=f_{i}$.
%
%  \item If $C'$ is an object of $\bC$ and $(f_{i})_{i\in I}$ is a family $f_{i}:C_{i}\to C'$ such that 
%   $$\sup_{J\subseteq I\ , |J|<\infty}  \| \sum_{j\in J}e_{j} f^{*}_{i} f_{j}p_{i}\|_{\bigoplus_{i\in I}C_{i}}\vee \|\sum_{j\in J} f_{j} f^{*}_{j} \|_{C'}$$   is finite, then there exists a unique $f:   \bigoplus_{i\in I}C_{i}\to C' $ such that $f e_{i}= f_{i}$.
%
%
%\item \fuli{Not clear that this will be needed.}  The inclusion of $\mathbb{K}(\bigoplus_{i\in I}C_{i})$ into $\End_{\bC}(\bigoplus_{i\in I}C_{i})$ induces an isomorphism
%$$M(\mathbb{K}(\bigoplus_{i\in I}C_{i})) \cong \End_{\bC}(\bigoplus_{i\in I}C_{i}) \ .$$ 
\end{ddd}

If we want to stress the surrounding category, then we talk about  an orthogonal sum in $\bC$.
% Here $M(\mathbb{K}(\bigoplus_{i\in I}C_{i}))$ denotes the multiplier algebra of  $\End_{\bC}(\bigoplus_{i\in I}C_{i})$.
%Condition  \ref{erogwfsfdvbsbfdbsdfbsfdbv}.\ref{ergiuwegegergwreg} in particular implies that $\mathbb{K}(\bigoplus_{i\in I}C_{i})$ is an essential ideal in 
%$ \End_{\bC}(\bigoplus_{i\in I}C_{i})$. 
%
%\begin{rem}
%We define
%$$\bK(D,\bigoplus_{i\in I}C_{i}):=\overline{\bigoplus_{i\in I}e_{i}\Hom_{\bC}(D,C_{i})}$$
%where the closure is taken in $\Hom_{\bC}(D,\bigoplus_{i\in I}C_{i})$.
%Similarly we define
%$$\bK(\bigoplus_{i\in I}C_{i},D):=\overline{\bigoplus_{i\in I}\Hom_{\bC}(C_{i},D)p_{i}}\ .$$
%A right multiplier $R$ from $D$ to $\bigoplus_{i\in I}C_{i}$ is a  family of continuous maps $(R_{D'})_{D'\in \bC}$
%$$R_{D'}: \bK(\bigoplus_{i\in I}C_{i},D')\to \Hom_{\bC}(D,D')$$
%such that $R_{D''}(hf)=h R_{D'}(f)$ for all $f$ in $\bK(\bigoplus_{i\in I}C_{i},D')$ and $h$ in $\Hom_{\bC}(D,D')$.
%A left multiplier $L$ from  $D$ to $\bigoplus_{i\in I}C_{i}$ is family $(L_{D'})_{D'\in \bC}$ of maps
%$$L_{D'}:\bK(D',\bigoplus_{i\in I}C_{i})\to  \Hom_{\bC}(D',D)$$ such that 
%$L_{D''}(fh)=L_{D'}h$ for all $h\in \Hom_{\bC}(D'',D)$. 
%We let  $RM(D, \bigoplus_{i\in I}C_{i})$ be the linear space of right multipliers and
%$LM( \bigoplus_{i\in I}C_{i},D)$ be the linear space of left-multipliers. 
%
% 
%\end{rem}

\begin{rem}\label{rem_bounded_inverse}
 Note that {the associated 
  multiplier transformations} \eqref{wergwergwggwegegr42t} and \eqref{wergwergwggwegegr42t1} are continuous linear maps between Banach spaces. Therefore, if they are bijective, then by the open mapping theorem their inverses are also continuous. In Proposition   \ref{lem_sum_isos_isometric} below we will show that  bijectivity  implies isometry. {In this case it then follows that} the families  of inverses $(m^{R,-1}_{D})_{D\in \Ob(\bC)}$ and $(m^{L,-1}_{D})_{D\in \Ob(\bC)}$ are uniformly bounded, and  the  transformations
 $m^{R}$ and $m^{L}$ in \eqref{vjfjosjvvsfvsvsfvsfv} and \eqref{vjfjosjvvsfvsvsfvsfv1}  are isomorphisms as well.
\hB
\end{rem}

\begin{rem}\label{qrguihqeriufwefqwefqewfqef}
If $I$ is a finite set, then it is an easy exercise to show that the notion of an orthogonal sum according to Definition \ref{erogwfsfdvbsbfdbsdfbsfdbv} coincides with the notion of an orthogonal sum 
according to Definition \ref{regiuhqrogefewfqwfqef}. \hB
%
% then we can replace  Condition  \ref{erogwfsfdvbsbfdbsdfbsfdbv}.\ref{ergiuwegegergwreg} by the condition that
%$$\sum_{i\in I} e_{i}e^{*}_{i}=\id_{\bigoplus_{i\in I}C_{i}}\ .$$ In this case
%$$\mathbb{K}(D,\bigoplus_{i\in I}C_{i})=\Hom_{\bC}(D,\bigoplus_{i\in I}C_{i})\ , \quad \mathbb{K}(\bigoplus_{i\in I}C_{i},D)=\Hom_{\bC}(\bigoplus_{i\in I}C_{i},D)$$
%and the canonical maps \eqref{wergwergwggwegegr42t} and  \eqref{wergwergwggwegegr42t1} are automatically isomorphisms.
%
%Because of the presence of the involution on $\bC$ it suffices to verify that \Alex{only} one of the transformations  \eqref{wergwergwggwegegr42t} and  \eqref{wergwergwggwegegr42t1} is an isomorphism.
%
%The notion of a direct sum is defined such that the object $\bigoplus_{i\in I}C_{i}$ represents the functor
%$RM(-,\bigoplus_{i\in I}C_{i})$ and corepresents the functor
%$LM(\bigoplus_{i\in I}C_{i},-)$.  
%
%
%
%For infinite $I$ the notion of a sum of objects in a $C^{*}$-category is defined such that it is compatible with the notion of a direct sum in the $C^{*}$-categories of Hilbert-$C^{*}$-modules over a $C^{*}$-algebra. \Alex{Will man das noch beweisen, oder ist es ``klar''?}
%\hB
\end{rem}

\begin{lem}\label{rgiojqeiovevqeve9}
An orthogonal sum of a family of objects is unique up to unique unitary isomorphism.%\fuli{An orthogonal sum of a family of objects is unique up to unique unitary multiplier isomorphism.}
\end{lem}
\begin{proof}
Let $(C,(e_{i})_{i\in I})$ and $(C',(e'_{i}))_{i\in I}$ be  two orthogonal sums of the family $(C_{i})_{i\in I}$. 
{In the following argument we will use repeatedly that the associated multiplier maps are bijective.}

By the Lemmas \ref{qoirgujoewgwergwergw} and \ref{qergijoqregqewfewfqewfwef}
 {the}  family $e' \coloneqq (e'_{i})_{i\in I}$ induces a left multiplier $L(e')$ in $\LM((C,(e_{i})_{i\in I}),C')$ satisfying $L(e')(e_{i})=e_{i}'$ for all $i$ in $I$. {It is the associated left multiplier of a uniquely determined 
 % Its preimage under  \eqref{wergwergwggwegegr42t1} is a
  morphism}  $v\colon C\to C'$ satisfying $ve_{i}=e_{i}'$ for all $i$ in $I$.
  
  Analogously, the family $e  \coloneqq (e_{i})_{i\in I}$ defines a left multiplier $L(e )$ in $\LM((C',(e_{i}')_{i\in I}),C)$ satisfying $L(e )(e'_{i} )=e_{i}$  for all $i$ in $I$ {which is the associated left mulitplier 
  %Its preimage under   \eqref{wergwergwggwegegr42t1} is 
  of a uniquely determined  morphism} $w\colon C'\to C$ such that $w e'_{i}=e_{i} $  for all $i$ in $I$.
 
   The associated left multiplier of  $v w$ is %under   \eqref{wergwergwggwegegr42t1}  i
  $v_{*}L(e )$ in $\LM((C',(e_{i}')_{i\in I}),C')$ satisfying
  \[(v_{*}L(e))(e'_{i})=ve_{i}=e_{i}'\]
  for all $i$ in $I$. 
   Since {the associated left multiplier} of $\id_{C'}$  %under  \eqref{wergwergwggwegegr42t1} 
   has the same property we conclude that $vw=\id_{C'}$. In a similar manner we show that $wv=\id_{C}$.
    
 We finally argue that $w=v^{*}$. 
 For every $i$ in $I$ and $f$ in $\Hom_{\bC}^{\fin}(C,D)$ we have for a sufficiently large finite subset $J$ of $I$, {using $f = \sum_{j \in J} f e_j e_j^*$,}
 \begin{eqnarray*}
  f v^{*}e_{i}'&=&({e^{\prime,*}_{i}} v f^{*})^{*}=
   {(e^{\prime,*}_{i} v \sum_{j \in J} e_j  e_j^* f^*)^*} =  {(e^{\prime,*}_{i} \sum_{j \in J} e^\prime_j  e_j^* f^*)^*}\\&=&
   (e_{i}^{*}f^{*})^{*}=fe_{i}=fwe_{i}'\, .
    \end{eqnarray*}
  This implies $v^{*}e_{i}'=we_{i}'$ for all $i$ in $I$. By the injectivity of the associated left multiplier map 
  \eqref{wergwergwggwegegr42t1} we get $v^{*}=w$.   
\end{proof}

The following  lemma  prepares the proof of Proposition~\ref{lem_sum_isos_isometric} which states that for an orthogonal sum the associated multiplier maps   \eqref{wergwergwggwegegr42t} and \eqref{wergwergwggwegegr42t1} are isometric. 
%We first prove   lemma  about multiplier norms on certain left-ideals.
Let $A$ be in $\Calg$ and $I$ be a left-ideal in $A$. Recall  {(\cite[Thm.~3.1.2]{murphy})} that left-ideals admit approximate right-units, i.e., there is a net $(u_\nu)_{\nu \in N}$ of positive elements of $I$ with $\lim_\nu x u_\nu = x$ for every $x$ in $I$. For an element $a$ of $A$ we define  its right-multiplier norm on $I$ by
\[
\|a\|_{\cR(I)} \coloneqq \sup_{x \in I, \|x\| \le 1} \|xa\|\,.
\]
%and correspondingly $\|a\|_{\cL(I)}$ for its left-multiplier norm on the set $I$, i.e.,
%\[
%\|a\|_{\cL(I)} \coloneqq \sup_{x \in I, \|x\|=1} \|ax\|\,.
%\]
{A  family of elements  $(v_{\kappa})_{\kappa\in K}$  of $A$ is called right-essential   if for every non-zero $a$ in $A$ exists some $\kappa$ in $K$ such that $a v_{\kappa}$} is not zero.
We define the notion of a left-essential  {subset }analogously using multiplication from the left.
\begin{lem}\label{lem_norm_multipliers_general}
If $I$ admits an approximate unit $(u_\nu)_{\nu \in N}$ which is right-essential in $A$, then
%Assume that multiplication from the right by the set $\{e_\nu\colon \nu \in N\}$ is essential on $A$, i.e., for every non-zero $a$ in $A$ exists some $\nu$ in $N$ such that $a e_\nu$ is not zero.
%
%Assume that $\rho_\pi$ is isometric.
%there exists a faithful representation $\rho\colon A \to B(H)$ of $A$ on a Hilbert space $H$ such that the net $(\rho(e_\nu))_{\nu \in N}$ converges strongly to $\rho(\id_A)$.
%Then 
%
for every $a$ in $A$ we have
\[
\|a\| = \|a\|_{\cR(I)}\,.
\]
\end{lem}
 
\begin{proof}
%In the first step we show that $I$ is both left and right essential in $A$, i.e., for every $a$ in $A$ we have that $Ia = \{0\}$ implies $a=0$, and also $aI = \{0\}$ implies $a=0$. We start with the case $Ia = \{0\}$. This implies $\rho(I)\rho(a) = \{0\}$, and since $\rho$ is faithful it suffices to show $\rho(a) = 0$. For any vector $v$ in $H$ we have
%\[
%\rho(a)(v) = \rho(\id_A a)(v) = \rho(\id_A)(\rho(a)(v)) = \lim_\nu \big( \rho(e_\nu)(\rho(a)(v)) \big) = 0\,,
%\]
%which shows that $\rho(a) = 0$. To treat the case $aI = \{0\}$ we note that this implies $I^* a^* = \{0\}$, and now the same argument as above works since $e_\nu$ is self-adjoint for every $\nu$ in $N$.
%
%Since we now know that $I$ is an essential left-ideal in $A$, general $C^*$-algebra theory tells us that $\|a\| = \|a\|_{\cL(I)}$, see \cite{ozawa_MO}.
%
%
%
%
%To show that $\|a\|_{\cR(I)}$ equals $\|a\|$, we modify Ozawa's proof from \cite{ozawa_MO}. We consider the double commutant $\rho(A)''$ of $\rho(A)$, and recall that von Neumann's double commutant theorem tells us that $\rho(A)$ is strongly closed in $\rho(A)''$. The strong closure of $\rho(I)$ is a strongly closed left-ideal in $\rho(A)''$ and therefore of the form $\rho(A)'' \pi$ for some projection $\pi$ in $\rho(A)''$. In fact, $\pi$ is the strong limit of any approximate right-unit of $\rho(I)$, and hence $\pi$ equals by our assumption $\rho(\id_A)$.
The inequality $\|a\|_{\cR(I)} \le \|a\|$ immediately follows from the sub-multiplicativity of the norm on $A$. We now show the reverse inequality.\footnote{The argument is a modification of the argument that Ozawa provided to answer the MathOverflow question \cite{ozawa_MO}.}

%To show the reverse inequality, we use a modification of \Alex{the argument that Ozawa provided to answer the MathOverflow question \cite{ozawa_MO}.}

We let $A^{**}$ be the von Neumann algebra given by the  double commutant   of the image  of  $A$ under its universal representation.\footnote{By \cite[III.5.2.7]{blackadar_operator_algebras} it is also {isometrically isomorphic to the} double dual of $A$ considered as a Banach space.  {This explains the notation $A^{**}$.}} The weak closure of $I$ in $A^{**}$ will be denoted by $I^{**}$. It is  a weakly closed left-ideal  in $A^{**}$ and therefore of the form
$A^{**}\pi_{I}$ for some projection $\pi_{I}$ in $A^{**}$. In fact{,} $\pi_{I}$ is the strong limit of  $(u_\nu)_{\nu \in N}$   in $A^{**}$. 
We refer to \cite[III.1.1.13]{blackadar_operator_algebras} for these statements.  

We let $\bar I$ denote the strong closure of $I$ in $A^{**}$. By   \cite[Thm.~4.2.7]{murphy}
we know that $\bar I$ is also weakly closed. Hence the canonical inclusion $\bar I\subseteq I^{**}$ is an equality.

 Since we assume that $(u_\nu)_{\nu \in N}$ is  right-essential
 we can conclude that the map $A \to A^{**}$, $a \mapsto a\pi_I$ is injective. Hence its extension to  a map $A^{**} \to A^{**}$, $z \mapsto z\pi_I$, is also injective \cite[III.5.2.10]{blackadar_operator_algebras}. But this implies   that $\pi_I=1_{A^{**}}$ and therefore   $\bar I=I^{**}=A^{**}$.

For every $a$ in $A$ we have the chain of equalities
$$
\|a\|_{\cR(I)}   = \sup_{x \in I, \|x\|\le 1} \|xa\|  \stackrel{!}= \sup_{y \in \bar I, \|y\|\le1} \|y a\| 
  \stackrel{!!}= \sup_{y \in A^{**}, \|y\|\le 1} \|y a\|\,,
$$
where in the equality marked $!$ we use that $\bar I $ is the strong   closure of $I$  in $A^{**}$ and that $A\to A^{**}$ is isometric, and  the equality marked by $!!$  follows from $I^{**}=A^{**}$ as shown above.
 
%
%For every $a$ in $A$ we have
%\begin{align*}
%\|a\|_{\cR(I)} & = \sup_{x \in I, \|x\|=1} \|xa\|\\
%& \stackrel{!}\ge \sup_{x \in I, \|x\|=1} \|\omega(xa)\|\\
%& = \sup_{x \in I, \|x\|=1} \|\omega(x)\omega(a)\|\\
%& \stackrel{!!}= \sup_{y \in I^{**}, \|y\|=1} \|y \omega(a)\|\\
%& \stackrel{!!!}= \sup_{y \in A^{**}\pi, \|y\|=1} \|y \omega(a)\|\\
%& \stackrel{!!!!}\ge \sup_{z \in A^{**}, \|z\|=1} \|z \pi \omega(a)\|\,,
%\end{align*}
%where in the inequality marked $!$ we used that $\omega$ is non-expansive since it is a representation, in the equality $!!$ we passed to the (strong) closure $I^{**}$ of $\omega(I)$ in $A^{**}$, in the equality $!!!$ we used that $I^{**}$ is of the form $A^{**}\pi$, and in the inequality $!!!!$ we used that $\|z\pi\| \le 1$ if $\|z\|=1$.
%%%
%From the assumption it first follows that the bounded linear map $A \to A^{**}$, $a \mapsto \omega(a) \pi$ is injective, and from this that its extension to the map $A^{**} \to A^{**}$, $z \mapsto z\pi$ is also injective \cite[III.5.2.10]{blackadar_operator_algebras}.

Since   $A^{**}$ is a von Neumann algebra it admits a measurable function calculus for self-adjoint operators. For any $\varepsilon $ in $(0,\infty)$  we can define the  projection $q \coloneqq 1_{[\|a\|-\varepsilon,\|a\|]}(|a a^*|^{1/2})$ in $A^{**}$.  Since 
$\sup\sigma(|aa^{*}|^{1/2})=\|a\|^{2}$ we have $ \sigma(|a^{*}a|^{1/2}) \cap [\|a\|-\varepsilon,\|a\|]\not=\emptyset$ and therefore $q\not=0$.
The spectral theorem implies 
%, and note that the product $q\pi$ is non-zero.
 the inequality
$
a a^* \ge (\|a\| - \varepsilon)^2 q
$
of self-adjoint operators. By  \cite[Thm.~2.2.5(2)]{murphy}) we then also have the inequality
$qaa^{*}q\ge  (\|a\| - \varepsilon)^2 q$. Using the $C^*$-identity for the first equality
we therefore get the estimate
%Using first the $C^*$-identity and then that conjugation by elements preserves these inequalities \cite[Thm.~2.2.5(2)]{murphy})
\[
{\|qa\|^2 = \|q a a^* q\| \ge (\|a\| - \varepsilon)^2 \|q\| \stackrel{q\not= 0}= (\|a\| - \varepsilon)^2\,.}
\]
%We consider the polar decomposition $q\pi = v \cdot |q\pi|$. The (non-zero) partial isometry $v$ is an element of $A^{**}$ \cite[Thm.~4.1.10]{murphy} and satisfies $qv = v = v \pi$. Furthermore, we have the estimate
%\[
%\|v^* a\|^2 = \|v^* a a^* v\| \ge (\|a\| - \varepsilon)^2 \|v^* q v\| = (\|a\| - \varepsilon)^2 \|v^* v\| = (\|a\| - \varepsilon)^2\,,
%\]
%where we have used for the last equality that $v$ is non-zero.
%%%%
%For any $\varepsilon > 0$ we define the projection $q \coloneqq 1_{[\|\rho_\pi(a)\|-\varepsilon,\|\rho_\pi(a)\|]}(|\rho_\pi(a)|)$.
%Since we use the Borel functional calculus, $q$ is a priori an element of the weak closure of the $C^*$-algebra generated by $|\rho_\pi(A)|$. But $\rho_\pi(A)$ is clearly a subset of $\rho(A)''$, and hence $|\rho_\pi(A)|$ an element of $\rho(A)''$ and the subalgebra it generates a subalgebra of $\rho(A)''$. It follows that $q$ is in the weak closure of $\rho(A)''$, and since the latter is weakly closed due to being a von Neumann algebra, $q$ is an element of it.
%
%We have $\pi q \not= 0$ \textbf{todotodotodo}
%We have $\pi \rho(A)'' q \not= 0$. Choosing any element $m$ in $\pi \rho(A)'' q$, we consider its polar decomposition $m = v |m|$. Here $v$ is a partial isometry in $\rho(A)''$ by \cite[Thm.~4.1.10]{murphy}, and hence $v$ is also in $\pi \rho(A)'' q$.
Finally  we get
\[\|a\|_{\cR(I)} =
\sup_{y \in A^{**}, \|y\|\le 1} \|y a\| \ge \|q a\| \ge \|a\|-\varepsilon\,.
\]
Since $\varepsilon$ was arbitrary, the desired inequalilty $\|a\|_{\cR(I)} \ge \|a\|$ follows.
%%%%
%\textbf{todotodotodo}
%
%Let $p$ in $A^{**}$ be any non-zero projection below $q$ and $\pi$. Especially, this implies $p \pi = p = p q$. From \eqref{eq_inequality_operators} we get
%\[
%\|p a\|^2 = \|p a a^* p^*\| \ge (\|a\|-\varepsilon)^2 \|p q p^*\|\,,
%\]
%which together with $\|p q p^*\| = \|p\| = 1$ finally implies
%\[
%\sup_{z \in A^{**}, \|z\|=1} \|z \pi a\| \ge \|p \pi a\| = \|p a\| \ge  \|a\|-\varepsilon\,.
%\]
%Since $\varepsilon$ was arbitrary, we conclude $\|a\|_{\cR(I)} \ge \|a\|$ finishing the proof.
%
%Now we have
%\[
%\sup_{z \in \rho(A)'', \|y\|=1} \|\pi z \rho_\pi(a)\| \ge \|\pi q \rho_\pi(a)\| \ge \|\rho_\pi(a)\| - \varepsilon\,.
%\]
%Since $\varepsilon$ was arbitrary, we conclude $\|a\|_{\cR(I)} \ge \|a\|$. The reverse inequality $\|a\|_{\cR(I)} \le \|a\|$ is trivial, and so the proof is finished. \textbf{Benutzen, dass die Erweiterung von Konjugation mit $\pi$ auf den starken Abschluss immer noch isometrisch ist, um das linke $\pi$ in obiger Abschätzung verschwinden zu lassen.}
\end{proof}

\begin{rem}
Since the {members}    of the net  $(u_\nu)_{\nu \in N}$   are positive and therefore self-adjoint, the assumption of Lemma~\ref{lem_norm_multipliers_general} is equivalent to  the assumption that {this net} % set $\{u_\nu\colon \nu \in N\}$  
is left-essential.
%demanding that multiplication from the \emph{left} by the set $\{e_\nu\colon \nu \in N\}$ is essential on $A$. To see this one just has to use the involution.
Hence the proof of Lemma~\ref{lem_norm_multipliers_general} also shows that $\|a\| = \|a\|_{\cL(I)}$, where
\[
\|a\|_{\cL(I)} \coloneqq \sup_{x \in I, \|x\|\le 1} \|ax\|
\]
is the norm of $a$ considered as a left-multiplier on $I$.
%
%Ozawa \cite{ozawa_MO} proved that under the assumption of Lemma~\ref{lem_norm_multipliers_general} we also have $\|a\| = \|a\|_{\cL(I)}$, where
%\[
%\|a\|_{\cL(I)} \coloneqq \sup_{x \in I, \|x\|=1} \|ax\|
%\]
%is the norm of $a$ considered as a left-multiplier on $I$. \fuli{ist das nicht obvious. Der gleiche Beweis? Warum diese Bemerkung}
\hB
\end{rem}

Let $\bC$ be in $\Ccat$.
Let $(C_{i})_{i\in I}$ be a family of objects in $\bC$ and assume that $(C,(e_{i})_{i\in I})$ represents an orthogonal sum of $(C_{i})_{i\in I}$ according to Definition \ref{erogwfsfdvbsbfdbsdfbsfdbv}.
This is equivalent to the fact that the associated multiplier maps  \eqref{wergwergwggwegegr42t} and \eqref{wergwergwggwegegr42t1} are bijective for every object $D$ in $\bC$. %We now show that in this case these natural transformations are actually even isometric.

\begin{prop}\label{lem_sum_isos_isometric}
The  associated multiplier maps  \eqref{wergwergwggwegegr42t} and \eqref{wergwergwggwegegr42t1} are isometric for every object $D$ in $\bC$.
%If \eqref{wergwergwggwegegr42t} and \eqref{wergwergwggwegegr42t1} are bijective, then they are isometric.
\end{prop}
 \begin{proof}
We only discuss the case of the associated right muliplier map \eqref{wergwergwggwegegr42t}. Then the case of left multipliers  \eqref{wergwergwggwegegr42t1} can be deduced using the involution of $\bC$.  Let $h\colon D\to C$ be a morphism, and denote by $R(h) \coloneqq m_{D}^{R}(h)$   the associated right multiplier.   The estimate in \eqref{vjfjosjvvsfvsvsfvsfv}   immediately implies $\|R(h)\| \le \|h\|$.

In order to show the reverse estimate, we claim that it suffices to prove it for endomorphisms of the object $C$. To show the claim assume that $\|f\| \le \|R(f)\|$ for all endomorphisms $f$ of $C$. Applying this to $f=hh^{*}$ with $h\not=0$ (the case $h=0$ is obvious) we conclude $\|hh^*\| \le \|R(hh^*)\|$. Using the $C^*$-identity and that the involution is an isometry, we further get the equality $\|h h^*\| = \|h^*\|^2 = \|h^*\| \|h\|$.
On the other hand, by the definition of the right multiplier norm we  have the inequality $\|R(hh^*)\| \le \|R(h)\| \|h^*\|$. Combining everything and dividing by $\|h^*\|$, we arrive at the desired inequality $\|h\| \le \|R(h)\|$.

In order to show $\|f\| \le \|R(f)\|$ for every endomorphism $f$ of $C$ we employ Lemma~\ref{lem_norm_multipliers_general}. Recall that $\mathbb{K}((C,(e_{i})_{i\in I}),C)$ is generated by morphisms $f' \colon C\to C$ of the form $f' = f'_{i} e_{i}^{*}$ for some morphism $f'_{i}\colon C_{i}\to C$. It follows that we have the inclusion
\[
\Hom_{\bC}(C,C) \cdot \mathbb{K}((C,(e_{i})_{i\in I}),C) \subseteq \mathbb{K}((C,(e_{i})_{i\in I}),C)\,,
\]
i.e., $\mathbb{K}((C,(e_{i})_{i\in I}),C)$ is a left-ideal in the $C^*$-algebra $\Hom_{\bC}(C,C)$.
For every   finite subset $J$ of $I$ we define  
$p_J \coloneqq \sum_{i \in J} e_i e_i^*$ in $\End_{\bC}(C)$. 
It is immediate that the family $(p_J)_{J }$ with   $J$ running through the poset of   finite subsets of $I$ is an approximate right-unit for $\mathbb{K}((C,(e_{i})_{i\in I}),C)$.
In order to   apply Lemma~\ref{lem_norm_multipliers_general} we must check that $(p_J)_{J }$ is right-essential in  $\End_{\bC}(C)$.   This follows from the injectivity of the associated left multiplier map.  Indeed, if $f$ is a non-zero morphism in  in $\End_{\bC}(C)$, then %by
%the injectivity of the associated left multiplier map 
 %Corollary~\ref{kor_morphisms_sum_characterization}.\ref{item_morphisms_sum_characterization_one}+
 there is an $i$ in $I$ such that $f e_i$, and hence also $f p_{\{i\}}$, is non-zero.
\end{proof}

\section{Morphisms into and out of orthogonal sums}
\label{sec_morphisms_sums}

In the following we 
 explain methods to produce morphisms into or out of an orthogonal sum. This will be  used in Proposition \ref{lem_sum_characterization_morphisms} to  provide an alternative  characterization of orthogonal sums. 
  We then show in Remark \ref{euwifhqiufqwfewffqewffqfef} that our definition of an orthogonal sum is equivalent with the one introduced in \cite{fritz}.
 {We  furthermore provide} some technical results preparing \cite{coarsek}.
% Given a family of objects indexed by a set $I$  we can consider a partition $(J_{k}
%)_{k\in K}$ of $I$. In this case we can  consider the sum  over $K$ over the family of  partial sums over 
%the subfamilies indexed by the subsets $J_{k}$.  We will show that the result is isomorphic to the total sum over $I$.
%Finally we will show in Proposition \ref{lem_sum_isos_isometric} that in the case of an orthonormal sum the associated multiplier maps are isometric.

 {Let $\bC$ be in $\Ccat$, let} $(C_{i})_{i\in I}$ be a family of objects of $\bC$, and assume that $(C_{i})_{i\in I}$ admits an orthogonal sum $({C}, (e_{i})_{i\in I})$. Let   $D$ be an object of $\bC$, and 
  let $(h_{i})_{i\in I}$ and $ (h'_{i})_{i\in I}$  be   families of morphisms $h_{i}\colon D\to C_{i}$ and $h'_{i}\colon C_{i}\to D$. 

\begin{kor}\label{qergijoqregqewfewfqewfwef1}\mbox{}
\begin{enumerate} 
\item\label{qerjgqkrfewfewfqfefqef111}
{  There exists a morphism $h\colon D\to C$ (often denoted by $\sum_{i\in I} e_{i}h_{i}$) with $e_{j}^{*}h=h_{j}$ for all $j$ in $I$ if and only if {$(h_{i})_{i\in I}$ is square summable.}%\fuli{AV: Assume that $\sum_{i\in I} e_{i}h$  converges strictly. This implies square-summability which gives the norm.} %\eqref{eq_condition_multiplier} is satisfied.}

{If {$(h_{i})_{i\in I}$ is square summable,} %\eqref{eq_condition_multiplier} is satisfied, 
then $h$ is uniquely determined and $\|h\|^{2}=\sup_{J} \|\sum_{i\in J} h_{i}^{*}h_{i}\|$, where $J$ runs over the finite subsets of $I$.} %{Moreover, if the family $(h_{i}^*)_{i\in I}$ is mutually orthogonal and $\sup_{i\in I} \|h_{i}\|<\infty$, then \eqref{eq_condition_multiplier} is satisfied.}
%If  $(h_{i})_{i\in I}$ is a family of morphisms $h_{i}\colon D\to C_{i}$ \Alex{satisfying \eqref{eq_condition_multiplier},} then there exists a unique morphism  $h\colon  D\to \bigoplus_{i\in I} C_{i}$ (often denoted by $\sum_{i\in I} e_{i}h_{i}$)   such that   $e_{j}^{*}h=h_{j}$ for all $j$ in $I$.
\item\label{qerjgqkrfewfewfqfefqef11111}
 There exists a morphism $h'\colon C \to D$ (often denoted by $\sum_{i\in I} h'_{i}e_{i}^{*}$) with $h' e_{j}=h'_{j}$ for all $j$ in $I$ if and only if {$(h'_{i})_{i\in I}$ is square summable.}% \eqref{eq_condition_multiplier_2} is satisfied.}

If {$(h'_{i})_{i\in I}$ is square summable,} % \eqref{eq_condition_multiplier_2} is satisfied, 
then $h'$ is uniquely determined and $\|h'\|^{2}=\sup_{J} \|\sum_{i\in J} h_{i}'h^{\prime,*}_{i}\|$, where $J$ runs over the finite subsets of $I$.} %{Moreover, if the family $(h'_{i})_{i\in I}$ is mutually orthogonal and $\sup_{i\in I} \|h'_{i}\|<\infty$, then \eqref{eq_condition_multiplier_2} is satisfied.}
%If $(h'_{i})_{i\in I}$ is a family of morphisms $h'_{i}\colon C_{i}\to D$ \Alex{satisfying \eqref{eq_condition_multiplier_2},} then there exists a unique morphism $h'\colon \Alex{\bigoplus_{i\in I} C_{i}}\to D$ (often denoted by $\sum_{i\in I} h'_{i}e_{i}^{*}$) such that $h' e_{j}=h'_{j}$ for all $j$ in $I$.
\end{enumerate}
\end{kor}

\begin{proof}
By Lemma \ref{qergijoqregqewfewfqewfwef} we obtain multipliers corresponding to $h$ or $h'$ if and only{ if the corresponding families are square summable.}
%\eqref{eq_condition_multiplier} or \eqref{eq_condition_multiplier_2} is satisfied, respectively. 
In view of the Definition \ref{erogwfsfdvbsbfdbsdfbsfdbv} of an orthogonal sum these multipliers lift {uniquely} to the desired morphisms under {the associated multiplier morphism maps} \eqref{wergwergwggwegegr42t} or \eqref{wergwergwggwegegr42t1}, respectively. The assertion about the norms follows from Proposition \ref{lem_sum_isos_isometric}.
%Uniqueness of these morphisms follows from the uniqueness of the corresponding multipliers together with the injectivity of \eqref{wergwergwggwegegr42t} or \eqref{wergwergwggwegegr42t1}, respectively.
%It has already been shown in Lemma \ref{qergijoqregqewfewfqewfwef} that  \eqref{eq_condition_multiplier} or \eqref{eq_condition_multiplier_2},  respectively,  
%\Alex{That \eqref{eq_condition_multiplier}, respectively \eqref{eq_condition_multiplier_2},} 
%are satisfied if the corresponding families are mutually orthogonal and  {uniformly bounded.} %, is already shown in Lemma \ref{qergijoqregqewfewfqewfwef}.
\end{proof}

The following corollary states that a map  into an orthogonal sum, or a map out of an orthogonal sum, respectively, is uniquely determined by its compositions with the structure maps of the sum. 
{We keep the notation introduced before Corollary \ref{qergijoqregqewfewfqewfwef1}. We consider  pairs of morphisms
 $f,f' \colon D \to C$,    $k,k' \colon C \to D$,  and  $g,g' \colon C \to C$.} 
\begin{kor}\label{kor_morphisms_sum_characterization}\mbox{} % \fuli{AV: ok: state everything for multiplier morphisms.}
\begin{enumerate}
\item\label{item_morphisms_sum_characterization_one}  If $e_j^* f = e_j^* f'$  for all $j$ in $I$, then $f = f'$.
 \item\label{item_morphisms_sum_characterization_one1} If  $k e_j = k' e_j$  for all $j$ in $I$, then $k = k'$.
\item\label{item_morphisms_sum_characterization_both}  If $e_i^* g e_j = e_i^* g' e_j$ for all $i,j$ in $I$, then $g= g'$.
\end{enumerate}
\end{kor}
\begin{proof}
Assertions~\ref{item_morphisms_sum_characterization_one} and \ref{item_morphisms_sum_characterization_one1} immediately follow from the uniqueness statements in   Corollary~\ref{qergijoqregqewfewfqewfwef1}. 

We show Assertion~\ref{item_morphisms_sum_characterization_both}.
Fixing $j$ in $I$, the uniqueness statement in  Corollary~\ref{qergijoqregqewfewfqewfwef1}.\ref{qerjgqkrfewfewfqfefqef111} (applied to the family of morphisms $(h_i)_{i \in I}\colon C_j \to C_i$ defined by $h_i\coloneqq e_i^* g e_j$ for all $i$ in $I$) implies that $g e_j = g' e_j$. Then the uniqueness statement in Corollary~\ref{qergijoqregqewfewfqewfwef1}.\ref{qerjgqkrfewfewfqfefqef11111} (applied to $h_i^\prime\coloneqq g e_i\colon C_i \to C$ for every $i$ in $I$) implies that $g = g'$.
\end{proof}
   
   Let $(C_{i})_{i\in I}$ and $(C_{i}')_{i\in I}$ be two families of objects in $\bC$ with the same index set.
 We assume that  they admit orthogonal sums $(C, (e_{i})_{i\in I} )$ and $  (C', (e'_{i})_{i\in I} )$.
Let $(f_{i})_{i\in I}$ be a uniformly bounded family of morphisms $f_{i} \colon C_{i}\to C'_{i}$.  %We assume that  $\sup_{i\in I}\|f_{i}\|<\infty$. 
By Lemma \ref{qoirgujoewgwergwergw} the families $(f_{i}e^{*}_{i})_{i\in I}$ 
and $(e'_{i}f_{i})_{i\in I}$ are  square summable.
Using Corollary \ref{qergijoqregqewfewfqewfwef1}.\ref{qerjgqkrfewfewfqfefqef111} applied to
$(f_{i}e_{i}^{*})_{i\in I}$ we get a unique morphism $f \colon C  \to  C'$ such that
$e_{j}^{\prime,*}f=f_{j}e_{j}^{*}$  for all $j$ in $I$. 
On the other hand, using Corollary \ref{qergijoqregqewfewfqewfwef1}.\ref{qerjgqkrfewfewfqfefqef11111} applied to the family
$(e'_{i}f_{i})_{i\in I}$ we get a unique  morphism 
$f'\colon C\to  C'$ satisfying   
$ f'e_{j}=e_{j}'f_{j} $  for all $j$ in $I$. 

\begin{lem}\label{ruihqiuwefwefqfqwefq}
 We have $f=f'$. 
 \end{lem} \begin{proof} For all $i,j$ in $I$ with $i\not=j$ we have
$e_{i}^{\prime,*}f e_{j}=0= e_{i}^{\prime,*}f' e_{j}$. Furthermore
$e_{i}^{\prime,*}fe_{i}=f_{i}e_{i}^{*}e_{i}=f_{i}=e_{i}^{\prime,*}e'_{i}f_{i}= {e_{i}^{\prime,*}f'e_{i}}$.
This first implies that $e_{i}^{\prime,*}(f-f')e_{j}=0$ for all $i,j$ in $I$,  {which in turn} implies $f=f'$ by Corollary \ref{kor_morphisms_sum_characterization}.\ref{item_morphisms_sum_characterization_both}.
\end{proof}
We will usually use the suggestive notation
\begin{equation}\label{ervfdvsdfvfdvfsdvsfdv}
\oplus_{i\in I} f_{i} \colon C \to C'
\end{equation}
for the morphism $f$ (or equivalently, $f'$) considered above.

% Since we have proved ``if and only if''-statements in Corollary \ref{qergijoqregqewfewfqewfwef1} and Lemma \ref{qergijoqregqewfewfqewfwef}, we can use them to characterize orthogonal sums. 

We consider a full subcategory $\bD\subseteq \bC$ in $\Ccat$  such that $\bC$ admits
all  finite orthogonal sums. 
We consider two  families of objects $(A_{i})_{i\in I}$ and
$(B_{j})_{j\in J}$ in $\bD$ and   assume that  there exist orthogonal sums 
$(A,(e_{i})_{i\in I})$ and $(B,(f_{j})_{j\in J})$  of these families in $\bC$.
By Corollary \ref{kor_morphisms_sum_characterization} every morphism $h \colon A\to B$ is uniquely determined by the family  
$(h_{ji} )_{i\in I,j\in J}$ of morphisms $h_{ji} \coloneqq f_{j}^{*}he_{i}:A_{i}\to B_{j}$ in $\bD$.  We claim that one can  describe the Banach space
$\Hom_{\bC}(A,B)$ completely in the language of $\bD$. In other words, $\bD$ determines 
 which families $(h_{ji})_{i\in I,j\in J}$
correspond to morphisms  $h$  and its norms.
In order to formulate  this claim technically  we consider a second full
inclusion $\bD\subseteq \bC'$ where $\bC'$  also admits all  finite small orthogonal sums  
and the   orthogonal sums
$(A',(e'_{i})_{i\in I})$ and $(B',(f'_{j})_{j\in J})$    in $\bC'$ of the families $(A_{i})_{i\in I}$ and
$(B_{j})_{j\in J}$.
 
\begin{prop}\label{woijotrgwergeferfwerfewrf}
For a family
$(h_{ji})_{i\in I,j\in J}$ of morphisms $h_{ji} \colon A_{i}\to B_{j}$ in $\bD$
 the  following assertions are equivalent:
\begin{enumerate}
\item \label{gwoeirgujowregregfwfer} There is a   morphism
$h \colon A\to B$ such that $h_{ji}=f_{j}^{*}
h e_{i}$ for all $i$ in $I$ and $j$ in $J$.
\item \label{gwoeirgujowregregfwfer1}  There is a   morphism
$h' \colon A'\to B'$ such that $h_{ji}=f_{j}^{*,\prime}
h' e'_{i}$ for all $i$ in $I$ and $j$ in $J$.
 \end{enumerate}
 If these conditions are satisfied, then $\|h\|=\|h'\|$.
\end{prop}
%By Corollary \ref{kor_morphisms_sum_characterization}   the morphisms $h$ and $h'$ are unique if they exist.
\begin{proof}
We let $\bD_{\oplus}$ denote the full subcategory of $\bC$
on objects which are isomorphic to finite sums of objects from $\bD$.
We define $\bD'_{\oplus}$ similarly.  
Then it is easy to construct an  equivalence $\bD_{\oplus}\to \bD_{\oplus}'$ in $\Clincat$ which extends the identity of $\bD$. This equivalence is then necessarily an equivalence in $\Ccat$.

By the symmetry of the assertions it suffices to show that Assertion \ref{gwoeirgujowregregfwfer} implies Assertion \ref{gwoeirgujowregregfwfer1}. Thus assume that $h$ exists.

  If we set $h_{i} \coloneqq he_{i} \colon A_{i}\to B$, then by Corollary \ref{qergijoqregqewfewfqewfwef1}  the  family $(h_{i})_{i\in I}$ is square summable  and we have 
$$\|h\|^{2} \coloneqq \sup_{L}\|\sum_{l\in L} h_{l}h_{l}^{*}\|\ ,$$ where $L$ runs over the finite subsets of $I$.  
We now fix $L$ and choose a sum $(C,(c_{l})_{l\in L})$ of the finite family
$(A_{l})_{l\in L}$ in $\bD_{\oplus}$.  %set $g=\sum_{l\in L} h_{l}c_{l}^{*}$. 
Then using Definition \ref{regiuhqrogefewfqwfqef}.\ref{asfvasvadsvsdvasdvisddddda1}
$$\|\sum_{l\in L} h_{l}h_{l}^{*}\|=\|\sum_{l\in L} h_{l}c^{*}_{l}( \sum_{l'\in L}h_{l'}c^{*}_{l'} )^{*}\|=\|\sum_{l\in L}  h_{l}c_{l}^{*}\|^{2}\ .$$  
Again by   Corollary \ref{qergijoqregqewfewfqewfwef1}  the morphism  $$g^{L} \coloneqq \sum_{l\in L}  h_{l}c_{l}^{*} \colon C\to B$$
 gives rise to  the 
  square summable family $(g_{j}^{L})_{j\in J}$ with 
$g^{L}_{j} \coloneqq f_{j}^{*} g^{L}=\sum_{l\in L} h_{jl}c_{l}^{*}$ and $$\|g^{L}\|^{2}= \sup_{M} \|\sum_{j\in M}g^{L,*}_{j} g^{L}_{j}\| % =   \sup_{M} \|%\sum_{m\in M} 
%(\sum_{l\in L}  h_{ml}c_{l}^{*})^{*}
%\sum_{l'\in L}  h_{ml'}c_{l'}^{*} \|
=  
\sup_{M} \|\sum_{j\in M} 
\sum_{l,l'\in L}  c_{l} h_{jl}^{*}
   h_{jl'}c_{l'}^{*} \|
\ .$$
On the right-hand side we have the norms of   endomorphisms of
$C$ which are completely determined by  the structure of $\bD_{\oplus}$.
We let
$(C',(e'_{l})_{l\in L})$ be the image of $(C,(c_{l})_{l\in L})$ under a unitary
equivalence $\bD_{\oplus}\to \bD_{\oplus}'$ under $\bD$.
We then consider the morphisms $g_{j}^{L,\prime} \coloneqq \sum_{l\in L}   h_{jl}c^{\prime,*}_{l} \colon C'\to B_{j}$.
We have the equality  
\begin{equation}\label{weqfihfiuhwefiuqhweufiwefqwefqwefqewf}
\|g^{L}\|^{2}= \sup_{M} \|\sum_{j\in M} 
\sum_{l,l'\in L}  c_{l} h_{jl}^{*}
   h_{jl'}c_{l'}^{*} \|=  \sup_{M} \|\sum_{j\in M} 
\sum_{l,l'\in L}  c'_{l} h_{jl}^{*}
   h_{jl'}c_{l'}^{\prime,*} \| = \sup_{M}\|\sum_{j\in M}g^{L,\prime,*}_{j} g^{L,\prime}_{j}\|\ .
\end{equation}     This implies that $(g^{L,\prime}_{j})_{j\in J}$ is a  square summable family and determines  by   Corollary \ref{qergijoqregqewfewfqewfwef1} a morphism
   $g^{L,\prime} \colon C'\to B'$ such that $f_{j}^{\prime,*} g^{L,\prime}=g^{L,\prime}_{j}$ for all $j$ in $J$ with $\|g^{L,\prime}\|=\|g^{L}\|$.
 We now set $h'_{i} \coloneqq g^{L,\prime}c'_{i} \colon A_{i}\to B'$.  Then $f_{j}'h'_{i}=h_{ji}$ for every $j$ in $J$ so that $h'_{i}$ does not depend on the choice of $L$ provided $i\in L$.
 We furthermore have  
 $g^{L,\prime}=\sum_{l\in L} h_{l}' c_{l}^{\prime,*}$
  and  $$\sup_{L}\|\sum_{l\in L} h'_{l}h_{l}^{\prime,*}\|=\sup_{L}\|g^{L,\prime}\|^{2}=\sup_{L}\|g^{L}\|^{2}= 
\sup_{L}\|\sum_{l\in L} h_{l}h_{l}^{*}\|=\|h\|^{2}\ .$$
This shows that the family $(h_{i}')_{i\in I}$ is square summable and determines by   Corollary \ref{qergijoqregqewfewfqewfwef1} a morphism $h' \colon A\to B$ such that $f_{j}^{\prime,*} h' e'_{i}=h_{ji}$
for all $i$ in $I$ and $j$ in $J$ and 
$ \|h'\|^{2} =\|h\|^{2}$.
 \end{proof} 

%\begin{rem}
%Note that we assumed that $\bC$ and $\bC'$ have all small sums
%in order to have a simple statement of  Proposition  \ref{woijotrgwergeferfwerfewrf}.   It actually  suffices to assume the existence of the few  sums which appear in the proof. \hB
%\end{rem}

{The following proposition provides an alternative characterization of orthogonal sums in terms of morphisms.}
Let $\bC$ be in $\Ccat$, let $(C_i)_{i \in I}$ be a family of objects of $\bC$, and let $C$ be an object of $\bC$ with a family of mutually orthogonal isometries $(e_i)_{i \in I}$ with $e_i\colon C_i \to C$ for every $i$ in $I$.
\begin{prop}\label{lem_sum_characterization_morphisms} %\fuli{AV:  must require strict convergence of $\sum_{i}e_{i}h_{i}$ instead of square summability}
$(C,(e_i)_{i \in I})$ is an orthogonal sum in $\bC$ of the family $(C_i)_{i \in I}$ if and only if the following two equivalent conditions are satisfied:%\footnote{{Note that by using the involution we see that Condition~\ref{item_sum_characterization_1} is actually equivalent to Condition~\ref{item_sum_characterization_2}.}}
\begin{enumerate}
\item\label{item_sum_characterization_1} For every object $D$ of $\bC$ and every {square summable family}   {$(h_i)_{i \in I}$ of morphisms  $h_i\colon D \to C_i$}  %satisfying \eqref{eq_condition_multiplier} 
there exists a unique morphism $h\colon D \to C$ with $e_i^* h = h_i$ for all $i$ in $I$.
\item\label{item_sum_characterization_2} For every object $D$ of $\bC$ and every {square summable family}  {$(h'_i)_{i \in I}$ of morphisms}    $h'_i\colon C_i \to D$ %satisfying \eqref{eq_condition_multiplier_2} 
there exists a unique morphism $h'\colon C \to D$ with $h' e_i = h'_i$ for all $i$ in $I$.
\end{enumerate}
\end{prop}
\begin{proof}
Using the involution we see that Condition~\ref{item_sum_characterization_1}    equivalent to Condition~\ref{item_sum_characterization_2}. 

If $(C,(e_i)_{i \in I})$ is an  orthogonal sum, then the  Conditions~\ref{item_sum_characterization_1} and~\ref{item_sum_characterization_2} are satisfied by Corollary \ref{qergijoqregqewfewfqewfwef1}.

To prove the {converse}
%contrary 
we assume Conditions~\ref{item_sum_characterization_1} and~\ref{item_sum_characterization_2}. We have to show that the associated multiplier  maps   \eqref{wergwergwggwegegr42t} and \eqref{wergwergwggwegegr42t1} are isomorphisms. We consider only the case {of the associated right multiplier map}  \eqref{wergwergwggwegegr42t} since the other case will then follow by using the involution. We fix an object $D$ of $\bC$ and let $R$ be in $\RM(D,C)$. We define a family of morphisms $(h_i)_{i \in I}$ with $h_i\colon D \to C_i$ for every $i$ in $I$ by setting $h_i \coloneqq R_{C_i}(e_i^*)$. By Lemma \ref{qergijoqregqewfewfqewfwef}.\ref{qerjgqkrfewfewfqfefqef} we see that {the family $(h_{i})_{i\in I}$ is square summable.}
%\eqref{eq_condition_multiplier} is satisfied. 
Condition~\ref{item_sum_characterization_1} then implies the existence of a unique morphism $h\colon D \to C$
whose associated right multiplier is $R$.
 %which {is} the preimage of $R$ under the transformation \eqref{wergwergwggwegegr42t}.
This shows that the associated right multiplier map \eqref{wergwergwggwegegr42t} is bijective.
\end{proof}

  In the case of $W^{*}$-categories orthogonal sums have   a particularly simple characterization.   Let $\bC$ be in $\Ccat$, $(C_{i})_{i\in I}$ be a family of objects in $\bC$, and $(C,(e_{i})_{i\in I})$ be a pair of an object and a mutually
 orthogonal family of   isometries $e_{i} \colon C_{i}\to C$.  
  In view of Remark \ref{euwifhqiufqwfewffqewffqfef} below the following
 proposition is equivalent to \cite[Thm. 5.1]{fritz}. Using the definition of an orthogonal sum in a $W^{*}$-category according to \cite[Sec. 6]{ghr}
 it even becomes a tautology.

 \begin{prop}\label{qr3gijoerwgergwefwerferwf}
 If 
  $\bC$ is a $W^{*}$-category, then  the following assertions are equivalent.
  \begin{enumerate}
  \item  \label{ejrgiowegewrgergrefwefe1}We have
 $\sum_{i\in I} \sigma(e_{i})\sigma(e_{i})^{*}=1_{\sigma(C)}$ weakly 
 for some unital, normal and faithful representation 
 $\sigma \colon \bC\to \Hilb(\C)^{\la}$ such that $\bC\cong \bC_{\sigma}''$.
 \item\label{ejrgiowegewrgergrefwefe2} 
 The pair  $(C,(e_{i})_{i\in I})$ represents the sum of the family $ (C_{i})_{i\in I}$ in $\bC$.
\item\label{ejrgiowegewrgergrefwefe3} We have $\sum_{i\in I}e_{i}e_{i}^{*}=1_{C}$ in the $\sigma$-weak   topology.
\item \label{ejrgiowegewrgergrefwefe4}
We have $\sum_{i\in I}e_{i}e_{i}^{*}=1_{C}$ in the weak  operator topology.
 \end{enumerate}
 \end{prop}
 \begin{proof}
 \ref{ejrgiowegewrgergrefwefe1}$\Rightarrow$\ref{ejrgiowegewrgergrefwefe2}: 
 Let    
 $\sigma:\bC\to \Hilb(\C)^{\la}$  be unital, normal and faithful representation such that $\bC\cong \bC_{\sigma}''$ and $\sum_{i\in I} \sigma(e_{i})\sigma(e_{i})^{*}=1_{\sigma(C)}$ weakly.
 %We assume that $ (C,(e_{i})_{i\in I})$ represents an AV-sum of $(C_{i})_{i\in I}$. Then $\sum_{i\in I } e_{i}e_{i}^{*}=1_{C}$
%strictly.  
  %We choose a faithful representation of $\sigma:\bM\bK\to \Hilb(\C)^{\la}$  and identify  $\bC\cong \bM\bK_{\sigma}^{''}$. 
%n particular, for $D,D'$ in $\bC$  we can consider $\Hom_{\bC}(D,D')$ as a subset of
%$ \Hom_{\Hilb(\C)^{\la}}(\sigma(D),\sigma(D'))$.
%The net $(\sum_{i\in J}  e_{i}e_{i}^{*})_{J }$ for $J$ running over the finite subsets of $I$   in $\bC$ is bounded. 
%By Lemma   \ref{ewtiojgowtrgwergwergf} we know that
%it converges  to $1_{C}$ in the weak operator topology.
   This
  implies that $(\sigma(e_{i}))_{i\in I}$ represents
$\sigma(C)$ as the orthogonal Hilbert sum of the family $(\sigma(C_{i}))_{i\in I}$.
 We  will use Proposition \ref{lem_sum_characterization_morphisms} in order to see that $ (C,(e_{i})_{i\in I})$ represents the sum of the family
$  (C_{i})_{i\in I}$ in $\bC$. To this end
we  verify the Condition   \ref{lem_sum_characterization_morphisms}.\ref{item_sum_characterization_1}. %The argument for 
 %Condition   \ref{lem_sum_characterization_morphisms}.\ref{item_sum_characterization_2} is similar.

We let $D$ be an object of $\bC$ and consider a square-summable family
$(h_{i})_{i\in I}$ of morphisms $h_{i}:D\to C_{i}$ in $\bC$. 
Then we must show that there exists a unique 
morphism $h:D\to C$ in $ \bC$ such that $e_{i}^{*}h=h_{i}$ for all $i$ in $I$.

%Let $x$ be in $ \sigma(D)$. Then  
We  want to define the operator $\tilde h:\sigma(D)\to \sigma(C)$ by  $$\tilde h(x):=\sum_{i\in I} \sigma(e_{i})  \sigma(h_{i})(x)\ , \quad x\in \sigma(D)\ .$$
 For every finite subset $J$ of $I$ we have 
$$\sum_{i\in J}\| \sigma(e_{i}) \sigma( h_{i})(x) \|^{2}= \| \sum_{i\in J} \sigma(e_{i}) \sigma( h_{i})(x) \|^{2}\le  \|  \sum_{i\in J}  h_{i}^{*}h_{i} \| \|x\|^{2}\le  \sup_{J\subseteq I}  \|\sum_{i\in J} h_{i}^{*}h_{i}\|  \|x\|^{2}\ .$$
 Since $(h_{i})_{i\in I}$ is square summable  the sum defining $\tilde h(x)$ converges in norm and defines 
 a bounded 
     bounded operator $\tilde h$. % with $$\|h\|\le \sup_{J\subseteq I} \sqrt{\|\sum_{i\in j} h_{i}^{*}h_{i}\| }\ .$$ 
 By construction we have
 $\sigma(e_{i})^{*}\tilde h=\sigma(h_{i})$. 
 
 It remains to show that $\tilde h$ belongs to $\sigma(\bC)$. 
 Let $v$ be in $\End_{\Rep(\bC)}(\sigma)$.
 Then   for every $i$ in $I$  we have the equalities 
 $ v_{C_{i}}\sigma(h_{i})=\sigma(h_{i})v_{D}$ and  
   $\sigma(e_{i})v_{C_{i}}=v_{C}\sigma(e_{i})$.
 This implies that $ v_{C}\sigma(e_{i}) \sigma(h_{i})=\sigma(e_{i}) \sigma( h_{i}) v_{D}$.
 Since $v_{C}$ and $v_{D}$ are continuous
 we conclude that
 $v_{C}\tilde h= \tilde hv_{D}$. Since $v$ is arbitrary this implies that $\tilde h$ belongs to $\bC_{\sigma}''$. 
  We finally let  $h$ in $\bC$ be the unique morphism such that $\sigma(h)=\tilde h$.
  
  \ref{ejrgiowegewrgergrefwefe2}$\Rightarrow$\ref{ejrgiowegewrgergrefwefe3}:  
  The uniformly bounded net $(\sum_{i\in J}e_{i}e_{i}^{*})_{J}$ of non-negative operators with $J$ running over finite subsets of $I$ is monotoneously increasing. Hence $\sum_{i\in I}e_{i}e_{i}^{*}=\lim_{J}\sum_{i\in J}e_{i}e_{i}^{*}$ $\sigma$-weakly converges  to  $\sup_{J}\sum_{i\in J}e_{i}e_{i}^{*}=:p$  in $\End_{\bC}(C)$.
  We have $ pe_{i}= e_{i}=1_{C}e_{i}$ for all $i$ in $I$.
  Hence by Corollary \ref{kor_morphisms_sum_characterization} we have 
 $p =1_{C}$.
   
   \ref{ejrgiowegewrgergrefwefe3}$\Rightarrow$\ref{ejrgiowegewrgergrefwefe4}:
   The implication follows from the fact that the $\sigma$-weak topology contains the weak operator topology.  
   
   \ref{ejrgiowegewrgergrefwefe4}$\Rightarrow$\ref{ejrgiowegewrgergrefwefe1}:  
  This implication is clear since by definition of the weak operator topology  any 
  normal $\sigma$ is continuous  for the  weak operator topology  on the domain and the weak topology on the target.
   %For the converse we   assume that   $ (C,(e_{i})_{i\in I})$ represents a sum in $\bW\bM\bC$. In order to show that it also represents an AV-sum we must show that $\lim_{i\in I} e_{i}e_{i}^{*} =1_{C}$ strictly. 
%Let $f:D\to C$ be a morphism in $\bC$. Then we must show that 
%$\lim_{i\in I}  e_{i}e_{i}^{*}f=f$ in norm. By Corollary \ref{qergijoqregqewfewfqewfwef1}.\ref{qerjgqkrfewfewfqfefqef111}  we know that  $(e_{i}^{*}f)_{i\in I}$ is a square summable family of morphisms in $\bC$.  
 \end{proof}

Using $W^{*}$-envelopes we can give the following extrinsic characterization  of orthogonal sums in an arbitrary  unital $C^{*}$-category 
  $\bC$. Let  $(C_{i})_{i\in I}$ be a family of objects in $\bC$ and $(C,(e_{i})_{i\in I})$ be a pair of an object $C$ of $\bC$ and a mutually orthogonal family of isometries $e_{i}:C_{i}\to C$.

\begin{prop} \label{geoirgjiskldfnmgsre} The pair
$(C,(e_{i})_{i\in I})$ is an orthogonal sum of the family $(C_{i})_{i\in I}$ in $\bC$ if 
it is an orthogonal sum of this family in $\bW\bC$ and one of the following equivalent conditions is satisfied:
\begin{enumerate}
\item \label{fpogjskfpdgsfgsfdgsfdg} For every object $D$ of $\bC$ and morphism $f$ in $\Hom_{\bW\bC}(C,D)$ the condition that $fe_{i}\in  \Hom_{\bC}(C_{i},D)$ for all $i$ in $I$ implies   that  $f\in \Hom_{\bC}(C,D)$.
\item \label{fpogjskfpdgsfgsfdgsfdg1}  For every object $D$ of $\bC$ and morphism $f $ in $\Hom_{\bW\bC}(D,C)$ the condition that $e_{i}^{*}f\in  \Hom_{\bC}(D,C_{i})$ for all $i$ in $I$   implies   that    $f\in \Hom_{\bC}(D,C)$.
\end{enumerate}
\end{prop}
\begin{proof}
We use the involution in order to see that the two conditions are equivalent.

Assume that $(C,(e_{i})_{i\in I})$ is an orthogonal sum of the family $(C_{i})_{i\in I}$  in $\bW\bC$.
Let $(f_{i})_{i\in I}$ be a square summable family of morphisms $f_{i}:D\to C_{i}$ in $\bC$. By 
  Corollary
\ref{qergijoqregqewfewfqewfwef1}.\ref{qerjgqkrfewfewfqfefqef111}
there exists a unique morphism 
  $f$   in $\Hom_{\bW\bC}(D,C)$ such that   $f_{i}:=
e_{i}^{*}f$ for every $i$ in $I$. Using    Condition \ref{fpogjskfpdgsfgsfdgsfdg1}
we conclude that $f$ belongs to $\bC$.

Since this holds for any $D$ and square integrable family as above we can 
  now use  Proposition \ref{lem_sum_characterization_morphisms}.\ref{item_sum_characterization_1} in order to conclude that
$(C,(e_{i})_{i\in I})$ is an orthogonal sum of the family $(C_{i})_{i\in I}$  in $\bC$.
 \end{proof}

% 
% \begin{rem}\label{wrthoijrtpgfregregwfref}
% The assumption of    Proposition \ref{qr3gijoerwgergwefwerferwf} 
% is in particular satisfied if $\sum_{i\in I} e_{i}^{*}e_{i}=1_{C}$ operator weakly.
%  This is because by the definition  of the operator weak topology on $\bC$  every normal representation is continuous with respect to the 
%  weak operator topology on its domain and the weak topology in the target. 
%% 
%%  By definition  of the operator weak topology on $\bC$  the representation    $\sigma$ is    continuous with respect to the weak operator topology on $\bC$ and the weak topology on $\Hilb(\C)$. Since $\sum_{i\in I} e_{i}^{*}e_{i}=1_{C}$ operator weakly   this 
%% 
%% shows that one can weaken the the assumption of this proposition further and only assume
%% that $\sum_{i\in I} \sigma(e_{i})\sigma(e_{i})^{*}=1_{C}$
%% for some unital, normal and faithful representation 
%% $\sigma:\bC\to \Hilb(\C)^{\la}$ such that $\bC=\bC_{\sigma}^{''}$.
%\end{rem}

Prior to the present paper  a notion of an orthogonal sum  of a family $(C_{i})_{i\in I}$ of objects in a unital $C^{*}$-category $\bC$   has already been introduced in   \cite{fritz}.  The family of objects gives rise to a functor  
  $S:\bC\to \Ban$ which sends an object $D$ in $\bC$   to the Banach space
  of square-summable (see Definition \ref{oijfqoifweeqfewfqefqfewf}) families $(h'_{i})_{i\in I}$ of morphisms $h'_{i} \colon C_{i}\to D$  
  with the norm from \eqref{eq_multiplier_norm_L}. 
  \begin{ddd}[{\cite[Def. 4.2]{fritz}}]\label{qwifhofqewfqwfewfq}
  An orthogonal sum of the family $(C_{i})_{i\in I}$ is an object $C$ of $\bC$ together with an isomorphism $\Hom_{\bC}(C,-)\cong S$ of $\Ban$-valued functors.
  \end{ddd}
   
 \begin{rem}\label{euwifhqiufqwfewffqewffqfef}
   It is an immediate consequence %\fuli{AV: no longer: Its false} 
   of Proposition \ref{lem_sum_characterization_morphisms} and the norm calculation in Lemma \ref{qergijoqregqewfewfqewfwef}.\ref{qerjgqkrfewfewfqfefqef1} that
   the Definitions  \ref{qwifhofqewfqwfewfq} and \ref{erogwfsfdvbsbfdbsdfbsfdbv} provide equivalent notions of {orthogonal} sums. 
   
For the  different notion of an orthogonal sum according to Antoun--Voigt see Section \ref{sec_comparison_other_definition} below.
   \hB
   \end{rem}

We  now present  two illustrative examples of    orthogonal sums of   infinite families of objects. In view of Remark \ref{euwifhqiufqwfewffqewffqfef} more examples are given by \cite[Prop.  5.3]{fritz}. The case of Hilbert $A$-modules will be discussed separately in Section \ref{hilbert}.

%Afterwards we  show  {that in the  case of Hilbert modules over unital $C^*$-algebras
%the abstract notion of an orthogonal sum introduced above coincides with the usual one.}

\begin{ex}\label{example_basic_orthogonal_sum_diagonal}% \fuli{AV. not an example}
Let $X$ be a countable infinite set.
%, and let $A$ be a sub-$C^*$-algebra of the bounded operators $B(\ell^2(X))$. Assume that $A$ contains the identity $\id_{\ell^2(X)}$ and the compact operators $K(\ell^2(X))$. This ensures that the $C^*$-category $\bX$ defined below is well-defined and unital.
We define a $C^*$-category $\bX$ as follows:
{
\begin{enumerate}
\item {objects:} The set of objects of $\bX$ is the set $X \cup \{X\}$. 
\item morphisms:  The   morphism spaces  are defined as subspaces of $B(\ell^{2}(X))$.   
For every two subsets $Y,Y'$ of $X$ we can consider
$B(\ell^{2}(Y),\ell^{2}(Y'))$ as a block subspace of $B(\ell^{2}(X))$ in the natural way.
\begin{enumerate}
\item For $x$ in $X$ the  algebra  $\End_{\bX}(x)$ is the  subalgebra $B(\ell^{2}(\{x\}))$. It is isomorphic to $\C$.
\item  For    {$x,x^\prime$ in} $X$ with $x \not= x^\prime$ we set 
 $\Hom_{\bX}(x,x^\prime) \coloneqq 0$. 
 \item The  algebra $\End_{\bX}(X)$ is the  subalgebra of diagonal operators in $B(\ell^{2}(X))$. It is isomorphic to $\ell^{\infty}(X)$.  
\item For $x$ in $X$  we let   $\Hom_{\bX}(x,X)$ be the subspace of 
$ B(\ell^{2}(\{x\}),\ell^{2}(X))$  generated by the canonical inclusion $e_{x}$. Similarly we let 
 $\Hom_{\bX}(X,x)$ be the subspace of $B(\ell^{2}(X), \ell^{2}(\{x\}))$ generated by $e_{x}^{*}$. These spaces are one-dimensional.
\end{enumerate}
\item The composition and the involution of $\bX$ are induced from $B(\ell^{2}(X))$.\end{enumerate}}
 
 We claim that $(X, (e_x)_{x \in X})$ is the orthogonal sum in $\bX$ of the family $(x)_{x \in X}$.
 In order to show the claim we use Proposition \ref{lem_sum_characterization_morphisms}.
 We consider only one of the four cases. The remaining are left as an exercise.
  Let $(h'_{x})_{x\in X}$ be a square summable family of morphisms $h'_{x} \colon x\to X$.  Then for every $x$ in $X$ we have $h'_{x}=\lambda'_{x}e_{x}$ for some uniquely determined $\lambda'_{x}$ in $\C$.  Furthermore $\|\sum_{x\in J} h'_{x}h_{x}^{\prime,*}\|= \|\sum_{x\in J} |\lambda'_{x}|^{2} e_{x}e_{x}^{*}\|=\max_{x\in J} |\lambda'_{x}|^{2}$ for every finite family $J$ of $X$. 
 Since $(h'_{x})_{x\in X}$ is square summable 
  it follows that $(\lambda'_{x})_{x\in X}$ is uniformly bounded. 
The unique morphism $h' \colon X\to X$ with $h'e_{x}=h'_{x}$  for all $x$ in $X$ is then given by the diagonal operator on $\ell^{2}(X)$ given by multiplication by the bounded function  $x\mapsto \lambda'_{x}$.
\hB
\end{ex}
 
\begin{ex}\label{example_basic_orthogonal_sum_full}%\fuli{AV: not an example}
Let $X$ be a countable infinite set. 
%and let $A$ be a sub-$C^*$-algebra of the bounded operators $B(\ell^2(X))$. Assume that $A$ contains the identity $\id_{\ell^2(X)}$ and the compact operators $K(\ell^2(X))$. This ensures that the $C^*$-category $\bX$ defined below is well-defined and unital.
We define a $C^*$-category $\bX^\prime$ as follows:
\begin{enumerate}
\item {objects:} The set of objects of $\bX^\prime$ is the set $X \cup \{X\}$.
\item morphisms: As in Example \ref{example_basic_orthogonal_sum_diagonal} the   morphism spaces  are defined as subspaces of $B(\ell^{2}(X))$.
\begin{enumerate}
 \item For $x,x'$ in $X$  we set  $\Hom_{\bX'}(x,x')$ as   $B(\ell^{2}(\{x\}),\ell^{2}(\{x'\}))$. It is one-dimensional.
\item For any $x$ in $X$ we set  $\Hom_{\bX^\prime}(x,X) \coloneqq B(\ell^2(\{x\}),\ell^2(X))$ and $\Hom_{\bX^\prime}(X,x) \coloneqq B(\ell^2(X), \ell^2(\{x\} ))$.
 \item $\End_{\bX'}(X) \coloneqq B(\ell^{2}(X))$.
 \end{enumerate}
\item The composition and the involution of $\bX^\prime$ are induced from $B(\ell^{2}(X))$.
\end{enumerate}
For every $x$ in $X$ we denote by $e_x$ the {canonical inclusion 
  in % ion 
%element $\IC \mapsto \delta_x$ in $\Hom_{\bX^\prime}(x,X) = 
$B(\ell^2(\{x\}),\ell^2(X))$.} We claim that 
  $(X, (e_x)_{x \in X})$ is the orthogonal sum in $\bX^\prime$ of the family $(x)_{x \in X}$.    In order to show this claim  we consider $\bX'$ as a full subcategory 
  of $\Hilb(\C)$ in the obvious manner. Since $(\ell^{2}(X), (e_{x})_{x\in X})$ is the classical orthogonal sum of  Hilbert $\C$-modules $(x)_{x\in X}$, by Proposition \ref{rthkprghtregwergwerg}
  it is
  an  orthogonal sum of $(x)_{x\in X}$ in $\Hilb(\C)$ in the sense of Definition \ref{erogwfsfdvbsbfdbsdfbsfdbv}.  By Corollary \ref{wetoighjowtgwergwregwergwerg}  it is also an orthogonal sum of this family in $\bX'$.
    \hB
\end{ex}

Let $\bC$ be in $\Ccat$.
If $(C,(e_{i})_{i\in I})$ is an orthogonal sum of a family
$(C_{i})_{i\in I}$ of objects in $\bC$, then  one can ask  whether subsets of $I$
determine subobjects of $C$ representing the orthogonal sum of    the corresponding subfamilies, and whether $C$ is the sum of its  subobjects 
corresponding to a partition of $I$. The following results show that
all expected  assertions are true.
 
Let $(C_{i})_{i\in I}$ be a family of objects in $\bC$ and assume that it admits  an orthogonal sum  $(\bigoplus_{i\in I}C_{i}, (e_{i})_{i\in I})$.   Let $J$ be a subset of $I$. Then we consider the family $(e^{*}_{j})_{j\in J}$ of morphisms $e_{j}^{*}\colon \bigoplus_{i\in I} C_{i}\to C_{j}  $.
 If we  extend this family by zero to a family indexed by $I$, then 
 {by}  Corollary \ref{qergijoqregqewfewfqewfwef1}.\ref{qerjgqkrfewfewfqfefqef111} applied with $D:= \bigoplus_{i\in I} C_{i}$ we can form $p:=\sum_{j\in J} e_{j}e_{j}^{*}$ in $\End_{\bC}( \bigoplus_{i\in I} C_{i})$. Note that $pe_{j}=e_{j}$ for all $j$ in $J$ and $pe_{i}=0$ for $i$ in $I\setminus J$.
 \begin{lem}\label{ethgiowgergwergerg}\mbox{} %\fuli{AV: still true: of course $p$ is multiplier.}
\begin{enumerate}
 \item \label{tihgjretigergewrgwerg} $p$ is a projection. 
 \item \label{jwqefiofwqefewfqfqewfqefq}If $p$ is effective and $(E,u)$ presents an image of $p$  {(Definition~\ref{gewergregwegegwerg})}, then $(E, (u^{*}e_{i} )_{i\in J})$
 represents the sum of the subfamily 
 $(C_{j})_{j\in J}$.
 \item \label{ewiorjgoewgergerwfe} If $\bC$ admits very small sums and the projections $e_{j}e_{j}^{*}$ are effective for  all
 $j$ in $J$, then $p$ is effective.
\end{enumerate}
\end{lem}
\begin{proof} For every $k$ in $I$ we have 
%We consider $f$   in $\Hom_{\bC}^{\fin}(D', \bigoplus_{i\in I} C_{i})$. Then for a sufficiently large finite subset $J'$  of $J$ (depending on $f$) we have, using that $(e_{i})_{i\in I}$ is a mutual orthogonal family of isometries,
$$
p^{2}e_{k}=\sum_{j\in J} e_{j}e_{j}^{*} \big(\sum_{i\in J}e_{i}e_{i}^{*} e_{k} \big) = c_{k}e_{k}=  %&=&\sum_{i,j\in J'} e_{j}e_{j}^{*} 
  %e_{i}e_{i}^{*}f\\
  \sum_{j\in J} e_{j}e_{j}^{*} e_{k}\\=pe_{k}\ ,$$
  where $$c_{k}:=\left\{\begin{array}{cc}1&k\in J\\0&else\end{array} \right.\ .$$  Using Corollary \ref{kor_morphisms_sum_characterization}.\ref{item_morphisms_sum_characterization_one1} we  conclude that $p^{2}=p$.  We verify similarly  that $p^{*}=p$. This finishes the proof of Assertion \ref{tihgjretigergewrgwerg}

In order to show Assertion \ref{jwqefiofwqefewfqfqewfqefq}
we  assume that $p$ is effective, and that $u\colon E\to \bigoplus_{i\in I} C_{i}$ presents an image of $p$. 
The family $( u^{*}e_{i})_{i\in J}$ of morphisms $u^{*}e_{i} \colon C_{i}\to E$  is a mutually orthogonal family of isometries.
%Indeed, for all $i,j$ in $J$ we have
%$$(u^{*}e_{j})^{*}u^{*}e_{i}={e_{j}^{*}u u^{*} e_{i}}=e_{j}^{*}p e_{i}=e_{j}^{*}e_{i}\, .$$  

We now consider the left and right multipliers for $(E,(u^{*}e_{i})_{i\in J})$.
We must show that the associated multiplier morphisms  %morphisms \eqref{wergwergwggwegegr42t} and  \eqref{wergwergwggwegegr42t1} 
$$m^{E,R}_{D} \colon \Hom_{\bC}(D,E)\to \RM(D,E)\, , \quad m^{E,L}_{D} \colon \Hom_{\bC}(E,D)\to \LM(E,D)$$
are isomorphisms for all $D$ in $\bC$, where we added a superscript $E$ in order to indicate the dependence on $E$.
  It suffices to  consider the case of $m_{D}^{E,R}$. %morphism $\Hom_{\bC}(D,E)\to \RM(D,E)$.   
  The other case then follows by applying the involution.

    We first show surjectivity. 
Let $R \coloneqq (R_{D'})_{D'\in \Ob(\bC)}$ be in $\RM(D,E)$. Pre-composition with {$u$} induces a map
$$-\circ u \colon \Hom_{\bC}^{\fin}( (\bigoplus_{i\in I} C_{i}, (e_{i})_{i\in I}),D' )\to \Hom_{\bC}^{\fin}(  (E,(u^{*}e_{i})_{i\in J}),D')$$
and therefore extends by continuity to a map
 $$-\circ u \colon \mathbb{K}( (\bigoplus_{i\in I} C_{i}, (e_{i})_{i\in I}),D')\to \mathbb{K}((E,(u^{*}e_{i})_{i\in J}),D')\, .$$
Then $Ru \coloneqq (R_{D'}\circ (-\circ u))_{D'\in \Ob(\bC)}$ belongs to $\RM(D , \bigoplus_{i\in I} C_{i})$. Hence there exists a uniqely determined 
 morphism $f\colon D\to \bigoplus_{i\in I} C_{i}$ such that $m^{C,R}_{D}(f)=Ru$. Then 
%This multiplier corresponds under  \eqref{wergwergwggwegegr42t} to a  unique  Then 
$m^{E,R}_{D}(u^{*}f)=R$.  

Assume now that $f\colon D\to E$ is a morphism such that  $m_{D}^{E,R}(f)=0$.  
This means that $hf=0$ for all objects $D'$ and generators $h$ of $\mathbb{K}(E,D')$. Note that ${e^{*}_{i}}u$ is a generator of  $\mathbb{K}(E,C_{i})$. Hence in
 particular we have $ e_{i}^{*}uf=0$ for all $i$ in $I$. This implies $uf=0$ and therefore $f=u^{*}uf=0$.
 
 We finally show Assertion \ref{ewiorjgoewgergerwfe}.
Using the assumption that $e_{j}e_{j}^{*}$ is effective for every $j$ in $J$ we choose   an image  $(D_{j},u_{j})$
of $e_{j}e_{j}^{*}$. Since $\bC$ admits very small orthogonal sums by assumption   we find  an orthogonal sum $(D,(f_{j})_{j\in J})$ of the family $(D_{j})_{j\in J}$. By Corollary \ref{qergijoqregqewfewfqewfwef1}  we get an isometry $v:=\sum_{j\in J} e_{j}f_{j}^{*}:D\to C$. The pair $(D,v)$ represents an image of $p$. 
   \end{proof}

The following results fit into the present discussion but will only  be used in the follow up {paper} \cite{coarsek}.
 Let $\bC$ be in $\Ccat$.  Let
  $(C_{i})_{i\in I}$ be a family  of objects in $\bC$ and $(C,(e_{i})_{i\in I})$ be an orthogonal sum of the family. Let furthermore  $(J_{k})_{k\in K}$ be a partition of the set $I$. For every $k$ in $K$ we can form the projection
  $p_{k}:=\sum_{i\in J_{k}} e_{i}e_{i}^{*}$ by Lemma \ref{ethgiowgergwergerg}.
\begin{lem}\label{etgoijoergergwegwergwrg}%\fuli{AV: still true}
Assume  that for any $k$ in $K$ 
 the projection $p_{k }$ is effective with image $(E_{k},u_{k})$.
 Then the sum of the family $(E_{k})_{k\in K}$ exists and is represented by
 $(C, (u_{k})_{k\in K})$.   
\end{lem}

\begin{proof}
Since the members of the family $(J_{k})_{k\in K}$ are mutually disjoint we have $p_{k}p_{k'}=0$ for all $k, k'$ in $K$ with $k\not=k$. This implies that
$(u_{k})_{k\in K}$ is a mutually orthogonal family of isometries.

We must show that the  associated multiplier morphisms \eqref{wergwergwggwegegr42t} and \eqref{wergwergwggwegegr42t1} 
\begin{align*}
m_{D}^{\prime,R} \colon \Hom_{\bC}(D,C) & \to \RM (D,(C,(u_{k})_{k\in K}))\,,\\
m_{D}^{\prime,L} \colon \Hom_{\bC}(C,D) & \to \LM ((C,(u_{k})_{k\in K}),D)
\end{align*}
are isomorphisms for all  objects $D$ in $\bC$. Here the superscript $\prime$ is added in order to distinguish these maps  from the associated multiplier maps $m_{D}^{R}$ and $m_{D}^{L}$ of $(C,(e_{i})_{i\in I})$.

We again consider the case of right multipliers.  The case of left multipliers then follows by applying the involution.
We have  inclusions $$\mathbb{K} ((C,(e_{i})_{i\in I}),D')\subseteq \mathbb{K} ((C,(u_{k})_{k\in K}),D')$$ for all $D'$ in $\bC$. 
Hence we have a restriction map $!$ fitting into the diagram
$$\xymatrix{&\Hom_{\bC}(D,C)\ar[dl]_-{m^{\prime,R}_{D}} \ar[dr]^-{m^{R}_{D}}_{\cong}&\\ \RM (D,(C,(u_{k})_{k\in K}))\ar[rr]^-{!}&& \RM (D,(C,(e_{i})_{i\in I}))}$$ 
This already implies that the map   $m^{\prime,R}_{D}$ is injective. 

 We will now show   that $!$ is injective.  
To this end we assume that $R=(R_{D'})_{D'\in \Ob(\bC)}$ is in $\RM (D,(C,(u_{k})_{k\in K}))$ and is sent to zero by $!$.
We must show that $R_{E_{k}}(u_{k}^{*})=0$ for all $k$ in $K$.
We have for all $i$ in $J_{k}$ that 
$$ e_{i}^{*}u_{k}R_{E_{k}}(u_{k}^{*})= R_{C_{i}}(e_{i}^{*}u_{k} u_{k}^{*})= R_{C_{i}}(e_{i}^{*}p_{k} )=
R_{C_{i}}(e_{i}^{*}  )=0\, ,$$
where for the last  equality we use the assumption on $R$. This implies by Lemma  \ref{ethgiowgergwergerg}.\ref{jwqefiofwqefewfqfqewfqefq}  and the uniqueness assertion in Corollary \ref{qergijoqregqewfewfqewfwef1}.\ref{qerjgqkrfewfewfqfefqef111} (applied to the sum $(E_{k},(e_{i}u^{*}_{k})_{i\in J_{k}})$ and the family 
of morphisms $ (e_{i}^{*}u_{k}R_{E_{k}}(u_{k}^{*}))_{i\in J_{k}}$)
 that  $R_{E_{k}}(u_{k}^{*})=0$.
 
The injectivity of $!$ implies  by a diagram chase that $m^{\prime,R}_{D}$  is surjective. 
\end{proof}

Let $\bC$ be in $\Ccat$,   
 $C$ be an object of $\bC$, and $(p_{i})_{i\in I}$ be a   mutually orthogonal    family of projections on $C$.
\begin{ddd}\label{rgoigsgsgfgg}
We say that $C$ is the orthogonal sum of the images of the family of projections if the following are satisfied:
\begin{enumerate}
\item \label{oiejfioqwjefewfewfqef} For every $i$ in $I$ the projection $p_{i}$ is effective  (see Definition \ref{regiuhjigwergwgrewgrg}).
 \item If $(D_{i},u_{i})$ is a choice of an image of $p_{i}$ for every $i$ in $I$ (see Definition \ref{gewergregwegegwerg}), 
 then $(C,{(u_{i})_{i\in I}})$ represents the orthogonal sum of the family $(D_{i})_{i\in I}$.
  \end{enumerate}
 \end{ddd}
 
Note that the validity of the conditions in   Definition \ref{rgoigsgsgfgg} does not depend on the choices involved in the images and the direct sum.

\section{Antoun--Voigt sums}
\label{sec_comparison_other_definition}

In Definition \ref{peofjopwebgwregwgreg}    we recall the notion of an orthogonal sum according to Antoun--Voigt \cite[Defn.\ 2.1]{antoun_voigt}.  The precise relation between this concept and the notion of an orthogonal sum  introduced in Definition \ref{erogwfsfdvbsbfdbsdfbsfdbv} will then be explained  in Theorem \ref{wergijowergerfwrefwerf}.   We furthermore provide some more technical statements which will be used in the subsequent papers \cite{coarsek} and \cite{bel-paschke}.

We consider $\bK$ in $\nCcat$, let  $\bM\bK$  in $\Ccat$ denote  a multiplier category  characterized in   Definition \ref{weighwoegerwferfwef}, and let 
  $\bC:= \bW\bM\bK$ be a $W^{*}$-envelope of $\bM\bK$ introduced in Definition \ref{wijtgoewgewgefw}.
By Proposition \ref{jigogewrgeferfref} we also have an isomorphism $\bC\cong \bW^{\mathrm{nu}}\bK$. By choosing appropriate models for the multiplier category and the $W^{*}$-envelope we can and will assume that 
we have  isometric inclusions \begin{equation}\label{qrefqewfdewdqwedwed}
\bK\subseteq \bM\bK\subseteq \bC
\end{equation}
which are identities on the level of objects.
Morphisms in   $\bM\bK$ will be called multiplier morphisms.
By Proposition \ref{jigogewrgeferfref} the category $\bM\bK$ is the idealizer of $\bK$ in $\bC$.

Let $(C_i)_{i \in I}$ be a family of objects of $\bK$, and let $C$ be an object of $\bK$ together with a family $(e_i)_{i \in I}$ of mutually orthogonal multiplier morphisms $e_i\colon C_i \to C$.
Recall the Definition \ref{erguihierverqerf} of the strict topology on the morphism spaces of the multiplier category $\bM\bK$.
\begin{ddd}[{\cite[Defn.\ 2.1]{antoun_voigt}}]\label{peofjopwebgwregwgreg} The pair 
$(C,(e_i)_{i \in I})$ is an orthogonal AV-sum in $\bK$ of the family $(C_i)_{i \in I}$ if the sum $\sum_{i \in I} e_i e_i^*$ converges strictly to the identity multiplier morphism of $C$.
\end{ddd}
We use the term AV-sum (AV stands for Antoun--Voigt) in order to distinguish this notion from the one defined in Definition \ref{erogwfsfdvbsbfdbsdfbsfdbv}.
  %If we want to stress the original category we   talk about
  %   $AV$-sums in $\bK$.
     \begin{rem}\label{rem_sums_unital_nonunital}
 An orthogonal AV-sum of a family $(C_{i})_{i\in I}$ of unital objects (see Definition \ref{ebgiojoigrevsfdvsdfvsfdvsdfv})   in $\bK$  can only be a unital object  of $\bK$  if the set of non-zero members of the family 
  is finite. In fact, if the $AV$-sum is unital,   then Lemma \ref{lem_unital_strict_equals_norm}.\ref{eqroigqwgwefqewfqf1}  implies 
  that the sum of mutually orthogonal projections 
  $\sum_{i \in I} e_i e_i^*$   converges in norm. This is only possible if the sum has finitely many non-zero terms.
   \hB
\end{rem}

 \begin{theorem}\label{wergijowergerfwrefwerf}
 If $(C,(e_{i})_{i\in I}) $ is an AV-sum in $\bK$  of a family  $(C_{i})_{i\in I}$, then
 it is also an orthogonal  sum in $\bC$
of this family in the sense of Definition \ref{erogwfsfdvbsbfdbsdfbsfdbv}.   
 \end{theorem}
\begin{proof}
 %We  choose a unital, normal and faithful representation 
 %$\sigma:\bC\to \Hilb(\C)^{\la}$ such that $\bC=\bC_{\sigma}^{''}$.
 
We assume that $ (C,(e_{i})_{i\in I})$ represents an AV-sum of $(C_{i})_{i\in I}$. Then $\sum_{i\in I } e_{i}e_{i}^{*}$ converges strictly to $1_{C}$.
  %We choose a faithful representation of $\sigma:\bM\bK\to \Hilb(\C)^{\la}$  and identify  $\bC\cong \bM\bK_{\sigma}^{''}$. 
%n particular, for $D,D'$ in $\bC$  we can consider $\Hom_{\bC}(D,D')$ as a subset of
%$ \Hom_{\Hilb(\C)^{\la}}(\sigma(D),\sigma(D'))$.
The net $(\sum_{i\in J}  e_{i}e_{i}^{*})_{J }$ for $J$ running over the finite subsets of $I$   in $\bC$ is bounded. 
By Lemma   \ref{ewtiojgowtrgwergwergf} we know that
it converges  to $1_{C}$ in the weak operator topology.
Now Proposition \ref{qr3gijoerwgergwefwerferwf} implies the assertion.
 \end{proof}

%Let $\bK$ be in $\nCcat$. % $(C_{i})_{i\in I}$ be a family of objects in $\bK$, and $(C,(e_{i})_{i\in I})$ be a pair of an object in $\bK$ and a mutually orthogonal family of multiplier isometries.

%\uli{the following corollary is  probably wrong!}
% 
%%take $\bC=\HilbK(\ell^{2})$ and $C_{i}=e_{i}K(\ell^{2})$. Then $K(\ell^{2})$ is the }
% \begin{kor} If   $(C,(e_{i})_{i\in I})$ is an orthogonal AV-sum of a family $(C_{i})_{i\in I}$ in $\bK$, then it is
% an orthonormal sum of the family $(C_{i})_{i\in I}$ in $\bM\bK$. 
% \end{kor}
%\begin{proof}
%By Theorem \ref{wergijowergerfwrefwerf} we know that $(C,(e_{i})_{i\in I})$
%is an orthonormal sum of the family $(C_{i})_{i\in I}$ in $\bW^{\mathrm{nu}}\bK\cong \bW\bM\bK$, see Theorem \ref{jigogewrgeferfref}. In order to conclude the assertion of the corollary, in view of Proposition \ref{geoirgjiskldfnmgsre}
%we must check that for any morphism $f:C\to D$ in $\bW\bK$ the condition
%$fe_{i}\in \bM\bK$ implies that $f\in \bM\bK$. Since $\sum_{i\in I}e_{i}e_{i}^{*}$ strictly converges to $\id_{C}$, for  every morphism $v$ in $\bK$ with target $C$   the sum
%$\sum_{i\in I} fe_{i}e_{i}^{*}v$ converges in norm to $fv$. Since the summands belong to $\bK$ we conclude that $ fv\in \bK$. For every morphism $w$ in $\bK$ with domain $D$ the sum $\sum_{i\in I} wfe_{i}e_{i}^{*}$
%\uli{not yet clear what to do}
%\end{proof}

 %Orthogonal AV-sums  enjoy the following universal property. 
 
 Let $(C_i)_{i \in I}$ be a family of objects of $\bK$, and let  $(C,(e_i)_{i \in I})$  represent an AV-sum of this family.
Let  $(h_i)_{i \in I}$ be a uniformly bounded family of morphisms $h_i\colon D   \to C_i$. %Similarly, let 
% $(h'_i)_{i \in I}$ be a uniformly bounded family of morphisms $h'_i\colon C_i \to D$.  %Similarly, let $(h_i)_{i \in I}$ be a uniformly bounded family of morphisms $h_i\colon D   \to C_i$.
\begin{lem} \label{wejbhviwevfvsdfvfdv}
\mbox{}
\begin{enumerate}
\item  \label{wetkgowpegopregwpegwergwreg0} If   $\sum_{i \in I} e_{i}h_{i}$ converges strictly, then there is a unique multiplier morphism $h \colon D \to C$ with $e_{i}^{*}h = h_i$ for all $i$ in $I$.  
\item \label{wetkgowpegopregwpegwergwreg}If $ \sum_{i\in I} e_{i}h_{i}$ converges strictly, then
$(h_{i})_{i\in I}$ is square summable.
\item \label{wetkgowpegopregwpegwergwreg1} If $(h_{i})_{i\in I}$ is square summable, then
$\sum_{i\in I} e_{i}h_{i}$ converges right-strictly.
\end{enumerate}
\end{lem}
\begin{proof} 

We start with Assertion \ref{wetkgowpegopregwpegwergwreg0}.
By assumption the sum converges strictly  to a multiplier morphism $ h$. Since the composition of multiplier morphisms is separately continuous for the strict topology 
  we have  $e_{i}^{*}h= \sum_{j \in I} e_{i}^{*}e_j  h_j =h_{i}$ for every $i$ in $I$. 
 If $h' $ is a second multiplier morphism  from $C$ to $D$ such that $e_{i}^{*}h' =h_{i}$ for all $i$ in $I$, then 
\[h=   \left(\sum_{i}e _{i}e_{i}^{*}\right)h  =\sum_{i}e _{i}e_{i}^{*} h= \sum_{i}e _{i}e_{i}^{*}h'= \left(\sum_{i}e _{i}e_{i}^{*}\right)h'=h'\, \ .\]

We now show Assertion \ref{wetkgowpegopregwpegwergwreg}.
If $\sum_{i\in I} e_{i}h_{i}$ converges strictly, then  by Assertion \ref{wetkgowpegopregwpegwergwreg0}
  there exists a unique multiplier morphism
$h:C\to D$ such that $e_{i}^{*}h=h_{i}$ for all $i$ in $I$.
By Theorem \ref{wergijowergerfwrefwerf} the pair  $(C,(e_{i})_{i\in I})$ represents a sum of the family $(C_{i})_{i\in I}$ in the sense of Definition \ref{erogwfsfdvbsbfdbsdfbsfdbv}.  Since $h$ is in particular a morphism in $\bC$ we conclude by  Corollary  \ref{qergijoqregqewfewfqewfwef1}.\ref{qerjgqkrfewfewfqfefqef111} that $(h_{i})_{i\in I}$ is square summable.

We  finally  show  Assertion \ref{wetkgowpegopregwpegwergwreg1}. We 
 assume that  $(h_{i})_{i\in I}$ is square summable. We must show that 
  for every
$f$ in $\Hom_{\bK}(C,D)$   the sum $\sum_{i\in I}   f e_{i} h_{i}
$ converges in norm. To this end 
first observe that the net 
$(\sum_{i\in J} e_{i} h_{i})_{J}$ for $J$ running over the finite subsets of $I$ is uniformly bounded since 
$$\|\sum_{i\in J} e_{i} h_{i}\|^{2}=
\|\sum_{i,j\in J} h_{j}^{*}e_{j}^{*}e_{i} h_{i}\| =\|\sum_{i\in J} h_{i}^{*}  h_{i}\| $$
and $(h_{i})_{i\in I}$ is square summable.
It therefore suffices to show the norm convergence of 
 $\sum_{i\in I}   f e_{i} h_{i}
$  for $f$ of the special form $f=\sum_{i\in J} e_{i}e_{i}^{*}g$ for some $g$ in $\Hom_{\bK}(D,C)$ since the subspace of those morphisms is dense in $ 
\Hom_{\bK}(D,C)$.  But if $f$ has this form, then
$\sum_{i\in J'}   f e_{i} h_{i}=\sum_{i\in J}   f e_{i} h_{i}$
for all finite subsets  $J'$ of $I$ containing $J$.  
\end{proof}
\begin{rem}\label{wijb9eovdfvsdfvsdvsdfv} 

In the situation of Lemma \ref{wejbhviwevfvsdfvfdv}.\ref{wetkgowpegopregwpegwergwreg1} in general we can not conclude that 
  $\sum_{i\in I}   e_{i} h_{i}$  converges left-strictly in the sense that
$ \sum_{i\in I}   e_{i} h_{i}f'$ converges in norm for all objects $E$ of $\bC$ and $f'$ in $\Hom_{\bK}(E,D)$.
An example  where this happens will be given in Example  \ref{qefguifggsdfgsdfgsfdg} below. This example also shows that the converse of  the assertion of Theorem \ref{wergijowergerfwrefwerf} is not true in general.
\end{rem}

%
%  \begin{lem}  
% If   the sum $\sum_{i \in I} e_{i}h_{i}$ converges strictly, then there exists a unique multiplier morphism $h \colon D \to C$ with $e_{i}^{*}h = h_i$ for all $i$ in $I$.  
% \end{lem}
% \begin{proof}
% By assumption the sum converges strictly  to a multiplier morphism $ h$. Since the composition of multiplier morphisms is separately continuous for the strict topology 
%  we have  $e_{i}^{*}h= \sum_{j \in I} e_{i}^{*}e_j  h_j =h_{i}$ for every $i$ in $I$. 
% If $h' $ is a second multiplier morphism  from $C$ to $D$ such that $e_{i}^{*}h' =h_{i}$ for all $i$ in $I$, then 
%\[h=   \left(\sum_{i}e _{i}e_{i}^{*}\right)h  =\sum_{i}e _{i}e_{i}^{*} h= \sum_{i}e _{i}e_{i}^{*}h'= \left(\sum_{i}e _{i}e_{i}^{*}\right)h'=h'\, .\qedhere\]
%  \end{proof}
%
 
 The following was stated in \cite{antoun_voigt}.
 \begin{prop}\label{weijgowgwerfwrefwr}
 An AV-sum is unique up to unique unitary multiplier morphism.
 \end{prop}
\begin{proof}
Let $(C,(e_{i})_{i\in I})$ and $(C',(e'_{i})_{i\in I})$ both represent AV-sums of the family $(C_{i})_{i\in I}$ in $\bK$. We will show    that $\sum_{i\in I} e_{i}'e_{i}^{*}$ converges strictly to a unitary  multiplier isomorphism $u:C\to C'$ such that $e_{i}'=ue_{i}$ for all $i$ in $I$.

The family $(e^{*}_{i})_{i\in I}$ is square integrable.  
By Lemma  \ref{wejbhviwevfvsdfvfdv}.\ref{wetkgowpegopregwpegwergwreg1}
the sum   $ \sum_{i\in I} e_{i}'e_{i}^{*}$ converges right strictly. 
Since also $(e^{\prime,*}_{i})_{i\in I}$ is  square integrable we can conclude, using the involution and  Lemma  \ref{wejbhviwevfvsdfvfdv}.\ref{wetkgowpegopregwpegwergwreg1} again, that $ \sum_{i\in I} e_{i}'e_{i}^{*}$ also converges left strictly. 
We then get $u:= \sum_{i\in I} e_{i}'e_{i}^{*}$ with the desired properties from 
 Lemma  \ref{wejbhviwevfvsdfvfdv}.\ref{wetkgowpegopregwpegwergwreg0}. 
%
%
%We first note that the net $(\sum_{i\in J} e_{i}'e_{i}^{*})_{J}$ with $J$ running over the finite subsets of $I$ is uniformly norm bounded.  Therefore
%$(e^{*}_{i})_{i\in I}$ is square integrable.
%
%
%Then we will show that $\sum_{i\in I} e_{i}'e_{i}^{*}$ converges strictly.  Lemma \ref{wejbhviwevfvsdfvfdv} applied to the AV-sum  $(C',(e'_{i})_{i\in I})$   and the family $(e_{i}^{*})_{i\in I}$
%then provides 
% the desired  unitary multiplier  isomorphism $u:C\to C'$  such that $e_{i}^{\prime,*}u=e^{*}_{i} $.
%
%We must show that for every $f$ in $\Hom_{\bK}(D,C)$ the sum $\sum_{i\in I} e_{i}'e_{i}^{*}f$ converges in norm, and that for   every $f'$ in $\Hom_{\bK}(C',D)$ the sum $ \sum_{i\in I} f'e_{i}'e_{i}^{*}$ coverges in norm.
%
%We only consider the first case. The argument for the second is analogous.
%We first observe that the net $(\sum_{i\in J} e_{i}'e_{i}^{*})_{J}$ with $J$ running over the finite subsets of $I$ is uniformly norm bounded.   
%Hence   it suffices to check norm convergence of $ \sum_{i\in I} e_{i}'e_{i}^{*}f$   for   the subset  of morphisms of the form
%$f=\sum_{i\in J'}e_{i}e_{i}^{*}h$ for some morphism $h:D\to C$ in $\bK$ and finite subset $J'$ of $I$. In fact, since   $(C,(e_{i})_{i\in I})$ is an AV-sum the this subset is dense in $\Hom_{\bK}(D,C)$.
% But if $f$ has this special form, then
%$  \sum_{i\in J} e_{i}'e_{i}^{*}\left( \sum_{i\in J'}e_{i}e_{i}^{*}h\right)=\sum_{i\in J'}
%e_{i}'e_{i}^{*}h$  for all $J$ containing $J'$.
\end{proof}

 We consider a commutative diagram \begin{equation}\label{qwefoijofwqed}
\xymatrix{\bK\ar[r]\ar[d]^{\phi}& \bM\bK\ar[d]^{\psi}\ar[r]&\bC:=\bW\bM\bK\ar[d]^{\bW\psi}\\ \bL\ar[r]& \bM\bL\ar[r]&\bD:=\bW\bM\bL}
\end{equation}
 in $\nCcat$ where $\psi$ is unital and strictly continuous on bounded subsets. By Proposition \ref{stkghosgfgsrgsegs} this situation,
e.g., arises if $\phi$ is a non-degenerate morphism in $\nCcat$, and
$\psi=\bM\phi$.
   The following is an AV-analog of Corollary \ref{wetoighjowtgwergwregwergwerg}.
   
    \begin{kor}\label{wriojtghowtrgfrefw}\mbox{}\begin{enumerate}
\item \label{ertihjeorger90g0weerg} $\psi$ preserves AV-sums.
\item \label{ertihjeorger90g0weerg1}  If $\psi$ is faithful, then it detects AV-sums.
\end{enumerate}
 \end{kor}
 \begin{proof}
 We start with Assertion \ref{ertihjeorger90g0weerg}.
 Let $(C_{i})_{i\in I}$ be a family of objects in $\bK$ and
$(C,(e_{i})_{i\in I})$ be a pair of an object  in $\bK$ and a family of multiplier morphisms $e_{i}:C_{i}\to C$.  
 If $(C,(e_{i})_{i\in I}) $ is an AV-sum in $\bK$ of the family  $(C_{i})_{i\in I}$, 
 then $\sum_{i\in I} e_{i}e_{i}^{*}$ strictly converges to $\id_{C}$.
 Since $\psi$ is strictly continuous and unital  
$\sum_{i\in I} \psi(e_{i})  \psi(e_{i})^{*}$  strictly converges to $ 1_{\psi(C)}$.  Hence $(\psi(C),(\psi(e_{i}))_{i\in I}) $ is an AV-sum in $\bL$ of the family  $(\psi(C_{i}))_{i\in I}$.

In order to show   Assertion \ref{ertihjeorger90g0weerg1} assume that  $\sum_{i\in I} \psi(e_{i})  \psi(e_{i})^{*}$  converges strictly to $1_{\psi(C)}$.
 Using that $\psi$ is faithful and unital we can then conclude that 
$\sum_{i\in I}e_{i}e_{i}^{*}$ converges strictly to $1_{C}$. 
% 
%  then 
%we have $\sum_{i\in I}e_{i}e_{i}^{*}=1_{C}$ strictly.
% 
% then $(\psi(C),(\psi(e_{i}))_{i\in I}) $ is an AV-sum in $\bL$ of the family  $(\psi(C_{i}))_{i\in I}$. If $\psi$ is faithful, then the converse conclusion
% is also true.
%  
%Since $\psi$ is strictly continuous and unital we have
%$$\sum_{i\in I} \psi(e_{i})  \psi(e_{i})^{*}=
%\psi( \sum_{i\in I} e_{i}e_{i}^{*})=\psi(1_{C})=1_{\psi(C)}\ .$$
%For the converse,  if $\sum_{i\in I} \psi(e_{i})  \psi(e_{i})^{*}=1_{\psi(C)}$ strictly, then 
%we have $\sum_{i\in I}e_{i}e_{i}^{*}=1_{C}$ strictly.
\end{proof}

% 
%  
 %
%
%  Let $(C_{i})_{i\in I}$ be a family of objects in $\bK$ and
%$(C,(e_{i})_{i\in I})$ be a pair of an object  in $\bK$ and a family of multiplier morphisms $e_{i}:C_{i}\to C$. 

%\begin{ex}\label{wrtjihordbvdfgbdb}
We now consider the AV-analog of  Lemma \ref{ruihqiuwefwefqfqwefq}.
Let $(C_{i})_{i\in I}$ and $(C_{i}')_{i\in I}$ be two families of objects in $\bK$ with the same index set.
 We assume that  they admit  AV-sums $( C, (e_{i})_{i\in I} )$ and $  ( C' , (e'_{i})_{i\in I} )$ in $\bK$. By Theorem \ref{wergijowergerfwrefwerf} 
 they are also orthogonal sums in $\bC$.  Let $(f_{i})_{i\in I}$ be a uniformly bounded family of multiplier morphisms $f_{i} \colon C_{i}\to C'_{i}$.   

\begin{lem}\label{wrtjihordbvdfgbdb}
The morphism $\oplus_{i\in I} f_{i}: C \to  C' $   in $\bC$  from \eqref{ervfdvsdfvfdvfsdvsfdv}
is a multiplier morphism {given by the strictly convergent sum $\sum_{i\in I} e_{i}'f_{i}e_{i}^{*}$.}
 \end{lem}
 \begin{proof}
 The family $(f_{i}e_{i}^{*})_{i\in I}$ is square summable which implies by  Lemma \ref{wejbhviwevfvsdfvfdv}.\ref{wetkgowpegopregwpegwergwreg1}   that 
$\sum_{i\in I} e_{i}'f_{i}e_{i}^{*}$ converges right strictly.  Using the involution and  Lemma \ref{wejbhviwevfvsdfvfdv}.\ref{wetkgowpegopregwpegwergwreg1}  again we 
 see in this  this case that  it also converges left strictly. 
 We conclude by Lemma \ref{wejbhviwevfvsdfvfdv}.\ref{wetkgowpegopregwpegwergwreg0}    that the morphism 
 $\oplus_{i\in I} f_{i}: C \to  C' $   is  
  a multiplier morphism.
 \end{proof}

Let $(C_{i})_{i\in I}$ be a family of objects in $\bK$ and assume that it admits  an AV-sum  $( C, (e_{i})_{i\in I})$. By Theorem \ref{wergijowergerfwrefwerf} it is also an orthogonal sum of this family interpreted in $\bC$.  
   Let $J$ be a subset of $I$ and form the projection  $p:=\sum_{j\in J} e_{j}e_{j}^{*}$ in $\End_{\bC}( C)$ as explained  before Lemma \ref{ethgiowgergwergerg}.
   The following lemma  is the AV-analog of Lemma \ref{ethgiowgergwergerg}.
\begin{lem}\label{ethgiowgergwergerggqewfuzgwef}\mbox{} %\fuli{AV: still true: of course $p$ is multiplier.}
\begin{enumerate}
 \item\label{fdvfdvsdfvsfdvfdvsv}  The projection $p$ is a multiplier morphism. 
 \item \label{jwqefiofwqefewfqfqewfwerferfrwefreqefq} If $p$ is effective  and $(E,u)$ presents an image of $p$ in $\bM\bK$ {(Definition~\ref{gewergregwegegwerg})}, then $(E, (u^{*}e_{i} )_{i\in J})$
 represents the AV-sum of the subfamily 
 $(C_{j})_{j\in J}$.
  \item \label{ewiorjgoewgergerwfeav} If $\bK$ admits very small AV-sums and the projection $e_{j}e_{j}^{*}$ is effective in $\bM\bK$ for every
 $j$ in $J$, then $p$ is effective in $\bM\bK$.
\end{enumerate}
\end{lem}
\begin{proof}
Assertion \ref{fdvfdvsdfvsfdvfdvsv}  is an immediate consequence of Lemma   \ref{wrtjihordbvdfgbdb}
applied to the family
$(f_{i})_{i\in I}$ given by
$$f_{i}:=     \left\{\begin{array}{cc}e_{i}e_{i}^{*}&i\in J\\0&else\end{array}\right. \ .$$
In fact we have $p=\oplus_{i\in I} f_{i}$.

For Assertion \ref{jwqefiofwqefewfqfqewfwerferfrwefreqefq}
we must show that 
$\sum_{i\in J}    u^{*}e_{i}e_{i}^{*}u$ converges strictly to the identity of $E$.
Note that we can replace $J$ by $I$ since the additional summands vanish.
The assumption in \ref{jwqefiofwqefewfqfqewfwerferfrwefreqefq} implies that $u$ is a multiplier morphism. Since $\bK$ is an ideal in $\bM\bK$ for every morphism $f:D\to E$ in $\bK$ we have $uf\in \bK$.  Since  $\sum_{i\in I}     e_{i}e_{i}^{*}$ converges strictly to $1_{C}$  we conclude that 
$\sum_{i\in J}    u^{*}e_{i}e_{i}^{*}uf$ converges in norm to $u^{*}uf=f$.
Hence
$\sum_{i\in J}    u^{*}e_{i}e_{i}^{*}u$ converges left strictly to $1_{E}$. 
Right-strict convergence of $\sum_{i\in J}    u^{*}e_{i}e_{i}^{*}u$ is seen similarly.

We finally show Assertion \ref{ewiorjgoewgergerwfeav}.
Using the assumption that $e_{j}e_{j}^{*}$ is effective for every $j$ in $J$ we choose   an image  $(D_{j},u_{j})$
of $e_{j}e_{j}^{*}$ in $\bM\bK$. Since $\bK$ admits very small AV-sums by assumption   we find  an  AV-sum $(D,(f_{j})_{j\in J}$ of the family $(D_{j})_{j\in J}$. 
Using that $C$ and $D$ are presented as AV-sums we 
 check that that the sum
$v:=\sum_{j\in J} e_{j}f_{j}^{*}:D\to C$ converges strictly.
Using Lemma 
  \ref{wejbhviwevfvsdfvfdv}.\ref{wetkgowpegopregwpegwergwreg0}  we  check that $v^{*}v=\id_{D}$ and $vv^{*}=p$.  Hence
the pair $(D,v)$ then represents an image of $p$ in $\bM\bK$. 
  \end{proof}

Next we consider the AV-analog of Lemma \ref{etgoijoergergwegwergwrg}.
  Let
  $(C_{i})_{i\in I}$ be a family  of objects in $\bK$ and $(C,(e_{i})_{i\in I})$ be an AV-sum of the family. Let furthermore  $(J_{k})_{k\in K}$ be a partition of the set $I$. For every $k$ in $K$ we can form the multiplier projection
  $p_{k}:=\sum_{i\in J_{k}} e_{i}e_{i}^{*}$ by Lemma \ref{ethgiowgergwergerggqewfuzgwef}.\ref{jwqefiofwqefewfqfqewfwerferfrwefreqefq}.

\begin{lem}\label{etgoijoergergwgwergerrewegwergwrg}%\fuli{AV: still true}
Assume  that for any $k$ in $K$ 
 the projection $p_{k }$ is effective in $\bM\bK$ with image $(E_{k},u_{k})$.
 Then the AV-sum of the family $(E_{k})_{k\in K}$ exists and is represented by
 $(C, (u_{k})_{k\in K})$.   
\end{lem}
\begin{proof}
We must show that $\sum_{k\in K} u_{k}u_{k}^{*}$ converges strictly to
$1_{C}$. We have $u_{k}u_{k}^{*}=p_{k}=\sum_{i\in J_{k}} e_{i}e_{i}^{*}$ strictly.
Hence
$\sum_{k\in K} u_{k}u_{k}^{*}=\sum_{k\in K}\sum_{i\in J_{k}} e_{i}e_{i}^{*}=\sum_{i\in I}e_{i}e_{i}^{*}=1_{C}$ strictly. 
 \end{proof}

Let $\bK$ be in $\nCcat$,   
 $C$ be an object of $\bK$, and $(p_{i})_{i\in I}$ be a   mutually orthogonal    family of projections on $\bM\bK$. The following is the AV-analog of Definition \ref{rgoigsgsgfgg}.
\begin{ddd}\label{rgoigsgsgfggav}
We say that $C$ is the AV-sum of the images of the family of projections if the following are satisfied:
\begin{enumerate}
\item \label{oiejfioqwjefewfewfqefav} For every $i$ in $I$ the projection $p_{i}$ is effective  in $\bM\bK$. % (see Definition \ref{regiuhjigwergwgrewgrg}).
 \item $\sum_{i\in I}p_{i}$ converges strictly to $\id_{C}$.
 %  If $(D_{i},u_{i})$ is a choice of an image  in $\bM\bK$ of $p_{i}$ for every $i$ in $I$, % (see Definition \ref{gewergregwegegwerg}), 
% then $(C,{(u_{i})_{i\in I}})$ represents the  AV-sum of the family $(D_{i})_{i\in I}$.
  \end{enumerate}
 \end{ddd}

  \color{black}

\section{Hilbert \texorpdfstring{$\bm{C^{*}}$}{Cstar}-modules} \label{hilbert}

%\color{ForestGreen}\textbf{Hier kann man noch den Vergleich über die Ind-Kategorie von Antoun--Voigt bringen, wenn wir wollen.}\color{black}

 In this section we discuss   the situation \eqref{qrefqewfdewdqwedwed}  for $\bK=\Hilb_{c}(A)$, see Example \ref{ex_hilbert_modules}, for a very small $C^{*}$-algebra $A$.  In Lemma \ref{wetiojgowegfgrefweferf} we first identify $\bM\Hilb_{c}(A)$ with $\Hilb(A)$. In Theorem \ref{ewrgowergwerfgrfwerf}  we then compare orthogonal AV-sums in $\Hilb_{c}(A)$ with classical orthogonal sums of Hilbert $A$-modules.
% Finally we discuss a version of the  Yoneda embedding for $C^{*}$-categories and its compatibility with orthogonal sums. 
 
 The following was asserted in \cite{antoun_voigt}. It generalizes the well-known statement that for a Hilbert $A$-module $C$  the $C^{*}$-algebra $B(C)$ of bounded adjointable operators  on $C$ is the multiplier algebra of the  $C^{*}$-algebra of compact (in the sense of Hilbert $A$-modules) operators $K(C)$ \cite[Thm. 13.4.1]{blackadar}.
 \begin{lem}\label{wetiojgowegfgrefweferf}
 We have a canonical isomorphism $\Hilb(A)\cong \bM\Hilb_{c}(A)$.
 \end{lem}
 \begin{proof}
 We have a morphism 
 $\phi:\Hilb(A)\to \bM\Hilb_{c}(A)$ which is the identity on objects, and  which sends a morphism in $\Hilb(A)$ to the multiplier morphism 
 given by composition with the morphism.
 
 We claim that this morphism is an isomorphism. 
 To this end we construct an inverse $\psi: \bM\Hilb_{c}(A)\to  \Hilb(A)$.
 The following argument is the straightforward generalization  of the proof of \cite[Thm. 13.4.1]{blackadar}.
 
 We take advantage of the following fact. 
 Let $C,D$ be in $\Hilb(A)$  and let $S:C\to D$ and $T:D\to C$ be maps  (of the underlying sets)  such that
 $$\langle Sc,d\rangle=\langle c,Td\rangle$$ for all $c$ in $C$ and $d$ in $D$, then
 $S$ und $T$  are morphism in $\Hilb(A)$ with $S^{*}=T$.
 Let $(L,R):C\to D$ be a  morphism in $\bM\Hilb_{c}(A)$. % (see Section \ref{weigjogfgsgdfgsfdg} for notation). 
 Then we define
 $S:C\to D$ and $T:D\to C$ by
 $$S(c):=\lim_{\epsilon\downarrow 0} L(\theta_{c,c})(c\cdot [\langle c,c\rangle+\epsilon]^{-1})$$
 and
 $$T(d):=\lim_{\epsilon\downarrow 0} R(\theta_{d,d})^{*}(d \cdot [\langle d,d\rangle+\epsilon]^{-1})\ ,$$ see \eqref{ssrgfdsgsreg} for notation. In the following we explain why  the limits exist.
 For any compact $K:C\to C$ and $u:D\to D$ we have 
, using \eqref{tokjpherthertgetrgetrg}, the operator inequality
 $$L(K)^{*}u^{*}uL(K)=K^{*}R(u)^{*}R(u) K \le K^{*}\|R(u)\|^{2}K\le \|L\|^{2}\|u\|^{2} K^{*}K\ .$$
 Letting $u$ run over a normalized approximate unit  we conclude that
 $$L(K)^{*} L(K)  \le \|L\|^{2} K^{*}K\ .$$
Using this operator inequality in the first step
%and $$[\langle c,c\rangle+\epsilon]^{-1}-  [\langle c,c\rangle+\epsilon']^{-1}=[\langle c,c\rangle+\epsilon]^{-1} [\langle c,c\rangle+\epsilon']^{-1} (\epsilon'-\epsilon)$$
 we estimate  for $\epsilon, \epsilon'$ in $(0,\infty)$  \begin{eqnarray*}
\lefteqn{\| L(\theta_{c,c})(c\cdot [\langle c,c\rangle+\epsilon]^{-1}- c\cdot [\langle c,c\rangle+\epsilon']^{-1})\|^{2}}&&\\&\le&
 \|L\|^{2} \|\theta_{c,c}(c\cdot [\langle c,c\rangle+\epsilon]^{-1}- c\cdot [\langle c,c\rangle+\epsilon']^{-1}) \|^{2}\\
 &=&\|L\|^{2} \| c \cdot \langle c ,c\rangle\cdot ( [\langle c,c\rangle+\epsilon]^{-1}-  [\langle c,c\rangle+\epsilon']^{-1}) \|^{2} \end{eqnarray*}
 We now claim that
 $$\lim_{\epsilon\downarrow 0}   c \cdot \langle c ,c\rangle\cdot  [\langle c,c\rangle+\epsilon]^{-1}=c\ .$$
 The family of maps $(S_{\epsilon})_{\epsilon\in (0,\infty)}$ given by 
 $c'\mapsto  S_{\epsilon  }(c'):=c' \cdot \langle c ,c\rangle\cdot  [\langle c,c\rangle+\epsilon]^{-1}$
 is uniformy bounded by $1$. On the dense subset of $C$ of elements of the form
 $c'=c'' \cdot\langle c,c\rangle^{\delta}$ for $\delta$ in $(0,\infty)$ we have
 $\lim_{\epsilon\downarrow 0} S_{\epsilon}(c')=c'$. This implies the claim.

 Using the claim we get 
  $$\lim_{\epsilon,\epsilon'\downarrow 0} \| L(\theta_{c,c})(c\cdot [\langle c,c\rangle+\epsilon]^{-1}- c\cdot [\langle c,c\rangle+\epsilon']^{-1})\|^{2}=0\ .$$
In particular note
% we consider $c$ in $C$ as a morphism $$\hat c:=(a\mapsto c\cdot a)\colon A\to C\ ,$$
%and similarly $\langle c,c\rangle$ as an endomorphism $$\widehat{\langle c,c\rangle} := (a\mapsto \langle c,c\rangle\cdot a)\colon  A\to A\ .$$
%We add subscripts indicating the component of the natural transformation $L$.  We have \begin{eqnarray*}
%L_{C}(\theta_{c,c})(c\cdot [\langle c,c\rangle+\epsilon]^{-1})&=&
%\lim_{u}L_{C}(\theta_{c,c})(c\cdot u\cdot [\langle c,c\rangle+\epsilon]^{-1})\\&=&\lim_{u}
%(L_{C}(\theta_{c,c})\circ \hat c)(u\cdot [\langle c,c\rangle+\epsilon]^{-1})\\&=
%&\lim_{u}L_{A}(\theta_{c,c}\circ \hat c)(u\cdot [\langle c,c\rangle+\epsilon]^{-1})\\&=&\lim_{u}
%L_{A}(\hat c\circ \widehat{\langle c,c\rangle})( u\cdot [\langle c,c\rangle+\epsilon]^{-1})\\&=&\lim_{u}
%L_{A}(\hat c)( \langle c,c\rangle\cdot u\cdot [\langle c,c\rangle+\epsilon]^{-1})\\&=&
%L_{A}(\hat c)( \langle c,c\rangle\cdot  [\langle c,c\rangle+\epsilon]^{-1})\\
%&=&L_{A}(\widehat{ c \cdot \langle c,c\rangle\cdot  [\langle c,c\rangle+\epsilon]^{-1}})\\
%&\stackrel{\epsilon\to 0}{\to}&L_{A}(\hat c)(1_{A)}\ ,
%\end{eqnarray*}
%where $u$ runs over an approximate unit of $A$.
 that $$\lim_{\epsilon\downarrow 0} 
 \theta_{c,c}(c \cdot [\langle c,c\rangle+\epsilon]^{-1})=c\ .$$
 We now calculate using that $\theta_{c,c}$ is selfadjoint
 \begin{eqnarray*}
\langle S(c),d\rangle&=&\lim_{\epsilon\downarrow 0}   \langle L(\theta_{c,c})(c\cdot [ \langle c,c\rangle+\epsilon]^{-1}),d\rangle\\&=&\lim_{\epsilon\downarrow 0}   \langle c\cdot [ \langle c,c\rangle+\epsilon]^{-1}), L(\theta_{c,c})^{*}\theta_{d,d}(d \cdot [\langle d,d\rangle+\epsilon]^{-1})\rangle\\&\stackrel{\eqref{3r2fpokpwerverfw}}{=}&\lim_{\epsilon\downarrow 0}   \langle c\cdot [ \langle c,c\rangle+\epsilon]^{-1},  \theta_{c,c} R(\theta_{d,d})^{*} (d \cdot [\langle d,d\rangle+\epsilon]^{-1})\rangle\\&=&
\lim_{\epsilon\downarrow 0}   \langle \theta_{c,c} (c\cdot[ \langle c,c\rangle+\epsilon]^{-1}),   R(\theta_{d,d})^{*} (d \cdot [\langle d,d\rangle+\epsilon]^{-1})\rangle\\&=&\langle c,T(d)\rangle\ .
\end{eqnarray*}
We define $\psi$ such that it sends $(L,R)$ to $S$.

It remains to  show $\psi$ and $\phi$ are inverses to each other. Note that this also implies automatically that $\psi$ is a morphism in $\Ccat$ so that we do not have to  check this fact separately.

Let $A:C\to D$ be a morphism in $\Hilb(A)$ and $(L,R)=(A\circ-, -\circ A)=\phi(A)$ be the corresponding multiplier morphism.
Then
$$\psi(L,R)(c)=\lim_{\epsilon\downarrow 0} A \theta_{c,c}(c\cdot [\langle c,c\rangle+\epsilon]^{-1})=A(c)\ .$$
This shows that $\psi\circ \phi=\id_{\Hilb(A)}$. 

In order to show that $\phi\circ \psi=\id_{\bM\Hilb_{c}(A)}$
we start with a multiplier morphism $(L,R)$ and let $S:=\psi(L,R)$.  We then consider a compact operator $\theta_{c,e}:E\to C$.  We must show that
$S\theta_{c,e}=L(\theta_{c,e})$. 
For every $d$ in $D$, $f$ in  some object $F$,   and $e'$ in $E$, setting    $c':=\theta_{c,e}(e') $ and using \eqref{3r2fpokpwerverfw}  twice, we  have  \begin{eqnarray*}
\theta_{f,d} S \theta_{c,e}(e') &= &
\lim_{\epsilon\downarrow 0}  \theta_{f,d} L(\theta_{c',c'})(c'\cdot[ \langle c',c'\rangle+\epsilon]^{-1})\\&=&
\lim_{\epsilon\downarrow 0}  R(\theta_{f,d}) \theta_{c',c'}(c'\cdot [\langle c',c'\rangle+\epsilon]^{-1})\\&=&R(\theta_{f,d})(c')\\&=&
R(\theta_{f,d})( \theta_{c,e}(e'))\\&=&
\theta_{f,d}L( \theta_{c,e})(e')\ .
\end{eqnarray*}
 This shows that  $S\theta_{c,e}=L(\theta_{c,e})$. 
 By a similar argument we show that $R(\theta_{f,d})=\theta_{f,d}S$ for any $f$ and $d$ as above.
 This finishes the verification of $\phi\circ \psi=\id_{\bM\Hilb_{c}(A)}$.
  \end{proof}

We let $\bW\Hilb(A)$ be the $W^{*}$-envelope of $\Hilb(A)$ as introduced in Definition \ref{wijtgoewgewgefw}. Then \begin{equation}\label{qerwffqwedwedewdewdqwd}
\Hilb_{c}(A)\subseteq \Hilb(A)\subseteq \bW\Hilb(A)
\end{equation}
  is an instance of \eqref{qrefqewfdewdqwedwed}.

In the following we discuss orthogonal sums and AV-sums in this context.
We first recall the construction of the classical  sum of Hilbert $A$-modules. 
 
\begin{rem} 
Note that orthogonal sums of a family of objects in $\Hilb(A)$ in the sense of Definition \ref{erogwfsfdvbsbfdbsdfbsfdbv} or AV-sums in the sense of Definition \ref{peofjopwebgwregwgreg} are objects of $\Hilb(A)$ with  additional  structure maps that are characterized by certain properties. In contrast, the classical sum of   a family of Hilbert $A$-modules is an object determined uniquely up to unitary isomorphism by the Construction \ref{ejirgowergrefweerf} below. \hB
\end{rem}

\begin{construction}\label{ejirgowergrefweerf}{\em
Let $(C_{i})_{i\in I}$ be a family in $\Hilb(A)$ indexed by a very small set. In order to construct the classical  orthogonal sum of this family  
we start with choosing
 an algebraic direct sum $$C^{\alg} \coloneqq \bigoplus_{i\in I} C_{i}$$  of $A$-right-modules with the $A$-valued scalar product
$$\langle \oplus_{i} c_{i}, \oplus_{i} c'_{i} \rangle \coloneqq \sum_{i\in I} \langle   c_{i},   c'_{i} \rangle_{i} \, ,$$
where $\langle -,-\rangle_{i}$ is the $A$-valued scalar product on $C_{i}$.
We then let $C$ be the closure of $C^{\alg}$ with respect to the norm induced by this scalar product. 
Note that for $c$ in $C$ we have
 \begin{equation}
\label{eq_norm_classical_Hilb_sum}
\|c\|^{2}=\|\sum_{i\in I} \langle e_{i}^{*}(c)  ,e_{i}^{*}(c)  \rangle_{i}\|\,.
\end{equation} 
The scalar product extends by continuity and equips $C$ with the structure of a Hilbert $A$-module. 
We have an obvious mutual orthogonal family $(e_{i})_{i\in I}$ of isometries $e_{i}\colon C_{i}\to C$.
 
 We will say that the pair  $(C,(e_{i})_{i\in I})$ represents   the classical orthogonal sum  of the family $(C_{i})_{i\in I}$ in  $\Hilb(A)$.}\hB
\end{construction}

%We now consider the ideal category $\Hilb_{c}(A)$ of $\Hilb(A)$ of compact operators. Then we have an isometric embedding $\Hilb(A)\to \bM \Hilb_{c}(A)$.  Using this embedding we can interpret $(C_{i})$ and $(C,(e_{i})_{i\in I})$ from above in the multiplier category $  \bM \Hilb_{c}(A)$.

 We now state the main theorem of this section.
 Let $(C_{i})_{i\in I}$ be a family of objects in $\Hilb_{c}(A)$ and
$(C,(e_{i})_{i\in I})$ be a pair of an object in $\Hilb_{c}(A)$   and a family of isometries
$e_{i}:C_{i}\to C$ in $\Hilb(A)$. 
We consider the following assertions: 
\begin{enumerate}
\item \label{q3roigjeiogrqegfewfqewf1}$(C,(e_{i})_{i\in I})$ represents the classical orthogonal sum of the family $(C_{i})_{i\in I}$ in $ \Hilb(A) $.
\item  \label{q3roigjeiogrqegfewfqewf2} $(C,(e_{i})_{i\in I})$ represents an $AV$-sum of the family $(C_{i})_{i\in I}$ in $  \Hilb_{c}(A) $.
\item  \label{q3roigjeiogrqegfewfqewf3}  $(C,(e_{i})_{i\in I})$ represents an orthogonal sum of the family $(C_{i})_{i\in I}$ in $\bW\Hilb(A)$.
\end{enumerate}
In order to interpret  Assertion   \ref{q3roigjeiogrqegfewfqewf2} % stated in \cite{antoun_voigt}
we use the identification of $\Hilb(A)$ with $\bM\Hilb_{c}(A)$ by  Lemma \ref{wetiojgowegfgrefweferf}.

\begin{theorem}  \label{ewrgowergwerfgrfwerf}The Assertions \ref{q3roigjeiogrqegfewfqewf1} \&
\ref{q3roigjeiogrqegfewfqewf2}  are equivalent. Furthermore, both imply 
 Assertion~\ref{q3roigjeiogrqegfewfqewf3}.
\end{theorem}
The proof of the theorem will be  finished later in this section after the verification of  partial statements.

Let $(C_{i})_{i\in I}$ be a family in $\Hilb(A)$ and let   $(C,(e_{i})_{i\in I})$ represent    the classical orthogonal sum  of this family. 
The following assertion  was stated in \cite{antoun_voigt}.
%we use the identification of $\Hilb(A)$ with $\bM\Hilb_{c}(A)$ by  Lemma \ref{wetiojgowegfgrefweferf}.
 
\begin{prop}\label{werjigowegwerfwrefwerf}
The pair $(C,(e_{i})_{i\in I})$ is an AV-sum  in $\Hilb_{c}(A)$ of the family $(C_{i})_{i\in I}$. 
\end{prop}
\begin{proof}
According to Definition \ref{peofjopwebgwregwgreg} we must show that
$\sum_{i\in I}e_{i}e_{i}^{*}$ converges strictly to the identity multiplier morphism of $C$. Let $f:D\to C$ be any morphism in $\Hilb_{c}(A)$.  Then we must show that $\sum_{i\in I} e_{i}e_{i}^{*}f=f$, where the sum converges in norm. Similary, for any morphism $f':C\to D$  in $\Hilb_{c}(A)$  we must show that
$\sum_{i\in I}f'e_{i}e_{i}^{*}=f'$ in norm. 

We consider the first case. The second is analoguous. 
We first observe that for any finite subset $J$ if $I$ we have 
$$\|\sum_{i\in J} e_{i}e_{i}^{*}f\|\le  \|\sum_{i\in J} e_{i}e_{i}^{*}\|\|f\|=\|f\|\ ,$$ since 
$ \sum_{i\in J} e_{i}e_{i}^{*}$ is an orthogonal projection.
Since $f$ is compact it can be approximated in norm by linear combinations of finite-dimensional operators of the form $\theta_{c,d}:D\to C $. % given by
%$d'\mapsto \theta_{c,d}(d'):=c\langle d,d'\rangle_{D}$.
 Therefore it suffices to show  that
\begin{equation}\label{werferfrefwefefrefef}
\sum_{i\in I} e_{i}e_{i}^{*}\theta_{c,d}=\theta_{c,d}
\end{equation}  in  norm  for all $c$ in $C$ and $d$ in $D$. To this end  we use 
  the identity
$$\sum_{i\in J} e_{i}e_{i}^{*} \theta_{c,d}- \theta_{c,d}=\theta_{\sum_{i\in J}e_{i}e_{i}^{*}c-c,d}\ .$$
Using that   $\|\theta_{c,d}\| \le \|c\| \|d\| $ and that  
  $ \sum_{i\in I}e_{i}e_{i}^{*}c=c$ in norm we conclude  \eqref{werferfrefwefefrefef}.
\end{proof}

Let $(C_{i})_{i\in I}$ be a family in $\Hilb(A)$ and let   $(C,(e_{i})_{i\in I})$ represent    the classical orthogonal sum  of this family. 
Then combining Proposition \ref{werjigowegwerfwrefwerf} with 
Theorem \ref{wergijowergerfwrefwerf} we get the following result.
 \begin{kor}\label{wrthgrgwefwer} The pair 
$(C,(e_{i})_{i\in I})$ represents an orthogonal sum in the sense of  Definition \ref{erogwfsfdvbsbfdbsdfbsfdbv}  in $\bW\Hilb(A)$.
\end{kor}

In particular we see that in the  context of \eqref{qerwffqwedwedewdewdqwd}    we have the existence of AV-sums and orthogonal sums in the sense of Definition \ref{erogwfsfdvbsbfdbsdfbsfdbv} 
for every very small family of objects.

 \begin{proof} [Proof of Theorem \ref{ewrgowergwerfgrfwerf}]
By Lemma \ref{werjigowegwerfwrefwerf} we know that  Assertion
 \ref{q3roigjeiogrqegfewfqewf1} implies Assertion \ref{q3roigjeiogrqegfewfqewf2}.
 By Theorem \ref{wergijowergerfwrefwerf}  we know that 
 Assertion
 \ref{q3roigjeiogrqegfewfqewf2} implies Assertion \ref{q3roigjeiogrqegfewfqewf3}.
 We finally show that Assertion \ref{q3roigjeiogrqegfewfqewf2} implies 
 Assertion \ref{q3roigjeiogrqegfewfqewf1}.

 We assume that  $(C,(e_{i})_{i\in I})$ represents the  AV-sum of the family $(C_{i})_{i\in I}$  in $\Hilb_{c}(A)$.  
  Let $(C',(e'_{i})_{i\in I})$ represent the classical sum of the family $(C_{i})_{i\in I}$.  Since  
 $(C',(e'_{i})_{i\in I})$ is also  an AV-sum of the family $(C_{i})_{i\in I}$
 by the implication  \ref{q3roigjeiogrqegfewfqewf1}$\Rightarrow$\ref{q3roigjeiogrqegfewfqewf2}, by  the uniqueness of AV-sums asserted in  Proposition \ref{weijgowgwerfwrefwr}
 there exists a   unique unitary  morphism $u  :C\to C'$ in $\Hilb(A) $ with
 $e_{i}^{\prime, *}u=e^{*}_{i}$.   Hence  $(C,(e_{i})_{i\in I})$ also  represents the classical sum of the family $(C_{i})_{i\in I}$.
 % 
% 
% It remains to show that $u$  belongs to $\Hilb(A)$. 
%  It suffices to show that composition with $u$ preserves $\Hilb_{c}(A)$.
%  Let $x$ be in $C$  and $y$ be in $D$ and consider the operator $\theta_{x,y} :D\to C$ in $\Hilb_{c}(A)$. 
% Then we have  $$ u\theta_{x,y}= u\sum_{i\in J} e_{i}e_{i}^{*}\theta_{x,y}=\sum_{i\in J} e'_{i} \theta_{e_{i}^{*}x,y}$$
% in norm. Hence $u\theta_{x,y}$ is compact. Hence $u$ is a left multiplier of $\Hilb_{c}(A)$. A similar argument also shows that $u$ is a right mutiplier of $ 
%  \Hilb_{c}(A)$. 
\end{proof}

 Let $(C_{i})_{i\in I}$ be  a family of objects in $\Hilb(\C)$ and let
$(C,(e_{i})_{i\in I})$ represent the classical sum.
The following proposition is not a special case of Corollary \ref{wrthgrgwefwer} for $A=\C$ since the inclusion $\Hilb(\C)\to \bW\Hilb(\C)$  is not an isomorphism {(see Lemma~\ref{lem_kn2rlwef} below).}
\begin{prop} \label{rthkprghtregwergwerg}The pair $(C,(e_{i})_{i\in I})$ represents an orthogonal sum  in $\Hilb(\C)$ in   the sense of  Definition \ref{erogwfsfdvbsbfdbsdfbsfdbv}. 
%Classical sums in $\Hilb(\C)$ are sums  in the sense of  Definition \ref{erogwfsfdvbsbfdbsdfbsfdbv}.
\end{prop}
\begin{proof}
The category $\Hilb(\C)$ is a $W^{*}$-category (Example \ref{ex_HilbC_Wstar}). 
%Let $(C_{i})_{i\in I}$ be  a family of objects in $\Hilb(\C)$ and let
%$(C,(e_{i})_{i\in I})$ represent the classical sum.
We have $\sum_{i\in I} e_{i}e_{i}^{*}=1_{C}$ in the weak topology.
Applying Proposition \ref{qr3gijoerwgergwefwerferwf} to the identity representation we  conclude that
$(C,(e_{i})_{i\in I})$ is a sum of the  family $(C_{i})_{i\in I}$
 in the sense of  Definition \ref{erogwfsfdvbsbfdbsdfbsfdbv}.
\end{proof}

\begin{ex}\label{qefguifggsdfgsdfgsfdg}
We show now by an example that Assertion \ref{q3roigjeiogrqegfewfqewf3} does not imply 
 Assertion \ref{q3roigjeiogrqegfewfqewf1} in general.

Assume that  
 $(C,(e_{i})_{i\in I})$ represents the orthogonal  sum of the family $(C_{i})_{i\in I}$  in $\bW\Hilb(A)$.  
  Let $(C',(e'_{i})_{i\in I})$ represent the classical sum of the family $(C_{i})_{i\in I}$.  Since  
 $(C',(e'_{i})_{i\in I})$ is also   a sum of the family $(C_{i})_{i\in I}$, 
 by the implication  \ref{q3roigjeiogrqegfewfqewf1}$\Rightarrow$\ref{q3roigjeiogrqegfewfqewf3} and the uniqueness of  sums stated in Lemma \ref{rgiojqeiovevqeve9}
 there exists a  unique  unitary  morphism $u  :C'\to C$ in $\bW\Hilb(A) $ such that 
 $ ue_{i}'=e_{i}$. The problem is that $u$ does not necessarily belong to $\Hilb(A)$.

Here is a concrete example where this happens. 
We consider the algebra $A:=B(\ell^{2})$ of bounded operators on the separable standard Hilbert space.  For $i$ in $\nat$ we let $p_{i}$ be the projection onto the one-dimensional subspace of $\ell^{2}$ generated by the $i$'th basis vector.

 We consider $C:=B(\ell^{2})$ as an object of $\Hilb(A)$.
We  consider the  submodules $C_{i}:=p_{i}C$ in $\Hilb(A)$ of $C$
and let $e_{i}:C_{i}\to C$ be the canonical inclusions. The adjoint of $e_{i}$ is given by left-multiplication by $p_{i}$.  One can check that the classical sum
of the family $(C_{i})_{i\in I}$ is represented by the pair  $(C', (e'_{i})_{i\in I})$, where $C'$ is the algebra of compact operators on $\ell^{2}$, and
$e_{i}':C_{i}\to C'$ is given by $e_{i}$ which happens to take values in compact operators.   We then have a unique unitary isomorphism $C'\to C$ in $\bW\Hilb(A)$ such that $ue_{i}'=e_{i}$. But this unitary does not belong to $\Hilb(A)$ since otherwise it must be the inclusion $K(\ell^{2})\to B(\ell^{2})$ which does not have an adjoint. 
Alternatively, if $u$ would belong to $\Hilb(A)$, then $(C,(e_{i})_{i\in I})$  also   represents an AV-sum of the family $(C_{i})_{i\in I}$.  But note that $C$ is a unital object of $\Hilb(A)$. This contradicts the observation made in Remark \ref{rem_sums_unital_nonunital}.

We consider the family $(e_{i}^{*})_{i\in I}$ of  morphisms $e_{i}^{*}\colon C\to C_{i}$ in $\Hilb(A)$. This family is square summable and  the sum 
$\sum_{i\in \nat} e_{i}'e_{i}^{*}$ converges right-strictly as shown in Remark \ref{wijb9eovdfvsdfvsdvsdfv}.    But it does not converge
left-strictly.  For, if it converged, then it would determine a unitary isomorphism 
between $C$ and $C'$ in $\Hilb(A)$ which does not exist.
\hB
\end{ex}

\begin{lem}\label{lem_kn2rlwef}
The inclusion $\Hilb(\C)\to \bW\Hilb(\C)$  is not an isomorphism.
\end{lem}
\begin{proof}
Let $H$ be an $\infty$-dimensional Hilbert space and consider the $C^*$-algebra $B(H)$ as a $C^*$-category with a single object. We have a fully faithful functor $\phi\colon B(H) \to \Hilb(\C)$ by mapping the unique object of $B(H)$ to the object $H$ of $\Hilb(\C)$. By Proposition~\ref{rjigowergwregregwfer} the functor $\bW\phi$ is also fully faithful. We consider the diagram
\[
\xymatrix{
B(H) \ar[r]^-{\phi} \ar[d] & \Hilb(\C) \ar[d] \\
\bW B(H) \ar[r]^-{\bW\phi} & \bW\Hilb(\C)
}
\]
Since $\phi$ and $\bW\phi$ are fully faithful, $\Hilb(\C)\to \bW\Hilb(\C)$ being an isomorphism would imply that $B(H) \to \bW B(H)$ is also an isomorphism. Now for a $C^*$-algebra the enveloping von Neumann algebra can be computed as its   double dual \cite[Sec.~III.5.2]{blackadar_operator_algebras} and hence $\bW B(H) \cong B(H)^{\ast\ast}$. But $B(H)$ is not isomorphic to its double dual $B(H)^{\ast\ast}$.
\end{proof}

\section{Isometric embeddings of \texorpdfstring{$\bm{C^*}$}{Cstar}-categories and orthogonal sums}
\label{sec_sums_many_examples}
By Corollary \ref{ewtkohijorthrhdrhrhrdth} the  notion of an orthogonal sum according to Definition \ref{erogwfsfdvbsbfdbsdfbsfdbv} is well adapted to normal morphisms between $W^{*}$-categories.
In this section we discuss  the interaction of the notion of an orthogonal sum with morphisms of $C^{*}$-categories further. The main result is Theorem \ref{weigjiowetgerwerer}. 
%By Corollary \ref{ewtkohijorthrhdrhrhrdth} the  notion of an orthogonal sum according to Definition \ref{erogwfsfdvbsbfdbsdfbsfdbv} is well adapted to normal morphisms between $W^{*}$-categories.
%Indeed, Proposition \ref{qr3gijoerwgergwefwerferwf} has the following consequence.
\begin{kor}[{\cite[Cor.~5.2]{fritz}}]\label{ewtkohijorthrhdrhrhrdth}
Every  morphism in $\Wcat$ preserves arbitrary orthogonal sums.
\end{kor}
\begin{proof}
Let $\phi \colon \bC\to \bD$ be a morphism in $\Wcat$. We consider a  family of objects $(C_{i})_{i\in I}$  in $\bC$  and assume that $(C,(e_{i})_{i\in I})$ represents the orthogonal  sum of this family. %We  must show that $(\phi(C),(\phi(e_{i}))_{i\in I})$ represents the orthogonal sum of the family 
%$(\phi(C_{i})_{i\in I})$ in $\bD$.
 %In fact, 
 By the conclusion \ref{ejrgiowegewrgergrefwefe2}$\Rightarrow$\ref{ejrgiowegewrgergrefwefe3} in Proposition \ref{qr3gijoerwgergwefwerferwf} we know that 
$\sum_{i\in I}e_{i}e_{i}^{*}$ converges $\sigma$-weakly to $\id_{C}$. 
Since $\phi$ is normal and hence $\sigma$-weakly continuous we see that $ \sum_{i\in I} \phi(e_{i})\phi(e_{i})^{*}$ converges $\sigma$-weakly to $\id_{\phi(C)}$. 
Applying 
 Proposition \ref{qr3gijoerwgergwefwerferwf}(\ref{ejrgiowegewrgergrefwefe3}$\Rightarrow$\ref{ejrgiowegewrgergrefwefe2}) we finally
  deduce that  $(\phi(C),(\phi(e_{i}))_{i\in I})$ represents the orthogonal sum of the family 
$(\phi(C_{i})_{i\in I})$ in $\bD$.\end{proof}

 We now turn back to functors between $C^{*}$-categories. 
  The property of being an orthogonal sum of a given family of objects in general depends  {on} the surrounding category. But our 
main result is the following.
\begin{theorem} \label{weigjiowetgerwerer}A fully faithful inclusion in $\Ccat$ detects and preserves orthogonal sums.
\end{theorem}
The proof of this theorem will be deduced from a collection of  more technical results below, some of  which also deal with   non-full subcategories.

 Assume that $\bD$ is a full subcategory   of $\bC$ in $\Ccat$,  that %\fuli{If $\bD$ is a full subcategory of $\bC$ (in $\nCcat$), then we have a restriction
%of multipliers of $\bC$ to multipliers of $\bD$. In order words, multipliers in $\bM\bC$ between objects of $\bD$ are in $\bM\bD$. This is injective. 
%} 
$(C_{i})_{i\in I}$ is a family of objects in $\bD$, and that $(C,(e_{i})_{i\in I})$ is an object of $\bD$ together with a mutually orthogonal  family of isometries $e_{i} \colon C_{i}\to C$. 

\begin{kor}\label{wetoighjowtgwergwregwergwerg}%\fuli{AV:This is true. One must interpret the relevant multipliers in $\bC$ on $\bD$.}
If $(C,(e_{i})_{i\in I})$ represents an orthogonal sum of the family $(C_{i})_{i\in I}$ in $\bC$, then it also  represents an orthogonal  sum of this family  in $\bD$. 
\end{kor}
\begin{proof}
This immediately follows from the characterization of orthogonal sums given in the {Proposition \ref{lem_sum_characterization_morphisms}}
which only involves conditions formulated in the language of $\bD$.
%
%Note that a verification of this corollary by checking the conditions in   Definition \ref{erogwfsfdvbsbfdbsdfbsfdbv} directly  seems to be much more tricky  {since one would have to extend multipliers in $\bD$ to multipliers in $\bC$.}
\end{proof}

%\uli{We then consider   Yoneda type embeddings  of $C^{*}$-categories into   categories of Hilbert $C^{*}$-modules  and use this  to provide a  characterization of orthogonal sums in terms of the classical notion of an orthogonal {sum} of Hilbert $C^{*}$-modules.}

%In this section we discuss how sums behave under isometric embeddings of $C^*$-categories. 
%{In particular we will show in Proposition \ref{kor_compare_orthogonal_sums} that fully faithful embeddings preserve sums.
%We then provide examples of functors which do not preserve sums.} 
 
%\fuli{If $\bD$ is a closed subcategory of $\bC$, then  the relation between $\bM\bD$ and $\bM\bC$ is not clear. If $\bD$ is an ideal, then  
%then $\bM\bC\cap \Ob(\bD)\to \bM\bD$. We again give up unitality, but assume that $\bD$ is an ideal.
%}

%\fuli{ We must consider a situation where $\bM\bD\subseteq \bM\bC$ is a full unital subcategory. The Yoneda functor is an (full) example of this situation. }
 Let  {$\bC$} be in $\Ccat$ and assume that $\bD$ is a {closed} unital sub-$C^*$-category of $\bC$. Let $(C_{i})_{i\in I}$ be a family of objects of $\bD$ and assume that it admits an orthogonal sum $(D, (e^\prime_{i})_{i\in I})$ in $\bD$ and an orthogonal sum $(C, (e_{i})_{i\in I})$ in $\bC$.
 
 \begin{prop}\label{eoitgjeogergwergwegwr}\mbox{}%\fuli{AV:  Unklar, ob das gilt: Man arbeitet in the multiplier categories. Wir m\"ussen annehmen, da{\ss} die $e_{i}'$ in $\bM\bC$ liegen. $\Hom_{\bM\bD}\to \Hom_{\bM\bC}$ existiert nicht.  }\fuli{das richt nach bedingter Erwartung und wird nur in einem von Neumann Kontext funktionieren.}
 \begin{enumerate}
 \item\label{lem_sums_vor_zurueck}
 There exists a unique isometry $h\colon C\to D$ in $\bC$  such that $he_{i}=e'_{i}$   for all $i$ in $I$.
\item \label{qriogqowfgqewfqewfqewfq} For any object $E$ of $\bD$ the maps
\begin{equation}\label{wefwefefewffef}
\Hom_{\bD}(D,E)\to \Hom_{\bC}(C,E)\, , \quad f\mapsto fh 
\end{equation}
and
\begin{equation}\label{wefwefefewffef1}
 \Hom_{\bD}(E,D)\to \Hom_{\bC}(E,C) \, , \quad f\mapsto h^{*}f
\end{equation}
are isometric inclusions.
\item The map
\begin{equation}\label{wefwefefewffef2}
 \End_{\bD}(D)\to \End_{\bC}(C) \, , \quad  f\mapsto h^{*}fh
\end{equation}
identifies the $C^{*}$-algebra $\End_{\bD}(D)$
with a corner of $ \End_{\bC}(C)$.
\end{enumerate}
\end{prop}
Here we omit to write the inclusion map of $\bD$ to $\bC$.
Note that $h$  identifies $C$ with a subobject of $D$ considered as an object of $\bC$.
\begin{proof} 
%\begin{lem}
We start with Assertion~\ref{lem_sums_vor_zurueck}.
%In $\bC$ there exist a unique isometry  $h\colon D \to C$  satisfying
%\begin{enumerate}
%\item\label{item_h_1} 
%$h^* e_j = e_j^\prime$ and  $h e_j^\prime = e_j$ for all $j$ in $I$.
%\item\label{item_h_2} $h e_j^\prime = e_j$ for all $j$ in $I$.
%\item\label{item_h_3} $h h^* = \id_C$.
%\item\label{item_h_4} $h^* h$ is a projection in $\End_{\bC}(D)$ satisfying $(e_j^\prime)^* h^* h e_i^\prime = (e_j^\prime)^* e_i^\prime$ for all $i,j$ in $I$.
%\end{enumerate}
%\end{lem}
We apply the Corollary~\ref{qergijoqregqewfewfqewfwef1}.\ref{qerjgqkrfewfewfqfefqef11111} to the family $(e_{i}^{\prime})_{i\in I}$ of morphisms $  e_i^\prime\colon C_i \to D$ to get a unique morphism $h \colon C \to D$ in $\bC$   satisfying $h e_i = e'_i$ for all $i $ in $I$. 
 %\begin{proof}
%We furhermore apply Corollary~\ref{qergijoqregqewfewfqewfwef1}.\ref{qerjgqkrfewfewfqfefqef111} to the orthogonal sum $(C, (e_{i})_{i\in I})$ in $\bC$ and the family $(h_{i})_{i\in I}$ of morphisms $h_i \coloneqq e^{\prime,*}_i\colon D \to C_i$ {in order} to get a unique morphism $h' \colon D \to C$ satisfying $e_i^* h^{*} = e_i^{\prime,*}$ for all $i $ in $I$. The uniqueness statements in Corollary~\ref{qergijoqregqewfewfqewfwef1} imply that $h^* = h^\prime$.  %This implies Assertion~\ref{item_h_1}.
 The composition 
$h^{*} h\colon C \to C$   satisfies $$ e_{j}^{*}h^{*} h e_i =    (e_j^\prime)^* e_i^\prime = {e_j^* e_i} \colon C_i \to C_j$$ for all $i,j$ in $I$.
By Corollary~\ref{kor_morphisms_sum_characterization}.\ref{item_morphisms_sum_characterization_both} these equalities together imply that $h^{*} h  = \id_C$, i.e., that $h$ is an isometry. In particular, 
$ 
p \coloneqq hh^{*}
 $ is a projection in $\End_{\bC}(D)$. 

%we can conclude that $hh^{*}=\id_{D}$.

%Finally  we get the relation
%$h^{*} e_i^\prime = h^{*} h e_i = e_i$ for every $i$ in $I$.

%\end{proof}

%Fixing $i$ in $I$, the uniqueness statement in Corollary~\ref{qergijoqregqewfewfqewfwef1}.\ref{qerjgqkrfewfewfqfefqef111} (applied to the family of morphisms $(h_j)_{j \in I}\colon C_i \to C_j$\fuli{die auch $h$ zu nennen ist confusing} defined by $h_j\coloneqq \id_{C_i}$ for $j = i$ and $h_j \coloneqq 0$ for $j \not= i$) implies that $h h^\prime e_i$ equals $e_i$. Then the uniqueness statement of Corollary~\ref{qergijoqregqewfewfqewfwef1}.\ref{qerjgqkrfewfewfqfefqef11111} (applied to the family given by $h_j^\prime\coloneqq e_j\colon C_j \to C$ for every $j$ in $I$) implies that $h h^\prime$ is the identity.

%{Since $hh'=\id_{C}$ by Assertion \ref{item_h_3}} we conclude that  $h^\prime h\colon D \to D$ is a projection in $\End_{\bC}(D)$.  {By Assertion~\ref{item_h_1} it satisfies} 
%\[
%(e_j^\prime)^* h^\prime h e_i^\prime = (e_j^\prime)^* h^\prime h h^\prime e_i = (e_j^\prime)^* h^\prime e_i = (e_j^\prime)^* %e_i^\prime
%\]
%for every $i,j$ in $I$, which is Assertion~\ref{item_h_4}.

%It remains to show Assertion~\ref{item_h_2}. For this we just note that
%\[h e_j^\prime = h h^\prime e_j = e_j\]
%for every $j$ in $J$.

%We stay in the situation of Lemma \ref{lem_sums_vor_zurueck} and start now comparing the morphism spaces {involving} the sums $D$ and $C$.

%\begin{rem}\label{rem_homs_spaces_corners}

We now show Assertion~\ref{qriogqowfgqewfqewfqewfq}.  %{In the following 
We will  use the notation
\begin{align*}
\LM_{\bD}(D,E) & :=\Hom^{\bd}_{\Fun(\bD^{\op},\Ban)}(\mathbb{K}_{{\bD}}(-,D), \Hom_{\bD}(-,E))\,,\\
\LM_{\bC}(C,E) & :=\Hom^{\bd}_{\Fun(\bC^{\op},\Ban)}(\mathbb{K}_{{\bC}}(-,C), \Hom_{\bC}(-,E))\,.
\end{align*}
%Note that in $\LM_{\bD}(D,E)$ we are allowed to plug in objects from $\bD$ for the blank entry, and in $\LM_{\bC}(D,E)$ objects from $\bC$ are allowed.
% For $L$  {in} $\LM_{\bC}(D,E)$ we denote by $L|_{\bD}$ its restriction to objects of $\bD$, but note that in general $L|_{\bD}$ is not an element of $\LM_{\bD}(D,E)$.
We define
\[
\LM_{\bC}|_{\bD}(D,E) \coloneqq \Hom^{\bd}_{\Fun(\bD^{\op},\Ban)}(\mathbb{K}_{{\bD}}(-,D), \Hom_{\bC}(-,E))
\] and 
 get a  restriction   map $$-|_{\bD}\colon \LM_{\bC}(D,E) \to \LM_{\bC}|_{\bD}(D,E)\, .$$ 
We also have a canonical isometric inclusion
\begin{equation}\label{f324f23oifj32iof234f243f23f2}
\LM_{\bD}(D,E) \to \LM_{\bC}|_{\bD}(D,E)
\end{equation}  given by the isometric  inclusion of $\Hom_{\bD}(-,E)$ into $\Hom_{\bC}(-,E)$.

 Since $h^{*}e_{i}^{\prime}=e_{i}$ and $he_{i}=e_{i}'$ for all $i$ in $I$, 
 left-composition with $h $ and $h^{*}$  induces  natural transformations $$\ell(h):\mathbb{K}_{{\bC}}(-,C)\to \mathbb{K}_{{\bC}}(-,D)\ ,   \quad    \ell(h^{*})\colon \mathbb{K}_{{\bC}}(-,D)\to \mathbb{K}_{{\bC}}(-,C)\ .$$ We show that 
 these transformations are inverse to each other. First note that $h^{*}h=\id_{C}$ immediately implies that 
  $\ell(h^{*})\circ \ell(h)=\id_{\mathbb{K}_{{\bC}}(-,C)}$.
  Furthermore, since
  $hh^{*}=p$ satisfies $ hh^{*}e^{\prime}_{i}= e^{\prime}_{i}$ for every $i$ in $I$, left composition with $p$ acts as the identity on $ \mathbb{K}_{{\bC}}(-,D)$ and therefore
  also $\ell(h)\circ \ell(h^{*})=\id_{\mathbb{K}_{{\bC}}(-,D)}$.

 Precomposition by $\ell(h)$  and $\ell(h^{*})$ gives an isomorphism  \begin{equation}\label{eq_Rmult_by_h_LM} % -\circ \ell(h)\colon \LM_{\bC}(D,E) \to \LM_{\bC}(C,E)\ , \quad 
-\circ \ell(h^{*})\colon \LM_{\bC}(C,E) \to \LM_{\bC}(D,E)
\end{equation}
with inverse $-\circ \ell(h)$.

Let $E$ be an object of $\bD$. Then we consider the  commutative {diagram}%\fuli{AV: $!$ wird die Einbettung. Ist isometrie.}%factorization
$$\xymatrix{
\Hom_{\bD}(D,E)\ar[d]\ar[drr]^{!} \ar[rr]^-{f\mapsto fh } && \Hom_{\bC}(C,E)\\
\Hom_{\bC}(D,E)  \ar[rr]^-{f\mapsto fp} && \Hom_{\bC}(D,E)p\ar[u]_{g\mapsto gh}
}$$
We must show that the upper horizontal map   is isometric. We   first observe that
the right vertical map is an isometry with inverse $l\mapsto lh^{*}$.  Note that $lh^{*}$ belongs to 
$\Hom_{\bC}(D,E)p$ since $lh^{*}=l(h^{*}h)h^{*}=(lh^{*})p$.

It remains to show that the map marked by $!$ is isometric. 
We consider the diagram
\begin{equation*}
\xymatrix{
\Hom_{\bC}(C,E) \ar[rr]^-{\eqref{wergwergwggwegegr42t1}, m^{L}_{E}}_-{\cong} \ar[d]^-{l\mapsto lh^{*}}_-{\cong} && \LM_{\bC}(C,E)\ar[d]^-{\eqref{eq_Rmult_by_h_LM}}_-{\cong}\\
\Hom_{\bC}(D,E)p \ar[r]^{\incl} &\Hom_{\bC}(D,E) \ar[r]^-{\eqref{wergwergwggwegegr42t1}, m^{L}_{D}}& \LM_{\bC}(D,E) \ar[d]^-{-|_{\bD}}\\
 && \LM_{\bC}|_{\bD}(D,E)\\
\Hom_{\bD}(D,E) \ar[uu]^{!}  \ar[rr]^-{\eqref{wergwergwggwegegr42t1}}_-{\cong} && \LM_{\bD}(D,E) \ar[u]_-{!!}
}
\end{equation*}
%We want to show that the map 
  %marked by $!$   (given by $f\mapsto fp$) is an isometry.  
   The commutativity of the lower hexagon requires that the morphism marked by $!!$ is given by the composition of right composition by $p$ composed with the restriction \eqref{f324f23oifj32iof234f243f23f2}.  
  Thereby multiplying from the right by $p$ on $\mathbb{K}_{{\bD}}(-,D)$  is well-defined  and acts as the identity   since $pe_{i}'=e_{i}'$ for all $i$.  In particular 
   the map marked by $!!$ is  equal to  the canonical inclusion \eqref{f324f23oifj32iof234f243f23f2}. 

%Since $!$ is injective  we can conclude that \eqref{eq_incl_corner}  also injective.

{In order to} prove that the map marked by $!$ is isometric, we first note that all  maps    in the above diagram are non-expansive. Furthermore, the associated left multiplier map \eqref{wergwergwggwegegr42t1} is isometric by Lemma~\ref{lem_sum_isos_isometric}, and the canonical inclusion \eqref{f324f23oifj32iof234f243f23f2}  is isometric as observed above.  %since the inclusion of $\Hom_{\bD}(-,E)$ into $\Hom_{\bC}(-,E)$ is isometric since $\bD$ is closed in $\bC$. 
{The combination of these facts}  implies that $!$  is isometric.

The other assertions of the proposition  are shown by similar arguments.\end{proof}

\begin{rem}
In relation to the very first paragraph of the previous proof:
Note that also 
%Similarly, 
$e_{j}^{\prime,*}hh^{*}e_{i}'=e_{j}^{*}e_{i}=e_{j}^{\prime,*} e_{i}'$ for all $i,j$ in $I$.
But this does not imply that $hh^{*}=\id_{D}$ since $hh^{*}$ is a morphism in $\bC$ and not necessarily belongs to $\bD$.  Since $(D,(e_{i}')_{i\in I})$ represents a sum in $\bD$ the characterization of endomorphisms of  $D$  in terms of its matrix components in Corollary~\ref{kor_morphisms_sum_characterization}.\ref{item_morphisms_sum_characterization_both}   only applies to morphisms in $\bD$. \hB
 \end{rem}

For the next proposition we retain the notation introduced before Proposition \ref{eoitgjeogergwergwegwr}.

\begin{prop}\label{lem_surjectivity_corners}\mbox{}%\fuli{AV:  Bleibt richtig unter richtigen annahmen}
\begin{enumerate}
\item\label{item_surjectivity_corners_homs_1} If $\Hom_{\bD}(C_i,E) = \Hom_{\bC}(C_i,E)$ for every $i$ in $I$, then \eqref{wefwefefewffef} is  an isomorphism.
\item\label{item_surjectivity_corners_homs_2} If $\Hom_{\bD}(E,C_i) = \Hom_{\bC}(E,C_i)$ for every $i$ in $I$, then \eqref{wefwefefewffef1} is an isomorphism.
\item\label{item_surjectivity_corners_endo} If $\Hom_{\bD}(C_i, C_j) = \Hom_{\bC}(C_i, C_j)$ for every $i,j$ in $I$, then \eqref{wefwefefewffef2} is an isomorphism.
\end{enumerate}
\end{prop}

\begin{proof}
We show Assertion \ref{item_surjectivity_corners_homs_1}. We have already seen in Proposition \ref{eoitgjeogergwergwegwr}.\ref{qriogqowfgqewfqewfqewfq} that  \eqref{wefwefefewffef}  is an isometric inclusion. It therefore suffices to show that this map is also surjective.

Let $g$ be in $\Hom_{\bC}(C,E)$. 
Then by assumption $ge_{i} \colon C_{i}\to E$ is a morphism in $\bD$ for every $i$ in $I$. We apply Corollary \ref{qergijoqregqewfewfqewfwef1}.\ref{qerjgqkrfewfewfqfefqef11111} (to the orthogonal sum $D$ in $\bD$) to the family of morphisms $(g e_i)_{i\in I}$ in order to get a  morphism $f^\prime\colon D \to E$ in $\bD$ with $f^\prime e_i^\prime = g e_i$ for every $i$ in $I$. By Proposition \ref{eoitgjeogergwergwegwr}.\ref{lem_sums_vor_zurueck}  the composition $f^\prime h$ satisfies $f^\prime h e_i = f^\prime e_i^\prime$ for every $i$ in $I$, and hence Corollary~\ref{kor_morphisms_sum_characterization}.\ref{item_morphisms_sum_characterization_one} implies that $f^\prime h = g$. Therefore $f^\prime$ is a preimage of $g$ under \eqref{wefwefefewffef}.

The {Assertion}~\ref{item_surjectivity_corners_homs_2} follows from {Assertion}~\ref{item_surjectivity_corners_homs_1} by using the involution, and the argument for  {Assertion}~\ref{item_surjectivity_corners_endo} is similar.
\end{proof}

%Adding the condition of fullness to the situation of Lemma \ref{lem_sums_vor_zurueck}, we can even conclude that the orthogonal sum is preserved. We recall the setting of Lemma \ref{lem_sums_vor_zurueck}: let $\bD$ and $\bC$ be in $\Ccat$, assume that $\bD$ is a unital sub-$C^*$-category of $\bC$, let $(C_{i})_{i\in I}$ be a family of objects of $\bD$, and assume that it admits an orthogonal sum $(D, (e^\prime_{i})_{i\in I})$ in $\bD$ and an orthogonal sum $(C, (e_{i})_{i\in I})$ in $\bC$.

We retain the notation introduced before Proposition \ref{eoitgjeogergwergwegwr}.
%Let $\bD$ be a closed unital subcategory of $\bC$. Let $(C_{i})_{i\in I}$ be a family of objects in $\bD$.
 \begin{prop}\label{kor_compare_orthogonal_sums} %\fuli{AV: gilt unter den richtigen Annahmen}
If  $\End_{\bD}(D) = \End_{\bC}(D)$, then the morphism $h \colon C\to D$ constructed in the Proposition \ref{eoitgjeogergwergwegwr}.\ref{lem_sums_vor_zurueck}  is an isomorphism between the orthogonal sums $(C,(e_{i})_{i\in I})$ and $(D,(e_{i}')_{i\in I})$.
  \end{prop}
\begin{proof}
%{By assumption we can choose  an orthogonal sum  
 %$(D, (e_{i}^\prime)_{i\in I})$ in $\bD$ and an   orthogonal sum $(C, (e_{i})_{i\in I})$ in $\bC$.}
 %We then show that $(D, (e_{i}^\prime)_{i\in I})$ is also an orthogonal sum in $\bC$.}
    %If the isometric embedding of $\bD$ into $\bC$ is full , then $(D, (e_{i}^\prime)_{i\in I})$ represents an orthogonal sum in~$\bC$ of the family $(C_{i})_{i\in I}$ and hence is uniquely unitarily isomorphic in~$\bC$ to $(C, (e_{i})_{i\in I})$.
%We apply Lemma~\ref{lem_sums_vor_zurueck} in order  to get  an isometry $h\colon D \to C$  intertwining the structure maps of the sums.
 It suffices to show that $ hh^{*}=\id_{D}$ in $\End_{\bC}(D)$. By assumption 
  $h h^{*}$ belongs to $\End_{\bD}(D)$.
Hence we may apply Corollary~\ref{kor_morphisms_sum_characterization}.\ref{item_morphisms_sum_characterization_both} to the sum $(D, (e_{i}^\prime)_{i\in I})$ in the category $\bD$ and the identities $e^{\prime,*}_{i}hh^{*}e_{j}'=e^{*}_{i}e_{j}=e^{\prime,*}_{i} e_{j}'$ for all $i,j$ in $I$  in order  to conclude
%We apply, similarly as above, the uniqueness statements of Corollary~\ref{qergijoqregqewfewfqewfwef1} to conclude 
that $hh^{*}=\id_{D}$. 
\end{proof}

 \begin{proof}[Proof of Theorem \ref{weigjiowetgerwerer}]
 A fully faithful inclusion detects orthogonal sums by Corollary \ref{wetoighjowtgwergwregwergwerg}. It preserves orthogonal sums
 by Proposition \ref{kor_compare_orthogonal_sums}.
  \end{proof}

\begin{ex}\label{ex_sum_becomes_different}%\fuli{AV: m\"u{\ss}te man modifizieren}
In this example we construct an inclusion $\bD\subseteq \bC$ where the inclusion $\End_{\bD}(D) \subseteq \End_{\bC}(D)$ is proper and $h$ is not an isomorphism. This shows that the assumption in Proposition \ref{kor_compare_orthogonal_sums} can not be dropped.

%We will deliberately fail the step in the proof of Proposition~\ref{kor_compare_orthogonal_sums} which needs the fullness of $\bD$ in $\bC$. In the proof we used Lemma~\ref{lem_sums_vor_zurueck} to get a morphism $h$ and fullness of $\bD$ in $\bC$ was used to conclude that the projection $h^* h$ is already in $\bD$. This allowed us to argue that $h^* h$ must be the identity. In this counter-example here we will have the situation that this particular projection will not be a morphism in $\bD$ and especially, it will not be the identity of $D$.

Let $X$ be a countable infinite set. We let $\sim$ be the equivalence relation on the power set $P(X)$ of $X$ given by $A \sim B$ if and only if  the symmetric difference $A \Delta B$ is finite. Let $[-]\colon P(X) \to P(X)/\sim$ be the quotient map. The set $P(X)$ is a Boolean algebra under the operations of forming unions, intersections and complements.  These operations descend to the quotient $P(X)/\sim$. Using Stone's representation theorem for Boolean algebras \cite{representation_boolean_algebra}, we get a set $Y$ and an injective homomorphism of Boolean algebras $s\colon (P(X)/\sim)  \to P(Y)$.
%We choose a section $s \colon P(X)/_{\!\sim} \to P(X)$ of $[-]$ with the following properties:
%\begin{enumerate}
%\item\label{Point_one_counterex_sum_retain} $s([\emptyset]) = \emptyset$.
%\item\label{Point_two_counterex_sum_retain} $s([X]) = X$.
%\item\label{Point_three_counterex_sum_retain} For every $A, B$ in $P(X)$ we have $s([A]) \cap s([B]) = s([A \cap B])$.
%\end{enumerate}
%\textbf{Existiert so ein Schnitt $s$ überhaupt??}

%Let $(e_x)_{x \in X}$ be the canonical orthonormal basis of the Hilbert space $\ell^2(X)$. 
For every $x$ in $X$ let $p_x$ be the orthogonal projection in $B(\ell^2(X))$ onto the one-dimensional subspace spanned by $x$, and for a subset $A$ of $X$ we consider the orthogonal projection $p_A:=\sum_{x \in A} p_x$  in $B(\ell^{2}(X))$ (the sum is  strongly convergent). Analogously we define for every subset $B$ of $Y$ an orthogonal projection $q_B$ in $B(\ell^2(Y))$.

%We let $X,X_{1},X_{2}$ be three copies of $X$ and 
%We consider the $C^{*}$-algebra
%$B(\ell^{2}(Y\cup X)$. 
We will use the notation conventions as in Example \ref{example_basic_orthogonal_sum_diagonal} in order to denote subspaces of  the algebra $B(\ell^{2}(Y\cup X))$. 
We define a $C^{*}$-category $\bC$ as follows:
\begin{enumerate}
\item objects: The set of objects of $\bC$ is $X\cup \{X ,Y\}$.
\item morphisms: The morphisms of $\bC$ are given as subspaces of $B(\ell^{2}(Y\cup X))$ as follows:
\begin{enumerate}
\item $\End_{\bC}(x) \coloneqq B(\ell^{2}(\{x\}))$ for $x$ in $X$.
\item \label{eoijqwofwefewfqwef}$\End_{\bC}(X)$ is the subalgebra of $B(\ell^{2}(Y\cup X))$ generated by the operators
$p'_{A}+q_{s(A)}$ for all subsets $A$ of $X$ and $B(\ell^{2}(X))$, where $p'_{A}$ is $p_{A}$ considered as an element of $\End_{\bC}(X)$. 
\item $\End_{\bC}(Y) \coloneqq B(\ell^{2}(X))$.
\item $\Hom_{\bC}(x,x') \coloneqq B(\ell^{2}(\{x\}),\ell^{2}(\{x'\}))$.
\item $\Hom_{\bC}(x,X) \coloneqq B(\ell^{2}(\{x\}),\ell^{2}(X))$ and  $\Hom_{\bC}(X,x) \coloneqq B(\ell^{2}(X),\ell^{2}(\{x\}))$.
\item \label{qetgoopkpofqwefewfq}$\Hom_{\bC}(x,Y) \coloneqq B(\ell^{2}(\{x\}),\ell^{2}(X))$ and  $\Hom_{\bC}(Y,x) \coloneqq B(\ell^{2}(X),\ell^{2}(\{x\}))$.
\item \label{qeoigjoqerfqfewf}$\Hom_{\bC}(X,Y) \coloneqq B(\ell^{2}(X))$ and  $\Hom_{\bC}(Y,X) \coloneqq B(\ell^{2}(X))$.
\end{enumerate}
\item involution and composition: are induced from $B(\ell^{2}(Y\cup X))$.
\end{enumerate}

We let  $e_{x}$ in $B(\ell^{2}(\{x\}),\ell^{2}(X))$
  be the canonical inclusion of $\ell^{2}(\{x\})$ into $\ell^{2}(X)$. We write
  $e_{x}''$ for the corresponding morphism from $x$ to $Y$ under the identification \ref{qetgoopkpofqwefewfq}.
The pair $(\{Y\},(e''_{x})_{x\in X})$ is an orthogonal sum of 
the family of objects $(x)_{x\in X}$ in $\bC$. This follows from a combination 
 of Example \ref{example_basic_orthogonal_sum_full} and Proposition \ref{kor_compare_orthogonal_sums} applied to the full subcategory  on the objects $X\cup \{Y\}$ of $\bC$.
We now describe an isometric embedding of the category $\bX$ from Example  \ref{example_basic_orthogonal_sum_diagonal} onto a subcategory $\bD$ of $\bC$.
\begin{enumerate}
\item objects: The embedding sends the objects $x$ and $\{X\}$ of $\bX$ to the corresponding objects of $\bC$ with the same name.
\item morphisms: 
\begin{enumerate}
\item The map $\End_{\bX}(x)\to \End_{\bC}(x)$ is given by the  identity of $B(\ell^{2}(\{x\}))$.
\item The map $\End_{\bX}(X)\to \End_{\bC}(X)$ sends the generator $p_{A}$ to $p'_{A}+q_{s(A)}$. 
\item Then map $ \Hom_{\bX}(x,X)\to \Hom_{\bC}(x,X)$ is the canonical inclusion
\[
\C e_{x}\to B(\ell^{2}(\{x\}),\ell^{2}(X))\,.
\]
\item  The  map $\Hom_{\bX}(X,x)\to \Hom_{\bC}(X,x)$    is the canonical inclusion
\[
\C e^{*}_{x}\to B(\ell^{2}(X),\ell^{2}(\{x\}) )\,.
\]
\end{enumerate}
\end{enumerate} The image in $\bD$ of the morphism $e_{x}$ in $\bX$ will be denoted by $e_{x}'$.
In order to see that this is compatible with the composition note that $e_{x}e_{x}^{*}=p_{x} $ in $\bX$ 
is sent to $p'_{x}+q_{s(\{x\})}=p'_{x}=e_{x}'e^{\prime,*}_{x}$  in $\bC$ since $s(\{x\})=\emptyset$ because of  $\{x\}\sim \emptyset$.

It is easy to see that $\End_{\bX}(X)\to \End_{\bC}(X)$ is injective and hence isometric. 
We conclude that $\bD$ is an isometric copy of $\bX$ in $\bC$. Note that $\End_{\bX}(X)\to \End_{\bC}(X)$ is not surjective.
  By   Example \ref{example_basic_orthogonal_sum_diagonal}
the pair $(X,(e_{x}')_{x\in X})$ is an orthogonal sum of the family $(x)_{x\in X}$ in $\bD$. 
The morphism $h \colon Y\to X$ constructed in Lemma \ref{lem_sums_vor_zurueck} is given by the identity of $B(\ell^{2}(X))$ under the identification \ref{qeoigjoqerfqfewf}.
The projection $hh^{*}$ in $ \End_{\bC}(X)$ is given by the image of $1_{B(\ell^{2}(X))}$ in $\End_{\bC}(X)$ under the identification \ref{eoijqwofwefewfqwef}. It is not the identity in $ \End_{\bC}(X)$ since, e.g., $hh^{*}(p'_{X}+q_{s(X)})=p'_{X}\not= p'_{X}+q_{s(X)}$.
 We conclude that $(X,(e'_{x})_{x\in X})$ does not represent the orthogonal sum of $(x)_{x\in X}$ in $\bC$ anymore.
 \hB
\end{ex}

\begin{ex}\label{ex_sum_vanishes}
We retain the notation from Example \ref{ex_sum_becomes_different}. Let $\bE$ be the full subcategory of $\bC$ with the same objects $X\cup \{X\}$ as $\bD$.  This $C^{*}$-category does not have any orthogonal sum anymore.  
%%We define $\bY$ similarly as $\bX$:
%\begin{enumerate}
%\item The set of objects of $\bY$ is the set $X \cup \{X\}$, i.e., the same objects as $\bX$.
%\item morphisms:
%\begin{enumerate}
%\item $\Hom_{\bY}(x,x) \coloneqq \IC$ for all $x \in X$.
%\item $\Hom_{\bY}(x,x^\prime) \coloneqq 0$ for all $x,x^\prime \in X$ with $x \not= x^\prime$.
%\item $\Hom_{\bY}(X,X) \coloneqq B(\ell^2(X) \oplus \ell^2(X))$.
%\item $\Hom_{\bY}(x,X) \coloneqq \IC$ and $\Hom_{\bY}(X,x) \coloneqq \IC$ for every $x \in X$.
%\end{enumerate}
%\item The composition and the involution of $\bY$ become clear when we interpret $\Hom_{\bY}(x,X)$ as the isometric embedding $\IC \to B(\ell^2(X) \oplus \ell^2(X))$ given by $1_{\IC} \mapsto p_x \oplus 0$ for every $x \in X$, where $p_x$ is the projection in $B(\ell^2(X))$ onto the one-dimensional subspace spanned by $e_x$.
%\end{enumerate}
%%%%%
%We construct $\bE$ out of $\bD$ by enlargening the endomorphism space of the object $X$: we declare $\Hom_{\bE}(X,X)$ as the sub-$C^*$-algebra of $B(\ell^2(X) \oplus \ell^2(Y))$ generated by the $C^*$-algebra $M$ together with $B(\ell^2(X)) \oplus 0$ (i.e., it is the same space as $\Hom_{\bC}(X,X)$ in Example~\ref{ex_sum_becomes_different}). The $C^*$-category $\bE$ does not admit any non-trivial sums.
%We construct $\bE$ out of $\bD$ by enlargening the endomorphism spaces: we declare $\Hom_{\bE}(a,b)$ to be $\Hom_{\bC}(a,b)$ for every object $a$ and $b$, where $\bC$ is the $C^*$-category from Example~\ref{ex_sum_becomes_different}. The $C^*$-category $\bE$ does not admit any non-trivial sums.
\hB
\end{ex}

\section{A Yoneda-type embedding}  \label{wtogwepgfereggwrferf}
  
  In the following we associate to every small unital $C^{*}$-category $\bK$ a small $C^{*}$-algebra $A(\bK)$ and construct a Yoneda type embedding $M \colon \bK \to \Hilb_{c}(A(\bK))$, where $  \Hilb_{c}(A(\bK))$ is the large $C^{*}$-category of  small right Hilbert $A(\bK)$-modules and compact operators.   The Yoneda type embedding gives rise to the following   instance of \eqref{qwefoijofwqed}: 
  % From now one we use the symbol $\bK$ for the $C^{*}$-category in $\nCcat$ which was denoted by $\bC$ above. 
%Then we  consider the following instance of \eqref{qwefoijofwqed} \footnote{For size issues we must interpret this diagram in the large universe.}
$$\xymatrix{\bK\ar[r]\ar[d]^{M}&\bM\bK\ar[d]^{ \bM M}\ar[r]&\bC:=\bW \bM\bK\ar[d]^{\bW\bM M} \\ \Hilb_{c}(A(\bK))\ar[r]&\Hilb(A(\bK))\ar[r]&\bW \Hilb (A(\bK)) }$$
%By Lemma \ref{wetiojgowegfgrefweferf}  we can identify  $\Hilb(A(\bK))$ with $\bM\Hilb_{c}(A(\bK))$ and therefore get a strict topology on the morphism spaces of $\Hilb(A(\bK))$.
%The middle vertical arrow is strictly continuous by 
%Lemma \ref{stkghosgfgsrgsegs} since the left vertical arrow
%is full  by Lemma \ref{wiotgowfrsfdgsfgefwrefrfwer}.

The main  result of this section is the following.
  \begin{theorem}\label{mainyondea}
  The Yoneda type embedding detects and preserves AV-sums  in $\bK$
  and   orthogonal sums in $\bC$. 
  \end{theorem}

The proof of this theorem will be given later in this section after recalling the construction of $A(\bK)$ and the Yoneda type embedding $M$.

%  
%  We will then see by using Corollary \ref{wetoighjowtgwergwregwergwerg} that this embedding preserves orthogonal sums. \fuli{check afterwards}
%So orthogonal sums in $\bK$ can be completely  understood  in terms of orthogonal sums in  $ \Hilb(A(\bK))$.\footnote{We thank Ch.\ Voigt for suggesting this idea.}\fuli{rewrite afterwards}
% 
 Let $\nCcat_{i}$ denote the wide subcategory of $\nCcat$ of morphisms which are injective on objects. 
 We consider the functor
\begin{equation}\label{frewfoirjviojvioeweverwvwev}
A\colon \nCcat_{i}\to \nCalg
\end{equation}
defined in  \cite[Def.~2]{joachimcat}, see also  \cite[Def.~6.5]{crosscat}.

\begin{rem}\label{rem_defn_A}
For the sake of self-containedness and in order to introduce the relevant notation we
  recall the construction of the functor~$A$. Let $\bK$ be in $\nCcat_{i}$. We start with the description of a $*$-algebra $A^{\alg}(\bK)$. The underlying $\C$-vector space of  $A^{\alg}(\bK)$ is the algebraic direct sum
\begin{equation}\label{efefwefeefefyyyyyyyefewfef}
A^{\alg}(\bK) \coloneqq \bigoplus_{C,C^{\prime} \in \Ob(\bK)} \Hom_{\bK}(C,C^{\prime})\, .
\end{equation}
A morphism $f:C\to C'$ in $\bK$ gives rise to an element in $A^{\alg}(\bK)$ which will  be denoted by  $f[C',C]$.
The product  on $A^{\alg}(\bK)$ is defined by  
\begin{equation}\label{fwerferwfwfwrefwrefer}
g[C''',C'']f[C',C] := \begin{cases} g  f[C''',C]& C'=C'' \\ 0 & \text{otherwise\,.}\end{cases}
\end{equation} 
The $*$-operation on $\bK$ induces an involution on $A^{\alg}(\bK)$ by
$f[C',C]^{*}:=f^{*}[C,C']$.  One can check that $A^{\alg}(\bK)$ is a pre-$C^{*}$-algebra.
We equip $A^{\alg}(\bK)$ with the maximal $C^{*}$-norm and define $A(\bK)$ as the completion  of $A^{\alg}(\bK)$.
We have a natural transformation of functors
\begin{equation}\label{qweflqjoijefqwefqwefqwef}
\id\to A\colon \nCcat_{i}\to \nCcat\, .
\end{equation} Its evaluation {at} $\bK$ is the  morphism $ \bK\to A(\bK)$ which sends all objects of $\bK$ to the unique object of $A(\bK)$ (we consider the $C^{*}$-algebra as a $C^{*}$-category with a single object), and every morphism $f:C\to C'$ in $\bK$ to the corresponding element  $f[C',C]$ of $A(\bK)$.  By \cite[Lem. 6.7]{crosscat} the morphism  $ \bK\to A(\bK)$ is isometric.

Note that the assignment $\bK \mapsto A(\bK)$ is not a functor on $\nCcat$ since non-composable morphisms in a $C^\ast$-category may become composable after applying a functor to another $C^\ast$-category which is incompatible with the product described in \eqref{fwerferwfwfwrefwrefer}.
%However, $\bK \mapsto A(\bK)$ is functorial on $\nCcat_{i}$.
\hB
\end{rem}

We consider $\bK$ in $\nCcat$ and its multiplier category $\bM\bK$.
The inclusion $\bK\to \bM\bK$ belongs to $\nCcat_{i}$ so that the functor $A$ can be applied.
Let $MA(\bK)$ denote the multiplier algebra of $A(\bK)$.
\begin{lem}\label{weirugwergewferfr}
There exists a unique homomorphism
$A(\bM\bK)\to MA(\bK)$ such that $$\xymatrix{&   A(\bK)\ar[dr]\ar[dl]&\\ A(\bM\bK) \ar[rr]&& MA(\bK)}$$
commutes.
\end{lem}
\begin{proof}
Since $A$ preserves  ideal inclusions by \cite[Prop. 8.9.2]{crosscat} and   $\bK\to \bM\bK$ is an ideal inclusion,  we have an   ideal inclusion $A(\bK)\to A(\bM\bK)\to A(\bM\bK)^{+}$.  The 
 universal property (see Definition \ref{weighwoegerwferfwef}) of the multiplier algebra then provides a unique unital homomorphism 
 $A(\bM\bK)^{+}\to MA(\bK)$ under $A(\bK)$. %such that 
  %$$\xymatrix{&   A(\bK)\ar[dr]\ar[dl]&\\ A(\bM\bK)^{+} \ar[rr]&& MA(\bK)}$$
  %commutes. 
  The desired homomorphism is then the composition 
  $A(\bM\bK)\to A(\bM\bK)^{+} \to MA(\bK)$.
%   
%Let $f:C\to C'$ be a morphism in $\bM\bK$. Then $f[C',C]$ in $A(\bM \bK)$
%acts as multiplier on $A(\bK)$ considered as a subalgebra of $A(\bM\bK)$
%and hence defines an element in $M A(\bK)$. 
%One checks that this prescription extends to a homomorphism
%$A^{\alg}(\bM\bK)\to MA(\bK)$. By the universal property of the completion
% it uniquely extends to the desired homomorphism $A(\bM\bK)\to MA(\bK)$.
 \end{proof}

%
%Let $\bK$ be in $\Ccat$ and form 
%     $A(\bK)$ in $\nCalg$.  Then we get a new (large) unital $C^{*}$-category 
%$\Hilb(A(\bK))$   of small right Hilbert $A(\bK)$-modules.  We let $\Hilb_{c}(A(\bK))$ be the   ideal of $\Hilb(A(\bK))$ of compact operators.
%We consider $A(\bK)$ as an object of $\Hilb_{c}(A(\bK))$ in the natural way.

Let $\bK$ be in $\nCcat$.
For every object $C$ in $\bK$ we have the multiplier $1_{C}$ in $\End_{\bM\bK}(C)$ and consider the projection
$1_{C}[C,C]$ in $A(\bM\bK)$ and therefore  in $MA(\bK)$ by applying the morphism constructed in  Lemma  \ref{weirugwergewferfr}.
Then $1_{C}[C,C]A(\bK)$ is a submodule of $A(\bK)$ which we consider as an object of $\Hilb_{c}(A(\bK))$.
%\uli{We furthermore consider the full subcategory $\pHilb(A(\bK))$ of projective modules.}
  
 \begin{ddd}\label{eiogjoewrgerfgerrfewfwer} We define the Yoneda type  functor
$M:\bK\to \Hilb_{c}(A(\bK))$ is follows.
\begin{enumerate}
\item objects: If $C$ is in $\bK$, then we set $M_{C}:=1_{C}[C,C] A(\bK)$. %Here $1_{C}$ is a morphism in $\bM\bK$ and we use Lemma \ref{weirugwergewferfr} in order to interpret this formula.
\item morphisms: If $f:C\to C'$ is a morphism in $\bK$, then we define
$M_{f}:=f[C',C]:M_{C}\to M_{C'}$. 
 \end{enumerate}
 \end{ddd}
A priori this defines a functor $M:\bK\to \Hilb(A(\bK))$.
In order to see that $M$  takes values in the ideal  $ \Hilb_{c}(A(\bK))$  we first consider $u$ in $\End_{\bK}(C')$.
Then $M_{uf}=\theta_{u,f^{*}[C,C']}$, i.e., $M_{uf}$ is compact. We now let $u$ run over an approximate unit of the $C^{*}$-algebra $\End_{\bK}(C')$ and get $M_{f}=\lim_{u} M_{uf}$. Hence $M_{f}$ is compact, too.

\begin{lem}\label{wkoegwgrregwerewferf}
The Yoneda  type functor extends canonically to 
 a functor $ M:\bM\bK\to \Hilb(A(\bK))$ such that
 $$\xymatrix{\bK\ar[r]^-{M}\ar[d]&\Hilb_{c}(A(\bK))\ar[d]\\   M\bK\ar[r]^-{\bM M}&\Hilb(A(\bK))}$$ commutes.
\end{lem}
\begin{proof}
We must define the extension on the level of morphisms.
For a multiplier $f:C\to C'$ in $\bM\bK$ we define
$ M_{f}:M_{C}\to M_{C'}$ as the morphism
$f[C',C]:M_{C}\to M_{C'}$, where this formula  must be interpreted using    Lemma 
\ref{weirugwergewferfr}. This prescription is compatible with the involution and the composition.
\end{proof}

% 
% We consider $A(\bC)$ as an  object of $\Hilb(A(\bC))$ in the natural way. 
% A morphism $f \colon C\to C'$ in $A^{\alg}(\bC)$ gives rise to the matrix
%$f[C',C]$  in $A(\bC)$ with the single non-trivial entry $f$ at position $(C',C)$.
%The element $1_{C}[C,C]$ is an orthogonal projection in $A(\bC)$. Its image is a submodule
%\begin{equation}\label{rfgrfzeiurfhirefwfqwefqwefewqfq}
%M_{C} \coloneqq 1_{C}[C,C] A(\bC)\, .
%\end{equation}
%%\uli{Note that $M_{C}$ belongs to $\pHilb(A(\bC))$.}
%A morphism $f \colon C\to C'$ induces an
%$A(\bC)$-linear map $M_{f} \colon M_{C}\to M_{C'}$ given by left multiplication by $f[C',C]$.
%One checks that this map   has an adjoint given by $M_{f}^{*}=M_{f^{*}}$. In this way we get a functor between $C^{*}$-categories \begin{equation}\label{ewfqfiojoijfqewfwqfew}
%M \colon \bC\to {\Hilb}(A(\bC))\, , \quad C\mapsto M_{C}\, , \quad f\mapsto M_{f}
%\end{equation}
% 
 
 \begin{lem}\label{wiotgowfrsfdgsfgefwrefrfwer}
The functors $M:\bK\to \Hilb_{c}(A(\bK))$   and
$ \bM M:\bM\bK\to \Hilb(A(\bK))$ are fully faithful.
\end{lem}
\begin{proof}
We first show that $ M:\bM\bK\to \Hilb(A(\bK))$ is faithful. Let $f \colon C\to C'$ be a morphism in $\bM\bK$.
 If $M_{f}=0$, then $0=M_{f}(u[C,C])$ for any $u$ in $\End_{\bK}(C)$. This implies $fu=0$ in $\Hom_{\bK}(C,C')$.  Since $u$ is arbitrary, this finally implies that $f  =0$.
 
 As a consequence also $M:\bK\to \Hilb_{c}(A(\bK))$  is faithful.

We now show that $  M:\bM\bK\to \Hilb(A(\bK))$ is full.
Let $F \colon M_{C}\to M_{C'}$ be a morphism in $\Hilb(A(\bK))$.
Then we define a multiplier $(L,R)$ (see Definition \ref{ewirjgoethwfegergewrg}) from $C\to C'$ as follows. 
We  define $L(g):E\to C'$ for any $g:E\to C$ uniquely by 
$L(g)[C',E]:=F(g[C,E])$. Furthermore, for    $h:C'\to D$ we  define  $R(h):C\to D$  by $F^{*}(h^{*}[C',D])=R(h)^{*}[C,D]$.
One checks that $(L,R)$ is indeed an algebraic double centralizer and hence
provides a morphism in $\bM\bK$ by Proposition \ref{wrtohijworthrtwregergwerg}.\ref{oiergjoiwegwrrgwr}. Furthermore, $M(L,R)=F$.
 
Assume now that $F:M_{C}\to M_{C'}$ belongs to $ \Hilb_{c}(A(\bK))$.
If $u$ runs over an approximate unit of $\End_{\bK}(C)$, then
$\lim_{u} F M_{u}= F$ by Lemma \ref{multmult}.  
Now note that  $FM_{u}=M_{R(u)}$
 with $R(u)$ in $\Hom_{\bK}(C,C')$.
 Since $M: \bM \bK\to \Hilb(A(\bK)) $ is fully faithful it is isometric. Consequently  $u\to R(u)$ converges in norm to $\lim_{u}R(u)$ in $\Hom_{\bK}(C,C')$ and 
  $F=M_{\lim_{u} R(u)}$.
  \end{proof}

By Lemma \ref{wetiojgowegfgrefweferf}  we can identify  $\Hilb(A(\bK))$ with $\bM\Hilb_{c}(A(\bK))$ and therefore get a strict topology on the morphism spaces of $\Hilb(A(\bK))$.
Then $\bM M:\bM\bK\to \Hilb(A(\bK)) $   is strictly continuous by 
Proposition \ref{stkghosgfgsrgsegs} since  $M$ is full
  by Lemma \ref{wiotgowfrsfdgsfgefwrefrfwer}.

\begin{proof}[Proof of Theorem \ref{mainyondea}]
Since $\bM M$ is fully faithful and strictly continuous, 
 Corollary \ref{wriojtghowtrgfrefw}   implies that $\bM M$ detects and preserves AV-sums.

 Since $\bW \bM M \colon \bC\to\bW  \Hilb (A(\bK))$ is   fully faithful  by Proposition \ref{rjigowergwregregwfer}  applied to $\phi=\bM M$ it follows from 
 Corollary \ref{weigjiowetgerwerer} that it detects and preserves orthogonal sums in $\bC$.
  \end{proof}

\begin{kor}\label{woiregjoiwergreggergw9333} 
Any small  $C^{*}$-category admits an  AV-sum   preserving embedding into a
large $C^{*}$-category admitting    AV-sums for all small families.
\end{kor}\begin{proof}
For $\bK$ in $\nCcat$ we can take the embedding $M:\bK\to \Hilb_{c}(A(\bK))$.
\end{proof}

%Similarly, as a consequence of Proposition \ref{rjigowergwregregwfer} applied to $\bM M$  we get: 
\begin{kor}\label{woiregjoiwergreggergw9} 
For every  small  $C^{*}$-category $\bK$  the catgeory $\bW\bM\bK$  admits an  orthogonal sum   preserving embedding into a
large $C^{*}$-category admitting orthogonal sums    for all small families.
\end{kor}\begin{proof}
We can take the embedding $\bW\bM M: \bW\bM\bK\to \bW\Hilb(A(\bK))$.
\end{proof}
 
}

   We   consider $A(\bK)$   as an object in  $\Hilb(A(\bK))$. In view of the  Definition \ref{eiogjoewrgerfgerrfewfwer}   the algebraic sum $ \bigoplus^{\alg}_{C\in \Ob(\bK)} M_{C}$ is naturally an $A(\bK)$-submodule of $A(\bC)$.
 The sum in the following lemma is the classical sum of Hilbert $C^{*}$-modules explained in Construction \ref{ejirgowergrefweerf}, but note Theorem \ref{ewrgowergwerfgrfwerf}.
\begin{lem} \label{rtiohjirotgwegfwergewgegwerf9}
We have an isomorphism  
$A(\bK)\cong \bigoplus_{C\in  \Ob(\bK)} M_{C}$
in $\Hilb(A(\bK))$.
\end{lem}
\begin{proof}
 We use that $ \bigoplus^{\alg}_{C\in  \Ob(\bK)} M_{C}$ is a dense subspace of both $A(\bK)$ and $\bigoplus_{C\in  \Ob(\bK)} M_{C}$.
Let $m=\oplus_{C}m_{C}
$ be an element of $\bigoplus^{\alg}_{C\in \Ob(\bK)} M_{C}$. Then 
$$\|m\|^{2}_{A(\bK)}=\|m^{*}m\|_{A(\bK)}=\|\sum_{C\in  \Ob(\bK)} m_{C}^{*}m_{C}\|=   \|  \langle m^{*},m\rangle_{ \bigoplus_{C\in  \Ob(\bK)} M_{C}}\|=\|m\|^{2}_{ \bigoplus_{C\in  \Ob(\bK)} M_{C}}\, .$$
By \eqref{eq_norm_classical_Hilb_sum} the right-hand side is the square of the norm of $m$ in the classical sum. 
This equality of norms implies the equality of closures. 
\end{proof}

\begin{rem} For an alternative proof of Lemma \ref{rtiohjirotgwegfwergewgegwerf9}
  one could  observe that 
$$\sum_{C\in \Ob(\bK)} 1_{C}[C,C]=1_{MA(\bK)}$$
in the strict topology of $$ MA(\bK)\cong \End_{\bM\Hilb_{c}(A(\bK))}(A(\bK))\cong \End_{\Hilb(A(\bK))}(A(\bK))\ .$$ In view of Definition \ref{peofjopwebgwregwgreg} this shows that
$A(\bK)$ is the AV-sum of the family $(M_{C})_{C\in \Ob(\bK)}$. Then one can apply Theorem \ref{ewrgowergwerfgrfwerf} in order to deduce the assertion of the lemma.\hB
\end{rem}

We now consider $C^{*}$-categories with a strict $G$-action and study the equivariance of the Yoneda type embedding. We will see that it is not strictly equivariant. But it extends to a weakly invariant morphism in the sense of   Definition \ref{rgjeqrgoieqrgrgqewgfq}.   We furthermore  extend Corollary \ref{woiregjoiwergreggergw9} to the equivariant case and study the
 compatibility of the Yoneda type embedding with equivariant morphisms.

 If $  \bK$ is in $\Fun(BG,\Ccat)$, then  $A(  \bK)$ is in $\Fun(BG,\nCalg)$, and we can consider the large $C^{*}$-category    $\Hilb_{c}(A(\bK))$ with the strict $G$-action as in Example \ref{qregiuheqrigegwegergwegergw}. 
Using Lemma \ref{wetiojgowegfgrefweferf} we will identify $\bM\Hilb_{c}(A(\bK))\cong \Hilb (A(\bK))$ in $\Fun(BG,\Ccat)$.   %Recall the Definition \ref{rgjeqrgoieqrgrgqewgfq} of a weakly invariant morphism. 

\begin{lem}\label{qroihofewfqffqwefqewfqefqef} 
The    Yoneda type embedding
\[
 M:   \bK\stackrel{}{\to}\Hilb_{c}(A(  \bK)) 
\]
defined in \ref{eiogjoewrgerfgerrfewfwer}
%from Lemma \ref{wkoegwgrregwerewferf} 
extends to a  weakly equivariant  morphism of $C^{*}$-categories.
\end{lem}
\begin{proof} For  $g$  in $G$ and object $C$ in $\bK$ we   define the  isomorphism
  \begin{equation}\label{efwqwefqewfwqdwed}
\sigma_{C}(g): {M_{C}= 1_{C}[C,C]A(\bK) \stackrel{m\mapsto {}^{g}m}{\to} 1_{gC}[gC,gC] A(\bK)=M_{gC}}
\end{equation}  of {complex} vector spaces. 
Using the notation introduced in Example \ref{qregiuheqrigegwegergwegergw} one checks that this isomorphism intertwines  the right $A(\bK)$-action $\cdot$ on $M_{gC}$ with the action $\cdot_{g}$ on $M_{C}$. Furthermore, it  is an isometry {if we equip $M_{gC}$ with the scalar product $\langle-,-\rangle_{M_{gC}}$ and} $ M_{C}$
with the scalar product ${}^{g} \langle -, -\rangle_{M_{C}}$. 
Therefore  $\sigma_{C}(g)$ can be interpreted as a unitary isomorphism between $gM_{C}$ and $M_{gC}$ in $\Hilb(A(\bC))$. % which is the multiplier category of $\Hilb_{c}(A(\bK))$.}

If $f \colon C\to C'$ is a morphism in $\bK$, then  $$\xymatrix{{g}M_{C}\ar[r]^{g M_{f}}\ar[d]^{\sigma_{C}(g)}&{g}M_{C'}\ar[d]^{\sigma_{C'}(g)}\\M_{gC}\ar[r]^{ M_{gf}}&M_{gC'}}$$ {obviously}
commutes.

This shows that the family $ \rho(g):=(g^{-1}\sigma_{C}(g))_{C\in \Ob(\bK)}$ is a   unitary natural multiplier isomorphism $ M\to  g^{-1} M g$. 
Then $\rho:=( \rho(g))_{g\in G}$ is the family of unitary  multiplier  isomorphisms which extends $  M$ to a weakly equivariant morphism.
\end{proof}

Combining Theorem \ref{mainyondea} and Lemma \ref{qroihofewfqffqwefqewfqefqef} we obtain the desired  equivariant generalization of Corollary \ref{woiregjoiwergreggergw9}. 
 \begin{kor}\label{woiregjoiwergrreggergw9} 
Any small  $C^{*}$-category $\bL$ with strict $G$-action admits a weakly equivariant   orthogonal sum   preserving embedding into a
large $C^{*}$-category $\bK$ with strict $G$-action  such that $\bW\bM\bK$  admits  all small orthogonal sums.
\end{kor}

A morphism $\phi \colon \bK\to \bL$ in $\nCcat_{i}$
  induces a homomorphism
$A(\phi):A(\bK)\to A(\bL)$ of $C^{*}$-algebras.
The latter induces a functor
$A(\phi)_{*} \colon \Hilb_{c}(A(\bK))\to \Hilb_{c}(A(\bL))$ given by
$$C\mapsto C\otimes_{A(\bK)} A(\bL), \quad f\mapsto f\otimes 1_{A(\bL)}\ ,$$ where $1_{A(\bL)}$ belongs to the multiplier algebra $MA(\bL)$. Note that the functor $A(\phi)_{*}$ depends on the choice of the tensor product and is therefore only unqiue up to a unitary multiplier isomorphism. Furthermore note that
$A(\phi_{*})$ has an obvious extension to the multiplier categories 
$\bM A(\phi_{*}):\Hilb (A(\bK))\to \Hilb (A(\bL))$ given by the same formulas.

     We now assume that $\phi: \bK\to  \bL$ belongs to 
$\Fun(BG,\nCcat_{i})$.   
\begin{lem}\label{wothpwgergewefwef}
The functor 
  $A(\phi)_{*}$ canonically extends to a weakly invariant functor %. Furthermore the compositions $A(\phi)_{*}\circ  \bM^{K}$ and $\bM^{\bL}\circ \phi$ are defined
such that   the   square 
$$\xymatrix{\bK\ar[r]^{\phi}\ar[d]^{ M^{\bK}}&\bL\ar[d]^{ M^{\bL}}\\ \Hilb_{c}(A(\bK))\ar[r]^{ A(\phi)_{*}}&\Hilb_{c}(A(\bL))}$$ commutes up to a canonical unitary  multiplier  morphism between weakly invariant functors.
\end{lem}
\begin{proof} 
%Note that the action of $G$ on Hilbert modules 
%only changes the right multiplication and the scalar product. 
For every $C$ in $\Hilb_{c}(A(\bK))$ 
we  define   the $\C$-linear  map
$$\rho_{C}(g):=\id_{C}\otimes A(g):C\otimes_{A(\bK)}A(\bL)\to g^{-1}(gC\otimes_{A(\bK)}A(\bL))\ .$$
One checks that it is well-defined and a unitary isomorphism in 
$\Hilb(A(\bL))$. The family
$\rho(g):=(\rho_{C}(g))_{C\in \Hilb(A(\bK))}$ is a unitary natural  multiplier   isomorphism
from $ A(\phi)_{*}$ to $g^{-1} A(\phi)_{*}g$. The family
$(\rho(g))_{g\in G}$ extends $ A(\phi)_{*}$ to a weakly invariant functor.

 Note that the  compositions $A(\phi)_{*}\circ  M^{\bK}$ and $M^{\bL}\circ \phi$ of weakly invariant functors are defined.
%The composition of weakly equivariant morphisms $A(\phi)_{*}\circ M^{\bK}$
%exists since $A(\phi)_{*}$ has the extension $\bM A(\phi)_{*}$.
%The composition $\bM^{\bL}\circ \phi$ exists since $\phi$ is strictly invariant
%and the identities of objects are multiplier morphisms.

For every  object
  $C$ of $\bK$ we have the unitary  isomorphisms in $\Hilb(A(\bL))$ \begin{eqnarray*} 
\lefteqn{\hspace{-3cm}\kappa_{C}:A(\phi)_{*}M^{\bK}(C)\cong 1_{C}[C,C] A(\bK)\otimes_{A(\bK)} A(\bL)\cong
\bM\phi(1_{C}[\phi(C),\phi(C)])A(\bL)} && \\ &&\cong 1_{\phi(C)}[\phi(C),\phi(C)]A(\bL)\cong M^{\bL}  (\phi(C))\ . 
\end{eqnarray*} The family $\kappa:=(\kappa_{C})_{C\in \Ob(\bK)}$
 is the desired unitary multiplier isomorphism filling the 
square.
\end{proof}

\section{Orthogonal sums of functors}\label{qergoijeroigjeroigjqer90gu9384tuergqegerggw}
%\label{rifoqwfwefqewfqewfqefqewfqef}

%\subsection{\Alex{Orthogonal sums of functors and flasque \texorpdfstring{$C^*$-categories}{Cstar-categories}}}

 {Let $\bC$ be in $\Ccat$ and consider a family $(\phi_{i})_{i\in I}$ of morphisms  $\phi_{i}:\bD\to \bC$. For every object $D$ in $\bD$ we get a family $(\phi_{i}(D))_{i\in I}$ of  objects in $\bC$ and can consider its orthogonal sum  in $\bC$ in the sense of Definition \ref{erogwfsfdvbsbfdbsdfbsfdbv}. In the present section we extend this to the notion of an orthogonal sum of functors $\oplus_{i\in I} \phi_{i}:\bD\to \bC$. In the equivariant case  we assume that 
$\bD,\bC$ belong to $\Fun(BG,\Ccat)$ and the morphisms $\phi_{i}$ are equivariant for all $i$ in $I$.  But due to the non-uniqueness of orthogonal sums of families of objects we can not expect that $ \oplus_{i\in I} \phi_{i}$ is again equivariant. 
But since orthogonal sums are unique up to unique unitary isomorphism by Lemma \ref{rgiojqeiovevqeve9} this sum of functors is still equivariant in a weaker sense. In the present section we discuss
 the details of these considerations. 
 
 We develop the case of AV-sums in $\bK$
 in a parallel manner using the notation from \eqref{qrefqewfdewdqwedwed}  by indicating the necessary modifications in brackets.  In this case $\bC=\bW\bM\bK$.

}

Assume that $\bC$ is in $\Ccat$ [or $\bK$ in $\nCcat$]. 
Let $I$ be a  very small set and assume that $\bC$ admits $I$-indexed orthogonal sums [or $\bK$ admits $I$-indexed AV-sums].
  %$|I|$-additive,  {see Definition~\ref{ergiehgijiorewgergwergweg}.}

%\Alex{Recall that $\Ccat$ is complete (see \cite{DellAmbrogio:2010aa} or \cite[Thm.~8.1]{startcats}). Especially, it admits all products indexed by very small sets: the objects of $\prod_{I}\bC$ are families $(C_{i})_{i\in I}$ of objects $C_i$ of $\bC$ for every $i$ in $I$ and morphisms are families $(f_{i})_{i\in I}\colon (C_{i})_{i\in I}\to (C'_{i})_{i\in I}$, where  $f_{i}\colon C_{i}\to C_{i}'$  for all $i$ in $I$, with the condition $\sup_{i\in I} \|f_{i}\|<\infty$.}

\begin{construction} \label{qergoieqgfreqfewfwefq} {\em % \fuli{AV: geht immer noch: Mu{\ss} $\bC$ zumeist als $\bM\bC$ interpretieren}
We  construct a functor
\begin{equation}\label{qwekjbkjwqc}
\bigoplus_{I}:\prod_{I}\bC\to \bC \quad \quad \quad \left[\bigoplus_{I}:\prod_{I}\bM\bK\to \bM\bK\right]
\end{equation}
  as follows:
\begin{enumerate}
\item\label{wgbiwhejrioefrvweerv} objects: For every object $(C_{i})_{i\in I}$ of $\prod_{I}\bC$ [or $\prod_{I}\bM\bK$] we choose an orthogonal  sum  in $\bC$ [or AV-sum in $\bK$] 
\[(\bigoplus_{i\in I}C_{i},(e_{i})_{i\in I})\,.\]
This determines the action of the functor on objects.
\item morphisms: Let   $(C_{i})_{i\in I}$  and $(C'_{i})_{i\in I}$  be objects and   
    $(f_{i})_{i\in I}\colon (C_{i})_{i\in I}\to (C'_{i})_{i\in I}$ be a morphism in $\prod_{I}\bC$ [or $\prod_{I}\bM\bK$], where  $f_{i}\colon C_{i}\to C_{i}'$ [is a multiplier morphism] for all $i$ in $I$. 
% Since the product is interpreted in $\Ccat$ 
Then we have $\sup_{i\in I} \|f_{i}\|<\infty$ and 
   \eqref{ervfdvsdfvfdvfsdvsfdv}  [or Lemma  \ref{wrtjihordbvdfgbdb}] provides the morphism [multiplier morphism]
$$\bigoplus_{I} (f_{i})_{i\in I} :=  \oplus_{i\in I}f_{i} \colon \bigoplus_{i\in I}C_{i}\to \bigoplus_{i\in I}C'_{i}\, .$$
 \end{enumerate} 
 
%By the uniqueness assertion  in Example \ref{ruihqiuwefwefqfqwefq} [or Example \ref{wrtjihordbvdfgbdb}]  
This construction is compatible with compositions and the involution.
Note that the functor \eqref{qwekjbkjwqc} depends on the choice of the objects representing the orthogonal sums.  By Lemma \ref{rgiojqeiovevqeve9}  [or Proposition \ref{weijgowgwerfwrefwr}] a different choice here leads to a uniquely unitarily isomorphic functor.  
} \hB
\end{construction}
%Let $\bC$ be in $\Ccat$ and assume that $\bC$ is additive and idempotent complete.
%Let  $C$ be an object of $\bC$ and $(P_{i})_{i\in I}$ be a family of mutually orthogonal projections, 
%i.e.,  $P_{i} P_{i'}=0$ all pairs $i,i'$ in $I$ with $i\not= i'$. By idempotent completeness of $\bC$ we can choose for  every $i$ in $I$  an image $p_{i}:C_{i}\to C$ of $P_{i}$. We let $ e_{i}:C_{i}\to \bigoplus_{i\in I} C_{i}$ be the canonical morphism. 
%
%\begin{lem}
%We have a canonical morphism
%$$\sum_{i\in I} p_{i}e_{i}^{*}:\bigoplus_{i\in I} C_{i}\to C$$ uniquely determined by the condition that
% $P_{i'}(\sum_{i\in I} p_{i}e_{i}^{*})=p_{i'}e_{i'}^{*}$ for all $i'$ in $I$.
%\end{lem}
%\begin{proof}
%We define the morphism $\sum_{i\in I} p_{i}e_{i}^{*}$ in $
%\end{proof}
%

\begin{construction}\label{eoiqjofcqwecqwc}{\em 
Let $\bD$  {and} $\bC$ be in $\Ccat$ [or $\bK$ and $\bL$ in $\nCcat$], and let $(\phi_{i})_{i\in I}$ be a  family of morphisms  in $\Hom_{\Ccat}(\bD,\bC)$
[or in $\Hom_{\Ccat}(\bM\bL,\bM\bK)$]. We   assume that
 $\bC$ admits $I$-indexed orthogonal sums [or $\bK$ admits $I$-indexed AV-sums]. 
 We  fix a choice for the functor \eqref{qwekjbkjwqc}. % \begin{ddd}\label{eoiqjofcqwecqwc} %\fuli{das geht durch, wenn man die $\phi_{i}$ und die Summe in $\bM\bC$ interpretiert}
 We  define the orthogonal sum  
\begin{equation}\label{eiogrbwegergwerg9} \oplus_{i\in I}\phi_{i}:\bD\to \bC\quad\quad \left[\oplus_{i\in I}\phi_{i}:\bM\bL\to \bM\bK\right] \end{equation} of the family $(\phi_{i})_{i\in I}$ as the composition
$$\oplus_{i\in I} \phi_{i}:\bD\xrightarrow{\diag} \prod_{i\in I} \bD\xrightarrow{\prod_{i\in I}\phi_{i}} \prod_{i\in I} \bC\xrightarrow{\bigoplus_{ I}} \bC\, .$$
$$\left[\oplus_{i\in I} \phi_{i}:\bM\bL\xrightarrow{\diag} \prod_{i\in I} \bM\bL\xrightarrow{\prod_{i\in I}\phi_{i}} \prod_{i\in I} \bM\bK\xrightarrow{\bigoplus_{ I}} \bM\bK\, .\right]$$
%\end{ddd}
Again, this sum depends on the choice adopted for $\bigoplus_{I}$. 
A different choice here leads to a uniquely unitarily isomorphic functor.
}\hB \end{construction}
%Let $I$ be a set and  $\bC$ be in $\Ccat$. Assume that $\bC$ is $|I|$-additive.

%\begin{ddd}\label{eiogrbwegergwerg9}
%We define the endofunctor $\bigoplus
%Assume that $\bC$ is $|I|$-additive.
%
%$\bigoplus_{I}:\bC\to \bC$ which sends $C
%$ in $\bC$ to a sum $\bigoplus_{  I} C$
%and a morphism $f:C\to C^{\prime}$ to the morphism $\oplus_{i\in I} f:\bigoplus_{  I} C\to \bigoplus_{  I} C'$
%\end{kor}
%\begin{proof}
%We construct  the functor $\bigoplus_{I}$ first on objects. To this end we choose for every 
% object $C$ in $\bC$ a sum  $(\bigoplus_{I}C,(e_{i})_{i\in I} )$. 
% If $f:C\to C^{\prime}$ is a morphism, then we define
% $\oplus_{I}f:\bigoplus_{I}C\to \bigoplus_{I}C'$ using Lemma \ref{wrigfoegwergergwegg}. 
% By the uniqueness assertion of  this lemma we conclude that this construction is compatible with compositions.
%  \end{proof}

{In order to show that a given $C^{*}$-category has trivial $K$-theory one often uses an Eilenberg swindle argument. In the present paper we formalize this using the notion of flasqueness. We   refer to Proposition \ref{ergiwoegwergwergwgrg} below for the application.}

 Let $\bC$ be in $\Ccat$.
 \begin{ddd}\label{bjiobdfbfsferq}
 $\bC$ is flasque  if it is    additive {(see Definition~\ref{ergiehjioferfqffrf})} and admits an endomorphism $S\colon \bC\to \bC$ such that $\id_{\bC}\oplus S$ is unitarily isomorphic to $ S$.
\end{ddd}
We say that $S$ implements flasqueness of $\bC$.
Note that the sum $\id_{\bC}\oplus S$ is defined  by   \eqref{eiogrbwegergwerg9}.

%\subsection{\Alex{Weakly invariant functors}}

%For convenience {we} adopt the following notational convention.  
%
%
%We use decorated symbols $\tilde C, \tilde C'$ to denote 
%  $G$-objects in $\Fun(BG,\cC)$, and similarly decorated letters  to denote morphisms $\tilde f\colon \tilde C\to \tilde C'$. 
%   In this case we write  $C:=\tilde C(*_{BG})$ for the underlying object in $\cC$, and $f\colon C\to C'$ for the morphism between the underlying objects induced by $\tilde f$. 

%\begin{ex}
%Let $\bC,\bD$ be in $\Fun(BG,\Ccat)$ and $(\phi,\rho)$ and $(\phi',\rho')$ be weakly
%invariant morphisms from $\bC$ to $\bD$. We assume that $\bD$ admits finite sums.
%Then we can form the sum
%$$ (\phi,\rho)\oplus (\phi',\rho') =(\phi\oplus \phi',\kappa):\bC\to \bD$$
%which is again a weakly invariant  morphism.
%Its underlying  morphism is 
%$\phi\oplus \phi'$  defined as in Definition \ref{qergoieqgfreqfewfwefq}.
%For every $g$  in $G$  we let $\kappa(g) $ be the isomorphism
%$$  \phi\oplus \phi'\stackrel{\rho(g)\oplus \rho'(g)}{\to}  (g^{-1}\phi g)\oplus (g^{-1}\phi' g)   \stackrel{!}{\to}  
%   g^{-1} ( \phi\oplus \phi )g\ ,
%$$
%where the unitary isomorphism marked by  $!$ is determined uniquely by its compatibility with the (implicitly chosen, but not written) structure maps of the sum of functors.
%
%In general,  even if $\phi$ and $\phi'$ were equivariant, the functor $\phi\oplus \phi'$ is only weakly invariant. \hB
%\end{ex}

Let     $ \bD, \bC$ be in $\Fun(BG,\Ccat)$ [or $\bL,\bK$ in $\Fun(BG,\nCcat)$, where then $\bM\bL$ and $\bM\bK$ have induced strict $G$-actions].   
Let $I$ be a set
 and $((\phi_i,\rho_{i}))_{i\in I}$ be a family   of weakly  equivariant (Definition \ref{rgjeqrgoieqrgrgqewgfq}) morphisms from $\bD$ to $\bC$ [or from $\bM\bL$ to  $\bM\bK$].
 Assume that $  \bC$   admits $I$-indexed orthogonal  sums [$\bK$ admits $I$-indexed AV-sums].
Then we can construct a morphism \begin{equation}\label{qefqwefewdqd}
\oplus_{i\in I} \phi_{i}\colon  \Res^{G}(\bD) \to \Res^{G}(\bC) \quad\quad \left[\oplus_{i\in I} \phi_{i}\colon  \Res^{G}(\bM\bL) \to \Res^{G}(\bM\bK) 
\right]
\end{equation}
  as in  Construction  \ref{eoiqjofcqwecqwc}. %{Due to} the choices involved in this construction we can not expect that  it comes from an invariant functor. The following lemma  morally provides the best approximation to its equivariance. 

%Recall the notion of a weakly equivariant morphism from \cite[Def.~7.10]{crosscat}. \Alex{Die Definition würde ich hier wiederholen.}
%Recall the Definition \ref{rgjeqrgoieqrgrgqewgfq} of a weakly equivariant morphism.
\begin{prop} \label{rgoihegiqwefqfqwfqwefqwefq} 
The morphism $\oplus_{i\in I} \phi_{i}$ in \eqref{qefqwefewdqd}  has a canonical refinement to
 a  weakly equivariant  morphism
$$(\oplus_{i\in I} \phi_{i},\theta)\colon \bD\to \bC \quad \quad \left[(\oplus_{i\in I} \phi_{i},\theta):\bM\bL\to \bM\bK
\right]\, .$$ 
 \end{prop}
\begin{proof}
We discuss the necessary modifications for the AV-case at the end.
 It remains 
 to construct  the family of unitary natural transformations  $\theta$. For  $D$ in $\bD$ we consider  the sum 
 $(\bigoplus_{i\in I}\phi_{i}(D),(e_{i})_{i\in I})$  with $e_{i}\colon \phi_{i}(D) \to \bigoplus_{i\in I}\phi_{i}(D)$
 underlying the construction of  the sum of morphisms in \eqref{qefqwefewdqd}. For  $g$ in $G$  we further consider 
 the object $ gD$  in $\bD$
 and let $(\bigoplus_{i\in I} \phi_{i}(gD), (  e^{g}_{i})_{i\in I} )$ with $e_{i}^{g}:\phi_{i}(gD) \to \bigoplus_{i\in I} \phi_{i}(gD)$  be the corresponding  choice of the sum in $\bC$ going into \eqref{qefqwefewdqd}. 
  Then the    object  $(g^{-1}\bigoplus_{i\in I} \phi_{i}(g D), (g^{-1}  e^{g}_{i}\circ \rho_{i}(g)_{D})_{i\in I} )$ also represents an orthogonal  sum for {the family of  objects} $(\phi_{i}(D))_{i\in I}$.    {From} Lemma \ref{rgiojqeiovevqeve9}
    we get a uniquely determined unitary  isomorphism $$\theta (g)_{D} \colon \bigoplus_{i\in I}\phi_{i}(D)\to g^{-1}\bigoplus_{i\in I} \phi_{i}(g D) $$ such  that  $\theta (g)_{D}e_{i}= g^{-1}  e^{g}_{i}\circ \rho_{i}(g)_{D} $  for all $i$ in $I$. One checks that the family $\theta(g):=(\theta(g)_{D})_{ D\in \bD}$
   is a natural transformation 
\[\theta(g)\colon\oplus_{i\in I} \phi_{i}\to {g^{-1} \big( \oplus_{I} \phi_{i}\big) g}\ .
\] Furthermore, the family $\theta:=(\theta(g))_{g\in G}$
    satisfies the cocycle relation  required in Definition \ref{rgjeqrgoieqrgrgqewgfq}.
     The pair $(\oplus_{i\in I}\phi_{i},\theta)$ is the desired canonical extension of $\oplus_{i\in I}\phi_{i}$ to a weakly equivariant  morphism from $\bD$ to $\bC$.
    [In the AV-case we replace $\bD$ by $\bM\bL$ and $\bC$ by $\bM\bK$.  We apply Proposition \ref{weijgowgwerfwrefwr} in order to
    get the unitary  multiplier morphisms $\theta(g)_{D}$.]
\end{proof}

  \begin{ex}\label{ex_countably_additive_flasque}%\fuli{AV:$\bM\bC$ is flasque.} 
   If $\bC$  in $\Ccat$ 
  is countably additive, then it is flasque.
Indeed,   according to   Definition \ref{eoiqjofcqwecqwc}  we can construct   the endofunctor $S:=\oplus_{\nat} \id_{\bC}\colon \bC\to \bC$. 
One easily  finds  a unitary isomorphism between
$\id_{\bC}\oplus S$ and $S$.

Similarly, if $\bK$ is countably AV-additive, then $\bM\bK$ is flasque by the same argument.

If $\bC$  is in $\Fun(BG,\Ccat)$ [or $\bK$ is in $\Fun(BG,\nCcat)$] for some group $G$  such that the underlying $C^{*}$-category admits countable orthogonal sums [or   AV-sums], then by Proposition \ref{rgoihegiqwefqfqwfqwefqwefq} the endomorphism $S$  above can be refined to a weakly invariant morphism such that the isomorphism $\id_{\bC}\oplus S\cong S$  becomes an isomorphism of weakly invariant functors.  This witnesses the fact  that $\bC$ (or $\bM\bK$, respectively) is flasque in the sense of $C^{*}$-categories with $G$-action \cite[Def.~6.16]{KKG}.
\hB
\end{ex}

\section{Reduced crossed products}
\label{sec_reduced_crossed_products}

The maximal crossed product of a $C^{*}$-category with a strict action of a group $G$ {was} introduced and studied in \cite{crosscat}. In the present paper we  will introduce the
 reduced crossed product; it is an important ingredient in the subsequent papers \cite{coarsek}, \cite{KKG} and \cite{bel-paschke}.

 {In the case of a $C^{*}$-algebra with $G$-action
 $A$, as recalled   in  Definition \ref{woigwgewgerg9},}  the reduced norm on the algebraic crossed product $  A \rtimes^{\alg}G$  is induced from a representation on the Hilbert-$A$-module
 $L^{2}(G,  A)$,  see \eqref{qewfljqlwefqwfqewfwf} below. 
 %or equivalently, on the $A$-Hilbert $C^{*}$-module $\bigoplus_{G}A$. 
  In Definition  \ref{ewgiowegfregfrwegrfwerferf}  we will employ $G$-indexed  orthogonal sums of objects  in order to define   an analogue of this Hilbert $A$-module  for $C^{*}$-categories.    
 The main result of this section can be formulated as follows:
 \begin{theorem}\label{ejgwoierferfewrferfwe}
 There exists a construction of a functor
 $$-\rtimes_{r}G:\Fun(BG,\nCcat)\to \nCcat$$
 which receives a natural transformation
 $i:-\rtimes^{\alg}G\to -\rtimes_{r}G $ in $\nClincat$
 such that $i_{\bK}:\bK\rtimes^{\alg}G\to \bK\rtimes_{r}G $
 has dense image  for every $\bK$ in $\Fun(BG,\nCcat)$
 and whose values on $G$-$C^{*}$-algebras coincide with the classical reduced crossed products.  Furthermore, the functor $-\rtimes_{r}G$ preserves faithful morphisms and  fully  faithful morphisms.
 \end{theorem}
  Most of the remainder of this section is devoted to the statements and proofs  of  various  partial results which all together imply this  theorem.      
   We further show that the reduced crossed product commutes with the functor $A$ from \eqref{frewfoirjviojvioeweverwvwev}, that for exact groups it preserves exact sequences and that  for amenable groups $G$ the canonical morphism from the maximal to the reduced crossed product is an isomorphism. Finally we show that for any subgroup $H$ of $G$ there is an isometric natural transformation $$\Res^{G}_{H}(-)\rtimes_{r}H\to (-)\rtimes_{r}G$$  of functors from $\Fun(BG,\nCcat)$ to $\nCcat$  
   extending the obvious natural transformation between the corresponding algebraic crossed products.

 We consider  $  \bK$   in $\Fun(BG,\nCcat)$. Note that $G$ acts by fully faithful morphisms on $\bK$ so that    Proposition \ref{stkghosgfgsrgsegs}  provides an extension of this action by unital and strictly continuous morphisms to the multiplier category $\bM\bK$ which then belongs to $\Fun(BG,\Ccat)$.
 We finally apply  the functor $\bW$ from the Theorem \ref{rijogegergwerg9}
 in order to define  $\bC:=\bW\bM\bK$ in $\Fun(BG,\Wcat)$. By construction,
 the group $G$ acts on $\bC$ by normal morphisms. We thus 
 get an equivariant analogue of \eqref{qrefqewfdewdqwedwed}
$$\bK\subseteq\bM\bK\subseteq \bC \ .$$
The $G$-actions  on these categories are   implemented by a  family $(g)_{g\in G}$ of isomorphisms in the respective category.
  % $g\colon \bK\to \bK$. We use the same notation 
 %  in order to denote its normal extension $g:=\bW\bM g:\bC\to \bC $ which exists by     Lemma \ref{stkghosgfgsrgsegs} and the since $\bW$ is a functor.
 
  We now assume that $  \bC$ admits orthogonal sums of cardinality $|G|$. %\ fuli{AV: mit der nun \"ublichen Uminterpretation} 
  Then we can apply   Construction \ref{eoiqjofcqwecqwc}  in order to define
  an {endomorphism} $$\oplus_{g\in G}  g\colon \bC\to \bC\, .$$ 
  
 Note that the category $\Ccat$ of  unital $C^{*}$-categories admits all limits (see \cite{DellAmbrogio:2010aa} or \cite[Thm.~8.1]{startcats} for an argument), so in particular  pull-backs.
  \begin{ddd}\label{ewgiowegfregfrwegrfwerferf}
  We define the category $\Lzwei(G,  \bC)$ as the pull-back in $\Ccat$
  \begin{equation}\label{fevrfervwevwevefvfvs}
\xymatrix{\Lzwei(G,  \bC)\ar[rr]\ar[d]^{!}&&\bC\ar[d]^{\eqref{sgbsfdvfsvfvdfvsv}}\\ 
0[\Ob( \bC)]\ar[rr]^{0[\Ob(\oplus_{g\in G}  g )]}&&0[\Ob(\bC)] }
\end{equation}
  \end{ddd}
%The category  $\Lzwei(G,  \bC)$ is determined uniquely up to unique  isomorphism.

\begin{rem}\label{rem_explicit_description_Lzwei}
We have the following explicit description of $\Lzwei(G,  \bC)$:
\begin{enumerate}
\item\label{wrtoijhwrthtgwerregw} objects:  The set of objects  of $\Lzwei(G, \bC)$ is  canonically identified with the set of    objects of $\bC$ using the arrow marked by $!$ in \eqref{fevrfervwevwevefvfvs}.
\item\label{wrtoijhwrthtgwerregw1} morphisms: The definition of the sum $\oplus_{g\in G} g$ involves the choice of an object  $(\oplus_{g\in G } gC, (e^{C}_{g})_{g\in G })$ for every object $C$ of $\bC$. The upper horizontal arrow in \eqref{fevrfervwevwevefvfvs}    then identifies the space of morphisms from $C$ to $C'$ in $\Lzwei(G,\bC)$   as follows:
\begin{equation}
\label{eq_morphisms_LtwoG}
\Hom_{\Lzwei(G,   \bC)}(C,C') \cong \Hom_{\bC}\big(\oplus_{g\in G}gC,\oplus_{g\in G}gC'\big)\,.
\end{equation}
\item The composition and the involutions are inherited from $\bC$.\hB
\end{enumerate}
The upper horizontal arrow in \eqref{fevrfervwevwevefvfvs}  is a fully faithful inclusion of 
  $\Lzwei(G,   \bC)$ into the $W^{*}$-category $\bC$. Therefore   $\Lzwei(G,   \bC)$ is itself a $W^{*}$-category.
\end{rem}

Using the universal property of the pull-back defining  $\Lzwei(G, \bC)$ we construct a  {morphism}
$\sigma\colon\bC\to \Lzwei(G,  \bC)$  in $\Ccat$ based on the following diagram: 
   \begin{equation}\label{erglkmnelfweqfwefqewfq}
\xymatrix{\bC\ar@{..>}[dr]^-{\sigma}\ar@/_1pc/[ddr]_-{\eqref{sgbsfdvfsvfvdfvsv}}\ar@/^1pc/[rrrd]^-{\oplus_{g\in G} g}&&\\
&\Lzwei(G,  \bC)\ar[rr]\ar[d] &&\bC\ar[d]^{\eqref{sgbsfdvfsvfvdfvsv}}\\ &
0[\Ob(\bC)]\ar[rr]^{0[\Ob(\oplus_{g\in G} g)] }&&0[\Ob(\bC)] }
\end{equation}

\begin{rem}\label{rem_explicit_description_sigma} 
Using the explicit description of $\Lzwei(G,   \bC)$ given in Remark~\ref{rem_explicit_description_Lzwei} we can give an explicit description of the {morphism} $\sigma$:
%\Alex{Recall from  the explicit description of $\Lzwei(G,\tilde  \bC)$.}
%The functor $\sigma$ has the following description:
\begin{enumerate}
\item objects: In view of the left   triangle in \eqref{erglkmnelfweqfwefqewfq}    
 the action of $\sigma$  on objects is the identity under the identification in Remark
\ref{rem_explicit_description_Lzwei}.\ref{wrtoijhwrthtgwerregw}.
 %objects of $\Lzwei(G,\tilde \bC)$ with the objects of $\bC$, $\sigma$ is the identity on objects 
 \item morphisms: Using the right  triangle in  \eqref{erglkmnelfweqfwefqewfq}  and   Remark  \ref{rem_explicit_description_Lzwei}.\ref{wrtoijhwrthtgwerregw1}  we see  {that} $\sigma$ sends a morphism $f\colon C\to C'$ to the morphism
$$\oplus_{g\in G} gf\colon \bigoplus_{g\in G}gC\to\bigoplus_{g\in G}gC'$$ in $\Lzwei(G,   \bC)$. Note that one can write this also as
\begin{equation}\label{trwiohiovfjdovsdfvfdvsdfvsfdv}
\sigma(f)=\sum_{g\in G}   e^{C'}_{g} g(f) e^{C,*}_{g}\ ,
\end{equation}
\end{enumerate}
where $(e^{C}_{g})_{g \in G}$ and $(e^{C'}_{g})_{g \in G}$ are the {families of} isometries from the choices of the orthogonal sums $(\bigoplus_{g\in G} gC, (e^{C}_{g})_{g\in G})$ and $(\bigoplus_{g\in G} gC', (e^{C'}_{g})_{g\in G})$.
The morphism $\sigma(f)$    is an instance of  \eqref{ervfdvsdfvfdvfsdvsfdv}.
\hB
\end{rem}

%Recall that $\bC$ is in $\Fun(BG,\Ccat)$. 
We next recall the notion of a covariant representation  \cite[Defn.~5.4]{crosscat} of $  \bC$   on an object $\bD$ in $\nCcat$.  For this definition  $\bC$ can be any object in $\Fun(BG,\nClincat)$.
\begin{ddd}\label{ergiuqehrguiergwergewferfwrefwref}
A covariant  representation of $ \bC$  on $\bD$ is a pair $(\sigma,\pi)$  consisting of:
\begin{enumerate}
\item a {morphism} $\sigma\colon\bC\to \bD$ {(in $\nClincat)$}
\item \label{iuhuierhuihriuhgqergrgqre} a family $\pi=(\pi(g))_{g\in G}$ of unitary natural multiplier isomorphisms  $\pi(g)\colon\sigma\to g^{*}\sigma $ such that $g^{*}\pi(g')\circ \pi(g)=\pi(g'g)$ for all $g,g'$ in $G$.
\end{enumerate}
\end{ddd}

\begin{rem}
The Definition \ref{ergiuqehrguiergwergewferfwrefwref} is slightly more general than \cite[Defn.~5.4]{crosscat} since here  we allow that $\pi$ takes values in multiplier morphisms instead of  just morphisms in $\bD$. The difference is relevant in the case wher $\bD$ is non-unital.  \hB  \end{rem}

Recall that $\bC=\bW\bM\bK$ is an object of $\Fun(BG,\Ccat)$.

 \begin{lem}%\fuli{AV: bleibt richtig}
\label{qeriuhiqerfqwefqwfqwefqewf}
The morphism $\sigma\colon\bC\to \Lzwei(G,  \bC)$ has a canonical extension to a covariant representation $(\sigma,\pi)$ of $  \bC$ on $\Lzwei(G,  \bC)$.
\end{lem}

\begin{proof}
We will use the explicit descriptions of $ \Lzwei(G, \bC)$ and $\sigma$ given  in
  Remarks~\ref{rem_explicit_description_Lzwei} and \ref{rem_explicit_description_sigma}. We must describe $\pi$.
 %For $h$ in $G$ the  functor $h^{*}\sigma$ sends the object $C$ of $\bC$ to the object $hC$ in $\Lzwei(G,  \bC)$, and the   morphism $f\colon C\to C'$ to the morphism \begin{equation}\label{eqwfwedqwdwedqewdqewdq}
%\oplus_{g\in G} ghf\colon \bigoplus_{g\in G}ghC\to\bigoplus_{g\in G}ghC'
%\end{equation}
 %in $\Lzwei(G , \bC)$.
For every object $C$ of $\bC$ we define, applying Corollary 
\ref{qergijoqregqewfewfqewfwef1}.\ref{qerjgqkrfewfewfqfefqef111} %\fuli{AV: still aplies} 
to the family
$(e^{C,*}_{gh})_{g\in G}$ of morphisms $ e^{C,*}_{gh} \colon \bigoplus_{g\in G}g C\to ghC$, the morphism
\begin{equation}\label{dfsviuhqr3iugerefsv}
\pi(h)_{C} \coloneqq \sum_{g\in G} e^{hC}_{g}e^{C,*}_{gh}\colon \bigoplus_{g\in G}g C\to  \bigoplus_{g\in G}  g hC\, .
\end{equation} 
It is straightforward to check that 
the family $\pi(h) \coloneqq (\pi(h)_{C})_{C\in \Ob(\bC)}$ is a unitary natural transformation from $ \sigma$ to $h^{*}\sigma$. One checks  furthermore   that the family $\pi \coloneqq (\pi(h)_{C})_{h\in G}$ satisfies the cocycle condition in Definition \ref{ergiuqehrguiergwergewferfwrefwref}.\ref{iuhuierhuihriuhgqergrgqre}.
\end{proof}

%
%  \begin{eqnarray*} 
%  g^{*}\pi(h')_{hC}\pi(h)_{C}&=& \sum_{g\in G} e^{h'hC}_{g}e^{hC,*}_{gh'}  \sum_{g'\in G} e^{hC}_{g'}e^{C,*}_{g'h}\\&=&
%   \sum_{g\in G} e^{h'hC}_{g} e^{C,*}_{gh'h}\\&=&\pi(h'h)_{C}
%  \end{eqnarray*}
%

% 
%We then define a $C^{*}$-category $C(G,\bC)$.
%\begin{enumerate}
%\item \label{qtgoijrblmwregw} objects: The objects of $C(G,\bC)$ are the objects of $\bC$. In order to define the space of morphisms below we adopt  for every object $C$ of $\bC$ a choice of a direct sum $(\bigoplus_{g\in G} gC, (e^{C}_{g})_{g\in G})$ of the family  $(gC)_{g\in G}$  of objects in $\bC$.
%\item morphisms: The space of morphisms from $C$ to $C'$ in $C(G,\bC)$ is given using the choices of sums in \ref{qtgoijrblmwregw} by
%$$\Hom_{C(G,\bC)}(C,C'):=\Hom_{\bC}(\bigoplus_{g\in G}gC,\bigoplus_{g\in G}gC')\ .$$
%\item The composition and the involutions are inherited from $\bC$.
%\end{enumerate}

%The category $\bC$ admits a covariant representation $(\sigma,\pi)$ in the $C^{*}$-category
%$\bC(G,\bC)$. The functor $\sigma:\bC\to \bC(G,\bC)$  has the following description:
%\begin{enumerate}
%\item objects: The functor is acts as identity on the objects.
%\item morphisms: If $A:C\to C'$ is a morphism in $\bC$, then $$\sigma(A):=\oplus_{g\in G} gA:\bigoplus_{g\in G}gC\to \bigoplus_{g\in G}gC'$$ in $C(G,\bC)$.
%\end{enumerate}

   According to \cite[Defn.~5.1]{crosscat} we can form  {the} 
   algebraic crossed product
$$  \bK\rtimes^{\alg}G$$ in $\Clincat$. 
Instead of repeating the definition of the crossed product we proceed   with observing that 
by  \cite[Lem.~5.7]{crosscat} the covariant representation 
$(\sigma,\pi)$ from Lemma \ref{qeriuhiqerfqwefqwfqwefqewf}  induces a   {morphism} 
\begin{equation}\label{ewrgfviubiuhuihfefwerfwerfw}
\rho\colon  \bK\rtimes^{\alg}G\to \Lzwei(G,  \bC)\, .
\end{equation} 
In our situation this functor is  wide and faithful, and we can describe the algebraic crossed product $  \bK\rtimes^{\alg}G$  directly
as a $\C$-linear $*$-subcategory of $\Lzwei(G, \bC)$:
\begin{enumerate}
\item objects: The {set of} objects of $  \bK\rtimes^{\alg}G$ {is the set of}  objects of $\bK$ and hence of $\Lzwei(G,  \bC)$.
\item morphisms: The $\C$-vector space of morphisms $ \Hom_{ \bK\rtimes^{\alg}G}(C,C')$ is {linearly} generated as a subspace of $\Hom_{\Lzwei(G,  \bC)}(C,C')$   by the morphisms \begin{equation}\label{qewfiuhuiqhefiuqwefqewfeqwfwefqwf}
(f,g):=\rho(f,g)=\pi(g)_{g^{-1}C'}\sigma(f)
\end{equation}  for all $g$ in $G$ and $f\colon C\to g^{-1}C'$ in $\bK$.\label{item_morphisms_reduced_crossed_product}
\item The composition and the involution are inherited from $\Lzwei(G,  \bC)$.
\end{enumerate}
One easily checks using the algebraic relations for a covariant representation  that this describes a well-defined subcategory which is equivalent to the algebraic crossed product $ \bK\rtimes^{\alg}G$ defined in \cite[Defn.~5.1]{crosscat}.

%Using the notation \uli{$(f,g)$ as in \eqref{qewfiuhuiqhefiuqwefqewfeqwfwefqwf}  for a  morphism $C\to C'$ in $\bC\rtimes^{\alg}G$ with   $f\colon C\to g^{-1}C'$ (see also    \cite[Defn.~5.1]{crosscat})}, 
For a morphism $(f,g):C\to C'$ in    $\bK\rtimes^{\alg}G$ we calculate, using the formulas \eqref{dfsviuhqr3iugerefsv} for $\pi $ and \eqref{trwiohiovfjdovsdfvfdvsdfvsfdv} of $\sigma$,  that  %{(using \cite[(5.5)]{crosscat} for the first equality)}
\begin{equation}\label{evidfnviojafoivsdvasdvasdvasdv}
\rho(f,g) = %\stackrel{\eqref{qewfiuhuiqhefiuqwefqewfeqwfwefqwf} }{=} \pi(g)_{g^{-1}C'} \sigma(f)\stackrel{\eqref{dfsviuhqr3iugerefsv},\eqref{trwiohiovfjdovsdfvfdvsdfvsfdv}}{=}   \sum_{\ell\in G} e^{C'}_{\ell }e^{g^{-1}C',*}_{\ell g}   \sum_{\ell'\in G}   e^{g^{-1}C'}_{\ell'} \ell'(f) e^{C,*}_{\ell'}      =
\sum_{\ell\in G}e^{C'}_{\ell}  (\ell g )f e^{C,*}_{\ell g}\, .
\end{equation}
%where we {used} the orthogonality relations for the family $({e}^{g^{-1}C'}_{\ell})_{\ell\in  G}$ for the last equality.

 \begin{rem} \label{wetkgijowegfregrwegerf}
The morphism $\pi(h)_{C}$  constructed in the proof of Lemma \ref{qeriuhiqerfqwefqwfqwefqewf}  and  $\rho(f,g)$ from \eqref{evidfnviojafoivsdvasdvasdvasdv}    are   morphisms in $\bC$. 

If $\bK$ admits AV-sums  of cardinality $|G|$, then  in view Theorem  \ref{wergijowergerfwrefwerf} we can choose AV-sums in the definition  of the morphism $\oplus_{g\in G}g$. In this case
  one can check using  Lemma \ref{wrtjihordbvdfgbdb} that $\pi(h)_{C}$ and  $\rho(f,g)$  actually belong to $\bM\bK$.  
  
   But note that 
  this property is not invariant under changes of the choices involved
  in the construction of $\Lzwei(G,\bC)$.  \hB
\end{rem}

%By \cite[Lemma 5.6]{crosscat} the covariant representation $(\sigma,\pi)$ induces a representation
%\begin{equation}\label{wrefnvweerwcfvs}
%\bC\rtimes^{\alg}G\to \bC(G,\bC)\ .
%\end{equation}

%Recall that $\bK$ is in $\Fun(BG,\nCcat)$, $\bC:=\bW\bM\bK$ is in $\Fun(BG,\Wcat)$, and that we assume that $\bC$ admits $G$-indxed  orthogonal sums.
%Let $ \bC$ be a $|G|$-additive object of $\Fun(BG,\Ccat)$.

Recall that our standing hypothesis is that $\bC=\bW\bM\bK$ admits sums of cardinality $|G|$.
\begin{ddd}\label{ewtiojgwergerggwggr}
The reduced crossed product $  \bK\rtimes_{r}G$ {is} {defined to be} the closure of  
$ \bK\rtimes^{\alg}G$ with respect to the norm induced by the  representation $\rho$ {in \eqref{ewrgfviubiuhuihfefwerfwerfw}.}
\end{ddd}

Equivalently, $  \bK\rtimes_{r}G$ is the closure of $ \bK\rtimes^{\alg}G$ viewed as a subcategory of  $\Lzwei(G, \bC)$. It follows from the uniqueness
of orthogonal sums up to unique unitary  isomorphism that the
  reduced crossed product is well-defined independently of the choices involved in the   construction of $\Lzwei(G,\bC)$ and  $\rho$.

%Let $\alpha$ be a very small cardinal and 
%Recall the notation introduced in Section \ref{qeoirghowergerwgwregwergw}.
%In particular,  $\widehat{\Ccat}^{|G|\add}$ is the subcategory of $\Ccat$ of $|G|$-additive categories and $|G|$-additive functors.%Recall that we denote by $\Ccat^{\alpha\add}$ the full subcategory of $\Ccat$ of $\alpha$-additive $C^*$-categories (Definition~\ref{ergiehgijiorewgergwergweg}), and we further denote by $\Ccat^{\alpha\add}_{F\alpha\add}$ the wide subcategory of $\Ccat^{\alpha\add}$ of $\alpha$-additive $C^*$-functors (Definition~\ref{defn_alpha_additive_functors}).}

%Let $\Ccat^{|G|\mathrm{add}}$ denote the full subcategory of $\Ccat$ of \Alex{$|G|$-additive $C^{*}$-categories.}

%Note that our standing hypothesis is that $\bC=\bW\bM\bK$ admits    orthogonal sums of cardinality $|G|$. %Below we will extend this construction to arbitrary $\bK$ in $ \Fun(BG,\nCcat)$.
We let $\nCcat_{s\add}$ denote the full subcategory of $\nCcat$ of categories $\bK$ with the property that $\bW\bM\bK$  admits  all very  small orthogonal sums.
\begin{lem}\label{qrgiojfwewfqdedqwedqwed} 
The construction of the reduced crossed product has a canonical extension to a  functor
\[- \rtimes_{r}G \colon \Fun(BG,\nCcat_{s\add})\to \nCcat\, .\]
The functor preserves fully faithfulness of morphisms.
\end{lem}
\begin{proof}
 Definition \ref{ewtiojgwergerggwggr} provides the 
 action of the  functor $- \rtimes_{r}G$  on objects. We must extend it to morphisms. Thus let
$  \phi \colon   \bK\to  \bK'$ be a morphism in $\Fun(BG,\nCcat_{s\add} )$. 
It induces a  {morphism}
$$ \phi  \rtimes^{\alg} G \colon     \bK   \rtimes^{\alg} G \to   \bK' \rtimes^{\alg} G$$
in $\nClincat$ in a functorial way.
We must show that it extends by continuity to the reduced crossed products.
To this end we construct a commutative  diagram
$$\xymatrix{\bK\ar[d]\ar[rr]^{\phi}&&\bK'\ar[d]\\\bC\ar[rr]^{\bW^{\mathrm{nu}}\phi}\ar[d]^-{\sigma}&&\bC'\ar[d]^-{\sigma'}\\\Lzwei(G,  \bC) \ar[rr]^{\Lzwei(G,  \phi)}&& \Lzwei(G,  \bC')}$$
In order to interpret the middle arrow we identify $\bC\cong \bW^{\mathrm{nu}} \bK$ and $\bC'\cong \bW^{\mathrm{nu}} \bK'$ using Theorem \ref{jigogewrgeferfref}.  If $\phi$ is  faithful and non-degenerate, then $ \bW^{\mathrm{nu}}\phi=\bW\bM\phi$, but for general $\phi$ the extension $\bM\phi$  to the multiplier category might not exist.
%This is important since for the existence of  $\bM\phi$ 

%The pull-back square in \eqref{fevrfervwevwevefvfvs} is also a pull-back square in  {the $(2,1)$-category} $\Ccat_{2,1}$ of unital $C^{*}$-categories, unital functors and unitary natural isomorphisms. 
 On objects the functor 
  $\Lzwei(G,  \phi) $ acts as $\phi$. In order to define the action of
  $\Lzwei(G,  \phi)$ on morphisms note that 
 by Corollary   \ref{ewtkohijorthrhdrhrhrdth} and the equivariance of $ \phi$, for every object $C$ of $\bK$ we have a unitary $$u_{C}: \phi(\oplus_{g\in G}gC)\to \oplus_{g\in G} g
 \phi(C)$$ in $\bC$ 
 which is uniquely determined by the condition that 
 $$ e_{g}^{\phi(C),*} u_{C} \bW^{\mathrm{nu}}  \phi(e_{k}^{C})=  \left\{\begin{array}{cc}\id_{g\phi(C)}&g=k,\\0&\text{else}.\end{array}\right.$$
 For $C,C'$ in $\bK$ we then define 
 $$\Lzwei(G,  \phi):\Hom_{\Lzwei(G,  \bC)}(C,C')\to 
 \Hom_{\Lzwei(G,  \bC')}(\phi(C),\phi(C'))$$ as
 \begin{eqnarray}
\Hom_{\Lzwei(G,  \bC)}(C,C')&\stackrel{\eqref{eq_morphisms_LtwoG}}{\cong}&\Hom_{\bC}(\oplus_{g\in G}gC,\oplus_{g\in G}gC') \label{ERKOHPERHTRHEHTRHE}
 \\&\stackrel{\bW^{\mathrm{nu}}\phi}{\to}&
 \Hom_{\bC')}(\phi(\oplus_{g\in G}gC),\phi(\oplus_{g\in G}gC') ) \nonumber\\&\stackrel{u_{C'}\circ-\circ u_{C}^{-1}}{\cong}& \Hom_{\bC')}( \oplus_{g\in G}g\phi(C),\oplus_{g\in G}g\phi(C') )
 \nonumber\\& \stackrel{\eqref{eq_morphisms_LtwoG}}{\cong}&\Hom_{\Lzwei(G,  \bC')}(\phi(C),\phi(C'))\ . \nonumber
\end{eqnarray}
 One checks that this description is compatible with the composition of morphisms and the involution.

%
%{morphism} $ \bW\bM \phi:\bC\to \bC'$ preserves the {morphism} $\oplus_{g\in G}g$ up to unique unitary isomorphism. For ev 
%
%  It follows that
%the {$  \phi$} induces a morphism of pull-back squares  of the shape \eqref{fevrfervwevwevefvfvs} in $\Ccat_{2,1}$, and therefore a
%{morphism} $$\Lzwei(G,  \phi) \colon \Lzwei(G,  \bC)\to \Lzwei(G,  \bC')$$ {in $\Ccat$}. %which is uniquely  determined up to unitary isomorphism. 
%Furthermore, in view of the construction of $\sigma$ in \eqref{erglkmnelfweqfwefqewfq} we get a commutative  diagram
%$$\xymatrix{\bK\ar[d]\ar[rr]^{\phi}&&\bK'\ar[d]\\\bC\ar[rr]^{\bW\bM\phi}\ar[d]^-{\sigma}&&\bC'\ar[d]^-{\sigma'}\\\Lzwei(G,  \bC) \ar[rr]^{\Lzwei(G,  \phi)}&& \Lzwei(G,  \bC')}\ .$$
%

 One now checks using the explicit descriptions that
$\Lzwei(G,  \phi)$ restricts to a morphism $\bK\rtimes^{\alg}G\to \bK'\rtimes^{\alg}G$  {in $\nClincat$, where} the algebraic crossed products are viewed as subcategories of $\Lzwei(G, \bC)$ and $ \Lzwei(G, \bC')$, respectively, and that this restriction is equivalent to
$  \phi \rtimes^{\alg} G$. {Thus we} can  define $  \phi  \rtimes_{r} G \colon   \bK \rtimes_{r}G \to   \bK' \rtimes_{r} G $ as the continuous extension of 
  $  \phi\rtimes^{\alg}G$, {given explicitly by the restriction of $\Lzwei(G,  \phi)$ to the crossed products viewed as   subcategories of $\Lzwei(G,  \bC)$ and $\Lzwei(G,  \bC')$, respectively.} 

One  finally checks in a straightforward manner that $-\rtimes_{r}G$ is compatible with the composition of morphisms in {$\Fun(BG,\nCcat_{s\add} )$.}

We now assume that $\phi$ is fully faithful. Then by Proposition \ref{stkghosgfgsrgsegs} the functor 
$\bM\phi $  is fully  faithful and  it follows  from  Proposition \ref{rjigowergwregregwfer} that  the functor 
$\bW\bM\phi \cong \bW^{\mathrm{nu}}\phi$  is 
 fully faithful, too.  This implies that  the maps \eqref{ERKOHPERHTRHEHTRHE}
 are  isomorphisms for all objects $C,C'$ in $\bK$.
 Hence   $\phi\rtimes_{r}G$ is fully faithful.
\end{proof}
 
 Recall that the reduced crossed product is constructed above under an additional additivity assumption. We must 
  extend the domain of the  functor $-\rtimes_{r}G$ from Lemma \ref{qrgiojfwewfqdedqwedqwed}
 to all of $\Fun(BG,\nCcat)$. We proceed with the following steps which will be referred to as steps of the construction of  the reduced crossed product.
\begin{enumerate}
\item \label{wkegjwoergergerwferf}  We construct the reduced crossed product $\bL\rtimes_{r}G$ for  all $\bL$ in $\Fun(BG,\nCcat)$ which admit a fully faithful morphism
$\bL\to \bK$ in $\Fun(BG,\nCcat)$ with $\bK$ in $ \Fun(BG,\nCcat_{s\add})$.
\item \label{wkegjwoergergerwferf15}   If $\phi:\bL\to \bL'$ is a fully faithful  morphism  $\Fun(BG,\nCcat)$ and  Step \ref{wkegjwoergergerwferf} applies to  $\bL'$, then it applies to $\bM\bL'$ and $\bL$. Furthermore,  
$\phi\rtimes^{\alg}G:\bL\rtimes^{\alg}G\to \bL'\rtimes^{\alg}G$
and
$\bL'\rtimes^{\alg}G\to \bM\bL'\rtimes^{\alg}G$  extend to
fully faithful  morphisms
$\phi\rtimes_{r} G:\bL\rtimes_{r}G\to \bL'\rtimes_{r}G $ and $ G:\bL'\rtimes_{r}G\to \bM \bL'\rtimes_{r}G $.
\item \label{wkegjwoergergerwferf1}  We   construct  the reduced crossed product  $\bH\rtimes_{r}G$ for all $\bH$ in $\Fun(BG,\nCcat)$ which receive  a unitary  equivalence     $\bL\to \bH$ in $\Fun(BG,\nCcat)$  from some 
 $\bL$ considered in Step \ref{wkegjwoergergerwferf}.   
  \item \label{wkegjwoergergerwferf2}  We verify that the categories appearing in Step \ref{wkegjwoergergerwferf1}  exhaust all of $\Fun(BG,\nCcat)$.
 \item  \label{wkegjwoergergerwferf3}  We   check that for every morphism $\phi:\bH\to \bH'$ in $  \Fun(BG,\nCcat)$ the morphism
 $\phi: \bH\rtimes^{\alg}G\to  \bH'\rtimes^{\alg}G$ has a continuous extension to the reduced crossed products. Furthermore, if $\phi$ is fully faithful, then so is $\phi\rtimes_{r}G$.
 \end{enumerate}

% 
% 
% stepwise first  to  the category  $\Fun(BG,\nCcat)^{'}  $ of $\Fun(BG,\nCcat)$  of $ C^{*}$-categories with strict $G$-action  in
% $\Fun(BG,\nCcat)$ which admit an equivariant fully faithful embedding into a an object of $ \Fun(BG,\nCcat_{s\add})$, and  then
% to all of $\Fun(BG,\nCcat)$.
% 
% 
 
 We start with Step \ref{wkegjwoergergerwferf}.
Assume that $\bL$ is in $\Fun(BG,\nCcat)$ and 
$\phi:\bL\to \bK$ is a fully faithful morphism in $\Fun(BG,\nCcat)$
such that $\bK$ belongs to $\Fun(BG,\nCcat_{s\add})$. Then  we get a fully faithful  morphism
$$ \phi  \rtimes^{\alg} G \colon     \bL   \rtimes^{\alg} G \to   \bK \rtimes^{\alg} G$$
in $\nClincat$. We want to construct
$    \bL   \rtimes_{r}G$ as the completion of $   \bL   \rtimes^{\alg} G$ with respect to the norm induced from $ \bK \rtimes_{r} G$ via $ \phi  \rtimes^{\alg} G$. We must check that this norm does not depend on the choice of the embedding
$\phi:\bL\to \bK$. To this end we consider a second such embedding 
 $\phi':\bL\to \bK'$.  
%we have   fully faithful  inclusions
%$$ \phi  \rtimes^{\alg} G \colon     \bL   \rtimes^{\alg} G \to   \bK \rtimes^{\alg} G\ , \quad  \phi'  \rtimes^{\alg} G \colon     \bL   \rtimes^{\alg} G \to   \bK' \rtimes^{\alg} G \ .$$
%The inclusions   $\bK \rtimes^{\alg} G\to   \bK \rtimes_{r}G$
%and  $\bK' \rtimes^{\alg} G\to   \bK' \rtimes_{r}G$ 
%induce the reduced norms on the respective algebraic crossed products.
 \begin{lem}\label{qefqwefewqfq}
 The norms on $   \bL   \rtimes^{\alg} G$ induced via $ \phi  \rtimes^{\alg} G$ and  $\phi'  \rtimes^{\alg} G$  are equal.
 \end{lem}
\begin{proof} 
By Proposition \ref{stkghosgfgsrgsegs} the functors
$\bM\phi:\bM\bL\to \bM\bK$ and $\bM\phi':\bM\bL\to \bM\bK'$ exist and are fully faithful. By Proposition \ref{rjigowergwregregwfer} the functors 
$\bW\bM\phi:\bW\bM\bL\to \bC:=\bW\bM\bK$ and $\bW\bM\phi':\bW\bM\bL\to \bC':=\bW\bM\bK'$
are fully faithful. 

Let $D_{0},D_{1}$ be two   objects of $\bL$ and hence of $\bL\rtimes^{\alg}G$.
For $i$ in $\{0,1\}$ we  set   $C_{i}:=\phi(D_{i})$ and $C_{i}':=\phi'(D_{i})$.    
 Then $C_{i}$ are also objects of $\bK\rtimes^{\alg}G$, and  $C_{i}'$ are objects of $\bK'\rtimes^{\alg}G$.
Using $\rho$ from \eqref{ewrgfviubiuhuihfefwerfwerfw}  and the isomorphism \eqref{eq_morphisms_LtwoG} we have identified  $\Hom_{\bK\rtimes^{\alg}G}(C_{0},C_{1})$
with  a linear subspace of 
$\Hom_{\bC}(\oplus_{g\in G} gC_{0},\oplus_{g\in G}gC_{1})$. 
Similarly, we have an inclusion $\rho'$ of 
   $\Hom_{\bK'\times^{\alg}G}(C'_{0},C'_{1})$
 as linear subspace  of $\Hom_{\bC'}(\oplus_{g\in G} gC'_{0},\oplus_{g\in G}gC'_{1})$.
 By Proposition \ref{woijotrgwergeferfwerfewrf}
 we have an isometry $$\Phi_{D,D'}:\Hom_{\bC}(\oplus_{g\in G} gC_{0},\oplus_{g\in G}gC_{1})\to \Hom_{\bC'}(\oplus_{g\in G} gC'_{0},\oplus_{g\in G}gC'_{1})$$ that is uniquely determined by the condition that
  $$e_{k}^{C_{1}',*}\Phi_{D,D'}(h)e_{g}^{C_{0}'}=\bW\bM\phi'((\bW\bM\phi)^{-1}(e_{k}^{C_{1},*}he_{g}^{C_{0}}))$$ for all $g,k$ in $G$ and $h$ in $ \Hom_{\bC}(\oplus_{g\in G} gC_{0},\oplus_{g\in G}gC_{1})$. Using the explicit formula \eqref{evidfnviojafoivsdvasdvasdvasdv}  for $\rho$ (and  similarly for $\rho'$) one checks that 
  the following  diagram commutes: $$\xymatrix{ &\Hom_{\bL\times^{\alg}G}(D_{0},D_{1} ) \ar[dr]^{ \phi'  \rtimes^{\alg} G}\ar[dl]_{\phi  \rtimes^{\alg} G}&\\ \Hom_{\bK\rtimes^{\alg}G}(C_{0},C_{1})\ar[rr]^{}\ar[d]^{\rho}&&\ar[d]^{\rho'}\Hom_{\bK'\rtimes^{\alg}G}(C'_{0},C'_{1})\\\Hom_{\bC}(\oplus_{g} gC_{0},\oplus_{g}gC_{1})
 \ar[rr]^{\Phi_{D,D'}}&&\Hom_{\bC'}(\oplus_{g\in G} gC'_{0},\oplus_{g\in G}gC'_{1})}$$
  Since $\rho$ and $\rho'$ are isometries by definition  this shows the assertion.  \end{proof}
  This finishes  Step \ref{wkegjwoergergerwferf}.
  
  We proceed with  Step  \ref{wkegjwoergergerwferf15}.
If $\phi:\bL\to \bL'$ is an equivariant fully faithul functor and $\psi:\bL'\to \bK$ is a fully faithful functor with $\bK$ in $\Fun(BG,\nCcat_{s\add})$, then 
   we can form    the diagram
   $$\xymatrix{\bL \ar[r]^{\phi}&\bL'\ar[r]^{\psi} \ar[d]&\bK\ar[d]&\\  &\bM\bL'\ar[r]^{\bM\psi} &\bM\bK\ar[r]&\bW\bM\bK}$$
   where the  functor  $\bM\psi$  exists and is  fully faithful by Proposition \ref{stkghosgfgsrgsegs}.   
   The lower line shows that Step \ref{wkegjwoergergerwferf}  applies to
   $\bM\bL'$, and the upper line shows that this step applies to $\bL$.
   Since  the  reduced norms on $\bL\rtimes^{\alg}G$ and $\bL'\rtimes^{\alg}G$ and
   $\bM\bL'\rtimes^{\alg}G$ are eventually all induced from
   the reduced  norm on $\bM\bK\rtimes^{\alg}G$ we see that 
   the  morphisms 
   $$\bL\rtimes_{r}G
\to    \bL'\rtimes_{r}G\to \bM\bL'\rtimes_{r}G$$ are all isometric.
The latter is an inclusion of an ideal. This finishes Step  \ref{wkegjwoergergerwferf15}.

We now consider Step \ref{wkegjwoergergerwferf1}.
  Let  $\bH$ be in $\Fun(BG,\nCcat)$ and assume that
  $\phi:\bL\to \bH$ is a unitary equivalence  in $\Fun(BG,\nCcat)$
  from an object $\bL$ to which    Step \ref{wkegjwoergergerwferf} applies.
  Then we get  a unitary equivalence
  $\phi\rtimes^{\alg}G:\bL\rtimes^{\alg}G\to \bH\rtimes^{\alg}G$.
  Note that $\bL\rtimes^{\alg}G$ has a well-defined reduced norm by
  Step \ref{wkegjwoergergerwferf}.
  We want to define the norm on $\bH\rtimes^{\alg}G$ such that
  $\phi\rtimes^{\alg}G$ becomes an isometry and then define
  $\bH\rtimes_{r}G$  as the completion. We must check that
the norm is well-defined.

 Let  $H_{0},H_{1}$ be objects of $\bH$. Then we can choose objects $L_{0},L_{1}$ in $\bL$ and
 unitary multiplier equivalences
 $u_{i}:\phi(L_{i})\to H_{i}$ in $\bH$ for $i$ in $\{0,1\}$.  
 As said above we want to define the norm
on $ \Hom_{\bH\rtimes^{\alg}G}(H_{0},H_{1})$ such that
$$   u_{1}\circ \phi(-)\circ u_{0}^{-1} :\Hom_{\bL\rtimes^{\alg}G}(L_{0},L_{1})  \to \Hom_{\bH\rtimes^{\alg}G}(H_{0},H_{1}) $$
becomes an isometry. We must check that this does not depend on the choices. 

\begin{lem} The norm  $\Hom_{\bH\rtimes^{\alg}G}(H_{0},H_{1})$ described above does not depend on the choices of $\phi$ and $L_{i}$ and $u_{i}$.
\end{lem}
\begin{proof}
For the moment we fix $\phi$ and let $L_{i}'$ and $u_{i}'$ denote  another choice.
Since $\bM\phi$ is fully faithful the unitary multipliers   $u_{i}^{\prime,-1}\circ u_{i}:  \phi(L_{i})\to \phi(L_{i}')$ lift to  unitary multipliers
 $v_{i}:L_{i}\to L_{i}'$ in $\bL$. We then have  unitaries 
   $(v_{i},e):L_{i}\to L_{i}'$ in $\bM\bL\rtimes^{\alg} G$. 
   Since  $\bM\bL\rtimes_{r}G$ is defined by Step \ref{wkegjwoergergerwferf15}
 we conclude that
 $(v_{i},e)$ induce  unitaries from $L_{i}$ to $L_{i}'$ in $\bM\bL\rtimes_{r}G$ and therefore
  unitary multiplier isomorphisms
between the same objects in $\bL\rtimes_{r}G$.
We can now conclude that
$$(v_{1},e)\circ -  \circ (v_{0},e)^{-1} :\Hom_{\bL\rtimes_{r}G}(L_{0},L_{1})
 \to   \Hom_{\bL\rtimes_{r}G}(L'_{0},L'_{1})
$$ is an isometry. Since
$$\xymatrix{\Hom_{\bL\rtimes_{r}G}(L_{0},L_{1})\ar[rr]^{(v_{1},e)\circ -  \circ (v_{0},e)^{-1}}\ar[dr]_{ u_{1}\circ \phi(-)\circ u_{0}^{-1}\qquad} && \Hom_{\bL\rtimes_{r}G}(L'_{0},L'_{1})\ar[dl]^{\quad u'_{1}\circ \phi(-)\circ u_{0}^{\prime,-1}}\\ 
&\Hom_{\bH\rtimes^{\alg}G}(H_{0},H_{1})& }$$
commutes the induced norm on $ \Hom_{\bH\rtimes^{\alg}G}(H_{0},H_{1})$ does not depend on the choices made above for fixed $\phi$.

We now  
 consider a second choice
 $\phi':\bL'\to \bH$. Since $\phi'$ is a unitary equivalence, by Lemma  \ref{ekgjosgwregrewgwerwfrwef} there exists a weakly equivariant
 inverse $\psi':\bH\to \bL'$.  Applying Lemma \ref{eoigjosegbsfdgsertsfd} to the weakly equivariant morphism $\psi'\circ \phi\circ p_{\bL}$ 
 we get an equivariant   morphism  $\xi:Q(\bL)\to \bL'$ togther with a unitary multiplier isomorphism of weakly equivariant 
  morphisms $\kappa: \psi'\circ \phi\circ p_{\bL}\to \xi$.
% we obtain 
  %a square
%$$\xymatrix{Q(\bL)\ar[r]^{\xi}\ar[d]^{p_{\bL}}&\bL'\ar[d]^{\phi'}\\\bL\ar[r]^{\phi}&\bH\ar@/^-1cm/@{..>}[u]^{\psi'}}$$
%which commutes up to a unitary multiplier isomorphism of weakly equivariant (in this case actually equivariant) morphisms.
Note that $\xi$ is also a unitary equivalence.
 We consider again objects  $H_{0},H_{1}$   of $\bH$. Then there exist objects
 $L_{0},L_{1}$ in $\bL$ and
 unitary multiplier equivalences
 $u_{i}:\phi(L_{i})\to H_{i}$ for $i$ in $\{0,1\}$.  We consider the lifts $(L_{i},e
)$ in $Q(\bL)$ and set $L_{i}':=\xi(L_{i},e)$.  By applying $\bM\phi'$ to suitable  values of $\kappa$ we get unitary multiplier isomorphisms
from $\phi'(L'_{i})$ to $ H_{i}$.  Thus we can take $L_{i}'$ as lifts of $H_{i}$ under $\phi'$. 

 By   Step \ref{wkegjwoergergerwferf15}  applied to $p_{\bL}:Q(\bL)\to \bL$ we know that  Step  \ref{wkegjwoergergerwferf} 
applies to $Q(\bL)$.
Since 
$$\Hom_{\bL\rtimes_{r}G }(L_{0},L_{1})  \stackrel{p_{\bL}\rtimes_{r}G}{\leftarrow} \Hom_{Q(\bL)\rtimes_{r}G }((L_{0},e),(L_{1},e))\stackrel{\xi\rtimes_{r}G}{\to}
 \Hom_{\bL'\rtimes_{r}G }(L'_{0},L'_{1}) $$
 are isometries (for $\xi$ we use Step  \ref{wkegjwoergergerwferf15}  again applied to $\xi :Q(\bL)\to \bL'$) we can conclude that the induced norm on
 $ \Hom_{\bH\rtimes^{\alg}G}(H_{0},H_{1})$ does not depend on
 the choice of $\phi:\bL\to \bH$. \end{proof}
  This finishes Step \ref{wkegjwoergergerwferf1}.

  We now do  Step \ref{wkegjwoergergerwferf2}.
  Let $\bH$ be in $\Fun(BG,\nCcat)$.
   Note that  the Yoneda type embedding $M:\bH\to \Hilb_{c}(A(\bH))$ from Definition \ref{eiogjoewrgerfgerrfewfwer} is  
  weakly equivariant by Lemma \ref{wothpwgergewefwef}.
   We apply  Lemma   \ref{eoigjosegbsfdgsertsfd}
  to the composition of weakly equivariant morphisms
  $M\circ p_{\bH}:Q(\bH)\to  \Hilb_{c}(A(\bH) )$ (which exists since $M$ is fully faithful) in order to get 
  a morphism $\phi:Q(\bH)\to  \Hilb_{c}(A(\bH))$ in $\Fun(BG,\nCcat)$
together with a unitary natural multiplier isomorphism
$ M\circ p_{\bH}\cong \phi$ between  weakly equivariant  morphisms. Since $
  M$ and $p_{\bH}$ are fully faithful, so is $\phi$.
  Since  $\Hilb_{c}(A(\bH))$  belongs to the large version of
  $\Fun(BG,\nCcat_{s\add})$  
   we can use $\phi$ to see that Step  \ref{wkegjwoergergerwferf} applies to $Q(\bH)$.
   But then we use $p_{\bH}$ in order to apply Step  \ref{wkegjwoergergerwferf1}  to $\bH$.

 We now do the final Step \ref{wkegjwoergergerwferf3}.  
 We consider a  morphism
   $\phi:\bH\to \bH'$   in $\Fun(BG,\nCcat)$.       
   \begin{lem} \label{erogjowergerwfewrf}If $\phi$ is injective on objects, then
the map $\phi\rtimes^{\alg} G:\bH\rtimes^{\alg}G\to \bH'\rtimes^{\alg}G$ is bounded with respect to the reduced norms.  If $\phi$ is in addition fully faithful, then $\phi\rtimes^{\alg} G$ is fully faithful and isometric. \end{lem}
 \begin{proof}
We  build the following diagram of weakly equivariant morphisms: \begin{equation}\label{asdvdsvascadscasdca}
\xymatrix{Q(\bH)\ar[rrr]^{Q(\phi)}\ar[ddd]^{\alpha}\ar[dr]^{p_{\bH}}&&& Q(\bH') \ar[dl]_{p_{\bH'}}\ar[ddd]^{\beta}&\\&\bH \ar@{-->}[d]^{M^{\bH}}\ar[r]^{\phi}&\bH'\ar@{-->}[d]^{M^{\bH'}} &\\
&\Hilb_{c}(A(\bH)) \ar@/_1cm/@{..>}[dl]\ar@{-->} [r]^{A(\phi)_{*}}&\Hilb_{c}(A(\bH')) \ar@/_-1cm/@{..>}[dr]&\\
\ar[ur]_{\qquad p_{ \Hilb_{c}(A(\bH))}} \ar[rrr]^{\gamma} Q(\Hilb_{c}(A(\bH)))&&&Q(\Hilb_{c}(A(\bH')))\ar[ul]^{p_{ \Hilb_{c}(A(\bH'))}\quad}}
\end{equation}
The bold morphisms are actually  equivariant while the remaining morphisms are weakly equivariant. In order to construct the morphisms marked by $\alpha,\beta,\gamma$
we choose weakly invariant inverses of $p_{ \Hilb_{c}(A(\bH))}$ and $p_{ \Hilb_{c}(A(\bH))}$ as indicated by the dotted arrows. Since they are fully faithful they  can be   right-composed with any further weakly equivariant morphism, so in particular with $M^{\bH}\circ p_{\bH}$, $M^{\bH'}\circ p_{\bH'}$, or 
$A(\phi)_{*}\circ p_{ \Hilb_{c}(A(\bH))}$ respectively.
We apply  Lemma   \ref{eoigjosegbsfdgsertsfd} to the respective compositions of arrows in order get 
the arrows marked by $\alpha,\beta,\gamma$
 together with fillers of the respective 
squares by unitary natural multiplier isomorphisms between
weakly equivariant functors. By  Lemma \ref{wothpwgergewefwef} the inner 
  square is also filled by such an isomorphism, while the upper square commutes on the nose.
All morphisms except the horizontal ones are fully faithful.

  We pause the proof for a moment to introduce the following contruction:

\begin{construction} \label{rtjoiherthertgeg}{\em 
In \cite[Prop. 7.12]{crosscat} we have shown that the functor $-\rtimes^{\alg}G$
extends to weakly equivariant morphisms  between unital $C^{*}$-categories and sends  uniformly bounded (unitary, respectively) natural transformations between them to  uniformly   bounded    (unitary) transformations. This extends to the non-unital case as follows.

If $(\phi,\rho):\bC\to \bD$ is weakly invariant and $(f,g):C\to C'$ is a morphism in $\bC\rtimes^{\alg}G$ with $f:C\to g^{-1}C'$ in $\bC$,
  then the induced morphism $(\phi,\rho)\rtimes^{\alg} G:\bC\rtimes^{\alg}G\to \bD\rtimes^{\alg}G$ sends
$(f,g)$  to  \begin{equation}\label{feqwfqwefqedqwe}
(\rho(g)_{g^{-1}C'} \phi(f),g):\phi(C)\to g^{-1}\phi(C')
\end{equation}  in $ \bD\rtimes^{\alg}G$. In contrast to the unital case,
 here  $\rho(g)_{C'}$ is only a multiplier morphism, but    the composition $ 
\rho(g)_{C'} \phi(f)$ still belongs to $\bD$.
If $\kappa:(\phi,\rho)\to (\phi',\rho')$ is a uniformly bounded  (unitary) natural multiplier  transformation 
between weakly equivariant morphisms, then
we get a uniformly   bounded    (unitary)    multiplier isomorphism
$\kappa\rtimes^{\alg}G:(\phi,\rho)\rtimes^{\alg} G\to (\phi',\rho')\rtimes^{\alg} G$. It is given by the  unitary  multiplier isomorphism  \begin{equation}\label{eqwfqwedwgtertgertge54}
 (\kappa\rtimes^{\alg}G)_{C}:=(\kappa(g)_{C},e):\phi(C)\to \phi'(C)
\end{equation} 
on any $C$ in $\bC\rtimes^{\alg}G$.}
\hB\end{construction}

We resume the proof of Lemma \ref{erogjowergerwfewrf}.
Using Construction \ref{rtjoiherthertgeg} we can  apply  the functor $-\rtimes^{\alg}G$ to the diagram   in \eqref{asdvdsvascadscasdca} in order to get  the new diagram
\begin{equation}\hspace{-1.5cm}
\xymatrix{Q(\bH)\rtimes^{\alg}G\ar[rrr]^{Q(\phi)\rtimes^{\alg}G}\ar[ddd]^{\alpha\rtimes^{\alg}G}\ar[dr]^{p_{\bH}\rtimes^{\alg}G}&&& Q(\bH') \rtimes^{\alg}G\ar[dl]_{p_{\bH'}\rtimes^{\alg}G}\ar[ddd]^{\beta\rtimes^{\alg}G}&\\
&\bH\rtimes^{\alg}G \ar@{-->}[d]^{M^{\bH}\rtimes^{\alg}G}\ar[r]^{\phi\rtimes^{\alg}G}&\bH'\rtimes^{\alg}G\ar@{-->}[d]^{M^{\bH'}\rtimes^{\alg}G} &\\
&\Hilb_{c}(A(\bH)) \rtimes^{\alg} G \ar@{-->} [r]^{A(\phi)_{*}\rtimes^{\alg}G}&\Hilb_{c}(A(\bH')) \rtimes^{\alg}G &\\
\ar[ur]_{p_{ \Hilb_{c}(A(\bH))}\rtimes^{\alg}G} \ar[rrr]^{\gamma\rtimes^{\alg}G} L(\Hilb_{c}(A(\bH)))\rtimes^{\alg}G&&&L(\Hilb_{c}(A(\bH')))\rtimes^{\alg}G\ar[ul]^{p_{ \Hilb_{c}(A(\bH'))}\rtimes^{\alg}G}}
\end{equation}
All squares are filled by unitary multiplier isomorphisms.
Again,  all morphisms except the horizontal ones are fully faithful.
Our task is to show that $\phi\rtimes^{\alg}G$ is bounded. 
Using the fact that the norms  on its domain and target are
induced from the norms on the domain and target of
$Q(\phi)\rtimes^{\alg}G$ via $p_{\bH}\rtimes^{\alg}G$ and $p_{\bH'}\rtimes^{\alg}G$ respectively, and since
the upper square commutes up to a unitary multiplier  isomorphism, 
 it suffices to show that $Q(\phi)\rtimes^{\alg}G$ is bounded.
By Lemma \ref{qrgiojfwewfqdedqwedqwed} the morphisms $\alpha\rtimes^{\alg}G$ and $\beta\rtimes^{\alg}G$ are fully faithful and isometric, and 
$\gamma\rtimes^{\alg}G$ is bounded.
 Since the big square is filled by a unitary multiplier isomorphism
 we can conclude that $Q(\phi)\rtimes^{\alg}G$ is bounded, too.   
 If $\phi$ is fully faithful, then  so is  the composition $A(\phi)_{*}\circ M^{\bH}$.
 Then also $\gamma\circ \alpha$ is fully faithful.
 This implies that $(\gamma\rtimes^{\alg}G)\circ (\alpha\rtimes^{\alg}G ) $ is  fully faithful and isometric  which implies that $Q(\phi)\rtimes^{\alg}G$ is  fully faithful and isometric.  Finally we conclude that
 $\phi\rtimes^{\alg}G$ is isometric and fully faithful.  \end{proof}
The Lemma \ref{erogjowergerwfewrf}  settles Step  \ref{wkegjwoergergerwferf3} for functors which are injective on objects.
 In order to finish  the argument for Step  \ref{wkegjwoergergerwferf3}
 we must remove the assumption about the injectivity of $\phi$ on objects.
 To this end we first note that   if  $\phi:\bH\to \bH'$ is a unitary equivalence, then   by Step
 \ref{wkegjwoergergerwferf2} we can find a further equivalence
 $\psi:\bL\to \bH$ such that Step
 \ref{wkegjwoergergerwferf} applies to $\bL$.
 Then $ (\psi\rtimes^{\alg} G)$ and the composition
 $(\psi\rtimes^{\alg} G)\circ (\phi\rtimes^{\alg}G)$ are  isometric
 equivalences by Step~\ref{wkegjwoergergerwferf1}. It follows that
 $\phi\rtimes^{\alg}G$ is  an isometric equivalence.

We now consider an arbitrary functor
$\phi:\bH\to \bH'$ in $\Fun(BG,\nCcat)$.
\begin{lem}\label{wkoetgerfregwr}
 $\phi\rtimes^{\alg}G:\bH\rtimes^{\alg}G\to \bH'\rtimes^{\alg}G$ is  bounded with respect to the reduced norms.  If $\phi$ is fully faithful, then $\phi\rtimes^{\alg}G$ is fully faithful and isometric.
 \end{lem}
 \begin{proof}
In order to deduce this from the preceding cases we form $  \bL$ in  $\Fun(BG,\nCcat)$ as follows:
\begin{enumerate}
\item objects: The set of objects of $  \bL$ is given by $\Ob(  \bH)\sqcup \Ob(  \bH')$.
\item morphisms: \[\Hom_{\bL}(L,L'):=\begin{cases}
\Hom_{\bH}(L,L') & \text{for }L,L' \in \bH\,,\\
 \Hom_{\bH'}(\phi(L),L') & \text{for } L\in \bH, L' \in \bH'\,,\\
  \Hom_{\bH'}(L,\phi(L')) & \text{for } L\in \bH', L' \in \bH\,,\\
    \Hom_{\bH'}(L,L') & \text{for } L,L'\in \bH'\,.
\end{cases} \]
\item {composition and involution}: these structures are defined in the canonical way.
\item the $G$-action is canonically induced from the $G$-actions on $  \bH$ and $\bH'$.
\end{enumerate}
 We have    inclusions 
 $$i\colon \bH\to \bL\, , \quad j\colon \bH'\to   \bL$$
 in $\Fun(BG,\nCcat_{i})$
 and a   projection $p\colon \bL\to \bH$ in $\Fun(BG,\nCcat)$  such that $p\circ j=\id_{\bH}$
 and $p\circ i=\phi$.
 Moreover, there is an obvious  invariant unitary multiplier isomorphism $\kappa: \id_{\bL}\to j\circ p $ given by
 $$\kappa_{L}:=\left\{\begin{array}{cc} \id_{\phi(L)}&L\in \bH,\\ \id_{L}&L\in \bH'.\end{array}\right.$$ 
We conclude that $j$ and $p$ are unitary equivalences.
We have a factorization
$$\phi\rtimes^{\alg} G=(p\rtimes^{\alg} G)\circ (i\rtimes^{\alg}G)\ .$$
Since $p$ is a unitary  equivalence, $p\rtimes^{\alg} G$ is an isometric equivalence.
Since $i$ is injective on objects the morphism $i\rtimes^{\alg}G$ is bounded
by Lemma \ref{erogjowergerwfewrf}.
We conclude that  $\phi\rtimes^{\alg} G$ is bounded. 

If $\phi$ is fully faithful, then so is $i$. 
 Then  $i\rtimes^{\alg}G$ is  fully faithful and isometric by Lemma \ref{erogjowergerwfewrf} and we conclude that $\phi\rtimes^{\alg} G$ is  fully faithful and isometric, too.
  \end{proof}
    This  completes     Step  \ref{wkegjwoergergerwferf3}.

We have finished the construction of the     reduced crossed product functor
$$-\rtimes_{r}G:\Fun(BG,\nCcat)\to \nCcat\ .$$

 Let $\bK,\bL$ be in $\Fun(BG,\nCcat)$. Recall from Definition \ref{rgjeqrgoieqrgrgqewgfq} that  a weakly equivariant morphism  $(\phi,\rho):\bK\to \bL$  is a pair $(\phi,\rho) $  of a morphism $\phi:\bK\to \bL$ and a cocycle $\rho$ of  natural  unitary multiplier transformations. In order to simplify the notation, as long as  we do not encounter   explicit formulas involving $\rho$,  we will just use the symbol
$\phi$ in order to denote weakly  equivariant morphisms. 
\begin{kor}\label{skjtrogergfdsgfgfgd}\mbox{}\begin{enumerate} \item
If $\phi:\bK\to \bL$ is a weakly equivariant morphism, then the induced morphism $\phi\rtimes^{\alg}G:\bK\rtimes^{\alg}G\to \bL\rtimes^{\alg}G$ extends by continuity to a morphism $\phi\rtimes_{r}G:\bK\rtimes_{r}G\to \bL\rtimes_{r}G$.
  If $\phi$ is fully faithful, then $\phi\rtimes_{r}G$ is fully faithful, too.
\item A uniformly   bounded    (unitary)   natural multiplier  transformation $\kappa:\phi\to \phi'$ between weakly equivariant morphisms extends to a  uniformly   bounded  (unitary) natural multiplier 
transformation $ \phi\rtimes_{r}G\to \phi'\rtimes_{r}G$.
\end{enumerate}
\end{kor}
\begin{proof}
Let $\phi:\bK\to \bL$ be a weakly equivariant morphism.
Applying Lemma   \ref{eoigjosegbsfdgsertsfd}  we get a diagram
$$\xymatrix{&Q(\bK)\ar[dl]_{p_{\bK}}\ar[dr]^{\psi}&\\\bK\ar[rr]^{\phi}&&\bL}$$
where $\psi$ is equivariant and which commutes up to a unitary natural multiplier   isomorphism between weakly equivariant morphisms. In view of Construction \ref{rtjoiherthertgeg}  we can apply  $-\rtimes^{\alg}G$ and get a triangle 
$$\xymatrix{&Q(\bK)\rtimes^{\alg}G\ar[dl]_{p_{\bK}\rtimes^{\alg}G}\ar[dr]^{\psi\rtimes^{\alg}G}&\\\bK\rtimes^{\alg}G\ar[rr]^{\phi\rtimes^{\alg}G}&&\bL\rtimes^{\alg}G}$$ which commutes up to a unitary natural multiplier   isomorphism. Since $p_{\bK}$ is a unitary equivalence $p_{\bK}\rtimes^{\alg}G$ is an isometry with respect to the reduced norms, and since $\psi$ is equivariant it follows from Lemma \ref{wkoetgerfregwr}  that $\psi\rtimes^{\alg}G$ is bounded.  

If $\phi$ is fully faithful, then so is $\psi$.  
By      Lemma \ref{wkoetgerfregwr}  we know that 
$\psi\rtimes^{\alg}G$ is fully faithful and  isometric which implies   
that
$\phi\rtimes^{\alg}G$ has these properties, too. 
Hence $\phi\rtimes_{r}G$ is fully faithful.

Assume now that  $\kappa:\phi\to \phi'$  is a uniformly   bounded    (unitary)  natural multiplier  transformation  between weakly equivariant morphisms. 
Then by Construction \ref{rtjoiherthertgeg}  we get the natural  multiplier morphism  $\kappa\rtimes^{\alg}G= ((\kappa_{C},e))_{C\in \Ob(\bK)}$
from $\phi\times^{\alg}G$ to $\phi'\times^{\alg}G$. One checks  using the formulas from  Remark \ref{rem_explicit_description_sigma} that  
$$\| (\kappa_{C},e)\|_{\Hom_{\bM(\bL\rtimes_{r}G)}(\phi(C),\phi'(C))}= \|\kappa_{C}\|_{\Hom_{\bL }(\phi(C),\phi'(C))}\ .$$
It follows that $ ((\kappa_{C},e))_{C\in \Ob(\bK)}$ is uniformly bounded and continuously extends  to a natural multiplier transformation 
$\kappa\rtimes_{r}G$ from $\phi\rtimes_{r}G$ to $\phi'\rtimes_{r}G$.
If $\kappa$ is unitary, then so is $\kappa\rtimes_{r}G$.
 \end{proof}

\begin{rem}\label{qreigjoqergergwegregwreg}
For the construction of the reduced crossed product it was useful to have the freedom to choose the embeddings into small additive categories freely.  But from  Corollary \ref{skjtrogergfdsgfgfgd} we obtain the following useful characterization of the reduced norm on the reduced crossed product.   Consider 
  $\bK$ in $\Fun(BG,\nCcat)$ and  let $M\colon \bK\to \Hilb_{c}(A(\bK))$ be the  Yoneda type embedding from Definition \ref{eiogjoewrgerfgerrfewfwer} which has a weakly invariant extension  by Lemma \ref{qroihofewfqffqwefqewfqefqef}. Since it is fully faithful we get a fully faithful
functor
$\phi\rtimes_{r}G\colon \bK\rtimes_{r}G\to  \Hilb_{c}(A(\bK))\rtimes_{r}G$.
Hence the reduced norm on $\bK\rtimes^{\alg}G$ is induced from
the embdding
$$\phi\rtimes^{\alg}G\colon \bK\rtimes^{\alg}G\to   \Hilb_{c}(A(\bK))\rtimes^{\alg} G \to \bL^{2}(G,\bW \Hilb(A(\bK) ))\ .$$
Using \eqref{evidfnviojafoivsdvasdvasdvasdv}, \eqref{feqwfqwefqedqwe} and the formulas obtained in the proof of Lemma \ref{qroihofewfqffqwefqewfqefqef}
we calculate that 
$\phi\rtimes^{\alg}G$ sends $(f,g):C\to C'$ in $ \bK\rtimes^{\alg}G$ to
\begin{equation}\label{qwefqewfqwefewfqew}
\sum_{\ell\in G} e^{M_{C'}}_{\ell} ({}^{ g}(-)\circ f[g^{-1}C',C])  e^{M_{C},*}_{\ell g}:\bigoplus_{g\in G}gM_{C}\to  \bigoplus_{g\in G}gM_{C'} \ .
\end{equation} 
In order to interpret this formula 
 we use that the underlying vector spaces of $\ell g  M_{C}$ and
$\ell M_{C'}$ are $M_{C}$ and $M_{C'}$, respectively; see Example \ref{qregiuheqrigegwegergwegergw} for the explicit description of the $G$-action on $\Hilb(A(\bK))$.  
We consider the multiplication by the one-entry matrix 
$ f[g^{-1}C',C]$ as a linear  map from $ M_{C}$ to $M_{g^{-1}C'}$ and ${}^{g}(-)$ as a linear map from $M_{g^{-1}C'}$ to $M_{C'}$.
The composition $ {}^{g}(-)\circ f[g^{-1}C',C]$ turns out to   be a morphism in $\Hilb(A(\bK))$ from $\ell g M_{C}$ to $\ell M_{C'}$.
  \hB
 \end{rem}

\begin{rem}
 Restricting to the unital case one can reformulate Corollary \ref{skjtrogergfdsgfgfgd}  in analogy with \cite[Prop. 7.12]{crosscat}  as in the corollary below.
There  $\widetilde{\Fun}(BG,\Ccat)$ is the $(2,1)$-category
of unital $C^{*}$-categories, weakly  equivariant morphisms and  unitary natural isomorphisms between weakly  equivariant morphisms, and
$\Ccat_{2,1}$ is the  $(2,1)$-category of unital $C^{*}$-categories, morphisms, and unitary isomorphisms between morphisms. \hB
\end{rem}
 \begin{kor} The reduced crossed product functor extends to a $2$-functor
$$ -\rtimes_{r}G:\widetilde{\Fun}(BG,\Ccat)\to \Ccat_{2,1}\ .$$
\end{kor}

In order to finish the proof of Theorem \ref{ejgwoierferfewrferfwe} we must show that the  restriction of the reduced crossed product functor % $-\rtimes_{r}G$ 
to $C^{*}$-algebras  has the desired values and that it preserves faithful morphisms.   
We start with 
 recalling the explicit description of the reduced crossed product  of $C^{*}$-algebras with $G$-action.

\begin{construction}\label{wtijghioegergewgrwe}{\em 
  Let $ A$ be in $\Fun(BG,\nCalg)$. We write the {result of the} action of $g$ on $a$ in $A$ as ${}^{g}a$. %Then we consider
The strict $G$-action on the  large $C^{*}$-category   $\Hilb(A)$ was described in  Example \ref{qregiuheqrigegwegergwegergw}.
We consider $ A$ as an object of  $\Hilb(  A)$ in the natural way and define
\begin{equation}\label{qewfljqlwefqwfqewfwf}
L^{2}(G, A) \coloneqq \bigoplus_{g\in G}  {g}A
\end{equation}
  in $\Hilb(A)$.
 An element  {$b$ in   in   the summand $gA$ (which is an element of the set $A$)   will be denoted by $[g,b]$.}  We define a covariant representation $(\rho,{\kappa})$ of $(A,G)$ on $L^{2}(G,A)$ as follows:
\begin{enumerate}
\item For $a$ in $A$ we define $\rho(a)$ in $\End_{\Hilb(A)}(L^{2}(G,A))$ such that $\rho(a){([g,b])}:=[g,ab]$.
\item For $h$ in $G$ we define the unitary  ${\kappa} 
(h)$ in $ \End_{\Hilb(A)}(L^{2}(G,A))$ by 
\[{\kappa}(h)([g,b]) \coloneqq [gh^{-1},{}^{h}b]\,.\]
\end{enumerate}
One checks the relation ${\kappa}(h)\rho(a){\kappa}(h^{-1})=\rho({}^{h}a)$.
 The formula \begin{equation}\label{fopgosifgsfgsfgsfgsgf}
\kappa(g)\rho(a)=\sum_{\ell\in G}e_{\ell} [(-)^{g}\circ a\cdot(-) ] e^{*}_{\ell g}
\end{equation}
will be useful later.
%We indeed have 
% $$\langle [g,b],[g',b'] \cdot_{\tw} a  \rangle_{\tw}= \delta_{g,g'} {}^{g^{-1}}b^{*}{}^{g^{-1}}b  a=
 %\langle [g,b],[g',b']   \rangle \cdot a\ .$$
% We let $L^{2}(G,A)$ be the completion of $L^{2,\alg}(G,A)$ in $\Hilb(A)$.
% \begin{enumerate} \item We consider the left action
% of $A$ on $L^{2}(G,A)$ given by
% $a [g,b]=[g,  a \cdot b]$.
% \item We consider 
 % the unitary representation of $G$ on $L^{2}(G,A)$ given by 
 %$h[g,b]=[h g,{}^{h}b]$.
% \end{enumerate}
}\hB \end{construction}
   \begin{ddd}\label{woigwgewgerg9}
 The reduced crossed product $ A\rtimes_{r}G$    is defined as the $C^{*}$-subalgebra of $\End_{\Hilb(A)}(L^{2}(G, A))$  generated by the operators ${\kappa}(h)\rho(a)$  for all  {$a$}  in $A$ and $h$ in $G$.
 \end{ddd}

Let $A$ be in $\Fun(BG,\nCalg)$. Temporarily  we write
$A\rtimes^{\Calg}_{r}G$ and $A\rtimes^{\Ccat}_{r}G$ for the reduced crossed products of $A$ with $G$ considered as a $C^{*}$-algebra or as a $C^{*}$-category with a single object. The following lemma
shows that the reduced crossed product for $C^{*}$-categories with $G$-action restricts to the classical  reduced crossed product for 
$C^{*}$-algebras with $G$-action.  
\begin{lem}\label{wiothgwrthgfregwreg}
The norms on $A\rtimes^{\alg}G$ induced from $A\rtimes^{\Calg}_{r}G$ and
$A\rtimes^{\Ccat}_{r}G$ are equal.
\end{lem}
\begin{proof}
By Remark \ref{qreigjoqergergwegregwreg}  the $C^{*}$-algebra $A\rtimes^{\Ccat}_{r}G$ is the closure of the image of 
 $A\rtimes^{\alg}G\to \bL^{2}(G,\bW\Hilb(A))$. Similarly, by Definition \ref{woigwgewgerg9}  the $C^{*}$-algebra $A\rtimes^{\Calg}_{G}$ is the closure of the image of $A\rtimes^{\alg}G\to 
 \End_{\Hilb(A)}(L^{2}(G, A))$.
  It therefore suffices to construct an isometric inclusion
 $ i:\End_{\Hilb(A)}(L^{2}(G, A))\to  \bL^{2}(G,\bW\Hilb(A))$
 such that \begin{equation}\label{qewfqwefqwqdewdqwd}
\xymatrix{&A\rtimes^{\alg}G\ar[dl]_{(1)}\ar[dr]^{(2)}\\  \End_{\Hilb(A)}(L^{2}(G, A))\ar[rr]&& \bL^{2}(G,\bW\Hilb(A))}
\end{equation}
 commutes. Using  \eqref{eq_morphisms_LtwoG}  
 and \eqref{qewfljqlwefqwfqewfwf} we see that  we can define $i$  as  the inclusion of  $\End_{\Hilb(A)}(L^{2}(G, A))$ into 
 $\End_{ \bL^{2}(G,\bW\Hilb(A))}(A)$, which explicitly is the inclusion
 $$\End_{\Hilb(A)}(L^{2}(G, A))\to \End_{\bW\Hilb(A)}(L^{2}(G, A))$$
 of the bounded operators on the Hilbert $A$-module $ L^{2}(G, A)$ into its von Neumann envelope.
 Comparing  the explicit formula \eqref{fopgosifgsfgsfgsfgsgf}  for $(1)$ with the formula \eqref{qwefqewfqwefewfqew} for $(2)$
   we check that the triangle in \eqref{qewfqwefqwqdewdqwd} commutes.
\end{proof}

 Recall the functor $A$ from \eqref{frewfoirjviojvioeweverwvwev}. % which is explained in detail in  {Remark} \ref{rem_defn_A}.
 Let $  \bK$ be in $\Fun(BG,\nCcat)$ and consider $A  (\bK)$ in $\Fun(BG,\nCalg)$.
We have an isomorphism
\begin{equation}\label{ervqqowdqwdwqd}
A^{\alg}(  \bK\rtimes^{\alg} G)\cong A^{\alg}(  \bK)\rtimes^{\alg}G\, .
\end{equation}
 In \cite[Thm.\ 6.9]{crosscat} we have shown that this isomorphism extends to an isomorphism 
 $$A(  \bK\rtimes G)\cong A(  \bK)\rtimes G$$ involving maximal crossed products.
 The following  result  is the analogue of this isomorphim for the reduced crossed products.  
  
% 
% \color{brown}
% 
% We consider
% $$A^{\alg}(  \bK)\rtimes^{\alg}G \cong A^{\alg}(  \bK\rtimes^{\alg} G)\to 
% A^{\alg}(  \bK\rtimes_{r}   G)\to A (  \bK\rtimes_{r}  G)$$
% Get extension $$A(  \bK)\rtimes^{\alg}G\to A (  \bK\rtimes_{r}  G)\ .$$
% Must show that extends continuously.
% $$ A^{\alg}(  \bK\rtimes^{\alg} G)   \cong A^{\alg}(  \bK)\rtimes^{\alg}G \to 
% A(  \bK)\rtimes^{\alg}   G\to A (  \bK)\rtimes_{r}  G$$
%Show that $\bK\rtimes^{\alg} G\to A (  \bK)\rtimes_{r}  G$ extends to 
% $\bK\rtimes_{r}  G\to A (  \bK)\rtimes_{r}  G$.  This is clear from functoriality of crossed product applied to $\bK\to A(\bK)$ (use here assertion of he Theorem, that algebra and category crossed products coincide). Then get extension
% $A (  \bK\rtimes_{r}G)\to A (  \bK)\rtimes_{r}  G$ by universal property of $A$.
% 
% \color{black}
 
\begin{theorem}\label{weoigjowergerwegergwerg}
The isomorphism \eqref{ervqqowdqwdwqd} extends to an isomorphism $$A(  \bK\rtimes_{r} G)\cong A(  \bK)\rtimes_{r} G\, ,$$
where the crossed product on the right-hand side is the classical one for $G$-$C^{*}$-algebras.
\end{theorem}

 \begin{proof}%[Proof of Theorem \ref{weoigjowergerwegergwerg}]
We apply the functor $-\rtimes_{r}G$ to the morphism
  $\bK\to A(\bK)$ in $\Fun(BG,\nCcat)$  in order to get a morphism 
  \begin{equation}\label{qwefqewfqewdwedweqdqewdqewd}
 \bK\rtimes_{r}G\to A(\bK)\rtimes_{r}G
\end{equation}   in $\nCcat$
  which extends \begin{equation}\label{qerqewrewrqerewqr}
\bK\rtimes^{\alg}G\to A^{\alg}(\bK)\rtimes^{\alg}G\ .
\end{equation} 
Here we must interpret   $A(\bK)\rtimes_{r}G$ as  the reduced crossed product of the single-object $C^{*}$-category $A(\bK)$ with $G$. But by 
  Lemma \ref{wiothgwrthgfregwreg}  it coincides with 
  the reduced crossed product in the sense of $C^{*}$-algebras described in Definition \ref{woigwgewgerg9}. 
   By the universal property of $A^{\alg}$ the  morphism \eqref{qerqewrewrqerewqr} induces  the underlying map of the  isomorphism   \eqref{ervqqowdqwdwqd}.    Correspondingly, by the universal property of $A$ the morphism \eqref{qwefqewfqewdwedweqdqewdqewd}  induces  a morphism
  $A(\bK\rtimes_{r}G)\to A(\bK)\rtimes_{r}G$ continuously extending \eqref{ervqqowdqwdwqd}. 
  
 For the other direction
 we consider the following   composition where the last two morphisms are    induced by the canonical embedding in completions \begin{equation}\label{adcq4fasdfa} A^{\alg}(\bK)\rtimes^{\alg}G \stackrel{\eqref{ervqqowdqwdwqd}}{\cong} 
A^{\alg}(  \bK\rtimes^{\alg} G)\to 
 A^{\alg}(  \bK\rtimes_{r}   G)\to A (  \bK\rtimes_{r}  G)\ .
\end{equation} 
It corresponds to a covariant representation
$(\rho^{\alg},\pi)$, where  $\rho^{\alg}:A^{\alg}(\bK)\to  A (  \bK\rtimes_{r}  G)$ is a morphism in $\nClincat$,  and $\pi:G\to M(A (  \bK\rtimes_{r}  G))$ is a unitary representation. 
By  the universal property  of $A(\bK)$ as the  completion of $A^{\alg}(\bK)$ the morphism $\rho^{\alg} $ continuously extends to a morphism $
\rho:A(\bK)\to   A (  \bK\rtimes_{r}  G)$ so that we get a covariant representation 
$(\rho,\pi)$ of $A(\bK)$ on $
A (  \bK\rtimes_{r}  G)$.  It induces a morphism
\begin{equation}\label{adfasdfdsfaewfae}
A(\bK)\rtimes^{\alg}G\to A (  \bK\rtimes_{r}  G)\ .
\end{equation}
{We must show that   it continuously extends further to the reduced crossed product.} We consider the full subcategory $\bD$ of $\bL^{2}(G,\bW\Hilb(A(\bK)))$ on the objects $M_{C}$ for $C$ in $\Ob(\bK)$.  
We will 
 construct 
  an  isometric embedding 
\begin{equation}\label{sfjksdjfdlsdfgsdfgsfg}
A(\bD)  \to \End_{\Hilb(A(\bK))}(L^{2}(G,A(\bK))) 
\end{equation} in $\nCalg$ % including
%$\End_{\Hilb(A(\bK))}(L^{2}(G,A(\bK)))$ into
%\End_{\bW\Hilb(A(\bK))}(L^{2}(G,A(\bK)))\cong  \End_{\bL^{2}(G,\bW\Hilb(A(\bK)))}(A(\bK))$
%considered as a corner of $A(\bL^{2}(G,\bW\Hilb(A(\bK))))$ corresponding to the object $A(\bK)$. We claim that the following %
such that the square  \begin{equation}\label{werwferfrewfwerfw}
\xymatrix{ A (  \bK)\rtimes^{\alg}  G\ar[r]^{\eqref{adfasdfdsfaewfae}}\ar[d]^{(1)}&A(\bK\rtimes_{r}G)\ar[d]^{(2)}\\\End_{\bW\Hilb(A(\bK))}(L^{2}(G,A(\bK))) & \ar[l]A(\bD)}
\end{equation}
 commutes.
Since we have an isometric inclusion 
$\bK\rtimes_{r}G\to \bD $ (see Remark \ref{qreigjoqergergwegregwreg}) and $A$ preserves isometric inclusions by \cite[Lem. 6.8.1]{crosscat}
the right vertical arrow $(2)$ is an isometric inclusion. Since the reduced norm 
on $ A (  \bK)\rtimes^{\alg}  G$ is by Definition \ref{woigwgewgerg9} induced by the left vertical arrow  $(1)$ the commutativity of  the square immediately then implies that 
\eqref{adfasdfdsfaewfae} continuously extends to the reduced crossed product in the domain.
In order to construct \eqref{sfjksdjfdlsdfgsdfgsfg}
we  consider the isomorphism 
$$L^{2}(G,A(  \bK))\stackrel{\eqref{qewfljqlwefqwfqewfwf}, \ \text{Lem.}\ref{rtiohjirotgwegfwergewgegwerf9}}{\cong} \bigoplus_{g\in G} \bigoplus_{C\in \Ob(\bK)}  g M_{C}\cong   \bigoplus_{C\in \Ob(\bK)}  \bigoplus_{g\in G} g M_{C} 
$$ in $\Hilb(A(\bK))$.
Considering the elements of $A^{\alg}(\bD)$ as matrices
indexed by $\Ob(\bK)$ with entries
in $\Hom_{\bW\Hilb(A(\bK))}(\oplus_{g\in G}  gM_{C}, \oplus_{g\in G} g M_{C'} )$ for pairs $C,C'$ in $\Ob(\bK)$
  we get an injective  homomorphism
 $A^{\alg}(\bD)\to \End_{\bW\Hilb(A(\bK))}(L^{2}(G,A(\bK)))$.  By \cite[Lem. 6.8.2]{crosscat} 
  this inclusion
 extends to an isometric inclusion
 $A (\bD)\to \End_{\bW\Hilb(A(\bK))}(L^{2}(G,A(\bK)))$.
   Comparing  the explicit formula \eqref{fopgosifgsfgsfgsfgsgf}  for $(1)$ with the formula \eqref{qwefqewfqwefewfqew} for $(2)$
   we check that the square in \eqref{werwferfrewfwerfw} commutes.
\end{proof}

The following proposition finishes the proof of Theorem \ref{ejgwoierferfewrferfwe}.

\begin{prop}\label{okgpweerfrfrfwr}
The reduced crossed product functor $-\rtimes_{r}G:\Fun(BG,\nCcat)\to \nCcat$ preserves faithful morphisms.
\end{prop}
\begin{proof}
We first observe that the assertion of the proposition  is true for the restriction of the reduced crossed product functor to $C^{*}$-algebras with $G$-action.
Assume that $A\to B$  is an isometric inclusion in $\Fun(BG,\nCalg)$. Then we get an isometric inclusion $L^{2}(G,A)\to L^{2}(G,B)$ of Banach spaces. This implies that the induced homomorphism $A\rtimes_{r}G\to 
B\rtimes_{r}G$ is isometric.
 
Let $\bK\to \bL$ be a faithful (or equivalently, an isometric) morphism in $\Fun(BG,\nCcat)$. We first assume that it is injective on objects.
Then we consider the commutative  diagram
$$\xymatrix{\bK\rtimes_{r}G\ar[d]\ar[r]&\bL\rtimes_{r}G\ar[d]\\A(\bK\rtimes_{r}G)\ar[r]\ar[d]^{\cong}&A(\bL\rtimes_{r}G)\ar[d]^{\cong }\\A(\bK)\rtimes_{r}G\ar[r]&A(\bL)\rtimes_{r}G}$$
The upper vertical morphisms are isometric by  \cite[Lem. 6.7]{crosscat}.
The lower vertical morphisms are isomorphisms by Theorem \ref{weoigjowergerwegergwerg}.
Since $A$ preserves isometric inclusions by \cite[Lem. 6.8.1]{crosscat} 
the morphism $A(\bK)\to A(\bL)$ is an isometry.  As explained above, this implies that
 the lower horizontal morphism is   isometric. This implies that the upper horizontal morphism is  isometric.

We finally remove the assumption that 
 $\bK\to \bL$ is injective on objects. In this case, as in the proof of Lemma \ref{wkoetgerfregwr}, we  can find a factorization of this morphism  as
$\bK\to \bL'\to \bL$, where the first map is faithful and injective on objects, and
$\bL'\to \bL$ is a unitary equivalence, hence fully faithful. 
 We  obtain a  factorization of  the  morphismin question as
 $\bK\times_{r}G\to \bL'\rtimes_{r}G\to \bL\rtimes_{r}G$.
 The first morphism is isometric by the special case above. Since
 $-\rtimes_{r}G$ preserves fully faithfulness the second morphism is fully faithful.  
Hence the composition is faithful.
 \end{proof}

Recall that a group $G$ is called exact if the functor
$-\rtimes_{r}G:\Fun(BG,\nCalg)\to \nCalg$ preserves exact sequences.
\begin{prop}\label{rgkohpgbdghdfghf}
If $G$ is exact, then $-\rtimes_{r}G:\Fun(BG,\nCcat)\to \nCcat$
preserves exact seqences.
\end{prop}
\begin{proof}
We use that the functor $A:\nCcat_{i}\to \nCalg$ preserves and detects exact sequences. If
$0\to \bC\to \bD\to \bQ\to 0$ is an exact sequence in $\Fun(BG,\nCcat)$,
then $0\to A( \bC)\to A( \bD)\to A(\bQ)\to 0$ is an exact sequence in $\Fun(BG,\nCalg)$. Since $G$ is exact we get the exact sequence
$  0\to A( \bC)\rtimes_{r}G\to A( \bD)\rtimes_{r}G\to A(\bQ)\rtimes_{r}G\to 0$.
By Theorem \ref{weoigjowergerwegergwerg} we see that 
$  0\to A( \bC \rtimes_{r}G)\to A( \bD \rtimes_{r}G)\to A(\bQ \rtimes_{r}G)\to 0$ is exact. We finally conclude that
$   0\to   \bC \rtimes_{r}G \to   \bD \rtimes_{r}G \to  \bQ \rtimes_{r}G \to 0$ is an exact sequence in $\nCcat$.
\end{proof}

In the following we compare the reduced and the maximal  versions of the crossed product.
Let $ \bK$ be in $\Fun(BG,\nCcat)$. By Definition~\ref{ewtiojgwergerggwggr} the norm on the reduced crossed product ${\bK} \rtimes_r G$ is induced by the representation $\rho$ from \eqref{ewrgfviubiuhuihfefwerfwerfw} which comes from the covariant representation 
$(\sigma,\pi)$ from Lemma \ref{qeriuhiqerfqwefqwfqwefqewf}.
Hence by the universal property of the maximal crossed product \cite[Cor.~5.10]{crosscat} we get a comparison functor 
\begin{equation}\label{eq_q}
q_{ {\bK}}\colon  {\bK} \rtimes G \to {\bK} \rtimes_r G
\end{equation}
in $\nCcat$. 
\begin{lem}
The functor
$q_{ {\bK}}$ is the identity on objects and surjective on morphism spaces.
\end{lem}
\begin{proof}
By construction, $q_{ {\bK}}$ is the identity on objects.

By definition of the reduced crossed product, the image of the functor  $q_{ {\bK}}$  contains the dense $*$-subcategory $ {\bK} \rtimes^\alg G$ of $ {\bK} \rtimes_r G$. Since functors between $C^*$-categories have closed ranges on the morphism spaces the claim follows.
\end{proof}

It is known that for an amenable group  $G$  the  canonical map $$q_{ A} \colon  A\rtimes G\to   A\rtimes_{r}G$$  is an isomorphism for all {$C^{*}$-algebras $A$ with $G$-action.}  %A$ in $\Fun(BG,\nCalg)$. 
In the following we generalize this fact  to $C^*$-categories.

Let $  \bK$ be in $\Fun(BG,\nCcat)$.
\begin{theorem}\label{thm_G_amenable} If $G$ is amenable, then the canonical  morphism $q_{{\bK}}\colon  {\bK} \rtimes G \to  {\bK} \rtimes_r G $ is an isomorphism.
\end{theorem}
\begin{proof} 
%\uli{Since $ \tilde{\bK} \rtimes^{\alg} G$ is dense in both, the domain and the image of $q_{\bK}$, it suffices that the maximal and reduced norms $\|-\|_{\max}$ and $\|-\|_{r}$ on this dense subcategory coincide.}
%%%
%Let $b$ be in $\tilde{\bK} \rtimes^{\alg} G$. \uli{Since by definition of the maximal norm  we always have $\|q_{\bK}(b)\|_{r}\le \|b\|_{\max}$ it suffices to show the converse inequality}  $\|b\|_{\max} \le \|q_{\uli{\bK}}(b)\|_{\uli{r}}$.    
%%%
Since for any $C^{*}$-category $\bD$ the canonical map $\rho_{\bD}:\bD\to A(\bD)$ is an isometry \cite[Lem.\ 6.7]{crosscat} it suffices to show that 
$A(q_{  \bK}) \colon A(  \bK\rtimes G)\to A(  \bK \rtimes_{r}G)$ is an isomorphism.
 %Since the functor $A$ is isometric\fuli{so kann man  das nicht sagen, use that $\rho_{\bK}:\bK\to A(\bK)$ is isometricnby \cite[Lemma 6.7]{crosscat}.}, this means that it suffices to show 
%\begin{equation*}
%\|\rho_{\bK\rtimes G}(b)\|_{A(\bK\rtimes G)} \le \|\rho_{\bK\rtimes_{r}G}(q_{\tilde{\bK}}(b))\|_{A(\bK\rtimes_{r} G)}\,.
%\end{equation*}
Recall that the  isomorphism %$\nu\colon A^{\alg}(\tilde \bK)\rtimes^{\alg}G \to A^{\alg}(\tilde \bK\rtimes^{\alg} G)$
from \eqref{ervqqowdqwdwqd} extends to  isomorphisms
\begin{equation}\label{eq_nu_max}
%\nu_{\max} \colon
 A( \bK)\rtimes G \xrightarrow{\cong} A( {\bK}\rtimes G)
\end{equation}
by \cite[Thm.\ 6.9]{crosscat}  and  
\begin{equation}\label{eq_nu_r}
 A( \bK)\rtimes_{r} G \xrightarrow{\cong} A( \bK\rtimes_{r} G)
\end{equation}
by Theorem~\ref{weoigjowergerwegergwerg}.  
We have a commutative diagram
\begin{equation}\label{eq_nu_isos}
\xymatrix{
A(  \bK)\rtimes G \ar[r]^-{\eqref{eq_nu_max}}_-{\cong} \ar[d]_-{q_{A( {\bK})}} & A( {\bK}\rtimes G) \ar[d]^-{A(q_{ {\bK}})}\\
A(\bK)\rtimes_r G \ar[r]^-{\eqref{eq_nu_r}}_-{\cong} & A({\bK}\rtimes_r G)
}
\end{equation}
The left vertical arrow $q_{A( {\bK})}$ is an isomorphism, because $G$  is   amenable and $A(\bK)$ is a $G$-$C^{*}$-algebra.
This implies that $A(q_{  \bK})$ is an isomorphism, too.
%Since $A(b)$ is an element of $A(\tilde{\bC}\rtimes G)$ and $A(q_{\tilde{\bC}}(b))$ an element of $A(\tilde{\bC}\rtimes_r G)$, the isomorphisms in the previous diagram imply that $\|A(b)\|_{\max} = \|A(q_{\tilde{\bC}}(b))\|$.
\end{proof}

We finally consider a subgroup $H$ of $G$ and $\bK$ in $\Fun(BG,\nCcat)$.
Then we have a canonical inclusion $i^{\alg}\colon\Res^{G}_{H}(\bK)\rtimes^{\alg} H\to \bK\rtimes^{\alg}G$.

\begin{prop}\label{efjviofbsdfbsdbsdfvvsdfvs}
$i^{\alg} $ continuously extends to an isometric inclusion
\[i\colon \Res^{G}_{H}(\bK)\rtimes_{r} H\to \bK\rtimes_{r} G\,.\]
\end{prop}
\begin{proof}
We omit the functor $\Res^{G}_{H}$ from the notation. 
We first assume that $\bC=\bW\bM\bK$ admits very small  orthogonal sums. In this case  we  define a wide isometric inclusion $j_{\bC}:\bL^{2}(H,\bC)\to \bL^{2}(G,\bC)$.  On objects it acts as the identity. In order to define $j_{\bC}$ on morphisms, 
   for every object $C$ of $\bC$
we let $(\oplus^{G}_{g\in G} gC,(e^{G,C}_{g})_{g\in G})$
and
$(\oplus^{H}_{g\in H} gC,(e^{H,C}_{g})_{g\in H})$
denote the choices of sums in the definitions of $\bL^{2}(G,\bC)$ and
$\bL^{2}(H,\bC)$. We have an isometry
$$u_{C}:=\oplus_{g\in H}e^{G,C}_{g}e^{H,C,*}_{g}:\oplus^{H}_{g\in H} gC\to \oplus^{G}_{g\in G} gC\ .$$
The morphism $j_{\bC}$ sends a morphism $f:C\to C'$ in $ \bL^{2}(H,\bC)$ to
$u_{C'}fu_{C}^{*}$.
 We now observe that
 $j_{\bC}$ restricts to the morphism $i^{\alg}$ from    $\bC\rtimes^{\alg}H$   to $\bC\rtimes^{\alg}G$ interpreted via \eqref{ewrgfviubiuhuihfefwerfwerfw} as subcategories of  $ \bL^{2}(H,\bC)$  and   $\bL^{2}(G,\bC)$, respectively.
  Hence
  $j_{\bC}$ restricts to an isometric inclusion
  $i:\bK\rtimes_{r}H\to \bK\rtimes_{r}G$ which is the asserted continuous extension of $i^{\alg}_{\bK}$. This shows the assertion of the proposition for all $\bK$ such that $\bW\bM\bK$ admits all very small orthogonal sums.
 
 Next we assume that there is a fully faithful morphism  $\bK\to \bK'$
 such that $\bW\bM\bK'$ admits all small orthogonal sums.
 %Then by a combination of Proposition \ref{stkghosgfgsrgsegs} and  Proposition \ref{rjigowergwregregwfer} we have an induced fully faithful morphism $\bC\to \bC'$. 
  As observed in Step \ref{wkegjwoergergerwferf} of the construction of  the reduced crossed product, the  horizontal
  arrows in   $$\xymatrix{\bK\rtimes^{\alg}H\ar[d]\ar[r]&\bK'\rtimes^{\alg}H\ar[d]\\
 \bK\rtimes^{\alg}G\ar[r]&\bK'\rtimes^{\alg}G }$$
 are isometric inclusions with respect to the reduced norms.
  By the special case discussed above, also the right vertical map is an isometric inclusion. This implies that the left vertical morphism is isometric, too. 
   This proves the assertion for $\bK$  going into Step \ref{wkegjwoergergerwferf} of the construction of the reduced crossed product.  
   
    We now assume that $\bK$ admits a fully faithful morphism $\bL\to \bK$
   such that $\bL$ admits a fully faithful morphism into $\bL'$ such that $\bW\bM\bL'$ admits all very small sums.
   Then we consider the square
         $$\xymatrix{\bL\rtimes^{\alg}H\ar[d]\ar[r]&\bK\rtimes^{\alg}H\ar[d]\\
 \bL\rtimes^{\alg}G\ar[r]&\bK\rtimes^{\alg}G }$$
By   Step \ref{wkegjwoergergerwferf1} of the construction the reduced crossed product the horizontal morphisms are fully faithful and  isometric for the reduced norms. By the case above also the left vertical morphism is an isometric inclusion. It follows that the right vertical morphism is an isometric inclusion.
In view of Step \ref{wkegjwoergergerwferf2} of the construction the reduced crossed product we have verified
the assertion of the proposition for all objects of $\Fun(BG,\nCcat)$.
\end{proof}

% \begin{ex}\label{qregihqiuwfewfweqfqwefqewf}
%%We continue with Example \ref{qregiuheqrigegwegergwegergw}. {%Recall that $\kappa$ is the notation for the $G$-action on $A$. %and \ref{ex_bsp_Lzwei_Hilb_Hom_A}. 
%Let $A$ be in $\Fun(BG,\nCalg)$ be very small.
%We consider $ A$  as an object of  $ {\Hilb_{c}( A)}$.
%Note that $\Hilb_{c}(A)$ admits $G$-indexed orthogonal sums and  that the left-action of $A$ on itself induces an isomorphism $A \cong \End_{\Hilb_{c}(A)}(A) $.    
%  We now use Example \ref{egiojeogergerferferfwrf} in order to identify the norm induced by $\rho$ on $\End_{\Hilb_{c}(A)\rtimes^{\alg}G}(A)$ with the reduced norm on $A\rtimes^{\alg}G$  induced by the representation on $L^{2}(G,A)$. We then get an 
% isomorphism of $C^{*}$-algebras
%$$\End_{ {\Hilb_{c}( A)}  \rtimes_{r} G }( A  )\cong   {A} \rtimes_{r} G \, . $$
%where  the right-hand side is the classical reduced crossed product of  $A$    with $G$. %If $A$ is unital\fuli{AV: das wird jetzt sogar allgemein richtig}, then 
%%$$\End_{ {\Hilb( A)}  \rtimes_{r} G }( A  )\cong   \uli{A} \rtimes_{r} G \, . $$
%
%
% 
%In this sense Definition \ref{ewtiojgwergerggwggr} generalizes the reduced crossed product from $C^{*}$-algebras with $G$-action to $C^{*}$-categories with $G$-action.
%\hB
%\end{ex}
%%

\section{Homological functors}\label{ewgiowegergregregergrwrf}

% Let   $\nCcat$ denote  the large category  of (possibly non-unital) small $C^{*}$-categories and (not necessarily unit preserving) functors. 
% Its subcategory of unital $C^{*}$-categories and unit preserving functors
% will be denoted by $\Ccat$, see \cite[Def.~3.4 and Diag.~(3.4)]{crosscat}.
% 
% The categories $\nCcat$ and $\Ccat$ are complete and cocomplete, see  \cite[Thm.~4.1]{crosscat},  \cite{DellAmbrogio:2010aa} and \cite[Thm.~8.1]{startcats}.
% 
% For every set $X$ we can functorially form the unital $C^{*}$-category $0[X]$ with  objects \Alex{the elements of} $X$  and whose morphism spaces \Alex{(between distinct objects)} are all zero. The functor $0[-]$ fits into various adjunctions \cite[Lem.~3.8]{crosscat}.
% 
% Note that morphisms in $\Ccat$ are functors.
%The category  $\Ccat$  is the underlying $1$-category of a strict $(2,1)$-category $\Ccat_{2,1}$.
% The two-isomorphisms in   $\Ccat_{2,1}$ are 
% unitary isomorphisms between functors.\footnote{\Alex{``Unitary isomorphism'' definieren.}} A morphism in $\Ccat$ is an equivalence if it is invertible up to unitary isomorphism.
 
 {The basic homotopy theoretic invariant of a $C^{*}$-category is its topological $K$-theory. Axiomatizing some of the fundamental properties of the  $K$-theory of $C^{*}$-categories we introduce the notion of a homological functor.    
  We then use these axioms in order to   derive various 
 properties  of homological functors. In the subsequent section we show that $K$-theory is indeed  an example of a homological functor.}

% In the present section we introduce the no
% 
% 
%  We first introduce the notions of a closed ideal in a $C^{*}$-category and of an  excisive square of $C^{*}$-categories. 
%  Then we define the notion of a homological functor and discuss some of its basic properties.
  
 Let $i\colon \bC\to \bD$ be a {morphism} in $\nCcat$. The following definition %is  \uli{taken} from \cite[Def.~8.2]{crosscat}.   It 
 generalizes the notion of a closed two-sided ideal in a $C^{*}$-algebra.
\begin{ddd}\label{weiogjwergferfwegwgregwregwgre}
 The morphism $i$ is an inclusion of an ideal % (is a kernel) 
 if it has the following properties:
 \begin{enumerate}
 \item\label{weiogjwergferfwegwgregwregwgre_eins} $i$ induces a bijection   between the sets of objects.
 \item $i$ induces closed embeddings of morphism spaces.
 \item\label{weiogjwergferfwegwgregwregwgre_drei}   The composition of a morphism in the image of $i$
 with any morphism of $\bD$ belongs again to the image of $i$.
  \end{enumerate}
\end{ddd}

Let $i\colon \bC\to \bD$  be a  {morphism} in $\nCcat$. %Since 
%the category $\nCcat$ is cocomplete \cite[Thm.~4.1.2.]{crosscat}  
The  quotient
$\bD/\bC$ is defined as    {the} push-out 
\[
\xymatrix{\bC\ar[rr]^-{i}\ar[d]_-{\eqref{sgbsfdvfsvfvdfvsv}}&&\bD\ar[d]^{q}\\0[\Ob(\bC)]\ar[rr]&&\bD/\bC}
\]
in $\nCcat$. We will say that $q$ presents $\bD/\bC$ as the quotient of $\bD$ by $\bC$.
If $i$ is the inclusion of an ideal 
  it is easy to describe the $C^{*}$-category $\bD/\bC$   explicitly. 
\begin{enumerate}
\item objects: The objects of $\bD/\bC$ are the objects of $\bD$ (which are in bijection with the objects of $\bC$ via $i$).
\item morphisms:  For objects $C,C'$ in $\bC$ we have
$$\Hom_{\bD/\bC}(i(C),i(C'))\cong \Hom_{\bD}(i(C),i(C'))/i(\Hom_{\bC}(C,C'))\, .$$
\item composition and involution:  The composition and $*$-operation are inherited from $\bD$.
\end{enumerate}
Since $i(\Hom_{\bC}(C,C'))$ is a closed subspace of $\Hom_{\bD/\bC}(i(C),i(C'))$ the quotient has an induced norm which exhibits $\bD/\bC$ as a $C^{*}$-category \cite[Cor.~4.8]{mitchc}.
If $\bD$ is unital, then so is $\bD/\bC$, and the projection map $\bD\to \bD/\bC$ is a morphism in $\Ccat$. %\uli{The morphism $\bD\to \bD/\bC$ is a quotient morphism in the sense of \cite[Def. 8.1]{crosscat}.}

We consider a sequence of morphisms $$\bC\stackrel{i}{\to} \bD\stackrel{q}{\to} \bQ$$ in $\nCcat$.

\begin{ddd}
The sequence is an exact sequence in $\nCcat$ if $i$ is an inclusion of  an ideal and $q$ presents $\bQ$ as the quotient $\bD/\bC$.
\end{ddd}

\begin{rem}
In the following we use the language of $\infty$-categories\footnote{more precisely, $(\infty,1)$-categories}. References are \cite{htt,Cisinski:2017}.  Ordinary categories will be considered as $\infty$-categories using the nerve functor. 
A typical target $\infty$-category for the homological functors introduced below is the stable $\infty$-category $\Sp$ of spectra.  We refer to \cite{HA} for an introduction to stable $\infty$-categories in general, and for $\Sp$ in particular.  
The $\infty$-categories considered in the present paper belong to the large universe.
A cocomplete $\infty$-category thus admits all colimits for small index categories.
\hB \end{rem}

 Let $\bS$ be an $\infty$-category.
 We consider a functor
 $$\Homol\colon\nCcat\to \bS\, .$$
 
% \begin{ddd}
% $\Homol$ is additive if for every $\bC$ in $\Ccat$ with a chosen zero object the morphism $H(s_{\bC}):\Homol(\bC\sqcup \bC)\to \Homol(\bC\times \bC)$ is an equivalence.
% \end{ddd}
%

\begin{ddd}\label{oihgjeorgwergergergwerg} $\Homol$ is a homological functor if the following conditions are satisfied:{
\begin{enumerate}
\item $\bS$ is stable.
\item $\Homol$ sends unitary equivalences in $\nCcat$ to equivalences.
\item $\Homol$   sends exact sequences to  fibre sequences. 
\end{enumerate}
}
\end{ddd}

In the following we will provide an equivalent characterization of homological functors which is very similar to the notion of a homological functor  for left-exact $\infty$-categories used in \cite{unik}.   The    properties listed in Lemma \ref{qerughqiregqewfewqffq}.\ref{sfgsdfgsfdgsfdgsfg1}  together with the additional property introduced in Definition \ref{qerughqiregqewfewqffq1} are  motivated by the applications in \cite{coarsek}.

We  consider a square 
\begin{equation}\label{asdv2e4fwdqwqev}
\xymatrix{\bA\ar[r]\ar[d]&\bB\ar[d]\\\bC\ar[r]&\bD}
\end{equation}
 in $\nCcat$.
 By the universal property of the quotients of the horizontal functors    we obtain an induced {morphism}  $\bB/\bA\to \bD/\bC$.
  The following is taken from \cite[Defn.~8.10]{crosscat}.
 \begin{ddd}\label{weiogwegerewrgwreg}
  The square \eqref{asdv2e4fwdqwqev} is called excisive if it satisfies the following conditions:
  \begin{enumerate}
  %\item \label{qiorjfioqerqwfqwefqwef} $\bB$ and $\bD$ are unital and the functor $\bB\to \bD$ is unital.
  \item \label{rgijqeroigqrefqewefqwefq}
  The {morphism} $\bA\to \bB$ and $\bC\to \bD$ are
   embeddings of  closed ideals.     
   \item \label{thgiojrtiobgrefrefrewfwevf}The quotients $\bB/\bA$ and $\bD/\bC$ are unital    \item \label{thgoijroihwthwhhteh} The induced {morphism} $ \bB/\bA \to  \bD/\bC $  is unital and a unitary equivalence. 
  \end{enumerate}
  \end{ddd}

%\begin{ex}
%The notion of an excisive square of $C^{*}$-categories generalizes the notion of an exact sequence of $C^{*}$-algebras in the following sense.
%Let $$0\to I\to A\to Q\to 0$$ be a short exact sequence of $C^{*}$-algebras where $I$ is a closed ideal in $A$ and $Q$ is unital. We can consider $C^{*}$-algebras as $C^{*}$-categories with a single object $*$. Then the square
%$$\xymatrix{I\ar[r]\ar[d]&A\ar[d]\\0[*]\ar[r]&Q}$$
%is an excisive square of $C^{*}$-categories. 
%\hB
%\end{ex}
%
%

% 
%\begin{rem}\label{rfwefwfwefwefwefewfweewfqr}
%The condition that $\Homol$ is reduced means that it sends
%zero categories to zero. More precisely we require that 
%for every  small set $X$ we have $\Homol(0[X])\simeq 0_{\bM}$.  \hB \end{rem}
%

%Let $\Homol:\nCcat\to \bM$ be a homological functor.
%\begin{lem}\label{qerughqiregqewfewqffq}
%$\Homol$ sends unitary equivalences 
%  between unital $C^{*}$-categories to equivalences.  \end{lem} \begin{proof}Consider a unitary equivalence $\bA\to \bB$ in $\Ccat$. Then we form the commutative square  $$\xymatrix{0[\Ob(\bA)]\ar[r]\ar[d]&\bA\ar[d]\\0[\Ob(\bB)]\ar[r]&\bB}$$ in $\nCcat$. It is excisive and send by $\Homol$ to a push-out square in $\bM$. Using that $\Homol$ is reduced the latter has the form 
%  $$\xymatrix{0_{\bM}\ar[r]\ar[d]&\Homol(\bA)\ar[d]\\0_{\bM} \ar[r]&\Homol(\bB)}\ .$$  
%Since the left vertical arrow is an equivalence, the right vertical arrow is an equivalence, too.   
%\end{proof}
%\color{black}

Let $\Homol:\nCcat\to \bS$ be a functor.
\begin{lem}\label{qerughqiregqewfewqffq}
The following conditions are equivalent:
\begin{enumerate}
\item \label{sfgsdfgsfdgsfdgsfg} $\Homol$ is a homological functor.
\item \label{sfgsdfgsfdgsfdgsfg1}
\begin{enumerate}
\item\label{qergiuhieofrgregeggwergewrg} The $\infty$-category  $\bS$ is stable.
%\item\label{qerughqiregqewfewqffq} $\Homol$ sends unitary equivalences between unital $C^{*}$-categories to equivalences.\fuli{this is a consequence of the next axiom: consider the square $$\xymatrix{0[\Ob(\bA)]\ar[r]\ar[d]&\bA\ar[d]\\0[Ob(\bB)]\ar[r]&\bB}$$
%where $\bA\to \bB$ is a unital unitary equivalence. This square is excisive. $\Homol$  sends this  to a push-out. Using axiom 4 and 3 we conclude that $\Homol(\bA)\to \Homol(\bB)$ is an equivalence:}
%\item \label{qerughqiregqewfewqffq1}  $\Homol$ preserves filtered colimits.
 \item  \label{qerughqiregqewfewqffq2}  $\Homol$ sends excisive squares to push-out squares. 
 \item \label{etghiwjeorgreergwg} $\Homol$ is reduced, i.e.,   
for every  small set $X$ we have $\Homol(0[X])\simeq 0_{\bS}$.    
 %For every  small set $X$ we have $\Homol(0[X])\simeq 0_{\bM}$.
 %\item \label{qregoijfovgsvfvsfvsfdv} $\Homol$  preserves cartesian products.}
 \end{enumerate}
%
%\begin{enumerate}
%\item $\bM$ is stable.
%\item $\Homol$ sends unitary equivalences in $\nCcat$ to equivalences.
%\item $\Homol$   sends exact sequences sequences to  fibre sequences. 
%\end{enumerate}
\end{enumerate}
\end{lem}
\begin{proof}
We first show that Assertion \ref{sfgsdfgsfdgsfdgsfg1} implies Assertion \ref{sfgsdfgsfdgsfdgsfg}.
 We start with showing that 
$\Homol$  sends unitary equivalences 
  between unital $C^{*}$-categories to equivalences.
  Consider a unitary equivalence $\bA\to \bB$ in $\Ccat$. Then we form the commutative square  $$\xymatrix{0[\Ob(\bA)]\ar[r]\ar[d]&\bA\ar[d]\\0[\Ob(\bB)]\ar[r]&\bB}$$ in $\nCcat$. It is excisive and send by $\Homol$ to a push-out square in $\bS$. Using that $\Homol$ is reduced the latter has the form 
  $$\xymatrix{0_{\bS}\ar[r]\ar[d]&\Homol(\bA)\ar[d]\\0_{\bS} \ar[r]&\Homol(\bB)}$$  
Since the left vertical arrow is an equivalence, the right vertical arrow is an equivalence, too.   

We next show that
$\Homol$ sends exact sequences $\bC\to \bD\stackrel{q}{\to}\bQ$ to fibre sequence provided that $q$ is a morphism in $\Ccat$.   In fact, under this assumption
$$\xymatrix{\bC\ar[r]\ar[d]&\bD\ar[d]^{q}\\0[\Ob(\bC)]\ar[r]&\bQ}$$
is an excisive square. Applying $\Homol$ we get the push-out
$$\xymatrix{\Homol(\bC)\ar[r]\ar[d]&\Homol(\bD)\ar[d]^{q}\\0_{\bS} \ar[r]&\Homol(\bQ)}$$
hence the asserted fibre sequence. %Note that the assertion, that a sequence of two composable maps in a stable $\infty$-category is a fibre sequence
%involves exhibiting the data of a zero homotopy of the composition which here is given by the filler of the square. 

Let now $\bC\to \bD\stackrel{}{\to}\bQ$ be a general exact sequence.
Then we consider the diagram
$$\xymatrix{\bC\ar@{=}[d]\ar[r]&\bD\ar[d]\ar[r]&\bQ\ar[d]\\
\bC\ar[r]&\bD^{+}\ar[d]\ar[r]&\bQ^{+}\ar[d]\\&\C[\Ob(\bC)]\ar@{=}[r]&\C[\Ob(\bC)]}$$
where the right vertical exact sequences  arise from   unitalization.
The horizontal sequences are also exact.
If we apply $\Homol$, then we get the diagram
$$\xymatrix{\Homol(\bC)\ar@{=}[d]\ar[r]&\Homol(\bD)\ar[d]\ar[r]&\Homol(\bQ)\ar[d]\\
\Homol(\bC)\ar[r]&\Homol(\bD^{+})\ar[d]\ar[r]&\Homol(\bQ^{+})\ar[d]\\&\Homol(\C[\Ob(\bC)])\ar@{=}[r]&\Homol(\C[\Ob(\bC)])}$$
By the special case shown above the two  vertical sequences and the middle horizontal one  are fibre sequences. Consequently, the 
upper right square is a pull-back. We can now conclude that
the upper sequence is also a fibre sequence. 

 We finally show that $\Homol$ sends all unitary equivalences to  equivalences.
 Let $\bC\to \bD$ be a unitary equivalence. Since this functor is fully faithful, by Proposition \ref{stkghosgfgsrgsegs} we can 
  consider the square
 $$\xymatrix{\bC\ar[d]\ar[r]&\bM\bC\ar[d]\\\bD\ar[r]&\bM\bD}$$ The horizontal maps are inclusions of ideals. 
By Definition \ref{ihjigwjegoerwgwerffwerfw} the right vertical map is a morphism in $\Ccat$ which is a unitary equivalence. Using that $\bC\to \bD$ is fully faithful one checks that  it induces a unitary  equivalence in $\Ccat$ of the quotients. Hence the square is excisive and send by $\Homol$ to the push-out square  
$$\xymatrix{\Homol(\bC)\ar[d]\ar[r]&\Homol(\bM\bC)\ar[d]\\\Homol(\bD)\ar[r]&\Homol(\bM\bD)}$$ By the special case above we know that the right vertical map is an equivalence. Hence the left vertical map is an equivalence, too.

We now show that  conversely that  Assertion \ref{sfgsdfgsfdgsfdgsfg} implies Assertion \ref{sfgsdfgsfdgsfdgsfg1}. %We assume that $\Homol$ has the properties listed in \ref{sfgsdfgsfdgsfdgsfg1}.
If $X$ is a set, then $$0[X]\stackrel{\id_{0[X]}}{\to} 0[X]\stackrel{\id_{0[X]}}{\to} 0[X]$$ is an exact sequence in $\nCcat$. Applying   $\Homol$
  we get a fibre sequence 
$$\Homol(0[X])\stackrel{\Homol(\id_{0[X]})}{\to} \Homol(0[X])\stackrel{\Homol(\id_{0[X]})}{\to} \Homol(0[X])$$ which immediately implies that $\Homol(0[X])\simeq 0_{\bS}$.
Hence $\Homol$ is reduced.  
If we are given an excisive square \eqref{asdv2e4fwdqwqev}, then we extend its horizontal maps to exact sequences in $\nCcat$ and apply $\Homol$. We then get the diagram
\begin{equation}\label{asdv2e4fwdqwdwdwdwdwdwdqev}
\xymatrix{\Homol(\bA)\ar[r]\ar[d]&\Homol(\bB)\ar[d]\ar[r]&\Homol(\bB/\bA)\ar[d]^{\simeq}\\\Homol(\bC)\ar[r]&\Homol(\bD)\ar[r]&\Homol(\bD/\bC)}
\end{equation}
 The horizontal sequences are send by $\Homol$ to fibre sequences and the right vertical map is an equivalence since $\bB/\bA\to \bD/\bC$ is a unitary equivalence. 
 Consequently, the left square is a push-out square.
Hence $\Homol$ sends excisive squares to push-out squares.
 \end{proof}

Let $\Homol\colon \nCcat\to \bS$ be a homological functor.
\begin{ddd}\label{qerughqiregqewfewqffq1}
$\Homol$ is finitary if $\bS$ is in addition cocomplete and $\Homol$  preserves   small  filtered colimits.
\end{ddd}

%\uli{\begin{rem} The choice of the axioms for a finitary homological functor is  partially motivated by the applications in \cite{coarsek}. In this subsequent  paper we construct coarse homology theories by a two-step procedure. In the first step we construct functors which associate to bornological coarse spaces    $C^{*}$-categories
%of controlled objects and morphisms. In the second step we apply a homological functor. In order to  show 
%that this composition satisfies the axioms of a coarse homology theory we only need the axioms of a finitary homological as introduced above. \hB
%\end{rem}
%}

In the remainder of the present section we study some general  properties of homological functors. 

By $\emptyset$ we denote the empty $C^{*}$-category. Note that $\emptyset\cong 0[\emptyset]$.
\begin{lem}\label{qregheriogwegerwf}
If $\Homol:\nCcat\to \bS$ is a homological functor, then $\Homol(\emptyset)\simeq 0_{\bS}$.
\end{lem}
\begin{proof}
%We have an isomorphism $\emptyset\cong 0[\emptyset]$ in $\nCcat$.
We   use that  $\Homol$ is reduced by Lemma \ref{qerughqiregqewfewqffq} 
to  conclude $\Homol(\emptyset)\simeq \Homol(0[\emptyset])\simeq 0_{\bS}$.
\end{proof}

A  {morphism} $f\colon \bC\to \bD$   {in} $\nCcat$
   is called a zero   {morphism} if it sends every morphism in $\bC$ to zero.
Let $\Homol\colon \nCcat\to \bS$ be a  homological functor. %and let $f\colon \bC\to \bD$ be  a   {morphism} {in} $\nCcat$.
\begin{lem}\label{gbgfbgbfgdgfb}If $f$ is a zero   {morphism}, then $\Homol(f)=0$.
 \end{lem}
\begin{proof}
The   {morphism} $f$ has an obvious factorization
$$\bC\to 0[\Ob(\bD)]\xrightarrow{\omega_{\bD}} \bD\, ,$$ where
$\omega_{\bD}$ is the  obvious inclusion.  %It} acts as the identity on objects, and  on morphisms it acts in  the only possible way.  
By functoriality of $\Homol$ we get a factorization of $\Homol(f)$ as 
$$\Homol(\bC)\to \Homol(0[\Ob(\bD)])\to \Homol(\bD)\, .$$ 
 $\Homol$ is reduced by    Lemma~\ref{oihgjeorgwergergergwerg}.\ref{etghiwjeorgreergwg}  and therefore $ \Homol(0[\Ob(\bD)])\simeq 0_{\bS}$; the claim follows.
\end{proof}

%\Alex{The following definition, though not relevant right now, will be used throughout the rest of the article.} 
 
Let $\Homol\colon \nCcat\to \bS$ be a  functor, and consider  $\bC,\bD$   in $\nCcat$.

\begin{lem}\label{erighi9wrteogwergrgregwergwreg}
If $\Homol$ is homological and $\bC$ and $\bD$ are not empty, then 
 the morphism $$(\Homol(\pr_{\bC}),\Homol(\pr_{\bD})):\Homol(\bC\times \bD)\to \Homol(\bC)\times \Homol(\bD)$$ is an equivalence.
%Assume:
%\begin{enumerate}
%\item $\bC$ and $\bD$ are unital.
%\item \label{wkrtohgwgregreg} {$\bC$ and $\bD$  are non-empty.}
%\item $\Homol$ is homological. 
%\end{enumerate}
%Then the morphism $(\Homol(\pr_{\bC}),\Homol(\pr_{\bD})):\Homol(\bC\times \bD)\to \Homol(\bC)\times \Homol(\bD)$ is an equivalence.
\end{lem}
\begin{proof}
We have an exact sequence
\begin{equation}\label{vsdfvdsfvdfvsdvrefersf}
0[\Ob(\bC)]\times \bD\stackrel{\omega_{\bC}\times\id_{\bD}}{\to}\bC\times \bD 
\stackrel{q_{\bC} }{\to} \bC\times 0[\Ob(\bD)]\ .
\end{equation} 
Here  the morphism $\omega_{\bC}$ is the obvious inclusion  and 
  the morphism $q_{\bC} $   acts as identity on objects and sends 
  a morphism $(f,g)$ in $ \bC\times \bD$ to 
 $(f,0)$ in $\bC\times 0[\Ob(\bD)]$. The sequence is split by
 $\id_{\bD}\times \omega_{\bD}: \bC\times 0[\Ob(\bD)]\to \bC\times \bD $.
 
 We have a factorization  $\pr_{\bC}:=p_{\bC}\circ q_{\bC}$, where
 $p_{\bC}: \bC\times 0[\Ob(\bD)]\to \bC$ is the   projection. 
 We now observe that $p_{\bC}$ is a unitary equivalence provided that $\bD$ is not empty. In fact $p_{\bC}$ is fully faithful, and if $C$ is an object of $\bC$, then
 $p_{\bC}(C, *_{\bD})\cong C$ by unitary multiplier $\id_{C}$, where  $*_{\bD}$ is some object of $\bD$ which exists since we assume that $\bD$ is not empty.
 %with inverse $i_{\bC}$ which sends an object $C$ in $\bC$ to $(C,*_{\bD})$ for some arbitrarily chosen object $*_{\bD}$ (which exists since we assume that $\bD$ is not empty), and a morphism $f$ to $(f,0_{*_{\bD}})$. Then
 %$p_{\bC}\circ i_{\bC}=\id_{\bC}$. Furthermore we have
 %a unitary  multiplier isomorphism $u\colon \id_{\bC\times 0[\Ob(\bD)]}\to i_{\bC}\circ p_{\bC}$. For an object $(C,D)$ in $ \bC\times 0[\Ob(\bD)]$
% it is given by $u_{(C,D)}:=(\id_{C},0):(C,D)\to (C,*_{\bD})$ in $\bC\times 0[\Ob(\bD) ]$.  
We have a similar factorization
 $\pr_{\bD} = p_{\bD}\circ q_{\bD}$, where $p_{\bD}$ is a unitary equivalence since $\bC$ is not empty.  

We  apply $\Homol$ to the split exact sequence \eqref{vsdfvdsfvdfvsdvrefersf}. We then get a split fibre sequence in $\bS$ and therefore an equivalence
$$(\Homol(q_{\bC}),\Homol(q_{\bD}))\colon \Homol(\bC\times \bD)\stackrel{\simeq}{\to}
\Homol( \bC\times 0[\Ob(\bD)])\times \Homol(  0[\Ob(\bC)]\times \bD) \ .$$
We compose with the equivalence $(\Homol(p_{\bC}),\Homol(p_{\bD}))$ 
  in order to conclude the assertion.  
\end{proof}

Our next result asserts that a homological functor is additive on  unital {morphisms} between unital $C^{*}$-categories.
We assume that $\bC,\bD$ are in $\Ccat$, {that $\bC$  is not empty,} and that $\bD$   is additive. If $\phi,\phi'\colon \bC\to \bD$ are two  morphisms in $\Ccat$, then we can define a morphism $\phi\oplus \phi'\colon \bC\to \bD$ by  Definition~\ref{eoiqjofcqwecqwc}. 

%uniquely up to unitary isomorphism as follows. 
%\begin{enumerate}\item \label{wergeroigjgwergregw} objects:
%For every object $C$ of $\bC$ we choose an object $\phi(C)\oplus \phi'(C )$ in $\bD$ together with structure map $e:\phi(C)\to   \phi(C)\oplus \phi'(C )$, $p:\phi(C)\oplus \phi'( C)\to \phi(C)$ such that $pe=\id_{\phi(C)}$
%and $ e':\phi'(C)\to   \phi(C)\oplus \phi'(C )$, $p':\phi(C)\oplus \phi'(C )\to \phi'(C)$
% such that $p'e'=\id_{\phi'(C)}$  and $ep+e'p'=\id_{\phi(C)\oplus \phi'( C)}$.
% We then define $(\phi\oplus \phi')(C):=\phi(C)\oplus \phi'(C)$.
%\item morphisms:
%If $f:C\to \tilde C $ is a morphism in $\bC$, then we define
%$$(\phi\oplus \phi')(f):(\phi\oplus \phi')(C)\to (\phi\oplus \phi')(\tilde C)$$ to be the morphism
%$$\tilde e \phi(f) p+\tilde e' \phi'(f) p':  \phi(C)\oplus \phi'(C )\to   \phi(\tilde C)\oplus \phi'(\tilde C )\ .$$
%\end{enumerate}
%One easily checks that $\phi\oplus \phi'$ is a well-defined morphism. 
%It depends on the choices made in \ref{wergeroigjgwergregw}. A second choice there would lead to a unitarily isomorphic morphism.
%

If $\bS$ is  a stable $\infty$-category, then its  morphism spaces are group-like abelian monoids in spaces $\Spc$.  
The {operation} $+$ in the following proposition is  induced by this structure.
\begin{prop}\label{rguiwtgwreergrgwrgg}
If $\Homol$ is a homological functor, then 
 we have 
an equivalence $$\Homol(\phi\oplus \phi')\simeq \Homol(\phi)+\Homol(\phi')\colon \Homol(\bC)\to \Homol(\bD)\, .$$
%We have an equivalence
%$$\Kcat(\phi\oplus \phi')\simeq \Kcat(\phi)+\Kcat(\phi'):\Kcat(\bC)\to \Kcat(\bD)$$
\end{prop}
\begin{proof}  
Since $\bD$, being additive,    admits the  orthogonal sum of an empty family and  therefore a zero object $0_{\bD}$, it is not empty.
%Then we can define a functor
%\begin{equation}\label{cdsniu3runvifevfvasdfvadsvadsvav}
%s_{\bD}:\bD\sqcup \bD\to \bD\times \bD
%\end{equation} using the universal properties of the product and coproduct uniquely such that 
%$\pr_{i}s \iota_{i}=\id_{\bC}$ and $\pr_{1-i}s\iota_{i}=0$ for all $i$ in $\{0,1\}$.
%Here $\pr_{i}:\bD\times \bD\to \bD$ and $\iota_{i}:\bD\to \bD\sqcup \bD$ for $i$ in $\{0,1\}$ are the canonical projections and inclusions, and $0$ is the constant functor with image the chosen zero object.  
We consider the diagram
\begin{equation}\label{ebwegrfreferwfwefwrefwev}
\xymatrix{
&\Homol(\bD\times \bD)\ar[dr]^-{\Homol(\bigoplus)}\ar[dl]^-{\simeq}_-{\Homol(\pr_{0})\oplus \Homol(\pr_{1})\quad\quad}\\
\Homol(\bD)\oplus \Homol(\bD)\ar[rr]^-{+}
&&\Homol(\bD)
}
\end{equation}
 where the %middle 
 left vertical %diagonal 
 morphism is an equivalence by Lemma  \ref{erighi9wrteogwergrgregwergwreg}. 
 We claim that \eqref{ebwegrfreferwfwefwrefwev}  commutes.

Let $z_{0}:\bD\to \bD\times \bD$ given by $D\mapsto (D,0_{\bD})$ and
$f\mapsto (f,0)$. Let $z_{1}$ be defined similarly switching the roles of the factors.
Then $\pr_{i}\circ z_{i}=\id_{\bD}$ and $\pr_{1-i}\circ z_{i}$ is a zero morphism.
In view of the universal property of $+$ this shows that $\Homol(z_{0})+\Homol(z_{1})$ is an
inverse of $ \Homol(\pr_{0})\oplus \Homol(\pr_{1})$.
Thus in order to show that \eqref{ebwegrfreferwfwefwrefwev} naturally commutes it 
 %Using the explicit inverse \eqref{qwefwefqewfewf1rqwer} of $\Homol(\pr_{0
 %})\oplus \Homol(\pr_{1})$ and the universal property of $+$
   suffices to show that the compositions
 $$\Homol(\bD)\stackrel{\iota_{i}}{\to}\Homol(\bD)\oplus \Homol(\bD)\xrightarrow{\Homol(z_{0})+\Homol(z_{1})} \Homol(\bD\times \bD)\xrightarrow{\Homol(\bigoplus)} \Homol(\bD)$$
 are equivalent to the identity, where $\iota_{i}\colon\Homol(\bD)\xrightarrow{\iota_{i}} \Homol(\bD)\oplus \Homol(\bD)$ denote the canonical inclusions for $i=0,1$. 

 In the case $i=0$ this composition is induced by applying $\Homol$ to the endofunctor
 $s \colon\bD\to \bD$  which   sends an object $D$ to the representative $D\oplus 0_{\bD} $ chosen in the construction of $\bigoplus$, and which 
  sends a morphism $f\colon D\to D'$ to the morphism $f\oplus 0\colon D\oplus 0 \to D'\oplus 0 $. 
We have a unitary equivalence $u\colon \id_{\bD} \to s $ given by the family $(u_{D})_{D\in \Ob(\bD)}$  {of the canonical inclusions}
$u_{D}\colon D\to D\oplus 0_{\bD}$.
 Hence $\Homol(s )\simeq \Homol(\id_{\bD})$.
 The case $i=1$ is analogous.

% We furthermore see that the left triangle commutes by the definition of $s_{\bD}$. This expresses the fact that
% the composition $\Homol(s_{\bD})\circ( \Homol(\iota_{0})+\Homol(\iota_{1}))$ is an inverse of $\Homol(\pr_{0})\oplus \Homol(\pr_{1})$.
%We now note that $$\Homol(\bigoplus)\Homol(s_{\bD})\Homol(\iota_{i})\simeq \Homol(\bigoplus\circ s_{\bD}\circ \iota_{i})\simeq \id_{\Homol(\bD)}$$ for $i=0,1$ since $\bigoplus\circ s_{\bD}\circ \iota_{i}$ is unitarily isomorphic to $\id_{\bD}$.
%This implies that the right triangle also commutes. 

We have the following  diagram in $\bS$:
\begin{equation}\label{svfdviuvhuisdfvsdfvsdfvsdfv}
\xymatrix{
&\Homol(\bC)\ar[dr]^{\diag_{\Homol(\bC)}}\ar[dl]_{\Homol(\diag_{\bC})\ } &\\
\Homol(\bC\times \bC)\ar[rr]_-{\simeq}^-{\Homol(\pr_{0})\oplus \Homol(\pr_{1})} \ar[d]^{\Homol(\phi\times \phi')}& &\Homol(\bC)\oplus \Homol(\bC)\ar[d]^{\Homol(\phi)\oplus \Homol(\phi')} \\
\Homol(\bD \times  \bD)\ar[dr]_{\Homol(\bigoplus)\ }\ar[rr]_-{\simeq}^-{\Homol(\pr_{0})\oplus \Homol(\pr_{1})} &  &\Homol(\bD)\oplus \Homol(\bD) \ar[dl]^{+}\\
&\Homol(\bD)&
}
\end{equation}      
The lower triangle is  \eqref{ebwegrfreferwfwefwrefwev} and  commutes as shown above.
The upper triangle and the middle square obviously commute. The left top-down path is the map
$\Homol(\phi\oplus \phi')$, while the right top-down path is $\Homol(\phi)+\Homol(\phi')$. 
The filler of \eqref{svfdviuvhuisdfvsdfvsdfvsdfv} now provides the desired equivalence between these morphisms.
\end{proof}

Since the operation $+$ occuring in Proposition~\ref{rguiwtgwreergrgwrgg} is abelian we immediately get the following consequence.
\begin{kor}
$\Homol(\phi \oplus \phi^\prime)$ is equivalent to $\Homol(\phi^\prime  \oplus \phi)$.
\end{kor}

Recall the notion  of a flasque  $C^{*}$-category  introduced in Definition~\ref{bjiobdfbfsferq}.
\begin{prop}\label{ergiwoegwergwergwgrg} 
A homological functor annihilates flasques.
\end{prop}
\begin{proof}
Let $\Homol$ be a homological functor.
Furthermore, let $\bC$ be in $\Ccat$  {and assume that it is flasque}.
We must show that $\Homol(\bC)\simeq 0$.

The case where $\bC$ is empty  follows from Lemma \ref{qregheriogwegerwf}.
We now assume that $\bC$ is not empty, and 
  that $S\colon\bC\to \bC$ implements the flasqueness of $\bC$.
Then using Proposition~\ref{rguiwtgwreergrgwrgg}   we have the relation
$$\Homol(S)= \Homol(\id_{\bC}\oplus S)=\id_{\Homol(\bC)}+\Homol(S) $$  in the abelian group 
$[\Homol(\bC),\Homol(\bC)]$. This implies    that $\Homol(\bC)\simeq 0$.
\end{proof}

\section{Topological \texorpdfstring{$\bm{K}$}{K}-theory of \texorpdfstring{$\bm{C^{*}}$}{Cstar}-categories}\label{ergiojerfrfqefqwefqwef}

The goal of this section is to provide a reference for the topological $K$-theory functor for $C^{*}$-categories. Most of the material is from \cite{joachimcat}. The main result (Theorem \ref{qrgkljwoglergerggwergwrg}) states that this  $K$-theory functor is a finitary homological functor {(Definitions~\ref{oihgjeorgwergergergwerg} and \ref{qerughqiregqewfewqffq1}).}

Our starting point is the topological $K$-theory functor $ \Kast$ for $C^{*}$-algebras.  {
Recall that  $\nCalg$ denotes}
%Let $\nCalg$ 
the category of  {%very 
small}, possibly non-unital $C^{*}$-algebras and not necessarily unit-preserving homomorphisms.   {We consider} $\nCalg$ as a full subcategory of $\nCcat$ consisting of   the $C^{*}$-categories with a single object. 
Topological $K$-theory of $C^{*}$-algebras is a functor
\[\Kast\colon \nCalg\to \Sp\, .\]
 References for the induced group-valued functor
\[\pi_{*}\Kast\colon \nCalg\to \Ab^{\Z/2\Z\mathrm{gr}}\]
 (whose construction predates the spectrum-valued version)
 are, e.g.\  \cite{blackadar,higson_roe}, while the spectrum-valued one is defined in \cite[Defn.~4.9]{joachimcat}
 and justified by \cite[Thm.~4.10]{joachimcat}. An alternative construction using spectrum-valued $K\!K$-theory 
can be based on \cite{Land:2016aa}, see also  {\cite[Sec. 8.4]{buen}, \cite{KKG}}, and the survey \cite{Bunke:2023aa}.

In the following we list all the properties which will be explicitly
%\footnote{Since the proof of Theorem  \ref{qrgkljwoglergerggwergwrg} also involves the equivalence \eqref{sfdbjoqi3gegergewgerg} it implicitly uses more properties of $\Kast$, at least stability and homotopy invariance in the proof of Lemma \ref{sfdbjoqi3gegergewgerg}.} 
used in the proof of Theorem~\ref{qrgkljwoglergerggwergwrg} below. 
%stating that  $\Kcat$ is a  {finitary} homological functor. 
 %are relevant  for showing that $\Kcat$ is a homological functor (Theorem \ref{qrgkljwoglergerggwergwrg}).  
 \begin{prop}\label{fiodgerwvfdsvdvsfv}
The functor $\Kast$ has the following properties.
\begin{enumerate}
\item \label{ergoijuefoiqewfqwfewfqwefqwf} $\Kast(0)\simeq 0$.
\item\label{wtgwergfervfdsv} $\Kast$  preserves  {small}  filtered colimits. 
\item \label{wtgwergfervfdsv1}$\Kast$ sends exact sequences of $C^{*}$-algebras  
to fibre sequences.  
\item \label{wegoijweroigjergwerg} $\Kast$ is $\mathbb{K}$-stable (see Remark \ref{wrthijowrtgergwergwegr}.\ref{qfoijirfoqfqwefqewfqewf}).
\item \label{wroigjiowertgerwgrewgwreg} $\Kast$ is homotopy invariant (see Remark \ref{wrthijowrtgergwergwegr}.\ref{qriogfjoerfwefqwefqwefqef}).
\item $\Kast$ is Bott periodic (see Remark \ref{wrthijowrtgergwergwegr}.\ref{wegiojwoegwergrwegwreferf}).
\end{enumerate}
\end{prop}

 \begin{rem}\label{wrthijowrtgergwergwegr}
In this remark we add some details to the statement of Proposition \ref{fiodgerwvfdsvdvsfv}.
 \begin{enumerate}
 \item \label{ewigj9weoferwfwefrwerf}
 %$0\to A\to B\to C\to 0$ to cocartesian squares   \begin{equation}\label{rewghgwioergwerferf}
%\xymatrix{\Kast(A)\ar[r]\ar[d]&\Kast(B)\ar[d]\\0_{\Sp}\ar[r]&\Kast(C)}
%\end{equation}
   %in $\Sp$ {(see Remark \ref{wrthijowrtgergwergwegr}.\ref{ewigj9weoferwfwefrwerf})}.
  An exact sequence of $C^{*}$-algebras % as in \ref{fiodgerwvfdsvdvsfv}.\ref{wtgwergfervfdsv1} 
  is a square $$\xymatrix{A\ar[r]\ar[d]&B\ar[d]\\0\ar[r]&C}$$
in $\nCalg$ which is a pull-back and a push-out at the same time. Assertion  \ref{fiodgerwvfdsvdvsfv}.\ref{wtgwergfervfdsv1} can be reformulated to saying that    $\Kast$ sends such squares to  cocartesian  (or equivalently by stability of $\Sp$, to  cartesian) squares   \begin{equation}\label{rewghgwioergwerferf}
\xymatrix{\Kast(A)\ar[r]\ar[d]&\Kast(B)\ar[d]\\0_{\Sp}\ar[r]&\Kast(C)}
\end{equation}
   in $\Sp$.
  Since the left-lower corner in \eqref{rewghgwioergwerferf} is the zero object in $\Sp$  such a  square  is the same as a fibre sequence in $\Sp$. %Thus the condition can reformulated to saying that $\Kast$ sends exact sequences of $C^{*}$-algebras to fibre sequences.}
 \item \label{qfoijirfoqfqwefqewfqewf}$\mathbb{K}$-stability:   Let $\mathbb{K}$ denote the $C^{*}$-algebra of compact operators on a separable Hilbert space. 
 Fixing a rank-one projection $p$ in $\mathbb{K}$ we get a   morphism $\C\to \mathbb{K}$, $\lambda\mapsto \lambda p$, in $\nCalg$. For every $C^{*}$-algebra $A$ we get an induced morphism $A\cong A\otimes \C\to A\otimes \mathbb{K}$ (all choices of a $C^{*}$-algebraic tensor product coincide in this case). 
 Stability then says that the induced map of spectra $$\Kast(A)\to \Kast(A\otimes \mathbb{K})$$
 is an equivalence.
 \item \label{qriogfjoerfwefqwefqwefqef} homotopy invariance: {The condition says that  for} every $C^{*}$-algebra $A$ the map $A\to C([0,1],A)$
 given by the inclusion of $A$ as constant functions  induces an equivalence 
 $$\Kast(A)\to \Kast(C([0,1],A))\, .$$
  \item\label{wegiojwoegwergrwegwreferf}  Bott periodicity:    For every $C^{*}$-algebra $A$
  we have a natural equivalence
  \[\Sigma^{2}\Kast( A)\simeq  \Kast(A)\, .\]
As observed by J.~Cuntz (the argument is reproduced in \cite[Cor. 6.10]{zbMATH07948612}) this property is actually a formal consequence of the other properties stated in Proposition \ref{fiodgerwvfdsvdvsfv}. 
  
In view of Bott periodicity, in order to show that a morphism $\Kast(A)\to \Kast(B)$ is an equivalence it suffices to
show that $\pi_{i}\Kast(A)\to \pi_{i}\Kast(B)$ is an isomorphism for $i=0,1$.\hB
\end{enumerate}
\end{rem}

The inclusion of $C^{*}$-algebras into $C^{*}$-categories is the right-adjoint of an adjunction
\begin{equation}
A^{f}:\nCcat\leftrightarrows \nCalg:\incl\, .
\end{equation}
We refer to  \cite[Lem.~3.9]{crosscat} for details. Note that
the functor $A^{f}$ has been first introduced in \cite{joachimcat}.
  % We have a left-adjoint functor 
%$$A^{f}:\nCcat\to \nAlgc\ ,$$ see \cite[Lem.~3.9]{crosscat}. \Alex{Die Definition davon hier bringen?}
Following \cite{joachimcat} we adopt the following definition.
\begin{ddd}\label{ergiowergregrdsvsv}
We define the topological $K$-theory functor for $C^{*}$-categories as the composition
$$\Kcat\colon \nCcat\xrightarrow{A^{f}} \nCalg\xrightarrow{\Kast} \Sp\, .$$
\end{ddd}
 {Note that Mitchener \cite{mitchener_KTh_Ccat} provided an alternative construction of a $K$-theory functor for $C^*$-categories.}

For the following theorem recall Definitions~\ref{oihgjeorgwergergergwerg} and \ref{qerughqiregqewfewqffq1}.

\begin{theorem}\label{qrgkljwoglergerggwergwrg}
The functor $\Kcat$ is a finitary homological functor.
\end{theorem}

\begin{proof}{Note that the $\infty$-category  $\Sp$ is stable.} Hence the theorem follows from 
%\fuli{Verweis auf
 %  Proposition \ref{qeriughq9wofewfqewfewqfqwff} gestrichen, die braucht man nur noch intern}
Proposition \ref{rgiowegwreergwegwerg} {(fibre sequences)} and %  Lemma \ref{reughiqefefqfqwefqef} \uli{(reduced)}, and 
 Lemma \ref{wtiowgergwgreg} {(finitary)} %, and Lemma \ref{sfdbhjqirjgiosdfjviosfvjodfbvsfvsfv}.
shown below.
\end{proof}

%\subsubsection{Proof of Theorem~\ref{qrgkljwoglergerggwergwrg}}
%\label{sec_K_is_Hg}

%By \cite[Theorem 4.1]{crosscat}   the category   $\nCcat$ is complete and cocomplete. In particular it admits filtered colimits.
 \begin{lem}\label{wtiowgergwgreg}
 The functor $\Kcat $ preserves small filtered colimits.
 \end{lem}
\begin{proof} By definition, 
the functor  $A^{f}$ is a left-adjoint and therefore preserves all {small} colimits. 
 The functor $\Kast $ preserves {small} filtered colimits by Proposition \ref{fiodgerwvfdsvdvsfv}.\ref{wtgwergfervfdsv}.
 Hence the composition $ \Kcat$ preserves {small} filtered colimits.
\end{proof}

We now use the functor $A$ from  \eqref{frewfoirjviojvioeweverwvwev}.
The universal property of $A^{f}$ together with \eqref{qweflqjoijefqwefqwefqwef} provides a natural transformation  
\begin{equation}\label{eq_transformation_Af_A}
\alpha\colon A^{f}\to A
\end{equation}
of functors   from $\nCcat_{i}$ to $\nCalg $,
see, e.g.,  \cite[Lem.\ 8.54]{buen}.

In order to provide a self-contained presentation we give the proof of the following lemma.
Let $\bC$ be in $\nCcat$.
\begin{lem}[{\cite[Prop.\ 8.55]{buen}}]\label{sfdbjoqi3gegergewgerg1}
The morphism
\begin{equation}\label{sfdbjoqi3gegergewgerg}
\Kast(\alpha_{\bC})\colon\Kast(A^{f}(\bC))\to \Kast(A(\bC)) 
\end{equation}
is an equivalence.
\end{lem}
\begin{proof}
 In the special case that  $\bC$ is unital and has a countable set of objects  the assertion of the  {lemma} has been shown by Joachim \cite[Prop.~3.8]{joachimcat}.
 
First assume that $\bC$ has countably many objects, but is possibly non-unital. Then the arguments from the proof of \cite[Prop.~3.8]{joachimcat} are applicable and show that the canonical map $\alpha_{\bC}\colon A^{f}(\bC)\to A(\bC)$  is a stable homotopy equivalence. 
%Since we need the details anyway in the proof of Lemma \ref{fwijwefoewfewfwfewfewfw} below l
Let use recall the construction of the stable inverse $$\beta\colon A(\bC)\to A^{f}(\bC)\otimes \IK\, ,$$
where $\IK \coloneqq \IK(H)$ {are the compact operators on} the Hilbert space $$H \coloneqq \ell^2(\Ob(\bC) \cup \{e\})\, ,$$  where  $e$ is an artificially added point.
The assumption on the cardinality of $\Ob(\bC)$ is made since we want that $\IK$ is the algebra of compact operators on a separable Hilbert space.
Two points $x,y$ in $\Ob(\bC) \cup \{e\}$ provide a rank-one operator $\Theta_{y,x}$ in  $\IK(H)$ which sends the basis vector corresponding to $x$ to the vector corresponding to $y$, and which vanishes on the orthogonal complement of $x$.  
The homomorphism $\beta$ is given on  $A$ in $\Hom_{\bC}(x,y)$ by $$\beta(A) := A \otimes \Theta_{y,x}\, .  $$
If $A$ and $B$ are composable morphisms, then the relation $\Theta_{z,y}\Theta_{y,x}=\Theta_{z,x}$ implies   that $\beta(B\circ A)=\beta(B)\beta(A)$. Moreover, 
if $A$,$B$ are not composable, then $\beta(B)\beta(A) = 0$. Finally, $\beta(A)^{*}=\beta(A^{*})$ since
$\Theta_{y,x}^{*}=\Theta_{x,y}$. It follows that  $\beta$ is a well-defined $^\ast$-homomorphism.

The argument now proceeds by showing that the composition
$(\alpha_{\bC}\otimes \id_{\IK(H)})\circ \beta$ is homotopic to $\id_{A(\bC)}\otimes \Theta_{e,e}$, 
and that the composition
$\beta\circ \alpha_{\bC}$ is homotopic to $\id_{A^{f}(\bC)}\otimes \Theta_{e,e}$.
Note that in our setting $\bC$ is not necessarily unital. In the following we directly refer to the  proof of \cite[Prop.~3.8]{joachimcat}.
The only step in the  proof of that proposition where the identity morphisms are used is the definition of the maps denoted by $u_x(t)$ in the reference. But they in turn are only used to define the map denoted by $\Xi$  later in that proof. The crucial observation is that we can define this map $\Xi$ directly without using any identity morphisms in $\bC$.

We  conclude that the canonical map $\Kast(\alpha_{\bC})\colon K(A^{f}(\bC))\to \Kast(A(\bC))$ is an equivalence for $C^\ast$-categories $\bC$ with countably many objects.

In order to extend this to all $C^{*}$-categories,
 we  use the fact that $A^{f}$ commutes with  {small} filtered colimits which implies that $A^f(\bC) \cong \colim_{\bC^\prime} A^f(\bC^\prime)$, where the colimit runs over the  filtered poset of all full subcategories with countably many objects.  
 The connecting maps of the indexing family of this colimit  are functors which are injections on objects. 
We now argue that also $A(\bC)\cong \colim_{\bC^{\prime}} A(\bC^{\prime})$. 
Note that  $A$ is the composition of the functor 
 $A^{\alg}\colon \nCcat\to \npAlgc$  (see \cite[Defn.~6.1 and Lem.~6.4]{crosscat}) with the completion functor $\Compl\colon \npAlgc\to \nCalg$ (\cite[(3.17]{crosscat}). By construction  the functor $A^{\alg}$  preserves filtered colimits    with connecting maps that are injective on objects. The completion functor is a left-adjoint and therefore preserves all {small} filtered colimits.   This implies that $A$ commutes with $\colim_{\bC'}$.
 
Since $\Kast$  commutes with {small} filtered colimits   
the morphism $$ \Kast(\alpha_{\bC})\colon \Kast(A^{f}(\bC))\to \Kast(A(\bC))$$  is equivalent to the morphism
$$\colim_{\bC'}\Kast(\alpha_{\bC'})\colon \colim_{\bC^{\prime}} \Kast(A^{f}(\bC^{\prime}))\to \colim_{\bC^{\prime}} \Kast(A(\bC^{\prime}))\, .$$
Since the categories $\bC^{\prime}$ appearing in the colimit   have at most countably many objects we have identified
$\Kast(\alpha_{\bC}) $ with a colimit of equivalences. Hence this morphism itself is an equivalence. 
\end{proof}

\begin{prop}\label{rgiowegwreergwegwerg}
The functor $\Kcat$ sends  
exact sequences to fibre sequences.
%excisive squares in $\nCcat$ to  cocartesian squares 
%\begin{equation}\label{asdv2e4fwdqdwdwdwdwdwqev}
%\xymatrix{\Kcat(\bA)\ar[r]\ar[d]&\Kcat(\bB)\ar[d]\\\Kcat(\bC)\ar[r]&\Kcat(\bD)}
%\end{equation}
%in $\Sp$.
\end{prop}
\begin{proof}
Let $0\to \bC\to \bD \to \bQ\to 0$ be an exact sequence  in $\nCcat$.
Then we get the following commutative diagram
$$\xymatrix{%0\ar[r] &A(\bA)\ar[r] &A(\bB)\ar[r] &A(\bB/\bA) \ar[r]&0\\ &A^{f}(\bA)\ar[d]\ar[r]\ar[u]^{\alpha_{\bA}}&A^{f}(\bB)\ar[d]\ar[r]\ar[u]^{\alpha_{\bB}}&A^{f}(\bB/\bA)\ar[d] \ar[u]^{\alpha_{\bB/\bA}}&\\
&A^{f}(\bC)\ar[r]\ar[d]_{\alpha_{\bC}}&A^{f}(\bD)\ar[r]\ar[d]_{\alpha_{\bD}}&A^{f}(\bQ)\ar[d]_{\alpha_{\bQ}}&\\
0\ar[r]&A(\bC)\ar[r]&A(\bD)\ar[r]&A(\bQ)\ar[r]&0}$$
where the lower sequence is exact since $A$ preserves exact sequences by \cite[Prop.~8.9.2]{crosscat}. We now apply $\Kast$ and use Definition \ref{ergiowergregrdsvsv} in order to express the entries in the upper line in terms of $\Kcat$  in order to get
\begin{equation}\label{vdsfoijqiojviofvdfvsdfvsdfv}
\xymatrix{%
%\cdots\ar[r] &\Kast(A(\bA))\ar[r] &\Kast(A(\bB))\ar[r] &\Kast(A(\bB/\bA)) \ar[r]&\cdots \\
%& \Kcat(\bA)\ar[u]_{\simeq}\ar[r]\ar[d]&\Kcat(\bB)\ar[d]\ar[u]_{\simeq}\ar[r]\ar[d]&\Kcat(\bB/\bA)\ar[d]^{\simeq}_{!}\ar[u]_{\simeq}&\\
&\Kcat(\bC)\ar[d]^{\simeq}\ar[r]&\Kcat(\bD)\ar[d]_{\simeq}\ar[r]&\Kcat(\bQ)\ar[d]^{\simeq}&\\
\cdots\ar[r]&\Kast(A(\bC))\ar[r]&\Kast(A(\bD))\ar[r]&\Kast(A(\bQ))\ar[r]&\cdots
}
\end{equation}
The vertical morphisms are equivalences by Lemma \ref{sfdbjoqi3gegergewgerg1}, and the lower sequence is a fibre sequence by Proposition~\ref{fiodgerwvfdsvdvsfv}.\ref{wtgwergfervfdsv1}. Hence the upper line is a fibre sequence.
\end{proof}

In Theorem~\ref{thm_G_amenable} we have seen that the comparison functor $$q_{ {\bC}}\colon {\bC} \rtimes G \to  {\bC} \rtimes_r G$$ is an isomorphism for any $  \bC$ in $\Fun(BG,\nCcat)$ provided that $G$ is amenable. 
 If one is interested in this isomorphism  only after applying $K$-theory, then one can
 weaken the assumption on $G$ from amenable to $K$-amenable. Since we consider discrete groups $G$ we can adopt the following definition:
 \begin{ddd}[{\cite[Def.\ 2.2]{cuntz-kam}, \cite[Sec.\ 1.3.2]{book_Haagerup}}]
 The discrete group $G$ is $K$-amenable if for every $  A$ in $\Fun(BG,\nCalg)$
 the morphism
 $$\Kast(q_{  A}) \colon \Kast( A\rtimes G)\to \Kast( A \rtimes_{r}G)$$
 is an equivalence.
\end{ddd}
  
The class of $K$-amenable groups contains all amenable groups, but  also all groups with the Haagerup property (also often called a-T-menability), and hence  for example also all Coxeter groups and all CAT(0)-cubical groups \cite[Sec.\ 1.2]{book_Haagerup}.

% in the case of $C^*$-algebras it holds for considerably more groups than just the amenable ones: the $K$-amenable groups \cite[Sec.\ 1.3.2]{book_Haagerup}. This class of groups contains all groups with the Haagerup property (also often called a-T-menability), and hence besides amenable groups for example also all Coxeter groups and all CAT(0)-cubical groups \cite[Sec.\ 1.2]{book_Haagerup}.

Let $  \bC$ be in $\Fun(BG,\nCcat)$.
\begin{theorem}\label{thm_Kamenable_equiv}
If $G$ is $K$-amenable, then the morphism $\Kcat(q_{  {\bC}}) \colon \Kcat(  \bC\rtimes G)\to \Kcat(  \bC\rtimes_{r}G)$ is an equivalence.
\end{theorem}
\begin{proof}
%We will prove that $\Kcat_*(q_{\tilde{\bC}})$ is an isomorphism of abelian groups. 
We have the following commutative diagram
\[
\xymatrix{
\Kcat( {\bC} \rtimes G) \ar[r]^-{\eqref{sfdbjoqi3gegergewgerg}}_-{{\simeq}} \ar[d]^-{\Kcat(q_{ {\bC}})} & \Kast(A( {\bC} \rtimes G)) \ar[d]^-{\Kast(A(q_{
\bC}))} && \Kast(A({\bC}) \rtimes G) \ar[ll]_-{\eqref{eq_nu_max}}^-{{\simeq}} \ar[d]^-{\Kast(q_{A( {\bC})})}\\
\Kcat({\bC} \rtimes_r G) \ar[r]^-{\eqref{sfdbjoqi3gegergewgerg}}_-{{\simeq}} & \Kast(A({\bC} \rtimes_r G)) && \Kast(A( {\bC}) \rtimes_r G) \ar[ll]_-{\eqref{eq_nu_r}}^-{{\simeq}}
}
\]
where the right square is obtained by applying $\Kast$ to  the square\eqref{eq_nu_isos}. Because $G$ is $K$-amenable, the right vertical arrow $\Kast(q_{A({\bC})})$ is an equivalence.  Therefore $\Kcat(q_{{\bC}})$ is an equivalence, too.
\end{proof}

\section{\texorpdfstring{$\bm{K}$}{K}-theory of products of \texorpdfstring{$\bm{C^{*}}$}{Cstar}-categories}\label{ergoihoifwefwefwefqwefqwef}
%We assume that $\bC$ is additive.
%For every finite subset  $F$ in of objects of $\bC$ we can define an object
%$C(F):=\bigoplus_{C\in F}$.  We get  $C^{*}$-algebra $E(F):=\End_{\bC}(C(F))$. An inclusion $F\to F^{\prime}$ induces an isometry
%$C(F)\to C(F')$ and an inclusion$E(F)\to E(F^{\prime})$.  
%We define
%$$E(\bC)\to \colim_{F}E(F)$$ where $F$ runs over the poset of finite sets of objects of $\bC$.
%A functor $\bC\to \bC'$ which in injective on objects induces a homomorphism $E(\bC)\to E(\bC')$.
%
%We now assume that $(\bC_{i})_{i\in I}$ is a family of additive $C^{*}$-categories.
%Then we can form the $C^{*}$-category $ \prod_{i\in I} \bC_{i}$.
%We have a canonical map
%$$E(\prod_{i\in I} \bC_{i})\to  \prod_{i\in I}E(\bC_{i})\ .$$
%Indeed, let $F$ be a finite set of objects of $\prod_{i\in I} \bC_{i}$.
%Every element of $F$ is a family $(C_{i})_{i\in I}$, where $C_{i}$ is an  object  of $\bC_{i}$. 
%So we get a finite family $F_{i}$ of objects in $\bC_{i}$.
%We have projections $E(F)\to E(F_{i})$ for all $i$ in $I$.
%They provide  isomorphism
%$$E(F)\cong \prod_{i\in I} E(F_{i})\ .$$
%We now take the colimit over $F$ and get
%$$E(\prod_{i\in I} \bC_{i})\cong \colim_{F} E(F)\cong \colim_{F} \prod_{i\in I} E(F_{i})\to   \prod_{i\in I}  \colim_{F_{i}}E(F_{i}) \ .$$
%Let $(F_{i})_{i\in I}$ be a family, where $F_{i}$ is a finite set of objects in $\bC_{i}$.
%Then we consider the object $(C(F_{i}))_{i\in I}$ in $ \prod_{i\in I} \bC_{i}$.
%
%
%
%xxxxx

The main result of this section is Theorem \ref{ojgweorergwerfewrfwerfw} stating that the $K$-theory of a product of additive unital $C^{*}$-categories is equivalent to the product of the $K$-theories of the factors.
For finite products,   this {holds for any homological functor in place of $\Kcat$ and follows immediately  from Lemma~\ref{erighi9wrteogwergrgregwergwreg}. In view of   Theorem~\ref{qrgkljwoglergerggwergwrg} this applies to 
  $\Kcat$. So the interesting case are infinite families, where  this property seems to be a speciality of $\Kcat$.}

In order to simplify the notation in this section we  use the notation $K_{*}(A) \coloneqq \pi_{*}\Kast(A)$ for the $K$-theory groups of a $C^{*}$-algebra $A$. 

Let $A$ be an algebra and $n,m$ in $\nat$. For $a$ in $A$ and $i$ in $\{1,\dots,n\}$ and $j$ in $\{1,\dots,m\}$ we let $a[i,j]$ in $\Mat_{n,m}(A)$ denote the matrix
whose only non-zero entry is $a$ in position $(i,j)$. For $i$ in $\{1,\dots,n\}$ we let \begin{equation}\label{sfdvoijio3rgfevfdsvf}
\epsilon_{A,n}[i] \colon A\to \Mat_{n}(A)
\end{equation} denote the injective (non-unital if $n\ge 2$) algebra homomorphism which sends $a$ in $A$ to $a[i,i]$.
 
Let $A,B$ be $*$-algebras. Recall from Definition~\ref{r9erqgrgqrg0} that  an element $u$ in $B$ is  a partial isometry if $uu^{*}$ and $u^{*}u$ are projections in $B$.  
Let $h\colon A\to B$ be a $*$-homomorphism such that $h u^{*}u=h$. Then $$h' \coloneqq u hu^{*}\colon A\to B$$ is another $*$-homomorphism.

If $A$ and $B$ are $C^{*}$-algebras and the homomorphisms $h,h'\colon A\to B$ are related as described above with $u$ in the multiplier algebra of $B$, then we have an equality between the induced maps on $K$-theory groups
\begin{equation}\label{vasfvlkmlksaf24fwf}
h_{*}=h'_{*}\colon K_{*} (A)\to K_{*} (B)\, ,
\end{equation} 
see, e.g., \cite[Rem.~8.44]{buen}.

If $A$ is a $C^{*}$-algebra, $n$ in $\nat$, and $i$ in $\{1,\dots,n\}$, then by the matrix stability of $\Kast$ the homomorphism of $K$-theory groups
\begin{equation}\label{webwpokvopwervwcervwev}
\epsilon_{A,n}[i]_{*}\colon K_{*}(A)\to K_{*}(\Mat_{n}(A))
\end{equation}
induced by the homomorphism  \eqref{sfdvoijio3rgfevfdsvf} of $C^{*}$-algebras is an isomorphism.
 
We consider $\bC$ in $\Ccat$. If $F$ is a finite subset of objects of $\bC$, then we have a unital subalgebra
\begin{equation}\label{nriqrnficjwenkcsdcdcsdacsdcasdcasdcad}
A(F) \coloneqq \bigoplus_{C,C'\in F} \Hom_{\bC}(C,C')
\end{equation}
of $A^{\alg}(\bC)$, see \eqref{efefwefeefefyyyyyyyefewfef} for notation. %cite[Defn.~6.1]{crosscat} for \Alex{the definition of} $A(-)^{\alg}$, \Alex{or \eqref{efefwefeefefyyyyyyyefewfef}}). 
For $n$ in $\nat$ the    inclusion  $A(F)\to A(\bC)$ induces   the homomorphism  of matrix algebras
$$h_{F,n}\colon \Mat_{n}(A(F))\to \Mat_{n}(A(\bC))\,,$$
where $A(\bC)$ is as in \eqref{frewfoirjviojvioeweverwvwev}.  
For an object $C$ of $\bC$   we use the notation
\begin{equation}\label{sfvniojqrogvfevfdsvsfdvsfdvvsfdv}
\ell_{C}\colon \End_{\bC}(C) \to A(\bC)
\end{equation} 
for the canonical inclusion.

Let $\bC$ be in $\Ccat$, 
  $F$ be a finite set of objects of $\bC$, and let $n$ in $\nat$.
\begin{lem}\label{rqgiowergwregwergwergeg}
Assume that $\bC$ is   additive. Then
there is a partial isometry $u$ in $\Mat_{n}(A(\bC))$ and an object $C(F,n)$ in $\bC$ such that $h_{F,n}u^{*}u=h_{F,n}$
and $h':=uh_{F,n}u^{*}$ has a factorization
\begin{equation}\label{urehfuiehisudhciousdcdsa}
h'\colon \Mat_{n}(A(F))\xrightarrow{\phi_{F,n}} \End_{\bC}(C(F,n))\xrightarrow{\ell_{C(F,n)}} A(\bC)\xrightarrow{\epsilon_{A(\bC),n}[1]} \Mat_{n}(A(\bC))
\end{equation}
where the isomorphism $\phi_{F,n}$ will be constructed in the proof.
\end{lem}
\begin{proof}
  We consider the family $$((C,i))_{C\in F,i\in \{1\dots,n\}}$$ of elements  in $F$, i.e.,  every element of $F$ is repeated $n$ times.
 We then  choose a sum 
   \begin{equation}\label{sdfbsdfgwgdfbvsbvfdvs}
 \left( C(F,n),(e_{(C,i)})_{C\in F,i\in \{1,\dots,n\}}\right)\end{equation} 
of this finite family, see Definition \ref{regiuhqrogefewfqwfqef}.
% In order to simplify the notation we introduce the abbreviation
 % \begin{equation}\label{sdfbsdfgwgdfbvsbvfdvs}
%C(F,n):=\bigoplus_{i=1}^{n}\bigoplus_{C\in F} C\ .
%\end{equation} 
%The morphisms  $e_{(C,i)}:C\to C(F,n)$  satisfy the relations \begin{equation}\label{bsbsgbsbs}
%e_{(C',i')}^{*}e_{(C,i)}=\left\{\begin{array}{cc}\id_{C}&C=C' \:\mbox{and}\: i=i'\\0&else\end{array}\right. 
%\end{equation}
% and
% \begin{equation}\label{ergregrgwggregwer}
%\sum_{i=1}^{n}\sum_{C\in F}e_{(C',i')}e^{*}_{(C,i)}= \id_{C(F,n)}\ .
%\end{equation}

We can view morphisms in $\bC$ as elements of $A(\bC)$ in a canonical way. A morphism between objects in $F$ is an element of $A(F)$.
   We have  an isomorphism
\begin{equation}\label{vasvfjor3gervvfsdvfdvsv}
\phi_{F,n}\colon \Mat_{n}(A(F))\to     \End_{\bC}(C(F,n))\, ,  
\end{equation}
that sends the matrix $f[i,i']$ with $f\colon C'\to C$ in $\Mat_{n}(A(F))$ to $e_{C,i}  fe^{*}_{C',i'}$ in $\End_{\bC}(C(F,n))$. One checks that
\begin{equation}\label{fvsdfvfvsv}
\epsilon_{A(\bC),n}[1]\circ \ell_{C(F,n)}\circ \phi_{F,n}(-)=\sum_{i,i'=1}^{n}\sum_{C,C^{\prime}\in F} e_{C,i}[1,i] (-)e_{C',i'}^{*}[i',1]
\end{equation}
% Using the embeddings $\ell_{C(F,n)}$ and $\epsilon_{A(\bC)}[1]$   we consider $  \End_{\bC}(C(F,n))$ as a subalgebra of $A(\bC)$ in the canonical way, and further as a subalgebra  of $\Mat_{n}(A(\bC))$.
{as maps $\Mat_{n}(A(F))\to \Mat_{n}(A(\bC))$.} We define a matrix in $\Mat_{n}(A( \bC))$ by
\begin{equation}\label{sfdvlknejkvndfkvervwevwevw}
u \coloneqq \sum_{i=1}^{n} \sum_{C\in F}e_{C,i}[1,i]\, .
\end{equation}
  Using the orthogonality relations for the family $(e_{C,i})_{C\in F,i\in \{1,\dots,n\}}$ considered as elements in $A(\bC)$ we    calculate  that 
\begin{equation}\label{regwerlkgmwelrgwergerg}
uu^{*}=\id_{C(F,n)}[1,1] \, , \quad 
u^{*}u=1_{\Mat_{n}(A(F))}\, .
\end{equation}
%\begin{equation}\label{regwerlkgmwelrgwergerg}uu^{*}=%\sum_{i,i'=1}^{n} \sum_{C,C'\in F }    e_{C,i} [1,i]e_{C',i'}^{*}[i',1]=\sum_{i=1}^{n}\sum_{C \in F} e_{C,i}e_{C,i}^{*}[1,1]\stackrel{\eqref{ergregrgwggregwer}}{=}
%  \id_{C(F,n)}[1,1] \ , \quad 
%  %Note that the reduction of the summation to pairs $C,C^{\prime}$ with $C=C'$ is due to the fact that
%  %we consider  $e_{C,i}$ as an element of $A(\bC)$. 
% % We furthermore have 
% %\begin{equation}\label{regwerlkgmwelrgwergerg}
%u^{*}u%=\sum_{i,i'=1}^{n}\sum_{C,C'\in F} e^{*}_{C,i}[i,1]e_{C',i'}[1,i'] \stackrel{\eqref{bsbsgbsbs}}{=} \sum_{i=1}^{n}\sum_{C\in F}  \id_{C}[i,i] 
%=1_{\Mat_{n}(A(F))}\ .
%\end{equation} 
 %Here we use that $\sum_{C\in F}  \id_{C}=1_{A(F)}$ for the last equality. 
The second equation in  \eqref{regwerlkgmwelrgwergerg} immediately implies that $$h_{F,n}=h_{F,n}u^{*}u\, .$$
 We now calculate
\[h':=u h_{F,n}u^{*}=\sum_{i,i'=1}^{n}\sum_{C,C'\in F} e_{C,i}[1,i]h_{F,n} e_{C',i'}^{*}[i',1] \stackrel{\eqref{fvsdfvfvsv}}{
=}\epsilon_{{A(\bC),n}}[1]\circ \ell_{C(F,n)}\circ  \phi_{F,n}\, .\qedhere\]
%
% 
%   The inclusion \begin{equation}\label{sfdvfvsvsdfvsdfv}i:\Mat_{n}(A(F))\to \Mat_{n}(A(\bC)) \end{equation}   satisfies $i(-) u^{*}u=i(-)$.
%We thus get a homomorphism
%\begin{equation}\label{sfdvfvsvsdfvsdfv1}i':\Mat_{n}(A(F))\to  \Mat_{n} (A(\bC))\ , \quad  i'(-):=ui(-)u^{*}\ . \end{equation}  Note that $i'$ has a factorization
%\begin{equation}\label{sfdvfvsvsdfvsdfv2}\Mat_{n}(A(F))\stackrel{j}{\to}    \End_{\bC}(C(F,n))\stackrel{\ell}{\to}  A(\bC) \ . \end{equation}
%%where the last morphism is the inclusion of  $ \End_{\bC}(C(F,n))$ into $A(\bC)$ follows by the  left upper corner  inclusion.
%The homomorphisms $i$ and $i'$ induce the same maps
%$$  K_{*}(\Mat_{n}(A(F)))\to K_{*} (A(\bC) )   \ .$$
\end{proof}
 
 \begin{rem}\label{rgboiwjoijvmfldkmvskldfvsfdv}
 In this remark we recall the standard way to present elements in $K_{0}(A)$ for a  $C^{*}$-algebra 
 $A$, see e.g.\ \cite{blackadar}.

Let $A^{+}$ denote the unitalization of $A$.
  If $P,\tilde P$ is a pair of projections in $\Mat_{n}(A^{+})$ such that $P\equiv \tilde P$ modulo $\Mat_{n}(A)$, then  we have a $K$-theory class $[P,\tilde P]$ in $K_{0}(A)$.
 Every class in $K_{0}(A)$ can be represented in this way.

 We let $[P,\tilde P]_{n}$ be the class represented by this pair of projections in $K_{0}(\Mat_{n}(A))$. Then using the isomorphism \eqref{webwpokvopwervwcervwev} we have the equality
\begin{equation}\label{ewvoiuhiuosdfqwfwed}
[P,\tilde P]=\epsilon_{A,n}[1]_{*}^{-1} [P,\tilde P]_{n}\, .
\end{equation}

  If $A$ is unital and $P$ is a projection in $A$, then we get a class $[P]$ in $K_{0}(A)$.

 If $[P,\tilde P]=0$, then after increasing $n$ if necessary, there exists a partial isometry $U$ in $ \Mat_{n}(A^{+})$ such that
 $UU^{*}=P{\oplus Q}$ and $U^{*}U=\tilde P{\oplus Q}$ for some projection $Q$ in $\Mat_{n}(A^{+})$ which is orthogonal to $P$ and $\tilde P$.\hB
 \end{rem}

 Let $\bC$ be in $\Ccat$.
 
 \begin{lem}\label{egoijotrgregewrgwerg}We assume that $\bC$ is   additive.
 \begin{enumerate}
 \item \label{eoijwoegregregwerg}For every class $p$ in $K_{0}(A(\bC))$ there exists an object $C$ and projections $P,\tilde P$ in $\End_{\bC}(C)$ such that
 $\ell_{C,*}([P]-[\tilde P])=p$.
 \item \label{eoijwoegregregwerg1} If $P,\tilde P$ in $\End_{\bC}(C)$ are projections such that $\ell_{C,*}([P]- [\tilde P])=0$, then there exists {an object $C'$, a projection  $Q$ in $\End_{\bC}(C')$, and} a partial isometry
 $U$ in $\End_{\bC}(C{\oplus C'})$ such that $UU^{*}=P{+Q}$ and $U^{*}U=\tilde P{+Q}$.   \end{enumerate}
 \end{lem}
\begin{proof}
Let $p$ be a class in $K_{0}(A(\bC))$. Then there exists an $n$ in $\nat$ and a pair of projections $P',\tilde P'$ in $\Mat_{n}(A(\bC)^{+})$ such that $P'\equiv \tilde P'$ modulo $\Mat_{n}(A(\bC))$ and $p=[P',\tilde P']$.

We first  note that the dense subalgebra $A^{\alg}(\bC)^{+}$  of $A(\bC)^{+}$ is closed under holomorphic function calculus.
Every element of $A^{\alg}(\bC)^{+}$ is contained in $A(F)^{+}$ for a sufficiently large finite set of objects of $\bC$. 
The same applies to $n$-by-$n$ matrices.
We can therefore modify the choices of $P'$ and $\tilde P'$ such that $P',\tilde P'$ belong to  $\Mat_{n}(A(F)^{+})$ for a sufficiently large set $F$ of objects of $\bC$.
We write $[P',\tilde P']_{F}$ for the corresponding class in $K_{0}(\Mat_{n}(A(F)))$.

Since $A(F)$ is unital we have  decompositions
\begin{equation}\label{bojioge0rbgrb}A(F)^{+}\cong A(F)\oplus \C \, , \quad 
\Mat_{n}(A(F)^{+})\cong \Mat_{n}(A(F))\oplus \Mat_{n}(\C)\, .
\end{equation}  If we  
  take the components $P'',\tilde P''$ of the projections $P'$, $\tilde P'$ in $\Mat_{n}(A(F))$, then we have the equality
  $$[P',\tilde P']_{F}=[P'']-[\tilde P'']$$ in $K_{0}( \Mat_{n}(A(F )))$. 

%For any algebra  $A$ we the homomorphism in $K$-theory induced by $\epsilon_{A}[1]$ is an isomorphisms.   
%  \begin{equation}\label{ewflkqwnefqlefewfqwefqwefqwefqwef}
%  \epsilon_{A}[1]_{*}= K(A)\to
%\end{equation} 
%
Using the notation introduced in Lemma \ref{rqgiowergwregwergwergeg} we set $C \coloneqq C(F,n)$,   
 $P \coloneqq \phi_{F,n}(P'')$ and $\tilde P \coloneqq \phi_{F,n}(\tilde P'')$. 
 %Using the notation introduced in Lemma \ref{rqgiowergwregwergwergeg}   w 
 We have the chain of equalities 
\begin{eqnarray*}p&=&[P',\tilde P']\\&\stackrel{\eqref{ewvoiuhiuosdfqwfwed}}{=}& \epsilon_{A(\bC),n}[1]^{-1}_{*} [P',\tilde P']_{n}\\
 &=& \epsilon_{A(\bC),n}[1]_{*}^{-1} h_{F,n,*}[P',\tilde P']_{F}\\&=& \epsilon_{A(\bC),n}[1]_{*}^{-1}h_{F,n,*}([P'']-[\tilde P''])\\&\stackrel{\eqref{vasfvlkmlksaf24fwf}}{=}& \epsilon_{A(\bC),n}[1]_{*}^{-1} h'_{*}([P'']-[\tilde P''])\\&\stackrel{\eqref{urehfuiehisudhciousdcdsa}}{=}&  \ell_{C(F,n),*} ([\phi_{F,n}(P'')]-[\phi_{F,n}(\tilde P'')])\\&\stackrel{ }{=}&\ell_{C,*} ([ P]-[ \tilde P])\, .
\end{eqnarray*}
  
This finishes the verification of Assertion~\ref{eoijwoegregregwerg}.

  We now show the Assertion~\ref{eoijwoegregregwerg1}.   For $n$ in $\nat$ 
  we set $P':= \ell_{C}(P)[1,1]$ and $\tilde P':= \ell_{C}(P)[1,1]$ in $\Mat_{n}(A(\bC)^{+})$.
   By assumption we can  choose $n$, {a projection $Q'$ in $\Mat_{n}(A(\bC)^{+})$ which is orthogonal to $P'$ and $\tilde P'$},    and a partial isometry
   $U'$ in $\Mat_{n}(A(\bC)^{+})$ such that $U'U^{\prime,*}=P'{\oplus Q'}$ and $ U^{\prime,*}U'=\tilde P'{\oplus Q'}$.
 
 {As in the argument for the first part we can assume that $Q'$ belongs to 
 $A(F)^{+}$ for some finite set $F$ of objects of $\bC$ containing $C$. 
 We can then  replace $U'$ by  $ (P'+Q')U' (\tilde P'+Q')$ in $A(F)^{+}$.
 Using the decompositions 
 \eqref{bojioge0rbgrb} and taking the $\Mat_{n}(A(F))$-component  we can further assume that
 $Q'$ and $U'$  belong to $\Mat_{n}(A(F))$.}
 
 %Note that $P'$ and $\tilde P'$ belong to the subalgebra $\Mat_{n}(A(\bC))$.  This implies that $U'$ belongs to 
 %$\Mat_{n}(A(\bC))$. \uli{\bf This part of the argument breaks down}

 %Let $U'':= P'U' \tilde P'$.
 %Then we calculate in a straightforward manner that
%\[
%U''U^{\prime\prime,*}=P'\, , \quad U^{\prime\prime,*}U''=\tilde P'\,.
%\]

{We can decompose  $C(F,n)\cong C\oplus C'$ in a canonical manner. Then   
 $U:=\phi_{F,n}(U') $   is a partial isometry and  $Q:=\phi_{F,n}(Q') 
  $  is a projection in $\End_{\bC}(C\oplus C')$, and we have the relations
$UU^{*}=P+Q$ and $U^{*}U=\tilde P+Q$.
Here we consider $P$ and $\tilde P$ as projections on $C\oplus C'$ acting by zero on the second summand.}
%  $$U''U^{\prime\prime,*}=% P'U'  \tilde P' \tilde P'U^{\prime,*} P'= P'U'   \tilde P'U^{\prime,*} P'=P'U'   U^{\prime,*}U' U^{\prime*} P'=P'P'P' P'=
% P'\ , \quad %$$
% %and similarly, 
% U^{\prime\prime,*}U''=\tilde P'\ .$$
% We furthermore observe that
% $U''= \ell_{C}(U)[1,1]$ for a uniquely determined partial isometry $U$ in $\End_{\bC}(C)$
% which satisfies
% $UU^{*}=P$ and $U^{*}U=\tilde P$.
\end{proof}

Let $A$ be a  unital $C^{*}$-algebra, $U$ be a unitary in $A$,  and $V\colon [0,1]\to \Mat_{n}(A)$  be a Lipschitz continuous path of unitaries from $(U -1_{A})[1,1]+1_{A,n}$ to  $1_{A,n}$.

The following lemma is inspired by \cite[Proof of 12.6.3]{willett_yu_book}.
It improves the Lipschitz constant of the path to $7\pi$ at the cost of increasing the size of matrizes.

\begin{lem}\label{ebiogregrvsvdfvsfdvsfd}
There exists $n'$ in $\nat$ and a $7\pi$-Lipschitz continuous path $V'\colon [0,1]\to \Mat_{n'}(A)$
{of unitaries} from {$(U-1_{A})[1,1]+1_{A,n'}$} to $1_{A,n'}$.
%from $U[1,1]$ to $1_{A,n'}$. 
\end{lem}

\begin{proof}
Assume that $V\colon [0,1]\to \Mat_{n}(A)$  is a Lipschitz continuous path of unitaries from $(U-1_{A})[1,1]+1_{A,n}$ to  $1_{A,n}$ with Lipschitz constant bounded by $C$.
Then we will construct   a new path $V'\colon [0,1]\to \Mat_{3n}(A)$ {of unitaries} with Lipschitz constant  bounded by $\frac{3\pi}{2}+\frac{3C}{4}$ from 
$(U-1_{A})[1,1]+1_{A,3n}$ to $1_{A,3n}$. 
To this end we write
$$(U-1_{A})[1,1]+1_{A,3n}=\left( \begin{array}{ccc}V(0)&0&0\\0&V(1/2)&0\\0&0&V(1)\end{array}\right)\left(\begin{array}{ccc}1&0&0\\0&V(1/2)^{*}&0\\0&0&V(1)^{*}\end{array}\right)\, .$$
We have a path defined on $[0,2/3]$
$$ \left(\begin{array}{ccc}1&0&0\\0&V(1/2-3t/4)^{*}&0\\0&0&V(1-3t/4)^{*}\end{array}\right)
$$
from $$\left(\begin{array}{ccc}1&0&0\\0&V(1/2)^{*}&0\\0&0&V(1)^{*}\end{array}\right) \:\:\mbox{to}\:\:
\left(\begin{array}{ccc}1&0&0\\0&V(0)^{*}&0\\0&0&V(1/2)^{*}\end{array}\right)\, .$$
This path has Lipschitz constant $3/4C$.
We furthermore have a rotation path defined on $[2/3,1]$ of speed $3\pi/2$
from 
$$\left(\begin{array}{ccc}1_{A}&0&0\\0&V(0)^{*}&0\\0&0&V(1/2)^{*}\end{array}\right)
\:\:\mbox{to}\:\: \left(\begin{array}{ccc}V(0)^{*}&0&0\\0&V(1/2)^{*}&0\\0&0&1_{A} \end{array}\right)\, .$$
The product of the concatenation of these paths with $$\left( \begin{array}{ccc}V(0)&0&0\\0&V(1/2)&0\\0&0&V(1_{A})\end{array}\right)$$ is a path
from $(U-1_{A})[1,1]+1_{A,3n}$ to $1_{A,3n}$ with Lipschitz constant bounded by $\frac{3\pi}{2}+\frac{3C}{4}$. The fixed point of the iteration
$$C\Rightarrow \frac{3\pi}{2}+\frac{3C}{4}$$ is  $6\pi$.

By iterating the construction above  sufficiently often we can produce a path as asserted.
\end{proof}

\begin{rem}\label{wetghjiuwogrefregw}
In this remark we recall the standard way to represent elements in $K_{1}(A)$ for a $C^{*}$-algebra $A$, see e.g.\ \cite{blackadar}.
 
  A unitary $U$ in $\Mat_{n}(A^{+})$ with $U\equiv 1_{n}$ modulo $\Mat_{n}(A)$ represents a class
$[U]$ in $K_{1}(A)$. Every class in $K_{1}(A)$ can be represented in this way.  

We let $[U]_{n}$ denote the class of $U$ in $K_{1}(\Mat_{n}(A))$. Then using the isomorphism \eqref{webwpokvopwervwcervwev} we have the equality
\begin{equation}\label{ewvoiuhiuosdfqwfwed1}
[U]=\epsilon_{A,n}[1]_{*}^{-1} [U]_{n}\, .
\end{equation}

If $A$ is unital, then a unitary $U$ as above is of the form $(U'-1_{A,n},1_{n})$ for a unitary $U'$ in $\Mat_{n}(A)$.
If $U'$ is a unitary in $\Mat_{n}(A)$, then we set $[U']:=[(U'-1_{A,n},1_{n})]$.

Assume that $U$ and $\tilde U$ are two such unitaries and that $[U]=[U']$. Then, after increasing $n$ if necessary, there exists 
a path $V\colon [0,1]\to \Mat_{n}(A^{+})$   of unitaries  from $U$ to $\tilde U$ such that $V(t)\equiv 1_{n}$ for all $t$ in $[0,1]$.
 If $A$ is unital, then the path is of the form $V=(V'-1_{A,n},1_{n})$, where $V'$ is a path of unitaries  in $\Mat_{n}(A)$ from $U'$ to $\tilde U'$.
 \hB
\end{rem}

Let $\bC$ be in $\Ccat$.
\begin{lem}\label{tegiotwrgerwgregwgergre}
We assume that $\bC$ is  additive.
\begin{enumerate}
\item \label{ergiojetwregewgg}For every class $u$ in $K_{1}(A(\bC))$ there exists an object $C$ and a unitary $U$ in $\End_{\bC}(C)$ such that
$u=\ell_{C,*}[U]$.
\item  \label{regwoihjoirgerwrgwerggrgwg}Assume that $U$  in $\End_{\bC}(C)$ is a unitary such that $\ell_{C,*}[U]=0$.
Then there exists an object $C'$, an  isometry $u\colon C\to C^{\prime}$, and a $7\pi$-Lischitz path
$V\colon [0,1]\to \End_{\bC}(C')$ from $uUu^{*}+(\id_{C'}-uu^{*})$ to
$ \id_{C'}$.
\end{enumerate}
\end{lem}
\begin{proof}
Let $u$ be a class $u$ in $K_{1}(A(\bC))$. Then there exists $n$ in $\nat$ and a unitary  $U'$  in  $\Mat_{n}(A(\bC)^{+})$
such that  $U'\equiv 1_{n}$ modulo $ \Mat_{n}(A(\bC))$ and $[U']=u$.
 As in the proof of Lemma \ref{egoijotrgregewrgwerg} we can  modify $U'$ such that it belongs to
 $\Mat_{n}(A(F)^{+})$ for a sufficiently large set $F$  of objects of $\bC$. 
 Since $A(F)$ is unital we obtain a unitary $U''$ in $\Mat_{n}(A(F))$ such that $U'=(U''-1_{A(F),n},1_{n})$.
 We let $[U'']$ denote the corresponding class in $K_{1}(\Mat_{n}(A(F)))$.

 We set $C:=C(F,n)$ and   define the unitary $U:=\phi_{F,n} (U'')$ in $\End_{\bC}(C)$, where $C(F,n)$ is as in  \eqref{sdfbsdfgwgdfbvsbvfdvs} and   $\phi_{F,n}$ is as in \eqref{vasvfjor3gervvfsdvfdvsv}.
We   have the following chain of equalities
 \begin{eqnarray*}
 u&=&[U']\\&\stackrel{\eqref{ewvoiuhiuosdfqwfwed1}}{=}&\epsilon_{A(\bC),n}[1]^{-1}_{*}[U']_{n}\\&=&
\epsilon_{A(\bC),n}[1]^{-1}_{*}h_{F,n,*} [U'']\\&\stackrel{\eqref{vasfvlkmlksaf24fwf}}{=}&
\epsilon_{A(\bC),n}[1]^{-1}_{*}h'_{*} [U'']\\&\stackrel{\eqref{urehfuiehisudhciousdcdsa}}{=}&
  \ell_{C(F,n),*}[\phi_{F,n} (U'')]\\&=&
   \ell_{C,*}[U]
 \end{eqnarray*}
This finishes the proof of Assertion \ref{ergiojetwregewgg}.

We now show Assertion \ref{regwoihjoirgerwrgwerggrgwg}.  Since $\ell_{C,*}[U]=0$
  there exists $n$ in $\nat$ and a path of unitaries $V'\colon [0,1]\to \Mat_{n}(A(\bC)^{+})$ from $((U-1_{\bC})[1,1],1_{n})$ to $1_{n}$
  such that $V'(t)\equiv 1_{n}$ for all $t$ in $[0,1]$. We can modify the path such that it takes values in $ \Mat_{n}(A(F)^{+})$ for a sufficiently large set of objects $F$ containing $C$. Since $A(F)$ is unital we can write
  $V':=(V''-1_{A(F),n},1_{n})$ for a path $V''$ of unitaries in $\Mat_{n}(A(F))$ from $(U-\id_{C})[1,1]+1_{A(F),n}$  to $1_{A(F),n}$.
  
  We now apply Lemma \ref{ebiogregrvsvdfvsfdvsfd}. It provides a $7\pi$-Lipschitz path
  $V'''\colon [0,1]\to \Mat_{n'}(A(F))$ of unitaries 
  from    $(U-\id_{C})[1,1]+1_{A(F),n'}$ to $ 1_{A(F),n'}$.
  
  We now consider object $C':=C(F,n')$ (see \eqref{sdfbsdfgwgdfbvsbvfdvs}) and the isometry $u:=e_{C,1}\colon C\to C'$.
We furthermore define the $7\pi$-Lipschitz path  $V:=\phi_{F,n'}(V''')$, where $\phi_{F,n'}$ is as in \eqref{vasvfjor3gervvfsdvfdvsv}.  
This path does the job since
$$V(0)=\phi_{F,n'}(V'''(0))=\phi_{F,n'}((U-\id_{C})[1,1]+1_{A(F),n'})=uUu^{*}+(\id_{C'}-uu^{*})$$
and
\[
V(1)=\phi_{F,n'}(V'''(0))=\phi_{F,n'}(1_{A(F),n'})=\id_{C'}\,.\qedhere
\]
 %It represents a class $[U]$ in $K_{1}(A(\bC))$. We can find a unitary $U'$ close to $U$ such that $U'\equiv 1_{n}$ modulo $\Mat_{n}(A(\bC))$, $[U]=[U']$, and $U'\in \Mat_{n}(A(F)^{+})$
% for a sufficiently 
%large finite set $F$ of objects of $ \bC$. 
%We let $[U']_{F}$ in $K_{1}(\Mat_{n}(A(F)^{+})$  be the corresponding class. Then
%we have
%$$[U]=[U']=i^{+}_{*}[U']_{F}=i^{\prime,+}_{*}[U']_{F}=\ell^{+}_{*}[j^{+}(U')]\ ,$$
%where the superscript $+$ indicates that we have extended the maps to the unitalizations.
%So every class in $ K_{1}(\bA(\bC))$ can be represented by a unitary of the form $(U,1)$ in $\End_{\bC}(C)^{+} 
%$ for   a suitable object  $C$ of $\bC$. Note that $U$ satisfies $U^{*}U+U+U^{*}=0$.
%
 \end{proof}
 
% Assume that $U:[0,1]\to  \Mat_{n}(A(\bC)^{+})$ is a path of unitaries 
% 
% 
% such that $U(t)\equiv 1_{n}$ for all $t$ in $[0,1$ such that the endpoints belong to
% $  \End_{\bC}(C)$ for some object $C$ of $\bC$. Then by an approximation argument and after enlarging the set $F$  there exists a path $U':[0,1]\to  \Mat_{n}(A(F)^{+})$ with the same endpoints and such that  $U'(t)\equiv 1_{n}$ for all $t$ in $[0,1$.
%

Let $(\bC_{i})_{i\in I}$ be a family in $\Ccat$. For every $i$ in $I$ the projection $p_{i}\colon \prod_{i\in I} \bC_{i}\to \bC_{i}$ 
induces a morphism of spectra
\[
\Kcat(p_{i})\colon K\big(\prod_{i\in I} \bC_{i}\big) \to \Kcat(\bC_{i})\, .
\]
\begin{theorem}\label{ojgweorergwerfewrfwerfw}
If $\bC_{i}$ is  additive for every $i$ in $I$, then the morphism   of spectra
\begin{equation}\label{sdfv3rio3uhiufvdfvsdfvsfv}
\Kcat\big(\prod_{i\in I} \bC_{i}\big) \to \prod_{i\in I} \Kcat(\bC_{i})
\end{equation}
induced by the  family  $(K(p_{i}))_{i\in I} $  is an equivalence.
\end{theorem}
\begin{proof}
We consider the diagram
\begin{equation}\label{asdvlnowevasdvadsva1}\xymatrix{A^{f}(\prod_{i\in I} \bC_{i})\ar[r]\ar[d]&\prod_{i\in I} A^{f}(\bC_{i})\ar[d]\\A(\prod_{i\in I} \bC_{i})&\prod_{i\in I}A(\bC_{i})}\end{equation}  in $\nCalg$,
where  left upper horizontal morphism is induced by the family $(A^{f}(p_{i}))_{i\in I}$.
The vertical maps are instances of {\eqref{eq_transformation_Af_A}}  and induce isomorphisms in $K$-theory groups  {by Lemma~\ref{sfdbjoqi3gegergewgerg1}.} Hence we get a square
\begin{equation}\label{asdvlnowevasdvadsva}
\xymatrix{
K_{*}(\prod_{i\in I}\bC_{i})\ar[rr]^{\eqref{sdfv3rio3uhiufvdfvsdfvsfv}}\ar[d]^{\cong}&&\prod_{i\in I}K_{*}(\bC_{i})\ar[d]^{\cong}\\
K_{*}(A^{f}(\prod_{i\in I} \bC_{i}))\ar[r]^{!} \ar[d]^{\cong}&K_{*}(\prod_{i\in I} A^{f}(\bC_{i}))\ar[r]^{!!}&\prod_{i\in I}K_{*}(A^{f}(\bC_{i}))\ar[d]^{\cong}\\
K_{*}(A(\prod_{i\in I} \bC_{i}))\ar[rr]^{?}&& \prod_{i\in I} K_{*}(A(\bC_{i}))
}
\end{equation} 
where the homomorphism marked by $!$ is induced from the horizontal  homomorphism in \eqref{asdvlnowevasdvadsva1}, and the  homomorphism $!!$ is the canonical  comparison homomorphism.
The upper vertical isomorphisms reflect  Definition \ref{ergiowergregrdsvsv}, while the lower vertical isomorphisms are instances of \eqref{sfdbjoqi3gegergewgerg}.  In order to show that \eqref{sdfv3rio3uhiufvdfvsdfvsfv} is an isomorphism it suffices to show 
 that the morphism  $?$ (defined as the up-right-down composition) is an isomorphism.
 In view of Bott periodicity  (Remark \ref{wrthijowrtgergwergwegr}.\ref{wegiojwoegwergrwegwreferf}) is suffices to consider the cases $*=0$ and $*=1$.
 
 In the following argument we will frequently use the following fact. Let $\bC$ be in $\nCcat$ and $C$ be an object of $\bC$. Then we have a commutative triangle
 \begin{equation}\label{vewrvonvkndfsvsfdvfvsvdfv}
\xymatrix{&\End_{\bC}(C)\ar[ld]_{\ell_{C}^{f}}\ar[dr]^{\ell_{C}}&\\ A^{f}(\bC)\ar[rr]^{ \eqref{sfdbjoqi3gegergewgerg} }&&A(\bC)}
\end{equation}
where both {diagonal} morphisms are inclusions of closed subalgebras.

 \textbf{surjectivity of  $?$ in \eqref{asdvlnowevasdvadsva} for $*=0$:}
 
 Let $(p_{i})_{i\in I}$ be a class in $ \prod_{i\in I} K_{0}(A(\bC_{i}))$.
 By Lemma \ref{egoijotrgregewrgwerg}.\ref{eoijwoegregregwerg} 
 for every $i$ in $I$ there exists an object $C_{i}$
 in $\bC_{i}$ and projections  $P_{i},\tilde P_{i}$   in $\End_{\bC_{i}}(C_{i})$ such that
\[p_{i}=\ell_{C_{i},*}([P_{i}]-[\tilde P_{i}])\,.\]
We can form 
    projections $(P_{i})_{i\in I}, (\tilde P_{i})_{i\in I}$  in $\End_{\prod_{i\in I} \bC_{i}}((C_{i})_{i\in I})$.
    Using \eqref{vewrvonvkndfsvsfdvfvsvdfv} we see that  the class
\[\ell_{(C_{i})_{i\in I},*}([(P_{i})_{i\in I}]-[(\tilde P_{i})_{i\in I}])\]
in $K_{0}(A(\prod_{i\in I} \bC_{i}))$ provides a preimage of the class $(p_{i})_{i\in I}$ under {the morphism}~$?$.

 \textbf{injectivity of  $?$ in \eqref{asdvlnowevasdvadsva} for $*=0$:}
   
 We note that the product category $\prod_{i\in I}\bC_{i}$ is again   additive. Indeed, we can form sums componentwise.  %(see Lemma \ref{lem_product_preserves_sums}).
 Let $p$ be a class in $K_{0}(A(\prod_{i\in I}\bC_{i}))$ which is sent to zero by~$?$. By   Lemma~\ref{egoijotrgregewrgwerg}.\ref{eoijwoegregregwerg} 
 there is an object $(C_{i})_{i\in I}$ of $\prod_{i\in I}\bC_{i} $  and projections
 $P,\tilde P$ in $\End_{\prod_{i\in I}\bC_{i}}((C_{i})_{i\in I})$ such that
\[\ell_{(C_{i})_{i\in I},*}([P]-[\tilde P])=p\,.\]
We have $P=(P_{i})_{i\in I}$ and $\tilde P=(\tilde P_{i})_{i\in I}$ for projections $P_{i},\tilde P_{i}$ in $\End_{\bC_{i}}(C_{i})$.  
By assumption on $p$ and  \eqref{vewrvonvkndfsvsfdvfvsvdfv}  for every $i$ in $I$ we have
$\ell_{C_{i},*}([P_{i}]-[\tilde P_{i}])=0$.  By Lemma \ref{egoijotrgregewrgwerg}.\ref{eoijwoegregregwerg1}  for every $i$ in $I$
 {there exists an object $C_{i}'$ in $\bC_{i}$,  a projection $Q_{i}$ orthogonal to $P_{i}$ and $\tilde P_{i}$, and a  partial isometry
 $U_{i}$ in $ \End_{\bC_{i}}(C_{i}\oplus C_{i})$ such that $U_{i}U_{i}^{*}=P_{i}+Q_{i} $ and $U_{i}^{*}U_{i}=\tilde P_{i} +Q_{i}$}.    Then $U:=(U_{i})_{i\in I}$ is a partial isometry in $\End_{\prod_{i\in I}\bC_{i}}((C_{i}{\oplus C_{i}'})_{i\in I})$  such that $UU^{*}=P {+(Q_{i})_{i\in I}}$ and $U^{*}U=\tilde P{+(Q_{i})_{i\in I}}$.  Then $[P]-[\tilde P]=0$ in $K_{0}(\End_{\prod_{i\in I}\bC_{i}}((C_{i}\oplus C_{i}')_{i\in I}))$ and 
  therefore $0=\ell_{(C_{i}{\oplus C_{i}')_{i\in I},*}}{([P]-[\tilde P])=\ell_{(C_{i})_{i\in I},*}([P]-[\tilde P])} =p$.

%  
% 
% Let   $P,\tilde P$ be projections in $\Mat_{n}(A(\prod_{i\in I}\bC_{i})^{+})$ such that $P\equiv \tilde P$ modulo $A(\prod_{i\in I}\bC_{i})$. We assume that the class
% $[P,\tilde P]$   in $K_{0}(A(\prod_{i\in I}\bC_{i}))$   is sent to zero under $?$.
% We note that the product category $\prod_{i\in I}\bC_{i}$ is again finitely additive. Indeed, we can form sums componentwise.
%By our preparatory discussion   we can assume that there exists an object $(C_{i})_{i\in I}$ of $\prod_{i\in I}\bC_{i} $ such that 
% $P,\tilde P$ belong to $\End_{\prod_{i\in I}\bC_{i}}((C_{i})_{i\in I})$. Hence $P=(P_{i})_{i\in I}$ and $\tilde P=(\tilde P_{i})_{i\in I}$ for projections $P_{i},\tilde P_{i}$ in $\End_{\bC_{i}}(C_{i})$.  
% 
    
 \textbf{surjectivity  of  $?$ in \eqref{asdvlnowevasdvadsva} for $*=1$:}
  
Let $(u_{i})_{i\in I}$ be a class in $\prod_{i\in I}K_{1}(A(\bC_{i}))$. 
By Lemma \ref{tegiotwrgerwgregwgergre}.\ref{ergiojetwregewgg} for every $i$ in $I$ there exists an object   $C_{i}$  in $\bC_{i}$ and a unitary
$ U_{i}$ in $ \End_{\bC_{i}}(C_{i}) $ such that  $\ell_{C_{i},*}[U_{i}] =u_{i}$.
The family $(U_{i})_{i\in I}$ is a unitary in $ \End_{\prod_{i\in I}\bC_{i}}((C_{i})_{i\in I})$. 
  Using \eqref{vewrvonvkndfsvsfdvfvsvdfv} we see that the class $\ell_{(C_{i})_{i\in I},*}[(U_{i})_{i\in I}]$ in
$K_{1}(A(\prod_{i\in I}\bC_{i}))$ 
  is the desired preimage of the class 
$(u_{i})_{i\in I}$ under $?$.

 \textbf{injectivity of  $?$ in \eqref{asdvlnowevasdvadsva} for $*=1$:}
  
     Let $u$ be a class in $K_{1}(A(\prod_{i\in I}\bC_{i}))$ which is sent to zero by $?$. By Lemma \ref{tegiotwrgerwgregwgergre}.\ref{ergiojetwregewgg} there is  an object $C \coloneqq (C_{i})_{i\in I}$ in $\prod_{i\in I}\bC_{i}$ and a unitary
$U$ in $\End_{\prod_{i\in I}\bC_{i}}(C)$ such that $\ell_{(C_{i})_{i\in I},*}[U]=u$. 
We have $U=(U_{i})_{i\in I}$ for unitaries $U_{i}$ in $\End_{\bC_{i}}(C_{i})$.
By assumption on $u$ and  \eqref{vewrvonvkndfsvsfdvfvsvdfv} we have $\ell_{C_{i},*}[U_{i}]=0$ for all $i$ in $I$.
By Lemma~\ref{tegiotwrgerwgregwgergre}.\ref{regwoihjoirgerwrgwerggrgwg} for every $i$ we can find an object $C_{i}'$ in $\bC_{i}$, an isometry
$u_{i}\colon C_{i}\to C_{i}'$, and a $7\pi$-Lipschitz path 
$V_{i}\colon [0,1]\to \End_{\bC_{i}}(C'_{i})$ from $u_{i}U_{i}u_{i}^{*}+(\id_{C'_{i}}-u_{i}u_{i}^{*})$ to $\id_{C'_{i}}$.
We define the object  $C' \coloneqq (C_{i}')_{\in I}$ in $\prod_{i\in I}\bC_{i}$ and  the isometry 
$u \coloneqq (u_{i})_{i\in I}\colon C\to C'$ in $\prod_{i\in I}\bC_{i}$.
Then $V \coloneqq (V_{i})_{i\in I}$ is a path in $\End_{\prod_{i\in I}\bC_{i}}(C')$ from
$uUu^{*}+(\id_{C'}-uu^{*})$ to $\id_{C'}$. 
At this point, in order to see that $V$ is continuous one needs the uniform bound on the Lipshitz constants of the paths $V_{i}$.
This shows that $[uUu^{*}+(\id_{C'}-uu^{*})]=0$ in $K_{1}(\End_{\prod_{i\in I}\bC_{i}}(C'))$.
We have $\ell_{C}=\ell_{C}u^{*}u$ in $A(\bC)$ and {the} factorization
\begin{equation}\label{earoihoivaevrvfvfdvvfsvd}
u\ell_{C}u^{*} \colon \End_{\prod_{i\in I}\bC_{i}}(C)\xrightarrow{\phi} \End_{\prod_{i\in I}\bC_{i}}(C')\xrightarrow{\ell_{C'}} A\big(\prod_{i\in I}\bC_{i}\big)\, ,
\end{equation}
where $\phi(-):=u(-) u^{*}$.
Note that these homomorphisms are not unital. To apply these maps to unitaries representing $K$-theory classes we must extend them to the unitalizations. This leads to the formula
$$\phi_{*}[U]=[\phi(U)+(\id_{C'}-\phi(\id_{C}))]=[uUu^{*}+(\id_{C'}-uu^{*})]\, .$$
The homotopy $V$ whitnesses the fact that $\phi_{*}[U]=0$. 
{Finally,} we have
\begin{align*}
u & \quad = \quad \ell_{(C_{i})_{i\in I},*}[U]\\
& \quad \stackrel{\mathclap{\eqref{vasfvlkmlksaf24fwf}}}{=} \quad
u\ell_{(C_{i})_{i\in I}}u^{*} [U]\\
& \quad \stackrel{\mathclap{\eqref{earoihoivaevrvfvfdvvfsvd}}}{=} \quad
\ell_{C',*} \phi_{*} [U]\\
& \quad = \quad 0\, .\qedhere
\end{align*}
%\Alex{finishing the proof.}
\end{proof}
 
   \section{Morita invariance}\label{erogijogergergwgerg9}

In this section we recall  the notion of a Morita equivalence between {unital} $C^{*}$-categories. % and extend it to the non-unital case. 
We %then 
{show} 
that the reduced crossed product preserves Morita equivalences. We then consider Morita invariant homological functors and verify that $\Kcat$ is Morita invariant. % \uli{In the subsquent Sections \ref{eiogjwoiggregwegewrgwegrwerg}  and \ref{eoigwjeoigregewrgwergwerg} we  will discuss generalizations to non-unital $C^{*}$-categories.}

Recall from Definition~\ref{ergiehjioferfqffrf} that $\bE$ in $\Ccat$ is called additive if it admits orthogonal sums for all finite families of objects.
 Let $i\colon \bD\to \bE$  be a morphism in $\Ccat$.
\begin{ddd} \label{rhrhhdgffg}
  The morphism  $i$ presents $\bE$ as the additive completion of $\bD$  if  the following conditions are satisfied:
  \begin{enumerate}
  \item  The morphism  $i $ is  fully faithful.
  \item The $C^{*}$-category  $\bE$ is additive.
  \item 
Every object of $\bE$ is  
unitarily isomorphic to a finite orthogonal sum of objects in the image of $i$.
\end{enumerate}
 \end{ddd}

If $i:\bD\to \bE$ and $i':\bD\to \bE'$ present $\bE$ and $\bE'$ as additive completions of $\bD$, then there exists a unitary equivalence $\bE\to \bE'$ such that
$$\xymatrix{&\bD\ar[dr]^{i'}\ar[dl]_{i}&\\ \bE\ar[rr]&&\bE'}$$
commutes up to a unitary natural transformation.

\begin{ex}\label{ex_emptyset_Morita}
If $X$ is a set, then the functor $\emptyset\to 0[X]$ presents $0[X]$ as an additive completion of $\emptyset$.
\hB
\end{ex}

Let $\Ccat_{\oplus}$ be the full subcategory of $\Ccat$ of additive $C^{*}$-categories.
Then there exists a functor and a natural transformation 
$$(-)_{\oplus}\colon \Ccat\to \Ccat_{\oplus}\, , \quad \id\to (-)_{\oplus}\, ,$$
such that for every $\bC$ in $\Ccat$ the morphism $\bC\to \bC_{\oplus}$ presents $\bC_{\oplus}$ as the additive completion of $\bC$, see   \cite[Sec.~2]{davis_lueck} or \cite[Defn.~2.8]{MR3123758}. 
Observe that in this model of the additive completion functor  the transformation $\bC\to \bC_{\oplus}$ is injective on objects. 
 
 \begin{rem}
If one {passes} to $\infty$-categories, then this additive completion functor fits into an adjunction. In greater detail, as in   \cite{startcats} we consider  the Dwyer--Kan localization 
$\Ccat_{\infty}$  of $\Ccat$ at the set of   unitary equivalences.
  % Let $W$ be the set of unitary equivalences in $\Ccat$. 
  Then $(-)_{\oplus}$ descends to the left-adjoint of an adjunction (see \cite[Lem.~2.12]{MR3123758} for a $2$-categorial formulation)
 $$(-)_{\oplus}\colon \Ccat_{{\infty}}%[W^{-1}]
 \leftrightarrows \Ccat_{{\infty},\oplus} :\incl\, ,$$
where $\Ccat_{{\infty},\oplus}$ is the full subcategory of $ \Ccat_{{\infty}}$ of additive $C^{*}$-categories.
  The details can be understood similarly as in the case of additive categories \cite[Cor.~2.62]{Bunke:2018aa}, using a Bousfield localization of model category structures as constructed in  \cite{MR3123758}. %   \cite{startcats}. %or \cite{startcats}. 
%See also \cite{startcats} and Remark \ref{qwrghiefgfegsgsfgfsgsdfg}.
\hB
\end{rem}

Recall from Definition \ref{regiuhjigwergwgrewgrg} that $\bE$ in $\Ccat$ is idempotent complete if every projection in $\bE$ is effective. 
We again consider a  morphism $i\colon \bD\to \bE$ in $\Ccat$.
\begin{ddd}\label{defn_idempotent_complete}
The morphism $i$ presents $\bE$ as the idempotent completion of {$\bD$} if  the following conditions are satisfied:
\begin{enumerate}
\item The functor $i$ is fully faithful.
\item The $C^{*}$-category $\bE$ is idempotent complete.
\item\label{defn_idempotent_complete_isometry} For every object $E$ in  $\bE$  there is some object $D$  in $ \bD$ and an isometry $u\colon E\to i(D)$.
\end{enumerate}
\end{ddd}

If $i:\bD\to \bE$ and $i':\bD\to \bE'$ present $\bE$ and $\bE'$ as idempotent completions of $\bD$, then there exists a unitary equivalence $\bE\to \bE'$ such that
$$\xymatrix{&\bD\ar[dr]^{i'}\ar[dl]_{i}&\\ \bE\ar[rr]&&\bE'}$$
commutes up to a unitary natural transformation.

%For simplicity we consider the idempotent completion only for additive $C^{*}$-categories.
Let $\Ccat^{\Idem}$ denote the full subcategory of $\Ccat $ of idempotent complete  $C^{*}$-categories. 
There exists a functor  and a natural transformation 
$$\Idem\colon \Ccat \to \Ccat^{\Idem}\, , \quad \id\to \Idem\, ,$$
such that for every $\bC$ in $\Ccat$ the morphism $\bC\to \Idem(\bC)$ presents $\Idem(\bC)$ as the idempotent completion of $\bC$.   

\begin{construction}\label{wijgowegwergf}{\em  
In this paper  we will  work with the 
  explicit model of the idempotent completion functor  described in \cite[Defn.~2.15]{MR3123758}. Let $\bC$ be in $\Ccat $. Then $\bC\to \Idem(\bC)$  is given as follows: \begin{enumerate}
\item objects: The objects of  $\Idem(\bC)$ are pairs $(C,p)$ of an object  $C$ of $\bC$ and a projection $p$ in $\End_{\bC}(C)$.
\item morphisms: The morphisms $A:(C,p)\to (C',p')$ in $\Idem(\bC)$ are
morphisms $A:C\to C'$ satisfying $A=p'A=Ap$. 
\item composition and involution: These structures are inherited from $\bC$.
\item  canonical morphism: $\bC\to \Idem(\bC)$ sends $C$ in $\bC$ to $(C,\id_{C})$ in $\Idem(\bC)$ and $A:C\to C'$ to $A:(C,\id_{C})\to (C',\id_{C'})$. 
\end{enumerate}
Observe that in this model $\bC\to \Idem(\bC)$ is injective on objects.}\hB
\end{construction}

  %This fact is already expressed in the notation used above since we wrote that the functor $\Idem$ maps to $\Ccat_{\oplus}^{\Idem}$. 
Let $\Ccat_{\oplus}^{\Idem}$ denote the full subcategory of $\Ccat_{\oplus}$ of idempotent complete and additive $C^{*}$-categories. 
By \cite[Rem.~2.19]{MR3123758} the idempotent completion of an additive $C^*$-category is again additive. The idempotent completion functor therefore restricts to a functor 
$$\Idem:\Ccat_{\oplus}\to \Ccat_{\oplus}^{\Idem}\ .$$
In the same remark \cite[Rem.~2.19]{MR3123758} it is explained that the operations of forming additive completions and of idempotent completions do not commute since the additive completion of an idempotent complete $C^*$-category  may fail  to be idempotent complete.

 \begin{rem}
The idempotent completion functor descends to an adjunction between $\infty$-categories (see \cite[Defn.~2.17]{MR3123758} for a $2$-categorical formulation)
$$\Idem: \Ccat_{{\infty},\oplus}\leftrightarrows \Ccat_{{\infty},\oplus}^{\Idem}:\incl\, ,$$
 where $\Ccat_{{\infty},\oplus}^{\Idem}$ is the full subcategory of $\Ccat_{{\infty},\oplus}$ of idempotent complete {(and  additive)} $C^{*}$-categories.
The details are again  similar to  the case of additive categories \cite[Cor.~3.7]{Bunke:2018aa}, again using a Bousfield localization of  model category structures constructed in  \cite{MR3123758}. %or \cite{startcats}.
\hB
\end{rem}

By composing the additive and idempotent completion functors  and the corresponding natural transformations we obtain a functor and a natural transformation
\begin{equation}\label{qwrf98u89euf89qweuf98ewfeqwfqe}
(-)^{\sharp}:=\Idem\circ (-)_{\oplus}:\Ccat\to \Ccat_{\oplus}^{\Idem}\, , \quad \id\to (-)^{\sharp}\, .
\end{equation} 
For every $\bC$ in $\Ccat$ the morphism $\bC\to \bC^{\sharp}$  fully faithful.    Furthermore, if $\bC$ is additive and idempotent complete, then the morphism
 $\bC\to \bC^{\sharp}$ is a unitary equivalence. This in particular applies to 
$\bC^{\sharp}\to (\bC^{\sharp})^{\sharp}$. 
Using the explicit models of the additive and idempotent completion
functors explained above we can arrange that $\bC\to \bC^{\sharp}$ is injective on objects.

% (the evaluation of the transformation in \eqref{qwrf98u89euf89qweuf98ewfeqwfqe} at $\bC^{\sharp}$) is a unitary equivalence since
%$\bC^{\sharp} $ is already additive and idempotent complete.

\begin{ddd}[{\cite[Defn.~4.4]{MR3123758}}]\label{wetgoijwtoigwerrefwerfwefwref}
We define the set  $W_{\Morita}$ of Morita equivalences to be the set of {morphisms} in $\Ccat$ which are sent to unitary  equivalences by $(-)^{\sharp}$.  \end{ddd}

 Unitary equivalences are Morita equivalences. For every $\bC$ in $\Ccat$ the
  canonical morphism $\bC\to \bC^{\sharp}$ is a Morita equivalence since
$\bC^{\sharp}\to (\bC^{\sharp})^{\sharp}$ is a unitary equivalence as noted above.  
For a similar reason for every $\bC$ in $\Ccat$ also  $\bC\to \bC_{\oplus}$ is a Morita equivalence. 

Furthermore,  for $\bC$ in $\Ccat$   we claim that  $ \bC\to \Idem(\bC)$  is a  Morita equivalence.
    In order to see  this 
  we first apply $(-)^{\sharp}=\Idem\circ (-)_{\oplus}$ to $\bC\to \Idem(\bC)\to  \Idem(\bC_{\oplus}) $ in order to get 
$$\bC^{\sharp}\simeq 
\Idem(\bC_{\oplus})\to  \Idem(\Idem(\bC)_{\oplus})\to  \Idem(\Idem(\bC_{\oplus})_{\oplus})\stackrel{\simeq}{\leftarrow}  \Idem(\bC_{\oplus})\simeq \bC^{\sharp}\ .$$
For the inverted unitary equivalence we use that $\Idem$ preserves additivity.
 To show the claim we must show that the first arrow is a unitary equivalence.
We know that the composition of the two arrows is a unitary equivalence. 
The second arrow is also a unitary equivalence because it is  fully faithful and also essentially surjective since
$\bC_{\oplus}$ is contained in $\Idem(\bC)_{\oplus}$. 
Therefore the first arrow is  a unitary equivalence as desired.

A  Morita equivalence  $\bC\to \bD$  is fully faithful.  In order to see this we form    the  commutative square
$$\xymatrix{\bC\ar[r]\ar[d]^{\eqref{qwrf98u89euf89qweuf98ewfeqwfqe}}&\bD\ar[d]^{\eqref{qwrf98u89euf89qweuf98ewfeqwfqe}}\\\bC^{\sharp}\ar[r]^{\simeq}&\bD^{\sharp}}$$
Since the vertical morphisms are fully faithful  we conclude that $\bC\to \bD$ is fully faithful, too.
 
\begin{rem}
We  consider   the Dwyer--Kan localization \begin{equation}\label{qewfojfiojqwoefqwefqfeqef}
\ell_{\Morita}\colon \Ccat\to \Ccat[W_{\Morita}^{-1}]
\end{equation}
of $\Ccat$ at the Morita equivalences.
The $\infty$-category $\Ccat[W_{\Morita}^{-1}]$ can be modeled by a cofibrantly generated simplicial model category structure on $\Ccat$ \cite[Thm.~4.9]{MR3123758}. There is a Bousfield localization
 $$ L_{\Morita}\colon \Ccat_{{\infty}}\leftrightarrows     \Ccat[W_{\Morita}^{-1}]\, .$$
  %\Alex{where $W$ is the set of unitary equivalences in $\Ccat$ and we of course have $W \subset W_{\Morita}$.}
 \hB\end{rem}
 
%We now extend the notion of a Morita equivalence to the possibly non-unital case.
%Let $\phi:\bC\to \bD$ be a morphism in $ \nCcat$.
%\begin{ddd} The morphism $\phi$ is a Morita equivalence if it satisfies:
%\begin{enumerate}
%\item \label{rgrgdfgdgwegfdsgsdgdfg} $\phi$ is fully faithful.
%\item $\bM\phi:\bM\bC\to \bM\bD$ is a Morita equivalence in the sense of Definition 
%\ref{wetgoijwtoigwerrefwerfwefwref} (note that $\bM\phi$ is defined under the Assumption \ref{rgrgdfgdgwegfdsgsdgdfg} by Proposition \ref{stkghosgfgsrgsegs}.
%\end{enumerate}
%\end{ddd}
% If $\phi$ is a morphism in $\Ccat$, then this Definition reproduces the old notion.
% 
%\begin{rem}In analogy to Remark \ref{retjheirotjhoergrtgertgetrgterg}.\ref{etherhthgetgetrgetget}
%  we  the following  equivalent characterizations  of a Morita equivalence. A morphism $\phi:\bC\to \bD$  in $\nCcat$.
%  is a Morita equivalence if and only if it is fully faithful and part of a square
%  $$\xymatrix{\bC\ar[d]\ar[r]&\ar[d]^{\psi}\bE\\\bD\ar[r]&\bF}$$
%whose horizontal morphisms are ideal inclusions and
%$\psi$ is a Morita equivalence.
%Indeed, if $\phi$ is a Morita equivalence, then we can take $\psi:=\bM\phi:\bM\bC\to \bM\bD$. Vice versa, let such square be given. Then we get the following diagram
%$$\xymatrix{\bE^{\sharp}\ar@/^-2cm/[ddddd]^{\psi^{\sharp}}_{\simeq}\ar[rr]&&\ar@/^2cm/[ddddd]^{\bM\phi^{\sharp}}_{\simeq}\bM\bC^{\sharp}\\ \bE\ar[u]\ar[ddd]^{\psi}\ar[rr]&&\ar[u]\bM\bC\ar[ddd]^{\bM\phi}\\&\bC\ar[ur]\ar[ul]\ar[d]&\\&\bD\ar[dr]\ar[dl]&\\ \bF\ar[d]\ar[rr]&&\ar[d]\bM\bD\\ \bF^{\sharp}\ar[rr]&&\bM\bD^{\sharp}}$$
%  \end{rem}

\begin{ex}\label{qergoijwergioejrgwergwregwergre}
Let $ A$ be a  very small unital $C^{*}$-algebra. % with a $G$-action by automorphism. 
We can then consider the $C^{*}$-category of very small Hilbert $A$-modules
$ \Hilb (  A)$ 
  %\uli{with strict $G$-action 
  {explained in} Example \ref{qregiuheqrigegwegergwegergw}. 
It contains the subcategory $ \Hilb  (A)^{\fg,\proj}$ of finitely generated, projective Hilbert $A$-modules, and we may consider the object $ A$ in  $ \Hilb ( A)^{\fg,\proj}$ as a $C^{*}$-category with a single object. The inclusion
$ A \to   \Hilb ( A)^{\fg,\proj}$ is a Morita equivalence{. In} order to see this we consider the chain
$$ A \to \Hilb( A)^{\fg,\free}\to   \Hilb(A)^{\fg,\proj}\, .$$
The first functor presents  $\Hilb ( A)^{\fg,\free}$ as the additive completion of $ A$, and the second functor  presents 
$\Hilb ( A)^{\fg,\proj}$ as the idempotent completion of $\Hilb ( A)^{\fg,\free}$.
\hB
\end{ex}

Our next goal is to show that the reduced {and the maximal crossed product functors} preserve Morita equivalences.
Let $  \bD\to   \bE$ be a morphism in $\Fun(BG,\Ccat)$.
It is called a Morita equivalence if the induced morphism    {between} the underlying $C^{*}$-categories is a Morita equivalence.

Let $\bC$ be in $\Fun(BG,\Ccat)$.
\begin{lem} \label{thjrteohrehergtrg}If  %(the underlying $C^{*}$-category of) 
$\bC$    is additive, then  $\bC\rtimes_{r}G$ and $\bC\rtimes G$ are    additive.  \end{lem} 
\begin{proof}
{We give the argument for the reduced crossed product. The case of the maximal crossed products is analogous.}
We consider   a finite family
 $(C_{i})_{i\in I}$    of objects in {$ {\bC}\rtimes_{r}G$}.  In view of the  equality $\Ob(\bC)= \Ob(\bC\rtimes_{r}G)$ and the assumption on $\bC$ 
we can   choose a   representative $(C,(e_{i})_{i\in I})$    of its orthogonal sum in $\bC$.  Then $(C,({(e_{i},e)})_{i\in I})$ {(see \eqref{qewfiuhuiqhefiuqwefqewfeqwfwefqwf}} for the notation for morphisms in crossed products) presents the orthogonal sum of the family $(C_{i})_{i\in I}$    in $   \bC\rtimes_{r}G$.    We call this representative a standard representative.  %This shows that claim. 
 \end{proof}

Note that the reduced crossed product preserves faithful morphisms and fully faithful morphisms by Theorem \ref{ejgwoierferfewrferfwe}. In contrast, the maximal crossed product in general does not preserve faithful morphisms.
It is classically well known that it preserves ideal inclusions, but  it does not preserve general injections between $C^{*}$-algebras.  But it does preserve surjections and corner inclusions between $C^{*}$-algebras. On the level of $C^{*}$-categories this generalizes to the following statement.

\begin{prop}\label{okgpweerfrfrfwr34er3r3r}
The maximal crossed product functor $-\rtimes G:\Fun(BG,\nCcat)\to \nCcat$ preserves fully faithful morphisms.
\end{prop}
 %Note that  the statement is weaker than in the case of the reduced product. The latter  does not only preserve fully faithful morphisms but also faithful ones.
 \begin{proof}
Let $\phi:\bD\to \bE$ be a fully faithful morphism in $\Fun(BG,\nCcat)$.
Since $\bD\rtimes^{\alg} G\to \bE\rtimes^{\alg}G$ is fully faithful,
the morphism $\bD\rtimes G\to \bE\rtimes G$ has clearly a dense image in the full subcategory generated by the objects of the image. It suffices to show that this morphism is isometric.

Assume first that $\phi$ is injective on the level of objects. 
Then we can form the  commutative diagram $$\xymatrix{\bD\rtimes G\ar[r]\ar[d]&\bE\rtimes G\ar[d]\\A(\bD\rtimes G) \ar[r]&A(\bE\rtimes G)\\A(\bD)\rtimes G\ar[r]\ar[u]^{\cong}&A(\bE)\rtimes G\ar[u]^{\cong}}$$
The lower vertical morphisms are isomorphisms by \cite[Thm. 6.9]{crosscat}.
The upper vertical morphisms are isometric by \cite[Lem. 6.7]{crosscat}.
In order to show that the upper horizontal morphism between $C^{*}$-categories is  isometric it suffices 
to show that the lower horizontal homomorphism of $C^{*}$-algebras is  isometric.

We will use the fact that $-\rtimes G$ preserves corner inclusions of $C^{*}$-algebras.
We have inclusions of $G$-$C^{*}$-algebras $A(\bE)\to  A(\bM\bE)\to  MA(\bE)$. In  
$MA(\bE)$ we have the invariant  projection
$P:=\sum_{D\in \bD} \id_{D}$, where the sum converges strictly  and $\id_{D}$ belongs to $ \End_{\bM\bD}(D)$ considered as a subalgebra of $MA(\bD)$ in the canonical way. 
The image of $A(\phi):A(\bD)\to A(\bE)$ is the corner
given by $P$. Since $P$ is $G$-invariant it induces a projection $(P,e)$ in $A(\bE)\rtimes G$. The latter determines a corner of $A(\bE)\rtimes G$ which is isomorphic to $A(\bD)\rtimes G$.
The inclusion of a corner is isometric.

We now remove the assumption that $\phi$ is injective on objects.
We can then find a $G$-$C^{*}$-category $\bC$ and a factorization of $\phi$ through fully faithful equivariant morphisms
 $$ \xymatrix{&\bC &\\\bD\ar[ur]^{\psi}\ar[rr]^{\phi}&&\bE\ar[ul]_{\kappa}}$$
 where $\psi$ and $\kappa$ are  injective on objects and $\kappa$ is a unitary equivalence (see the construction given at the end of the proof of Lemma \ref{qwriufhqeifqfewfffqwfe} below).
 Then in  $$\xymatrix{&\bC\rtimes G &\\\bD\rtimes G\ar[ur]^{\psi\rtimes G}\ar[rr]^{\phi\rtimes G}&&\bE\rtimes G\ar[ul]_{\kappa\rtimes G}}$$ we know that $\psi\rtimes G$ is fully faithful and $\kappa\rtimes G$ is a unitary equivalence. 
 This implies that $\phi\rtimes G$ is also fully faithful.
\end{proof}

We now consider a  morphism  
  $ \phi: \bD\to  \bE $ in $\Fun(BG,\Ccat)$.

\begin{prop}\label{werijguhwerigvwergwer9}
If  $  \phi:\bD\to  \bE$ is a Morita equivalence, then $   \phi\rtimes_{r}G:\bD\rtimes_{r}G\to   \bE\rtimes_{r}G$
{and $\phi\rtimes G:\bD\rtimes G\to \bE\rtimes G$ are 
  Morita equivalences.}
\end{prop}
\begin{proof}
 {We present the argument for the reduced crossed product. For the maximal one it is the same   using Proposition \ref{okgpweerfrfrfwr34er3r3r} instead of the last assertion of Theorem \ref{ejgwoierferfewrferfwe}.}
 
 We  first   show that    the morphism $   \bD\rtimes_{r}G\to
 \bD_{\oplus}\rtimes_{r}G$ presents the additive completion of $\bD\rtimes_{r}G$. To this end we verify the conditions listed in Definition \ref{rhrhhdgffg}.
Since $\bD\to \bD_{\oplus}$ is fully faithful we conclude from Theorem \ref{ejgwoierferfewrferfwe} that $\bD\rtimes_{r}G\to
 \bD_{\oplus}\rtimes_{r}G$ is fully faithful. By Lemma \ref{thjrteohrehergtrg}  we know that $\bD_{\oplus}\rtimes_{r}G$ is additive. Finally, argueing similarly as in the proof of      Lemma \ref{thjrteohrehergtrg}  
  we see that every object of $\bD_{\oplus}\rtimes_{r}G$ is  unitarily isomorphic to a finite orthogonal sum of objects of $\bD\rtimes_{r}G$.

 We now form the  commutative diagram
$$\xymatrix{\bD\rtimes_{r}G\ar[r]^{!}\ar[d]^{\phi\rtimes_{r}G}&\bD_{\oplus}\rtimes_{r}G\ar[d]^{\phi_{\oplus}\rtimes_{r}G}\ar[r]&\Idem(\bD_{\oplus}\rtimes_{r}G)\ar[d]_{!!}^{ \Idem(\phi_{\oplus}\rtimes_{r}G)} \\\bE\rtimes_{r}G\ar[r]^{!}&\bE_{\oplus}\rtimes_{r}G\ar[r]&\Idem(\bE_{\oplus}\rtimes_{r}G)}$$
 Since the morphisms marked by $!$ present additive completions, the horizontal
compositions are instances of the transformation \eqref{qwrf98u89euf89qweuf98ewfeqwfqe}. 
We must show that the morphism marked by $!!$   is a unitary equivalence. 
First of all, since the horizontal morphisms and the left vertical morphism are fully faithful, the morphism $!!$ is also fully faithful.
 It remains to show that it is essentially surjective. 
 
 We use the explicit model of the functor $\Idem$ described above.
 Let $(E,p)$ be an object of $ \Idem(\bE_{\oplus}\rtimes_{r}G)$. 
 Since $\bD\to \bE$
is a Morita equivalence  
 there exists a finite family $(D_{i})_{i\in I}$ of objects
   in $\bD$, an orthogonal  sum $(D,(e_{i})_{i\in I})$ of this family in $\bD_{\oplus}$,   and an isometry $u: E\to \phi_{\oplus}(D)$.
 Then  we have the  unitary isomorphism $$(u,e)p:(E,p)\to   \Idem(\phi_{\oplus}\rtimes_{r}G) (D, (\phi_{\oplus}\rtimes_{r}G)^{-1}[(u,e)p(u,e)^{*}]) $$  
in $  \Idem(\bE_{\oplus}\rtimes_{r}G)$, where we use that $\phi_{\oplus}\rtimes_{r}G$ is fully faithful in order to define its inverse. Hence $(E,p)$  belongs to the essential image of $  \Idem(\phi_{\oplus}\rtimes_{r}G)$. \end{proof}

We finally study Morita invariant  functors. 
Let $\Homol\colon \Ccat\to \bS$ be a functor with values in some $\infty$-category.
\begin{ddd}\label{ewrgiuhwerogwregrefwferfwrf}
$\Homol$ is Morita invariant if it sends  Morita equivalences to equivalences.
\end{ddd}
More generally, if $\Homol:\nCcat\to \bS$ is a functor, then we call it Morita invariant if its restriction to $\Ccat$ is so.
%In general, given a generic  morphism in $\Ccat$ it is not obvious how to check the condition of Definition \ref{ewrgiuhwerogwregrefwferfwrf} directly. Therefore 
The following characterization of Morita invariance turns out to be very useful, e.g.,  to verify that $\Kcat$ is Morita invariant in the proof of Theorem \ref{wtoijgwergergrewfwergrg}.

Let  $\Homol\colon \Ccat\to \bS$ be a functor.
  \begin{lem}\label{rgoijeroigwergergrgwg}
  \mbox{}
  The following assertions are equivalent:
  \begin{enumerate}
  \item\label{weoigjowergefvfvsvsdv1} $\Homol$ is Morita invariant.
  \item\label{weoigjowergefvfvsvsdv} $\Homol$ sends the following {morphisms}  in $\Ccat$ to equivalences:
  \begin{enumerate}
  \item Unitary equivalences.
  \item  \label{fbiojefvbfdvsfdvfv}
%The localization $\ell_{Morita}:\Ccat\to \Ccat[W_{Morita}^{-1}]$ is generated by the following sets of functors $i:\bD\to \bE$  in $\Ccat$:
%\begin{enumerate}
%\item unitary equivalences
%\item functors  which present $\bE$ as the additive completion of $\bD$,
% \item functors   which present $\bE$ as the idempotent completion of $\bD$.
 %\end{enumerate}
 % The functor $\Homol$ is Morita invariant if and only if it sends  these generators to equivalences. 
  
  Fully faithful {morphisms} $i\colon\bD\to \bE$  satisfying:
 \begin{enumerate}
 \item  \label{wtoigjoergfregewgre90} $i$ is injective on objects.
 \item \label{wtoigjoergfregewgre9}$\bE$ is additive and idempotent complete.
 \item \label{wtoigjoergfregewgre91} $i$ presents $\bE$ as the additive and idempotent completion of $\bD$.

 \end{enumerate}
 \end{enumerate}
 \end{enumerate}
 \end{lem}
 Note that Condition \ref{wtoigjoergfregewgre91} means that for every object $E$ in $\bE$ there is a finite family $(D_{k})_{k\in K}$ of objects in $D$ and an isometry
$E\to \bigoplus_{k\in K} i(D_{k})$. 
\begin{proof}\mbox{}
$\eqref{weoigjowergefvfvsvsdv1}\Rightarrow \eqref{weoigjowergefvfvsvsdv}$: Unitary equivalences and functors $i$ as in \ref{fbiojefvbfdvsfdvfv} are Morita equivalences. If $\Homol$ is Morita invariant, then it sends these functors to equivalences.
  
 $\eqref{weoigjowergefvfvsvsdv}\Rightarrow \eqref{weoigjowergefvfvsvsdv1}$: 
Let $\bD\to \bE$ be a Morita equivalence. We must show that $\Homol(\bD)\to \Homol(\bE)$ is  an equivalence.

%We have a commuting square
%$$\xymatrix{\bD\ar[r]\ar[d]&\bE\ar[d]\\\bD^{\sharp}\ar[r]^{\simeq}&\bE^{\sharp}}\ .$$
%Since the vertical morphisms are fully faithful  we conclude that $\bD\to \bE$ is fully faithful.
{Since Morita equivalences are fully faithful we   have} a factorization $\bD\to \bD'\to \bE$, where $\bD'$ is {the} full subcategory of $\bE$ given by  the image of the {morphism} $\bD\to \bE$. Then $\bD\to \bD'$ is a unitary equivalence
 and $\bD'\to \bE$ a Morita equivalence.  Since $\Homol(\bD)\to \Homol(\bD')$ is an equivalence it remains to show that
 $\Homol(\bD')\to \Homol(\bE)$ is an equivalence.  
 To this end we consider the commutative square 
 $$\xymatrix{  \bD'\ar[r]\ar[d]& \bE\ar[d]\\ \bD^{\prime ,\sharp}\ar[r] & \bE^{\sharp}}$$ where we arrange  that the vertical morphisms are injective on objects.   
 They satisfy the conditions in \ref{fbiojefvbfdvsfdvfv}.  The lower horizontal morphism is a unitary equivalence.
We then apply $\Homol$ and get the commutative square
 $$\xymatrix{\Homol(\bD')\ar[r]\ar[d]^{\simeq}&\Homol(\bE)\ar[d]^{\simeq}\\\Homol(\bD^{\prime ,\sharp})\ar[r]^{\simeq}&\Homol(\bE^{\sharp})}$$
 %The {morphisms} $\bD'\to  \bD^{\prime,\sharp}$ and $\bE\to \bE^{\sharp}$ satisfy the conditions in \ref{fbiojefvbfdvsfdvfv}. 
 where the indicated equivalences follow from the 
  assumptions on $\Homol$. We conclude that the upper horizontal morphism is an equivalence. 
\end{proof}

%    Equivalently,   
%  the  functor $\Homol$ is Morita invariant if and only if it has a factorization 
%$$\xymatrix{\Ccat\ar[rr]^{\Homol}\ar[dr]^{\ell_{Morita}}&& \bM\\& \Ccat[W_{Morita}^{-1}]\ar[ru]^{\tilde \Homol}&}\ .$$
%\hB

%Let $\Homol\colon \Ccat \to \bM$ be a functor with values in \uli{some $\infty$-category and $X$ be any set.}  % and with $\Homol(0[X]) \simeq 0_{\bM}$ for any small set $X$.
%\begin{kor}
% If $\Homol$ is Morita invariant, \uli{then the morphism $\emptyset\to 0[X]$ induces an equivalence 
% $\Homol(\emptyset) \stackrel{\simeq}{\to} \Homol(0[X])$.}
%\end{kor} 
%\begin{proof}
%\Alex{By Example~\ref{ex_emptyset_Morita} the functor $\emptyset \to 0[X]$ for any set $X$ is an additive completion and hence a Morita equivalence. The claim follows since $\Homol(0[X]) \simeq 0_{\bM}$ by assumption.}
%\end{proof}
%

Let $\Homol \colon \nCcat \to \bS$ be a homological functor.
By definition it sends zero categories to zero, and it preserves some finite  products by Lemma \ref{erighi9wrteogwergrgregwergwreg}. But it is not clear that it preserves finite coproducts. %, or how to characterize  the value $H(\emptyset)$.  
Morita invariance improves the situation.
% with values in a cocomplete $\infty$-category.
\begin{lem}\label{igjotgergweferfwef}\label{lem_Morita_preserves_finite_coprods}
If $\Homol$ is {a Morita invariant homological functor}, then it preserves finite coproducts.
\end{lem}
\begin{proof} 
Let $(C_i)_{i \in I}$ be a finite family in $\Ccat$.
We must show that the canonical map
\begin{equation}\label{rhqifhjifqwefqwefwefq}
\coprod_{i\in I} \Homol ( \bC_{i})  \to \Homol \big( \coprod_{i\in I} \bC_{i} \big) 
\end{equation}
is an equivalence.

We consider the Dwyer--Kan localization $\ell_{\Morita}\colon \Ccat\to \Ccat[W_{\Morita}^{-1}]$
from \eqref{qewfojfiojqwoefqwefqfeqef}.
 By the universal property of the Dwyer--Kan localization the functor $\Homol$  has an essentially unique factorization
\begin{equation}\label{weoigjwegerregwerg}
\xymatrix{\Ccat\ar[rr]^-{\Homol}\ar[dr]_-{\ell_{\Morita}}&&\bS\,.\\
&\Ccat[W_{\Morita}^{-1}]\ar[ur]_-{\quad \Homol_{\Morita}}&}
\end{equation} 
 We use the factorization
  \eqref{weoigjwegerregwerg}
in order to factorize \eqref{rhqifhjifqwefqwefwefq} as 
\begin{equation}\label{qwefoijoqiwefqewfewfqwfe}
\coprod_{i\in I} \Homol_{\Morita} (\ell_{\Morita}( \bC_{i}) ) \to \Homol_{\Morita} \big(
\coprod_{i\in I}\ell_{\Morita}( \bC_{i}) \big)  \to   \Homol_{\Morita} \big(
 \ell_{\Morita} \big( \coprod_{i\in I}\bC_{i} \big) \big)\, .
\end{equation} 
The $\infty$-category $\Ccat[W_{\Morita}^{-1}]$ can be modeled by a cofibrantly generated simplicial model category structure on $\Ccat$ \cite[Thm.~4.9]{MR3123758} in which every object is cofibrant. By the latter property
  the canonical map $\coprod_{i\in I}\ell_{\Morita}( \bC_{i})  \to    \ell_{\Morita}( \coprod_{i\in I}\bC_{i})
$ is an equivalence in $\Ccat[W_{\Morita}^{-1}]$. Hence the second morphism in  \eqref{qwefoijoqiwefqewfewfqwfe}
is an equivalence.
We now claim that $\Homol_{\Morita}$ preserves  finite products.  

For the moment assume the claim. Since $\bS$ (as a stable $\infty$-category) and $\Ccat[W_{\Morita}^{-1}]$ by \cite[Thm.\ 1.4]{MR3123758}
are semi-additive, the functor $ \Homol_{\Morita}$ then also preserves finite coproducts. This implies that 
the first morphism in  \eqref{qwefoijoqiwefqewfewfqwfe} is an equivalence, too.
Hence assuming the claim we conclude that   \eqref{qwefoijoqiwefqewfewfqwfe}  and  therefore
 \eqref{rhqifhjifqwefqwefwefq}  are equivalences.

 It remains to show the claim.
 So let $(\ell_{\Morita}(\bE_{j}))_{j\in J}$ be a finite family of objects of $\Ccat[W_{\Morita}^{-1}]$.
 Since every object in $\Ccat$ is Morita equivalent to an additively and idempotently complete
 object (apply e.g.\ the  functor  $(-)^\sharp$ from \eqref{qwrf98u89euf89qweuf98ewfeqwfqe})
 we can assume without loss of generality that $\bE_{j}$ is additively and idempotently complete for every $j$ in $J$. Since such objects are fibrant in the model category structure  of 
  \cite{MR3123758}  
  we can conclude that the canonical map 
  \begin{equation}\label{fqweiuiqjefioqwefewqef}
\ell_{\Morita} \big(\prod_{j\in J} \bE_{j} \big)\to \prod_{j\in J} \ell_{\Morita} (\bE_{j})
\end{equation}
 is an equivalence. Applying $\Homol_{\Morita}$ we get the equivalence 
 \begin{equation}\label{fqweiuiqjefioqwefewqewfefefefef}
 \Homol_{\Morita}\big(\ell_{\Morita}\big(\prod_{j\in J} \bE_{j}\big)\big)\xrightarrow{\simeq} \Homol_{\Morita}\big(\prod_{j\in J} \ell_{\Morita} (\bE_{j})\big)\,.
\end{equation}
Since $\bE_{j}$ is additive and hence {non-empty} %admits a zero object for every $j$ in $J$
 we can apply 
  Lemma~\ref{erighi9wrteogwergrgregwergwreg}  in order to conclude that the lower horizontal morphism in  the commutative diagram
 $$\xymatrix{\Homol_{\Morita}(\ell_{\Morita}(\prod_{j} \bE_{j}))\ar[r]^{!}\ar[d]^{\eqref{weoigjwegerregwerg}}_{\simeq}&\prod_{j}   \Homol_{\Morita}(\ell_{\Morita}( \bE_{j})) \ar[d]^{\eqref{weoigjwegerregwerg}}_{\simeq} \\ \Homol (\prod_{j\in J} \bE_{j})\ar[r]^{\text{Lem.\ }\ref{erighi9wrteogwergrgregwergwreg}}&\prod_{j\in J} \Homol ( \bE_{j})}$$
 is an equivalence. Hence the arrow marked by $!$ is an equivalence.
 Composing this arrow with the inverse of \eqref{fqweiuiqjefioqwefewqewfefefefef} provides the desired equivalence 
 $$\Homol_{\Morita}\big(\prod_{j\in J} \ell_{\Morita} (\bE_{j})\big)\stackrel{\simeq}{\to} \prod_{j}   \Homol_{\Morita}(\ell_{\Morita}( \bE_{j}))\, .$$
This finishes the verification of the claim and therefore the proof  of the lemma.
\end{proof}

\begin{rem}
 If $\Homol \colon \Ccat\to \bS$ is a Morita invariant homological functor, then it in particular preserves the empty coproduct, i.e., the canonical map
 $0_{\bS}\to \Homol(\emptyset)$ is an equivalence. But this is already true without the assumption of Morita invariance by Lemma \ref{qregheriogwegerwf}.
\hB
\end{rem}

Let $\Homol\colon \nCcat\to \bS$ be a {functor.}   
\begin{kor}\label{kor_Morita_finitary_preserves_coprods}
If $\Homol$ is {a Morita invariant  finitary  homological functor} (see {Definitions \ref{oihgjeorgwergergergwerg} and} \ref{qerughqiregqewfewqffq1}), then $\Homol$ preserves all small coproducts.
\end{kor}
\begin{proof}
Every small coproduct is a   small filtered colimit of finite coproducts. Hence the claim follows from the previous Lemma \ref{lem_Morita_preserves_finite_coprods} and the fact that  $\Homol$ preserves small filtered colimits by assumption.
\end{proof}

Recall Definition \ref{ergiowergregrdsvsv} of the functor $\Kcat\colon \nCcat\to \Sp$.
 
\begin{theorem}\label{wtoijgwergergrewfwergrg}
The functor $\Kcat$ is Morita invariant.
\end{theorem}

\begin{proof}
We use the characterization of Morita invariant functors  provided by Lemma \ref{rgoijeroigwergergrgwg}.

%\uli{From Lemma \ref{qerughqiregqewfewqffq}} we already know that 
The functor $\Kcat$ sends unitary  equivalences to equivalences since it is a homological functor {by Theorem \ref{qrgkljwoglergerggwergwrg}}.

Let $ \bD\to \bE$ be a fully faithful {morphism in $\Ccat$}  satisfying the Conditions \ref{wtoigjoergfregewgre90}, \ref{wtoigjoergfregewgre9} and \ref{wtoigjoergfregewgre91}.
  We identify $\bD$ with a full subcategory of $\bE$.
Then we  must show that $$ \Kcat(\bD)\to \Kcat(\bE)$$ is an equivalence. Note that the homomorphism of $C^{*}$-algebras $ A(\bD)\to A(\bE)$ (see \eqref{frewfoirjviojvioeweverwvwev} for $A(-)$) is defined since $\bD\to \bE$ is injective on objects. Since $ \bD\to \bE$  is fully faithful,  for every finite set $F$ of objects $F$ in $\bD$ the composition $A(F)\to A(\bD)\to A(\bE)$ is an embedding, see \eqref{nriqrnficjwenkcsdcdcsdacsdcasdcasdcad} for  $A(F)$.

In view of Lemma \ref{sfdbjoqi3gegergewgerg1}
it suffices to show that the induced homomorphism
$$ \phi\colon \pi_{*}\Kast(A(\bD))\to \pi_{*}\Kast(A(\bE))$$ between $K$-theory groups is an isomorphism.   In view of Bott periodicity  (Remark \ref{wrthijowrtgergwergwegr}.\ref{wegiojwoegwergrwegwreferf}) is suffices to consider the cases $*=0$ and $*=1$.

We will   use the shorter notation $K_{*}:=\pi_{*}\Kast$.

\textbf{surjectivity for $*=0$:}

  Let $p$ be in $K_{0}(A(\bE))$. We must show that $p$ is in the image of  $\phi\colon K_{0}(A(\bD))\to K_{0}(A(\bE))$.

   Since $\bE$ is additive, by Lemma \ref{egoijotrgregewrgwerg}.\ref{eoijwoegregregwerg}  we can find an object $E$ in $\bE$ and a pair of projections $P,\tilde P$ in $\End_{\bE}(E)$ such that
$\ell_{E,*}([P]-[\tilde P])=p$, where $\ell_{E}\colon \End_{\bE}(E)\to A(\bE)$ is the canonical (in general non-unital) embedding \eqref{sfvniojqrogvfevfdsvsfdvsfdvvsfdv}, see also Remark \ref{rgboiwjoijvmfldkmvskldfvsfdv} for notation.

By the assumption on the functor {$\bD\to \bE$}
   we can   choose a  family of  objects $(D_{i})_{i=1,\dots,m}$ of $\bD$  and an isometry 
$u\colon E\to \bigoplus_{i=1}^{m} D_{i}$.   For every $i$ in $\{1,\dots,m\}$ we define the morphism
$$u_{i}:=e_{i}^{*}u\colon E\to D_{i}\, ,$$
where $(e_{i})_{i=1}^{m}$ is the family of structure maps for the sum $\bigoplus_{i=1}^{m} D_{i}$.    Then  the $m\times m$-matrix with entries in $A(\bE)$ $$u':=\sum_{i=1}^{m} u_{i}[i,1] $$ is a   partial isometry in $ \Mat_{m}(A(\bE))$.   We consider the finite subset $F:=\{D_{1},\dots,D_{m}\}$ of objects in $\bD$.
 The conjugation map $u'(-)u^{\prime,*}\colon \Mat_{m}(A(\bE))\to \Mat_{m}(A(\bE))$  has values in the subalgebra
  $\Mat_{m}(A(F))$ of $\Mat_{m}(A(\bE))$, and $u^{\prime,*}u'= h(\id_{E})$, where   $$h:=\epsilon_{A(\bE),{m}} [1]\circ \ell_{E}\colon \End_{\bE}(E)\to \Mat_{m}(A(\bE))$$ 
and
   $\epsilon_{A(\bE),{m}}[1]$  is as in \eqref{sfdvoijio3rgfevfdsvf}.     By construction %Using \eqref{ewvoiuhiuosdfqwfwed} 
   we have
   $$p=  \epsilon_{A(\bE),{m}}[1]_{*}^{-1} h_{*}([P]-[\tilde P])\, .$$  We consider   $\tilde h:=u'hu^{\prime,*}$ as a homomorphism from 
   $ \End_{\bE}(E)$ to $\Mat_{m}( A(F))$ and let $h'$ be its composition with $\kappa\colon \Mat_{m}(A(F))\to\Mat_{m}( A(\bD)) $ and $ \Mat_{m}( A(\bD))\to \Mat_{m}( A(\bE))$.
  Then the chain of equalities  
\[p=  \epsilon_{A(\bE),{m}}[1]_{*}^{-1} h_{*}([P]-[\tilde P])\stackrel{\eqref{vasfvlkmlksaf24fwf}}{=}   \epsilon_{A(\bE),{m}}[1]_{*}^{-1} h'_{*}([P]-[\tilde P]) =\phi ( \epsilon_{A(\bD),{m}}[1]_{*}^{-1}\kappa_{*}\tilde h_{*} ([P]-[\tilde P]))
\]
shows that $p$ is in the image of $\phi$.

%   
%
% For simplicity, in the following formulas we omit to write the inclusions into unitalizations.
% 
%  %We have a homomorphism
% % $$v:\End_{\bE}(E)\stackrel{\ell_{E}}{\to}A(\bE)\stackrel{\epsilon_{A(\bE)}[1]}{\to}\Mat_{m}(A(\bE)^{+})\stackrel{u'(-)u^{\prime,*}}{\to}     \Mat_{m}(A(\bD)^{+})  \ ,$$
% 
%We consider the homomorphism $h:=\epsilon_{A(\bE)}\circ \ell_{E}[1]:\End_{\bE}(E)\to \Mat_{n}(A(\bE))$. 
%Then $$P':=u' h(P) u^{\prime,*}\ , \quad \tilde P':=u' h(\tilde P ) u^{\prime,*}$$ are projections in  $\Mat_{m}(A(\bD)^{+})$ such that
%$P'\equiv \tilde P'$ modulo $\Mat_{m}(A(\bD))$. 
%We get a class $p':=[P',\tilde P']$ in $K_{0}(A(\bD))$ (see the text before Lemma \ref{egoijotrgregewrgwerg}).
%We claim that this class maps to $p$ under $K_{0}(A(\bD))\to K_{0}(A(\bE))$.
%The homomorphism
%$h:=\epsilon_{A(\bE)}\circ \ell_{E}[1]:\End_{\bE}(E)\to \Mat_{n}(A(\bE))$ satisfies 
%$h u^{\prime,*}u=h$.  
%Using \eqref{ewflkqwnefqlefewfqwefqwefqwefqwef} we have
%$$p=[h(P),h(\tilde P)]= [u^{\prime}h(P)u^{\prime,*}, u^{\prime}h(\tilde P)u^{\prime,*} ] =  =[P',\tilde P']\ .$$
%where in the last term we consider $P'$ and $\tilde P'$ naturally as projections in $\Mat_{m}(A(\bE)^{+})$.
 
  \textbf{injectivity for $*=0$:}
  
  Let $p$ be in $K_{0}(A(\bD))$ such that $\phi(p)=0$. As explained in  Remark \ref{rgboiwjoijvmfldkmvskldfvsfdv} and the beginning of the proof of Lemma \ref{egoijotrgregewrgwerg} there exists a finite subset $F$ of objects in $\bD$ and 
 projections  $P,\tilde P$     in $\Mat_{n}(A(F))$ %with $P\equiv P'$ modulo $\Mat_{n}(A(\bD))$ 
 such that $p=\kappa_{*}([P]-[\tilde P])$, where $\kappa\colon A(F)\to A(\bD)$ is the inclusion and $[P],[\tilde P]$ are considered in $K_{0}(A(F))$.

Using the inclusion $A(F)\to A(\bE)$ we can consider  the projections $P$ and $\tilde P$ as elements in $\Mat_{n}(\bE)$.
Since $\phi(p)=0$, after increasing $n$ if necessary there exists a 
   partial isometry $U$ and {a projection $Q$ (orthogonal to $P$ and $\tilde P$) in $\Mat_{n}(A(\bE)^{+})$ such that $UU^{*}=P+Q$ and $U^{*}U=\tilde P+Q$.   As in the proof of Lemma \ref{egoijotrgregewrgwerg}.\ref{eoijwoegregregwerg1} we can increase $F$ and change $U$ and $Q$ such that they  belong} to the subalgebra $\Mat_{n}(A(F))$ of $\Mat_{n}(A(\bE)^{+})$. Consequently, $[P]=[\tilde P]$ and hence $p=0$.

 \textbf{surjectivity for $*=1$:}
 
 Let $u$ be in $K_{1}(A(\bE))$. Since $\bE$ is additive, by Lemma \ref{tegiotwrgerwgregwgergre}.\ref{ergiojetwregewgg} we can find an object $E$ in $\bE$ and a unitary $U$  in $\End_{\bE}(E)$ {with} $\ell_{E,*}[U]=u$. 
 Then as  in the argument for surjectivity for $*=0$ we have
 $$u=\epsilon_{A(\bE),{m}}[1]_{*}^{-1}  h_{*}[U]=\epsilon_{A(\bE),{m}}[1]_{*}^{-1}  h'_{*}[U]=\phi(\epsilon_{{A(\bD),m}}[1]_{*}^{-1}  \kappa_{*}\tilde h_{*}([U]))$$
 so that $u$ is in the image of of $\phi$.

\textbf{injectivity for $*=1$:}

Let $u$  in $K_{1}(A(\bD))$ be such that $\phi(u)=0$.   As in the proof of Lemma   \ref{tegiotwrgerwgregwgergre}.\ref{ergiojetwregewgg} 
there exists a finite set of objects $F'$   of $\bD$  and $n$ in $\nat$ such that there is  an unitary   $U$ in $\Mat_{n}(A(F'))$  with 
$ [U]=u$. We let $[U]_{F',n}$ in $K_{1}(\Mat_{n}(A(F'))$ denote the corresponding class and $\kappa_{F',n}\colon \Mat_{n}(A(F'))\to \Mat_{n}(A(\bD))$ be the inclusion. Then we have the equality
\begin{equation}\label{gbnkjgbervfdvsdfvsdfv}
u=[U]\stackrel{\eqref{ewvoiuhiuosdfqwfwed1}}{=}\epsilon_{A(\bD),n}[1]_{*}^{-1} \kappa_{F',n,*} ([U]_{F',n})\, .
\end{equation}

We can further find an object $E$ in $\bE$ and a homomorphism
$\psi\colon\Mat_{n}(A(F'))\to \End_{\bE}(E)$ such that $\phi(u)=\ell_{E,*}(\psi_{*}([U]_{F',n}))$. 

Since $\phi(u)=0$,
 by  Lemma  \ref{tegiotwrgerwgregwgergre}.\ref{regwoihjoirgerwrgwerggrgwg}, after enlarging $E$ if necessary, we can assume 
that   $\psi_{*}[U]_{F',n}=0$. 

We let $F''$ be the union of the family $F$ chosen above  in order to represent $E$ as a subobject  and the  family $F'$.
Let now $\tilde h\colon \End_{\bE}(E)\to \Mat_{m}(A(F''))$ be as above. 
 Then
\begin{equation}\label{adsvoiqjiorfavsa}
\tilde h_{*}\psi_{*}[U]_{F',n}=0\, .
\end{equation}

By an inspection of the construction one observes that $$\tilde h  \circ \psi\colon \Mat_{n}(A(F'))\to \Mat_{m}(A(F''))$$ is  {the} conjugation $w( -) w^{*}$ by an element
in $w$ in $\Mat(m,n,A(F''))$ such that $w^{*}w=1_{A(F'),n}$. This implies that
$\tilde h_{*}\psi_{*}\colon K_{1}(\Mat_{n}(A(F')))\to K_{1}(\Mat_{n}(A(F'')))$  is equal to the map  induced by the inclusion $\iota\colon \Mat_{n}(A(F'))\to \Mat_{m}(A(F''))$. Consequently
\begin{align*}
u & \quad \stackrel{\mathclap{\eqref{gbnkjgbervfdvsdfvsdfv}}}{=} \quad \epsilon_{A(\bD),n}[1]_{*}^{-1} \kappa_{F',n,*} ([U]_{F',n})\\
& \quad = \quad \epsilon_{A(\bD),m}[1]_{*}^{-1} \kappa_{F'',m,*} \iota_{*}([U]_{F',n}) \\
& \quad = \quad \epsilon_{A(\bD),m}[1]_{*}^{-1} \kappa_{F'',m,*}    \tilde h_{*}\psi_{*}   ([U]_{F',n})\\
& \quad \stackrel{\mathclap{\eqref{adsvoiqjiorfavsa}}}{=} \quad 0\,.\qedhere
\end{align*}
\end{proof}

\begin{kor}\label{kor_Kcat_preserves_coprods}
$\Kcat$ preserves all  very small coproducts.
\end{kor}
\begin{proof}
By Theorem \ref{wtoijgwergergrewfwergrg} the functor  $\Kcat$ is Morita invariant, and by Theorem \ref{qrgkljwoglergerggwergwrg} it is finitary. The claim now follows  from Corollary \ref{kor_Morita_finitary_preserves_coprods}.
\end{proof}

\begin{rem}
In \cite[Rem.~10.12]{MR3123758} the authors review various definitions of {$K_{0}$-groups} for $C^{*}$-categories  appearing in the literature and compare them with their functor
$$K^{\text{D'A-T}}_{0}(\bC) \coloneqq \Hom_{\Ho(\Ccat[W_{\Morita}^{-1}]%^{\Idem}_{\oplus}
)} 
(\C^{\sharp},\bC^{\sharp})$$
(see \eqref{qwrf98u89euf89qweuf98ewfeqwfqe} for $\sharp$) which is
Morita invariant by definition.
In particular in Point~(iii) {of that remark} they mention the version $\pi_{0}\Kcat(\bC)$ considered in the present paper.
It is not clear that these two $K_{0}$-functors are isomorphic.
\hB
\end{rem}

\section{Relative  Morita equivalences and Murray--von Neumann equivalent morphisms}\label{eiogjwoiggregwegewrgwegrwerg}

The notion of a Morita equivalence   is  only defined for unital  morphisms between  unital $C^{*}$-categories. The reason is that finite orthogonal sums or the  canonical embedding
$\bC\to \Idem(\bC)$ defined by $C\mapsto (C,\id_{C})$  require  the existence of identity endomorphisms.  In the {present section} we  extend   the notion  of a Morita equivalence   to the  relative situation of an ideal in a unital $C^{*}$-category. We then    show that Morita invariant homological functors
send   relative Morita equivalences    to equivalences.
As a  particular example of a relative Morita equivalence we  discuss the relative idempotent
completion of an ideal.
We furthermore
introduce the notion of Murray-von Neumann (MvN) equivalence between morphisms in $\nCcat$ and 
  show that    Morita invariant homological functors send MvN-equivalent morphisms  to equivalent morphisms.

 Let $\phi:\bK\to \bL$ be a morphism in $\nCcat$.
 \begin{ddd}\label{sitogdghsgfgsfg}
 The morphism $\phi$ is a relative Morita equivalence if it extends to a morphism of exact sequences in $\nCcat$
 \begin{equation}\label{therthtrgetrgtregtg}
 \xymatrix{0\ar[r]&\bK\ar[r]\ar[d]^{\phi}&\bC\ar[d]^{\psi}\ar[r]&\bC/\bK\ar[r]\ar[d]^{\kappa}&0\\0\ar[r]&\bL\ar[r]&\bD\ar[r]&\bD/\bL\ar[r]&0}
\end{equation}
 such that $\bC$ and $\bD$ are unital and $\psi$ and $\kappa$ are Morita equivalences.
 \end{ddd}
 Note that $\psi$ is implicitly assumed to be unital, and that the  assumptions imply that   the quotient categories and $  \kappa$   are unital, too.

 \begin{prop}\mbox{} \begin{enumerate}
 \item {For any group $G$ the functor $-\rtimes G:\Fun(BG,\nCcat)\to \nCcat$ preserves  relative Morita equivalences.}\item 
 If $G$ is an exact group, then  $-\rtimes_{r}G:\Fun(BG,\nCcat)\to \nCcat$ preserves  relative Morita equivalences.
 \end{enumerate}
 \end{prop}
\begin{proof}
{We write out the details  in the case of the reduced crossed product.
Since the   maximal  crossed product preserves exact sequences for every group 
in this case we can argue in an analogous manner
  without any restriction on $G$.}

Assume that we are given a   morphism of exact sequences  as in \eqref{therthtrgetrgtregtg}. Since the functor 
 $-\rtimes_{r}G$ preserves exact sequences by   Proposition \ref{rgkohpgbdghdfghf}
we get a morphism of exact sequences 
\begin{equation}\label{therthtrgetrgtregtg333}
 \xymatrix{0\ar[r]&\bK\rtimes_{r}G\ar[r]\ar[d]^{\phi\rtimes_{r}G}&\bC\rtimes_{r}G\ar[d]^{\psi\rtimes_{r}G}\ar[r]&(\bC/\bK)\rtimes_{r}G\ar[r]\ar[d]^{\kappa\rtimes_{r}G}&0\\0\ar[r]&\bL\rtimes_{r}G\ar[r]&\bD\rtimes_{r}G\ar[r]&(\bD/\bL)\rtimes_{r}G\ar[r]&0}
\end{equation}
By Proposition \ref{werijguhwerigvwergwer9} the morphisms  $\psi\rtimes_{r}G$ and 
$\kappa\rtimes_{r}G$ are again Morita equivalences. Therefore
$\phi\rtimes_{r}G$ is a relative Morita equivalence.
\end{proof}

% \begin{rem}
% 
% \uli{Show that $\rtimes G$ preserves weak Morita equivalences without exctness}
% The obvious idea to get rid of the  exactness assumption on $G$ by working with maximal crossed products does not work because of the problem noted in Remark \ref{weug89wegwrefwrefw}. \Alex{Geht doch, weil das maximal crossed product doch fully faithful erhält.}
% \end{rem}

We consider a functor $\Homol:\nCcat\to \bS$.

 \begin{prop}\label{qrgojqeorgqfewfeqwfqewf1}
 If $\Homol$ is a Morita invariant homological functor, then it sends relative Morita equivalences to equivalences.
   \end{prop}
   \begin{proof}
   Applying $\Homol$ to the diagram \eqref{therthtrgetrgtregtg} we get a morphism of fibre sequences
    \begin{equation*}\label{therthtrgetrgtregtg1}
 \xymatrix{ \Homol(\bK)\ar[r]\ar[d]^{\Homol(\phi)}&\Homol(\bC)\ar[d]^{\Homol(\psi)}\ar[r]&\Homol(\bC/\bK) \ar[d]^{\Homol(\kappa)}\\ \Homol(\bL)\ar[r]&\Homol(\bD)\ar[r]&\Homol(\bD/\bL) }
\end{equation*}
The assumptions imply that $\Homol(\psi)$ and $\Homol(\kappa)$ are equivalences. Hence $\Homol(\phi)$ is an equivalence, too.
      \end{proof}

We now turn to the notion of a relative  idempotent completion.
Let $\bK$ be in $\nCcat$ and assume that $\bK\to \bC$
is an ideal inclusion with $\bC$ in $\Ccat$.

%We consider  $\bC$ in $\Ccat$ and an ideal $\bK$ in $\bC$.

\begin{ddd}\label{qrihweiofdwefffqfwefqwef}
The idempotent completion  $\bK\to \Idem^{\bC}(\bK)$ of $\bK$
relative to $\bC$ is  the inclusion of $\bK$ into the wide subcategory $\Idem(\bC)$ of morphisms belonging to $\bK$.
\end{ddd}

 \begin{rem}\label{rem_defn_Idem_ideal}
{Unfolding the definition and using the explicit model   % \cite[Defn.~2.15]{MR3123758}   
 of the idempotent completion $\Idem(\bC)$ 
described in Construction \ref{wijgowegwergf} %after Definition \ref{defn_idempotent_complete} 
we get the following explicit description of $\Idem^{\bC}(\bK)$:}
\begin{itemize}
\item objects: The  objects of  $\Idem^{\bC}(\bK)$ are the objects of $\Idem(\bC)$, i.e.,   pairs $(C,p)$ of an object $C$ of $\bC$ and a projection $p$ on $C$ belonging to $\bC$.  
\item  morphism: The morphisms $A\colon(C,p)\to (C',p')$ in $\Idem^{\bC}(\bK)$  
 are morphisms in $\Idem(\bC)$ 
 with the additional property  that $A$ belongs to $\bK$.
\end{itemize}
Note that $\Idem^{\bC}(\bK)$ depends on the embedding of $\bK$ into $\bC$.
The canonical inclusion $\bC\to \Idem(\bC)$ restricts to   the   morphism $\bK\to \Idem^{\bC}(\bK)$.  \hB
\end{rem}

Using the explicit description given in Remark \ref{rem_defn_Idem_ideal} one easily sees   that $ \Idem^{\bC}(\bK)\to \Idem(\bC)$  is an ideal inclusion.

\begin{ex}
For $\bK$ in $\nCcat$ a natural choice of an ideal inclusion is the embedding
$\bK\to \bM\bK$ of $\bK$ into its multiplier category. This leads to an idempotent completion
$\Idem^{\bM\bK}(\bK)$ which only depends on $\bK$. But since the transition to the multiplier category is only functorial for a restricted class of morphisms (see Proposition \ref{stkghosgfgsrgsegs}) one can not expect to get an idempotent completion functor for not necessarily unital $C^{*}$-categories in this way.  
 \hB
\end{ex}

%Let $\Homol\colon \nCcat\to \bS$ be %a homological 
%{functor} and let $\bK\to \bC$ be an ideal inclusion with $\bC$ in $\Ccat$.
%We consider the canonical morphism  $\bK\to \Idem^{\bC}(\bK)$.  

\begin{prop}\label{qrgojqeorgqfewfeqwfqewf}
A relative idempotent completion is a relative Morita equivalence.
%If $\Homol$ is a Morita invariant {homological 
%functor}, then the canonical morphism induces an equivalence 
 %$\Homol(\bK)\to \Homol(\Idem^{\bC}(\bK))$.
\end{prop}
\begin{proof}
 Let $\bK\to \bC$ be an ideal inclusion with $\bC$ in $\Ccat$. We must show that   
 the canonical morphism  $\bK\to \Idem^{\bC}(\bK)$ is a relative Morita equivalence.
We consider the exact sequence
$$0\to \bK\to \bC \to \bC/\bK\to 0$$ in $\nCcat$.
%where $\bQ$ is defined to be the quotient.
We then get an exact sequence
$$0\to \Idem^{\bC}(\bK)\to \Idem(\bC)\to    \bQ\to 0\ ,$$
where $  \bQ$ is defined as the quotient.
 We have a canonical  morphism
$ \bQ\to \Idem(\bC/\bK)$ which sends $(C,p)$ to $(C,[p])$, 
and which is the obvious map on morphisms.  
Here $[p]$ denotes the image in  $\bC/\bK$ of a morphism $p$ in $\bC$.
Unfolding the definition  we see that this morphism is  faithful.  In order 
 to see that it is also full    note that 
if $[A]\colon (C,[p])\to (C',[p'])$  is a morphism in ${\Idem(\bQ)}$, then  the relations  $[p'][A]=[A]=[A][p]$ imply that
$[A]=[p'Ap]$. Hence $[A]$ can be  lifted to a morphism $p'Ap\colon (C,p)\to (C',p')$ in $\Idem(\bC)$.
Thus we can identify    $  \bQ$ with  the full subcategory of $\Idem(\bC/\bK)$ consisting of objects $(C,[p])$ such that $[p]$ lifts to a projection in $\bC$.
We obtain   the following {commutative} diagram:
$$\xymatrix{0\ar[r]&\bK\ar[r]\ar[d]^{!!!}&\bC\ar[r] \ar[d]^{!}&\bC/\bK\ar[r]\ar[d]^{!!}\ar@/^3.5cm/[dd]^{!}\ar[d]&0\\
0\ar[r]&\Idem^{\bC}(\bK)\ar[r]&\Idem(\bC)\ar[r]& \bQ\ar[r]\ar[d]^{!}&0\\&&&\Idem(\bC/\bK)&}$$
%By assumption $\bC$, and hence also $\bQ$ are additive. Then $\tilde \bQ$ is additive, too.
The arrows marked by $!$  present idempotent completions of   $C^{*}$-categories and therefore are    Morita equivalences, as observed after Definition \ref{wetgoijwtoigwerrefwerfwefwref}.   It is immediate from Definition \ref{wetgoijwtoigwerrefwerfwefwref}  that  Morita equivalences satisfy the two-out-of-three principle. Therefore 
 the morphism marked by $!!$ is a Morita equivalence. In view of Definition \ref{sitogdghsgfgsfg}
  the morphism marked by $!!!$ is a relative Morita equivalence.
%They induce equivalences  after applying $\Homol$. This implies that also $!!$ induces an equivalence after applying $\Homol$. Applying $\Homol$ to the
%  horizontal sequences yields fibre sequences in $\bS$. 
%By the Five Lemma we can then conclude that also the arrow marked by $!!!$ induces an equivalence after applying $\Homol$.
\end{proof}

%We now consider two objects $\bK,\bK'$ in $\nCcat$ and a \uli{morphism} $f\colon \bK\to \bK'$. 

%\begin{ddd}\label{qergiojqoeqwefewfqweffqwe}
%We call $f$ a relative equivalence if there exists a pull-back diagram
%$$\xymatrix{\bK\ar[r]\ar[d]^{f}&\bC\ar[d]^{\hat f}\\\bK'\ar[r]&\bC'}$$
%such that the horizontal morphisms are inclusions of ideals and $\hat f$ a \uli{unitary} equivalence in $\Ccat$.
%\end{ddd}
%
%Let $\Homol\colon \nCcat\to \bM$ be a \uli{functor.} 
%\begin{prop}\label{qeroigjoiwergqwefeqewfq9}
%\uli{If $\Homol$ is a homological functor and} $f\colon \bK\to \bK'$ is a relative equivalence, then
%$\Homol(f)\colon \Homol(\bK)\to \Homol(\bK')$ is an equivalence.
%\end{prop}
%
%\begin{proof}
%We have the following  morphism
%$$\xymatrix{0\ar[r]&\bK\ar[r]\ar[d]^{f} &\bC\ar[r] \ar[d]^{\hat f}&\bQ\ar[r] \ar[d]^{\bar f}&0\\
%0\ar[r]& \bK'\ar[r]& \bC'\ar[r]& \bQ'\ar[r] &0}$$
%of exact sequences in $\nCcat$, where $\bQ$ and $\bQ'$ are defined as the quotients.
%Since $\hat f$ is fully faithful and the left square is a pull-back, we see that
%$\bar f$ is also fully faithful.
%Since $\hat f$ is essentially surjective, so is $\bar f$.
%Hence $\bar f$ is a \uli{unitary} equivalence, too.
%
%We now use that $\Homol$ sends the horizontal sequences to fibre sequences and the two right vertical morphisms to equivalences. By the Five Lemma we conclude that it also sends the left vertical morphism to an equivalence.
%\end{proof}

Let $f,g\colon \bC\to\bD$ two {morphisms} in $\nCcat$.  %Let $i:\bD\to \bM\bD$
%be the canonical inclusion of $\bD$ into its multiplier category. 

\begin{ddd}\label{weoigjwtgergregweg}
We say  that $f$ and $g$ are  unitarily isomorphic if there exists a unitary multiplier isomorphism between $f$ and $g$.  %compositions
%$i\circ f$ and $i\circ g$ are uitarily isomorphism. 
%there exists an ideal inclusion  $\omega:\bL\to \bD$  for  some $\bD$ in $\Ccat$ such that the compositions $\omega\circ f$ and $\omega \circ g$ are unitarily isomorphic.
%and extensions $\hat f,\hat g\colon \bC\to \bD$ of $f,g$ such that $\hat f$ and $\hat g$ are unitarily \uli{isomorphic}.
\end{ddd}

If $\bD$ is unital, then this definition reduces to the usual notion of unitarily isomorphic morphisms.

\begin{rem} \label{gjirogggfdg}Two morphisms $f,g\colon \bC\to\bD$  in $\nCcat$  are unitarily isomorphic if and only if there exists an ideal inclusion $i:\bD\to \bE$ such that $i\circ f$ and $i\circ g$ are unitarily isomorphic. In one direction,  if $f$ and $g$ are 
unitarily isomorphic, then we can take the ideal inclusion $\bD\to \bM\bD$. Vice versa, if $u:i\circ f\to i\circ g$ is a unitary isomorphism for some ideal inclusion
$i:\bD\to \bE$, then the image of $u$ under the canonical morphism
$\bE\to \bM\bD$ gives an unitary multiplier isomorphism between $f$ and $g$.
\hB \end{rem}

Let $\Homol\colon \nCcat\to \bS$ be a functor.

\begin{lem}\label{efijoijgoqegrqgqrgqrg}
If $\Homol$ sends unitary equivalences to equivalences and
 $f$ and $g$ are unitarily isomorphic, then we have an equivalence  $\Homol(f)\simeq \Homol(g)$. 
\end{lem}
\begin{proof}
We define a category $\bE $ in $\nCcat$ as follows playing the role of an interval objects.
\begin{enumerate}
\item objects: $\Ob(\bE  ):=\Ob(\bC)\sqcup \Ob(\bC)$.
 For $C$ in $\bC$ we let $C_{0}$ and $C_{1}$ denote the two copies of $C$ in $\bE $.
 \item morphisms: For $C,C'$ in $\bC$ and   $i,j$ in $\{0,1\}$ we
 set $\Hom_{\bE }(C_{i},C_{j}'):=\Hom_{\bC}(C,C')$. 
 \item composition and involution are defined in the obvious way.
\end{enumerate}
We have two inclusions $\iota_{0},\iota_{1}\colon \bC\to \bE $ sending $C$ to $C_{0}$ and $C_{1}$, respectively. We further have a projection $p\colon \bE \to \bC$
defined in the obvious way. Note that $p \circ \iota_{0} = p \circ \iota_{1} = \id_{\bC}$. Furthermore, we have unitary multiplier  isomorphisms $v_{i}\colon \iota_{i}\circ p\to \id$. For example, ${v_{0}}$ is given by
$v_{0,C_{0}}=\id_{C_{0}}$ and $v_{0,C_{1}}=\id_{C}$ in $\Hom_{\bM\bE }(C_{0},C_{1})$.
 We conclude that $p$ is a unitary equivalence and $\iota_{0}$, $\iota_{1}$ are both unitary equivalences which are inverse to $p$.   In particular we have an equivalence \begin{equation}\label{qwefqwefewdasafeqwfdaf}
\Homol(\iota_{0})\simeq \Homol(\iota_{1})\ .
\end{equation}

%\color{red}
% \begin{rem} The  appearance of the construction of the   $C^{*}$-category  $\bE$ in this argument is not an accident. Note that
% the $\bE$ is isomorphic to $\beins_{\Ccat}\otimes \bC$, where $\beins_{\Ccat}$ is the unitary morphism classifyer in $\Ccat$ \cite[Def. 4.10]{startcats}.  
%The functor $  \beins_{\Ccat}\otimes-$  together with the natural transformations  implemented by $\iota_{0},\iota_{1}$ and $p$  can be taken as a cylinder functor on $\Ccat$. The tensor product of $C^{*}$-categories in this case has the obvious definition since $\beins_{\Ccat}$ has finite-dimensional morphism spaces, but also see  \cite[Sec. 7]{KKG} for the general situation.  \hB
% \end{rem}
%\color{black}

Let $u\colon f\to g:\bC\to \bD$ be the unitary multiplier  isomorphism. 
We then define a {morphism}
$h\colon \bE\to \bD$ as follows:
\begin{enumerate}
\item objects: For $C$ in $\bC$ we set $h(C_{0}) \coloneqq f(C)$ and $h(C_{1}) \coloneqq g(C)$.
\item morphisms: We distinguish the following four cases.
\begin{enumerate}
\item $c\colon C_{0}\to C_{0}'$ is sent to $h(c) \coloneqq f(c)$.
\item $c\colon C_{0}\to C_{1}'$ is sent to $h(c) \coloneqq u_{C'} f(c)$.
\item $c\colon C_{1}\to C_{0}'$ is sent to $h(c) \coloneqq  u^{*}_{C'} g(c)$.
\item $c\colon C_{1}\to C_{1}'$ is sent to $h(c) \coloneqq g(c)$.
\end{enumerate}
 \end{enumerate}
One checks that this defines a   morphism in $\nCcat$. We note that 
$h\circ \iota_{0}=f$ and $h\circ \iota_{1}=g$.
We now conclude
\[
\Homol(f)\simeq \Homol(h)\circ \Homol(\iota_{0})\stackrel{\eqref{qwefqwefewdasafeqwfdaf}}{\simeq} \Homol(h)\circ \Homol(\iota_{1})\simeq \Homol(g)\,.\qedhere
\]
\end{proof}

In the next proposition we weaken the assumption in Lemma \ref{weoigjwtgergregweg} from  unitarily isomorphic to  Murray--von Neumann
 equivalent. We start with defining this notion for a pair    of  {morphisms}
  $f,g\colon \bC\to \bD$ in $\nCcat$.

\begin{ddd}\label{wkjrthowrtgewgwergw}
 We say  that $f$ and $g$ are   Murray--von Neumann equivalent (MvN equivalent) if there exists  a natural multiplier transformation $u \colon   f\to   g$ given by $u=(u_{C})_{C\in \bC}$, where $u_{C}$ is a partial  isometry in $\bM\bD$ for every object $C$ of $\bC$ such that 
{$u^{*}_{C'}u_{C'}f(k) = f(k)$ and $g(k)u_{C}u_{C}^{*} = g(k)$}
%\Alex{deine Variante war: $u^{*}_{C'}u_{C'}f(k)=g(k)u_{C}u_{C}^{*}$. Aber das ist nicht gut.} 
for all morphisms $k \colon C\to C'$ in $\bC$.
%and extensions $\hat f,\hat g\colon \bC\to \bD$ of $f,g$ such that $\hat f$ and $\hat g$ are unitarily \uli{isomorphic}.
\end{ddd}

\begin{rem}In analogy to Remark \ref{gjirogggfdg}
   $f$ and $g$ are   MvN equivalent  if and only if there exists an ideal inclusion  $i \colon \bD\to \bE$    and a natural transformation $u \colon i\circ f\to i\circ g$ given by $u=(u_{C})_{C\in \bC}$, where $u_{C}$ is a partial isometry in $\bE$ for every object $C$ of $\bC$ such that 
{$u^{*}_{C'}u_{C'}f(k) = f(k)$ and $g(k)u_{C}u_{C}^{*} = g(k)$}
%\Alex{deine Variante war: $u^{*}_{C'}u_{C'}f(k)=g(k)u_{C}u_{C}^{*}$. Aber das ist nicht gut.} 
for all morphisms $k \colon C\to C'$ in $\bC$. \hB
%and extensions $\hat f,\hat g\colon \bC\to \bD$ of $f,g$ such that $\hat f$ and $\hat g$ are unitarily \uli{isomorphic}.
\end{rem}

%\begin{ddd}
%$f$ and $g$ are partially isometrically equivalent if there exist
%embeddings $\bK\to \bC$ and $\bL\to \bD$ as ideals of unital $C^{*}$-categories and a natural transformation
%$u\colon f\to g$ which consists of partial isometries in $\bD$
%\Alex{satisfying $u^* u f = f$ and $g u u^* = g$.}
%\end{ddd}
%\begin{rem}The family 
%So $u$ is given by a family $(u_{C})_{C\in \bC}$ of partial isometries $u_{C}\colon f(C)\to g(C)$ in $\bD$ such that for any morphism $k\colon C\to C'$ in $\bK$ we have\color{blue}
%\begin{equation}
%\label{eq_conditions_part_isom_equiv}
%u_{C'} f(k)=g(k) u_{C}\,,\quad u_{C^\prime}^* u_{C^\prime} f(k) = f(k)\,,\quad g(k) u_C u_C^* = g(k)\,.
%\end{equation}
% 
%The second equation in \eqref{eq_conditions_part_isom_equiv} implies $f(k) u^*_C u_C = f(k)$ as follows:
 %The first and second equation in \eqref{eq_conditions_part_isom_equiv} imply together $u^*_{C^\prime} g(k) u_C = f(k)$.
%
%Similarly, the third equation in \eqref{eq_conditions_part_isom_equiv} implies $u_{C^\prime} u_{C^\prime}^* g(k) = g(k)$, and the first and third together imply $u_{C^\prime} f(k) u_C^* = g(k)$.
%\hB
%\end{rem}
%
%\color{black}

%We let $\Kcat_{*}:=\pi_{*}\Kcat$ be the group-valued $K$-theory functor for $C^{*}$-categories.
%\begin{prop}
%If $f$ and $g$ are  partially isometrically equivalent, then we have the equality $\Kcat_{*}(f)=\Kcat_{*}(g)$.
%\end{prop}

We consider   two  {morphisms}
  $f,g\colon \bC\to \bD$  in $\nCcat$ and a functor $\Homol\colon \nCcat\to \bS$.
\begin{prop}\label{wkthgwfgwegwesdf}
If $f$ and $g$ are    MvN equivalent  and $\Homol$ is a Morita invariant homological functor, then $\Homol(f) \simeq \Homol(g)$.
\end{prop}
\begin{proof}
Let $u=(u_{C})_{C\in \Ob(\bC)}\colon f \to g$ be the natural multiplier transformation implementing  the MvN equivalence between the  morphisms $f$ and $g$. %considered as functors with values in $\bD$ (we drop $\omega$ from the notation for better readability). 
 For every object $C$ of $\bC$ we
 have projections $p_{C}:=u^{*}_{C}u_{C}$ on $f(C)$  and $q_{C}:=u_{C}u_{C}^{*}$ on $g(C)$ belonging to $\bM\bD$.
 
 If $k \colon C\to C'$ is a morphism in $\bC$, then  we have
  $p_{C'}f(k)=f(k)$ by assumption.  Note also that
  \begin{align*}
f(k)  p_{C} & = (f(k) p_{C})^*{}^*
  = (p_{C} f(k^*))^*\\
& = f(k^*)^*
  = f(k)\,.
\end{align*}
%Recall that $\Idem(\bL)$ is defined relatively to the ideal inclusion of $\bL$ into $\bD$.
 Consequently  the {morphism}  $f$ canonically induces a {morphism}
$\tilde f \colon \bC\to \Idem^{\bM\bD}(\bD)$ {in $\nCcat$}  given as follows:
\begin{enumerate}
\item objects: The {morphism} $\tilde f$ sends the object $C$ in $\bC$  to the object $(f(C),p_{C})$ of $\Idem^{\bM\bD}(\bD)$.
\item morphisms: The {morphism} $\tilde f$ sends a morphism $k \colon C\to C'$ in $\bC$ to the morphism $f(k) \colon (f(C),p_{C})\to (f(C'),p_{C'})$.
\end{enumerate}
%The observations above show that this functor is well-defined. 
We have a similarly defined {morphism}
$\tilde g \colon \bC\to \Idem^{\bM\bD}(\bD)$.

We let $\Emb \colon \Idem^{\bM\bD}(\bD)\to \Idem^{\bM\bD}(\bD)$ be the endo{morphism} given as follows:
\begin{enumerate}
\item objects: $\Emb$
%The functor $\Emb$ 
sends the object $(D,p)$   to the object $D=(D,\id_{D})$. 
 \item morphisms:   $\Emb$ sends  a morphism $\phi \colon (D,p)\to (D',p')$  to the morphism $\phi \colon D\to D'$.
\end{enumerate}
 We have the following {commutative} diagram
\begin{equation}\label{eq_diag_f_factors_Idem}
\xymatrix{
\bC \ar[rr]^-{f} \ar[d]_-{\tilde{f}} && \bD \ar[d]^{c} \\
\Idem^{\bM\bD}(\bD) \ar[rr]_-{\Emb} && \Idem^{\bM\bD}(\bD)
}
\end{equation}
where 
$c$ is the canonical inclusion.
We have a similar diagram for $g$.
%n $u_C^* u_C$ and $u_C u_C^*$ are projections in $\b  t TheD$ for any object $C$ of $\bC$, and $f$ canonically defines a morphism $\tilde{f}\colon \bC \to \Idem(\bL)$ as follows (recall Remark~\ref{rem_defn_Idem_ideal}): for any object $C$ of $\bC$ we set $\tilde{f}(C) \coloneqq (f(C),u_C^* u_C)$, and for a morphism $k\colon C \to C^\prime$ in $\bC$ we set $\tilde{f}(k) \coloneqq f(k)$.where $\Emb$ is the canonical embedding of the subobjects into their parent object: it acts on an object $(C,p)$ as $\Emb((C,p)) \coloneqq (C,\id_C)$ and on a morphism $A \colon (C,p) \to (C^\prime,p^\prime)$ as $\Emb(A) \coloneqq A$.
%Analogously as $\tilde{f}$ we then define $\tilde{g}\colon \bC \to \Idem(\bL)$ on objects $C$ of $\bC$ as $\tilde{g}(C) \coloneqq (g(C),u_C u_C^*)$ and on a morphism $k\colon C \to C^\prime$ in $\bC$ as $\tilde{g}\coloneqq g(k)$, and also get a corresponding commutative diagram for $g, \tilde{g}$ of the form of \eqref{eq_diag_f_factors_Idem}.

We now note that $u$ defines a  unitary multiplier  isomorphism   $\tilde{u}\colon \tilde{f} \to \tilde{g}${.
I}ndeed, we have $\tilde{u}=(\tilde u_{C})_{C\in \Ob(\bC)}$, where 
\[
\tilde u_{C}=q_{C}u_{C}p_{C} \colon (f(C),p_{C})\to (g(C),q_{C})
\]
is a unitary  multiplier isomorphism in $\Idem^{\bM\bD}(\bD)$.
% consisting of unitaries in $\Idem(\bD)$: $\tilde{u}$ is given by the family $(\tilde{u}_C)_{C \in \bC}$ of unitaries $\tilde{u}_C \colon \tilde{f}(C) \to \tilde{g}(C)$ with $\tilde{u}_C \coloneqq u_C$ considered as morphisms in $\Idem(\bD)$. 
By Lemma~\ref{efijoijgoqegrqgqrgqrg} we conclude that  $\Homol(\tilde{f}) \simeq \Homol(\tilde{g})$. This implies $\Homol(\Emb \circ \tilde{f}) \simeq \Homol(\Emb \circ \tilde{g})$. Applying $\Homol$ to the {commutative} square  \eqref{eq_diag_f_factors_Idem}  this equivalence  implies   the equivalence
$\Homol(c\circ f)\simeq \Homol(c\circ g)$.
Since we assume that  $\Homol$ is  Morita invariant
 we know by Propositions \ref{qrgojqeorgqfewfeqwfqewf} and \ref{qrgojqeorgqfewfeqwfqewf1}
  that $\Homol(c)$ is an equivalence. We conclude that
$\Homol(f)\simeq \Homol(g)$.  
\end{proof}

\section{Weak Morita equivalences}\label{eoigwjeoigregewrgwergwerg}

In this section we introduce the notion of a weak Morita equivalence in $\nCcat$ and show that a weak Morita equivalence induces an equivalence in $K$-theory. 
In contrast to the algebraic notion {of} Morita {equivalence} as introduced in Section~\ref{erogijogergergwgerg9} {the} {notion of a} weak Morita {equivalence is}
 of analytic nature. {It involves} the possibility of norm-approximating morphisms in a larger category by morphisms in a smaller one. The typical example of a weak Morita equivalence  is the left upper corner  inclusion of $\C$ into the compact operators on a Hilbert space which is considered as a functor between single-object $C^{*}$-categories.

% which are\Alex{Contrary to Morita equivalences as introduced in Section~\ref{erogijogergergwgerg9}, weak Morita equivalences are not an algebraic notion since they involve norm approximations.} \Alex{That weak Morita equivalences induce an equivalence in $K$-theory is a crucial fact that we use in the companion paper \cite{compass}.}

Let $\bD$ be in $\nCcat$ and  %Recall that $\bD^{u}$ denotes the subcategory of unital objects of $\bD${, i.e., the full subcategory of all objects which have a unit endomorphism}.
% Let
$S$ be a subset of the set of objects of $\bD$.
  
\begin{ddd}\label{q3rioghoergerwgregwg9}
$S$ is weakly generating if
for every $D$ in $\bD$, {any} finite family
$(f_{i})_{i\in I}$ of morphisms $f_{i}\colon D_{i}\to D$ in $\bD$, and any $\varepsilon$ in $(0,\infty)$ 
there exists a multiplier isometry $u\colon C\to D$  in $\bD$ such that $\|f_{i}-uu^{*}f_{i}\|\le \varepsilon$ for all $i$ in $I$ and $C$ is unitarily isomorphic in $\bM\bD$ to a finite orthogonal sum in $\bM\bD$ of objects in $S$.
\end{ddd}

\begin{rem}\label{weitjgowegrefgrewfr}
If $\bM\bD$  admits finite orthogonal sums, then the condition in Definition \ref{q3rioghoergerwgregwg9} can be simplified. In this case it suffices to check that for every morphism
$f:D'\to  D$ in $\bD$ and $\epsilon$ in $(0,\infty)$ there exists 
a multiplier  isometry $u:C\to D$ from an object  which is unitarily isomorphic to a finite sum in $\bM\bD$  of objects of $S$ such that
$\|f-uu^{*}f\|\le \epsilon$.

In fact, given a family $(f_{i})_{i\in I}$   as in the Definition  \ref{q3rioghoergerwgregwg9} we choose an orthogonal sum $(\bigoplus_{i\in I} D_{i}, (e_{i})_{i\in I})$ of the family $(D_{i})_{i\in I}$ in $\bM \bD$. We then consider  the morphism 
$f:=\sum_{i\in I}f_{i} e_{i}^{*}:\bigoplus_{i\in I}D_{i}  \to  D$ in $\bD$. Assume that  $u:C\to \bigoplus_{i\in I}D_{i}$ is a  multiplier isometry    such that $\|f-u u^{*}f\|\le \epsilon$. Then we have
 $$\|f_{i} -   uu^{*} f_{i}\|=\|(f-uu^{*} f )e_{i}\|\le \epsilon$$
 for all $i$ in $I$.
\hB
\end{rem}
  
Let  $\phi\colon \bC\to \bD$ be a {morphism} in $\nCcat$. 

\begin{ddd}\label{ergiowergerfwrfwfref}
The {morphism}
$\phi$ is a weak Morita equivalence if it has the following properties: 
\begin{enumerate}
%\item $\bC$ is unital.
\item $\phi$ is fully faithful.
\item\label{weiorgwegrewfwe}  $\phi(\Ob(\bC))$ is weakly generating.
\end{enumerate}
\end{ddd}

\begin{rem}
%If a morphism $\phi:\bC\to \bD$ in $\Ccat$ is a Morita equivalence with additive target, then it is a weak Morita equivalence. In order to see this note that $\bC$ is unital by assumption, and that a Morita equivalence is fully faithful. By   Remark \ref{weitjgowegrefgrewfr} it suffices to approximate  a single morphism $f:D'\to D$.  Since $\phi$ is a Morita equivalence there exists an isometry $v:D\to C$ where $C$ is unitarily isomorphic to a finite sum of objects in the image of $\phi$. 
%Then $f= v*vf$
%
%
%
The notion of a weak Morita equivalence should not be confused
with the notion of a Morita equivalence. In general, a Morita equivalence need not be a weak Morita equivalence or vice versa, {see Example \ref{rtihortegertgrtget9} below}. Our motivation to use the term \emph{Morita} also in this situation is that a weak Morita equivalence $\phi:\bC\to \bD$ gives
rise to a Morita $(A(\bC),A(\bD))$-bi-module which is at the heart of the proof of Theorem \ref{eowigjwoigwregrgwegrwreg}.
\hB
\end{rem}

\begin{ex}\label{rtihortegertgrtget9}
Let $X$ be a very small set and consider the $C^{*}$-algebra $L^{\infty}(X)$ as an object of $\Ccat$.
Then  the morphism $L^{\infty}(X)\to L^{\infty}(X)^{\sharp}$ in $\Ccat$ is a Morita equivalence. If $X$ has more than one point, then it is not a weak Morita equivalence. In fact, let $Y$ be a proper non-empty subset of $X$. Then we can consider
$D:=(L^{\infty}(X),\chi_{Y})$ as an object of $ L^{\infty}(X)^{\sharp}$, where $\chi_{Y}$ denotes the projection given by the multiplication by the characteristic function of $Y$. We consider the morphism $\id_{D}:D\to D$ in $ (L^{\infty}(X),\chi_{Y})$.
It can not be approximated by morphisms which factorize over objects which are unitarily isomorphic to orthogonal sums of copies of the object $L^{\infty}(X)$. In fact,  if $I$ is finite, but not empty, then there does not exist any isometry
$\bigoplus_{i\in I} L^{\infty}(X)\to D$.

The left-upper corner inclusion $\C\to K(\ell^{2})$ considered as a morphism in $\nCcat$ 
is the prototypical example of a weak Morita equivalence. In fact,   $K(\ell^{2})$ is an ideal in the additive $C^{*}$-category $B(\ell^{2})$, and  therefore by Remark \ref{weitjgowegrefgrewfr} the Condition \ref{ergiowergerfwrfwfref}.\ref{weiorgwegrewfwe}  is equivalent to the condition that
every  element of $ K(\ell^{2})$ can be approximated
by finite-dimensional operators.
But $\C\to K(\ell^{2})$ is not a Morita equivalence since $K(\ell^{2})$ is not unital.  \hB
 \end{ex}

%\begin{ex}
%Let $\phi:\bC\to \bD$ be a functor in $\Ccat$.
%If $\phi$ is a Morita equivalence and $\bC$ is additive and idempotent complete, then  $\phi$ is a weak Morita equivalence. 
%
%In fact $\bC$ is unital by assumption and $\phi$, being a Morita equivalence, is fully faithful. Since $\phi$ is a Morita equivalence  and $\bC$ is additive and idempotent complete it is essentially surjective.
%
% 
%\end{ex}

%Let $\bC$ be  in $\Ccat$,  $\bD$ be in $\nCcat$ and $\phi:\bC\to \bD$ be a morphism in $\nCcat$.  Note that we assume that $\bC$ is unital.

%\Alex{The main result of this section is the following theorem:}
Let  $\phi\colon \bC\to \bD$ be a {morphism} in $\nCcat$. 

\begin{theorem}\label{eowigjwoigwregrgwegrwreg}
If $\phi$ is a weak Morita equivalence, then
\[
\Kcat(\phi)\colon \Kcat(\bC)\to \Kcat(\bD)
\]
is an equivalence.
\end{theorem}

Theorem \ref{eowigjwoigwregrgwegrwreg} has the following consequence which in the unital case  has already been observed in \cite{joachimcat} {and \cite{mitchener_KTh_Ccat}.} Assume that $\phi\colon \bC\to \bD$ is a morphism in $\nCcat$.
\begin{kor}\label{qwrgioqfeqwfqedfefdqe}
If $\phi\colon\bC\to \bD $ is a unitary equivalence, then
$\Kcat(\phi)$ is an equivalence.
 \end{kor}
\begin{proof}
We show that $\phi$ is a weak Morita equivalence. 
%By assumption, $\bC$ is unital. 
Since $\phi$ is a unitary  equivalence it is fully faithful. It remains to show that $\phi(\Ob(\bC))$ is weakly generating. In this case we have a much stronger property:
Let $D$ be an object of $\bD$.  Since $\bM\phi$ is essentially surjective, there exists $C$ in $\bC$ and a unitary multiplier $u\colon \phi(C)\to D'$. Then for any $f\colon D'\to D$ we have $f=uu^{*}f$.
\end{proof}

\begin{rem}
The specialization of the proof of the Theorem \ref{eowigjwoigwregrgwegrwreg} to the special  case considered in Corollary \ref{qwrgioqfeqwfqedfefdqe}  is essentially equivalent to the proof of the assertion of the corollary  given in \cite{joachimcat}.
\hB
\end{rem}

The idea of the proof of Theorem \ref{eowigjwoigwregrgwegrwreg} is to reduce the assertion to the Morita invariance of the $K$-theory of $C^{*}$-algebras. We first recall some of the basic facts, see e.g, \cite[Sec.~5]{arXiv:1006.4975}.

%\begin{rem}\label{evqlkvjoewvqecwcwqecwcewq}

Let $A$ and $B$ be in $\nCalg$.
Recall that a {Hilbert $B$}-module $(H,\langle -,-\rangle_{B})$ is called full if $\langle H,H\rangle_{B}$ is dense in $B$.
 
%Apparently there are two definitions of a Morita $(A,B)$-module.
%Here is the first:
%\begin{ddd}
%A Morita $(A,B)$-bimodule is a pair $(H,\phi)$, where $H$ is a Hilbert $B$-module and $\phi:A\to B(H)$ is a homomorphism such that:
%\begin{enumerate}
%\item $H$ is full (i.e.,  $\overline{\langle H,H\rangle_{B}}=B$)
%\item $\phi$ is an isomorphism onto $K(H)$.
%\end{enumerate}
%\end{ddd}
%Such a bimodule is called 
%imprimitivity bimodule in \cite{Meyer:aa}.
%
%Here is the second:
\begin{ddd}\label{weotgpwergwerrgrgregwger}
A Morita $(A,B)$-bimodule is a triple  $(H,{}_{A}\langle -,-\rangle, \langle-,-\rangle_{B})$, where $H$ is an  $(A,B)$-bimodule, ${}_{A}\langle -,-\rangle$ is an $A$-valued scalar product on $H$ and $ \langle-,-\rangle_{B}$ is a $B$-valued scalar product {on $H$} such that 
\begin{enumerate}
\item $(H ,\langle-,-\rangle_{B})$  is a full  {Hilbert $B$}-module.
\item  $(H, {}_{A}\langle -,-\rangle)$ is a full {Hilbert $A$}-module.
\item For all $h,h',h''$ in $H$  we have the relation
\begin{equation}
{}_{A}\langle h,h'\rangle h''=h\langle h',h''\rangle_{B}\,.
\end{equation}
\end{enumerate}
\end{ddd}

\begin{rem}\label{qrgoihiqoregrfewfeqfqwefqwef}
The datum of a Morita $(A,B)$-bimodule is equivalent to the datum of a 
triple $(H,\langle-,-\rangle_{B},\phi)$  of a {Hilbert $B$}-module $(H,  \langle-,-\rangle_{B})$ together with a homomorphism $\phi\colon A\to B(H)$ such that
\begin{enumerate}
\item $(H,  \langle-,-\rangle_{B})$ is full.
\item $\phi$ is an isomorphism from $A$ to $K(H)$.
\end{enumerate}
In this case one can reconstruct the $A$-valued scalar product by ${}_{A}\langle h,h'\rangle:=\phi^{-1}(\theta_{h,h'})$, where   $\theta_{h,h'}$ is  as in \eqref{ssrgfdsgsreg}. 
In the other direction, assuming the data in Definition \ref{weotgpwergwerrgrgregwger},  the relation 
$$\theta_{h,h'}(h'')=h\langle h',h''\rangle_{B}={}_{A}\langle h,h'\rangle h''$$
   shows that
$\theta_{h,h'}$ is given by the multiplication by an element of $A$. This extends to an isomorphism $\phi$ between $ A$ and $K(H)$.
\hB
\end{rem}

\begin{ddd}
The datum of a  Morita $(A,B)$-bimodule is called a strong Morita--Rieffel equivalence between $A$
and $B$. 
\end{ddd}

\begin{rem} If $A$ and $B$ are unital, then
a strong  Morita{--Rieffel} equivalence between $A$
and $B$ induces an equivalence  \begin{equation}\label{ewfiufheqiwufefeqfqwefqewfqewfqewf}
\Hilb(A)^{\fg,\proj}\ni M\mapsto M\otimes_{A}H\in \Hilb(B)^{\fg,\proj}
\end{equation}
of the topologically enriched categories 
of finitely generated, projective modules over $A$ and $B$.
It is possible  to construct the topological $K$-theory spectrum of $C^{*}$-algebras
from this category in a functorial way. Using such a construction in the background,
a   strong Morita{--Rieffel} equivalence between $A$
and $B$ gives rise to an {equivalence between} $K$-theory spectra $\Kast(A)\to \Kast(B)$.
We will not go into this direction since in the present paper we use the $K$-theory of $C^{*}$-algebras in an axiomatic way and therefore only have functoriality for homomorphisms between $C^{*}$-algebras.
 %do not want to use a specific model for the 
 %For our purpose can use any construction of $\Kast$ which is functorial for homomorphisms.
\hB
\end{rem}

Let $f\colon A\to B$ be a morphism in $\nCalg$.

\begin{ddd} \label{reogijoergegwegwergw}
We say that $f$ induces a strong Morita--Rieffel equivalence if the following conditions are satisfied:
\begin{enumerate}
\item $H \coloneqq \overline{f(A)B}$ with the $B$-valued scalar product given by $(b,b')\mapsto  b^{*}b'$ is a full right {Hilbert $B$}-module.  
\item \label{hwiuvhqivevcqwevcw} $f\colon A\to \End_{B}(H)$ identifies $A$ with $K(H)$.
\end{enumerate}
\end{ddd}
In view of Remark \ref{qrgoihiqoregrfewfeqfqwefqwef} the homomorphism  $f$ gives rise to a strong Morita{--Rieffel} equivalence between $A$ and $B$.  

\begin{lem}\label{qwriufhqeifqfewfffqwfe}
If $f\colon A\to B$  induces a strong Morita{--Rieffel} equivalence, then the induced morphism $\Kast(f)\colon \Kast(A)\to \Kast(B)$ is an equivalence. 
\end{lem}
\begin{proof}
Using Bott periodicity it suffices to check that $\Kast_{*}(f)\colon \Kast_{*}(A)\to \Kast_{*}(B)$ is an isomorphism for $*=0,1$.   The point is now that the well-known isomorphism 
between $\Kast_{*}(A)$ and $\Kast_{*}(B)$ induced by the Morita {$(A,B)$-}bimodule  given in Definition \ref{reogijoergegwegwergw}  is precisely the homomorphism  $\Kast_{*}(f)$.
 \end{proof}

 %we can define  scalar product
%${}_{A}\langle -,-\rangle$ by
%${}_{A}\langle h,h'\rangle:=\phi^{-1}(h  \langle h',-\rangle_{B})$,
%where we consider 
%$h  \langle h',-\rangle_{B})\in K(H)$.
%In the other direction, \uli{not so clear...?}
%
%%\end{rem}

\begin{proof}[Proof of Theorem~\ref{eowigjwoigwregrgwegrwreg}]

We first assume that $\phi\colon \bC\to \bD$ is injective on objects.
Then we have a commutative diagram 
$$\xymatrix{
\Kcat(\bC)\ar[rr]^-{\Kcat(\phi)}\ar[d]^-{\simeq} && \Kcat(\bD)\ar[d]^-{\simeq}\\
\Kast(A(\bC))\ar[rr]^-{\Kast(A(\phi))} && \Kast(A(\bD)) 
}$$
where, using Definition \ref{ergiowergregrdsvsv}, the vertical equivalences are induced by the natural transformation $\alpha:A^{f}\to A$ from \eqref{eq_transformation_Af_A}, {see Lemma~\ref{sfdbjoqi3gegergewgerg1}.}
The assumption on $\phi$ is needed since $A$ is only functorial for morphisms which are injective on objects. It suffices to show that
$\Kast(A(\phi))$ is an equivalence.

We claim that {$A(\phi)$ induces} a {strong Morita--Rieffel} equivalence in the sense of Definition \ref{reogijoergegwegwergw}. {To this end we verify the conditions listed in Definitions \ref{weotgpwergwerrgrgregwger}.}
 Recall that $A(\bD)$ is a closure of the matrix algebra
 $$A^{\alg}(\bD)=\bigoplus_{D,D'\in \bD}\Hom_{\bD}(D',D)\, .$$
 It is an $(A^{\alg}(\bC),A^{\alg}(\bD))$-bimodule.
 Using that $\phi$ is injective on objects,  we can consider 
  the $(A^{\alg}(\bC),A^{\alg}(\bD))$-bimodule
$$H^{\alg} \coloneqq \bigoplus_{C\in \bC,D\in \bD} \Hom_{\bD}(D,\phi(C))\, .$$
%$$M_{0}(\phi):=\bigoplus_{C\in \bC,D\in \bD} \Hom_{\bD}(D,\phi(C))$$
as a sub-bimodule of $A^{\alg}(\bD)$. Its 
elements  will be written as   families $(h_{CD})_{C\in \bC,D\in \bD}$ with finitely many non-zero members. A similar notation will be used for the elements of $A^{\alg}(\bC)$ and $A^{\alg}(\bD)$.
The action of $A^{\alg}(\bC)$ is given by
$$(a h)_{CD} \coloneqq \sum_{C'\in \bC} \phi(a_{CC'})h_{C'D}$$
for all $a$ in $A^{\alg}(\bC)$ and $h$ in $H^{\alg}$.
Similarly, the action of $A^{\alg}(\bD)$ is given by 
$$(hb)_{CD} \coloneqq \sum_{D'\in \bD } h_{CD'}b_{D'D}$$
for all $h$ in $H^{\alg}$ and $b$ in $A^{\alg}(\bD)$.
In this notation the $A(\bD)$-valued scalar product is given by 
$$(\langle h,h'\rangle_{A(\bD)})_{D'D} \coloneqq \sum_{C\in \bD} h^{*}_{CD'}h'_{CD}$$
for all $h,h'$ in $H^{\alg}$.
Furthermore, we define an $A^{\alg}(\bC)$-valued scalar product by
$$({}_{A(\bC)}\langle h,h'\rangle)_{C'C} \coloneqq \phi^{-1}\big(\sum_{D\in \bD} h_{C'D}h^{\prime,*}_{CD}\big)$$
for all $h,h'$ in $H^{\alg}$.
Here we use that $\phi$ is fully faithful.

One checks the relation
\begin{equation}\label{refewewrfwefqfewfqefewf}
{}_{A(\bC)}\langle h,h'\rangle h''=h \langle h',h''\rangle_{A(\bD)}
\end{equation}
for all $h,h',h''$ in $H^{\alg}$.

We let $H$ be the closure of $H^{\alg}$ with respect to the norm induced by the $A(\bD)$-valued scalar product, or equivalently, the closure in $A(\bD)$.  
Then $H$ is a  right  {Hilbert $A(\bD)$-}module. In the notation of Definition \ref{reogijoergegwegwergw} this is $\overline{A(\phi)A(\bD)}$.

We {next show} that  the scalar product $ {}_{A(\bC)}\langle -,-\rangle$ on $H^{\alg}$
extends by continuity to an $A(\bC)$-valued scalar product on $H$.
The relation \eqref{refewewrfwefqfewfqefewf} implies 
$$\|{}_{A(\bC)}\langle h,h'\rangle h''\|= \|h \langle h',h''\rangle_{A(\bD)}\| \le \|h\|\|h'\|\|h''\|\, .$$
All these norms are defined using the $A(\bD)$-valued scalar product and the norm in $A(\bD)$.
Hence we can estimate the operator norm of ${}_{A(\bC)}\langle h,h'\rangle$ on $H$ by
\begin{equation}\label{asfdvqrmgpefvvsf}
\| {}_{A(\bC)}\langle h,h'\rangle \| \le \|h\| \|h'\|\, .
\end{equation} 
For every $C'$ in $\bC$ the module
$H$ contains the closed $(A^{\alg}(\bC),\End_{\bD}(\phi(C')))$-submodule  {$H_{C'}$} generated by $\bigoplus_{C\in \bC}\Hom_{\bD}(\phi(C'),\phi(C))$. This module is isomorphic to the $(A^{\alg}(\bC),\End_{\bC}(C'))$-module generated by $\bigoplus_{C\in \bC} \Hom_{\bC}(C',C)$ (again since $\phi$ is fully faithful). 
It is known by \cite[Sec. 3]{joachimcat} {(see also the proof of \cite[Lem. 6.7]{crosscat} for an argument)} that the maximal norm on $A(\bC)$ is induced by the  {family of modules $(H_{C'})_{C'\in \Ob(\bC)}$.}  %representions of $A^{\alg}(\bC)$ on  modules for running $C'$ in $\Ob(\bC)$.
 It follows that the operator norm on $H$ induces the norm on 
  $A^{\alg}(\bC)$.
  The estimate \eqref{asfdvqrmgpefvvsf}  now implies   
   that $ {}_{A(\bC)}\langle -,-\rangle$ extends by continuity to an $A(\bC)$-valued scalar product  on $H$. Furthermore, the action of $A^{\alg}(\bC)$ on $H^{\alg}$ extends to an action of $A(\bC)$ on $H$ such that $H$ is a pre-Hilbert $A(\bC)$-left module. %\Alex{That it is a Hilbert $A(\bC)$-left module will be shown further below in this proof.}

We now show that \begin{equation}\label{asdvasdacvdsc}
A(\bC)= \overline{{}_{A(\bC)}\langle H,H\rangle}\, .
\end{equation}   
Let $[f_{CC'}]$ be a one-entry matrix in $A(\bC)$.  We consider the one-entry matrices $[h_{C\phi(C)}]$ with $h_{C\phi(C)} \coloneqq \phi(h)$ for $h$ in $\End_{\bC}(C)$ and $[h'_{C'\phi(C)}] \coloneqq \phi(f^{*}_{CC'})$ in $H$.
Then $$ {}_{A(\bC)}\langle [h_{C\phi(C)}],[h'_{C'\phi(C)}]\rangle=[
(hf)_{CC'}]\ .$$ 
We now  use  that $A(\bC)$ is generated by one-entry matrices
and that the linear span of elements of the form
$hf$ for $h$ in $\End_{\bC}(C)$ and $f$ in $\Hom_{\bC}(C',C)$ 
is dense in $\Hom_{\bC}(C',C)$ in order 
 to {conclude} the equality \eqref{asdvasdacvdsc}.

 Up to this point we have used that $\phi$ is fully faithful, but in the following argument
 use the assumption  that $\phi$
 is a weak Morita equivalence.
 We will show that
 $$A(\bD)=\overline{\langle H,H\rangle_{A(\bD)}}\, .$$
Let $f \colon D'\to D$ be a morphism in $\bD$ such that there is an object $C$ in {$\bC$}
and a multiplier  isometry  $u \colon \phi(C)\to D$ such that $f=uu^{*}f$.
We will call such a {morphism}  special.
Let $[f_{DD'}]$ be the one-entry matrix in $A^{\alg}(\bD)$ with $f_{DD'}=f$.   
Then we consider the one-entry matrices
$[h_{CD}]$ in $H$ with $h_{CD} \coloneqq u^{*}v^{*}$ for $v$ in $\End_{\bD}(D)$
and $[h'_{CD'}]$ in $H$ with $h'_{CD'} \coloneqq u^{*}f$.
Then $$\langle [h_{CD}],[h'_{CD'}]\rangle_{A(\bD)}=
[vf_{DD'} ]\, .$$
We claim that one-element matrices with special entries
 generate $A(\bD)$. Since $A(\bD)$ is generated by one-element matrices  and we can choose $v$ arbitrary (e.g. members in an approximate unit of $\End_{\bD}(D)$) it suffices to show that
 special elements generate a dense subspace of $\Hom_{\bD}(D',D)$ for all  objects $D,D'$ in $\bD$.
We consider $f \colon D'\to D$ and $\varepsilon$ in $(0,\infty)$. 
Since $\phi(\Ob(\bC))$ is weakly generating
  there exists a finite family $(C_{i})_{i\in I}$ of objects in $\bC$, the orthogonal sum $(\bigoplus_{i\in I} \phi(C_{i}),(e_{i})_{i\in I})$ in $\bM\bD$, and
   {a multiplier  isometry}
$u\colon \bigoplus_{i\in I} \phi(C_{i})\to D$ such that
$\|f-uu^{*}f\|\le \varepsilon$.
 Then $uu^{*}f=\sum_{i\in I} ue_{i} e_{i}^{*}u^{*}f$.
 The  summands  $ ue_{i}e_{i}^{*}u^{*}f$ are special. Hence $uu^{*}f$ is a finite sum of special elements.

We now show that the {pre-Hilbert  $A(\bC)$}-module $H$ is actually a {Hilbert $A(\bC)$}-module. 
 We let $\|-\|'$ denote the norm on $H$ induced by the $A(\bC)$-valued scalar product.
 We will show that $\|-\|$ is equivalent to $\|-\|'$, where $\|-\|$ is the norm on $H$ induced by the $A(\bD)$-valued scalar product. We then use  that $H$ is complete with respect to $\|-\|$ by construction. 
 
 From
  \eqref{asfdvqrmgpefvvsf}  we get the estimate
  $\|-\|'\le \|-\|$.
  By  \eqref{refewewrfwefqfewfqefewf} we get
$$\|h\langle h',h'\rangle_{A(\bD)}\|=\|{}_{A(\bC)}\langle 
h,h'\rangle h'\| \le \|h'\|\|{}_{A(\bC)}\langle 
h,h'\rangle\| \le \|h'\| \|h\|' \|h'\|'\, .$$
Taking the supremum over all $h$ in $H$ with $\|h\| \le 1$ %\uli{(note that this condition implies 
%{(and hence
% $\|h\|' \le 1$)}
we conclude that
\begin{equation}\label{fewfqlknkfqefqewf}
\| \langle h',h'\rangle_{A(\bD)}\|'' \le  \|h'\| \|h'\|'\, ,
\end{equation}
where $\|-\|''$ is the norm on $A(\bD)$ induced from the operator norm on $H$.
We claim that $\|-\|''$ is equal to the norm of $A(\bD)$.
The claim together with \eqref{fewfqlknkfqefqewf}
then implies that $\|h'\|\le \|h'\|'$ for all $h$ in $H$ and hence $\|-\|\le \|-\|'$.
  
We now show the claim that $\|-\|''$ is equal to the norm of $A(\bD)$. Let $b$ be in $A^{\alg}(\bD)$ such that $\|b\|=1$. We have to show that $\|b\|''=1$. For every $D'$ in $\bD$ we let $M_{D'}$ be the right Hilbert $A(\bD)$-module generated by ${M_{D'}^{\alg}:=}\bigoplus_{D\in \bD} \Hom_{\bD}(D,D')$. It is a direct summand of $A(\bD)$. We choose $\varepsilon$ in $(0,\infty)$. 
{Again by} \cite[Sec. 3]{joachimcat}  the family  of modules $(M_{D'})_{D'\in \bD}$  induces the norm on $A(\bD)$.
 Hence
  there exists $D'$ in
${\bD}$ and $m$ in ${M_{D'}^{\alg}}$
such that $\|m\|\le 1$ and $\|mb\|\ge  1-\varepsilon/2$.
  Note that the number $R$ of  non-zero members of the family 
  $m=(m_{D'D})_{D\in \bD}$   is finite. We furthermore have
    $\|m_{D'D}\|\le 1$ for all $D$ in $\bD$. {Since $\phi(\Ob(\bC))$ is weakly generating,}
  there  exists a finite family of objects $(C_{i})_{i\in I}$ in $\bC$, a pair  $(E,(e_{i})_{i\in I})$, $e_{i}\colon \phi(C_{i})\to E$, representing the orthogonal sum  of the family $(\phi(C_{i}))_{i\in I}$ in $\bM\bD$, and a multiplier isometry $u\colon E\to D'$
 such that $\|m_{D'D}-uu^{*}m_{D'D}\|\le \frac{\varepsilon}{2(R+1)}$ for all $D$ in $\bD$.
 Then $\|m-uu^{*}m\|\le \varepsilon/2$.
 We consider the right Hilbert {$A(D)$-}module $M_{E}$.   We note that $u$ induces an isometry $M_{E}\to M_{D'}$. We set $m' \coloneqq u^{*}m$ in $M_{E}$.
 Then we have
 $$\|m'b\|=\|um'b\|=\|uu^{*}mb\|\ge \|mb\|-\|(m-uu^{*}m)b\|\ge 1-\varepsilon\, .$$
For every $i$ in $I$ we have an isometric inclusion of right Hilbert
$A(\bD)$-modules $f_{i}\colon M_{\phi(C_{i})}\to M_{{E}}$
sending $(m_{{\phi(C_{i})}D})_{D\in \bD}$ to $ (e_{i}m_{{\phi(C_{i})}D})_{D\in \bD}$.   Hence
we get an isometric inclusion
\[
f\coloneqq \oplus_{i\in I} f_{i}\colon M_{E}\to \bigoplus_{i\in I} H\,.
\]
The  diagonal representation of $A(\bD)$ on
$ \bigoplus_{i\in I} H$ induces the same norm as the representation on $H$. 
We then have $\|f(m')b\| {\ge} 1-\varepsilon$. Since {$\|m'\|=\|u^{*}m\|\le \|m \|\le 1$} this implies that $\|b\|''\ge 1-\varepsilon$.
Since $\varepsilon$ was arbitrary we conclude that $\|b\|''=1$. 

This finishes the verification that $(H,{}_{A(\bC)}\langle -,-\rangle,\langle-,-\rangle_{A(\bD)})$ is an $(A(\bC),A(\bD))$-Morita bimodule, {and that $\Kast(A(\phi))$ induces a strong Morita--Rieffel equivalence.}
By Lemma \ref{qwriufhqeifqfewfffqwfe} we can conclude that
$\Kast(A(\phi)):\Kast(A(\bC))\to \Kast(A(\bD))$ is an equivalence.
This implies the assertion of Theorem \ref{eowigjwoigwregrgwegrwreg} and therefore also of Corollary \ref{qwrgioqfeqwfqedfefdqe}
for functors $\phi$ which are injective on objects.

We finally drop  the assumption that $\phi$ is injective on objects.
Let $\phi\colon  \bC\to  \bD$ be a weak Morita equivalence.
 Then we form $  \bE$ in  $\nCcat$ as follows:
\begin{enumerate}
\item objects: The set of objects of $  \bE$ is given by $\Ob(  \bC)\sqcup \Ob(  \bD)$.
\item morphisms: \[\Hom_{\bE}(E,E'):=\begin{cases}
\Hom_{\bC}(E,E') & \text{for }E,E' \in \bC\,,\\
 \Hom_{\bD}(\phi(E),E') & \text{for } E\in \bC, E' \in \bD\,,\\
  \Hom_{\bD}(E,\phi(E')) & \text{for } E\in \bD, E' \in \bC\,,\\
    \Hom_{\bD}(E,E') & \text{for } E,E'\in \bD\,.
\end{cases} \]
\item {composition and involution}: these structures are defined in the canonical way.
\end{enumerate}
 We have  inclusions 
 $$i\colon \bC\to \bE\, , \quad j\colon \bD\to   \bE$$
 and a projection $p\colon \bE\to \bD$ such that $p\circ j=\id_{\bD}$
 and $p\circ i=\phi$.
 Moreover, there is an obvious unitary multiplier isomorphism $j\circ p\cong \id_{\bE}$. We conclude that   $p$  is a unitary equivalence and therefore $\Kcat(p)$ is an equivalence.  Moreover, $i$ is again  a weak Morita equivalence which is in addition injective on objects.
 By the special case already shown, $\Kcat(i)$ is an equivalence. Hence $\Kcat(\phi)\simeq \Kcat(p)\circ \Kcat(i)$ is an equivalence.
\end{proof}
 
%\Alex{We finish this section with the discussion of an example that occurs in the companion paper \cite{compass}.}

\begin{ex}\label{qrgiooergwefe9}
Let $A$ be  in $\Calg$ and consider the wide subcategory ${\Hilb_{c}(A)}$ of compact morphisms in $\Hilb(A)$, {cf.~Example~\ref{ex_hilbert_modules}.}  Recall that  ${\Hilb_{c}(A)^{u}}$ denotes the full subcategory of  ${\Hilb_{c}(A)}$ of unital objects, see Definition \ref{ebgiojoigrevsfdvsdfvsfdvsdfv}.   A Hilbert  $A$-module is in ${\Hilb_{c}(A)^{u}}$ if and only if it is algebraically finitely generated, and all such modules are projective \cite[Ex.~15.O and Cor.~15.4.8]{wegge_olsen}.
Considering $A$ itself as an object of $\Hilb_{c}(A)^{u}$ we get the inclusion
$A\to  \Hilb_{c}(A)^{u}$. %Recall the Definition \ref{ebgiojoigrevsfdvsdfvsfdvsdfv} of the subcategory of unital objects $\uli{\Hilb_{c}(A)^{u}}$ in $\uli{\Hilb_{c}(A)}$. 
%A Hilbert  $A$-module is in ${\Hilb_{c}(A)^{u}}$ if and only if it is algebraically finitely generated, and all such modules are projective \cite[Ex.~15.O and Cor.~15.4.8]{wegge_olsen}.

We let $\Hilb(A)_{\std}$ be the full subcategory of $\Hilb(A)$ of objects which are isomorphic to   classical orthogonal  sums (see Construction \ref{ejirgowergrefweerf})  of   {very small} families of objects of ${\Hilb_{c}(A)^{u}}$ and set ${\Hilb_{c}(A)_{\std}} \coloneqq{\Hilb_{c}(A)}\cap \Hilb(A)_{\std}$.
Note that  ${( \Hilb_{c}(A)_{\std})^{u}} ={\Hilb_{c}(A)^{u}}$. We further have the following commutative diagram of inclusion functors
\begin{equation}
\xymatrix{
A\ar[rr]^-{ } \ar[dr]_-{\text{Morita}} &&{{\Hilb_{c}(A)}_{\std}}\,.\\
&{\Hilb_{c}(A)}^{u}\ar[ur]_-{\ \ \text{weak Morita}} &
}
\end{equation}
Applying $K$-theory  and using  the Theorems \ref{wtoijgwergergrewfwergrg} and \ref{eowigjwoigwregrgwegrwreg}  we obtain equivalences 
\begin{equation}
\xymatrix{
\Kast(A)\ar[rr]^-{\simeq} \ar[dr]_-{\simeq} && {\Kcat({\Hilb_{c}(A)_{\std}})}\,.\\
&\Kcat({\Hilb_{c}(A)}^{u})\ar[ur]_-{\simeq} & 
}
\end{equation}
These equivalences will be used in companion paper \cite{bel-paschke}.
%The full subcategory $\bK^{u}$  of unital objects (i.e. finitely generated projective modules) is additive.
%%%
%Since $A$ is unital we can consider $A$ as an object of  \uli{$\bK(A)^{u}$} in a natural way.
%%%
%
%
%The set $\{A\}$ is weakly generating for $\bK(A)_{\std}$.
% 
%The full subcategory $\bK(A)^{u}_{\std}$  of unital objects (i.e. finitely generated, projective modules)    contains $A$ and is hence   also weakly generating \Alex{for $\bK(A)_{\std}$.} \Alex{Note that $\bK(A)^{u}_{\std} = \bK(A)^{u}$.}
% 
%Hence the inclusions $A\to \bK(A)_{\std}$ and  $\bK(A)^{u}\to \bK(A)_{\std}$  are weak Morita equivalences, \Alex{and therefore they induce the equivalences on $K$-theory
%$ \Kast(A)\to  \Kcat(\bK(A)_{\std})$ and $\Kcat(\bK(A)^{u})\to \Kcat(\bK(A)_{\std})$.}
\hB
\end{ex}

\section{Functors on the orbit category}\label{qroigjqowefwefwfwefqwef}
\label{qeroigjeoigergergewgwegewrgwerg}

For a group  $G$ we consider the orbit category  $G\Orb$  of transitive $G$-sets and equivariant maps.  It  plays a fundamental role in $G$-equivariant homotopy theory{. By} Elmendorf's theorem \cite{elmendorf} {(and subsequent work thereon)} the category $\PSh(G\Orb)$ models the equivariant homotopy theory of $G$-topological spaces. For  a cocomplete $\infty$-category $\bS$ the $\infty$-category of $\bS$-valued equivariant homology theories     is equivalent to the $\infty$-category of functors from $G\Orb$ to   $\bS$, see Section \ref{oiejgoregrqrgqgregw}.  Such functors  are the main ingredients of assembly maps, see  e.g.\ \cite[Sec.\ 1]{desc} for more information.
The present section {is about the  construction of such functors starting from the datum of a  $C^{*}$-category with a $G$-action.}
{Much of  theory developed in the preceding sections will be employed to calculate the values of the resulting  functors. The outcomes will be further used in the subsequent papers \cite{coarsek} and \cite{bel-paschke}.}

Our first construction  uses the homotopy theory of unital $C^{*}$-categories modeled by the    Dwyer--Kan localization \begin{equation}\label{qwefwewdqwedwedqeedqwed}
\ell \colon \Ccat\to \Ccat_{\infty}
\end{equation}
of $\Ccat$ at the unitary equivalences % see Section \ref{eogijwegerergwg} {and} 
\cite{startcats}. 

Let $\Homol \colon \Ccat\to \bS$ be a functor which sends unitary equivalences to equivalences. Our main example is the restriction of $ \Kcat$ to unital $C^{*}$-categories. 
By the universal property of the Dwyer--Kan localization it  has an essentially unique factorization
\begin{equation}\label{234toijoig23gergwregreg}
\xymatrix{\Ccat\ar[rr]^-{\Homol}\ar[dr]_-{\ell}&&\bS\,.\\
&\Ccat_{\infty}\ar[ur]_-{\Homol_{\infty}}&}
\end{equation}

We can consider the set $G$ with the left action as an object  of $G\Orb$.
The right-action of $G$ on itself induces an isomorphism of monoids $G\cong \End_{G\Orb}( G)$. We therefore have an   embedding of categories \begin{equation}\label{wefjqwoeifjoqwefqwefwefqefqewfq}
j^{G} \colon BG\to G\Orb
\end{equation} which sends
the unique object $*_{BG}$ of $BG$ to the {left} $G$-set $ G$. We let 
  $j_{!}^{G}$ denote the left Kan extension functor  along $j^{G}$.
  
  \begin{ddd}
  We define the functor
  $$\Homol^{G}_{\infty}\colon \Fun(BG,\Ccat_{\infty})\to \Fun(G\Orb,\bS)\ , \quad \bC_{\infty}\mapsto \Homol^{G}_{\infty,\bC_{\infty}}:= \Homol_{\infty}\circ j_{!}^{G}(\bC_{\infty})\ .$$
  \end{ddd}
We also use the symbol \[
\ell  \colon \Fun(BG,\Ccat)\to \Fun(BG,\Ccat_{\infty})
\] be  the {functor}
  for the  post-composition with $\ell$ {from \eqref{qwefwewdqwedwedqeedqwed}}.
 Given a unital $C^{*}$-category with $G$-action $ {\bC}$ in $\Fun(BG,\Ccat)$ we   define a functor
\begin{equation}\label{fqwwefojiofqwefqwefqewf}
j_{!}^{G}(\ell ( {\bC})) \colon G\Orb\to \Ccat_{\infty}\, .
\end{equation}
%At this point the order of applying the localization and the Kan extension is crucial in order to get the desired values. 
If 
 $H$ is a subgroup of $G$, then one can calculate the value of the functor \eqref{fqwwefojiofqwefqwefqewf} at  the object $G/H$ in $G\Orb$.

\begin{lem}\label{qroigjqofrgqrgwergreg}
We have an equivalence 
$$j_{!}^{{G}}(\ell ( {\bC})){(G/H)} \simeq \ell( {\bC}\rtimes H)\, .$$
\end{lem}
\begin{proof}
We use the pointwise formula for the left Kan extension which gives 
$$j_{!}^{{G}}(\ell ( {\bC}))(G/H)\simeq \colim_{BG_{/G/H}}  \ell ( {\bC}) \, .$$
We consider the functor $BH\to B{G}_{/G/H}$ which sends the unique object $*_{BH}$
of $BH$ to the projection map $  G\to G/H$ considered as a object of the slice category  $B{G}_{/G/H}$, and the morphism $h$ in $H=\End_{BH}(*_{BH})$  to the endomorphism of $G\to G/H$   given by right-multiplication with $h^{-1}$ on $G$. This functor is an 
 equivalence  of categories.  We can therefore replace the slice category in the index of the colimit by $BH$.
  We further observe that 
  the restriction of the functor $
  \ell  ({\bC})$ along $BH\to B{G}_{/G/H}$ is given by $\ell (\Res^{G}_{H}( {\bC}))$, where $\Res^{G}_{H} \colon \Fun(BG,\Ccat)\to\Fun(BH,\Ccat)$ is the restriction {of the group action.}  We therefore get the equivalence $$\colim_{BH} \ell (\Res^{G}_{H}( {\bC}))\simeq  \colim_{BG_{/G/H}}  \ell ( {\bC})\ .$$
   In \cite[Thm.\ 7.8.{2}]{crosscat} we have
  seen that $$\colim_{BH} \ell (\Res^{G}_{H}( {\bC}))\simeq \ell( {\bC}\rtimes H)\, ,$$
  where $-\rtimes H$ denotes the maximal crossed product. Combining the three displayed equivalences we get the  equivalence asserted in the lemma.
  \end{proof}
  
%  \begin{rem}
%  For most constructions in this first part we could start with an object $  \bD_{\infty}$ in $\Fun(BG,\Ccat_{\infty})$ in place if $\ell_{BG}(  \bD)$. The reason to prefer to start with the choice of $  \bD$  in $\Fun(BG,\Ccat)$ 
%  is precisely that in this case by Lemma \ref{qroigjqofrgqrgwergreg} we have a preferred representative of $j_{!}^{{G}}(\ell_{BG}({ {\bD}}))$ given by the maximal crossed product. See Remark \ref{qrgoijwoegwergregw} for a continuation {of this discussion}.
%\hB 
%\end{rem}
%
  
  \begin{ddd}\label{we98gu9egergwergwe9}
  For $ {\bC}$ in $\Fun(BG,\Ccat)$   
we define the functor
\begin{equation}\label{eq_defn_functor_max_orbit}
\Homol^{G}_{{\bC},\max} \coloneqq \Homol^{G}_{\infty, \ell ({\bC)}} \colon G\Orb\to \bS\, .
\end{equation} 
%Composing $j^{G}_{!}(\ell_{BG}(\tilde{\bD}))$ with $\Homol_{\infty}$ we get  a functor
\end{ddd}
By Lemma \ref{qroigjqofrgqrgwergreg} its value on  the orbit $G/H$ is given by
\begin{equation}\label{qrwgojoiqrfwefqwefqwef}
\Homol^{G}_{{\bC},\max}(G/H)\simeq \Homol({\bC}\rtimes H)\, .
\end{equation}
The subscript $max$ indicates that the values of this functor involve the maximal crossed product.
 
% \begin{rem}\label{qrgoijwoegwergregw}
% We actually have defined a functor
% $\Fun(BG,\Ccat)\to \Fun(G\Orb,\bM)$,
% which sends $ \bD$ to $\Homol^{G}_{ \bD,\max}$. 
% By construction this functor only depends on $\ell_{BG}( \bD)$ and therefore has a canonical factorization over $\ell_{BG} \colon \Fun(BG,\Ccat)\to \Fun(BG,\Ccat_{\infty})$. Hence for $ \bD_{\infty}$ in $ \Fun(BG,\Ccat_{\infty})$
% we have a well-defined functor
%$$\Homol_{\bD_{\infty},\max}^{G} \coloneqq \Homol_{\infty}(j^{G}_{!}(\bD_{\infty})) \colon G\Orb\to \bM\, .$$
%This will be relevant for the  statement of Proposition \ref{qrfiuwhiufewfqwefqwef}.
%\hB
%\end{rem}

The homotopy theoretic construction of $\Homol^{G}_{  \bC,\max}$ has the advantage that it is easy to derive some of its formal properties. As an example, 
the next proposition states the compatibility of the construction of $\Homol^{G}_{  \bC,\max}$ above with the induction along the inclusion of $G$ into a larger 
  group $K$. 
  We have a 
  commutative diagram of categories
  \begin{equation}\label{qewfqwefewfewfewfqe}
\xymatrix{BG\ar[r]^{j^{G}}\ar[d]^{i}&G\Orb\ar[d]^{i^{K}_{G}}\\
BK\ar[r]^{j^{K}}&K\Orb}
\end{equation}
where $i \colon BG\to BK$ is given by applying $B$ to the inclusion of $G$ into $K$, and $i^{K}_{G}$ sends the $G$-orbit $S$ to the $K$-orbit $K\times_{G}S$. 
For a functor $E^{G} \colon G\Orb\to \bS$ we let $E^{G}(X)$ also denote the value of the corresponding $\bS$-valued equivariant homology theory (given by Elemendorf's theorem) on the $G$-topological space $X$. Furthermore, we let $i^{K}_{G,!}$ denote the left Kan extension functor along $i^{K}_{G}$.

Let $  \bC_{\infty}$ be in $\Fun(BG,\Ccat_{\infty})$ and $\Homol\colon \Ccat\to \bS$ be a functor. %to a cocomplete   target which sends unitary equivalences to equivalences.

\begin{prop}\label{qrfiuwhiufewfqwefqwef} 
Assume:
\begin{enumerate}
\item $\bS$ is cocomplete.
\item $\Homol$ sends unitary equivalences to equivalences.
\item $\Homol$ preserves small coproducts. 
\end{enumerate}
Then we have the following assertions:
 \begin{enumerate}
\item \label{qeroigjoqergqgqfqewfqewf}
We have an equivalence $\Homol^{K}_{\infty,i_{!} \bC_{\infty}}\simeq i^{K}_{G,!} \Homol_{\infty, \bC_{\infty}}^{G}$ of functors from $K\Orb$ to $\bS$.
\item \label{qeroigjoqergqgqfqewfqewf1} For every $K$-topological space  $X$ we have
$\Homol^{K}_{\infty,i_{!}  \bC_{\infty}}(X)\simeq \Homol^{G}_{ \infty, \bC_{\infty}}(\Res^{K}_{G}(X))$.
\end{enumerate}
\end{prop}
\begin{proof}
We first show that $\Homol_\infty$ preserves small coproducts. Then the claims of the proposition will be
%These facts are 
consequences of general considerations that will be given in the appendix to this section.

The Dwyer--Kan localization $\ell\colon \Ccat\to \Ccat_{\infty}$ of $\Ccat$ at the set of unitary equivalences
% in the realm of $\infty$-categories. This $\infty$-category 
is modeled by a combinatorial 
model category structure on $\Ccat$   (for details we refer to \cite{DellAmbrogio:2010aa}, {see also}\cite{startcats}). 
%By the universal property of the Dwyer--Kan localization the functor $\Homol$ has an essentially unique factorization
%\begin{equation}
%\xymatrix{\Ccat\ar[rr]^-{\Homol}\ar[dr]_-{\ell}&&\bM\,.\\
%&\Ccat_{\infty}\ar[ur]_-{\Homol_{\infty}}&}
%\end{equation}
Since in  {this model category structure} all objects of $\Ccat$ are cofibrant, for any small  family 
$(\bC_i)_{i \in I}$ in $\Ccat$  the canonical morphism
$$\coprod_{i\in I} \ell(\bC_{i})\to \ell\big(\coprod_{i\in I} \bC_{i}\big)$$
is an equivalence.
 %an object $X$ of $\Ccat$ is the coproduct of a family if and only if $\ell(X)$ is a coproduct in $\Ccat_\infty$ of the family $(\ell(X_i))_{i \in I}$. 
 It follows that $\Homol$ preserves small  coproducts if and only if $\Homol_\infty$ preserves small  coproducts.

We now turn to the actual proof of the proposition.
Because ${\Homol_{\infty}}$ preserves small coproducts, applying Lemma \ref{q3roigerwgwergwergwerg} with $B={\Homol_{\infty}}$ and $A=j^{G}_{!} \bC_{\infty}$
we get
$$i^{K}_{G,!} \Homol_{ \infty, \bC_{\infty}}^{G}\simeq \Homol_{\infty}\circ i^{K}_{G,!}j^{G}_{!}  \bC_{\infty}\, .$$
We now use the {commutative} square \eqref{qewfqwefewfewfewfqe} {and the functoriality of Kan extension functors}  in order to rewrite the right-hand side 
$$ \Homol_{\infty}\circ i^{K}_{G,!}j^{G}_{!}  \bC_{\infty}\simeq \Homol_{\infty}\circ j^{K}_{!} i_{!}  \bC_{\infty}\simeq \Homol^{K}_{\infty,i_{!}  \bC_{\infty}}\, .$$
The concatenation of {these} two equivalences gives the  equivalence asserted in 
 \ref{qeroigjoqergqgqfqewfqewf}.
Assertion~\ref{qeroigjoqergqgqfqewfqewf1} is now an immediate consequence of Assertion  \ref{qeroigjoqergqgqfqewfqewf} and Lemma \ref{wejiogwergergwergwerg}.
\end{proof}
 {\begin{ex}If $\Homol$ is a finitary Morita invariant homological functor, then  the assumption on $\Homol$ in Proposition \ref{qrfiuwhiufewfqwefqwef} is satisfied by Corollary \ref{kor_Morita_finitary_preserves_coprods}. By the Theorems  \ref{qrgkljwoglergerggwergwrg} and  \ref{wtoijgwergergrewfwergrg} this applies  e.g. to $\Kcat$ in place of $\Homol$.\hB \end{ex}}

For an application of Proposition \ref{qrfiuwhiufewfqwefqwef} see  Proposition \ref{eroighqeirgoregwergwergewrgewrgwe} below.

%We now explain what we will do in the remainder of of the present section.
We can apply  the construction {of $\Homol^{G}_{\bC,\max}$}   to a unital $C^{*}$-algebra {$A$} with $G$-action in place of $ {\bC}$.    {If $A$ has a  trivial $G$-action} one could try to compare  {$\Homol^{G}_{\bC,\max}$} with the functor {$\Homol^{\DL,G}_{ A}:G\Orb\to \bS$} constructed  following ideas of Davis--L\"uck \cite{davis_lueck}, {see Construction  \ref{weijgowergrewferff}.}   {An immediate difference between these functors  is that 
 the Davis--L\"uck functor satisfies
  $\Homol^{\DL,G}_{ A}(G/H)\simeq \Homol(A\rtimes_{r}H)$, i.e., it involves 
   the reduced crossed product instead of the maximal one as   \eqref{qrwgojoiqrfwefqwefqwef} .} 
  
  {In the remainder of the present section we} construct a functor
  $$\Homol^{G}_{ {\bC},r}  \colon G\Orb\to \Sp$$
  whose values on orbits $G/H$ are given by \begin{equation}\label{wefqwefewfwqdqdqwedewd}
\Homol^{G}_{ {\bC},r}{(G/H)}\simeq \Homol( {\bC}\rtimes_{r} H)\, .
\end{equation}
 We furthermore provide 
    a comparison map
   $$c \colon \Homol^{G}_{ {\bC},\max}\to \Homol^{G}_{ {\bC},r}$$
   whose component at $G/H$ is   
   the canonical morphism between the maximal and reduced crossed products. 
%In particular, the restriction of $c$ to $G_{\Fin}\Orb$ is an equivalence. 
{If $A$ is a unital $C^{*}$-algebra with trivial $G$-action we will also provide an equivalence between  the Davis--L\"uck functor $\Homol^{\DL,G}_{ A}$ and $\Homol^{G}_{\Hilb_{c}(A)^{u},r}$.}

\begin{construction}\label{qrfiohqrioffdewfeqwdqewdqed}{\em 
%For the followi%ng  construction $\bC$ could be any  object of $\Fun(BG,\Ccat)$ which admits all very small orthogonal sums.  
We consider  $\bC $ in $\Fun(BG,\Ccat)$.  We assume that $\bC$ admits finite orthogonal sums. This assumption implies that a finite sum of mutually orthogonal effective projections is again an effective projection. We  introduce a functor 
\begin{equation}\label{veqvwefvfvvvsfv}
\bC[-]\colon \Fun(BG,\Set)\to \Fun(BG,\Ccat)\, ,
\end{equation}
where $\Set$ is the  small category of  very small sets. 
%\Alex{In the following we will also write $G\Set$ for $\Fun(BG,\Set)$.}
 %We assume that $\bC$ is idempotent complete.  
%\begin{ddd}\label{qrfiohqrioffdewfeqwdqewdqed}\mbox{}
\begin{enumerate}
\item objects: For  $  X$   in $\Fun(BG,\Set)$ we  define $\bC[  X]$  in $\Fun(BG,\Ccat)$ as follows:
\begin{enumerate} 
\item\label{ergiojwergergwgregw} objects: The  objects of $ \bC[  X]$ are pairs $(C,(p_{x})_{x\in X})$ of an object
$C$ of $\bC$ and a commuting and mutually orthogonal family of  effective projections $p_{x}$ in $\End_{\bC}(C)$ such that  its support
$$\supp(C,(p_{x})_{x\in X}):=\{x\in X\:|\: p_{x}\not=0\}$$ is finite   and $C$ is isomorphic to the orthogonal  sum of the images of the family $(p_{x})_{x\in X}$  (see Definition \ref{rgoigsgsgfgg}).   \item morphisms: A morphism
\[A\colon (C,(p_{x})_{x\in X})\to (C',(p'_{x})_{x\in X})\]
in $  \bC[  X]$  is  a morphism $A\colon C\to C'$ in $\bC$ such that  for all $x,x'$ we have 
$p'_{x}Ap_{x}=0$  unless $x=x'$.
\item composition and involution: These structures are inherited from $\bC$.
\item  The group $G$ acts on $  \bC [X]$ {by}
\[g(C,(p_{x})_{x\in X}) \coloneqq (gC,(gp_{g^{-1}x})_{x\in X})\, .\]
The action of $G$ on morphisms is inherited from $  \bC$.  
\end{enumerate} 
\item morphisms: For a morphism   $f\colon   X\to  X'$  in $\Fun(BG,\Set)$  we define the morphism $\bC[f]\colon \bC[ X]\to  \bC[X']$ in $\Ccat$ as follows:
\begin{enumerate}
\item objects: The functor $\bC[f]$ sends the object $(C,(p_{x})_{x\in X})$  of $ \bC[ X]$ to  the object
$(C,(p_{x'})_{x'\in X'})$ of $ \bC[ X']$, where
\begin{equation}\label{fqweiuhuihfiqwefqwefqwefqew}
p_{x'} \coloneqq \sum_{x\in f^{-1}(\{x'\})} p_{x}\ .
\end{equation}
Since $\bC$ is finitely additive we see that $p_{x'}$ is again   an effective projection. % by Lemma \ref{ethgiowgergwergerg}. 
 \item morphisms:  The functor $\bC[f]$ sends a morphism  $A$  from $(C,(p_{x})_{x\in X})$ to $ (C',(p'_{x})_{x\in X})$ in $ \bC[X]$ to the same morphism $A\colon C\to C'$ considered as a morphism from $\bC[f]((C,(p_{x})_{x\in X}))$ to $\bC[f]((C',(p'_{x})_{x\in X}))$  in $ \bC[ X']$.
\end{enumerate}
\end{enumerate}
   If $f:X\to X'$ is a map in $\Fun(BG,\Set)$, then we have
\begin{equation}\label{fdvadivojdfoivasdvasdvdvdvav}
\supp(\bC[f]((C,(p_{x})_{x\in X})))\subseteq f(\supp((C,(p_{x})_{x\in X})))\, .
\end{equation}  

%Idempotent completeness is used to construct $\bC[f]$  since for the application of Lemma~\ref{etgoijoergergwegwergwrg} we must know that the projections $\sum_{x\in f^{-1}(\{x'\})} p_{x}$ are effective for all $x'$ in $X'$.

Let  $\Fun(BG,\Set)_{i}$ denote the wide subcategory of $\Fun(BG,\Set)$ of morphisms which are injective. If we drop the assumption that $\bC$ is additive, then  the construction above still gives a functor  $ \bC [-]:\Fun(BG,\Set)_{i}\to \Fun(BG,\Ccat)$.  

The construction $\bC\mapsto (X\mapsto \bC[X])$ extends to a functors $$\Ccat\to \Fun(\Fun(BG,\Set)_{i},\Fun(BG,\Ccat)) $$  and $$\Ccat^{f\add}\to \Fun(\Fun(BG,\Set),\Fun(BG,\Ccat)) $$ in the obvious way, where $f\add$ stands for the full subcategory of finitely additive $C^{*}$-categories.
}\hB
\end{construction}

\begin{ex}\label{weoigjwioegergwergerg}
Let $C$ be an object of $\bC$ and $y$ be a point in $X$. Then we consider the object $C_{x}$ in $\bC[X]$ given by $(C,(p^{y}_{x})_{x\in X})$, where $p^{y}_{x}=0$ for all $x$ in $X$ except for $x=y$ where $p^{y}_{y}=\id_{C}$. We say that $C_{y}$ is the object $C$ placed at the point  $y$ in $X$. We have $\supp(C_{y})=\{y\}$. 

{If $(C,p)$ with $p=(p_{x})_{x\in X}$ is a general object of $\bC[X]$, then we can choose images $(C(x),u_{x})$ in $\bC$ of the projections $p_{x}$ for all $x$ in $X$. We then observe that  $((C,p),(u_{x})_{x\in X})$ is an orthogonal sum in $\bC[X]$ of the family $(C(x)_{x})_{x\in X}$ of objects in $\bC[X]$.}
\hB
\end{ex}

   \begin{rem} Let $\bK$ be in $\Fun(BG,\nCcat)$ admit all very small orthogonal AV-sums.
In \cite[(7.4)]{coarsek} we introduce unital $C^{*}$-categories categories $\tilde{\bar{\bK}}^{\mathrm{ctr}}_{\mathrm{lf}}(X)$ of objects in $\bM\bK$ which are controlled by   $G$-bornological coarse spaces $X$.  Then by  \cite[Prop. 9.2.1]{coarsek}
%by \cite[Prop.~8.2]{coarsek}
the functor $ \bK^{u}[-] $ defined  in Construction \ref{qrfiohqrioffdewfeqwdqewdqed}  
%in Definition \ref{rguihqrfewdqedqwdq}
{is equivalent}  to the functor $  \tilde{\bar \bK}^{\mathrm{ctr}}_{\mathrm{lf}}((-)_{\textit{min},\textit{max}})$.
%It is the functor  $\bar \bC_{\mathrm{lf}}^{\mathrm{ctr}}((-)_{min,max})$  defined in \cite[Def. 4.21]{coarsek} (for the trivial group) with the $G$-action induced by functoriality from the action on $\bC$ and $X$, see    \cite[(4.6)]{coarsek}.  
\hB
\end{rem}

%\begin{rem}
%We denote by $\Fun(BG,\Set)_{i}$ the wide subcategory of $\Fun(BG,\Set)$ of morphisms which are injective. If $\bK^{u}$ is not hereditarily additive in $\bC$, then ${}_{\bK}\bC_{\mathrm{lf}}[-]$ \uli{still} gives rise to a functor 
% $\Fun(BG,\Set)_{i}\to {\Fun(BG,\Ccat)}$. 
% %This situation occurs in the examples at the end of this section.
%\hB
%\end{rem}

%Since we   assume that $\bC$ is $|G|$-additive, the $C^{*}$-category 
%  $\bC[ X]$ is also $|G|$-additive.  Indeed, if  $(C_{i},(p_{i,x})_{x\in X})_{i\in I}$ is a family of objects in $\bC[ X]$ with $|I|\le |G|$, then we can choose a sum
%  $(\bigoplus_{i\in I}C_{i}, (e_{i})_{i\in I})$
%  in $\bC$. Using Proposition \ref{lem_sum_characterization_morphisms} one then  checks that $$(\uli{(}\bigoplus_{i\in I}C_{i}, (\oplus_{i\in I} p_{i,x})_{x\in X}\uli{)}, (e_{i})_{i\in I})$$ represents the sum of this family in $\bC[ X]$. If $f \colon X\to  X$ is a morphism in $\Fun(BG,\Set)$, then $\bC[f]$ preserves these sums.
% The functor $\bC[-]$ from \eqref{veqvwefvfvvvsfv} therefore defines  a functor from $\Fun(BG,\Set)$ to $\Fun(BG,\widehat{\Ccat}^{|G|\add})$ (see Section \ref{qeoirghowergerwgwregwergw} for notation). 
 %We now apply the reduced crossed product functor from Theorem  \ref{ejgwoierferfewrferfwe} to $\bC[-]$. % in order to get a functor
%\begin{equation}\label{f32f3334f}
%\bC[-]\rtimes_{r} G\colon \Fun(BG,\Set) \to \Ccat\, .
%\end{equation}

  We consider an additive   $\bC$  in  $\Fun(BG,\Ccat)$. 
  % For the next definition we assume that   $\bK^{u}$ is hereditarily additive in $\bC$.  Note that the set of objects of $  \bC[  X]\rtimes_{r} G$ is the set of objects of $ \bC[  X]$.
  \begin{ddd} \label{ergfevbfsvfvsvfsdvsfvs9} %\fuli{AV: define the crossed product directly, not larger category needed}
  We define the functor
  $$ \bC [-]\rtimes_{r}G:=(-)\rtimes_{r}G\circ \bC[-]\colon \Fun(BG,\Set)\to \Ccat$$  
  %as the subfunctor of $\bC[-]\rtimes_{r} G$ which
   %sends $  X$ in $\Fun(BG,\Set)$ to the full subcategory 
 % ${}_{\bK} \bC_{\mathrm{lf}}[  X]\rtimes_{r}G$  of $ \bC[  X]\rtimes_{r} G$ 
  %on the objects of ${}_{\bK}\bC_{\mathrm{lf}}[ X]$.
  \end{ddd}

%\begin{rem}\label{eqrgioqwfqewfewqddeqwd}
%Note that $ {{}_{\bK}}\bC_{\mathrm{lf}}[  X]\rtimes_{r}G$ is an abuse of notation. The reduced crossed   product ${}_{\bK}\bC_{\mathrm{lf}}[  X]\rtimes_{r}G$ is not defined intrinsically. It  \uli{may depend} on the choice of the $|G|$-additive surrounding category $\bC[ X]$. The more precise notation would be ${}_{\bK}\bC_{\mathrm{lf}}[  X]\rtimes_{r}^{\bC[ X]}G$ (see Remark \ref{3rgiuheruigweqfqwfqwefeqwfqewfqwefqwef} and Lemma \ref{lem_reduced_crossed_product_independent}).
% %But from Lemma~\ref{lem_reduced_crossed_product_independent} we know that ${}_{\bK}\bC_{\mathrm{lf}}[\tilde X]\rtimes_{r}G$ is essentially independent of the choice of the $|G|$-additive surrounding category $\bC[\tilde X]$.}
%
%
%
%The point of introducing the local finiteness condition by means of the  subcategory  $  \bK$ is to prevent
%the resulting category ${}_{\bK}\bC_{\mathrm{lf}}[ X]\rtimes_{r}G$ from being   countably additive, hence flasque  and  {therefore} $K$-theoretically uninteresting (Example~\ref{ex_countably_additive_flasque} and Proposition \ref{ergiwoegwergwergwgrg}).
%%The definition of the reduced crossed product requires $|G|$-additivity. So   \Alex{cf.\ Remark~\ref{3rgiuheruigweqfqwfqwefeqwfqewfqwefqwef}. But from Lemma~\ref{lem_reduced_crossed_product_independent} we know that ${}_{\bK}\bC_{\mathrm{lf}}[\tilde X]\rtimes_{r}G$ is essentially independent of the choice of the $|G|$-additive surrounding category $\bC[\tilde X]$.}
%\hB
%\end{rem}

\begin{rem}\label{rem_restricted_functor}
{If $\bC$ is not   additive, then $ \bC [-]\rtimes_{r}G$ is  still defined  as a functor from  $\Fun(BG,\Set)_i$ to $\Ccat$.}
\hB
\end{rem}

% We  now calculate the value $ \bC [ {G/H}]\rtimes_{r}G $ for $H$ a subgroup of $G$. %Recall that 
 % $  \bC$  is $\Fun(BG,\Ccat)$ and that $  \bK$   a closed, $G$-invariant subcategory of $  \bC$. We assume that $\bC$ is $|G|$-additive and idempotent complete.
   %\Alex{Note that for the following proposition we do not need to assume that $\bK^{u}$ is hereditarily additive in $\bC$.}
%Assume that $\bC$ is $|G|$-additive.   
% Note that we do not require that $\bK^{u}$ is  additive in any sense.
%By abuse of notation we define the crossed product
%\begin{equation}\label{ewfqqefqwefqwefqewdqqrwefqwef}
%  \bK^{u}\rtimes_{r} H \coloneqq   \bK^{u}\rtimes^{ \bC}_{r} H\, ,
%\end{equation}
%(see Remark~\ref{3rgiuheruigweqfqwfqwefeqwfqewfqwefqwef} for the decoration of the crossed product). 
 Recall  the notion of a Morita equivalence from Definition \ref{wetgoijwtoigwerrefwerfwefwref}. Let $\bC$ be in $\Fun(BG,\Ccat)$.
%Let $H$ be a subgroup of $G$.
\begin{prop}\label{wetgiojweirogwergrwegwergw}
For every subgroup $H$ of $G$ we
%: \begin{enumerate}
%\item  
 %$\bC$ is idempotent complete, then
%\item $\Homol$ is Morita invariant (Definition \ref{ewrgiuhwerogwregrefwferfwrf}). \footnote{We only need that $\Homol$ to sends additive completions and equivalences to equivalences.}
%\end{enumerate} 
  have a  Morita equivalence
 \begin{equation}\label{sthrtshdfgvfdsgsfgfd}
 i_{H}:\Res^{G}_{H}( \bC)\rtimes_{r}H\to  \bC [ {G/H}]\rtimes_{r}G\,.
\end{equation}\end{prop}
\begin{proof}
%\uli{Exercise: Show that
%Put this in contrast to the functor
%$\Kcat(J^{G}(\tilde \bC)):G\Orb\to \Sp$ whose  value  on $G/H$  is  $$\Kcat(H\rtimes  \Res^{G}_{H}\tilde \bC)$$ (maximal crossed product), where $J^{G}$ is as in \cite[Def. 14.10]{startcats}. Which one appears in the Baum-Connes assembly map?}
%\Alex{By abuse of notation we write $\tilde{\bC}\rtimes^{\alg}H$ for $\Res^{G}_{H}(\tilde{\bC})\rtimes^{\alg}H$.}
%For better readability we will not write the functor $\Res^{G}_{H}$.
%We start with the description of the   {morphism}   
%$$i^{\alg}\colon  \Res^{G}_{H}( \bC) \rtimes^{\alg}H\to \bC[ {G/H}]\rtimes^{\alg} G$$ {in $\Ccat$.}
%\begin{enumerate}
%\item {objects:} $i^{\alg}$ sends the object $C$ of $ \Res^{G}_{H}( \bC)\rtimes^{\alg}H$ 
%(i.e., an object of $\bC$) to the object $C_{eH}$ 
% in $\bC [ {G/H}]\rtimes^{\alg} G$ (i.e.,  {the object $C$ placed at $eH$ in}  $\bC [ {G/H}]$), see Example \ref{weoigjwioegergwergerg}. 
%\item {morphisms:} $i^{\alg}$ sends a morphism $(f,h)\colon C\to C'$  in  $   \Res^{G}_{H}( \bC)\rtimes^{\alg}H$ with $f\colon C\to h^{-1}C'$  (see  {\eqref{qewfiuhuiqhefiuqwefqewfeqwfwefqwf}} for notation) to the morphism $(f,h)\colon C_{eH}\to   C'_{eH}$ in $\bC [ {G/H}]\rtimes^{\alg} G$. This is well-defined since
%$h^{-1}(C_{eH})  = (h^{-1}C')_{eH}  $. 
%\end{enumerate}
%The  {morphism} $i^{\alg}$ is injective on objects and morphisms. It identifies $  \Res^{G}_{H}( \bC)\rtimes^{\alg}H$ with the full subcategory of $\bC[ {G/H}]\rtimes^{\alg} G$ of objects which are supported on the class $H$ in $G/H$.
%

We let $k\colon \bD\to \bC[ G/H] $ denote the  inclusion of the full $G$-invariant subcategory of  $\bC[ G/H] $ of objects which are supported
on a single point of $G/H$.   We then have a $H$-equivariant inclusion
$j\colon\Res^{G}_{H}(\bC)\to \Res^{G}_{H}(\bD)$ which identifies $\Res^{G}_{H}(\bC)$ with  the full subcategory of objects supported on the class $H$ in $G/H$. 
We   define $i_{H}$ as the composition $$i_{H}\colon \Res^{G}_{H}(\bC)\rtimes_{r}H
\stackrel{j\rtimes_{r}G}{\to} \Res^{G}_{H}(\bD)\rtimes_{r}H\stackrel{\ell}{\to}  \bD\rtimes_{r}G
\stackrel{k\rtimes_{r}G}{\to} \bC[G/H]\rtimes_{r} G\ ,$$
where $\ell$ is induced by the inclusion of $H$ into $G$, see Proposition \ref{efjviofbsdfbsdbsdfvvsdfvs}.
The following assertions imply that $i_{H}$ is a Morita equivalence:
\begin{enumerate}
 \item  \label{wrgoibjowrgbfgbgfw}$j\rtimes_{r}G$ is fully faithful.
\item \label{wrgoibjowrgbfgbgfw1} $\ell$ is isometric.
\item \label{wrgoibjowrgbfgbgfw2}  $\ell \circ (j\rtimes_{r}G)$ is full.
\item \label{wrgoibjowrgbfgbgfw3} $\ell \circ (j\rtimes_{r}G)$ is essentially surjective.
\item \label{wrgoibjowrgbfgbgfw4} $k\rtimes_{r}G$ is a Morita equivalence.
\end{enumerate}
In fact, the first four assertions  together imply that $\ell \circ (j\rtimes_{r}G)$
is a unitary equivalence so that $i_{H}$ is the composition of a Morita equivalence and a unitary equivalence and hence itself a unitary equivalence.

In order to see Assertion \ref{wrgoibjowrgbfgbgfw}  note that $j$ is fully faithful, and therefore $j\rtimes_{r}G$ is fully faithful by Theorem \ref{ejgwoierferfewrferfwe}. 

For Assertion \ref{wrgoibjowrgbfgbgfw1}  note that  the morphism $\ell$ is isometric by Proposition \ref{efjviofbsdfbsdbsdfvvsdfvs}. 

We now show Assertion  \ref{wrgoibjowrgbfgbgfw2}.
Let $C,C'$ be objects of $\Res^{G}_{H}(\bC)\rtimes_{r}H$ (  i.e.,  objects of $\bC$)
and $\sum_{g\in G} (f_{g},g)$ be a morphism 
  $ C_{eH}\to   C_{eH}'$ in $\bD\rtimes_{r}G$ (see Example \ref{weoigjwioegergwergerg} for notation), where $f_{g}:C_{eH}\to g^{-1}C'_{eH}$.  For $g\not\in H$   we have $\supp(g^{-1}C'_{eH})=g^{-1}H\not=H$ and hence 
  $f_{g}=0$. Since by the first two assertions
$\ell\circ (j\rtimes_{r}G)$ is isometric, $\sum_{g\in H} (f_{g},g)$ converges  in $\Res^{G}_{H}(\bC)\rtimes_{r}H$ and provides    a morphism $C\to C'$  which is the desired preimage.

 In order to show Assertion  \ref{wrgoibjowrgbfgbgfw3}
we consider an object   of $\bD\rtimes_{r}G$.
It is of the form $C_{g H}$ for some object $C$ of $\bC$ and $g$ in $G$. Then $(\id_{gC},g^{-1}):(g^{-1} C)_{eH}\to C_{g H}$ is a unitary isomorphism 
in $\bD\rtimes_{r}G$ from an object in the image of $ \ell \circ (j\rtimes_{r}G)$.

%
%
% By restricting the target, $i^{\alg}$ induces a morphism $$j^{\alg}\colon  \bC\rtimes^{\alg}H\to  \bD^{\alg}$$  in $\Clincat$. We claim that $ j^{\alg}$ is a unitary equivalence. {Since $j^{\alg}$ is fully faithful it remains to show that it is essentially surjective. Any object of $\bD^{\alg}$ is of the form $C_{\ell H}$ for some object $C$ of $\bC$ and $\ell $ in $G$. Then $(\id_{\ell C},\ell^{-1}):(\ell^{-1} C)_{eH}\to C_{\ell H}$ is a unitary isomorphism from an object in the image of $j^{\alg}$ to
%$C_{\ell H}$.} Using Proposition \ref{efjviofbsdfbsdbsdfvvsdfvs} we see that
%$j^{\alg}$ extends 
%by continuity to a  unitary equivalence $j\colon  \bC\rtimes_{r}H\to  \bD$.
%We define $$i:\bC\rtimes_{r}H\to \bD\stackrel{k}{\to} \bC[G/H]\rtimes_{r}G\ ,$$
%where $k$ is the canonical inclusion. 

It remains to show Assertion \ref{wrgoibjowrgbfgbgfw4}.  %that $k$ is a Morita equivalence.
  We will actually show the stronger statement  
    that every object in $  \bC [ {G/H} ]\rtimes_{r}  G$
  is isomorphic to a finite orthogonal sum of objects in $  \bD\rtimes_{r}  G$. 
  %This implies that
%\begin{equation}
%\label{eq_k_oplus_unitary_equiv}
%k_{\oplus}\colon ({}_{\bK}\bar \bD)_{\oplus}\to ({}_{\bK}\bC_{\mathrm{lf}}[ {G/H} ]%\rtimes_{r}  G)_{\oplus}
%\end{equation}
%is already a unitary equivalence.
   Let $(C,p^{C})$ be an object of $ \bC [ {G/H} ]\rtimes_{r}  G$.
 We choose images $(C(gH),u_{gH})$ of the   projections $p^{C}_{gH}$ for all $gH$ in the finite set $ \supp(C,p^{C})$.    {Then $C(gH)_{gH}$ (see Example \ref{weoigjwioegergwergerg}) belongs to $   \bD\rtimes_{r}  G$ and 
  $((C,p^{C}),(u_{gH},e)_{gH\in \supp(C,p^{C})})$ is  an orthogonal sum of the finite family $(C(gH)_{gH})_{gH\in \supp(C,p^{C})}$ in $  \bD\rtimes_{r}  G$. %  .
  }\end{proof}

  %Furthermore, $\supp(C,p^{C})$
 % only finitely many of these images are non-trivial. %, and 
  %we have $\id_{C(gH)}\in \End_{\bK}(C(gH))$ for all $gH$ in $G/H$.
%  By  Definition \ref{qrfiohqrioffdewfeqwdqewdqed}.\ref{ergiojwergergwgregw} we have an isomorphism $$\sum_{gH\in \supp(C,p^{C})} u_{gH}e_{gH}^{*}\colon \bigoplus_{gH\in  \supp(C,p^{C})} C(g\uli{H})\to C$$
%  in $\bC$.  For any $gH$ in  $ \supp(C,p^{C})$
%  we \uli{have the  object  $C(g\uli{H})_{gH}$ %, p^{gH})$ 
%  in ${}_{\bK}\bar \bD$, see Example \ref{weoigjwioegergwergerg}.} %, where   $p^{gH}=(p^{gH}_{\ell H})_{\ell H\in G/H}$ is such that
%  %$p^{gH}_{\ell H}=0$ if $\ell H\not=g H$, and $p^{gH}_{g H}=\id_{C(gH)}$.
%  We then have a unitary  isomorphism
%  $$\sum_{gH\in  \supp(C,p^{C})} (u_{gH}e_{gH}^{*},e)\colon \bigoplus_{gH\in  \supp(C,p^{C})} \uli{C(gH)_{gH}}%, p^{gH})
%  \to (C,p^{C})
%  $$
%in ${}_{\bK}\bC_{\mathrm{lf}}[ {G/H} ]\rtimes_{r}  G$. 
%  Therefore $(C,p^{C})$ is unitarily  isomorphic to a finite  orthogonal sum of objects of ${}_{\bK}\bar \bD$.
 
% Let $\Homol\colon \nCcat\to \bM$ be a  functor. Recall that we are considering
% a unital  $|G|$-additive idempotent complete $C^{*}$-category $  \bC$ with strict $G$-action in $\Fun(BG,\Ccat)$ with an invariant closed subcategory $ \bK$. We in addition
% assume  that  $\bK^{u}$ is hereditarily additive {(Definition~\ref{ddd_hereditarily_additive})} in $\bC$.
  Note that $G\Orb$ is a full subcategory of $\Fun(BG,\Set)$. 
  We consider   $\bC$  in  $\Fun(BG,\Ccat)$ and assume that it is additve. 
  \begin{ddd}\label{iuhquifhiwefqewfqwefqwefef}
We define the  functor   $\Homol_{  \bC,r}^{G}\colon G\Orb\to \bS$ as the composition
\begin{equation}\label{eq_defn_functor_reduced_orbit}
\Homol_{  \bC,r}^{G}\colon
G\Orb\to \Fun(BG,\Set)\xrightarrow{ \bC [-]\rtimes_{r}G} \Ccat\xrightarrow{\Homol}
\bS\,.
\end{equation}
\end{ddd}
Using the functoriality of the construction $\bC\to \bC[-]$ with respect to the  $C^{*}$-category $\bC$ we see that 
  we actually have constructed a functor
$\Homol^{G}_{r}:\Ccat\to \Fun(G\Orb,\bS)$.

The next corollary of Proposition \ref{wetgiojweirogwergrwegwergw} 
shows that the values   of functor constructed above are indeed as desired.% in   \eqref{wefqwefewfwqdqdqwedewd}.}  %the notation announced in Remark \ref{gwioejgoergergwegwergwerg}.

\begin{kor}\label{kor_computation_orbit_Morita}
If $\Homol$ is Morita invariant,  then  for every subgroup $H$ of $G$ we have an equivalence 
$$\Homol(i_{H})\colon \Homol(  \Res^{G}_{H}(\bC)\rtimes_{r}H)\stackrel{{\simeq}}{\to} \Homol_{  \bC,r}^{G}(G/H)    \, .$$
\end{kor}

Let $\Homol \colon \Ccat\to \bS$ be a functor   which 
sends unitary equivalences to equivalences. 
{In the case $\Homol=\Kcat$ we will use the more readable notation $\K^{G}_{\bC,{\max}} \coloneqq (\Kcat)^{G}_{  \bC,{\max}}$ for the functor in \eqref{eq_defn_functor_max_orbit}, and $\K^{G}_{  \bC,r} \coloneqq (\Kcat)^{G}_{ \bC,r}$ for the functor in \eqref{eq_defn_functor_reduced_orbit}.}  {For   a family   $\cF$ of subgroups of $G$ we let $G_{\cF} \Orb$ denote the full subcategory of  $G\Orb$ of transitive $G$-sets  with stabilizers in $\cF$.}
Let   $\bC$ be in  $\Fun(BG,\Ccat) $ and assume that it is additive.

\begin{prop}\label{oiqrjgpoiregreewegegw}\mbox{}
\begin{enumerate}
\item \label{reoigjewrogwergewrgwerg}
There  is a  canonical natural transformation $c \colon \Homol^{G}_{  \bC,\max}\to 
\Homol^{G}_{  \bC,r}$.
\item \label{wgwegrecewcs} If $\Homol$ is a Morita invariant, then the evaluation of $c$ at $G/H$  corresponds 
 under the equivalences from Corollary \ref{kor_computation_orbit_Morita} and \eqref{qrwgojoiqrfwefqwefqwef}  to the canonical morphism
$$\Homol(q_{\bC}):  \Homol(  \bC\rtimes H)\to \Homol( \bC\rtimes_{r} H)\ ,$$ see \eqref{eq_q}.
\item\label{reoigjewrogwergewrgwerg1}
If $\Homol$ is {Morita invariant} and every member of $\cF$ is amenable, then 
$$c_{|G_{\cF}\Orb} \colon ( \Homol^{G}_{ \bC,\max})_{|G_{\cF}\Orb}\to 
(\Homol^{G}_{ \bC,r})_{|G_{\cF}\Orb}$$
%the restriction of the natural transformation  $c$ from Assertion \ref{reoigjewrogwergewrgwerg}  to 
%$G_{\Fin}\Orb$ 
%$G_{\cF}\Orb$ 
 is an equivalence.
\item\label{prop_item_Haagerup_equiv_functor}
If every member of $\cF$ is $K$-amenable, then % the restriction of 
$$c_{|G_{\cF}\Orb} \colon (\K^{G}_{  \bC,{\max}})_{|G_{\cF}\Orb} \to (\K^{G}_{  \bC,r})_{|G_{\cF}\Orb}$$
% to $G_{\cF}\Orb$ 
is an equivalence.
\end{enumerate}
\end{prop}

{The main difficulty in the construction of the transformation $c$ is that its domain and target are constructed in very different manners. In fact, the domain of $c$ is given by an  $\infty$-categorical   left Kan extension functor, while the target is given by an explicit one-categorical construction. }
{Before we start the {actual} proof of Proposition~\ref{oiqrjgpoiregreewegegw} we  {therefore} prove two intermediate  assertions.} {The main outcome is Lemma \ref{qeroighwoiegwergwergwregwerg} providing a   one-categorical model of the $\infty$-categorical left Kan extension $j^{G}_{!}(\ell_{BG}( \bD))$. 
The idea is, of course, to relate the $\infty$-categorical construction with the version of \eqref{eq_defn_functor_reduced_orbit} with the reduced crossed product replaced by the maximal one.
}

We let $G'$ be a second copy of $G$.   Then we can form the   functor  
$\phi\colon G\Orb\to \Fun(BG',\Set)$
which  sends $S$ in $G\Orb$ to $S$ considered as a $G'$-set.  {Using the exponential law we interpret $\phi$ as a functor  $$\phi\colon G\Orb\times BG'\to \Set\ .$$}{We consider the group $G$ as  an object $\tilde G$} in $\Fun(BG \times BG',\Set)$, where $G'$-action is the right-action and  the $G$-action is the left action on $\tilde G$. 

We let $\delta \colon \Set\to \Spc$ denote the canonical  functor and for any category $\cC$ we also write $\delta:\Fun(\cC,\Set)\to \Fun(\cC,\Spc)$ for the functor give by post-composition with $\delta$. Finally 
recall the embedding of categories $j^{G} \colon BG\to G\Orb$ from \eqref{wefjqwoeifjoqwefqwefwefqefqewfq}. %which sends
%the unique object $*_{BG}$ of $BG$ to the $G$-set $\tilde G$.}
 \begin{lem}\label{woitjowergergwergreg}
We have an equivalence  \begin{equation}\label{vewfoihvjiefvwevrev}(j^{G}\times\id_{BG'})_{!}\delta (\tilde G)\simeq \delta (\phi)  
\end{equation}
  in $\Fun(  \Orb\times BG', \Spc)$.  
\end{lem}
\begin{proof}
The inverse map $g\mapsto g^{-1}$ on $\tilde G$ induces an isomorphism
  $\tilde G\stackrel{\cong}{\to} (j^{G}\times \id_{BG'})^{*}\phi$ in $\Fun(BG\times BG',\Set)$.  We  get the morphism
$$(j^{G}\times \id_{BG'})_{!}\delta (\tilde G)\stackrel{{\simeq}}{\to} (j^{G}\times \id_{BG'})_{!} (j^{G}\times \id_{BG'})^{*}\delta (\phi)\stackrel{\text{counit}}{\to} \delta (\phi)\, .$$
We must show that the counit is an  equivalence. To this end we calculate {its evaluation} at $G/H$ in $G\Orb$ and get
\begin{eqnarray*}
 ((j^{G}\times \id_{BG'})_{!} (j^{G}\times \id_{BG'})^{*}\delta (\phi))(G/H)&\simeq & \colim_{(G\to G/H)\in  BG_{/G/H}}  \delta(\phi)(G) \\&\simeq& \colim_{BH}\delta (\phi)(G) \\& \simeq& \delta (\phi)(G/H)  \, ,
\end{eqnarray*}
{where} for the last equivalence we use that $H$ acts freely on $\tilde G$ {from the right} and that therefore we can calculate the colimit over $BH$  before applying $\delta $.
 \end{proof}

Since $\Ccat$ has {all} coproducts it is tensored over $\Set$. For $\bD$ in $\Ccat$
the functor $\bD\otimes -\colon \Set\to \Ccat$ is essentially uniquely determined by an  isomorphism 
$\bD\otimes *\cong \bD$  and  the property that it preserves coproducts.
If  $  \bD$ is in $ \Fun(BG,\Ccat )$ and
$  S$ is in $\Fun(BG,\Set)$, then we can consider $  \bD\otimes S$ 
in $\Fun(BG,\Ccat)$ {using the diagonal action}.

Similarly, the $\infty$-category   $\Ccat_{\infty}$  is cocomplete and hence tensored over $\Spc$. For $\bD_{\infty}$ in $\Ccat_{\infty}$ the functor $\bD_{\infty}\otimes - \colon \Spc\to \Ccat_{\infty}$ is essentially uniquely determined by an equivalence $\bD_{\infty}\otimes *\simeq \bD_{\infty}$ and the property that it preserves colimits.
If $  \bD_{\infty}$ is in  
 $\Fun(BG,\Ccat_{\infty})$ and if $  X$ is $\Fun(BG,\Spc)$, then we can consider 
 $  \bD_{\infty}\otimes   X$  in  $\Fun(BG,\Ccat_{\infty})$. 
 
 The functor $\delta \colon \Set \to \Spc$ preserves coproducts.
  Since the localization $\ell \colon \Ccat\to \Ccat_{\infty}$ also preserves coproducts {(this was explained in the  proof of Proposition \ref{qrfiuwhiufewfqwefqwef})}, {for all $\bD$ in $\Ccat$} and $S$ in $\Set$
 we have a {canonical} equivalence $\ell(\bD\otimes S)\simeq \ell(\bD)\otimes \delta(S)$. 
 {Similarly, for all $\bD$ in $ \Fun(BG,\Ccat)$ and $S$ in $\Fun(BG,\Set)$ we have a canonical equivalence}
\begin{equation}\label{eq_l_commute_delta}
\ell( \bD\otimes  S)\simeq \ell( \bD)\otimes \delta(  S)
\end{equation}
in $\Fun(BG,\Ccat_{\infty})$.
%for {all} $\tilde \bD$ in $\Fun(BG,\Ccat)$
%and $\tilde S$ in $\Fun(BG,\Set)$.

 Let $  \bD$ be in $\Fun(BG,\Ccat)$. We write $  \bD' $ in $\Fun(BG',\Ccat )$  for $  \bD $ considered with
the $G'$-action.  
 \begin{lem}\label{qeroighwoiegwergwergwregwerg}
 We have an equivalence \begin{equation}\label{qeroighwoiegwergwergwregwerg111}
j^{G}_{!}(\ell( \bD))\simeq \ell(( \bD'\otimes \phi)\rtimes G')
\end{equation}
   in $\Fun(  G\Orb, \Ccat_{\infty})$.
 \end{lem}
\begin{proof}

We have the equivalence
\begin{equation}\label{erfowjfoiwerfwerfwrefwerf}
\ell (  \bD')\otimes \delta (\tilde G)\stackrel{\eqref{eq_l_commute_delta}}{\simeq}  \ell ( \bD'\otimes \tilde G) {\stackrel{\simeq}{\to}}  \ell ( \bD\otimes \tilde G)\stackrel{\eqref{eq_l_commute_delta}}{\simeq}  \ell ({ {\bD}})\otimes \delta (\tilde G) 
\end{equation}
in $\Fun(BG\times BG',\Ccat_{\infty})$, 
where the middle equivalence is given by $ (C,h)\mapsto (h C,h)$, and  where we implicitly consider, e.g., $\bD'$ in $\Fun(BG\times BG',\Ccat)$ with the trivial $G$-action.
It sends the   diagonal action of $G'$  to the right action of $G'$ on $\tilde G$, and the    left action of $G$ on $\tilde G$ to the diagonal action. 
We have $$\colim_{BG'} \delta  (\tilde G)  \simeq *$$  since $G'$ acts freely on $\tilde G$ so that we can calculate the colimit before going from sets to spaces.  
Applying $\colim_{BG'}$  to \eqref{erfowjfoiwerfwerfwrefwerf}  
we get the equivalence 
$$\colim_{BG'}(\ell (  \bD')\otimes  \delta (\tilde G))\simeq \ell (  \bD)$$
in $\Fun(BG,\Ccat_{\infty})$.
We now apply {$j^{G}_{!}$}
%$(j^{G}\times \id_{BG'})_{!}$
and use that this left Kan extension functor preserves colimits  {to} get
\begin{eqnarray}\label{qoihjoiwefqwfqewfq}
\colim_{BG'}   (j^{G}\times \id_{BG'})_{!}      (\ell ( \bD')\otimes \delta (\tilde G))&\simeq&j^{G}_{!} \colim_{BG'}(\ell ( \bD')\otimes  \delta (\tilde G))\nonumber\\&\simeq& j^{G}_{!}(\ell ( \bD))\, .
\end{eqnarray}
Finally, using Lemma \ref{woitjowergergwergreg} and that  $\ell ( \bD')\otimes -$ preserves colimits
%and Lemma \ref{woitjowergergwergreg}
we can rewrite the domain of \eqref{qoihjoiwefqwfqewfq}
as \begin{eqnarray*}
\colim_{BG'}  (j^{G}\times \id_{BG'})_{!}(\ell ( \bD')\otimes  \delta (\tilde G) )&\stackrel{\eqref{vewfoihvjiefvwevrev}}{ \simeq}& \colim_{BG'}(\ell (   \bD') \otimes \delta  (\phi))\\&\stackrel{\eqref{eq_l_commute_delta}}{\simeq} & \colim_{BG'}\ell (   \bD'  \otimes  \phi) \\
&\simeq& \ell ((  \bD'\otimes \phi) \rtimes G')\, ,
\end{eqnarray*}
where for the last equivalence we use \cite[Thm.\ 7.8]{crosscat}
\end{proof}

\begin{proof}[Proof of Proposition~\ref{oiqrjgpoiregreewegegw}]
%We now come back to our original situation and specialize the above results to 
%$  \bD \coloneqq   \bK^{u}$. 
We write $\bC'$ for $\bC$ considered with the action of $G'$.
We define a transformation $$\nu \colon \bC^{\prime}  \otimes \phi\to  \bC [\phi(-)] $$
 of functors from $G\Orb$ to $\Fun(BG',\Ccat)$. Note that for $T$ in $G\Orb$ an object of $ \bC^{\prime }\otimes \phi(T)$ is given by a pair $(C,t)$ of an object $C$ of $\bC$ and a point $t$ in $T$.   {Recall that $C_{t}$  in $  \bC [\phi(T)]$ denotes the object $C$ placed at the point $t$, see  Example \ref{weoigjwioegergwergerg}.}
 \begin{enumerate}
 \item objects: The  evaluation $\nu_{T}$ of $\nu$ at $T$ sends the object $(C,t)$ in  $   \bC^{\prime }\otimes \phi(T)$   to the object $C_{t}$ in $ \bC [T]$.  
 \item morphisms: A  {non-zero} morphism  $(C,t)\to (C',t')$ in $  \bC^{\prime}\otimes \phi(T)$  only exists if $t=t'$.   {A morphism $(C,t)\to (C',t)$ is} given by a morphism $f \colon C\to C'$ in $\bC$.
 The   evaluation $\nu_{T}$ of $\nu$ at $T$ sends this morphism to the morphism $f_{t} \colon C_{t}\to C_{t}'$.
 \end{enumerate} 
One checks %in a straightforward manner 
that $\nu_{T}$ is  a well-defined morphism between $C^{*}$-categories and $G'$-equivariant. Furthermore, the family $\nu=(\nu_{T})_{T\in G\Orb}$ is a natural transformation.
We get an induced transformation
\begin{equation}\label{qefjqoiwejfoiqjwefqwefqwefqwef}
\nu\rtimes G \colon (  \bC^{\prime}  \otimes \phi)\rtimes G'\to  \bC [\phi(-)]\rtimes G'\stackrel{!}{\to}  \bC [\phi(-)]\rtimes_{r} G'\cong  \bC [-]\rtimes_{r} G\, ,
\end{equation}
where the marked natural transformation is the transformation from the maximal to the reduced crossed product \eqref{eq_q}.
We furthermore apply
$\Homol_{\infty}\circ \ell_{G\Orb}$ and get  the transformation\begin{eqnarray}
c\colon \Homol^{G}_{  \bC ,\max}&\stackrel{\text{Def.} \ref{we98gu9egergwergwe9}}{\simeq} &
\Homol_{\infty}(j^{G}_{!}(\ell ( \bC))) \label{quirhfgiuerfqfweewfqdweq}\\&\stackrel{\text{Lem.} \ref{qeroighwoiegwergwergwregwerg} }{\simeq}& \Homol_{\infty}( \ell ((  \bC^{\prime }  \otimes \phi)\rtimes G')) \nonumber\\&\stackrel{\Homol_{\infty}( \ell (\eqref{qefjqoiwejfoiqjwefqwefqwefqwef}))}{\to}  & \Homol_{\infty}( \ell (  \bC [-]\rtimes_{r} G)) \nonumber \\&\stackrel{\eqref{234toijoig23gergwregreg}}{\simeq}& 
\Homol (  \bC [-]\rtimes_{r} G) \nonumber\\
&\stackrel{\text{Def.} \ref{iuhquifhiwefqewfqwefqwefef}}{\simeq}&
\Homol^{G}_{\bC,r} \nonumber\,.
\end{eqnarray}
This finishes the construction of the morphism $c$ in Assertion \ref{reoigjewrogwergewrgwerg}.

We now show  Assertion \ref{wgwegrecewcs}. Recall the  morphism $q_{\bC}$ from \eqref{eq_q} and $i_{H}$ from \eqref{sthrtshdfgvfdsgsfgfd}.
We have the following commutative  diagram in $\Ccat_{\infty}$:
\begin{equation}\label{vwervwevwvdvsfdvsdv}  
\xymatrix{\ar@/^0.75cm/[rr]^{!!}\ell(\bC\rtimes H)\ar[r]^-{!}\ar[d]^{\ell(q_{\bC})}&    \ell( \bC [G/H] \rtimes G) \ar[d]^{\ell(q_{\bC [G/H})}&\ar[l]_-{\ell(\nu\rtimes G)} \ell (( \bC'\otimes \phi(G/H))\rtimes G')\ar[d]^{\ell(q_{ \bC'\otimes \phi(G/H)})}     &\ar[l]^-{\simeq}_-{\eqref{qeroighwoiegwergwergwregwerg111}} j^{G}_{!}  \ell(\bC)(G/H)\ar[l] \\\ar@/_0.5cm/[rr]_{!!}
\ell(\bC\rtimes_{r} H) \ar[r]^-{\ell(i_{H})}&\ell( \bC [G/H]\rtimes_{r} G)&  \ar[l]_-{\ell(\nu\rtimes_{r} G)}  \ell (( \bC'\otimes \phi(G/H))\rtimes_{r} G') & } \end{equation} 
The  arrow  marked by $!$  is the analogue of $\ell(i_{H})$ for the  maximal crossed product. 
The  morphism $\nu\rtimes_{r}G$ is equivalent to the Morita equivalence 
$k\rtimes_{r}G$ in the proof of Propositon \ref{wetgiojweirogwergrwegwergw}. The morphism $\nu\rtimes G$
is also a Morita equivalence  with the same argument. % as for Assertion \ref{wrgoibjowrgbfgbgfw4} in the proof of Propositon \ref{wetgiojweirogwergrwegwergw}.
The lower arrow marked by $!!$ is equivalent to the composition $\ell\circ (j\rtimes_{r}H)$ in the proof of Proposition \ref{wetgiojweirogwergrwegwergw} and therefore a unitary equivalence. Analysing the argument for this fact we see that $\ell\circ (j\rtimes_{r}H)$ restricts to a unitary equivalence
$\bC\rtimes^{\alg}H\to \bD\rtimes^{\alg}G$ which extends by continuity to a unitary equivalence 
$ \bC\rtimes H\to \bD\rtimes G$. Therefore the upper arrow marked by $!!$ is also induced by a unitary equivalence.

%We now argue that  that both arrows are induced by Morita equivalences. 
%For the reduced crossed product this is precisely the assertion of  Proposition \ref{wetgiojweirogwergrwegwergw}. The proof in the case
%of the maximal crossed product is similar, but simpler since we do not have to consider $\Lzwei(G,\dots)$.

We apply $\Homol_{\infty}$ to the diagram \eqref{vwervwevwvdvsfdvsdv}, delete 
the third column, and add the definition of $ \Homol_{\bK^{u},r}^{G}(G/H) $ and the transformation $c$. Then we get the commutative diagram
\begin{equation}\label{erwggwerfweff}  \xymatrix{\Homol(\bC\rtimes H)\ar[r]^-{\simeq} \ar[d]^{\Homol(q_{\bC})}
&    \Homol (\bC [G/H]) \rtimes G) \ar[d]^{\Homol(q_{ \bC [G/H]})}
& \ar[l]_-{\simeq} \Homol_{\bC,\max}^{G}(G/H)\ar[l]\ar[d]^{c_{G/H}} \\\Homol(\bC\rtimes_{r} H) \ar[r]^-{\simeq}&\Homol( \bC [G/H]\rtimes_{r} G)\ar[r]_-{\text{Def.~}\ref{iuhquifhiwefqewfqwefqwefef}}^-{\simeq}& \Homol_{\bC,r}^{G}(G/H)}
\end{equation}
whose right square reflects the definition of $c_{G/H}$ above. 
The left horizontal morphisms  are equivalences since we assume that $\Homol$ is Morita invariant.
 The diagram \eqref{erwggwerfweff} gives Assertion \ref{wgwegrecewcs}.

%To see Assertions \ref{reoigjewrogwergewrgwerg1} and \ref{prop_item_Haagerup_equiv_functor} note that the evaluation of $c$ at $G/H$ reduces under the equivalences in \ref{kor_computation_orbit_Morita} and \eqref{qrwgojoiqrfwefqwefqwef} to the canonical morphism
%$$\Homol(q_{\tilde \bK^{u}})\colon \Homol(\tilde \bK^{u}\rtimes H)\to \Homol(\tilde \bK^{u}\rtimes_{r} H)$$
%from the maximal to the reduced crossed product {by $H$.}
{We now show Assertions \ref{reoigjewrogwergewrgwerg1} and \ref{prop_item_Haagerup_equiv_functor}.  If} $H$ is {amenable}, then it $c_{G/H}$ is an equivalence {by Theorem \ref{thm_G_amenable}}. If $H$ is $K$-amenable, then in the special case of $\Homol=\Kcat$ the morphism $c_{G/H}$ is an equivalence by Theorem~\ref{thm_Kamenable_equiv}.
\end{proof}

{We now} relate the functor $K^{\DL,G}_{\C} \colon G\Orb\to \Sp$ introduced  by
Davis--L\"{u}ck in \cite{davis_lueck} with the constructions of the present paper.  
We will actually consider its straightforward generalization  {$$\Homol^{\DL,G}_{ A} \colon G\Orb\to \bS $$}to the case of a $C^{*}$-algebra $ A$ in $\nCalg$  in place of $\C$
and a functor $\Homol \colon \nCcat\to \bS$ which %sends unitary equivalences to equivalences and 
{is Morita invariant} in place of $\Kcat$. 
{The precise  description  of $\Homol^{\DL,G}_{ A} $ will be recalled in Construction \ref{weijgowergrewferff} below.} {The value of $ \Homol^{\DL,G}_{ A} $ }  %We then get 
 %a functor 
%$$\Homol^{\DL,G}_{ A} \colon G\Orb\to \bM$$ whose 
 on the orbit $G/H$ is given by
\begin{equation}\label{qewfijioqjweofqwefqwefqwefqewf}
\Homol^{\DL,G}_{ A}(G/H)\simeq \Homol( A\rtimes_{r}H)\, .
\end{equation}

  \begin{construction}\label{weijgowergrewferff}{\em Let $\Groupoids^{\mathrm{faith}}$ denote the category of very small groupoids and faithful morphisms.
%We start  with  explaining  the construction of $\Homol^{\DL,G}_{ A}$ following \cite{davis_lueck}. 
We have a functor
$\Fun(BG,\Set)\to \Groupoids^{\mathrm{faith}}$ which sends   $S$   in $\Fun(BG,\Set)$ to the action groupoid $S\curvearrowleft G$. The latter  has the following description:
\begin{enumerate}
 \item objects: The set of objects of $S\curvearrowleft G$ is the set $S$.
 \item morphisms:  For $s,s'$ in $S$ the  set of morphisms from $s$ to $s'$  is the subset 
 \[\{g\in G\:|\:gs=s' \}\]
  of $G$.  
 \item The composition is inherited from the multiplication in $G$.
 \end{enumerate}
 A morphism $f\colon S\to S'$ in $\Fun(BG,\Set)$ induces a morphism $$f\curvearrowleft G\colon S\curvearrowleft G\to S'\curvearrowleft G$$ in $\Groupoids^{\mathrm{faith}}$
 which sends $s$ in $S$ to $f(s)$ in $s'$ and acts as natural inclusions on morphism sets.

 For $A$ in $\nCalg$ we have a functor $$\bC_{A,r}^{*} \colon \Groupoids^{\mathrm{faith}}\to \nCcat$$
 defined  as in  \cite{davis_lueck} as follows.
 For a groupoid $\cS$  
 %We can now define the $C^{*}$-category
% $ A\otimes_{r}(S\curvearrowleft G)$ following the description in  \cite{davis_lueck}.
 we first form  the algebraic tensor product $A\otimes^{\alg} \cS$ in $\nClincat$
as in \cite[Sec.\ 6]{startcats} (this construction naturally extends to the non-unital case). Its objects are the objects of $\cS$.  But instead of completing in the maximal norm (which would give $A\otimes_{\max}  \cS$) we complete in the reduced norm described in \cite[Sec.\ 6]{davis_lueck}. 
To do this, for any two objects $s,s'$ in $\cS$ we canonically embed $\Hom_{A\otimes^{\alg} \cS}(s,s')$ into the adjointable bounded operators between Hilbert $A$-modules 
$B(L^{2}(\Hom_{\cS}(s_{0},s),A),L^{2}(\Hom_{\cS}(s_{0},s'),A))$
and take the supremum of the norms of the images over all choices {of} $s_{0}$ in $S$. 
We let $\bC_{A,r}^{*}(\cS)$ be the completion of $A\otimes^{\alg} \cS$.
A morphism $f\colon \cS\to \cS'$ in $\Groupoids^{\mathrm{faith}}$ induces a morphism $ \bC_{A,r}^{*}(\cS)\to \bC_{A,r}^{*}(\cS')$ in the natural way. At this point  it is important that we only consider faithful morphisms between groupoids.
The functor $\bC^{*}_{A,r}$ extends to  a functor between $2$-categories 
(of groupoids, faithful morphisms and equivalences  on the one hand; and $C^{*}$-categories, functors and unitary   equivalences on the other {hand)}
and  sends equivalences of groupoids to  equivalences of $C^{*}$-categories.

 %A\otimes_{r}(f\curvearrowleft G)\colon A\otimes_{r}(S\curvearrowleft G)\to  A\otimes_{r}(S'\curvearrowleft G)\, .$$
The    functor $\Homol_{ A}^{\DL,G}$  is then defined  {as the composition}  %(see  \cite{davis_lueck}, with the correction by \cite{joachimcat})
$$\Homol_{A}^{\DL,G}\colon G\Orb\xrightarrow{S\mapsto  S\curvearrowleft G} \Groupoids^{\mathrm{faith}}\xrightarrow{\bC_{A,r}^{*}}
\nCcat\xrightarrow{\Homol}\bS\, .$$
 
If  $H$ is a subgroup of $G$, then we have
an equivalence of groupoids $$(*\curvearrowleft H)\stackrel{\simeq}{\to} ((G/H)\curvearrowleft G)$$ which sends $*$ to the class $H$.
This equivalence induces a unitary equivalence 
\begin{equation}\label{oijvoijoijvoiasjdvasvsvasdvsddsfsfsdfdfav}
A\rtimes_{r}H \cong \bC^{*}_{A,r} (*\curvearrowleft H)\stackrel{\simeq}{\to} \bC^{*}_{A,r} ((G/H)\curvearrowleft G)
\end{equation}
 in $\Ccat$
 which yields \eqref{qewfijioqjweofqwefqwefqwefqewf} by applying $\Homol$.}\hB
 \end{construction}

{In the following 
  we consider $A$  in  $\nCalg$ and   $ {\Hilb_{c}( A)}$ in $ \Ccat$. Then we have its full  subcategory of unital objects $ \Hilb_{c}( A)^{u}$ in $\Ccat$ and refer  to Example \ref{qrgiooergwefe9} for the  explicit description of the latter in the case that $A$ is unital. Furthermore, let 
 $\Homol:\Ccat\to \bS$ be a Morita invariant functor.
\begin{prop}\label{prop_compare_DL}
If $A$ is unital, then 
 {we have}  a canonical equivalence 
\begin{equation}\label{qergiuhqeruiffwefqwe}
\Homol^{\DL,G}_{ A}\stackrel{\simeq}{\to} \Homol^{G}_{ {\Hilb_{c}}( A)^{u},r}
\end{equation}
in $\Fun(G\Orb,\bS)$.
\end{prop}
 \begin{proof}
{We %now  
 define} a natural transformation of  functors 
 \begin{equation}\label{vrvewtrklmlekvewve}
\kappa\colon \bC^{*}_{A,r}  (-\curvearrowleft G)\to \Hilb_{c}( A)^{u}[ {-}]\rtimes_{r}G
\end{equation} 
 {from $G\Orb$ to $\Ccat$ and obtain the desired transformation in \eqref{qergiuhqeruiffwefqwe} by applying $\Homol$.}
 The evaluation  $\kappa_{S}$ of $\kappa$ at $S$  in $G\Orb$ is {the morphism in $\Ccat$} given as follows:
 %To this end we describe the functor $\kappa_{S}\colon A\otimes_{r}(S\curvearrowleft G)\to {}_{\Alex{\IK}}\bC_{\mathrm{lf}}[\widetilde{S}]\rtimes_{r}G$ for $S$ in $G\Orb$.
 \begin{enumerate}
 \item objects: $\kappa_{S}$ sends the object $s$ in $S=\Ob(  \bC^{*}_{A,r}  (S\curvearrowleft G))$ to
 the object $A_{s}$ in   $  \Hilb_{c}( A)^{u}[ {S}]\rtimes_{r}G$ (see Example \ref{weoigjwioegergwergerg}), where we can consider $A$ as an object of ${\Hilb_{c}}(A)^{u}$ since   $A$ is unital by assumption.
    \item morphisms: Let $s,s'$ be in $S$, let  $g$ in $G$ be such that $gs=s'$, and {let $a$ be in $A$.} 
 Then we can consider $(a,g)$ as a morphism  in $A\otimes^{\alg} (S \curvearrowleft G)$, and therefore as a morphism in $ A\otimes_{r}(S \curvearrowleft G)$.
 We can consider the right-multiplication by $a$ as a morphism
 $a\colon A_{s}\to A_{s'}=g A_{s}$ in  $ \Hilb_{c}( A)
^{u}[ {S}]$. 
 The functor $\kappa_{S}$ sends $(a,g)$ to the morphism $(a,g)\colon A_{s}\to A_{s'}$ in
 $ \Hilb_{c}( A)^{u}[ {S}]\rtimes_{r}G$.
 
  We extend $\kappa_{S}$ by linearity and continuity. 
    \end{enumerate}
    One  checks that $\kappa_{S}$ is well-defined and that the family $\kappa:=(\kappa_{S})_{S\in G\Orb}$
    is a natural transformation. In order to check that $\kappa_{S}$ extends by continuity we do not have to consider estimates. We just check that for a subgroup $H$ of $G$ the functor 
    $\kappa_{G/H}$ identifies $\bC^{*}_{A,r}((G/H)\curvearrowleft G)$ with  
the subcategory $ \bD\rtimes_{r}G$ of $A[ G/H]\rtimes_{r}G$ appearing in the proof of Proposition \ref{wetgiojweirogwergrwegwergw}.
This follows from the fact that both receive unitary equivalences from $A\rtimes_{r}H$ by \eqref{oijvoijoijvoiasjdvasvsvasdvsddsfsfsdfdfav} and
%the map $(1)$ in \eqref{asviuhefiuhsadvasdvcsadvdasvasdv}.

We consider $A$ as a $G$-invariant one-object subcategory of $ {\Hilb_{c}}(A)^{u}$. Let $H$ be a subgroup of $G$.  The inclusion  induces a  morphism
\begin{equation}\label{gewrgwergregegwergregwe}
 A[G/H]\rtimes_{r} G
 \to       \Hilb_{c}( A)^{u} [G/H]\rtimes_{r} G  \, .
 \end{equation}
  Note that  in general $A$ considered as a $C^{*}$-category with a single object is not  additive   so that   we do not have   naturality of the  morphism \eqref{gewrgwergregegwergregwe} with respect to  the argument $G/H$. The morphism \eqref{gewrgwergregegwergregwe}
  in turn  induces the morphism {in} the statement below by applying $\Homol$.

The following  lemma  is an essential step in the proof of Proposition \ref{prop_compare_DL} but might be interesting in its own right.
Its statement  uses the functoriality of $\bC\to \Homol^{G}_{\bC,r}$.

Recall that  by assumption $\Homol \colon \nCcat\to \bS$ is  a  {Morita invariant functor and  $A$ is unital.}\begin{lem}\label{qeorigjoegergregwgre}
{The inclusion of $C^{*}$-categories $A\to \Hilb_{c}(A)^{u}$ induces for every $G/H$ in $G\Orb$ an} equivalence
\begin{equation}\label{ewfuzhfquzfhiwef}
\Homol^{G}_{A,r}(G/H)\stackrel{\simeq}{\to} \Homol^{G}_{{\Hilb_{c}}(A)^{u},r}(G/H)\, .
\end{equation}
\end{lem}
\begin{proof}
%with an action of $G$ and consider the $C^{*}$-category with $G$-action $\tilde \bC=\widetilde{\Hilb^{G}(A)}$ as before. In $\tilde \bC$ we consider the $G$-invariant ideal $\tilde{\mathbb{K}}=\widetilde{\Hilb^{G}(A)_{c}} $  of compact operators.  By definition, for any two objects  $C$ and $C'$ in $ \bC$ the space of morphisms in $\Hom_{\tilde{\mathbb{K}}}(C,C')$ is the subspace of morphisms in
%$\Hom_{\bC}(C,C')$ of operators from $C$ to $C'$ which are compact in the sense of $A$-Hilbert $C^{*}$-modules. Using Remark~\ref{rem_ideal_hereditarily_additive} we conclude that we get a well-defined functor ${}_{\mathbb{K}}\bC_{G,r}\colon G\Orb \to \Ccat$, and Theorem~\ref{qrfoiqfewewfqfewfeqf} applies to $\Kcat{}_{\mathbb{K}}\bC_{G,r}$ and states that it is a CP-functor.
%
%In the following we argue that the restriction of the functor $\Homol{}_{\mathbb{K}}\bC_{G,r}$ to $G\Orb_i$ is equivalent to $\Homol{}_{(A)}\bC_{G,r}$ in the case that $G$ acts trivially on $A$ and $\Homol$ is Morita invariant.
%
%
% 
%
%
% So we assume that $G$ acts trivially on $A$.
% \Alex{Since $A$ is unital, we have $\Hom_{(\tilde{A})}(A,A) \cong A$; and furthermore, we have $\Hom_{\tilde{\IK}}(A,A) \cong A$.}
%\Alex{This means that we have a well-defined} inclusion
% $(\tilde A)\to \tilde{\mathbb{K}}^{u}$, which induces a natural transformation between functors on the \Alex{restricted} orbit category \Alex{$G\Orb_i$:}
%\begin{equation}\label{vevdfvfdvvdfsvfdvsfdv}
%\Homol{}_{(A)}\bC_{G,r}\to \Homol{}_{\mathbb{K}}\bC_{G,r}\, .
%\end{equation}   
%We must  verify that
%\eqref{vevdfvfdvvdfsvfdvsfdv} is an equivalence.
% 
Under {the equivalence provided by} Corollary~\ref{kor_computation_orbit_Morita} the morphism  {in} \eqref{ewfuzhfquzfhiwef}  
 corresponds to
 \begin{equation}\label{erfrefefrwfreggergw}
\Homol(A \rtimes_{r}H)\to    \Homol({\Hilb_{c}(A)}^{u}  \rtimes_{r}H)
\end{equation}  induced by the inclusion 
the inclusion $A\to {\Hilb_{c}(A)}^{u}$.
As observed in Example \ref{qrgiooergwefe9}, using that $A$ is unital, we have an equality
$ {\Hilb_{c}(A)}^{u}=  \Hilb(A)^{\fg,\proj} $.
The inclusion
 $A\to \Hilb^{G}(A)^{\fg,\proj}$ 
is a Morita equivalence by Example \ref{qergoijwergioejrgwergwregwergre}.
%
%
%We observe that the subcategory of unital objects in $\mathbb{K}$ is given by $ \mathbb{K}^{u}=\Hilb^{G}(A)^{\fg,\proj}$.   
%Since also $(\tilde A)=(\tilde A)^{u}$  the problem boils down to showing that
%\begin{equation}\label{qergoivhefiovcsfasv} 
%(\tilde A) \rtimes^{\widetilde{\Hilb^{G}}(A)}_{r}H\to  \widetilde{\Hilb^{G}}(A)^{\fg,\proj}\rtimes^{\widetilde{\Hilb^{G}}(A)}_{r}H
%\end{equation}
%is a Morita equivalence.
% The inclusion
% $(\tilde A)\to \widetilde{\Hilb^{G}}(A)^{\fg,\proj}$ 
%is a Morita equivalence by Example \ref{qergoijwergioejrgwergwregwergre}.
 By Proposition \ref{werijguhwerigvwergwer9} we conclude that  $$ A \rtimes_{r}H \to    {\Hilb_{c}(A)}^{u}  \rtimes_{r}H $$ is a Morita equivalence. Since $\Homol$ is Morita invariant we see that   \eqref{erfrefefrwfreggergw} is an equivalence.
 \end{proof}

{We  {now  finish} the proof of Proposition \ref{prop_compare_DL}.} {Let $\kappa$ be as in \eqref{vrvewtrklmlekvewve}.}  As in the proof of Proposition \ref{wetgiojweirogwergrwegwergw} let  $k\colon \bD\to A[ G/H] $ denote the  inclusion of the full $G$-invariant subcategory of  $A[ G/H] $ of objects which are supported
on a single point of $G/H$. As noted above,    $\kappa_{G/H}$ induces a unitary equivalence between $\bC^{*}_{A,r}((G/H)\curvearrowleft G)$  and 
  $ \bD\rtimes_{r}G$. 
The evaluation of $\Homol(\kappa)$  {at} $G/H$ has the following factorization:
 \begin{eqnarray*} \Homol^{\DL,G}_{ A}(G/H)&\stackrel{\text{Def.}}{\simeq}
 &
 \Homol(\bC^{*}_{A,r}  ((G/H)\curvearrowleft G))\\&{\stackrel{ \Homol(\kappa_{G/H})}{\simeq}}%\stackrel{ \Homol(\kappa_{G/H}),\eqref{vrvewtrklmlekvewve}}{\simeq}
 & \Homol(  \bD\rtimes_{r}G)\\&\stackrel{\Homol(k\rtimes_{r}G)}{\simeq}
% \stackrel{\Homol(k) ,  \eqref{asviuhefiuhsadvasdvcsadvdasvasdv}}{\simeq}
&  \Homol(A [ {G/H}]\rtimes_{r}G)\\&\stackrel{\text{Lem.} \ref{qeorigjoegergregwgre}}{\simeq}&\Homol( \Hilb( A)^{u}[ {G/H}]\rtimes_{r}G)\\&\stackrel{\text{Def.}}{\simeq}&
 \Homol^{G}_{{\Hilb_{c}(A)^{u},r}}(G/H)
\end{eqnarray*}
through equivalences, where we use that $k\rtimes_{r}G$ is  a Morita  equivalence
as shown in the proof of Proposition \ref{wetgiojweirogwergrwegwergw} (Assertion \ref{wrgoibjowrgbfgbgfw4}).
\end{proof}
%\end{ex}

\begin{rem}
 {For a generalization of Construction \ref{weijgowergrewferff} to $C^{*}$-algebras with non-trivial $G$-action we refer to \cite{kranz} and the review in \cite[Sec. 15]{bel-paschke}. 
%The corresponding generalization of Proposition \ref{prop_compare_DL}
%is \cite[Prop. 1\uli{6}.18]{bel-paschke}.
\hB}
\end{rem}

%\begin{ex}\label{oitghjrtiohrthetrhehh}
Let $ \bC$  in $\Fun(BG,\Ccat)$ be    additive.  We consider   an  inclusion of groups $i:G\to K$  and let  $Bi \colon BG\to BK$ denote the induced functor. 
We can then choose an object 
$\Ind_{G}^{K}(  \bC)$  in $\Fun(BK,\Ccat)$  
such that there is  an equivalence $Bi_{!}\ell_{BG}(  \bC) \simeq \ell_{BK}(\Ind_{G}^{K}(  \bC))$.
Note that $\Ind_{G}^{K}(\bC)$ is well-defined  up to unitary equivalence.
  Let $\Homol \colon \nCcat\to \bS$ be  a  functor.    {The following result is the analogue of \cite[Prop. 2.5.8]{kranz}.}
  We use the notation introduced in Definition \ref{iuhquifhiwefqewfqwefqwefef} for functors on the orbit category and the same symbols also for the corresponding equivariant homology theories.
 % {such that $\Homol$ preserves coproducts.} 
\begin{prop}\label{eroighqeirgoregwergwergewrgewrgwe}\mbox{}
\begin{enumerate}\item 
If $\bS$ is cocomplete and $\Homol$ is Morita invariant and  preserves small coproducts, then
for every $K$-CW-complex $X$  with amenable  stabilizers we have an equivalence
\begin{equation}\label{qerijioqwejrqwerewr}
\Homol^{K}_{\Ind_{G}^{K}(  \bC),r}(X)\simeq \Homol^{G}_{ \bC,r}(\Res^{K}_{G}(X))\, .
\end{equation}
\item For every $K$-CW-complex $X$  with $K$-amenable stabilizers  we have an equivalence
\begin{equation}\label{eq_equiv_KAmen_ind_res}
\K^{K}_{\Ind_{G}^{K}( \bC),r}(X)\simeq \K^{G}_{  \bC,r}(\Res^{K}_{G}(X))\, .
\end{equation}
\end{enumerate}
\end{prop}
\begin{proof}
We let $\Am$ and $\KAm$ denote the families of amenable and $K$-amenable subgroups.
The presheaves $Y^{K}(X)$ and $Y^{G}(\Res^{K}_{G}(X))$ (see \eqref{eqwfuhiouqwfhqwfeqef}) are supported on $K_{{\Am}}\Orb$ and $G_{{\Am}}\Orb$, respectively (or on $K_{\KAm}\Orb$, resp.\ $G_{\KAm}\Orb$ in the second case). In view of \eqref{qewfoiuioqwefqwefewfqefew} and Proposition \ref{oiqrjgpoiregreewegegw}.\ref{reoigjewrogwergewrgwerg1}
we have equivalences
\begin{equation}\label{eq_equiv_cor_1}
\Homol^{K}_{\Ind_{G}^{K}(  \bC),\max}(X)\simeq \Homol^{K}_{\Ind_{G}^{K}(  \bC),r}(X)
\end{equation}
and
\begin{equation}\label{eq_equiv_cor_2}
\Homol^{G}_{ \bC,\max}(\Res^{K}_{G}(X))\simeq  \Homol^{G}_{  \bC,r}(\Res^{K}_{G}(X))\, ,
\end{equation}
By Proposition \ref{qrfiuwhiufewfqwefqwef}.\ref{qeroigjoqergqgqfqewfqewf1}
we have an equivalence
$$ \Homol^{K}_{\Ind_{G}^{K}(  \bC),\max}(X)\simeq \Homol^{G}_{  \bC,\max}(\Res^{K}_{G}(X){)}\, .$$
The combination of these equivalences yields the equivalence  \eqref{qerijioqwejrqwerewr}.
%asserted in the corollary. 
For \eqref{eq_equiv_KAmen_ind_res} we use Proposition~\ref{oiqrjgpoiregreewegegw}.\ref{prop_item_Haagerup_equiv_functor} to conclude the equivalences \eqref{eq_equiv_cor_1} and \eqref{eq_equiv_cor_2} in the case of $\Homol=\Kcat$ (note that $\Kcat$ preserves small coproducts by Corollary \ref{kor_Kcat_preserves_coprods}).
\end{proof} 

\begin{rem}
We apply  Proposition \ref{eroighqeirgoregwergwergewrgewrgwe} to $X=E_{\Fin}K$. Since $\Res^{K}_{G}(E_{\Fin}K)\simeq E_{\Fin}G$ we get an equivalence 
 \begin{equation}\label{ewqfwefqwfqewfefqewf}
\Homol^{K}_{\Ind_{G}^{K}(  \bC),r}(E_{\Fin}K)\simeq \Homol^{G}_{  \bC,r}( E_{\Fin}G)\, .
\end{equation} 
 In the case of $\Homol=\Kcat$ the left and right hand sides of this equivalence 
 constitute the domains of corresponding Baum--Connes assembly maps.   
  In this case such an equivalence 
  (with a completely different model of equivariant $K$-homology and a completely different proof) has first been  
 obtained by \cite{oyono}, see \cite[Thm.\ 2.2]{MR1836047}.  %The comparison of models will be discussed further in \cite{compass}.
\hB \end{rem}

% 
% 
% On the orbit $T$ in $G\Orb$ it sends the pair $(C,t)$ of an object
%  to the object $C_{t}$.
% For $h'$ in $G'$ we have
% $h'(C,t)=(h'C,h't)$ which is sent to $(h'C)_{h't}$.
% On the other hand
% $h'C_{t}=(h'C)_{h't}$.
% This functor is a Morita equivalence.
% In fact, it is fully faithful and every object in the target is a finite sum of objects in the image.
%Using that $-\rtimes G$ preserves Morita equivalences we get a Morita  equivalence
%$$j^{G}_{!}\bC \to  \bC[-]\rtimes G\ .$$
%We have a canonical morphism 
%  $$\bC[-]\rtimes G\to \bC[-]\rtimes G$$
%  which is an isomorphism on orbits with finite stablizers (since in this case the crossed products are algebraic).
%  Let $\Homol$ be any homolgical functor. Then we have an equivalence
%  of the restriction of functors
%  $$\Homol \bC_{G}\stackrel{\simeq}{\to}   \Homol(\bC[-]\rtimes_{r}
%G)$$ to $G_{\Fin}\Orb$.  
%This implies that the values of the corresponding homology theories are equivalent on all $GW$-CW complexes with finite stabilizers.
%
%Let $H$ be a subgroup of $G$. Note that $\Res^{G}_{H}E_{\Fin}G\simeq E_{\Fin}H$.
%
%\begin{kor}\label{qreoigergrewgwer9}
%We have an equivalence
%$$ \Homol(\bC[-]\rtimes_{r}
%H)(E_{\Fin} H)\simeq \Homol((i_{!}\bC)[-]\rtimes_{r}G)(E_{\Fin}G)\ .$$
%\end{kor}
%
%

\begin{rem}\label{wtoigjwotgergwegwerg}
The following Theorem \ref{qrfoiqfewewfqfewfeqf} is one of the main  results of  {the subsequent paper} 
 \cite{coarsek} for which the present paper provides the   foundations
 concerning $C^{*}$-categories. Let $  \bK$ be in $\Fun(BG,\nCcat)$.  We consider the case $\Homol=\Kcat$  and  use the more readable notation 
 $\K^{G}_{  \bK^{u},r} \coloneqq (\Kcat)^{G}_{ \bK^{u},r}$ for the functor defined in Definition \ref{iuhquifhiwefqewfqwefqwefef}.
 {For the   notion of  a CP-functor $G\Orb\to \bS$   we refer to  \cite{desc} or \cite{coarsek}.} As explained in  \cite{desc}, \cite[Sec.\ 1]{coarsek}, or  in \cite[Sec.\ 6.5]{unik}, being a CP-functor has interesting consequences for the injectivity of assembly maps involving this functor.
 \hB
\end{rem}
\begin{theorem}[{\cite[Thm.~13.4]{coarsek}}]\label{qrfoiqfewewfqfewfeqf}
If  $\bK$ admits all very small AV-sums,  then
$$\K^{G}_{  \bK^{u},r} \colon G\Orb\to \Sp\, $$
is a CP-functor. %\Alexx{which is hereditary for surjective group homomorphisms $\phi\colon G \to Q$ with $K$-amenable kernel.}
\end{theorem}

\subsection{Appendix: Some equivariant homotopy theory}\label{oiejgoregrqrgqgregw}

Let $K$ be a group and $K\Top$ be the category of $K$-topological spaces. 
A morphism $f \colon X\to X'$ in $K\Top$ is an equivariant weak equivalence if
it induces  weak equivalences between the fixed-points sets
$f^{H} \colon X^{H}\to X^{\prime,H}$ for all subgroups $H$ of $K$. 
 In the following let $\Map_{K\Top}(-,-)$ denote the topological mapping space of equivariant maps 
 and $\ell \colon \Top\to \Spc$ be the canonical morphism which presents the $\infty$-category $\Spc$ as the Dwyer--Kan localization of $\Top$ at the weak equivalences.
 By Elmendorf's theorem  the functor
 \begin{equation}\label{eqwfuhiouqwfhqwfeqef}
Y^{K} \colon K\Top\to \PSh(K\Orb)\, , \quad X\mapsto (S\mapsto \ell(\Map_{K\Top}(S_{\disc},X)))
\end{equation}
 presents {$\PSh(K\Orb)$ as the localization}  of $K\Top$ at the  equivariant weak {equivalences.}  Here $S_{\disc}$ denotes the $K$-orbit $S$ considered as discrete $K$-topological space.

For a subgroup  $G$   of $K$
we have an adjunction
$$\Ind_{G}^{K}:G\Top\leftrightarrows K\Top:\Res^{K}_{G}\, ,$$
where the induction functor is given by
$$X\mapsto \Ind_{G}^{K}(X) \coloneqq K\times_{G}X\, .$$
Considering
  the orbit category $K\Orb$ as a full subcategory of $K\Top$ of discrete transitive $K$-topological spaces,
the induction functor restricts to the functor
$$i_{G}^{K} \colon G\Orb\to K\Orb\, .$$ 
It is a formal consequence of the definitions 
that
\begin{equation}\label{qewfoijoqwefqwefqewf}
\xymatrix{K\Top\ar[r]^{\Res^{K}_{G}}\ar[d]^{Y^{K}}&G\Top\ar[d]^{Y^{G}}\\
\PSh(K\Orb)\ar[r]^{i^{K,*}_{G}}&\PSh(G\Orb)}
\end{equation}
 commutes.
A functor $E^{G} \colon G\Orb\to \bS$ with  cocomplete target represents an $\bS$-valued $G$-equivariant homology theory  $E^{G} \colon G\Top\to \bS$  denoted by the same symbol.
We form the left Kan extension
$$\xymatrix{G\Orb\ar[rr]^-{E^{G}}\ar[dr]_{\text{Yoneda\ }}\ar@{}[drr]^{\Rightarrow}&& \bS\,.\\
&\PSh(G\Orb)\ar[ur]_-{\hat E^{G}}&}$$
Then the value of the homology theory on $X$ in $K\Top$ is given by \begin{equation}\label{qewfoiuioqwefqwefewfqefew}
E^{G}(X)\simeq \hat E^{G}(Y^{G}(X))\, .
\end{equation}
 We form the left Kan extension    $E^{K} \coloneqq i_{G,!}^{K}E^{{G}} \colon K\Orb\to \bS$ of $E^{G}$ as in 
$$\xymatrix{G\Orb\ar[rr]^-{E^{G}}\ar[dr]_-{i^{K}_{G}}\ar@{}[drr]^{\Rightarrow}&& \bS\,.\\
&K\Orb\ar[ur]_-{E^{K}}&}$$
It represents a $K$-equivariant homology theory. Let $X$ be in $K\Top$.
\begin{lem}\label{wejiogwergergwergwerg}
We have a natural equivalence
$ E^{K}(X)\simeq E^{G}(\Res^{K}_{G}(X))$.
\end{lem}
\begin{proof}
We have
$$ E^{G}(\Res^{K}_{G}(X))\simeq \hat E^{G}(Y^{G}(\Res^{K}_{G}(X)))\stackrel{\eqref{qewfoijoqwefqwefqewf}}{\simeq} \hat E^{G}(i^{K,*}_{G}(Y^{K}(X)))\, .$$
Let $y_{K\Orb} \colon K\Orb\to \PSh(K\Orb)$ denote the Yoneda embedding.
Then we have an equivalence  
$E^{K}\simeq i^{K}_{G,!}E^{G}\simeq  \hat E^{G}\circ i^{K,*}_{G}\circ y_{K\Orb}$ which implies   
 $ \hat E^{K}\simeq   \hat E^{G}\circ i^{K,*}_{G}$. 
We get  $$\hat E^{G}(i^{K,*}_{G}(Y^{K}(X)))\simeq
 \hat E^{K}(Y^{K}(X))\simeq  E^{K}(X)\, .$$ 
The desired equivalence follows from concatenating the two displayed chains of equivalences.
\end{proof}

The left Kan extension functor $i^{K}_{G,!}$ only involves forming coproducts. More precisely, we have the following assertion.
Let $A \colon G\Orb\to \bA$ be a functor with a cocomplete target and $B \colon \bA\to \bB$ be a second functor to a cocomplete target $\bB$.
\begin{lem}\label{q3roigerwgwergwergwerg}
If $B$  preserves small   coproducts, then the canonical transformation  is an equivalence
$ i^{K}_{G,!}(B\circ A) \simeq  B\circ i^{K}_{G,!}A $.
\end{lem}
\begin{proof}
We have a natural transformation
$ i^{K}_{G,!}(B\circ A)\to B\circ i^{K}_{G,!}A$. We use the pointwise formula for the left Kan extension in order to
evaluate this transformation at $S$ in $K\Orb$.  
The objects of $G\Orb_{/S}$ are morphisms $ K\times_{G}T\to S$ for $T$ in $G\Orb$  which are in bijection with morphisms
$T\to \Res^{K}_{G}(S)$ in $\Fun(BG,\Set)$. Hence the category $G\Orb_{/S}$ decomposes into a union of categories
$G\Orb_{/R}$, where $R$ runs over the set $G\backslash S$ of $G$-orbits in $\Res^{G}_{K}(S)$. Each component has a final object $R$. 
Hence we get the following chain of equivalences:
\begin{align*}
(i^{K}_{G,!}(B\circ A))(S)& \simeq \coprod_{R\in G\backslash S} B(A(R))
 \simeq B\big( \coprod_{R\in G\backslash S}  A(R) \big)\\ 
& \simeq B(( i^{K}_{G,!}A)(S))
 \simeq  (B\circ i^{K}_{G,!}A)(S)\,.\qedhere
\end{align*}
\end{proof}

\bibliographystyle{alpha}
\bibliography{forschung}

\end{document}